\documentclass[11pt,a4paper]{amsbook}

\usepackage{amsmath,amssymb,amsbsy,amscd,amsthm}

\newtheorem{thm}{Theorem}[section]
\newtheorem{prop}[thm]{Proposition}
\newtheorem{lem}[thm]{Lemma}
\newtheorem{cor}[thm]{Corollary}
\newtheorem{claim}[thm]{Claim}

\newtheorem{tthm}{Theorem}
\theoremstyle{definition}
\newtheorem{definition}{Definition}[section]

\newtheorem*{tgp}{Tame Generators Problem}

\newtheorem{conj}{Conjecture}
\newtheorem{qqq}[conj]{Question}
\newtheorem{pp}[conj]{Problem}

\newtheorem{q}{Question}
\newtheorem{problem}[q]{Problem}

\theoremstyle{remark}

\begin{document}

\title{Wildness of polynomial automorphisms 
in three variables}

\author{Shigeru Kuroda}

\address{Department of Mathematics and Information Sciences, 
Tokyo Metropolitan University,  
1-1Minami-Osawa, Hachioji, 
Tokyo 192-0397, Japan}
\email{kuroda@tmu.ac.jp}

\date{}

\subjclass[2000]{Primary 14R10; Secondary 13F20. }

\maketitle

\thanks{Partly supported by the Grant-in-Aid for 
Young Scientists (B) 21740026, 
The Ministry of Education, Culture, 
Sports, Science and Technology, Japan.}

\tableofcontents

\newcommand{\E }{{\rm E}}
\newcommand{\J }{{\rm J}}
\newcommand{\T }{{\rm T}}
\newcommand{\U }{{\cal U}}
\newcommand{\I }{{\cal I}}
\newcommand{\Aff }{{\rm Aff}}
\newcommand{\Affo }{{\rm E_0}}
\newcommand{\Aut }{{\rm Aut}}
\newcommand{\GL }{{\it GL}}
\newcommand{\Exp }{\mathop{\rm Exp}\nolimits}
\newcommand{\Diag }{\mathop{\rm Diag}\nolimits}
\newcommand{\End }{\mathop{\rm End}\nolimits}
\newcommand{\lc }{\mathop{\mathrm{lc}}\nolimits}
\newcommand{\lt }{\mathop{\mathrm{lt}}\nolimits}
\newcommand{\gr }{\mathop{\rm gr}\nolimits}
\newcommand{\trd }{\mathop{\rm trans.deg}\nolimits}

\newcommand{\lcm}{{\rm lcm}}
\newcommand{\Der}{\mathop{\rm Der}\nolimits}
\newcommand{\lnd}{\mathop{\rm LND}\nolimits}
\newcommand{\rank}{\mathop{\rm rank}\nolimits}
\newcommand{\degw}{\deg_{\w}}
\newcommand{\degv}{\deg_{\vv}}
\newcommand{\degu}{\deg_{\uu}}

\newcommand{\rr}{{\bf r}}
\newcommand{\qq}{{\bf q}}
\newcommand{\hh}{{\bf h}}
\newcommand{\w}{{\bf w}}
\newcommand{\id	}{{\rm id}}
\newcommand{\zs}{\{ 0\} }
\newcommand{\sm}{\setminus}
\newcommand{\A}{G_{\w }}
\newcommand{\B}{{\mathcal B}}
\newcommand{\aA}{{\mathcal A}^{\w}}
\newcommand{\C}{{\bf C}}
\newcommand{\R}{{\bf R}}
\newcommand{\Q}{{\bf Q}}
\newcommand{\N}{{\bf N}}
\newcommand{\Z}{{\bf Z}}
\newcommand{\e}{{\bf e}}
\newcommand{\vv}{{\bf v}}

\newcommand{\uu}{{\bf u}}
\newcommand{\bt}{{\bf t}}

\newcommand{\Zn}{{\Z _{\geq 0}}}
\newcommand{\Gammap}{{\Gamma _{> 0}}}
\newcommand{\ep}{\epsilon}
\newcommand{\sym}{\mathfrak{S}}
\newcommand{\x}{{\bf x}}
\newcommand{\kx}{k[{\bf x}]}
\newcommand{\Rx}{R[{\bf x}]}
\newcommand{\Ry}{R[{\bf y}]}
\newcommand{\Kx}{K[{\bf x}]}
\newcommand{\Lx}{L[{\bf x}]}
\newcommand{\Sx}{S[{\bf x}]}
\newcommand{\kapx}{\kappa [{\bf x}]}

\newcommand{\Gi}{G(x_1)}

\renewcommand{\Theta }{F}

\newcommand{\kxr}{k({\bf x})}
\newcommand{\y}{{\bf y}}
\newcommand{\ky}{k[{\bf y}]}
\newcommand{\pl }{\mathop{\rm pl}\nolimits}

\newcommand{\el}{l}

\chapter*{Introduction}
\label{sect:intro}
\setcounter{equation}{0}

For each integral domain $R$, 
we denote by 
$\Rx=R[x_1,\ldots ,x_n]$ the polynomial ring 
in $n$ variables over $R$, 
where 
$\x =\{ x_1,\ldots ,x_n\} $ is a set of variables 
and $n\in \N $. 
For an $R$-subalgebra $A$ of $\Rx $, 
we consider the automorphism group 
$\Aut (\Rx /A)$ of the ring $\Rx $ over $A$. 
We say that $\phi \in \Aut (\Rx /R)$ is {\it affine} 
if $\deg \phi (x_i)=1$ for $i=1,\ldots ,n$, 
or equivalently 
$$
\phi (x_i)=\sum _{j=1}^na_{i,j}x_j+b_i
$$
for $i=1,\ldots ,n$ 
for some $(a_{i,j})_{i,j}\in \GL (n,R)$ 
and $b_1,\ldots ,b_n\in R$. 
Here, 
$\deg f$ denotes the total degree of $f$ for each $f\in \Rx $. 
We say that $\phi \in \Aut (\Rx /R)$ is {\it elementary} 
if $\phi $ belongs to $\Aut (\Rx /A_i)$ 
for some $i\in \{ 1,\ldots ,n\} $, 
where 
\begin{equation}\label{eq:A_i}
A_i:=R[\x \sm \{ x_i\} ]. 
\end{equation}
If this is the case, 
then we have $\phi (x_i)=\alpha x_i+f$ 
for some $\alpha \in R^{\times }$ and $f\in A_i$, 
since $\Aut (\Rx /A_i)=\Aut (A_i[x_i]/A_i)$. 
By $\Aff (R,\x )$, $\E (R,\x )$, 
and $\T (R,\x )$, we denote 
the subgroups of $\Aut (\Rx /R)$ generated by 
all the affine automorphisms, 
all the elementary automorphisms, 
and $\Aff (R,\x )\cup \E (R,\x )$, 
respectively. 
We caution that 
$\Aff (R,\y )$ (resp.\ $\E (R,\y )$ 
and $\T (R,\y )$) may not be equal to 
$\Aff (R,\x )$ (resp.\ $\E (R,\x )$ 
and $\T (R,\x )$) 
for another set $\y $ of variables, 
i.e., a subset of $\Rx $ with $n$ elements such that $\Ry=\Rx $. 
When $R$ and $\x $ are clear from the context, 
we sometimes say that $\phi \in \Aut (\Rx /R)$ 
is {\it tame} if $\phi $ belongs to $\T (R,\x )$, 
and {\it wild} otherwise.

The following problem is 
a fundamental problem in polynomial ring theory.

\begin{tgp}\rm 
When is $\T (R,\x )$ equal to $\Aut (\Rx /R)$? 
\end{tgp}

The equality holds when $n=1$. 
In fact, 
every element of $\Aut (\Rx /R)$ 
is affine and elementary.

When $n=2$, 
it is known that $\T (R,\x )=\Aut (\Rx /R)$ 
if and only if $R$ is a field. 
Here, 
the ``if" part is due to Jung~\cite{Jung} 
in the case where $R$ is a field of characteristic zero, 
and to van der Kulk~\cite{Kulk} in the general case. 
The ``only if" part is due to 
Nagata~\cite[Exercise 1.6 of Part 2]{Nagata}.

Throughout this monograph, 
$k$ denotes an arbitrary field of characteristic zero. 
In the case of $n=3$, 
Shestakov-Umirbaev~\cite{SU} gave a criterion for 
deciding whether a given element of 
$\Aut (\kx /k)$ belongs to $\T (k, \x )$. 
As a corollary, 
they proved a statement 
equivalent to the following theorem 
(\cite[Corollary 10]{SU}).

\begin{tthm}[Shestakov-Umirbaev]\label{thm:tame23}
If 
$n=3$ and $k$ is a field of characteristic zero, 
then it holds that 
$$
\Aut (\kx /k[x_3])\cap \T (k,\x )
=\T (k[x_3],\{ x_1,x_2\} ). 
$$
\end{tthm}

It was previously known that some elements of 
$\Aut (\kx /k[x_3])$ 
do not belong to $\T (k[x_3],\{ x_1,x_2\} )$. 
For example, 
Nagata defined $\psi \in \Aut (\kx /k[x_3])$ by 
\begin{equation}\label{eq:Nagataautom}
\psi (x_1)=
x_1-2(x_1x_3+x_2^2)x_2-(x_1x_3+x_2^2)^2x_3,\ \ 
\psi (x_2)=
x_2+(x_1x_3+x_2^2)x_3, 
\end{equation}
and proved that $\psi $ does not belong to 
$\T (k[x_3],\{ x_1,x_2\} )$. 
He also conjectured that 
$\psi $ does not belong to $\T (k,\x )$ 
(see Theorem 1.4 and Conjecture 3.1 of Part 2 of \cite{Nagata}). 
By means of Theorem~\ref{thm:tame23}, 
Shestakov and Umirbaev 
decided that these elements of 
$\Aut (\kx /k[x_3])$ do not belong to $\T (k,\x )$. 
Thereby, 
Nagata's conjecture was also solved in the affirmative. 
For this work, 
Shestakov and Umirbaev 
were awarded the E.\ H.\ Moore Prize for 2007 
by the American Mathematical Society. 
The Tame Generators Problem is not solved in the other cases.

The purpose of this monograph is to study when 
elements of $\Aut (\kx /k)$ do not belong to 
$\T (k,\x )$ in the case of $n=3$ in more detail. 
Let $\phi $ be an element of $\Aut (\kx /k)$. 
If $\phi $ fixes two variables, 
then $\phi $ is elementary, 
and hence belongs to $\T (k,\x )$. 
Assume that $\phi $ fixes exactly one variable. 
Then, 
$\phi $ 
may not belong to $\T (k,\x )$ anymore. 
Thanks to Theorem~\ref{thm:tame23}, 
the study of such automorphisms 
is reduced to 
the case of two variables. 
For this reason, 
Part~\ref{part:domain} of this monograph 
is devoted to developing a theory of 
automorphisms 
in two variables over a domain. 
As applications, 
we prove the wildness of elements of $\Aut (\kx /k)$ 
fixing one variable 
in various situations.

Recently, 
the author generalized 
the theory of Shestakov and Umirbaev 
in \cite{SU1} and \cite{SU2}. 
This makes it possible to 
decide efficiently 
whether a given element of $\Aut (\kx /k)$ 
belongs to $\T (k,\x )$. 
In Part~\ref{chap:GSU} of this monograph, 
we deal with elements of $\Aut (\kx /k)$ 
which fix no variable. 
We prove the wildness of such elements of $\Aut (\kx /k)$ 
using the generalized Shestakov-Umirbaev theory.

Our strategy for studying wildness of automorphisms 
is as follows. 
For an $R$-subalgebra $A$ of $\Rx $, 
consider the {\it tame intersection}
$$
\Aut (\Rx /A)\cap \T (R,\x )
$$
with the automorphism group over $A$. 
If an element of $\Aut (\Rx /A)$ 
does not belong to this subgroup, 
then the element 
does not belong to $\T (R,\x )$ 
by definition. 
Hence, 
determining the tame intersection 
is the same as giving a sufficient condition for wildness 
of elements of $\Aut (\Rx /A)$. 
For example, 
Theorem~\ref{thm:tame23} 
describes the structure of the above 
subgroup for $R=k$, $n=3$ and $A=k[x_3]$, 
from which it was decided that 
some elements of $\Aut (\kx /k[x_3])$ 
do not belong to $\T (k,\x )$. 
We will prove the wildness of 
a great many automorphisms 
by studying tame intersections.

The method of tame intersection 
is particularly effective 
in the study of 
so-called exponential automorphisms. 
We say that 
$D\in \Der _R\Rx $ is {\it locally nilpotent} 
if $D^l(f)=0$ holds for some $l\in \N $ for each $f\in \Rx $. 
We denote by $\lnd _R\Rx $ 
the set of all the locally nilpotent derivations 
of $\Rx $ over $R$. 
Assume that $R$ is a $\Q $-domain, 
and $D$ is an element of $\lnd _R\Rx $. 
Then, 
we can define the {\it exponential automorphism} 
$\exp D\in \Aut (\Rx /R)$ by 
$$
(\exp D)(f)=\sum _{l\geq 0}
\frac{D^l(f)}{l!}
$$ 
for $f\in \Rx $. 
Since various automorphisms of interest have this form, 
it is of great importance 
to find conditions for wildness of exponential automorphisms. 
Note that $\ker D$ is an $R$-subalgebra of $\Rx $, 
and $(\exp D)(f)=f$ holds 
for each $f\in \ker D$. 
Hence, $\exp D$ always belongs to $\Aut (\Rx /\ker D)$. 
Therefore, 
the study of the wildness of $\exp D$ 
is reduced to the study of 
the tame intersection
$$
\Aut (\Rx /\ker D)\cap \T (R,\x ). 
$$

Generalizing the idea of ``tame intersection", 
we also formulate the notion 
of ``W-test polynomial" in Chapter \ref{sect:GSU} 
(Definition~\ref{defn:W-test}). 
This is an effective new technique 
for proving the wildness of automorphisms 
which are not explicitly defined.

Let us list problems and questions 
to be studied in this monograph. 
For $f\in \Rx $ and $D\in \Der _R\Rx $, 
it is known that the element 
$fD$ of $\Der _R\Rx $ 
is locally nilpotent 
if and only if $D(f)=0$ and $D$ is locally nilpotent 
(cf.~\cite[Corollary 1.3.34]{Essen}). 
It is also known that, 
even if $\exp D$ is tame, 
$\exp fD$ may be wild for some $f\in \ker D$. 
It is natural to ask 
whether the converse is true in general	.

\begin{q}\label{p:fixed points}\rm
Let $D$ be a nonzero element of $\lnd _R\Rx $. 
Does $\exp D$ always belong to $\T (R,\x )$ 
if $\exp fD$ belongs to $\T (R,\x )$ 
for some $f\in \ker D\sm \zs $? 
\end{q}

We say that $D\in \Der _R\Rx $ is {\it triangular} 
if $D(x_i)$ belongs to $R[x_1,\ldots ,x_{i-1}]$ 
for $i=1,\ldots ,n$. 
If this is the case, 
then $D$ is locally nilpotent. 
Moreover, 
we have 
$$
(\exp D)(x_i)=x_i+f_i
$$ 
for some $f_i\in R[x_1,\ldots ,x_{i-1}]$ 
for $i=1,\ldots ,n$. 
Hence, 
$\exp D$ belongs to $\E (R,\x )$, 
and so belongs to $\T (R,\x )$. 
In general, however, 
$\exp fD$ does not necessarily belong to $\T (R,\x )$ 
for $f\in \ker D\sm R$. 
For example, 
assume that $n=3$, 
and define $D\in \Der _k\kx $ by 
$$
D(x_1)=-2x_2,\quad D(x_2)=x_3,\quad D(x_3)=0. 
$$
Then, 
$D$ is triangular if $x_1$ and $x_3$ are interchanged, 
and $f=x_1x_3+x_2^2$ belongs to $\ker D$. 
Moreover, 
$\exp fD$ is equal to Nagata's automorphism 
defined in (\ref{eq:Nagataautom}).

The following problem arises naturally.

\begin{problem}\label{prob:triangular}\rm
Assume that $D\in \lnd _R\Rx $ is triangular. 
When does $\exp fD$ belong to $\T (R,\x )$ 
for $f\in \ker D\sm R$? 
\end{problem}

As mentioned, 
it is possible that $\T (R,\x )$ 
is not equal to $\T (R,\y )$ 
for another set $\y $ of variables. 
Namely, 
$\phi ^{-1}\circ \T(R,\x )\circ \phi $ 
may not be equal to $\T (R,\x )$ 
for $\phi \in \Aut (\Rx /R)$. 
A member of a set of variables 
is called a {\it coordinate}. 
More precisely, 
we call $f\in \Rx $ a {\it coordinate} 
of $\Rx $ over $R$ if 
$R[f,f_2,\ldots ,f_n]=\Rx $ 
for some $f_2,\ldots ,f_n\in \Rx $, 
or equivalently 
$f=\phi (x_1)$ for some $\phi \in \Aut (\Rx /R)$. 
A coordinate $f$ of $\Rx $ over $R$ 
is said to be {\it tame} 
if $f=\phi (x_1)$ for some $\phi \in \T (R,\x )$, 
and {\it wild} otherwise. 
We mention that Jie-Tai Yu conjectured that 
the two coordinates of $\kx $ over $k$ 
given in (\ref{eq:Nagataautom}) 
are wild 
(the strong Nagata conjecture). 
The conjecture was solved in the affirmative 
by Umirbaev-Yu~\cite{UY}.

For polynomials, 
we introduce three notions of wildness. 
Here, 
when $R$ is just a domain containing $\Z $, 
we mean by $\exp D$ the automorphism 
$\exp \bar{D}$ of $\bar{R}[\x]:=\Q \otimes _{\Z }\Rx $ 
for each $D\in \lnd _R\Rx $, 
where $\bar{D}$ is the natural extension 
of $D$ to $\bar{R}[\x ]$. 
Since $\Rx $ is identified with a subring of $\bar{R}[\x ]$, 
it induces an element of $\Aut (\Rx /R)$ 
if $(\exp D)(\Rx )=\Rx $.

\begin{definition}
\label{defin:wild coordinates}\rm
Let $f$ be an element of $\Rx $. 

\noindent(i) 
We say that $f$ 
is {\it totally wild} if 
$\Aut (\Rx /R[f])\cap \T (R,\x )=\{ \id _{\Rx }\} $. 

\noindent(ii) 
We say that $f$ is {\it quasi-totally wild} 
if $\Aut (\Rx /R[f])\cap \T (R,\x )$ 
is a finite group. 

\noindent(iii) 
Assume that $R$ contains $\Z $. 
We say that $f$ is {\it exponentially wild} 
if there does not exist $D\in \lnd _R\Rx \sm \zs $ 
such that $D(f)=0$ and $\exp D$ induces 
an element of $\T (R,\x )$. 
\end{definition}

In order to avoid ambiguity, 
we sometimes call $f$ a totally 
(resp.\ quasi-totally, exponentially) wild 
element of $\Rx $ over $R$ if $f$ satisfies 
(i) (resp.\ (ii), (iii)) of Definition~\ref{defin:wild coordinates}. 
If a coordinate $f$ of $\Rx $ over $R$ is a totally 
(resp.\ quasi-totally, exponentially) wild element of 
$\Rx $ over $R$, 
then we call $f$ a totally 
(resp.\ quasi-totally, exponentially) wild coordinate of 
$\Rx $ over $R$.

By definition, 
totally wild polynomials are quasi-totally wild. 
We claim that quasi-totally wild coordinates 
are wild when $n\geq 2$. 
In fact, 
if $f$ is a tame coordinate, 
and $\phi \in \T (R,\x )$ 
is such that $\phi (x_1)=f$, 
then $\Aut (\Rx /R[f])\cap \T (R,\x )$ 
contains infinitely many 
automorphisms defined for each $l\geq 0$ 
by $\phi (x_i)\mapsto \phi (x_i)$ for $i\neq 2$, 
and $\phi (x_2)\mapsto \phi (x_2+x_1^l)$.

Assume that $R$ contains $\Z $, 
and $D$ is a nonzero element of $\lnd _R\Rx $. 
Then, 
we have $(\exp D)^l=\exp lD\neq \id _{\bar{R}[\x ]}$ 
for each $l\geq 1$. 
Hence, 
$\exp D$ has an infinite order. 
If $\exp D$ induces an element of $\Aut (\Rx /R)$, 
then it also has an infinite order. 
Thus, 
if $f\in \Rx $ is quasi-totally wild, 
then $\exp D$ does not 
induce an element of 
$\Aut (\Rx /R[f])\cap \T (R,\x )$. 
Consequently, 
there does not exist $D\in \lnd _R\Rx \sm \zs $ 
such that $D(f)=0$ and 
$\exp D$ induces an element of $\T (R,\x )$. 
Therefore, 
we know that quasi-totally wild polynomials are exponentially wild. 
When $n=1$, 
no coordinate is wild or exponentially wild. 
We claim that exponentially wild coordinates 
are wild when $n\geq 2$. 
In fact, 
if $f$ is a tame coordinate, 
and $\phi \in \T (R,\x )$ 
is such that $\phi (x_1)=f$, 
then the locally nilpotent derivation 
\begin{equation}\label{eq:exp pd}
D:=\phi \circ \left( 
\frac{\partial }{\partial x_2}
\right)\circ\phi ^{-1}
\end{equation}
kills $f$ and $\exp D$ induces an element of $\T (R,\x )$.

\begin{problem}\label{prob:awcoord}\rm
Find coordinates which are 
totally (quasi-totally, exponentially) wild. 
\end{problem}

The existence of such coordinates means 
the existence of a very large class of wild 
automorphisms. 
It is noteworthy that, 
if $f$ is a totally wild coordinate, 
and $f_2,\ldots ,f_n\in \Rx $ are such that 
$k[f,f_2,\ldots ,f_n]=\Rx $, 
then the automorphism defined by 
\begin{equation}\label{eq:scalar auto}
f\mapsto f,\ \ f_2\mapsto cf_2\text{ \ and \ }
f_i\mapsto f_i\text{ \ for \ }i=3,\ldots ,n
\end{equation}
does not belong to $\T (R,\x )$ 
for each $c\in R^{\times }\sm \{ 1\} $. 
Similarly, 
if $R$ contains $\Z $, 
and if $\phi \in \Aut (\Rx /R)$ is such that 
$\phi (x_1)$ is exponentially wild, 
then the exponential automorphism for 
the locally nilpotent derivation defined as in (\ref{eq:exp pd}) 
induces an element of $\Aut (\Rx /R)$ 
not belonging to $\T (R,\x )$.

We say that $D\in \Der _R\Rx $ is {\it triangularizable} 
if $\tau ^{-1}\circ D\circ \tau $ is triangular 
for some $\tau \in \Aut (\Rx /R)$. 
If this is the case, 
then $D$ is locally nilpotent. 
Since 
\begin{equation}\label{eq:exp normal}
\exp \tau ^{-1}\circ D\circ \tau 
=\tau ^{-1}\circ (\exp D)\circ \tau , 
\end{equation}
it follows that $\exp D$ belongs to $\T (R,\x )$ if 
$\tau $ belongs to $\T (R,\x )$. 
If $\tau $ does not belong to $\T (R,\x )$, 
however, 
$\exp D$ does not belong to $\T (R,\x )$ in general.

Due to the following theorem of 
Rentschler~\cite{Rentschler}, 
every element of $\lnd _k\kx $ is triangularizable 
when $n=2$.

\begin{tthm}[Rentschler]
\label{thm:Rentschler original}
Assume that $n=2$. 
For each $D\in \lnd _k\kx \sm \zs $, 
there exist $\tau \in \T (k,\x )$ 
and $f\in k[x_1]\sm \zs $ such that 
$$
\tau ^{-1}\circ D\circ \tau =f\frac{\partial }{\partial x_2}. 
$$ 
Hence, 
there exits a coordinate $g$ of $\kx $ over $k$ 
such that $\ker D=k[g]$. 
\end{tthm}

When $n\geq 3$, 
there exist elements of $\lnd _k\kx $ 
which are not triangularizable 
due to Bass~\cite{Bass} and Popov~\cite{Popov}. 
Combining this result and 
a result of Smith~\cite{Smith}, 
Freudenburg showed that, 
for each $n\geq 4$, 
there exists $D\in \lnd _k\kx $ 
such that $D$ is not triangularizable 
and $\exp D$ belongs to $\T (k,\x )$ 
(cf.\ Lemma 3.36 of \cite{Fbook} 
and the remark following it). 
So 
he asked the following question 
(cf.~\cite[Section~5.3.2]{Fbook}).

\begin{q}[Freudenburg]\label{q:Freudenburg}\rm
Assume that $n=3$. 
Let $D\in \lnd _k\kx $ be such that 
$\exp D$ belongs to $\T (k,\x )$. 
Is $D$ always triangularizable? 
\end{q}

We say that $D\in \Der _R\Rx $ is {\it affine} 
if $\deg D(x_i)\leq 1$ for $i=1,\ldots ,n$. 
If this is the case, 
then $\deg D^l(x_i)\leq 1$ holds for each $l\geq 0$. 
Hence, if $D\in \lnd _R\Rx $ is affine, 
then we have $\deg {(\exp D)(x_i)}=1$ for $i=1,\ldots ,n$. 
Thus, $\exp D$ belongs to $\Aff (R,\x )$, 
and so belongs to $\T (R,\x )$.

We say that $D\in \Der _R\Rx $ 
is {\it tamely triangularizable} 
if $\tau ^{-1}\circ D\circ \tau $ is triangular 
for some $\tau \in \T(R,\x )$, 
and is {\it tamely affinizable} 
if $\tau ^{-1}\circ D\circ \tau $ is affine 
for some $\tau \in \T(R,\x )$. 
If $D\in \lnd _R\Rx $ is 
tamely triangularizable or tamely affinizable, 
then it is obvious that 
$\exp D$ belongs to $\T (R,\x )$. 

We pose the following problem. 

\begin{problem}\label{q:strong}\rm
Find conditions under which 
$D$ is tamely triangularizable or tamely affinizable 
if and only if $\exp D$ belongs to $\T (R,\x )$. 
\end{problem}

There exists a relation between 
Question~\ref{p:fixed points} 
and Problem~\ref{q:strong} as follows. 
Let $\mathcal{D}$ be a subset of $\lnd _R\Rx $ such that 
$fD$ belongs to $\mathcal{D}$ for each 
$D\in \mathcal{D}$ and $f\in \ker D$. 
Assume that $D$ is tamely triangularizable 
or tamely affinizable for every $D\in \mathcal{D}$ 
such that $\exp D$ belongs to $\T (R,\x )$. 
Then, 
Question~\ref{p:fixed points} has an affirmative answer 
for each $D\in \mathcal{D}$. 
Actually, 
if $\exp fD$ belongs to $\T (R,\x )$ for 
$D\in \mathcal{D}$ and $f\in \ker D\sm \zs $, 
then there exists $\tau \in \T (R,\x )$ such that 
$$
D':=\tau ^{-1}(f)(\tau ^{-1}\circ D\circ \tau ) 
=\tau ^{-1}\circ (fD)\circ \tau 
$$
is triangular or affine by assumption, 
since $fD$ belongs to $\mathcal{D}$. 
If $D'$ is triangular, 
then $D_0:=\tau ^{-1}\circ D\circ \tau $ 
must be triangular, 
since $\tau ^{-1}(f)\neq 0$. 
Hence, 
$\exp D$ belongs to $\T (R,\x )$. 
If $D'$ is affine, 
then $D_0$ must be affine, 
since $\deg D'(x_i)=\deg \tau ^{-1}(f)+\deg D_0(x_i)$. 
Hence, 
$\exp D$ belongs to $\T (R,\x )$.

We can modify 
Question~\ref{q:Freudenburg} 
as follows.

\begin{q}\label{q:strong2}\rm
Assume that $n=3$. 
Let $D\in \lnd _k\kx $ be such that 
$\exp D$ belongs to $\T (k,\x )$. 
Is $D$ always tamely triangularizable? 
\end{q}

The {\it rank} $\rank D$ of $D\in \Der _k\kx $ 
is by definition the minimal number $r\geq 0$ 
for which there exists $\sigma \in \Aut (\kx /k)$ 
such that $D(\sigma (x_i))\neq 0$ for $i=1,\ldots ,r$ 
(cf.~\cite{Frank}). 
For example, 
if $n=2$, 
then we have $\rank D\leq 1$ for each $D\in \lnd _k\kx $ 
by Theorem~\ref{thm:Rentschler original}. 
We remark that, 
if $n\geq 2$ and $D$ is triangular, 
then $D$ always kills a tame coordinate of $\kx $ over $k$. 
In fact, 
$D$ kills $x_1$ if $D(x_1)=0$, 
and $D(x_1)x_2-f$ if $D(x_1)\neq 0$, 
where $f\in k[x_1]$ is such that 
$\partial f/\partial x_1=D(x_2)$. 
Hence, 
if $n\geq 2$ and $D$ is triangular, 
then we have $\rank D\leq n-1$. 
Consequently, 
if $n\geq 2$ and $D$ is triangularizable, 
then we have $\rank D\leq n-1$. 
Freudenburg~\cite{Frank} 
gave the first examples of $D\in \lnd _k\kx $ 
with $\rank D=n$ for each $n\geq 3$.

In connection with Question~\ref{q:Freudenburg}, 
we are interested in the following question.

\begin{q}\label{prob:rank3}\rm
Assume that $n\geq 3$. 
Does there exist $D\in \lnd _k\kx $ such that 
$\rank D=n$ and $\exp D$ belongs to $\T (k,\x )$? 
\end{q}

We mention that, 
when $n\geq 3$, 
tameness and wildness of $\exp D$ 
are not previously determined 
for any $D\in \lnd _k\kx $ with $\rank D=n$.

The study of polynomial automorphisms 
is closely related to the study of 
locally nilpotent derivations, 
so we are interested in locally nilpotent derivations. 
The following problem is an important problem 
in polynomial ring theory with little progress.

\begin{problem}
\label{p:classify}\rm
Describe all the elements of $\lnd _k\kx $. 
\end{problem}

When $n=1$, 
every element of $\lnd _k\kx $ has the form 
$c(\partial /\partial x_1)$ for some $c\in k$. 
When $n=2$, 
Theorem~\ref{thm:Rentschler original} gives 
a solution to the problem. 
Assume that $n=3$. 
If $\rank D\leq 2$, 
then $\sigma ^{-1}\circ D\circ \sigma $ 
belongs to $\lnd _{k[x_3]}\kx $ 
for some $\sigma \in \Aut (\kx /k)$. 
Since each element of $\lnd _{k[x_3]}\kx $ naturally 
extends to an element of $\lnd _{k(x_3)}k(x_3)[x_1,x_2]$, 
the problem is reduced to the case of $n\leq 2$. 
When $\rank D=3$, 
the problem remains open. 
Actually, 
there is no simple way to 
find $D\in \lnd _k\kx $ with $\rank D=3$. 
The problem is also open 
when $n\geq 4$ and $\rank D\geq 3$.

The outline of this monograph is as follows 
(see also~\cite{bessatsu} 
for a summary of an early 
version of this monograph). 
Part~\ref{part:domain} 
deals with automorphisms in two variables 
over a domain as mentioned. 
In Chapter~\ref{chap:taoad}, 
we give a useful criterion for deciding 
wildness of elements of $\Aut (\Rx /R)$. 
In the study of polynomial automorphisms, 
the notion of ``elementary reduction" 
is of great importance. 
As an analogue of this notion, 
we introduce the notion of ``affine reduction". 
Note that $\Aff (R,\x )$ 
is contained in $\E (R,\x )$ 
for some kinds of $R$. 
For example, if $R$ is a local ring or a Euclidean domain, 
then every element of $\GL (n,R)$ 
is a product of elementary matrices. 
Consequently, 
$\Aff (R,\x )$ is contained in $\E (R,\x )$. 
In general, 
however, 
$\Aff (R,\x )$ is not contained in $\E (R,\x )$, 
and so 
$\T (R,\x )$ is not equal to $\E (R,\x )$. 
To analyze elements of $\T (R,\x )\sm \E (R,\x )$, 
the notion of ``affine reduction" is essential. 
For a subgroup $G$ of $\Aff (R,\x )$, 
let $G^+$ be the subgroup of 
$\T (R,\x )$ generated by $G$ and $\E (R,\x )$. 
Then, 
we can naturally formulate a criterion for 
deciding whether a given element of $\Aut (\Rx /R)$ 
belongs to $G^+$ 
using the notions of 
``affine reduction" and ``elementary reduction" 
(Theorem~\ref{thm:AEreduction}). 
In the case where $G=\Aff (R,\x )$, 
this gives a tameness criterion for elements of $\Aut (\Rx /R)$.

Chapter~\ref{chap:trc} is devoted to developing the machinery 
to be used in Chapters~\ref{chap:t&t} and \ref{chap:atit}. 
For $f\in \Rx \sm R$, 
we say that $f$ is {\it tamely reduced} over $R$ 
if 
$$
\deg _{x_1}\tau (f)+\deg _{x_2}\tau (f)
\leq \deg _{x_1}f+\deg _{x_2}f
$$
holds for every $\tau \in \T (R,\x )$. 
Here, 
$\deg _{x_i}f$ denotes the degree of $f$ in $x_i$ 
for each $i$. 
We investigate the properties of such a polynomial 
using the tameness criterion given in 
Chapter~\ref{chap:taoad}. 
As a consequence, 
we determine the structure of 
the tame intersection 
$$
H(f):=\Aut (\Rx /R[f])\cap \T (R,\x )
$$
in a coarse sense 
(Theorems~\ref{prop:intersection theorem} and 
\ref{thm:reduced coordinate}). 
This result plays key roles in the proofs 
of various interesting theorems 
in the following two chapters.

In Chapter~\ref{chap:t&t}, 
we study tameness and 
wildness of exponential automorphisms. 
In Section~\ref{sect:triangularizability}, 
we answer Question~\ref{p:fixed points} 
in the affirmative when 
$n=2$ (Theorem~\ref{cor:fixed points}), 
and when $n=3$, $R=k$ 
and $D$ kills a tame coordinate of $\kx $ over $k$ 
(Theorem~\ref{thm:anstoF} (ii)). 
We solve Problem~\ref{q:strong} when $n=2$ 
(Theorem~\ref{thm:triangularizability}). 
We answer Question~\ref{q:strong2} 
in the affirmative when 
$D$ kills a tame coordinate 
of $\kx $ over $k$ (Theorem~\ref{thm:anstoF} (i)). 
In Section~\ref{sect:triangular}, 
we solve Problem~\ref{prob:triangular} 
in the cases where $n=2$ (Theorem~\ref{thm:hD}), 
and where $n=3$ and $R=k$ (Theorem~\ref{thm:triangular3}). 
As an application, 
we describe all the wild automorphisms of $\kx $ over $k$ 
of the form $\exp fD$ for some $D\in \lnd _k\kx $ 
and $f\in \ker D$ in which $D$ is triangular 
(Proposition~\ref{prop:triangular family}). 
In Section~\ref{sect:affine lnd}, 
we study a problem similar to Problem~\ref{prob:triangular} 
for affine locally nilpotent derivations 
instead of triangular derivations, 
and solve this problem for $n=2$ 
(Theorem~\ref{thm:affine lnd}).

Assume that $R$ contains $\Z $, 
and let $S$ be an over domain of $R$, 
i.e., a domain which contains $R$ as a subring. 
Consider a coordinate $f\in \Rx $ of $\Sx $ 
over $S$ with 
$H(f)\neq \{ \id _{\Rx }\} $ 
and 
$\deg _{x_1}f\geq \deg _{x_2}f$  
which is tamely reduced over $R$. 
The main result of Chapter~\ref{chap:atit} 
is a classification of such 
elements of $\Rx $. 
If $\deg _{x_2}f=0$, 
then $f$ belongs to $R[x_1]$. 
Since $f$ is a coordinate of $\Sx $ over $S$, 
it follows that $f$ is a linear polynomial in $x_1$ over $R$. 
In the case where $\deg _{x_2}f\geq 1$, 
we classify such polynomials into 
five types of polynomials 
(Definition~\ref{def:invariant coord}, 
Theorem~\ref{thm:Gamma (f)}). 
Moreover, 
we describe the structure of $H(f)$ 
for the five types of 
polynomials (Theorem~\ref{thm:H(f)}). 
As a consequence, 
we show that $f$ is exponentially wild 
if and only if $f$ is quasi-totally wild 
for a coordinate $f\in \Rx $ of $\Sx $ over $S$ 
(Corollary~\ref{cor:expw=>qtw} (i)). 
In Section~\ref{sect:tame intersect appl}, 
we apply the results to the coordinate 
$f_i:=(\exp hD)(x_i)$ of $\Rx $ over $R$ for $i=1,2$. 
Here, 
$R$ is a $\Q $-domain, 
and $D\in \lnd _R\Rx $ and $h\in \ker D\sm R$ 
are such that $D$ is triangular and $\exp hD$ is wild. 
Then, we completely determine wildness, 
quasi-totally wildness 
and totally wildness of $f_i$ 
(Theorem~\ref{thm:nagata qtw}), 
and thereby solving 
Problem~\ref{prob:awcoord} when $n=2$ and $R$ is a $\Q $-domain.

Part~\ref{chap:GSU} 
contains highly technical applications of 
the generalized Shestakov-Umirbaev theory. 
Throughout, 
we assume that $n=3$, 
and study the 
wildness of elements of $\Aut (\kx /k)$. 
In Chapter~\ref{sect:GSU}, 
we briefly review 
the 
generalized Shestakov-Umirbaev 
theory, 
and derive some consequences needed later. 
In Chapter~\ref{chapter:atcoord} 
we prove that some coordinates of $\kx $ over $k$ 
are quasi-totally wild 
or totally wild (Corollary~\ref{cor:aw}). 
Thus, 
we solve Problem~\ref{prob:awcoord} 
for $n=3$ and $R=k$. 
This is one of the most difficult result 
in this monograph.

In Chapter~\ref{chapter:rank3}, 
we study 
Question~\ref{prob:rank3}. 
First, we construct large families of elements of $\lnd _k\kx $ 
(Theorems~\ref{thm:lsc1} (i) and \ref{thm:lsc2} (i)) 
by generalizing the construction of Freudenburg~\cite{Flsc}. 
Then, 
we check that most of the members of the families 
have rank three 
(Theorems~\ref{prop:rank} and ~\ref{prop:rank2}) 
by using a technique based on ``plinth ideal" 
(Proposition~\ref{prop:rank3criterion}). 
Finally, 
we completely determine the tameness and wildness of 
the exponential automorphisms by means of 
``W-test polynomials" 
(Theorems~\ref{thm:lsc1} (iii) and \ref{thm:lsc2} (ii)). 
The result is that 
$\exp D$ is wild whenever $\rank D=3$. 
This gives a partial affirmative answer 
to Question~\ref{prob:rank3}. 
In the last section, 
toward the solution of Problem~\ref{p:classify} for $n=3$, 
we discuss how to get more examples of 
locally nilpotent derivations of rank three.

We conclude this monograph with 
problems, questions and conjectures.

\part{Automorphisms in two variables over a domain}
\label{part:domain}

\chapter{Tame automorphisms over a domain}
\label{chap:taoad}

\section{Graded structures}\setcounter{equation}{0}
\label{sect:grading}

Graded structures on $\Rx $ 
play important roles in the study of $\Aut (\Rx /R)$. 
Let $\Gamma $ be a 
{\it totally ordered additive group}, i.e., 
an additive group 
equipped with a total ordering such that 
$\alpha \leq \beta $ implies 
$\alpha +\gamma \leq \beta +\gamma $ 
for each $\alpha ,\beta ,\gamma \in \Gamma $. 
Then, $\Gamma $ is torsion-free. 
In this monograph, 
we assume that a totally ordered additive group 
is always finitely generated without mentioning it. 
Then, 
it follows that $\Gamma $ is free. 
Hence, 
we sometimes regard $\Gamma $ as a subgroup of 
the $\Q $-vector space $\Q \otimes _{\Z }\Gamma $.

Let $\w =(w_1,\ldots ,w_n)$ 
be an $n$-tuple of elements of $\Gamma $. 
We define the $\w $-{\it weighted grading} 
$$\Rx =\bigoplus _{\gamma \in \Gamma }\Rx _{\gamma }$$ 
by setting $\Rx _{\gamma }$ 
to be the $R$-submodule of $\Rx $ 
generated by the monomials 
$x_1^{a_1}\cdots x_n^{a_n}$ 
for $a_1,\ldots ,a_n\in \Zn $ 
with $\sum _{i=1}^na_iw_i=\gamma $ 
for each $\gamma \in \Gamma $. 
Here, 
$\Zn $ denotes the set of nonnegative integers. 
The set of positive integers will be denoted by $\N $. 
We say that $f\in \Rx \sm \zs $ 
is $\w $-{\it homogeneous} 
if $f$ belongs to $\Rx _{\gamma }$ for some $\gamma \in \Gamma $. 
Write $f\in \Rx \sm \zs $ as 
$f=\sum _{\gamma \in \Gamma }f_{\gamma }$, 
where $f_{\gamma }\in 
\Rx _{\gamma }$ for each $\gamma \in \Gamma $. 
Then, 
we define the $\w $-{\it degree} of $f$ by 
$$
\degw f=\max \{ \gamma \in \Gamma \mid f_{\gamma }\neq 0\} . 
$$
We define $f^{\w }=f_{\delta }$, 
where $\delta :=\deg _{\w }f$. 
When $f=0$, 
we define $f^{\w }=0$ and $\deg _{\w } f=-\infty $. 
Here, $-\infty $ is 
a symbol which is less than any element of $\Gamma $. 
If $\Gamma =\Z $ and $w_i=1$ for $i=1,\ldots ,n$, 
then the $\w $-degree of $f$ 
is the same as the total degree $\deg f$ of $f$. 
When $\Gamma =\Z ^n$, 
we denote $\deg _{x_i}f=\deg _{\e _i}f$ for $i=1,\ldots ,n$, 
where $\e _1,\ldots ,\e _n$ 
are the coordinate unit vectors of $\R ^n$. 
Let $f=g/h$ be an element of the field of 
fractions of $\Rx $, 
where $g,h\in \Rx $ with $g\neq 0$. 
Then, 
we define 
$$
\degw f=\degw g-\degw h 
\quad \text{and}\quad 
f^{\w }=\frac{g^{\w }}{h^{\w }}. 
$$
We note that this definition does not depend 
on the choice of $g$ and $h$.

Next, 
let $\Omega $ be the module of differentials of 
$\Rx $ over $R$, and $\omega $ an element of 
the $r$-th exterior power $\bigwedge ^r\Omega $ 
of the $\Rx $-module $\Omega $ 
for $r\in \N $. 
Then, 
we may uniquely express
\begin{equation}\label{eq:omega}
\omega =\sum _{1\leq i_1<\cdots <i_r\leq n}
f_{i_1,\ldots ,i_r}dx_{i_1}\wedge \cdots \wedge dx_{i_r},
\end{equation}
where $f_{i_1,\ldots ,i_r}\in \Rx $ for each $i_1,\ldots ,i_r$. 
Here, 
$df$ denotes the differential of $f$ for each $f\in \Rx $. 
We define the $\w $-{\it degree} of $\omega $ by 
\begin{equation*}
\deg _{\w }\omega =\max \{ \deg_{\w }
f_{i_1,\ldots ,i_r}x_{i_1}\cdots x_{i_r}\mid 
1\leq i_1<\cdots <i_r\leq n\} . 
\end{equation*}
Then, 
we have 
\begin{equation}\label{eq:deg df = deg f}
\deg _{\w }df
=\max \{ \deg _{\w }(\partial f/\partial x_i)x_i\mid i=1,\ldots ,n\} 
\leq \deg _{\w }f
\end{equation}
for each $f\in \kx $, 
since $df=\sum _{i=1}^n(\partial f/\partial x_i)dx_i$. 
Here, 
we note that the equality holds 
if $f$ does not belong to $R$, 
$R$ contains $\Z $, 
and $w_i>0$ for $i=1,\ldots ,n$. 
For $f\in \Rx $, $\omega \in \bigwedge ^r\Omega $ 
and $\eta \in \bigwedge ^s\Omega $, 
we have 
$$
\degw f\omega =\degw f+\degw \omega 
\quad \text{and}\quad 
\degw \omega \wedge \eta \leq \degw \omega +\degw \eta . 
$$ 
Take $f_1,\ldots ,f_r\in \Rx $ 
and set $\omega =df_1\wedge \cdots \wedge df_r$. 
Then, 
it follows that 
\begin{equation}\label{eq:inequomega}
\degw \omega 
\leq \sum _{i=1}^r\degw df_i
\leq \sum _{i=1}^r\degw f_i. 
\end{equation}
When $f_1,\ldots ,f_r$ are $\w $-homogeneous, 
we have $\degw \omega =\sum _{i=1}^r\degw f_i$ 
if and only if $\omega \neq 0$. 
Indeed, 
if $\omega $ is written as in (\ref{eq:omega}), 
then the $\w $-degree of each monomial appearing in 
$f_{i_1,\ldots ,i_r}x_{i_1}\cdots x_{i_r}$ 
is equal to $\sum _{i=1}^r\degw f_i$. 
Let $K$ be the field of fractions of $R$. 
Then, 
it is well-known that $\omega =0$ 
if $f_1,\ldots ,f_r$ 
are algebraically dependent over $K$ 
(cf.~\cite[Section 26]{Matsumura}). 
Thus, 
we have 
$\degw \omega <\sum _{i=1}^r\degw f_i$ 
if $f_1^{\w },\ldots ,f_r^{\w }$ 
are algebraically dependent over $K$.

For an endomorphism $\phi $ of 
the $R$-algebra $\Rx $, 
we define an $n\times n$ matrix by 
$$
J\phi =\left( 
\frac{\partial \phi (x_i)}{\partial x_j}
\right) _{i,j}. 
$$
Then, 
we have 
$d\phi (x_1)\wedge \cdots \wedge d\phi (x_n)
=(\det J\phi )dx_1\wedge \cdots \wedge dx_n$. 
If $\phi $ is an element of $\Aut (\Rx /R)$, 
then $\det J\phi $ belongs to $R^{\times }$. 
Hence, 
we obtain the inequality 
\begin{align}\begin{split}\label{eq:ineq-wedge}
&\deg _{\w }\phi 
:=\sum _{i=1}^n\degw \phi (x_i)
\geq 
d\phi (x_1)\wedge \cdots \wedge d\phi (x_n) \\
&\quad
=\degw (\det J\phi )dx_1\wedge \cdots \wedge dx_n
=\sum _{i=1}^nw_i=:|\w |. 
\end{split}\end{align}
If $\phi (x_1)^{\w},\ldots ,\phi (x_n)^{\w }$ 
are algebraically dependent over $K$, 
then we have $\degw \phi >|\w |$ 
by the discussion above.

For an $R$-submodule $A$ of $\Rx $, 
we denote by $A^{\w }$ the $R$-submodule of $\Rx $ 
generated by $\{ f^{\w }\mid f\in A\} $. 
If $A$ is an $R$-subalgebra of $\Rx $, 
then $A^{\w }$ forms an $R$-subalgebra of $\Rx $, 
since $(fg)^{\w }=f^{\w }g^{\w }$ 
holds for each $f,g\in \Rx $. 
Clearly, 
$R[f_1^{\w },\ldots ,f_r^{\w }]$ 
is contained in 
$R[f_1,\ldots ,f_r]^{\w }$ 
for $f_1,\ldots ,f_r\in \Rx $. 
We remark that, 
if $f_1^{\w },\ldots ,f_r^{\w }$ 
are algebraically independent over $K$, 
then we have 
$$
R[f_1^{\w },\ldots ,f_r^{\w }]
=R[f_1,\ldots ,f_r]^{\w }. 
$$
To see this, 
take any $h=\sum _{i_1,\ldots ,i_r}c_{i_1,\ldots ,i_r}
f_1^{i_1}\cdots f_r^{i_r}\in R[f_1,\ldots ,f_r]\sm \zs $, 
and define $\mu $ to be the maximum among 
$\degw f_1^{i_1}\cdots f_r^{i_r}$ 
for $i_1,\ldots ,i_r$ 
with $c_{i_1,\ldots ,i_r}\neq 0$, 
and $h'$ to be the sum of 
$$
(c_{i_1,\ldots ,i_r}
f_1^{i_1}\cdots f_r^{i_r})^{\w }
=c_{i_1,\ldots ,i_r}
(f_1^{\w })^{i_1}\cdots (f_r^{\w })^{i_r}
$$ 
for $i_1,\ldots ,i_r$ such that 
$\degw f_1^{i_1}\cdots f_r^{i_r}=\mu $. 
Then, 
$h'$ belongs to $\Rx _{\mu }$, 
and is nonzero 
by the assumption that 
$f_1^{\w },\ldots ,f_r^{\w }$ 
are algebraically independent over $K$. 
Hence, we know that $h^{\w }=h'$. 
Since $h'$ belongs to 
$R[f_1^{\w },\ldots ,f_r^{\w }]$, 
it follows that so does $h^{\w }$.

\begin{lem}\label{lem:minimal autom}
For any $\w \in \Gamma ^n$ and $\phi \in \Aut (\Rx /R)$, 
we have 
$$
\degw \phi \geq |\w |.  
$$
Furthermore, 
it holds that $\degw \phi =|\w |$ if and only if 
$\phi (x_1)^{\w },\ldots ,\phi (x_n)^{\w }$ 
are algebraically independent over $K$. 
\end{lem}
\begin{proof}
Thanks to (\ref{eq:ineq-wedge}) 
and the note 
following it, 
it suffices to check the ``if" part 
of the last statement. 
Assume that $\phi (x_1)^{\w },\ldots ,\phi (x_n)^{\w }$ 
are algebraically independent over $K$. 
Then, 
it holds that 
$$
R[\phi (x_1)^{\w },\ldots ,\phi (x_n)^{\w }]
=R[\phi (x_1),\ldots ,\phi (x_n)]^{\w }=\Rx ^{\w }
=\Rx 
$$
as remarked. 
Hence, we may define $\psi \in \Aut (\Rx /R)$ 
by $\psi (x_i)=\phi (x_i)^{\w }$ for $i=1,\ldots ,n$. 
Then, we have 
$$
\omega :=
d\phi (x_1)^{\w }\wedge \cdots \wedge d\phi (x_n)^{\w }
=(\det J\psi )dx_1\wedge \cdots \wedge dx_n. 
$$ 
Since $\det J\psi $ belongs to $R^{\times }$, 
we get $\degw \omega =|\w |$. 
On the other hand, 
we obtain 
$$
\degw \omega =\sum _{i=1}^n
\degw \phi (x_i)^{\w } 
$$ 
by the $\w $-homogeneity of $\phi (x_i)^{\w }$'s. 
Since $\degw \phi (x_i)^{\w }=\degw \phi (x_i)$ 
for each $i$, 
this is equal to $\degw \phi $. 
Therefore, 
we conclude that $\degw \phi =|\w |$. 
\end{proof}

For a permutation $x_{i_1},\ldots ,x_{i_n}$ 
of $x_1,\ldots ,x_n$, 
we define 
\begin{equation*}
J(R;x_{i_1},\ldots ,x_{i_n})
\end{equation*} 
to be the set of $\phi \in \Aut (\Rx /R)$ 
such that 
$\phi (R[x_{i_1},\ldots ,x_{i_l}])$ 
is contained in 
$R[x_{i_1},\ldots ,x_{i_l}]$ 
for $l=1,\ldots ,n$. 
Let $\phi $ be any element of 
$J(R;x_{i_1},\ldots ,x_{i_n})$. 
Then, 
$\phi $ induces an automorphism of 
$R[x_{i_1},\ldots ,x_{i_l}]$ for $l=1,\ldots ,n$. 
By induction on $n$, 
it is easy to check that 
\begin{equation}\label{eq:jonquiere}
\phi (x_{i_l})=a_lx_{i_l}+h_l
\end{equation}
for some $a_l\in R^{\times }$ 
and $h_l\in R[x_{i_1},\ldots ,x_{i_{l-1}}]$ 
for $l=1,\ldots ,n$. 
From this, 
we see that $J(R;x_{i_1},\ldots ,x_{i_n})$
is a subgroup of $\E (R,\x )$.

In the rest of this section, 
we assume that $n=2$, 
and investigate the $\w $-degrees of coordinates 
of $\Rx $ over $R$.

\begin{lem}\label{lem:deg coordinate}
Let $f$ be a coordinate of $\Rx $ over $R$, 
and $\w =(w_1,w_2)$ an element of $\Gamma ^2$. 
Then, the following assertions hold$:$

\noindent{\rm (i)} 
If $f$ belongs to $R[x_1]$, 
then $\degw f$ is equal to $w_1$ or zero, 
and is greater than or equal to $w_1$.

\noindent{\rm (ii)} If $f$ does not belong to $R[x_i]$ 
and $w_j>0$ for $i,j\in \{ 1,2\} $ with $i\neq j$, 
then we have $\degw f\geq w_j$. 
\end{lem}\begin{proof}
(i) Let $\phi \in \Aut (\Rx /R)$ be such that $\phi (x_1)=f$. 
Then, $\phi $ belongs to $J(R;x_1,x_2)$, 
since $f$ belongs to $R[x_1]$ by assumption. 
Hence, 
we have $f=ax_1+b$ 
for some $a\in R^{\times }$ and $b\in R$. 
Thus, 
we get $\degw f=\max \{ w_1,\degw b\} $. 
Since $\degw b$ is equal to zero or $-\infty $, 
it follows that 
$\degw f$ is equal to $w_1$ or zero, 
and is greater than or equal to $w_1$.

(ii) We show that the monomial $x_j^t$ 
appears in $f$ for some $t\geq 1$. 
Supposing the contrary, 
$f-c$ is divisible by $x_i$ for some $c\in R$. 
Since $f-c$ is also a coordinate of $\Rx $ over $R$, 
we have $f-c=ax_i$ for some $a\in R^{\times }$. 
It follows that 
$f=ax_i-c$ belongs to $R[x_i]$, a contradiction. 
Hence, $x_j^t$ appears in $f$ for some $t\geq 1$. 
Thus, we get $\degw f\geq tw_j$. 
Since $w_j>0$ by assumption, 
we have $tw_j\geq w_j$. 
Therefore,
we conclude that $\degw f\geq w_j$. 
\end{proof}

Let $\phi $ be an element of $\Aut (\Rx /R)$. 
Then, $\phi (x_1)$ or $\phi (x_2)$ 
does not belong to $R[x_i]$ for $i=1,2$. 
Hence, 
if $w_1>0$ or $w_2>0$, 
then we have 
\begin{equation}\label{eq:max phi}
\max \{ \degw \phi (x_1),\degw \phi (x_2)\} >0
\end{equation}
by Lemma~\ref{lem:deg coordinate} (ii).

\section{Affine reductions and elementary reductions}
\setcounter{equation}{0}
\label{sect:AE}

For a subgroup $G$ of $\Aff (R,\x )$, 
we define $G^+$ 
to be the subgroup of $\T (R,\x )$ 
generated by $G\cup \E(R,\x )$. 
By definition, 
we have $\Aff (R,\x )^+=\T (R,\x )$ 
and $\{ \id _{\Rx }\} ^+=\E (R,\x )$. 
In this section, 
we give a criterion for deciding whether 
a given element of $\Aut (\Rx /R)$ 
belongs to $G^+$ when $n=2$.

Let $\phi $ be an element of $\Aut (\Rx /R)$, 
and $\w $ an element of $\Gamma ^n$. 
We say that $\phi $ admits a $G$-{\it reduction} 
for the weight $\w $ 
if there exists $\alpha \in G$ such that 
$$
\degw \phi \circ \alpha <\degw \phi . 
$$
Here, 
the composition is defined by 
$(\sigma \circ \tau )(f)=\sigma (\tau (f))$ 
for each $\sigma ,\tau \in \Aut (\Rx /R)$ and $f\in \Rx $ 
as usual. 
We call a $G$-reduction an 
{\it affine reduction} if $G=\Aff (R,\x )$. 
We say that $\phi $ admits an {\it elementary reduction} 
for the weight $\w $ 
if there exists $\ep \in \Aut (\Rx /A_i)$ 
for some $i\in \{ 1,\ldots ,n\} $ such that 
$$
\degw \phi \circ \ep <\degw \phi , 
$$
where we define $A_i$ as in (\ref{eq:A_i}). 
This condition is equivalent to the condition that 
$\phi (x_i)^{\w }$ 
belongs to $R[\{ \phi (x_j)\mid j\neq i\}]^{\w }$ 
for some $i$. 
If there is no fear of confusion, 
we omit to mention the weight $\w $.

We define $\Affo(R,\x )$ to be the subgroup of 
$\Aff (R,\x )\cap \E (R,\x )$ 
generated by 
$$
\bigcup _{i=1}^n\Aut (\Rx /A_i)\cap \Aff(R,\x ), 
$$
and $\bar{G}$ to be the subgroup of $\Aff (R,\x )$ 
generated by $G\cup \Affo(R,\x )$. 
Clearly, 
$\bar{G}$ is contained in $G^+$. 
If $G$ contains $\Affo(R,\x )$, 
then $\bar{G}$ is equal to $G$. 
Hence, we have $\overline{\Aff (R,\x )}=\Aff (R,\x )$.

\begin{thm}\label{thm:AEreduction}
Assume that $n=2$. 
Let $G$ be a subgroup of $\Aff (R,\x )$, 
and $\w \in \Gamma ^2$ 
such that $w_1>0$ or $w_2>0$. 
If $\degw \phi >|\w |$ holds for $\phi \in G^+$, 
then $\phi $ admits a $\bar{G}$-reduction 
or elementary reduction for the weight $\w $. 
\end{thm}

We mention that a similar result is known for 
$G=\Aff (R,\x )$ in the case where $\Gamma =\Z $ and $\w =(1,1)$ 
(see for example \cite[Proposition 1]{Furter}). 
In the following, 
we prove Theorem~\ref{thm:AEreduction} 
using a technique similar to \cite{SU}.

Assume that $n=2$. 
We define $\iota \in \Aut (\Rx/R)$ by 
$$
\iota (x_1)=x_2\quad \text{and} \quad \iota (x_2)=x_1. 
$$
Then, 
we have 
$\iota =\iota _1\circ \iota _2\circ \iota _3$, 
where we define 
$\iota _1,\iota _2,\iota _3\in \Aut (\Rx /R)$ 
by 
\begin{align*}
\iota _1(x_1)&=x_1-x_2 &  \iota _2(x_1)&=x_1 & 
\iota _3(x_1)&=x_2-x_1 \\
\iota _1(x_2)&=x_2 & \iota _2(x_2)&=x_2+x_1 & 
\iota _3(x_2)&=x_2. 
\end{align*}
Hence, 
$\iota $ belongs to $\Affo (R ,\x )$. 
For $i,j\in \{ 1,2\} $ with $i\neq j$, 
we define 
$$
J_{i,j}=J(R;x_i,x_j)\cup J(R;x_i,x_j)\circ \iota . 
$$ 
Then, 
$J_{i,j}$ is equal to the set of 
$\phi \in \Aut (\Rx /R)$ 
such that $\phi (x_1)$ or $\phi (x_2)$ belongs to $R[x_i]$. 
Since 
\begin{equation}\label{eq:JJJ}
\iota \circ J(R;x_i,x_j)=J (R;x_j,x_i)\circ \iota , 
\end{equation}
we get 
\begin{align*}
\iota \circ J_{i,j}
&=
\iota \circ J(R;x_i,x_j)\cup \iota \circ J(R;x_i,x_j)\circ \iota \\
&=J(R;x_j,x_i)\circ \iota \cup J(R;x_j,x_i)
=J_{j,i}. 
\end{align*}

\begin{lem}\label{lem:JJ}
For $i,j\in \{ 1,2\} $ with $i\neq j$, 
we have the following$:$

\noindent{\rm (i)} 
$\psi ^{-1}$ and 
$\psi ^{-1}\circ \tau $ 
belong to $J:=J_{1,2}\cup J_{2,1}$ 
for each $\psi ,\tau \in J_{i,j}$.

\noindent{\rm (ii)} 
$\phi \circ \tau $ belongs to $J_{i,j}$ 
for each $\phi \in J (R;x_i,x_j)\circ \iota $ 
and $\tau \in J_{j,i}$. 
\end{lem}
\begin{proof}
(i) 
By assumption, 
$\psi $ belongs to $J(R;x_i,x_j)$ 
or $J(R;x_i,x_j)\circ \iota $. 
If $\psi $ belongs to $J(R;x_i,x_j)$, 
then $\psi ^{-1}$ belongs to $J(R;x_i,x_j)$, 
since $J(R;x_i,x_j)$ is a group. 
Since $\tau $ is an element of $J_{i,j}$ by assumption, 
it follows that 
$\psi ^{-1}\circ \tau $ belongs to $J_{i,j}$. 
If $\psi $ belongs to $J(R;x_i,x_j)\circ \iota $, 
then $\psi ^{-1}$ belongs to 
$\iota \circ J(R;x_i,x_j)$. 
By (\ref{eq:JJJ}), 
it follows that 
$\psi ^{-1}$ belongs to $J (R;x_j,x_i)\circ \iota $, 
and hence belongs to $J_{j,i}$. 
Since $\psi ^{-1}$ and $\tau $ 
belong to $\iota \circ J(R;x_i,x_j)$ and $J_{i,j}$, 
respectively, 
we know that $\psi ^{-1}\circ \tau $ 
belongs to $\iota \circ J_{i,j}=J_{j,i}$. 
Therefore, 
$\psi ^{-1}$ and 
$\psi ^{-1}\circ \tau $ 
belong to $J$ in either case.

(ii) 
Since $\phi \circ \iota $ belongs to $J(R;x_i,x_j)$, 
and $\iota \circ \tau $ belongs to $\iota \circ J_{j,i}=J_{i,j}$, 
we know that 
$\phi \circ \tau =(\phi \circ \iota )\circ (\iota \circ \tau )$ 
belongs to $J_{i,j}$. 
\end{proof}

By Lemma~\ref{lem:JJ} (i), 
we see that $J$ is closed under the inverse operation. 
Since $\bar{G}$ is a group, 
it follows that $\bar{G}\cup J$ 
is closed under the inverse operation. 
Note that $J$ is contained in $\E (R,\x )$, 
since $\iota $ and $J(R;x_i,x_j)$'s 
are contained in $\E (R,\x )$. 
Hence, 
$\bar{G}\cup J$ is contained in $G^+$. 
Since $J$ contains all the elementary automorphisms, 
and $\bar{G}$ contains $G$, 
we know that $G^+$ is generated by $\bar{G}\cup J$. 
Because $\bar{G}\cup J$ 
is closed under the inverse operation, 
this implies that $G^+$ is generated by $\bar{G}\cup J$ 
as a semigroup.

Let $\phi $ and $\w $ be any 
elements of $\Aut (\Rx /R)$ and $\Gamma ^2$, 
respectively. 
Then, we show that 
$\phi $ admits an elementary reduction for the weight $\w $ 
if and only if $\degw \phi \circ \sigma <\degw \phi $ 
for some $\sigma \in J$. 
Since every elementary automorphism of $\Rx $ 
belongs to $J$, 
the ``only if" part is clear. 
To prove the ``if" part, 
assume that $\degw \phi \circ \sigma <\degw \phi $ 
for some $\sigma \in J$. 
Take $i,j\in \{ 1,2\} $ with $i\neq j$ 
such that $\sigma $ belongs to $J_{i,j}$. 
Then, 
$\sigma (x_1)$ or $\sigma (x_2)$ belongs to $R[x_i]$. 
Hence, we may write 
\begin{equation}\label{eq:JJ-red}
\sigma (x_p)=\alpha x_i-g\quad 
\text{and}\quad \sigma (x_q)=\beta x_j-h, 
\end{equation}
where $p,q\in \{ 1,2\} $ with $p\neq q$, 
$\alpha ,\beta \in R^{\times }$, $g\in R$ 
and $h\in R[x_i]$. 
Then, 
we have 
\begin{equation}\label{eq:J-red}
\degw \phi \circ \sigma 
=\degw (\alpha \phi (x_i)-g)
+\degw (\beta \phi (x_j)-\phi (h)). 
\end{equation}
Since (\ref{eq:J-red}) is less than 
$\degw \phi =\degw \phi (x_i)+\degw \phi (x_j)$ 
by assumption, 
it follows that 
$$
\degw (\alpha \phi (x_i)-g)<\degw \phi (x_i)\quad 
\text{or}\quad 
\degw (\beta \phi (x_j)-\phi (h))<\degw \phi (x_j). 
$$
This implies that 
$\alpha \phi (x_i)^{\w }=g$ or 
$\beta \phi (x_j)^{\w }=\phi (h)^{\w }$. 
Since 
$\alpha ^{-1}g$ and $\beta ^{-1}\phi (h)^{\w }$ belong to 
$R[\phi (x_j)]^{\w }$ and $R[\phi (x_i)]^{\w }$, 
respectively, 
we know that 
$\phi (x_i)^{\w }$ 
belongs to $R[\phi (x_j)]^{\w }$, 
or $\phi (x_j)^{\w }$ belongs to $R[\phi (x_i)]^{\w }$. 
Therefore, 
$\phi $ admits an elementary reduction for the weight $\w $.

Now, 
we prove Theorem~\ref{thm:AEreduction}. 
We define $G_{\w }$ 
to be the set of $\phi \in \Aut (\Rx /R)$ 
for which there exist $l\in \N $, 
$\phi _1,\ldots ,\phi _{l-1}\in \Aut (\Rx /R)$ 
and $\phi _l\in \bar{G}\cup J$ 
as follows: 
\begin{align}
& \phi _1=\phi \text{ and } \degw \phi _l=|\w |; 
\tag{A}\label{eq:Gw1} \\
& \degw \phi _{i+1}<\degw \phi _i\text{ and } 
\phi _{i+1}=\phi _i\circ \tau _i
\text{ for some }\tau _i\in \bar{G}\cup J 
\tag{B} \label{eq:Gw2} \text{ for } 1\leq i<l. 
\end{align}

Assume that $\phi \in G_{\w }$ 
satisfies $\degw \phi >|\w |$. 
Then, 
we have $l\geq 2$ in view of (\ref{eq:Gw1}). 
Hence, 
we see from (\ref{eq:Gw2}) that 
$\degw \phi \circ \tau <\degw \phi $ 
for some $\tau \in \bar{G}\cup J$. 
Thus, 
$\phi $ admits a $\bar{G}$-reduction 
or elementary reduction. 
Therefore, 
it suffices to verify that $G^+$ 
is contained in $G_{\w }$.

The following is a key proposition.

\begin{prop}\label{prop:A}
$\phi \circ \tau $ belongs to $\A $ 
for each $\phi \in \A $ and $\tau \in \bar{G}\cup J$. 
\end{prop}

Clearly, 
$\id _{\Rx }$ belongs to $\bar{G}\cup J$, 
and satisfies $\degw \id _{\Rx }=|\w |$. 
Hence, 
$\id _{\Rx }$ belongs to $G_{\w }$. 
Since $G^+$ is generated by $\bar{G}\cup J$ 
as a semigroup, 
we know by Proposition~\ref{prop:A} 
that $G^+$ is contained in $G_{\w }$. 
Therefore, 
Theorem~\ref{thm:AEreduction} 
follows from Proposition~\ref{prop:A}.

We remark that, 
if there exists $\tau \in \bar{G}\cup J$ such that 
$\degw \psi \circ \tau <\degw \psi $ 
and $\psi \circ \tau $ belongs to $G_{\w }$ 
for $\psi \in \Aut (\Rx /R)$, 
then $\psi $ belongs to $G_{\w }$ by the definition of $G_{\w }$. 
From this, 
we see that Proposition~\ref{prop:A} holds when 
$\degw \phi \circ \tau >\degw \phi $. 
In fact, 
$(\phi \circ \tau )\circ \tau ^{-1}=\phi $ 
belongs to $G_{\w }$ by assumption, 
and $\tau ^{-1}$ belongs to $\bar{G}\cup J$ 
because $\bar{G}\cup J$ is closed under the inverse operation.

\begin{lem}\label{claim:AE1}
Let $\phi \in \Aut (\Rx /R)$ and $\sigma \in \bar{G}\cup J$ 
be such that $\degw \phi \circ \sigma \leq \degw \phi $. 
Then, the following statements hold$:$

\noindent{\rm (i)} 
If $\degw \phi (x_1)=\degw \phi (x_2)$, 
then $\sigma $ belongs to $\bar{G}$.

\noindent{\rm (ii)} 
If $\degw \phi (x_i)<\degw \phi (x_j)$ 
for $i,j\in \{ 1,2\} $ with $i\neq j$, 
then $\sigma $ belongs to $J_{i,j}$. 
\end{lem}
\begin{proof}
(i) 
In view of (\ref{eq:jonquiere}), 
we see that $J(R;x_i,x_j)\cap \Aff (R,\x )$ 
is contained in $\Affo (R ,\x )$ 
for each $i,j\in \{ 1,2\} $ with $i\neq j$. 
Since $\iota $ is affine, 
we have 
$$
J(R;x_i,x_j)\circ \iota \cap \Aff (R,\x )
=\bigl(J(R;x_i,x_j)\cap \Aff (R,\x )\bigr)\circ \iota . 
$$
Since $\iota $ belongs to $\Affo (R ,\x )$, 
the right-hand side of this equality 
is also contained in $\Affo (R,\x )$. 
Hence, $J_{i,j}\cap \Aff (R,\x )$ 
is contained in $\Affo (R,\x )$. 
Thus, $J\cap \Aff (R,\x )$ 
is contained in $\Affo (R,\x )$, 
and therefore contained in $\bar{G}$.

Now, 
suppose to the contrary that 
$\sigma $ does not belong to $\bar{G}$. 
Then, 
$\sigma $ belongs to $J$. 
Since $J\cap \Aff (R,\x )$ 
is contained in $\bar{G}$, 
it follows that 
$\sigma $ does not belong to $\Aff (R,\x )$. 
Write $\sigma (x_1)$ and $\sigma (x_2)$ 
as in (\ref{eq:JJ-red}), 
where $i,j\in \{ 1,2\} $ with $i\neq j$. 
Then, 
we have $\deg _{x_i}h\geq 2$. 
Since $w_1>0$ or $w_2>0$, 
and $\delta :=\degw \phi (x_1)=\degw \phi (x_2)$ 
by assumption, 
we get $\delta >0$ by (\ref{eq:max phi}). 
Thus, 
we see from (\ref{eq:J-red}) that 
$$
\degw \phi \circ \sigma 
=\degw \phi (x_i)+\degw \phi (h)
\geq 3\delta >2\delta =\degw \phi , 
$$ 
a contradiction. 
Therefore, 
$\sigma $ belongs to $\bar{G}$.

(ii) 
Since $\degw \phi (x_i)<\degw \phi (x_j)=:\delta $ 
by assumption, 
we have $\degw \phi <2\delta $, 
and $\delta >0$ by (\ref{eq:max phi}). 
Suppose to the contrary that 
$\sigma $ does not belong to $J_{i,j}$. 
Then, 
$\sigma (x_1)$ and $\sigma (x_2)$ do not belong to $R[x_i]$, 
and $\sigma $ belongs to $\bar{G}$ or $J_{j,i}$. 
First, 
assume that $\sigma $ belongs to $\bar{G}$. 
Then, 
$\sigma $ is affine. 
Hence, 
$(\phi \circ \sigma )(x_l)$ 
is a linear polynomial in 
$\phi (x_i)$ and $\phi (x_j)$ over $R$ 
for $l=1,2$. 
Moreover, 
$(\phi \circ \sigma )(x_l)$ 
does not belong to $R[\phi (x_i)]$, 
since $\sigma (x_l)$ does not belong to $R[x_i]$. 
Thus, 
we know that 
$\degw (\phi \circ \sigma )(x_l)=\degw \phi (x_j)=\delta $. 
Therefore, 
we get $\degw \phi \circ \sigma =2\delta >\degw \phi $, 
a contradiction. 
Next, 
assume that $\sigma $ belongs to $J_{j,i}$. 
Then, we may write 
$\sigma (x_p)=\alpha x_j-g$ 
and $\sigma (x_q)=\beta x_i-h$, 
where $p,q\in \{ 1,2\} $ with $p\neq q$, 
$\alpha ,\beta \in R^{\times }$, $g\in R$ 
and $h\in R[x_j]$. 
Since $\sigma (x_q)$ does not belong to $R[x_i]$, 
we have $\deg _{x_j}h\geq 1$. 
Hence, we know that 
$\degw \phi \circ \sigma 
=\degw \phi (x_j)+\degw \phi (h)
\geq 2\delta >\degw \phi $, 
a contradiction. 
Therefore, 
$\sigma $ belongs to $J_{i,j}$. 
\end{proof}

Let $W$ be the set of $\w \in \Gamma ^2$ 
such that $w_1>0$ or $w_2>0$, 
and 
$$
\Sigma _{\w }:=
\{ \degw \psi \mid \psi \in \Aut (\Rx /R)\} 
$$ 
is a well-ordered subset of $\Gamma $. 
If $w_i\geq 0$ for $i=1,2$, 
then 
$\{ l_1w_1+l_2w_2\mid l_1,l_2\in \Zn \} $ 
is a well-ordered subset of $\Gamma $ 
(cf.~\cite[Lemma 6.1]{SU2}). 
Hence, $\Sigma _{\w }$ is also well-ordered. 
Thus, 
$W$ contains the set of 
$\w \in \Gamma ^2$ 
such that $w_1>0$ and $w_2\geq 0$, 
or $w_1\geq 0$ and $w_2>0$. 
Consider the following statements: 

\smallskip 

\noindent{\rm (I)} 
Proposition~\ref{prop:A} 
holds for each $\w \in W$. 

\noindent{\rm (II)} 
If $w_1>0$ or $w_2>0$ for $\w \in \Gamma ^2$, 
then $\Sigma _{\w }$ 
is a well-ordered subset of $\Gamma $. 

\smallskip 

\noindent
By the discussion after Proposition~\ref{prop:A}, 
(I) implies that Theorem~\ref{thm:AEreduction} 
holds for each $\w \in W $. 
On the other hand, 
(II) implies that $\w \in \Gamma ^2$ 
belongs to $W$ if $w_1>0$ or $w_2>0$. 
Thus, 
Theorem~\ref{thm:AEreduction} 
follows from (I) and (II). 
We prove (I) in the rest of this section. 
By making use of it, 
we prove (II) at the end of 
Section~\ref{sect:coordinate}. 
To prove (I), 
we may assume that 
$\degw \phi \circ \tau \leq \degw \phi $ 
by the remark after Proposition~\ref{prop:A}. 
To prove (I) and (II), 
we may assume that $w_1\leq w_2$ and $w_2>0$ 
by interchanging $x_1$ and $x_2$ if necessary.

Let us prove (I). 
Take any $\w \in W$. 
Then, 
$\Sigma _{\w }$ is a well-ordered subset of $\Gamma $. 
Since $\degw \phi $ belongs to $\Sigma _{\w }$ 
for each $\phi \in G_{\w }$, 
we prove the statement of Proposition~\ref{prop:A} 
by induction on $\degw \phi $. 
By Lemma~\ref{lem:minimal autom}, 
$\mu :=\min \{ \degw \phi \mid \phi \in G_{\w }\} $ 
is at least $|\w |$. 
Since $\id _{\Rx }$ belongs to $G_{\w }$, 
we get $\mu =|\w |$. 
So assume that $\degw \phi =|\w |$. 
Then, 
we have $\degw \phi \circ \tau \leq |\w |$ 
by the assumption that 
$\degw \phi \circ \tau \leq \degw \phi $. 
This implies that 
$\degw \phi \circ \tau =|\w |$ 
because of Lemma~\ref{lem:minimal autom}. 
Note that 
$\sigma $ belongs to $G_{\w }$ 
if and only if $\sigma $ belongs to $\bar{G}\cup J$ 
for $\sigma \in \Aut (\Rx /R)$ with $\degw \sigma =|\w |$. 
Hence, it suffices to show that 
$\phi \circ \tau $ belongs to $\bar{G}\cup J$. 
Since $\phi $ is an element of $G_{\w }$ 
with $\degw \phi =|\w |$, 
we know that $\phi $ belongs to $\bar{G}\cup J$. 
More precisely, 
$\phi $ satisfies the following conditions.

\begin{claim}\label{claim:AE2}
The following statements hold$:$ 

\noindent{\rm (1)} 
If $w_1=w_2$, 
then $\degw \phi (x_1)=\degw \phi (x_2)$ 
and $\phi $ belongs to $\bar{G}$. 

\noindent{\rm (2)} 
If $w_1<w_2$, 
then one of the following conditions holds$:$\\
{\rm (a)} 
$\degw \phi (x_1)<\degw \phi (x_2)$ 
and $\phi $ belongs to $J(R;x_1,x_2)$. \\
{\rm (b)} 
$\degw \phi (x_2)<\degw \phi (x_1)$ 
and $\phi $ belongs to $J(R;x_1,x_2)\circ \iota $. 
\end{claim}
\begin{proof}
Take $(i,j)\in \{ (1,2),(2,1)\} $ such that $\phi (x_j)$ 
does not belong to $R[x_1]$. 
Then, 
we have $\degw \phi (x_j)\geq w_2$ 
by Lemma~\ref{lem:deg coordinate} (ii), 
since $w_2>0$ by assumption. 
Whether or not $\phi (x_i)$ belongs to $R[x_1]$, 
we have $\degw \phi (x_i)\geq w_1$ 
by (i) and (ii) of Lemma~\ref{lem:deg coordinate}, 
since $w_1\leq w_2$ by assumption. 
Thus, 
we conclude that 
$\degw \phi (x_i)=w_1$ and $\degw \phi (x_j)=w_2$ 
from the assumption that $\degw \phi =|\w |$.

(1) Since $w_1=w_2$ by assumption, 
the first part follows from the preceding discussion. 
Since $\phi $ belongs to $\bar{G}\cup J$, 
we show that 
$\phi $ belongs to $\bar{G}$ when $\phi $ belongs to $J$. 
As remarked in the proof of Lemma~\ref{claim:AE1}, 
$J\cap \Aff (R,\x )$ is contained in $\bar{G}$. 
So we show that $\phi $ belongs to $\Aff (R,\x )$. 
By the definition of the $\w $-degree, 
we have $\degw h=w_1\deg h$ for each $h\in \Rx $, 
since $w_1=w_2$ and $w_2>0$. 
Hence, we get 
$w_1\deg \phi (x_l)=\degw \phi (x_l)$ for $l=1,2$. 
Since $\degw \phi (x_l)=w_1$, 
it follows that $\deg \phi (x_l)=1$ for $l=1,2$. 
Thus, 
$\phi $ belongs to $\Aff (R,\x )$. 
Therefore, 
$\phi $ belongs to $\bar{G}$.

(2) 
We claim that $\phi (x_i)$ belongs to $R[x_1]$. 
In fact, 
if not, 
$w_1=\degw \phi (x_i)$ is not less than 
$w_2$ by Lemma~\ref{lem:deg coordinate} (ii), 
since $w_2>0$ 
by assumption. 
This contradicts that $w_1<w_2$. 
If $(i,j)=(1,2)$, then 
it follows that $\phi $ belongs to $J(R;x_1,x_2)$. 
Since $w_1<w_2$, 
we know that 
$\degw \phi (x_1)=\degw \phi (x_i)=w_1$ 
is less than $\degw \phi (x_2)=\degw \phi (x_j)=w_2$. 
Therefore, $\phi $ satisfies (a). 
If $(i,j)=(2,1)$, 
then $\phi (x_2)=\phi (x_i)$ belongs to $R[x_1]$. 
Hence, 
$\phi $ belongs to $J(R;x_1,x_2)\circ \iota $. 
Since $w_1<w_2$, 
we know that 
$\degw \phi (x_2)=\degw \phi (x_i)=w_1$ 
is less than $\degw \phi (x_1)=\degw \phi (x_j)=w_2$. 
Therefore, $\phi $ satisfies (b). 
\end{proof}

Now, 
we prove that $\phi \circ \tau $ belongs to $\bar{G}\cup J$. 
Since $\degw \phi \circ \tau \leq \degw \phi $, 
the statements (i) and (ii) of Lemma~\ref{claim:AE1} 
hold for $\sigma =\tau $. 
If $w_1=w_2$, 
then we know by Claim~\ref{claim:AE2} (1) 
and Lemma~\ref{claim:AE1} (i) that 
$\phi $ and $\tau $ belong to $\bar{G}$. 
Hence, 
$\phi \circ \tau $ belongs to $\bar{G}$. 
If $w_1<w_2$, 
then we have (a) or (b) of Claim~\ref{claim:AE2} (2). 
In the case of (a), 
$\tau $ belongs to $J_{1,2}$ 
by Lemma~\ref{claim:AE1} (ii). 
Since $\phi $ belongs to $J(R;x_1,x_2)$, 
it follows that 
$\phi \circ \tau $ belongs to $J_{1,2}$. 
In the case of (b), 
$\tau $ belongs to $J_{2,1}=\iota \circ J_{1,2}$ 
by Lemma~\ref{claim:AE1} (ii). 
Since $\phi $ belongs to $J(R;x_1,x_2)\circ \iota $, 
it follows that 
$\phi \circ \tau $ belongs to $J_{1,2}$. 
Hence, 
$\phi \circ \tau $ belongs to $\bar{G}\cup J$ 
in every case. 
Thus, $\phi \circ \tau $ belongs to $G_{\w }$. 
Therefore, 
the statement of Proposition~\ref{prop:A} 
holds when $\degw \phi =|\w |$.

Next, assume that $\degw \phi >|\w |$. 
Since $\phi $ is an element of $\A $, 
we may find $\psi \in \bar{G}\cup J$ such that 
$\degw \phi \circ \psi <\degw \phi $ and 
$\phi \circ \psi $ belongs to $G_{\w }$ 
in view of (\ref{eq:Gw2}). 
Then, 
$(\phi \circ \psi )\circ \sigma $ belongs to $G_{\w }$ 
for any $\sigma \in \bar{G}\cup J$ 
by induction assumption. 
We show that 
$\psi ^{-1}\circ \tau $ belongs to $\bar{G}\cup J$. 
Then, 
it follows that 
$\phi \circ \tau 
=(\phi \circ \psi )\circ (\psi ^{-1}\circ \tau )$ 
belongs to $G_{\w }$. 
If $\degw \phi (x_1)=\degw \phi (x_2)$, 
then $\psi $ belongs to $\bar{G}$ 
by Lemma~\ref{claim:AE1} (i) 
since $\degw \phi \circ \psi <\degw \phi $ 
by the choice of $\psi $. 
Since $\degw \phi \circ \tau \leq \degw \phi $ 
by assumption, 
$\tau $ also belongs to $\bar{G}$ 
by Lemma~\ref{claim:AE1} (i). 
Hence, 
$\psi ^{-1}\circ \tau $ belongs to $\bar{G}$. 
Likewise, 
if $\degw \phi (x_i)<\degw \phi (x_j)$ 
for some $i,j\in \{ 1,2\} $ with $i\neq j$, 
then $\tau $ and $\psi $ belong to $J_{i,j}$ 
by Lemma~\ref{claim:AE1} (ii). 
Hence, 
$\psi ^{-1}\circ \tau $ belongs to $J$ 
by Lemma~\ref{lem:JJ} (i). 
Therefore, 
$\psi ^{-1}\circ \tau $ belongs to $\bar{G}\cup J$ 
in either case. 
This completes the proof of (I), 
and thereby proving Theorem~\ref{thm:AEreduction} 
for each $\w \in W$. 
In particular, 
Theorem~\ref{thm:AEreduction} is verified 
for each $\w \in \Gamma ^2$ such that $w_1>0$ and $w_2\geq 0$, 
or $w_1\geq 0$ and $w_2>0$.

\section{Analysis of reductions}
\setcounter{equation}{0}

Throughout this section, 
we assume that $n=2$. 
We investigate properties of affine reductions 
and elementary reductions.

To deal with affine reductions, 
the following notion is crucial. 
Let $K$ be the field of fractions of $R$. 
We define $V(R)$ to be the set of 
$\alpha /\beta \in K^{\times }$ 
for $\alpha ,\beta \in R\sm \zs $ 
such that $\alpha R+\beta R=R$. 
Note that $\alpha R+\beta R=R$ if and only if 
$\alpha $ and $\beta $ are the entries of 
a row or column 
vector of an element of $\GL (2,R)$ 
for $\alpha ,\beta \in R$.

\begin{lem}\label{lem:V(R)}
\noindent{\rm (i)} 
$R\sm \zs $ is contained in $V(R)$.

\noindent{\rm (ii)} 
For $\alpha ,\beta \in R\sm \zs $, 
it holds that $\alpha /\beta $ 
belongs to $V(R)$ if and only if 
$\alpha R+\beta R$ is a principal ideal of $R$.

\noindent{\rm (iii)} 
For $\gamma \in V(R)$ and $(a_{i,j})_{i,j}\in \GL (2,R)$, 
it holds that 
$$
\frac{a_{1,1}\gamma +a_{1,2}}{a_{2,1}\gamma +a_{2,2}}
$$
belongs to $V(R)$. 
\end{lem}
\begin{proof}
(i) For each $a\in R\sm \zs $, 
we have $aR+1R=R$. 
Hence, $a=a/1$ belongs to $V(R)$. 

(ii) 
Assume that $\alpha /\beta $ belongs to $V(R)$ 
for $\alpha ,\beta \in R\sm \zs $. 
Then, 
there exist $\alpha ',\beta '\in R\sm \zs $ 
such that $\alpha '/\beta '=\alpha /\beta $ 
and $\alpha 'R+\beta 'R=R$. 
Set $\gamma =\alpha /\alpha '=\beta /\beta '$ 
and take $a,b\in R$ such that $\alpha 'a+\beta 'b=1$. 
Then, 
we have $\alpha a+\beta b=\gamma $. 
Hence, 
$\gamma $ belongs to $R$, 
and $\gamma R$ is contained in $\alpha R+\beta R$. 
Since $\alpha =\gamma \alpha '$ 
and $\beta =\gamma \beta '$ belong to $\gamma R$, 
we know that $\alpha R+\beta R$ 
is contained in $\gamma R$. 
Thus, 
we get $\alpha R+\beta R=\gamma R$. 
Therefore, 
$\alpha R+\beta R$ is a principal ideal of $R$. 
Next, 
assume that $\alpha R+\beta R=\gamma R$ 
for some $\gamma \in R$. 
Then, 
$\alpha $ and $\beta $ belong to $\gamma R$. 
Hence, 
we have $\alpha =\alpha '\gamma $ 
and $\beta =\beta '\gamma $ 
for some $\alpha ',\beta '\in R$. 
Then, we get $\alpha 'R+\beta 'R=R$. 
Therefore, 
$\alpha /\beta =\alpha '/\beta '$ belongs to $V(R)$.

(iii) 
Let $\gamma _1,\gamma _2\in R\sm \zs $ be such that 
$\gamma =\gamma _1/\gamma _2$ and $\gamma _1R+\gamma _2R=R$. 
Then, 
we have 
$$
\delta :=
\frac{a_{1,1}\gamma +a_{1,2}}{a_{2,1}\gamma +a_{2,2}}
=\frac{a_{1,1}\gamma _1+a_{1,2}\gamma _2}{
a_{2,1}\gamma _1+a_{2,2}\gamma _2}
$$
and 
$$
(a_{1,1}\gamma _1+a_{1,2}\gamma _2)R
+(a_{2,1}\gamma _1+a_{2,2}\gamma _2)R
=\gamma _1R+\gamma _2R=R. 
$$
Hence, $\delta $ belongs to $V(R)$. 
\end{proof}

By Lemma~\ref{lem:V(R)} (ii), 
it follows that 
$V(R)=K^{\times }$ if $R$ is a PID\null. 
If $r$ belongs to $V(R)$, then $r^{-1}$ 
and $ur$ belong to $V(R)$ for each $u\in R^{\times}$ 
by Lemma~\ref{lem:V(R)} (iii).

The following lemma naturally follows from 
the definition of affine reduction 
and elementary reduction.

\begin{lem}\label{lem:AE initial}
Let $\phi ,\tau \in \Aut (\Rx /R)$ and $\w \in \Gamma ^2$ 
be such that $\degw \phi (x_i)>0$ for $i=1,2$ 
and $\degw \phi \circ \tau <\degw \phi $. 

\noindent{\rm (i)} 
If $\tau $ belongs to $\Aff (R,\x )$, 
then we have $\phi (x_1)^{\w }=a\phi (x_2)^{\w }$ 
for some $a\in V(R)$, 
and $(\phi \circ \tau )(x_i)^{\w }=b\phi (x_1)^{\w }$ 
for some $i\in \{ 1,2\} $ and $b\in K^{\times }$. 

\noindent{\rm (ii)} 
Assume that $\tau $ belongs to $J_{i,j}$ 
for some $i,j\in \{ 1,2\} $ with $i\neq j$. 
Then, 
we have $\phi (x_j)^{\w }=b(\phi (x_i)^{\w })^t$ 
for some $b\in R\sm \zs $ and $t\geq 1$, 
and $(\phi \circ \tau )(x_p)^{\w }=\alpha \phi (x_i)^{\w }$ 
for some $p\in \{ 1,2\} $ and $\alpha \in R^{\times }$. 
\end{lem}
\begin{proof}
(i) Take $A=(a_{i,j})_{i,j}\in \GL (2,R)$ 
and $b_1,b_2\in R$ such that 
$$
\tau (x_i)=a_{i,1}x_1+a_{i,2}x_2+b_i
$$ 
for $i=1,2$, 
and put 
$$
\delta _i=\max \{ \degw a_{i,1}\phi (x_1),\degw a_{i,2}\phi (x_2)\} 
$$ 
for $i=1,2$. 
Then, 
we have $\delta _i>0$ for $i=1,2$, 
since no row of $A$ is zero, 
and $\degw \phi (x_j)>0$ for $j=1,2$ 
by assumption. 
Since no column of $A$ is zero, 
we know that 
$\delta _1+\delta _2\geq \degw \phi $. 
We claim that 
$a_{j,1}\phi (x_1)^{\w }+a_{j,2}\phi (x_2)^{\w }=0$ 
for some $j\in \{ 1,2\} $. 
In fact, 
if not, we have 
$$
\degw \bigl((\phi \circ \tau )(x_i)-b_i\bigr) 
=\degw (a_{i,1}\phi (x_1)+a_{i,2}\phi (x_2))
=\delta _i
$$
for $i=1,2$. 
Since $b_i$ 
is an element of $R$, 
and $\delta _i$ is positive, 
this implies that 
$\degw (\phi \circ \tau )(x_i)=\delta _i$ 
for $i=1,2$. 
Hence, 
we get 
$\degw \phi \circ \tau =\delta _1+\delta _2\geq \degw \phi $, 
a contradiction. 
Thus, 
we may find $j\in \{ 1,2\} $ as claimed. 
Then, 
$a_{j,1}$ and $a_{j,2}$ are nonzero, 
since $(a_{j,1},a_{j,2})\neq (0,0)$, 
and $\phi (x_l)^{\w }\neq 0$ for $l=1,2$. 
Put $a=-a_{j,2}/a_{j,1}$. 
Then, 
$a$ belongs to $V(R)$, 
and satisfies $\phi (x_1)^{\w }=a\phi (x_2)^{\w }$. 
Take $i\in \{ 1,2\} $ with $i\neq j$. 
Then, 
$b:=a_{i,1}+a_{i,2}a^{-1}$ belongs to $K^{\times }$, 
since $\det A\neq 0$. 
Hence, 
we get 
$$
a_{i,1}\phi (x_1)^{\w }+a_{i,2}\phi (x_2)^{\w }
=a_{i,1}\phi (x_1)^{\w }+a_{i,2}(a^{-1}\phi (x_1)^{\w })
=b\phi (x_1)^{\w }\neq 0. 
$$ 
Since $\degw \phi (x_1)=\degw \phi (x_2)$, 
this implies that 
$$
\bigl((\phi \circ \tau )(x_i)-b_i\bigr) ^{\w }
=\bigl(a_{i,1}\phi (x_1)+a_{i,2}\phi (x_2)\bigr) ^{\w }
=b\phi (x_1)^{\w }. 
$$
Because $b_i$ is a constant 
and $\degw \phi (x_1)>0$, 
it follows that 
$(\phi \circ \tau )(x_i)^{\w }=b\phi (x_1)^{\w }$.

(ii) Since $\tau $ is an element of $J_{i,j}$, 
we have an expression as (\ref{eq:JJ-red}) 
with $\sigma =\tau $. 
Then, 
we may write $(\phi \circ \tau )(x_p)=\alpha \phi (x_i)-g$. 
Since $\degw \phi (x_i)>0$ and $g$ is a constant, 
it follows that 
$(\phi \circ \tau )(x_p)^{\w }=\alpha \phi (x_i)^{\w }$, 
proving the latter part. 
Consequently, 
we get $\degw (\phi \circ \tau )(x_p)=\degw \phi (x_i)$.  
Since $\degw \phi \circ \tau <\degw \phi $ by assumption, 
we know that the $\w $-degree of 
$(\phi \circ \tau )(x_q)=\beta \phi (x_j)-\phi (h)$ 
is less than $\degw \phi (x_j)$. 
Because $\beta \neq 0$, 
this implies that 
$\beta \phi (x_j)^{\w }=\phi (h)^{\w }$. 
Since $\phi (h)$ belongs to $R[\phi (x_i)]$, 
it follows that 
$\phi (x_j)^{\w }=\beta ^{-1}\phi (h)^{\w }$ 
belongs to $R[\phi (x_i)]^{\w }=R[\phi (x_i)^{\w }]$. 
Hence, 
we may write $\phi (x_j)^{\w }=b(\phi (x_i)^{\w })^t$ 
in view of the $\w $-homogeneity of $\phi (x_j)^{\w }$, 
where $b\in R\sm \zs $ and $t\geq 0$. 
Then, we have $t\geq 1$, 
since $\degw \phi (x_l)>0$ for $l=1,2$. 
\end{proof}

From (i) and (ii) of Lemma~\ref{lem:AE initial}, 
we get the following statement. 
Assume that $\phi \in \Aut (\Rx /R)$, 
$\tau \in \Aff (R,\x )\cup J$ and $\w \in \Gamma ^2$ 
satisfy 
$\degw \phi (x_l)>0$ for $l=1,2$ and 
$\degw \phi \circ \tau <\degw \phi $. 
Then, 
we have 
$$
\phi (x_j)^{\w }=a(\phi (x_i)^{\w })^t
\quad \text{and} \quad 
\phi (x_i)^{\w }=b(\phi \circ \tau )(x_q)^{\w }
$$
for some $i,j,q\in \{ 1,2\} $ with $i\neq j$, 
$a,b\in K^{\times }$ and $t\geq 1$. 
Hence, 
$\phi (x_1)^{\w }$ and $\phi (x_2)^{\w }$ 
are powers of $(\phi \circ \tau )(x_q)^{\w }$ 
multiplied by elements of $K^{\times }$ 
for some $q\in \{ 1,2\} $.

Now, 
let $\kappa $ be any field, 
and $\w \in \Gamma ^2$ such that $w_i>0$ for $i=1,2$. 
Then, 
$\w $ belongs to the set $W$ 
defined after Lemma~\ref{claim:AE1}. 
Hence, 
Proposition~\ref{prop:A} holds for this $\w $ by (I). 
As a consequence, we know that 
$\T (\kappa ,\x )=\Aff (\kappa ,\x )^+$  
is contained in $\Aff (\kappa ,\x )_{\w }$. 
On the other hand, 
$\T (\kappa ,\x )$ is equal to $\Aut (\kapx /\kappa )$ 
due to Jung~\cite{Jung} and van der Kulk~\cite{Kulk}. 
Thus, 
$\Aut (\kapx /\kappa )$ 
is contained in $\Aff (\kappa ,\x )_{\w }$.

\begin{lem}\label{lem:power}
Assume that $w_i>0$ for $i=1,2$. 
Then, 
for each $\phi \in \Aut (\kapx /\kappa )$ 
and $p\in \{ 1,2\} $, 
there exist $a\in \kappa ^{\times }$, 
$\psi \in \Aff (\kappa ,\x )\cup J$ and $t\geq 1$ 
such that 
$\phi (x_p)^{\w }=a(\psi (x_1)^{\w })^{t}$. 
\end{lem}
\begin{proof}
By the discussion above, 
$\phi $ belongs to $\Aff (\kappa ,\x )_{\w }$. 
Hence, 
there exist $l\in \N $, 
$\phi _1,\ldots ,\phi _{l-1}\in \Aut (\kapx /\kappa )$ 
and $\phi _l\in T:=\Aff (\kappa ,\x )\cup J$ 
satisfying 
(\ref{eq:Gw1}) and (\ref{eq:Gw2}). 
We prove the lemma by induction on $l$. 
If $l=1$, 
then we have $\phi =\phi _l$ by (\ref{eq:Gw1}). 
Hence, $\phi$ belongs to $T$. 
Thus, 
if $p=1$, 
then the statement holds for $a=1$, $t=1$ and $\psi =\phi $. 
Note that $(\phi \circ \iota )(x_1)=\phi (x_2)$ 
and $\phi \circ \iota $ belongs to $T$. 
Hence, if $p=2$, 
then the statement holds for $a=1$, $t=1$ 
and $\psi =\phi \circ \iota $. 
Assume that $l\geq 2$. 
Then, 
for $q=1,2$, 
there exists $\psi _q\in T$ such that 
$\phi _{2}(x_q)^{\w }$ is a power of 
$\psi _q(x_1)^{\w }$ multiplied by a nonzero constant 
by induction assumption. 
By (\ref{eq:Gw2}), 
there exists $\tau _1\in T$ such that 
$\phi _2=\phi \circ \tau _1$ 
and $\deg \phi \circ \tau _1=\degw \phi_2<\degw \phi $. 
Since $\phi (x_l)$ is not a constant, 
we have $\degw \phi (x_l)>0$ for $l=1,2$ 
by the choice of $\w $. 
Hence, 
we know from the remark after 
Lemma~\ref{lem:AE initial} that 
$\phi (x_1)^{\w }$ and $\phi (x_2)^{\w }$ 
are powers of $(\phi \circ \tau _1)(x_q)^{\w }$ 
multiplied by nonzero constants 
for some $q\in \{ 1,2\} $. 
Therefore, 
$\phi (x_p)^{\w }$ is a 
power of $\psi _q(x_1)$ 
multiplied by a nonzero constant, 
since so is 
$(\phi \circ \tau _1)(x_q)^{\w }=\phi _{2}(x_q)^{\w }$. 
\end{proof}

\chapter{Tamely reduced coordinates}
\label{chap:trc}

\section{Structure of coordinates}
\label{sect:coordinate}
\setcounter{equation}{0}

Throughout this chapter, 
we assume that $n=2$. 
The purpose of this chapter is to discuss reductions of polynomials 
by tame automorphisms. 
The notion of ``tamely reduced" coordinates 
introduced in Section~\ref{sect:reduced coordinate} 
will play a crucial role in the next two chapters.

In this section, 
we study the structure of coordinates. 
Let $\kappa $ be any field. 
For each $f\in \kapx \sm \zs $, 
we define an element $\w (f)$ of $(\Zn )^2$ by 
$$
\w (f)=(\deg _{x_2}f,\deg _{x_1}f). 
$$
Set $p_i=\deg _{x_i}f$ for $i=1,2$. 
Then, 
the monomial $x_1^{p_1}x_2^{t}$ 
appears in $f$ for some $t\geq 0$. 
Hence, 
we have 
\begin{equation}\label{eq:w(f)}
\deg _{\w (f)}f\geq 
\deg _{\w (f)}x_1^{p_1}x_2^{t}
=p_1p_2+tp_1\geq p_1p_2. 
\end{equation}
From this, we see that $x_1^{p_1}$ appears in $f$ 
if $\deg _{\w (f)}f=p_1p_2$ and $p_1>0$. 
If this is the case, 
then $x_1^{p_1}$ appears in $f^{\w (f)}$, 
since $\deg _{\w (f)}x_1^{p_1}=p_1p_2=\deg _{\w (f)}f$. 
Likewise, 
if $\deg _{\w (f)}f=p_1p_2$ and $p_2>0$, 
then $x_2^{p_2}$ appears in $f^{\w (f)}$.

\begin{lem}\label{lem:w(f)}
Assume that $p_i>0$ and $x_j^{q}$ 
appears in $f^{\w (f)}$ 
for some 
$i,j\in \{ 1,2\} $ with $i\neq j$ 
and $q\geq 0$. 
Then, 
we have $q=p_j$ and $\deg _{\w (f)}f=p_1p_2$. 
Hence, 
both $x_1^{p_1}$ and $x_2^{p_2}$ appear in $f^{\w (f)}$. 
\end{lem}
\begin{proof}
Since $x_j^{q}$ appears in $f^{\w (f)}$, 
we have $\deg _{\w (f)}f=\deg _{\w (f)}x_j^{q}=p_iq$. 
Hence, 
we get $p_iq\geq p_ip_j$ in view of (\ref{eq:w(f)}). 
Since $p_i>0$ by assumption, 
it follows that $q\geq p_j$. 
On the other hand, 
we have $q\leq \deg _{x_j}f=p_j$, 
since $x_j^{q}$ is a monomial appearing in $f$. 
Thus, 
we get $q=p_j$, 
and therefore 
$\deg _{\w (f)}f=p_iq=p_1p_2$. 
Since $p_i>0$, 
this implies that $x_i^{p_i}$ 
appears in $f^{\w (f)}$ 
by the remark. 
Consequently, 
both $x_1^{p_1}$ and $x_2^{p_2}$ appear in $f^{\w (f)}$. 
\end{proof}

Now, 
assume that $f$ 
is a coordinate of $\kapx $ over $\kappa $. 
Then, 
we have 
$$
|\w (f)|=\deg_{x_2}f+\deg _{x_1}f\geq 1, 
$$ 
since $f$ is not a constant. 
We remark that $|\w (f)|=1$ if and only if 
$f$ is a linear polynomial in $x_i$ 
over $\kappa $ for some $i\in \{ 1,2\} $, 
and hence if and only if 
$f$ belongs to $\kappa [x_i]$ for some $i\in \{ 1,2\} $.

\begin{prop}\label{prop:slope}
Let $f$ be a coordinate of $\kapx $ 
over $\kappa $ with $|\w (f)|>1$. 
Then, 
there exist $i,j\in \{ 1,2\} $ with $i\neq j$, 
$l,m\in \N $ and $a,b\in \kappa ^{\times }$ such that 
$$
\deg _{x_i}f=m,\ \ \deg _{x_j}f=lm
\text{ \ and \ }
f^{\w (f)}=a(x_i-bx_j^l)^m. 
$$
\end{prop}
\begin{proof}
Since $|\w (f)|>1$ by assumption, 
$f$ does not belong to $\kappa [x_i]$ for $i=1,2$. 
Hence, 
we have $p_i:=\deg _{x_i}f>0$ for $i=1,2$. 
By Lemma~\ref{lem:power}, 
we know that 
$f^{\w (f)}=c(\psi (x_1)^{\w (f)})^{m}$ 
for some 
$c\in \kappa ^{\times }$, 
$\psi \in \Aff (\kappa ,\x )\cup J$ 
and $m\geq 1$. 
Then, 
$\psi (x_1)^{\w (f)}$ has the form 
$\alpha x_i^l$ or $\alpha (x_i+bx_j^l)$ for some 
$\alpha ,b \in \kappa ^{\times }$, $l\geq 1$ 
and $i,j\in \{ 1,2\} $ with $i\neq j$. 
Thus, 
$f^{\w (f)}$ is written as $ax_i^{lm}$ or 
$a(x_i+bx_j^l)^m$, 
where $a:=c\alpha ^m$. 
If $f^{\w (f)}=ax_i^{lm}$, 
then the two monomials $x_1^{p_1}$ and $x_2^{p_2}$ 
appear in $f^{\w (f)}$ by Lemma~\ref{lem:w(f)}, 
a contradiction. 
Therefore, 
we conclude that 
$f^{\w (f)}=a(x_i+bx_j^l)^m$. 
Since 
$x_i^m$ and $x_j^{lm}$ appear in $f^{\w (f)}$, 
it follows from Lemma~\ref{lem:w(f)} that 
$\deg _{x_i}f=p_i=m$ and $\deg _{x_j}f=p_j=lm$. 
\end{proof}

Next, 
let $f\in \kapx \sm \kappa $ 
be such that $p_i:=\deg _{x_i}f>0$ for $i=1,2$. 
Then, 
the following conditions are equivalent:

\smallskip

\noindent{\rm (a)} 
$\deg _{\w (f)}f=p_1p_2$.

\noindent{\rm (b)} 
$x_1^{p_1}$ and $x_2^{p_2}$ appear in $f^{\w (f)}$.

\smallskip

\noindent 
In fact, since $p_i>0$ for $i=1,2$ by assumption, 
(a) implies (b) by the remark after (\ref{eq:w(f)}), 
and the converse is obvious. 
In view of Proposition~\ref{prop:slope}, 
we see that 
the equivalent conditions are satisfied 
if $f$ is a coordinate of $\kapx $ over $\kappa $ with $|\w (f)|>1$.

Put 
$$
q_i=\frac{p_i}{\gcd (p_1,p_2)}
$$ 
for $i=1,2$. 
Then, 
every $\w (f)$-homogeneous element of $\kapx $ 
is written as a product of a monomial 
and irreducible binomials of the form 
$x_1^{q_1}-bx_2^{q_2}$ for some $b\in \bar{\kappa }^{\times }$. 
Here, $\bar{\kappa }$ denotes an algebraic closure of $\kappa $. 
Hence, 
if $f$ satisfies 
(a) and (b), 
then we may write 
\begin{equation}\label{eq:f^{w(f)}}
f^{\w (f)}=a\prod _{i=1}^m(x_1^{q_1}-b_ix_2^{q_2}) 
\end{equation}
because of (b) and the $\w (f)$-homogeneity of $f^{\w (f)}$, 
where $a\in \kappa ^{\times }$, 
$m\in \N $ and $b_i\in \bar{\kappa }^{\times }$ 
for $i=1,\ldots ,m$. 
Then, 
we have $p_i=mq_i$ for $i=1,2$.

In this situation, 
we have the following proposition.

\begin{prop}\label{prop:slopee}
Let $f$ be as above, 
and let $\w \in \Gamma ^2$ and $\tau \in \Aut (\kapx /\kappa )$ 
be such that $w_1>0$ or $w_2>0$, 
and $(\tau (x_1)^{\w })^{q_1}\neq b_i(\tau (x_2)^{\w })^{q_2}$ 
for $i=1,\ldots ,m$. 
Then, we have 
\begin{equation}\label{eq:slopeedeg}
\degw \tau (f)=
\max \{ p_1\degw \tau (x_1),p_2\degw \tau (x_2)\} . 
\end{equation}
\end{prop}
\begin{proof}
We may extend $\tau $ to an element of 
$\Aut (\bar{\kappa }[\x ] /\bar{\kappa })$ 
naturally. 
Then, 
we have 
$$
\tau (f^{\w (f)})
=a\prod _{i=1}^m
(\tau (x_1)^{q_1}-b_i\tau (x_2)^{q_2}). 
$$
Put $\xi _i=\degw \tau (x_i)$ for $i=1,2$. 
Then, 
we obtain 
$$
\degw (\tau (x_1)^{q_1}-b_i\tau (x_2)^{q_2})
=\max \{ q_1\xi_ 1,q_2\xi _2\} 
$$
for $i=1,\ldots ,m$ by the assumption that 
$(\tau (x_1)^{\w })^{q_1}\neq b_i(\tau (x_2)^{\w })^{q_2}$ 
and $b_i\neq 0$. 
Hence, it follows that 
$$
\degw \tau (f^{\w (f)})
=m\max \{ q_1\xi_ 1,q_2\xi _2\} =\max \{ p_1\xi_ 1,p_2\xi _2\} =:d. 
$$
Thus, 
it suffices to check that 
$\degw \tau (f-f^{\w (f)})<d$. 
Let $x_1^{a_1}x_2^{a_2}$ be any monomial 
appearing in $f-f^{\w (f)}$. 
Then, 
we have 
$$
a_1p_2+a_2p_1=\deg _{\w (f)}x_1^{a_1}x_2^{a_2}
<\deg _{\w (f)}f=p_1p_2, 
$$
since $f$ satisfies the condition (a) by assumption. 
This yields that 
\begin{equation}\label{eq:slopeeq}
a_2<p_2-a_1p_2p_1^{-1}
\quad \text{and}\quad 
a_1<p_1-a_2p_1p_2^{-1}. 
\end{equation}
First, 
assume that $p_1\xi _1\leq p_2\xi _2$. 
Then, 
we have $d=p_2\xi _2$. 
Since $w_1>0$ or $w_2>0$ by assumption, 
we know by (\ref{eq:max phi}) that 
$\max \{ \xi _1,\xi _2\} >0$. 
Hence, we get $\xi _2>0$. 
Thus, 
it follows from 
the first inequality of (\ref{eq:slopeeq}) that 
\begin{align*}
&\degw \tau (x_1^{a_1}x_2^{a_2})
=a_1\xi _1+a_2\xi _2<
a_1\xi _1+
(p_2-a_1p_2p_1^{-1}) \xi _2 \\
&\quad =a_1p_1^{-1}(p_1\xi _1-p_2\xi _2)+p_2\xi _2
\leq p_2\xi _2=d. 
\end{align*}
Next, assume that $p_1\xi _1>p_2\xi _2$. 
Then, 
we have $d=p_1\xi _1$, 
and $\xi _1>0$ by (\ref{eq:max phi}) as above. 
Hence, 
it follows from the second inequality of (\ref{eq:slopeeq}) 
that 
\begin{align*}
&\degw \tau (x_1^{a_1}x_2^{a_2})
=a_1\xi _1+a_2\xi _2<
(p_1-a_2p_1p_2^{-1})\xi _1+a_2\xi _2 \\
&\quad =
a_2p_2^{-1}(p_2\xi _2-p_1\xi _1)+p_1\xi _1<p_1\xi _1=d. 
\end{align*}
Thus, we have proved 
$\degw \tau (f-f^{\w (f)})<d$. 
Therefore, 
we get $\degw \tau (f)=d$. 
\end{proof}

Let $f$ be a coordinate of $\kapx $ 
over $\kappa $ with $|\w (f)|>1$. 
Then, 
$f$ satisfies the assumption of 
Proposition~\ref{prop:slopee}. 
Write $f^{\w (f)}$ as in Proposition~\ref{prop:slope}, 
and take $\w \in \Gamma ^2$ 
such that $w_1>0$ or $w_2>0$. 
Clearly, 
$x_i^{\w }$ 
is not equal to $b(x_j^{\w })^l$. 
Hence, 
we get 
\begin{equation}\label{eq:slopedeg}
\degw f=\max \{ m\degw x_i,lm\degw x_j\} 
=\max \{ mw_i,lmw_j\} 
\end{equation}
by applying Proposition~\ref{prop:slopee} with $\tau =\id _{\kapx }$.

Now, 
we prove the statement (II) after Lemma~\ref{claim:AE1}. 
We may assume that $w_1<0$ and $w_2>0$ 
as noted. 
First, we show that 
$\degw f$ belongs to $\{ tw_2\mid t\in \N \} $ 
for each coordinate $f$ of $\Rx $ over $R$ 
not belonging to $R[x_1]$. 
Since $w_2>0$ by assumption, 
this is clear if $f$ belongs to $R[x_2]$. 
Assume that $f$ does not belong to $R[x_2]$. 
Then, 
we have $|\w (f)|>1$, 
since $f$ also does not belong to $R[x_1]$ by assumption. 
Note that $f$ is regarded as a coordinate of $\Kx $ over $K$. 
Hence, 
we know by (\ref{eq:slopedeg}) that 
$\degw f=\max \{ t_1w_1,t_2w_2\} $ 
for some $t_1,t_2\in \N $. 
Since $w_1<0$ and $w_2>0$, 
we get $\degw f=t_2w_2$. 
Thus, 
$\degw f$ belongs to $\{ tw_2\mid t\in \N \} $. 
Now, 
observe that 
\begin{equation}\label{eq:Sigma}
\{ tw_2\mid t\in \N \} 
\cup \{ w_1+w_2\}  
\end{equation}
is a  well-ordered subset of $\Gamma $. 
We show that $\Sigma _{\w }$ is contained in (\ref{eq:Sigma}). 
Take any $\phi \in \Aut (\Rx /R)$. 
If $\phi (x_i)$ does not belong to $R[x_1]$ for $i=1,2$, 
then $\degw \phi (x_i)$ belongs to $\{ tw_2\mid t\in \N \} $ 
for $i=1,2$ by the preceding discussion. 
Hence, 
$\degw \phi $ belongs to (\ref{eq:Sigma}). 
Assume that 
$\phi (x_i)$ belongs to $R[x_1]$ for some $i\in \{ 1,2\} $. 
Then, 
$\degw \phi (x_i)$ is equal to zero or $w_1$ 
by Lemma~\ref{lem:deg coordinate}~(i). 
Because 
$\phi $ belongs to $J_{1,2}$, 
we have $\phi (x_j)=ax_2+g$ 
for some $a\in R^{\times }$ and $g\in R[x_1]$ 
for $j\neq i$. 
Hence, we get $\degw \phi (x_j)=w_2$, 
since $\degw ax_2=w_2>0$ and $\degw g=(\deg _{x_1}g)w_1\leq 0$. 
Thus, 
$\degw \phi $ is equal to $w_2$ or $w_1+w_2$, 
and so belongs to (\ref{eq:Sigma}). 
Therefore, 
$\Sigma _{\w }$ is contained in (\ref{eq:Sigma}), 
proving the well-orderedness of $\Sigma _{\w }$. 
This completes the proof of (II), and 
thereby completing the proof 
of Theorem~\ref{thm:AEreduction}.

\section{Reduction of a polynomial}
\setcounter{equation}{0}
\label{sect:reduced coordinate}

Let $f$ be an element of $\Rx \sm R$. 
Recall that 
$f$ is said to be {\it tamely reduced} over $R$ 
if 
$$
|\w (\tau (f))|\geq |\w (f)|
$$ 
holds for every $\tau \in \T (R,\x )$. 
Obviously, 
$f$ is tamely reduced over $R$ if $|\w (f)|=1$, 
since $f$ is a linear polynomial in $x_i$ over $R$ 
for some $i\in \{ 1,2\} $.

Now, assume that 
$p_i:=\deg _{x_i}f>0$ for $i=1,2$, 
and $f$ satisfies the equivalent conditions (a) and (b) 
before Proposition~\ref{prop:slopee}. 
Then, 
$f^{\w (f)}$ is written as in (\ref{eq:f^{w(f)}}) 
with $\kappa =K$.

With this notation, 
we have the following proposition.

\begin{prop}\label{prop:reduced polynomials}
Let $f\in \Rx $ be as above. 
Then, 
the following conditions are equivalent$:$

\noindent
{\rm (1)} $f$ is tamely reduced over $R$.

\noindent {\rm (2)} 
The following statements hold$:$ 

\noindent 
{\rm (i)} 
If $p_1=p_2$, 
then $b_1,\ldots ,b_m$ do not belong to $V(R)$. 

\noindent 
{\rm (ii)} 
If $p_1<p_2$ and $p_1$ divides $p_2$, 
then $b_1,\ldots ,b_m$ do not belong to $R$. 

\noindent 
{\rm (iii)} 
If $p_1>p_2$ and $p_2$ divides $p_1$, 
then $b_1^{-1},\ldots ,b_m^{-1}$ do not belong to $R$.

\noindent
{\rm (3)} 
Let $\Gamma $ be any totally ordered additive group, 
$\w \in \Gamma ^2$ such that $w_1>0$ or $w_2>0$, 
and $\tau $ any element of $\T (R,\x )$. 
Then, we have 
$$
\degw \tau (f)=
\max \{ p_i\degw \tau (x_i)\mid i=1,2\} . 
$$ 
\end{prop}

To prove this proposition, 
we use the following lemmas.

\begin{lem}\label{lem:for prop:reduced polynomials}
Let $f$ be as above, 
and $g\in \Rx $ such that 
$\deg _{\w (f)}g\leq \deg _{\w (f)}f$ 
and $\deg _{x_i}g^{\w (f)}<\deg _{x_i}f$ 
for some $i\in \{ 1,2\} $. 
Then, we have $|\w (g)|<|\w (f)|$. 
\end{lem}
\begin{proof}
It suffices to show that 
$\deg _{x_i}g<p_i$ 
and $\deg _{x_j}g\leq p_j$, 
where $j\neq i$. 
Suppose to the contrary that 
$\deg _{x_i}g\geq p_i$. 
Then, 
the monomial $x_i^{p_i+s}x_j^t$ appears in $g$ 
for some $s,t\geq 0$. 
Since $p_l>0$ for $l=1,2$ by assumption, 
we have 
$$
p_ip_j\leq (p_i+s)p_j+tp_i
=\deg _{\w (f)}x_i^{p_i+s}x_j^t
\leq \deg _{\w (f)}g
\leq \deg _{\w (f)}f. 
$$
On the other hand, we know that 
$\deg _{\w (f)}f=p_ip_j$ 
by the condition (a) before Proposition~$\ref{prop:slopee}$. 
Hence, 
it follows 
that $s=t=0$ and $\deg _{\w (f)}x_i^{p_i}=\deg _{\w (f)}g$. 
Thus, 
$x_i^{p_i}$ appears in $g^{\w (f)}$. 
Therefore, we get 
$\deg _{x_i}g^{\w (f)}\geq p_i=\deg _{x_i}f$, 
a contradiction. 
This proves that $\deg _{x_i}g<p_i$. 
Next, suppose to the contrary that 
$\deg _{x_j}g>p_j$. 
Then, 
the monomial $x_i^sx_j^{p_j+t}$ appears in $g$ 
for some $s\geq 0$ and $t\geq 1$. 
Hence, we have 
$$
\deg _{\w (f)}g\geq 
\deg _{\w (f)}x_i^sx_j^{p_j+t}=sp_j+(p_j+t)p_i>p_ip_j 
=\deg _{\w (f)}f, 
$$
a contradiction. 
Therefore, we get $\deg _{x_j}g\leq p_j$. 
\end{proof}

We say that $\tau \in \Aut (\Rx /R)$ 
is $\w $-{\it homogeneous} if 
$\tau (x_i)$ is $\w $-homogeneous 
and $\degw \tau (x_i)=w_i$ for each $i$. 
If this is the case, 
then $\tau (\Rx _{\gamma })$ 
is contained in $\Rx _{\gamma }$ 
for each $\gamma \in \Gamma $. 
Hence, 
we have $\degw \tau (h)=\degw h$ and 
$\tau (h)^{\w }=\tau (h^{\w })$ 
for each $h\in \Rx $.

\begin{lem}\label{lem:w-homoge autom}
Let $f$ be as above, 
and assume that $\tau \in \Aut (\Rx /R)$ is $\w (f)$-homogeneous. 
If $\deg _{x_i}\tau (f^{\w (f)})<\deg _{x_i}f^{\w (f)}$ 
for some $i\in \{ 1,2\} $, 
then we have $|\w (\tau (f))|<|\w (f)|$. 
\end{lem}
\begin{proof}
Since $\tau $ is $\w (f)$-homogeneous, 
we have $\deg _{\w (f)}\tau (f)=\deg _{\w (f)}f$ 
and $\tau (f)^{\w (f)}=\tau (f^{\w (f)})$. 
Since $\deg _{x_i}\tau (f^{\w (f)})<\deg _{x_i}f^{\w (f)}$ 
by assumption, 
it follows that 
$\deg _{x_i}\tau (f)^{\w (f)}
<\deg _{x_i}f^{\w (f)}\leq \deg _{x_i}f$. 
Therefore, we get 
$|\w (\tau (f))|<|\w (f)|$ 
by applying Lemma~\ref{lem:for prop:reduced polynomials} 
with $g=\tau (f)$. 
\end{proof}

Now, 
let us prove Proposition~\ref{prop:reduced polynomials}. 
First, 
we prove that (1) implies the three statements of 
(2) by contradiction. 
Suppose that 
$p_1=p_2$ and 
$b_s$ belongs to $V(R)$ for some $s$. 
Then, 
we have $q_1=q_2=1$. 
Since $b_s$ belongs to $V(R)$, 
there exists $(a_{i,j})_{i,j}\in \GL (2,R)$ 
such that $b_s=a_{1,1}/a_{2,1}$. 
Define $\tau \in \Aff (R,\x )$ by 
$\tau (x_i)=a_{i,1}x_1+a_{i,2}x_2$ for $i=1,2$. 
Then, 
$\tau $ is $\w (f)$-homogeneous, 
and satisfies $\deg _{x_1}\tau (x_1-b_ix_2)\leq 1$ 
for 
$i=1,\ldots ,m$. 
By the choice of $(a_{i,j})_{i,j}$, 
we have 
$$
\tau (x_1-b_sx_2)=
a_{1,1}x_1+a_{1,2}x_2-b_s(a_{2,1}x_1+a_{2,2}x_2)
=(a_{1,2}-b_sa_{2,2})x_2. 
$$ 
Hence, we know that 
$$
\deg _{x_1}
\tau (f^{\w (f)})=
\deg _{x_1}
a\prod _{i=1}^m\tau (x_1-b_ix_2)
<m=\deg _{x_1}f^{\w (f)}. 
$$
Thus, 
we conclude that $|\w (\tau (f))|<|\w (f)|$ 
by Lemma~\ref{lem:w-homoge autom}. 
Since $\tau $ is affine, 
this contradicts that 
$f$ is tamely reduced over $R$. 
Therefore, 
(1) implies (i) of (2).

Next, 
suppose that $p_1<p_2$, $p_1$ divides $p_2$, 
and $b_s$ belongs to $R$ for some $s$. 
Then, we have $q_1=1$ and $q_2=p_2/p_1$. 
Since $b_s$ is an element of $R$, 
we may define $\tau \in \Aut (\Rx /R[x_2])$ by 
$\tau (x_1)=x_1+b_sx_2^{q_2}$. 
Then, $\tau $ is $\w (f)$-homogeneous, 
and 
$$
\tau (f^{\w (f)})
=a\prod _{i=1}^m\tau (x_1-b_ix_2^{q_2})
=a\prod _{i=1}^m\bigl(x_1+(b_s-b_i)x_2^{q_2}\bigr). 
$$
From this, we know that 
$\deg _{x_2}\tau (f^{\w (f)})<mq_2=\deg _{x_2}f^{\w (f)}$. 
Thus, we 
conclude that 
$|\w (\tau (f))|<|\w (f)|$ 
by Lemma~\ref{lem:w-homoge autom}. 
Since $\tau $ is elementary, 
this contradicts that 
$f$ is tamely reduced over $R$. 
Therefore, (1) implies (ii) of (2). 
We can check that (1) implies (iii) of (2) similarly.

Next, 
we prove that (2) implies (3). 
Take any totally ordered additive group $\Gamma $, 
$\w \in \Gamma ^2$ with $w_1>0$ or $w_2>0$, 
and any $\tau \in \T (R,\x )$. 
Then, 
we may regard $\tau $ as an element of $\Aut (\Kx /K)$. 
Hence, 
we may use Proposition~\ref{prop:slopee} 
for $\w $ and $\tau $ by taking $\kappa =K$, 
since $w_1>0$ or $w_2>0$ by assumption. 
Thus, 
it suffices to check that 
$(\tau (x_1)^{\w })^{q_1}\neq b_i(\tau (x_2)^{\w })^{q_2}$ 
for $i=1,\ldots ,m$. 
Suppose to the contrary that 
\begin{equation}\label{eq:proof:reduced1}
(\tau (x_1)^{\w })^{q_1}=
b_s(\tau (x_2)^{\w })^{q_2}
\end{equation}
for some $s$. 
Then, 
$\tau (x_1)^{\w }$ and $\tau (x_2)^{\w }$ 
are algebraically dependent over $K$. 
Hence, 
we have $\degw \tau >|\w |$ by Lemma~\ref{lem:minimal autom}. 
Since $\tau $ belongs to 
$\T (R,\x )=\Aff (R,\x )^+$, 
and $w_1>0$ or $w_2>0$ by assumption, 
it follows from Theorem~\ref{thm:AEreduction} that 
$\tau $ admits an affine reduction 
or elementary reduction for the weight $\w $.

First, 
assume that $\tau $ admits an affine reduction. 
Since $w_1>0$ or $w_2>0$, 
we have $\degw \tau (x_i)>0$ 
for some $i\in \{ 1,2\} $ by (\ref{eq:max phi}), 
and hence for any $i\in \{ 1,2\} $ by (\ref{eq:proof:reduced1}). 
On account of Lemma~\ref{lem:AE initial} (i) with $\phi =\tau $, 
we know that $\tau (x_1)^{\w }=c\tau (x_2)^{\w }$ 
for some $c\in V(R)$. 
Then, 
we have $c^{q_1}(\tau (x_2)^{\w })^{q_1}=b_s(\tau (x_2)^{\w })^{q_2}$ 
in view of (\ref{eq:proof:reduced1}). 
Since $\degw \tau (x_2)>0$, 
it follows that $q_1=q_2$. 
This implies that $q_1=q_2=1$. 
Hence, we get $c=b_s$. 
Thus, 
$b_s$ belongs to $V(R)$. 
On the other hand, 
we have $p_1=p_1$ since $q_1=q_2$. 
This contradicts (i) of (2).

Next, 
assume that $\tau $ admits an elementary reduction. 
Then, by Lemma~\ref{lem:AE initial}~(ii) with $\phi =\tau $, 
we know that $\tau (x_i)^{\w }=c(\tau (x_j)^{\w })^t$ 
for some $i,j\in \{ 1,2\} $ with $i\neq j$, 
$c\in R\sm \zs $ and $t\geq 1$. 
Note that $c$ and $c^{-1}$ belong to $V(R)$. 
Hence, 
we can get a contradiction 
as in the previous case when $t=1$. 
Assume that $t\geq 2$. 
If $(i,j)=(1,2)$, 
then we have 
$c^{q_1}(\tau (x_2)^{\w })^{tq_1}
=b_s(\tau (x_2)^{\w })^{q_2}$. 
Since $\degw \tau (x_2)>0$, 
this gives that $tq_1=q_2$. 
Hence, we get $q_1=1$, $q_2=t\geq 2$ 
and $b_s=c$. 
Thus, 
we know that $p_1<p_2$, 
$p_1$ divides $p_2$, 
and $b_s$ belongs to $R$. 
This contradicts (ii) of (2). 
If $(i,j)=(2,1)$, 
then we have 
$(\tau (x_1)^{\w })^{q_1}
=b_sc^{q_2}(\tau (x_1)^{\w })^{tq_2}$. 
Since $\degw \tau (x_1)>0$, 
this gives that $q_1=tq_2$. 
Hence, we get $q_2=1$, $q_1=t\geq 2$ 
and $b_s^{-1}=c$. 
Thus, 
we know that $p_1>p_2$, 
$p_2$ divides $p_1$, 
and $b_s^{-1}$ belongs to $R$. 
This contradicts (iii) of (2). 
Therefore, 
$(\tau (x_1)^{\w })^{q_1}\neq b_i(\tau (x_2)^{\w })^{q_2}$ 
holds for $i=1,\ldots ,m$. 
This proves that (2) implies (3).

Finally, 
we prove that (3) implies (1). 
Without loss of generality, 
we may assume that $p_1\leq p_2$ 
by interchanging $x_1$ and $x_2$ 
if necessary. 
Take any $\tau \in \T(R,\x )$. 
Then, 
we have 
\begin{align}\begin{split}\label{eq:max}
\deg _{x_j}\tau (f)
=\max \{ p_i\deg _{x_j}\tau (x_i)\mid i=1,2\} 
\end{split}\end{align}
for $j=1,2$ 
by applying (3) with $\Gamma =\Z $ and $\w =\e _j$. 
First, take $s\in \{ 1,2\} $ such that 
$\deg _{x_s}\tau (x_2)\geq 1$. 
Then, 
we obtain 
$\deg _{x_s}\tau (f)\geq p_2$ from (\ref{eq:max}) with $j=s$. 
Next, take $l\in \{ 1,2\} $ 
such that $\deg _{x_t}\tau (x_l)\geq 1$ 
for $t\in \{ 1,2\} $ with $t\neq s$. 
Then, 
we get 
$\deg _{x_t}\tau (f)\geq p_l$ 
by (\ref{eq:max}) with $j=t$. 
Since $p_1\leq p_2$ by assumption, 
it follows that $\deg _{x_t}\tau (f)\geq p_1$. 
Hence, 
we have 
$$
|\w (\tau (f))|
=\deg _{x_s}\tau (f)+\deg _{x_t}\tau (f)\geq 
p_2+p_1=|\w (f)|. 
$$
Thus, 
$f$ is tamely reduced over $R$. 
Therefore, 
(3) implies (1). 
This completes the proof of 
Proposition~\ref{prop:reduced polynomials}.

For $f\in \Rx \sm R$, 
we consider the subgroup 
$$
H(f):=\Aut (\Rx /R[f])\cap \T (R,\x )
$$
of $\Aut (\Rx /R)$. 
It is worthwhile to mention that, 
if $R=k[x_3]$, 
then we have 
\begin{equation}\label{eq:k[x_3,f]}
H(f)=\Aut (k[x_1,x_2,x_3]/k[x_3,f])\cap 
\T (k,\{ x_1,x_2,x_3\} )
\end{equation}
by virtue of Theorem~\ref{thm:tame23}. 
Here, 
we recall that $k$ 
is an arbitrary field of characteristic zero 
throughout this monograph.

By means of Proposition~\ref{prop:reduced polynomials}, 
we obtain the following theorem.

\begin{thm}\label{prop:intersection theorem}
Assume that 
$f\in \Rx $ satisfies $\deg _{x_i}f>0$ for $i=1,2$ 
and the equivalent conditions {\rm (a)} and {\rm (b)} 
before Proposition~$\ref{prop:slopee}$. 
If $f$ is tamely reduced over $R$, 
then the following assertions hold$:$ 

\noindent
{\rm (i)} If $\deg _{x_1}f=\deg _{x_2}f$, 
then $H(f)$ 
is contained in $\Aff (R,\x )$. 

\noindent
{\rm (ii)} If $\deg _{x_i}f>\deg _{x_j}f$ 
for $i,j\in \{ 1,2\} $ with $i\neq j$, 
then $H(f)$ is contained in $J(R;x_i,x_j)$. 
\end{thm}
\begin{proof}
Since $f$ satisfies the assumption of 
Proposition~\ref{prop:reduced polynomials}, 
and is tamely reduced over $R$, 
we know that $f$ satisfies the equivalent conditions 
of Proposition~\ref{prop:reduced polynomials}. 
In particular, 
(\ref{eq:max}) holds for every 
$\tau \in \T (R,\x )$ due to (3).

Now, 
take any $\tau \in H(f)$. 
Then, $\tau $ belongs to $\T (R,\x )$ by definition. 
Hence, 
$\tau $ satisfies (\ref{eq:max}). 
Since $\tau (f)=f$, 
it follows that 
\begin{align}\begin{split}\label{eq:maxx}
\deg _{x_j}f=\max \{ 
(\deg _{x_i}f)\deg _{x_j}\tau (x_i)\mid i=1,2\} 
\end{split}\end{align}
for $j=1,2$.

(i) Since $\deg _{x_1}f=\deg _{x_2}f$, 
we get 
$\max \{ \deg _{x_j}\tau (x_i)\mid i=1,2\} =1$ 
for $j=1,2$ by (\ref{eq:maxx}). 
Put $\w _i=\w (\tau (x_i))$ for $i=1,2$. 
Then, 
it follows that $|\w _i|=1$ or $\w _i=(1,1)$. 
If $|\w _i|=1$, 
then $\tau (x_i)$ is a linear polynomial. 
The same holds when $\w _i=(1,1)$ 
in view of Proposition~\ref{prop:slopee}. 
Thus, 
$\tau $ belongs to $\Aff (R,\x )$. 
Therefore, 
$H(f)$ is contained in $\Aff (R,\x )$.

(ii) Since $\deg _{x_i}f>\deg _{x_j}f$, 
we have $\deg _{x_i}f>0$, and 
$$
\deg _{x_i}f>\deg _{x_j}f
\geq (\deg _{x_i}f)\deg _{x_j}\tau (x_i)
$$
by (\ref{eq:maxx}). 
Hence, 
we get $\deg _{x_j}\tau (x_i)<1$, 
and so 
$\tau (x_i)$ belongs to $R[x_i]$. 
Thus, 
$\tau $ belongs to $J(R;x_i,x_j)$. 
Therefore, 
$H(f)$ is contained in $J(R;x_i,x_j)$. 
\end{proof}

Now, 
let $S$ be an over domain of $R$. 
In the rest of this section, 
we consider the case where 
$f\in \Rx $ is a coordinate of $\Sx $ over $S$. 
Assume that $|\w (f)|>1$. 
Then, 
$f$ satisfies $\deg _{x_i}f>0$ for $i=1,2$ 
and the equivalent conditions (a) and (b) 
before Proposition~\ref{prop:slopee}. 
Hence, 
$f$ satisfies the assumption of 
Proposition~\ref{prop:reduced polynomials}. 
Let $L$ be the field of fractions of $S$. 
Then, 
we may regard $f$ as a coordinate of $\Lx $ over $L$. 
Hence, 
there exist $i,j\in \{ 1,2\} $ with $i\neq j$, 
$a,b\in L^{\times }$ and $l,m\in \N $ such that 
$$
\deg _{x_i}f=m,\ \ \deg _{x_j}f=lm
\text{ \ and \ }
f^{\w (f)}=a(x_i-bx_j^l)^m
$$
by Proposition~\ref{prop:slope} with $\kappa =L$.

With this notation, 
the following proposition follows from 
Proposition~\ref{prop:reduced polynomials}.

\begin{prop}\label{prop:reduced coordinate}
Let $f\in \Rx $ be a coordinate of $\Sx $ over $S$ 
such that $|\w (f)|>1$. 
Then, the following assertions hold$:$

\noindent
{\rm (i)} Assume that $\deg _{x_1}f=\deg _{x_2}f$, i.e., $l=1$. 
Then, 
$f$ is tamely reduced over $R$ 
if and only if $b$ does not belong to $V(R)$. 

\noindent
{\rm (ii)} Assume that $\deg _{x_i}f<\deg _{x_j}f$, 
i.e., $l\geq 2$. 
Then, 
$f$ is tamely reduced over $R$ 
if and only if $b$ does not belong to $R$. 
\end{prop}

In fact, 
we obtain (i) and (ii) from the equivalence between 
(1) and (2) of Proposition~\ref{prop:reduced polynomials}, 
since $\deg _{x_i}f=m$ divides $\deg _{x_j}f=lm$ 
in the case of (ii).

\begin{cor}\label{lem:R/S}
Assume that $V(R)=L^{\times }$. 
Let $f\in \Rx $ be a coordinate of $\Sx $ over $S$. 
If $f$ is tamely reduced over $R$, 
then we have $\deg _{x_1}f\neq \deg _{x_2}f$. 
\end{cor}
\begin{proof}
Suppose to the contrary that $\deg _{x_1}f=\deg _{x_2}f$. 
Then, we have $|\w (f)|>1$, 
since $f$ is not a constant. 
Hence, 
$f^{\w (f)}$ is written as above. 
Since $b$ is an element of $L^{\times }$, 
and $L^{\times }=V(R)$ by assumption, 
it follows that $b$ belongs to $V(R)$. 
Thus, 
$f$ is not tamely reduce over $R$ by 
Proposition~\ref{prop:reduced coordinate} (i). 
This is a contradiction. 
Therefore, 
we get $\deg _{x_1}f\neq \deg _{x_2}f$. 
\end{proof}

Theorem~\ref{prop:intersection theorem} 
immediately implies the following theorem. 
This theorem plays a crucial role 
in Chapters 3 and 4.

\begin{thm}\label{thm:reduced coordinate}
Let $f\in \Rx $ be a coordinate of $\Sx $ over $S$ 
with $|\w (f)|>1$. 
If $f$ is tamely reduced over $R$, 
then the following assertions hold$:$ 

\noindent
{\rm (i)} If $\deg _{x_1}f=\deg _{x_2}f$, 
then $H(f)$ 
is contained in $\Aff (R,\x )$.

\noindent
{\rm (ii)} If $\deg _{x_i}f>\deg _{x_j}f$ 
for $i,j\in \{ 1,2\} $ with $i\neq j$, 
then $H(f)$ is contained in $J(R;x_i,x_j)$. 
\end{thm}

\chapter{Triangularizability and tameness} 
\label{chap:t&t}

\section{Triangularizability}
\setcounter{equation}{0}
\label{sect:triangularizability}

Throughout this chapter, 
we assume that $R$ is a $\Q $-domain, 
and $K$ is the field of fractions of $R$. 
We study tameness and wildness of $\exp D$ 
for $D\in \lnd _R\Rx $ 
using the theory developed in the previous chapter.

The following theorem is a key result in this section. 
This theorem gives a solution to 
Problem~\ref{q:strong} for $n=2$.

\begin{thm}\label{thm:triangularizability}
Assume that $n=2$. 
Let $D\in \lnd _R\Rx $ be such that 
$\exp D$ belongs to $\T (R,\x )$. 
Then, 
there exists $\tau \in \T (R,\x )$ such that 
$\tau ^{-1}\circ D\circ \tau $ is triangular or affine. 
If $V(R)=K^{\times }$, 
then there exists $\tau \in \T (R,\x )$ such that 
$\tau ^{-1}\circ D\circ \tau $ is triangular. 
\end{thm}

Here, we recall that $D\in \lnd _R\Rx $ 
is said to be {\it affine} 
if $\deg D(x_i)\leq 1$ for $i=1,\ldots ,n$.

In view of the remark after Problem~\ref{q:strong}, 
Theorem~\ref{thm:triangularizability} 
implies the following theorem. 
From this theorem, 
we know that the answer to Question~\ref{p:fixed points} 
is affirmative when $n=2$.

\begin{thm}\label{cor:fixed points}
Assume that $n=2$. 
Let $D$ be an element of $\lnd _R\Rx $. 
If $\exp fD$ belongs to $\T (R,\x )$ 
for some $f\in \ker D\sm \zs $, 
then $\exp D$ belongs to $\T (R,\x )$. 
\end{thm}

With the aid of 
Theorem~\ref{thm:tame23}, 
we get the following theorem from 
Theorems~\ref{thm:triangularizability} 
and~\ref{cor:fixed points}. 
This theorem gives partial answers to 
Question~\ref{p:fixed points} and Problem~\ref{q:strong}.

\begin{thm}\label{thm:anstoF}
Assume that $n=3$. 
Let $D$ be an element of $\lnd _k\kx $ 
which kills a tame coordinate of $\kx $ over $k$. 
Then, 
the following assertions hold$:$

\noindent{\rm (i)} 
$\exp D$ belongs to $\T (k,\x )$ if and only if 
$\tau ^{-1}\circ D\circ \tau $ is triangular 
for some $\tau \in \T (k,\x )$. 

\noindent{\rm (ii)} 
If $\exp fD$ belongs to $\T (k,\x )$ 
for some $f\in \ker D\sm \zs $, 
then $\exp D$ belongs to $\T (k,\x )$. 
\end{thm}
\begin{proof}
By assumption, 
there exists a tame coordinate $g$ of $\kx $ over $k$ 
such that $D(g)=0$. 
Take $\sigma \in \T (k,\x )$ such that $\sigma (x_1)=g$, 
and put $D'=\sigma ^{-1}\circ D\circ \sigma $. 
Then, we have $D'(x_1)=0$. 
Hence, 
$D'$ belongs to $\lnd _{k[x_1]}\kx $,  
and $\exp D'$ belongs to $\Aut (\kx /k[x_1])$.

(i) The ``if" part is clear. 
We prove the ``only if" part. 
Assume that $\exp D$ belongs to $\T (k,\x )$. 
Then, $\exp D'$ belongs to $\T (k,\x )$. 
Since $\exp D'$ belongs to $\Aut (\kx /k[x_1])$, 
it follows that $\exp D'$ belongs to 
$\T (k[x_1],\{x_2,x_3\} )$ 
on account of Theorem~\ref{thm:tame23}. 
Since $k[x_1]$ is a PID, 
we have $V(k[x_1])=k(x_1)^{\times }$. 
Regard $D'$ as a 
derivation of the polynomial ring in 
the two variables $x_2$ and $x_3$ over $k[x_1]$. 
Then, 
it follows from the last part of 
Theorem~\ref{thm:triangularizability} 
that 
$$
D'':=
\tau ^{-1}\circ D'\circ \tau 
=(\sigma \circ \tau )^{-1}\circ D\circ (\sigma \circ \tau )
$$ 
is triangular for some $\tau \in \T (k[x_1],\{ x_2,x_3\} )$, 
i.e., 
$D''(x_1)=0$, $D''(x_2)$ belongs to $k[x_1]$, 
and $D''(x_3)$ belongs to $k[x_1,x_2]$. 
This implies that 
$D''$ is a triangular derivation of $\kx $ over $k$. 
Clearly, 
$\sigma \circ \tau $ belongs to $\T (k,\x )$. 
Therefore, the ``only if" part is true.

(ii) 
Since $\exp fD$ belongs to $\T (k,\x )$ by assumption, 
and $\sigma $ is an element of $\T (R,\x )$, 
we see that 
$$
\exp \sigma ^{-1}(f)D'=\sigma ^{-1}\circ (\exp fD)\circ \sigma 
$$ 
belongs to $\T (k,\x )$. 
We may regard $\sigma ^{-1}(f)D'$ 
as an element of $\lnd _{k[x_1]}\kx $. 
Hence, 
$\exp \sigma ^{-1}(f)D'$ 
belongs to $\Aut (\kx /k[x_1])$. 
Thanks to Theorem~\ref{thm:tame23}, 
it follows that 
$\exp \sigma ^{-1}(f)D'$ belongs to 
$\T (k[x_1],\{ x_2,x_3\} )$. 
Thus, we know that 
$\exp D'$ belongs to $\T (k[x_1],\{ x_2,x_3\} )$ 
by virtue of Theorem~\ref{cor:fixed points}. 
This implies that $\exp D'$ belongs to $\T (k,\x )$. 
Therefore, 
$\exp D$ belongs to $\T (k,\x )$. 
\end{proof}

Now, 
we prove Theorem~\ref{thm:triangularizability}. 
The following is a key proposition.

\begin{prop}\label{prop:sankakuka}
Assume that $n=2$. 
Let $f\in \Rx $ be a coordinate of $\Sx $ over $S$ 
for some over domain $S$ of $R$, 
and $D\in \lnd _R\Rx $ such that $D(f)=0$ 
and $\exp D$ belongs to $\T (R,\x )$. 
If $f$ is tamely reduced over $R$, 
then the following assertions hold$:$

\noindent{\rm (i)} 
If $\deg _{x_1}f>\deg _{x_2}f$, 
then $D$ is triangular. 

\noindent{\rm (ii)} 
If $\deg _{x_1}f=\deg _{x_2}f$, 
then $D$ is affine. 
\end{prop}

Let $f$ and $D$ be as in Proposition~\ref{prop:sankakuka}. 
Then, 
we have $(\exp D)(f)=f$, 
since $D(f)=0$ by assumption. 
Hence, 
$\exp D$ belongs to $H(f)$ 
due to the assumption that 
$\exp D$ belongs to $\T (R,\x )$. 
Since $f$ is tamely reduced over $R$, 
we know by Theorem~\ref{thm:reduced coordinate} 
that $H(f)$ is contained in $J(R;x_1,x_2)$ 
in the case of (i), 
and in $\Aff (R,\x )$ 
in the case of (ii). 
Thus, 
$\exp D$ belongs to $J(R;x_1,x_2)$ in the case of (i), 
and to $\Aff (R,\x )$ in the case of (ii). 
Therefore, 
Proposition~\ref{prop:sankakuka} 
is proved by the following lemma.

\begin{lem}\label{lem:exp T}
For each $D\in \lnd _R\Rx $, 
the following assertions hold, 
where $n\in \N $ may be arbitrary. 

\noindent{\rm (i)} 
If $\exp D$ belongs to $J(R;x_1,\ldots ,x_n)$, 
then $D$ is triangular. 

\noindent{\rm (ii)}
If $\exp D$ belongs to $\Aff (R,\x )$, 
then $D$ is affine. 
\end{lem}
\begin{proof}
Consider the power series 
$$
p(z)=\exp z-1=\sum _{i=1}^{\infty }\frac{z^i}{i!}, 
$$ 
where $z$ is a variable. 
Since $z=\log \bigl((\exp z-1)+1\bigr)$, 
we get the identity 
\begin{equation}\label{eq:log exp}
z=\sum _{i=1}^{\infty }(-1)^{i+1}\frac{p(z)^i}{i}
\end{equation}
in the formal power series ring $\Q [[z]]$, 
where the right-hand side of 
(\ref{eq:log exp}) makes sense 
because $p(z)$ is an element of 
$z\Q [[z]]$. 
Now, 
let $\End _R\Rx $ be the ring of the 
$R$-linear endomorphisms of $\Rx $, 
$\Q [D]$ the $\Q $-subalgebra 
of $\End _R\Rx $ generated by $D$, 
and $\Q [[D]]$ the completion of $\Q [D]$ 
by the maximal ideal of $\Q [D]$ 
generated by $D$. 
Then, 
$\Q [[D]]$ is contained in $\End _R\Rx $ 
by the assumption that $D$ is locally nilpotent. 
Since $p(D)$ is a multiple of $D$, 
we see that $p(D)$ is locally nilpotent, 
i.e., 
for each $f\in \Rx $, there exists $l\in \N $ 
such that $p(D)^l(f)=0$.

(i) 
Since $\exp D$ belongs to $J(R;x_1,\ldots ,x_n)$ 
by assumption, 
we may write 
$$(\exp D)(x_i)=a_ix_i+g_i$$ for $i=1,\ldots ,n$, 
where $a_i\in R^{\times }$ and $g_i\in R[x_1,\ldots ,x_{i-1}]$. 
Then, it is easy to check that 
$(\exp D)^j(x_i)$ has the form $a_i^jx_i+g_{i,j}$ 
for some $g_{i,j}\in R[x_1,\ldots ,x_{i-1}]$ 
for each $j\geq 0$ by induction on $j$. 
Hence, 
we get 
\begin{align*}
&p(D)^l(x_i)=(\exp D-\id _{\Rx })^l(x_i) \\
&\quad =\sum _{j=0}^l(-1)^{l-j}\binom{l}{j}(\exp D)^{j}(x_i) 
=\sum _{j=0}^l(-1)^{l-j}\binom{l}{j}a_i^{j}x_i+h_{i,l}
=(a_i-1)^lx_i+h_{i,l} 
\end{align*}
for each $l\in \N $, 
where $h_{i,l}\in R[x_1,\ldots ,x_{i-1}]$. 
Since $p(D)$ is locally nilpotent, 
it follows that $a_i=1$ for each $i$. 
Hence, 
$p(D)^l(x_i)=h_{i,l}$ 
belongs to $R[x_1,\ldots ,x_{i-1}]$. 
In view of (\ref{eq:log exp}) with $z$ replaced by $D$, 
we know that $D(x_i)$ belongs to 
$R[x_1,\ldots ,x_{i-1}]$ for each $i$. 
Therefore, 
$D$ is triangular. 

(ii) Since $\exp D$ is affine by assumption, 
$p(D)(x_i)=(\exp D)(x_i)-x_i$ 
has total degree at most one 
for $i=1,\ldots ,n$. 
Hence, 
we have $\deg p(D)^l(x_i)\leq 1$ for each $l\in \N $. 
In view of (\ref{eq:log exp}) with $z$ replaced by $D$, 
we get $\deg D(x_i)\leq 1$ for each $i$. 
Therefore, 
$D$ is affine. 
\end{proof}

We mention that the method used in the proof of this lemma 
is similar to that used in 
van den Essen~\cite[Proposition 2.1.3]{Essen} 
and Nowicki~\cite[Proposition 6.1.4 (6)]{Now}.

Now, 
let us complete the proof of Theorem~\ref{thm:triangularizability}. 
If $D=0$, 
then the theorem is clear. 
Assume that $D\neq 0$. 
Let $\tilde{D}$ be the natural extension of $D$ 
to an element of $\lnd _K\Kx $. 
Since $R$ is a $\Q $-domain, 
$K$ is of characteristic zero. 
Hence, 
there exists a coordinate $f$ of $\Kx $ over $K$ 
such that $\ker \tilde{D}=K[f]$ 
by Theorem~\ref{thm:Rentschler original}. 
Multiplying by an element of $K^{\times }$, 
we may assume that $f$ belongs to $\Rx $. 
Take $\tau \in \T (R,\x)$ such that 
$|\w (\tau (f))|$ is minimal. 
Then, 
$f':=\tau (f)$ remains a coordinate of $\Kx $ over $K$, 
and is tamely reduced over $R$. 
Without loss of generality, 
we may assume that 
$\deg _{x_1}f'\geq \deg _{x_2}f'$ 
by changing $\tau $ if necessary. 
Put $D':=\tau \circ D\circ \tau ^{-1}$. 
Then, 
$D'$ kills $f'$. 
Therefore, 
by applying 
Proposition~\ref{prop:sankakuka} with $S=K$, 
we know that $D'$ is triangular 
if $\deg _{x_1}f'>\deg _{x_2}f'$, 
and affine 
if $\deg _{x_1}f'=\deg _{x_2}f'$. 
This proves the first part of 
Theorem~\ref{thm:triangularizability}.

To prove the last part, 
assume that $V(R)=K^{\times }$. 
Then, 
we have $\deg _{x_1}f'\neq \deg _{x_2}f'$ 
thanks to Lemma~\ref{lem:R/S}, 
because $f'$ is tamely reduced over $R$. 
Since $\deg _{x_1}f'\geq \deg _{x_2}f'$ by assumption, 
it follows that 
$\deg _{x_1}f'>\deg _{x_2}f'$. 
Therefore, 
$D'$ is triangular 
by Proposition~\ref{prop:sankakuka} (i), 
proving the last part. 
This completes the proof of 
Theorem~\ref{thm:triangularizability}.

\section{Nagata type wild automorphisms}
\setcounter{equation}{0}
\label{sect:triangular}

In this section, 
we study Problem~\ref{prob:triangular}. 
First, 
assume that $n=2$, 
and let $D$ be a triangular derivation of $\Rx $ over $R$. 
We consider when $\exp fD$ 
belongs to $\T (R,\x )$ 
for $f\in \ker D$. 
If $f$ belongs to $R$, 
then $fD$ is triangular. 
Hence, $\exp fD$ belongs to $\T (R,\x )$. 
If $D(x_i)=0$ for some $i\in \{ 1,2\} $, 
then $\exp fD$ belongs to $\Aut (\Rx /R[x_i])$, 
and hence belongs to $\T (R,\x )$ for any $f\in \ker D$.

Assume that $f$ does not belong to $R$, 
and $D(x_i)\neq 0$ for $i=1,2$. 
Write 
\begin{equation}\label{eq:triangular/R}
D(x_1)=a\quad\text{ and }\quad 
D(x_2)=\sum _{i=0}^lb_ix_1^i,
\end{equation}
where $l\in \Zn $, 
and $a,b_0,\ldots ,b_l\in R$ 
with $a\neq 0$ and $b_l\neq 0$. 
Then, 
\begin{equation}\label{eq:triangular kernel h/R}
h:=ax_2-\sum _{i=0}^l\frac{b_i}{i+1}x_1^{i+1}
\end{equation}
is a coordinate of $\Kx $ over $K$ 
such that $D(h)=0$. 
By the following theorem, 
it follows that $\ker D$ is contained in $K[h]$.

\begin{thm}[Rentschler]\label{thm:Rentschler}
Assume that $n=2$. 
If $D(f)=0$ holds for $D\in \lnd _k\kx \sm \zs $ 
and a coordinate $f$ of $\kx $ over $k$, 
then we have $\ker D=k[f]$. 
\end{thm}

We remark that this theorem is a consequence 
of Theorem~\ref{thm:Rentschler original}. 
Indeed, 
we have $\ker D=k[g]$ 
for some coordinate $g$ of $\kx $ over $k$ 
by Theorem~\ref{thm:Rentschler original}. 
Hence, if $f$ is a coordinate of $\kx $ over $k$ 
belonging to $\ker D$, 
then $f$ is a linear polynomial in $g$ over $k$. 
Therefore, 
we get $k[f]=k[g]=\ker D$.

Since $f$ is an element of $\ker D$, 
it follows that $f$ belongs to $K[h]$. 
We denote by $\deg _hf$ the degree of $f$ 
as a polynomial in $h$ over $K$. 
Then, it holds that $\deg _hf=\deg _{x_2}f$, 
since $\deg _{x_2}h=1$.

We define  
\begin{equation}\label{eq:II'}
I=\{ i\in \{ 0,\ldots ,l\} \mid b_i\not\in aR\} . 
\end{equation}
Then, 
we have the following theorem. 
This theorem gives a complete solution to 
Problem~\ref{prob:triangular} 
in the case of $n=2$.

\begin{thm}\label{thm:hD}
Let $D$ be as above, 
and $f$ an element of $\ker D\sm R$. 
Then, 
$\exp fD$ belongs to $\T (R,\x )$ 
if and only if one of the following 
conditions holds$:$ 

\smallskip 

{\rm (1)} $I=\emptyset $. \quad 
{\rm (2)} $I=\zs $, 
and $b_0/a$ belongs to $V(R)$ 
or $\deg _{x_2}f=1$. 

\smallskip 

\noindent
In particular, if $V(R)=K^{\times }$, 
then $\exp fD$ belongs to 
$\T (R,\x)$ 
if and only if 
$b_i$ belongs to $aR$ for $i=1,\ldots ,l$. 
\end{thm}
\begin{proof}
Define 
$\tau \in \Aut (\Rx /R[x_1])$ by 
\begin{equation}\label{eq:tau:triangular/R}
\tau (x_2)=x_2+\sum _{i\in I'}\frac{b_i}{(i+1)a}x_1^{i+1}, 
\end{equation}
where $I':=\{ 0,\ldots ,l\} \sm I$. 
Then, 
we have 
\begin{equation}\label{eq:triangular kernel h0/R}
h_0:=\tau (h)=
ax_2-\sum _{i\in I}\frac{b_i}{(i+1)}x_1^{i+1}. 
\end{equation}
Put $D_0=\tau \circ D\circ \tau ^{-1}$ 
and $f_0=\tau (f)$. 
Then, 
$\phi :=\exp fD$ is tame if and only if 
$\phi _0:=\exp f_0D_0=\tau \circ \phi \circ \tau ^{-1}$ 
is tame. 
Note that  $D_0(h_0)=\tau (D(h))=0$ and 
$D_0(x_1)=\tau (D(\tau ^{-1}(x_1)))=\tau (D(x_1))=a$. 
Hence, 
we have $\phi _0(h_0)=h_0$ 
and 
\begin{equation}\label{eq:ah'}
\phi _0(x_1)=(\exp f_0D_0)(x_1)=x_1+f_0D_0(x_1)=x_1+af_0. 
\end{equation}

Now, 
assume that $I\neq \emptyset $ and $I\neq \zs $, 
i.e., $t:=\max I\geq 1$. 
Then, we prove that $\phi $ is wild. 
Suppose to the contrary that $\phi $ is tame. 
Then, $\phi _0$ is tame. 
Since $\phi _0(h_0)=h_0$, 
it follows that $\phi _0$ belongs to $H(h_0)$. 
From (\ref{eq:triangular kernel h0/R}), 
we see that $\deg _{x_1}h_0=t+1>1=\deg _{x_2}h_0$, 
and 
\begin{equation}\label{eq:w(h_0)}
h_0^{\w (h_0)}=a(x_2-bx_1^{t+1}), 
\quad \text{where}\quad b:=\frac{b_t}{(t+1)a}. 
\end{equation}
Since $t$ is an element of $I$, 
we know that $b_t$ does not belong to $aR$, 
and so $b$ does not belong to $R$. 
Hence, 
$h_0$ is tamely reduced over $R$ by 
Proposition~\ref{prop:reduced coordinate} (ii). 
By Theorem~\ref{thm:reduced coordinate} (ii), 
it follows that $H(h_0)$ is contained in $J(R;x_1,x_2)$. 
Thus, 
$\phi _0$ belongs to $J(R;x_1,x_2)$. 
Because of 
(\ref{eq:ah'}), 
this implies that $af_0$ belongs to $R[x_1]$. 
Since 
$D_0(f_0)=0$ and $D_0(x_1)=a\neq 0$, 
we conclude that $f_0$ belongs to $R$. 
Thus, 
$f$ belongs to $R$, a contradiction. 
Therefore, 
$\phi $ is wild 
if $I\neq \emptyset $ and $I\neq \zs $.

Next, assume that $I=\emptyset $. 
Then, we have $h_0=ax_2$. 
Since $\phi _0(h_0)=h_0$, 
it follows that $\phi _0$ 
belongs to $\Aut (\Rx /R[x_2])$. 
Hence, $\phi _0$ is elementary. 
Therefore, 
$\phi $ is tame.

Finally, assume that $I=\zs $. 
Then, 
we have $h_0=ax_2-b_0x_1$ with $b_0\neq 0$. 
Hence, 
the total degree of $f_0$ is equal to 
$\deg _{h_0}f_0=\deg _hf=\deg _{x_2}f$. 
Since $h_0=\phi _0(h_0)$, 
we have 
$$
ax_2-b_0x_1=h_0=\phi _0(h_0)=\phi _0(ax_2-b_0x_1)
=a\phi _0(x_2)-b_0(x_1+af_0) 
$$
in view of (\ref{eq:ah'}). 
This gives that $\phi _0(x_2)=x_2+b_0f_0$. 
We prove that $\phi$ is tame if and only if 
$\deg _{x_2}f=1$ or $b_0/a$ belongs to $V(R)$. 
First, 
assume that $\deg _{x_2}f=1$. 
Then, 
we have $\deg f_0=1$ as mentioned. 
Hence, 
we get $\deg \phi _0(x_i)=1$ for $i=1,2$. 
Thus, 
$\phi _0$ is affine. 
Therefore, $\phi $ is tame. 
Next, 
assume that $b_0/a$ belongs to $V(R)$. 
Then, there exists 
$\sigma \in \Aff (R,\x )$ 
such that $\sigma (h_0)=cx_2$ for some $c\in R\sm \zs $. 
Put $\phi _1=\sigma \circ \phi _0\circ \sigma ^{-1}$. 
Then, 
we have 
$\phi _1(cx_2)=\sigma (\phi _0(h_0))=\sigma (h_0)=cx_2$. 
Hence, 
$\phi _1$ belongs to $\Aut (\Rx /R[x_2])$. 
Thus, $\phi _1$ is elementary. 
Therefore, $\phi $ is tame. 
Finally, 
assume that $\deg _{x_2}f\neq 1$ and 
$b_0/a$ does not belong to $V(R)$. 
Then, 
we have $\deg _{x_2}f\geq 2$, 
since $f$ is not an element of $R$ by assumption. 
Hence, we get $\deg f_0\geq 2$, 
and so $\phi _0$ is not affine 
in view of (\ref{eq:ah'}). 
Since 
$b_0/a$ does not belong to $V(R)$, 
we know by Proposition~\ref{prop:reduced coordinate} (i) 
that $h_0=ax_2-b_0x_1$ is tamely reduced over $R$. 
Hence, 
$H(h_0)$ is contained in $\Aff (R,\x )$ 
by Theorem~\ref{thm:reduced coordinate} (i). 
Thus, 
$\phi _0$ does not belong to $H(h_0)$. 
Since $\phi _0(h_0)=h_0$, 
this implies that $\phi _0$ is wild. 
Therefore, $\phi $ is wild. 
This proves that 
$\phi $ is tame if and only if 
(1) or (2) holds.

To prove the last part, 
assume that $V(R)=K^{\times }$. 
Then, we claim that 
(2) is equivalent to $I=\zs $. 
In fact, if $I=\zs $, 
then we have $b_0\neq 0$, 
and hence $b_0/a$ always belongs to $K^{\times }=V(R)$. 
Therefore, 
the first part of the theorem implies 
that $\phi$ is tame if and only if 
$I=\emptyset $ or $I=\zs $. 
By the definition of $I$, 
this condition is equivalent to the condition that 
$b_i$ belongs to $aR$ for $i=1,\ldots ,l$. 
\end{proof}

Next, assume that $n=3$, 
and let $D$ be a triangular derivation of $\kx $ over $k$. 
We consider when $\exp fD$ belongs to $\T (k,\x )$ 
for $f\in \ker D\sm k$.

(i) Assume that 
$D(x_1)=0$ and $f$ belongs to $k[x_1]$. 
Then, $fD$ is triangular. 
Hence, 
$\exp fD$ belongs to $\T (k,\x )$.

(ii) Assume that $D(x_i)=D(x_j)=0$ 
for some $1\leq i<j\leq 3$. 
Then, 
$\exp fD$ belongs to $\Aut (\kx /k[x_i,x_j])$. 
Therefore, 
$\exp fD$ belongs to $\T (k,\x )$.

(iii) Assume that $D(x_1)\neq 0$. 
Then, $D(x_1)$ belongs to $k^{\times }$ 
by the triangularity of $D$. 
Hence, 
$s:=x_1/D(x_1)$ is an element of $k[x_1]$. 
Define 
$\tau \in J(k[x_1];x_2,x_3)$ 
by 
$$
\tau (x_i)=
\sum _{l\geq 0}\frac{D^l(x_i)}{l!}(-s)^l
$$ 
for $i=2,3$. 
Then, 
we have $D(\tau (x_i))=0$ for $i=2,3$. 
In fact, 
since $D(s)=1$, 
it follows that 
$$
D\left( 
\frac{D^l(x_i)}{l!}(-s)^l
\right) =g_{l+1}-g_{l}
\text{ for each }l\geq 0,
\text{ where }
g_{l}:=l\frac{D^{l}(x_i)}{l!}(-s)^{l-1}. 
$$
Hence, 
we know that 
$D_0:=\tau ^{-1}\circ fD\circ \tau $ 
kills $x_i$ for $i=2,3$. 
Thus, 
$\exp D_0$ belongs to $\Aut (\kx /k[x_2,x_3])$. 
Therefore, 
$\exp fD$ belongs to $\T (k,\x )$.

In the rest of the case, 
tameness of $\exp fD$ is determined 
by the following theorem.

\begin{thm}\label{thm:triangular3}
Assume that $n=3$. 
Let $D$ be a triangular derivation of $\kx $ over $k$ 
such that $D(x_1)=0$ and $D(x_i)\neq 0$ for $i=2,3$, 
and $f$ an element of $\ker D\sm k[x_1]$. 
Then, 
$\exp fD$ belongs to $\T (k,\x )$ 
if and only if $\partial D(x_3)/\partial x_2$ 
belongs to $D(x_2)k[x_1,x_2]$. 
\end{thm}

By the triangularity of $D$, 
we may regard $D(x_3)$ as a polynomial in $x_2$ over $k[x_1]$. 
Then, 
$\partial D(x_3)/\partial x_2$ 
belongs to $D(x_2)k[x_1,x_2]$ 
if and only if the coefficient of 
$x_2^i$ in $D(x_3)$ 
belongs to $D(x_2)k[x_1]$ 
for each $i\geq 1$.

Theorem~\ref{thm:triangular3} 
is derived from Theorem~\ref{thm:hD} 
with the aid of Theorem~\ref{thm:tame23} as follows. 
Let $R=k[x_1]$, 
$y_i=x_{i+1}$ for $i=1,2$ 
and $\y =\{ y_1,y_2\} $. 
Then, 
we may regard $D$ 
as a triangular derivation of $\Ry $ over $R$. 
By assumption, 
$f$ does not belong to $R=k[x_1]$, 
and $D(y_i)=D(x_{i+1})\neq 0$ for $i=1,2$. 
Hence, 
$D$ fulfills the assumption of Theorem~\ref{thm:hD}. 
Since $k[x_1]$ is a PID, 
we have $V(k[x_1])=k(x_1)^{\times }$. 
Thus, 
we know by the last part of Theorem~\ref{thm:hD} 
that $\phi :=\exp fD$ 
belongs to $\T (R,\y )$ if and only if 
the coefficient of $y_1^i=x_2^i$ in $D(y_2)=D(x_3)$ 
belongs to $D(y_1)R=D(x_2)k[x_1]$ 
for each $i\geq 1$. 
This condition is equivalent to the condition that 
$\partial D(x_3)/\partial x_2$ 
belongs to $D(x_2)k[x_1,x_2]$ as remarked. 
Thanks to Theorem~\ref{thm:tame23}, 
$\phi $ belongs to $\T (R,\y )=\T (k[x_1],\{ x_2,x_3\} )$ 
if and only if $\phi $ belongs to $\T (k,\x )$, 
since $\phi $ is an element of $\Aut (\Kx /k[x_1])$. 
Therefore, 
$\phi $ belongs to $\T (k,\x )$ 
if and only if $\partial D(x_3)/\partial x_2$ 
belongs to $D(x_2)k[x_1,x_2]$. 
This proves Theorem~\ref{thm:triangular3}.

As an application of 
Theorem~\ref{thm:triangular3}, 
we describe all the wild automorphisms 
of $\kx $ over $k$ of the form 
$\exp fD$ for some triangular derivation $D$ of $\kx $ over $k$ 
and $f\in \ker D$. 
Let $\Lambda $ be the set of 
$(g,h)\in (k[x_1]\sm \zs )\times (x_2k[x_1,x_2]\sm \zs )$ 
as follows: 

\smallskip 

\noindent (i) 
$g$ and $h$ have no common factor; 

\noindent (ii) 
$g$ is a monic polynomial in $x_1$; 

\noindent (iii) 
$\partial ^2h/\partial x_2^2$ 
does not belong to $gk[x_1,x_2]$. 

\smallskip 

\noindent
Here, 
by ``no common factor", 
we mean ``no non-constant common factor". 
For each $(g,h)\in \Lambda $, 
we define a triangular derivation $T_{g,h}$ 
of $\kx $ over $k$ by 
\begin{equation}\label{eq:T_{g,h}}
T_{g,h}(x_1)=0,\quad 
T_{g,h}(x_2)=g,\quad 
T_{g,h}(x_3)=-\frac{\partial h}{\partial x_2}. 
\end{equation}
Then, 
we have 
$$
T_{g,h}(gx_3+h)
=gT_{g,h}(x_3)+
\frac{\partial h}{\partial x_2}T_{g,h}(x_2)=0. 
$$
Hence, $k[x_1,gx_3+h]$ is contained in $\ker T_{g,h}$. 
Take any $f\in k[x_1,gx_3+h]\sm k[x_1]$. 
Then, 
it follows from Theorem~\ref{thm:triangular3} that 
$$
\Phi _{g,h}^f:=\exp fT_{g,h}
$$ 
does not belong to $\T (k,\x )$, 
since 
$\partial T_{g,h}(x_3)/\partial x_2=-\partial ^2h/\partial x_2^2$ 
does not belong to $T_{g,h}(x_2)k[x_1,x_2]=gk[x_1,x_2]$ by (iii).

\begin{prop}\label{prop:triangular family}
Assume that $n=3$. 
Let $D$ be a triangular derivation of $\kx $ over $k$ 
and $f\in \ker D$ such that $\exp fD$ 
does not belong to $\T (k,\x )$. 
Then, 
there exist unique $(g,h)\in \Lambda $ and 
$f_0\in k[x_1,gx_3+h]\sm k[x_1]$ 
such that $fD=f_0T_{g,h}$. 
\end{prop}
\begin{proof}
First, 
we prove the existence of $g$, $h$ and $f_0$. 
Due to Theorem~\ref{thm:triangular3} and 
the discussion before this theorem, 
we know that 
$D(x_1)=0$, 
$f$ does not belong to $k[x_1]$, 
and $\partial D(x_3)/\partial x_2$ 
does not belong to $D(x_2)k[x_1,x_2]$. 
Moreover, 
$D(x_2)$ and $D(x_3)$ are nonzero elements of 
$k[x_1]$ and $k[x_1,x_2]$, 
respectively. 
Hence, 
we can construct $h_1\in x_2k[x_1,x_2]$ 
such that $\partial h_1/\partial x_2=-D(x_3)$ 
by integrating $-D(x_3)$ in $x_2$. 
Take the highest degree polynomial 
$f_1\in k[x_1]\sm \zs $ such that 
$f_1$ divides both $D(x_2)$ and $h_1$, 
and that $g:=D(x_2)/f_1$ is a monic polynomial. 
Then, 
$h:=h_1/f_1$ belongs to $x_2k[x_1,x_2]$, 
and $g$ and $h$ have no common factor 
by the maximality of $\deg _{x_1}f_1$. 
By the definition of $h$, $f_1$ and $h_1$, 
we have 
$$
f_1\frac{\partial ^2h}{\partial x_2^2}
=\frac{\partial ^2(f_1h)}{\partial x_2^2}
=\frac{\partial ^2h_1}{\partial x_2^2}
=-\frac{\partial D(x_3)}{\partial x_2}. 
$$ 
Since this polynomial 
does not belong to $D(x_2)k[x_1,x_2]=f_1gk[x_1,x_2]$, 
it follows that 
$\partial ^2h/\partial x_2^2$ 
does not belong to $gk[x_1,x_2]$. 
Thus, 
$(g,h)$ belongs to $\Lambda $. 
Moreover, 
we have $D=f_1T_{g,h}$,	
since 
$$
D(x_1)=0,\quad 
D(x_2)=f_1g\quad\text{and}\quad
D(x_3)=-\frac{\partial h_1}{\partial x_2}
=-f_1\frac{\partial h}{\partial x_2}. 
$$
Set $f_0=ff_1$. 
Then, 
we get $fD=f_0T_{g,h}$. 
Furthermore, 
$f_0$ belongs to $\ker D\sm k[x_1]$, 
since $f$ and $f_1$ belong to $k[x_1]\sm \zs $ 
and $\ker D\sm k[x_1]$, respectively. 
Because $\ker D=\ker T_{g,h}$, 
it remains only to show that 
$\ker T_{g,h}=k[x_1,gx_3+h]$. 
We prove this by means of the 
``kernel criterion" (cf.~\cite[Proposition~5.12]{Fbook}). 
Since $gx_3+h$ is a coordinate of 
$k(x_1)[x_2,x_3]$ over $k(x_1)$ 
such that $T_{g,h}(gx_3+h)=0$, 
we know by Theorem~\ref{thm:Rentschler} that 
the kernel of the extension of 
$T_{g,h}$ to $k(x_1)[x_2,x_3]$ 
is generated by $gx_3+h$ over $k(x_1)$. 
Hence, 
$\ker T_{g,h}$ is contained in $k(x_1,gx_3+h)$. 
Since $g$ and $h$ have no common factor, 
and are elements of $k[x_1]\sm \zs $ 
and $x_2k[x_1,x_2]\sm \zs $, respectively, 
it follows that $g$ and $\partial h/\partial x_2$ 
have no common factor. 
Hence, 
$T_{g,h}(x_2)$ and $T_{g,h}(x_3)$ 
have no common factor. 
This implies that 
$T_{g,h}$ is irreducible, 
i.e., 
$T_{g,h}(\kx )$ is contained in 
no proper principal ideal of $\kx $. 
In this situation, 
we may conclude that 
$\ker T_{g,h}=k[x_1,gx_3+h]$ 
by virtue of the ``kernel criterion". 
Therefore, 
$f_0$ belongs to $k[x_1,gx_3+h]\sm k[x_1]$. 
This proves the existence of $g$, $h$ and $f_0$.

To prove the uniqueness, 
assume that $f_1T_{g_1,h_1}=f_2T_{g_2,h_2}$ 
for some $(g_i,h_i)\in \Lambda $ 
and $f_i\in k[x_1,g_ix_3+h_i]\sm k[x_1]$ 
for $i=1,2$. 
Then, 
we have 
$$
f_1g_1
=f_2g_2\quad \text{and}\quad 
f_1\frac{\partial h_1}{\partial x_2}
=f_2\frac{\partial h_2}{\partial x_2}. 
$$
Since 
$g_1$ and $g_2$ are elements of $k[x_1]$, 
and $f_1$ and $f_2$ are nonzero, 
this gives that 
$$
\frac{\partial g_1h_2}{\partial x_2}
=g_1\frac{\partial h_2}{\partial x_2}
=(g_1f_2^{-1})f_2\frac{\partial h_2}{\partial x_2}
=(g_2f_1^{-1})f_1\frac{\partial h_1}{\partial x_2}
=g_2\frac{\partial h_1}{\partial x_2}
=\frac{\partial g_2h_1}{\partial x_2}. 
$$
Because $g_1h_2$ and $g_2h_1$ belong to $x_2k[x_1,x_2]$, 
it follows that $g_1h_2=g_2h_1$. 
Since $g_i$ and $h_i$ have no common factor 
for $i=1,2$, 
we may write $g_1=cg_2$ and $h_1=ch_2$, 
where $c\in k^{\times }$. 
Then, 
we have $c=1$ by the assumption that 
$g_1$ and $g_2$ are monic polynomials. 
Thus, we get $g_1=g_2$ and $h_1=h_2$, 
and therefore $f_1=f_2g_2/g_1=f_2$. 
This proves the uniqueness of $g$, $h$ and $f_0$. 
\end{proof}

\section{Affine locally nilpotent derivations}
\setcounter{equation}{0}
\label{sect:affine lnd}

Similarly to Problem~\ref{prob:triangular}, 
we can consider the following problem.

\begin{problem}\label{prob:affine}\rm
Assume that $D\in \lnd _R\Rx $ is affine. 
When does $\exp fD$ belong to $\T (R,\x )$ 
for $f\in \ker D\sm R$? 
\end{problem}

We remark that this problem 
is reduced to Problem~\ref{prob:triangular} 
when $n=2$ and $V(R)=K^{\times }$. 
Indeed, 
if $D\in \lnd _R\Rx $ is affine, 
then $\exp D$ belongs to $\T (R,\x )$. 
Hence, 
$D$ is tamely triangularizable 
due to the last part of 
Theorem~\ref{thm:triangularizability}.

The following is the main result of this section.

\begin{thm}\label{thm:affine lnd}
Assume that $n=2$, 
and that $D\in \lnd _R\Rx $ is affine, 
and $\psi ^{-1}\circ D\circ \psi $ is not triangular 
for any $\psi \in \Aff (R,\x )$. 
Then, 
$\exp fD$ belongs to $\T (R,\x )$ 
if and only if $f$ belongs to $R$ 
for $f\in \ker D$. 
\end{thm}

The following is a key lemma.

\begin{lem}\label{lem:affine lnd}
Let $D$ be as in Theorem~$\ref{thm:affine lnd}$. 
Then, the following assertions hold$:$ 

\noindent{\rm (i)} 
We have $\deg _{x_i}D(x_j)=1$ for every $i,j\in \{ 1,2\} $. 

\noindent{\rm (ii)} 
There exists $h\in \Rx $ satisfying the following conditions$:$ 

\noindent{\rm (1)} $D(h)=0$. 

\noindent{\rm (2)} $1\leq \deg _{x_1}h=\deg _{x_2}h\leq 2$. 

\noindent{\rm (3)} $h$ is a coordinate of $\Kx $ over $K$. 

\noindent{\rm (4)} $h$ is tamely reduced over $R$. 
\end{lem}

By assuming this lemma, 
we can prove Theorem~\ref{thm:affine lnd} as follows. 
The ``if" part of the theorem is clear. 
We prove the ``only if" part. 
Assume that $\phi :=\exp fD$ belongs to $\T (R,\x )$ 
for some $f\in \ker D$. 
Take $h\in \Rx $ as in Lemma~\ref{lem:affine lnd} (ii). 
Then, we have $\phi (h)=h$ by (1). 
Hence, 
$\phi $ belongs to $H(h)$, 
since $\phi $ belongs to $\T (R,\x )$ by assumption. 
Because $h$ satisfies (2), (3) and (4), 
we know by Theorem~\ref{thm:reduced coordinate} (i) that 
$H(h)$ is contained in $\Aff (R,\x )$. 
Thus, 
$\phi $ belongs to $\Aff (R,\x )$. 
By Lemma~\ref{lem:exp T} (ii), 
it follows that $fD$ is affine. 
By Lemma~\ref{lem:affine lnd} (i), 
this implies that $f$ belongs to $R$. 
Therefore, 
the ``only if" part of Theorem~\ref{thm:affine lnd} 
follows from Lemma~\ref{lem:affine lnd}.

Let us prove Lemma~\ref{lem:affine lnd}. 
Write 
$$
(D(x_1),D(x_2))=(x_1,x_2)A+(b_1,b_2), 
$$ 
where $A\in M(2,R)$ and $b_1,b_2\in R$. 
Then, $A$ is a nilpotent matrix. 
Since $D$ is not triangular, 
we have $A\neq O$. 
Hence, we know from linear algebra that 
$$
P^{-1}AP=
\left(\begin{array}{@{\,}cc@{\,}}
0 & 1 \\ 0 & 0
\end{array}\right) 
$$
for some $P\in {\it GL}(2,K)$. 
From this, we see that $A$ has the form 
$$
A=t\left(\begin{array}{@{\,}cc@{\,}}
\alpha _1\alpha _2& -\alpha _1^2 \\
\alpha _2^2 & -\alpha _1\alpha _2
\end{array}\right) 
$$
for some $t\in K^{\times }$ and $\alpha _1,\alpha _2\in K$. 
Then, 
$\alpha _1$ is nonzero, 
for otherwise $D$ is triangular 
if $x_1$ and $x_2$ are interchanged. 
Similarly, 
$\alpha _2$ is also nonzero. 
Hence, 
we have $\deg _{x_i}D(x_j)=1$ for every $i,j\in \{ 1,2\} $. 
This proves Lemma~\ref{lem:affine lnd} (i). 
We show that $\alpha _1/\alpha _2$ does not belong to $V(R)$. 
Suppose to the contrary that 
$\alpha _1/\alpha _2$ belongs to $V(R)$. 
Then, 
there exist $\alpha \in K^{\times }$ 
and $\beta _1,\beta _2\in R$ such that 
$(\alpha \alpha _1)\beta _2-(\alpha \alpha _2)\beta _1=1$, 
and $\alpha \alpha _1$ and $\alpha \alpha _2$ belong to $R$. 
Put $p=\alpha _1x_1+\alpha _2x_2$, 
and define $\psi \in \Aff (R,\x )$ by 
$$
\psi (x_1)=\alpha p,\quad 
\psi (x_2)=\beta _1x_1+\beta _2x_2. 
$$
Then, 
we have $\delta :=D(p)=\alpha _1b_1+\alpha _2b_2$. 
Hence, 
$$
(\psi ^{-1}\circ D\circ \psi )(x_1)
=\psi ^{-1}(D(\alpha p))=\alpha \delta 
$$
is a constant. 
Since $D':=\psi ^{-1}\circ D\circ \psi $ 
is locally nilpotent, 
this implies that $D'(x_2)$ belongs to $R[x_1]$. 
Thus, 
$D'$ is triangular, 
a contradiction. 
Therefore, 
$\alpha _1/\alpha _2$ does not belong to $V(R)$. 
When $\delta =0$, 
we define $h=t\alpha _2p$. 
Then, 
$h$ belongs to $\Rx $, 
since $t\alpha _2\alpha _i$ is an entry of 
$A$ for $i=1,2$. 
Since $\delta =0$, 
we get $D(h)=t\alpha _2\delta =0$, 
proving (1). 
Since $\alpha _1$ and $\alpha _2$ are nonzero, 
we see that $h$ is a coordinate of $\Kx $ over $K$ such that 
$\deg _{x_1}h=\deg _{x_2}h=1$, 
proving (3) and (2). 
Since $\alpha _1/\alpha _2$ does not belong to $V(R)$, 
we know by 
Proposition~\ref{prop:reduced coordinate} (ii) 
that $h$ is tamely reduced over $R$, 
proving (4). 
Thus, 
$h$ satisfies all the conditions of 
Lemma~\ref{lem:affine lnd} (ii). 
Next, assume that $\delta \neq 0$. 
We define 
$$
h=t\alpha _1\delta \left( 
x_1-\frac{b_1}{\delta }p-\frac{t\alpha _2}{2\delta }p^2
\right) . 
$$
Then, $h$ belongs to $\Rx $, 
since $t\alpha _1\delta $, $t\alpha _1p$ 
and $t^2\alpha _1\alpha _2p^2$ belong to $\Rx $. 
Since 
$$
D(h)=t\alpha _1\delta \left( 
(t\alpha _2p+b_1)-\frac{b_1}{\delta }\delta 
-\frac{t\alpha _2}{2\delta }(2p)\delta 
\right) =0, 
$$
we get (1). 
It is easy to check that 
$\deg _{x_i}h=\deg _{x_i}p^2=2$ for $i=1,2$, 
proving (2). 
Since $K[h,p]=K[x_1,p]=\Kx $, 
we see that $h$ is a coordinate of $\Kx $ over $K$, 
proving (3). 
Since $h^{\w (h)}$ is equal to $p^2$ 
up to a nonzero constant multiple, 
and since $\alpha _1/\alpha _2$ does not belong to $V(R)$, 
we know that $h$ is tamely reduced over $R$ by 
Proposition~\ref{prop:reduced coordinate} (ii). 
Hence, $h$ satisfies (4). 
Therefore, 
$h$ satisfies all the conditions of 
Lemma~\ref{lem:affine lnd} (ii). 
This completes the proof of 
Lemma~\ref{lem:affine lnd}, 
and thereby completing the proof of 
Theorem~\ref{thm:affine lnd}.

\chapter{Tame intersection theorem}
\label{chap:atit}

\section{Main result}
\setcounter{equation}{0}
\label{sect:invariant}

Assume that $n=2$. 
Let $S$ be an over domain of $R$, 
and $f\in \Rx $ a coordinate of $\Sx $ over $S$ 
which is tamely reduced over $R$. 
As we have seen in the previous chapter, 
the tame intersection
$$
H(f)=\Aut (\Rx /R[f])\cap \T (R,\x ) 
$$
plays important roles in proving the wildness of automorphisms. 
This motivates us to study $H(f)$ in detail.

Throughout this chapter, 
we assume that $R$ contains $\Z $ 
unless otherwise stated. 
Under this assumption, 
we investigate coordinates $f\in \Rx $ of $\Sx $ over $S$ 
which are tamely reduced over $R$, and for which 
$H(f)\neq \{ \id _{\Rx } \} $. 
Our goal is to give a complete classification of such $f$'s, 
and to describe the concrete structures of $H(f)$'s.

We mention that 
$\Aut (\Rx /R[f])$ itself 
is an infinite group for the following reason. 
Recall that, 
for each $D\in \lnd _R\Rx $, 
we mean by $\exp D$ the exponential automorphism for 
the natural extension of $D$ to $\bar{R}:=\Q \otimes _{\Z }\Rx $. 
Since $D$ is locally nilpotent, 
we may find $m\in \N $ such that $D^m(x_i)=0$ for $i=1,2$. 
Then, 
$\phi :=\exp m!D$ 
induces an element of $\Aut (\Rx /R)$, 
since $\phi (x_i)$ 
and $\phi ^{-1}(x_i)=(\exp -m!D)(x_i)$ 
belong to $\Rx $ for $i=1,2$. 
Now, 
define $\Delta _f\in \Der _R\Rx $ by 
\begin{equation}\label{eq:Delta _f:tr}
\Delta _f(x_1)=-\frac{\partial f}{\partial x_2}
\quad \text{and}\quad 
\Delta _f(x_2)=\frac{\partial f}{\partial x_1}. 
\end{equation}
Then, 
we have $\Delta _f(f)=0$. 
We show that $\Delta _f$ is locally nilpotent. 
It suffices to check that $\Delta _f$ 
extends to a locally nilpotent derivation of $\Sx $. 
Since $f$ is a coordinate of $\Sx $ over $S$, 
there exists $\psi \in \Aut (\Sx /S)$ 
such that $\psi (x_1)=f$. 
Then, 
we have $\Delta _f(\psi (x_1))=\Delta _f(f)=0$, 
and 
$$
\Delta _f(\psi (x_2))
=-\frac{\partial f}{\partial x_2}
\frac{\partial \psi (x_2)}{\partial x_1}
+\frac{\partial f}{\partial x_1}
\frac{\partial \psi (x_2)}{\partial x_2}
=\det J\psi . 
$$
Since $\det J\psi $ belongs to $S^{\times }$, 
it follows that $\Delta _f^2(\psi (x_2))=0$. 
Hence, 
$\Delta _f$ extends 
to a locally nilpotent derivation 
of $\Sx =S[\psi (x_1),\psi (x_2)]$. 
Thus, 
$\Delta _f$ 
is locally nilpotent. 
As remarked, 
there exists $c\in \N $ such that 
$\phi :=\exp c\Delta _f$ 
induces an element of $\Aut (\Rx /R)$. 
Then, 
$\phi $ belongs to $\Aut (\Rx /R[f])$, 
since $\Delta _f(f)=0$. 
Because $\phi $ has an infinite order, 
we conclude that 
$\Aut (\Rx /R[f])$ is an infinite group

Now, 
let $K$ be the field of fractions of $R$. 
Then, $K$ is of characteristic zero 
by the assumption that $R$ contains $\Z $. 
Thanks to the following lemma, 
we may assume that 
$S=K$ without loss of generality.

\begin{lem}\label{lem:R/Sc}
If $f\in \Rx $ is a coordinate of $\Sx $ over $S$, 
then $f$ is a coordinate of $\Kx$ over $K$. 
\end{lem}
\begin{proof}
Let $U$ be the set of $f\in \Kx $ such that 
$f$ is not a coordinate of $\Kx $ over $K$, 
but is a coordinate of $\Lx $ over $L$, 
where $L$ is the field of fractions of $S$. 
Then, $\tau (U)$ is contained in $U$ 
for each $\tau \in \Aut (\Kx /K)$. 
Now, 
suppose to the contrary that 
there exists a coordinate $f'\in \Rx $ 
of $\Sx $ over $S$ 
which is not a coordinate of $\Kx $ over $K$. 
Then, 
$f'$ is a coordinate of $\Lx $ over $L$. 
Hence, $f'$ belongs to $U$. 
Thus, $U$ is not empty. 
Take $f\in U$ so that $|\w (f)|$ is minimal. 
Then, 
$\tau (f)$ belongs to $U$ 
for each $\tau \in \Aut (\Kx /K)$. 
Hence, 
we get $|\w (\tau (f))|\geq |\w (f)|$ 
by the minimality of $|\w (f)|$. 
Thus, 
$f$ is tamely reduced over $K$. 
Since $f$ is not a coordinate of $\Kx $ over $K$, 
we see that $f$ is not a linear polynomial. 
Because $f$ is not a constant, 
we get $|\w (f)|>1$. 
By applying Proposition~\ref{prop:slope} with $\kappa =L$, 
we may write $f^{\w (f)}=a(x_i+bx_j^l)^m$, 
where $a,b\in L^{\times }$, 
$i,j\in \{ 1,2\} $ with $i\neq j$ 
and $l,m\in \N $. 
Then, 
$ax_i^m$ and $mabx_i^{m-1}x_j^l$ 
belong to $\Kx $, 
since $f$ is an element of $\Kx $. 
Hence, 
$a$ belongs to $K^{\times }$. 
Since $m\geq 1$, 
and $K$ is of characteristic zero, 
it follows that $b$ belongs to $K^{\times }$. 
Thus, $b$ belongs to $K$ and $V(K)$. 
Thanks to 
Proposition~\ref{prop:reduced coordinate}, 
this implies that $f$ is not tamely reduced over $K$, 
a contradiction. 
Therefore, 
every coordinate $f\in \Rx $ of $\Sx $ over $S$ 
is a coordinate of $\Kx $ over $K$. 
\end{proof}

When $R$ does not contain $\Z $, 
a statement similar to Lemma~\ref{lem:R/Sc} 
does not hold in general. 
We will give a counterexample 
at the end of the next section.

Let us define five types of elements of $\Rx $. 
In (1) and (2) of the following definition, 
$c\in R\sm \zs $ denotes 
the leading coefficient of the polynomial 
$g\in R[x_1]$. 
It is easy to check that the following 
five types of polynomials are 
coordinates of $\Kx $ over $K$.

\begin{definition}\label{def:invariant coord}
Let $f$ be an element of $\Rx $.

\noindent (1) 
We say that $f$ is of {\it type I} if 
$$
f=f_1:=ax_2+g 
$$
for some $a\in R\sm \zs $ and $g\in R[x_1]$ 
such that $\deg _{x_1}g\geq 2$, 
and $c$ does not belong to $aR$.

\noindent (2) 
We say that $f$ is of {\it type II} if 
$$
f=f_2:=a'x_1+h 
$$
for some $a'\in K^{\times }$ 
and $h\in \Kx $ as follows:

\noindent
(a) There exist $\zeta \in R\sm \{ 1\} $ 
and $e\geq 2$ such that $\zeta ^e=1$ 
and $a'$ belongs to $R':=R[(\zeta -1)^{-1}]$.

\noindent
(b) There exists 
$g\in R[x_1]$ such that $\deg _{x_1}g\geq 2$, 
$c$ does not belong to $(\zeta -1)R$ 
and $h$ belongs $R'[y_2^{e}]\sm R'$, 
where 
$y_2:=(\zeta -1)x_2+g$.

\noindent (3) 
We say that $f$ is of {\it type III} if 
$$
f=f_3:=a_1x_1+a_2x_2+b
$$
for some $a_1,a_2\in R\sm \zs $ and $b\in R$ 
such that $a_1/a_2$ does not belong to $V(R)$.

\noindent (4) 
We say that $f$ is of {\it type IV} if 
$$
f=f_4:=a\tau _4(x_1)^2+\tau _4(x_2) 
$$
for some $a\in R\sm \zs $, 
and $\tau _4\in \Aff (K,\x )$ 
for which there exist $\alpha _1,\alpha _2\in K^{\times }$ 
such that $\tau _4(x_1)=\alpha _1x_1+\alpha _2x_2$ 
and $\alpha _1/\alpha _2$ does not belong to $V(R)$.

\noindent (5) 
We say that $f$ is of {\it type V} if 
$$
f=f_5:=\tau _5(x_2+g') 
$$
for some $\tau _5\in \Aff (K,\x )$ 
and $g'\in K[x_1]\sm \zs $ 
with $\deg _{x_1}g'\geq 3$ as follows: 

\noindent (a) 
$g'$ belongs to $x_1^{e'}K[x_1^{e'}]$ for some $e'\geq 2$. 

\noindent (b) 
Let $\alpha _i,\beta _i\in K$ for $i=0,1,2$ 
be such that 
$$
\tau _5(x_1)=\alpha _1x_1+\alpha _2x_2+\alpha _0
\quad \text{and}\quad 
\tau _5(x_2)=\beta _1x_1+\beta _2x_2+\beta _0. 
$$
Then, we have 
$\alpha _i\neq 0$ for $i=1,2$, and 
$\alpha _1/\alpha _2$ does not belong to $V(R)$. 

\noindent (c) 
There exists $\zeta '\in R\sm \{ 1\} $ 
such that $(\zeta ')^{e'}=1$, 
and 
$$
\gamma _{i,j}(\zeta '):=\frac{(\zeta '-1)\alpha _i\beta _j}
{\alpha _1\beta _2-\alpha _2\beta _1}
$$ 
belongs to $R$ for $i=0,1,2$ and $j=1,2$. 
\end{definition}

By definition, 
no element of $\Rx $ is 
of type III or IV or V if $V(R)=K^{\times }$. 
If $R$ is a $\Q $-domain 
and $\zeta \in R\sm \{ 1\} $ 
is a root of unity, 
then $\zeta -1$ is a nonzero element of 
$\Q [\zeta ]=\Q (\zeta )$, 
and hence belongs to $R^{\times }$. 
Thus, 
no element of $\Rx $ is of type II 
by (b) of (2). 
If $R$ is a $\Q $-domain, 
then (b) and (c) of (5) imply that 
$\beta _i\neq 0$ for $i=1,2$. 
In fact, 
if $\beta _1=0$, 
then we have 
$\gamma _{2,2}(\zeta ')
=(\zeta '-1)\alpha _2/\alpha _1$. 
Since $\zeta '-1$ belongs to $R^{\times }$, 
it follows from (c) that 
$\alpha _2/\alpha _1$ belongs to $R$. 
Hence, $\alpha _1/\alpha _2$ belongs to $V(R)$, 
a contradiction to (b). 
Similarly, 
we get a contradiction if $\beta _2=0$, 
since 
$\gamma _{1,1}(\zeta ')=(1-\zeta ')\alpha _1/\alpha _2$ 
belongs to $R$.

Put $\lambda =\deg _{x_1}g$, 
$\lambda _1=\deg _{y_2}h$ 
and $\lambda _2=\deg _{x_1}g'$, 
where we regard $h$ 
as a polynomial in $y_2$ over $R'$. 
Then, 
we have $\lambda \geq 2$, $\lambda _1\geq 2$, 
$\lambda _2\geq 3$ and 
\begin{equation}\label{eq:w(f_i)}
\w (f_i)=(\deg _{x_2}f_i,\deg _{x_1}f_i)
=\left\{ 
\begin{array}{cl}
(1,\lambda )& \text{if }i=1 \\
(\lambda _1,\lambda \lambda _1)& \text{if }i=2 \\
(1,1)& \text{if }i=3 \\
(2,2)& \text{if }i=4 \\
(\lambda _2,\lambda _2)& \text{if }i=5. 
\end{array}
\right. 
\end{equation}
Let $c_1\in R'$ and $c_2\in K$ be 
the leading coefficients of $h$ and $g'$, 
respectively. 
Then, we have 
\begin{equation}\label{eq:f_i^w}
f_i^{\w (f_i)}=\left\{ 
\begin{array}{cl}
ax_2+cx_1^{\lambda }& \text{if }i=1 \\
c_1\bigl((\zeta -1)x_2+cx_1^{\lambda }\bigr)^{\lambda _1}& 
\text{if }i=2 \\
a_1x_1+a_2x_2& \text{if }i=3 \\
a(\alpha _1x_1+\alpha _2x_2)^2& \text{if }i=4 \\
c_2(\alpha _1x_1+\alpha _2x_2)^{\lambda _2}
& \text{if }i=5. 
\end{array}
\right. 
\end{equation}
By assumption, 
$c/a$ does not belong to $R$ when $i=1$, 
$c/(\zeta -1)$ 
does not belong to $R$ when $i=2$, 
$a_1/a_2$ does not belong to $V(R)$ when $i=3$, 
and $\alpha _1/\alpha _2$ 
does not belong to $V(R)$ when $i=4,5$. 
Hence, 
we know by Proposition~\ref{prop:reduced coordinate}  
that $f_i$ is tamely reduced over $R$ 
for $i=1,\ldots ,5$. 
This implies that, 
if $f_i$ is a coordinate of $\Rx $ over $R$, 
then $f_i$ is wild. 
Indeed, 
if a tame coordinate of $\Rx $ over $R$ 
is tamely reduced over $R$, 
then it must be a linear polynomial in $x_i$ over $R$ 
for some $i\in \{ 1,2\} $.

The following is the main theorem of this chapter.

\begin{thm}\label{thm:Gamma (f)}
Assume that $n=2$. 
Let $R$ be a domain containing $\Z $, 
$K$ the field of fractions of $R$, 
and $f\in \Rx $ a coordinate of $\Kx $ over $K$. 
Then, 
$f$ is of one of the types I through V 
if and only if the following three conditions hold$:$ 

\noindent{\rm (A)} $H(f)$ is not equal to $\{ \id _{\Rx }\} $. 

\noindent{\rm (B)} $f$ is tamely reduced over $R$. 

\noindent{\rm (C)} $\deg _{x_1}f\geq \deg _{x_2}f\geq 1$. 

\end{thm}

As discussed in the next section, 
there exist many elements of $\Rx $ 
which are coordinates of $\Kx $ over $K$ 
satisfying (B) and (C), 
but are of none of the types I through V\null. 
For such $f\in \Rx $, 
we have $H(f)=\{ \id _{\Rx }\} $ 
due to Theorem~\ref{thm:Gamma (f)}. 
Since $\Aut (\Rx /R[f])$ itself is an infinite group, 
this means the existence of 
a large number of wild automorphisms.

When $R$ does not contain $\Z $, 
a statement similar to the ``if" part of 
Theorem~\ref{thm:Gamma (f)} 
does not hold in general. 
We will give a counterexample 
at the end of the next section.

Next, 
we define a subset
$H_i$ of $\Aut (\Rx/R)$ 
for the polynomial $f_i$ 
in Definition~\ref{def:invariant coord} 
for $i=1,\ldots ,5$.

\begin{definition}\label{def:H_i}
\noindent (1) 
For $f_1$, we define 
$$H_1=\{ 
\phi \in J(R;x_1,x_2)\mid 
\phi (x_2)=x_2+a^{-1}(g-\phi (g))\} .
$$

\noindent (2) 
For $f_2$, 
let $\mu $ be the maximal integer 
such that $h$ belongs to $R'[y_2^{\mu }]$, 
and $Z$ the set of $\omega \in R^{\times }$ 
such that $\omega ^{\mu }=1$ and 
$g_{\omega }:=(\omega -1)(\zeta -1)^{-1}g$ belongs to $R[x_1]$. 
Then, we define 
$$
H_2=\left\{ 
\phi \in \Aut (\Rx /R[x_1])\mid
\phi (x_2)=\omega x_2+g_{\omega }
\text{ for some }\omega \in Z
\right\} . 
$$

\noindent (3) 
For $f_3$, 
we define $H_3$ 
to be the set of $\phi \in \Aff (R,\x )$ 
defined by 
\begin{equation}\label{eq:affine}
(\phi (x_1),\phi (x_2))=(x_1,x_2)A+(b_1,b_2)
\end{equation}
for some $A\in \GL (2,R)$ and $b_1,b_2\in R$ 
such that 
\begin{equation}\label{eq:H_3condition}
A{\bf a} 
={\bf a}
\quad \text{and}\quad
a_1b_1+a_2b_2=0,\quad \text{where}\quad 
{\bf a}:=\left(\begin{array}{@{\,}c@{\,}}
a_1\\ a_2
\end{array}\right) .
\end{equation}

\noindent (4) 
For $f_4$, 
we define $H_4$ 
to be the set of $\phi \in \Aff (R,\x )$ such that 
\begin{equation}\label{eq:H_4condition}
\bigl((\tau _4^{-1}\circ \phi \circ \tau _4)(x_1),
(\tau _4^{-1}\circ \phi \circ \tau _4)(x_2)\bigr)
=(x_1,x_2)
\left(\begin{array}{@{\,}cc@{\,}}
\ep  & 2\ep \alpha a\\
0 & 1
\end{array}\right)-(\alpha ,\alpha ^2a)
\end{equation}
for some $\ep \in \{ 1,-1\} $ and $\alpha \in K$.

\noindent (5) 
For $f_5$, 
let $\mu '$ be the maximal integer 
such that $g'$ belongs to $K[x_1^{\mu '}]$, 
and $Z'$ the set of $\omega \in R^{\times }$ 
such that $\omega ^{\mu '}=1$ and 
$\gamma _{i,j}(\omega )$ belongs to $R$ 
for $i=0,1,2$ and $j=1,2$. 
Then, 
we can define 
$\phi _{\omega }\in \Aff (R,\x )$ by 
\begin{equation}\label{eq:h5phi}
\begin{aligned}
\phi _{\omega }(x_1)&=(1+\gamma _{1,2}(\omega ))x_1
+\gamma _{2,2}(\omega )x_2+\gamma _{0,2}(\omega ) \\
\phi _{\omega }(x_2)&=-\gamma _{1,1}(\omega )x_1
+(1-\gamma _{2,1}(\omega ))x_2-\gamma _{0,1}(\omega )
\end{aligned}
\end{equation}
for each $\omega \in Z'$. 
Actually, 
we have $\det J\phi _{\omega }=\omega $, 
since 
\begin{equation*}
\begin{aligned}
&(1+\gamma _{1,2}(\omega ))(1-\gamma _{2,1}(\omega ))
-(-\gamma _{1,1}(\omega ))\gamma _{2,2}(\omega ) \\
&\quad =
1+
\frac{(\omega -1)(\alpha _1\beta _2-\alpha _2\beta _1)}{
\alpha _1\beta _2-\alpha _2\beta _1}
-\frac{(\omega -1)^2\alpha _1\beta _2\alpha _2\beta _1}
{(\alpha _1\beta _2-\alpha _2\beta _1)^2}
+\frac{(\omega -1)^2\alpha _1\beta _1\alpha _2\beta _2}{
(\alpha _1\beta _2-\alpha _2\beta _1)^2}=\omega . 
\end{aligned}
\end{equation*}
We define 
$H_5=\{ \phi _{\omega }\mid \omega \in Z'\} $. 

\end{definition}

In the notation above, 
the following theorem holds.

\begin{thm}\label{thm:H(f)}
Let $f_1,\ldots ,f_5$ 
be as in Definition~$\ref{def:invariant coord}$, 
where $f_5$ need not to satisfy {\rm (c)} of {\rm (5)}. 
Then, 
we have $H(f_i)=H_i$ for $i=1,\ldots ,5$. 
\end{thm}

We remark that $H_2$ is a finite set with at most $\mu $ elements, 
since $Z$ consists of $\mu $-th roots of unity. 
Similarly, 
$H_5$ is a finite set with at most $\mu '$ elements. 
Thus, 
$H(f_2)$ and $H(f_5)$ are finite groups 
due to Theorem~\ref{thm:H(f)}. 
Therefore, 
if $f$ is of type II or V for $f\in \Rx $, 
then $f$ is quasi-totally wild.

\begin{prop}\label{prop:H_i infinite}
$f_1$, $f_3$ and $f_4$ are not exponentially wild. 
In particular, 
$H(f_1)$, $H(f_3)$ and $H(f_4)$ are infinite groups. 
\end{prop}
\begin{proof}
For $i=1,3,4$, 
define $\Delta _{f_i}$ as in (\ref{eq:Delta _f:tr}). 
Then, $\Delta _{f_i}$ is locally nilpotent, 
since $f_i$ is a coordinate of $\Kx $ over $K$. 
Take $m_i\in \N $ 
such that $\Delta _{f_i}^{m_i}(x_j)=0$ 
for $j=1,2$, 
and set $D_i=m_i!\Delta _{f_i}$. 
Then, 
we have $D_i(f_i)=0$, 
and $\exp D_i$ induces an element of $\Aut (\Rx /R)$. 
From the definition of $f_1$ and $f_3$, 
we see that 
$\partial f_i/\partial x_2$ belongs to $R$, 
and $\partial f_i/\partial x_1$ belongs to $R[x_1]$ 
for $i=1,3$. 
Hence, 
$D_1$ and $D_3$ are triangular. 
Thus, 
$\exp D_1$ and $\exp D_3$ 
belong to $\T (R,\x )$. 
Since $\partial f_4/\partial x_i$ is a linear polynomial 
for $i=1,2$, 
we see that $D_4$ is affine. 
Hence, 
$\exp D_4$ 
belongs to $\T (R,\x )$. 
Therefore, 
$f_i$ is not exponentially wild for $i=1,3,4$. 
Since $\exp D_i$ belongs to $H(f_i)$ 
and has an infinite order, 
we know that $H(f_i)$ is an infinite group for $i=1,3,4$. 
\end{proof}

As a consequence of 
Theorems~\ref{thm:Gamma (f)} and \ref{thm:H(f)} 
and Proposition~\ref{prop:H_i infinite}, 
we get the following corollary.

\begin{cor}\label{cor:expw=>qtw}
Assume that $f\in \Rx $ 
is a coordinate of $\Kx $ over $K$. 

\noindent{\rm (i)} 
$f$ is quasi-totally wild 
if and only if 
$f$ is exponentially wild.

\noindent{\rm (ii)} 
Assume that $R$ 
is a $\Q $-domain such that $V(R)=K^{\times }$. 
If $H(f)$ is not equal to $\{ \id _{\Rx }\} $, 
then $H(f)$ is an infinite group. 
\end{cor}
\begin{proof}
(i) 
It suffices to prove the ``if" part. 
Assume that $f$ is exponentially wild. 
Without loss of generality, 
we may assume that $H(f)\neq \{ \id _{\Rx }\} $. 
By replacing $f$ with $\tau (f)$ for some $\tau \in \T (R,\x )$, 
we may assume further that $f$ is tamely reduced over $R$ 
and $\deg _{x_1}f\geq \deg _{x_2}f$. 
Since $f$ is exponentially wild by assumption, 
$f$ is not killed by $\partial /\partial x_2$. 
Hence, 
we get $\deg _{x_2}f\geq 1$. 
Thus, 
$f$ satisfies (A), (B) and (C) of 
Theorem~\ref{thm:Gamma (f)}. 
Therefore, 
$f$ must be of one of the types I through V\null. 
By Proposition~\ref{prop:H_i infinite}, 
$f$ is not of type I or III or IV\null. 
Hence, 
$f$ must be of type II or V\null. 
Therefore, 
$f$ is quasi-totally wild 
as mentioned after Theorem~\ref{thm:H(f)}.

(ii) 
By replacing $f$ with $\tau (f)$ for some $\tau \in \T (R,\x )$, 
we may assume that 
$f$ is tamely reduced over $R$ and 
$\deg _{x_1}f\geq \deg _{x_2}f$. 
If $\deg _{x_2}f=0$, 
then $f$ is a linear polynomial in $x_1$ over $R$. 
Hence, 
we have $H(f)=H(x_1)$. 
Thus, $H(f)$ is an infinite group. 
Assume that $\deg _{x_2}f\geq 1$, 
and $H(f)\neq \{ \id _{\Rx }\} $. 
Then, 
$f$ satisfies (A), (B) and (C) of Theorem~\ref{thm:Gamma (f)}. 
Hence, 
$f$ must be of one of the types I through V\null. 
Since $R$ is a $\Q $-domain 
such that $V(R)=K^{\times }$, 
we know that $f$ is of none of the types II through V 
by the remark after Definition~\ref{def:invariant coord}. 
Thus, 
$f$ is of type I\null. 
Therefore, $H(f)$ is an infinite group by 
Proposition~\ref{prop:H_i infinite}. 
\end{proof}

Thanks to Theorem~\ref{thm:tame23}, 
Corollary~\ref{cor:expw=>qtw} (ii) 
implies the following corollary.

\begin{cor}\label{cor:225}
Assume that $n=3$. 
Let $f\in \kx $ 
be a coordinate of $k(x_3)[x_1,x_2]$ over $k(x_3)$. 
Then, $\Aut (\kx /k[x_3,f])\cap \T (k,\x )$ 
is equal to $\{ \id _{\kx }\} $ or an infinite group. 
\end{cor}
Actually, 
$R=k[x_3]$ is a $\Q $-domain such that $V(R)=K^{\times }$, 
and 
$$
H(f)=\Aut (\kx /k[x_3,f])\cap \T (k[x_3],\{ x_1,x_2\} )
=\Aut (\kx /k[x_3,f])\cap \T (k,\x )
$$ 
by Theorem~\ref{thm:tame23}.

\section{Examples}
\setcounter{equation}{0}

In this section, 
we construct various examples. 
First, 
we give the five types of polynomials. 
It is easy to find polynomials of types I through IV\null. 
For example, 
let $R=\Z $. 
Then, 
$g_1:=2x_2+x_1^2$ is of type I, 
and 
$$
g_2:=x_1+(2x_2+x_1^2)^2
$$
is of type II with 
$\zeta =-1$, $e=2$ and $g=-x_1^2$. 
Let $R=\Z [a_1,a_2]$ be the polynomial ring 
in $a_1$ and $a_2$ over $\Z $. 
Then, $a_1/a_2$ does not belong to $V(R)$ 
by Lemma~\ref{lem:V(R)} (ii). 
Hence, $g_3:=a_1x_1+a_2x_2$ is of type III, 
and 
$$
g_4:=(a_1x_1+a_2x_2)^2+x_2
$$ 
is of types IV with $a=1$ and $\tau _4\in \Aff (K,\x )$ 
defined by $\tau _4(x_1)=a_1x_1+a_2x_2$ 
and $\tau _4(x_2)=x_2$.

To construct a polynomial of type V, 
consider the subring $R:=\Z [y,2z,yz]$ 
of the polynomial ring $\Z [y,z]$. 
Then, 
it is easy to see that 
$y$ is an irreducible element of $R$, 
$z$ does not belong to $R$, 
and $I:=yR+yzR$ is not equal to $R$. 
We claim that $I$ 
is not a principal ideal of $R$. 
In fact, 
if $I=pR$ for some $p\in R\sm R^{\times }$, 
then $p$ divides $y$, 
since $y$ belongs to $I$. 
Hence, 
we get $pR=yR$ 
by the irreducibility of $y$. 
Since $yz$ belongs to $I$, 
it follows that $y$ divides $yz$. 
Hence, $z$ belongs to $R$, 
a contradiction. 
Thus, 
$I$ is not a principal ideal of $R$. 
Therefore, 
$yz/y$ does not belong to $V(R)$ 
by Lemma~\ref{lem:V(R)} (ii). 
Define $\tau _5\in \Aff (K,\x )$ by 
$\tau _5(x_1)=yx_1+yzx_2$ and $\tau _5(x_2)=x_2$. 
Then, 
$$
g_5:=\tau _5(x_2+x_1^4)
$$ 
is an element of $\Rx $ 
of the form of $f_5$ with $g'=x_1^4$, 
and 
$\tau _5$ satisfies (b) of 
Definition~\ref{def:invariant coord} (5). 
Since $x_1^4$ belongs to $x_1^4K[x_1^4]$, 
we see that (a) holds for $e'=4$. 
Observe that 
$\gamma _{1,2}(-1)=-2y/y=-2$, 
$\gamma _{2,2}(-1)=-2yz/y=-2z$ 
and $\gamma _{i,j}(-1)=0$ 
if $i=0$ or $j=1$. 
Hence, 
$\gamma _{i,j}(-1)$ belongs to $R$ 
for $i=0,1,2$ and $j=1,2$. 
Since $(-1)^4=1$, 
we conclude that (c) holds for $\zeta =-1$. 
Therefore, $g_5$ is of type V.

We remark that $g_1,\ldots ,g_5$ above 
are not coordinates of $\Rx $ over $R$. 
This can be verified by using the following lemma 
and proposition.

\begin{lem}\label{lem:notcoordinate}
Let $f$ be a coordinate of $\Rx $ over $R$, 
$\mathfrak{p}$ a prime ideal of $R$, 
and $\bar{f}$ the image of $f$ in $(R/\mathfrak{p})[\x ]$. 
If $\bar{f}$ belongs to $(R/\mathfrak{p})[x_i]$ 
for some $i\in \{ 1,2\} $, 
then we have $\deg _{x_i}f=1$. 
\end{lem}
\begin{proof}
Since $f$ is a coordinate of $\Rx $ over $R$, 
we may find $\phi \in \Aut (\Rx /R)$ such that $\phi (x_i)=f$. 
Set $\bar{R}=R/\mathfrak{p}$. 
Then, 
$\bar{\phi }:=\id _{\bar{R}}\otimes \phi $ 
is an automorphism of 
$\bar{R}[\x ]=\bar{R}\otimes _R\Rx $ over $\bar{R}$. 
Since $\bar{\phi }(x_i)=\bar{f}$ belongs to $\bar{R}[x_i]$ 
by assumption, 
$\bar{\phi }$ induces an element of 
$\Aut (\bar{R}[x_i]/\bar{R})$. 
Because $\bar{R}$ is a domain, 
this implies that 
$\bar{\phi }(x_i)$ is a linear polynomial. 
Therefore, we get $\deg _{x_i}\bar{f}=1$. 
\end{proof}

Since the images of $g_1$ and $g_2$ in $(\Z /2\Z )[\x ]$ 
are $x_1^2$ and $x_1+x_1^4$, respectively, 
we know by Lemma~\ref{lem:notcoordinate} that 
$g_1$ and $g_2$ are not coordinates of $\Z [\x ]$ over $\Z $. 
Since the image of $g_4$ 
in $(\Z [a_1,a_2]/(a_1))[\x ]$ is $(a_2x_2)^2+x_2$, 
we know that 
$g_4$ is not a coordinate of 
$(\Z [a_1,a_2])[\x ]$ over $\Z [a_1,a_2]$ similarly. 
As for $g_5$, 
observe that $\mathfrak{p}=(y,2z,2)$ 
is a prime ideal of $R=\Z [y,2z,yz]$. 
Since the image of $g_5$ in $(R/\mathfrak{p})[\x ]$ 
is $(yzx_2)^4+x_2$, 
we know that 
$g_5$ is not a coordinate of $\Rx $ over $R$.

By the following proposition, 
$g_3$ is not a coordinate of $\Rx $ over $R$.

\begin{prop}\label{prop:f1f3}
No coordinate of $\Rx $ over $R$ is of type III. 
\end{prop}
\begin{proof}
Suppose to the contrary that 
$\phi (x_1)=f_3$ for some $\phi \in \Aut (\Rx /R)$. 
Then, 
$\det J\phi $ belongs to $a_1R+a_2R$, 
since $\partial f_3/\partial x_i=a_i$ for $i=1,2$. 
Because $\det J\phi $ belongs to $R^{\times }$, 
we get $a_1R+a_2R=R$. 
This contradicts that $a_1/a_2$ 
does not belong to $V(R)$. 
Therefore, 
$f_3$ is not a coordinate of $\Rx $ over $R$. 
\end{proof}

Next, 
we give examples of coordinates of $\Rx $ over $R$ 
which are of types I, IV and V\null. 
The following fact is well-known 
(see \cite[Lemma 1.1.8]{Essen} for a more general statement).

\begin{lem}\label{lem:Essen:Lem1.1.8}
Let $\phi $ be an endomorphism of the $R$-algebra $\Rx $ 
such that $\det J\phi $ belongs to $R^{\times }$. 
If $K[\phi (x_1),\phi (x_2)]=\Kx $, 
then $\phi $ belongs to $\Aut (\Rx /R)$. 
\end{lem}

Assume that $R=\Z $. 
Then, 
it is easy to see that 
$$
h_1:=4x_2+1+x_1+2x_1^2 
$$
is a type I element of $\Rx $. 
We show that $h_1$ 
is a coordinate of $\Rx $ over $R$. 
Since 
$$
2h_1^2-3h_1\equiv 
2(1+x_1)^2+(1+x_1+2x_1^2)
\equiv x_1-1\pmod{4\Rx }, 
$$ 
we see that 
$$
h_1':=\frac{1}{4}\left(x_1-1-(2h_1^2-3h_1)\right) 
$$
belongs to $\Rx $. 
Hence, 
we may define an endomorphism $\phi $ 
of the $R$-algebra $\Rx $ 
by $\phi (x_1)=h_1$ and $\phi (x_2)=h_1'$. 
Then, 
we have $\det J\phi =-1$, 
since 
$$
dh_1\wedge dh_1'=\frac{1}{4}dh_1\wedge dx_1=-dx_1\wedge dx_2. 
$$
Thus, $\det J\phi $ belongs to $R^{\times }$. 
Since $K[h_1,h_1']=K[h_1,x_1]=\Kx $, 
we conclude that $\phi $ belongs to $\Aut (\Rx /R)$ 
by Lemma~\ref{lem:Essen:Lem1.1.8}. 
Therefore, 
$h_1$ is a coordinate of $\Rx $ over $R$.

Next, 
let $R=\Z [\alpha _1,\alpha _2]$ 
be the polynomial ring in 
$\alpha _1$ and $\alpha _2$ over $\Z $. 
Define elements of $\Rx $ by 
$$
h_4=\alpha _1(\alpha _1x_1+\alpha _2x_2)^2+x_2,\quad 
h_4'=\alpha _2(\alpha _1x_1+\alpha _2x_2)^2-x_1. 
$$
Then, $h_4$ is of type IV 
with $a=\alpha _1$ and 
$\tau _4\in \Aff (K,\x )$ defined by 
$\tau _4(x_1)=\alpha _1x_1+\alpha _2x_2$ 
and $\tau _4(x_2)=x_2$. 
Since 
$\alpha _2h_4-\alpha _1h_4'=\alpha _1x_1+\alpha _2x_2$, 
we see that $x_1$ and $x_2$ belong to $R[h_4,h_4']$. 
Hence, we get $R[h_4,h_4']=\Rx $. 
Therefore, 
$h_4$ is a coordinate of $\Rx $ over $R$.

Finally, we construct a coordinate of type V\null. 
Let $R_0$ be the polynomial ring 
in four variables $\alpha _i$ and $\beta _i$ 
for $i=1,2$ over $\Z $. 
Put $\delta =\alpha _1\beta _2-\alpha _2\beta _1$, 
and consider the subring 
$$
R:=R_0\left[2\alpha _1\delta ^{-1},
2\alpha _2\delta ^{-1},
4\beta _1\delta ^{-1},
4\beta _2\delta ^{-1},
(\alpha _1+2\beta _1)\delta ^{-1},
(\alpha _2+2\beta _2)\delta ^{-1}\right]
$$
of $R_0[\delta ^{-1}]$. 
Define $\tau _5\in \Aff (K,\x )$ by 
$$
\tau _5(x_1)=\alpha _1x_1+\alpha _2x_2+2\quad 
\text{and} 
\quad 
\tau _5(x_2)=\beta _1x_1+\beta _2x_2. 
$$
Then, 
$h_5:=\tau _5(x_2+x_1^4)$ is an element of $\Rx $ 
of the form of $f_5$ with $g'=x_1^4$, 
and satisfies (a) of 
Definition~\ref{def:invariant coord} (5) 
for $e'=4$. 
We check (b) and (c).

Define a homomorphism $\psi :R_0\to \Z $ 
of $\Z $-algebras by 
$$
\psi (\alpha _1)=0,\quad 
\psi (\alpha _2)=2,\quad 
\psi (\beta _1)=-2,\quad 
\psi (\beta _2)=3. 
$$
Then, we have $\psi (\delta )=4$. 
Hence, 
$\psi $ extends to a homomorphism 
$\bar{\psi }:R_0[\delta ^{-1}]\to \Z [1/4]$ 
of $\Z $-algebras. 
Then, 
we have 
$\bar{\psi }\bigl(4\beta _i\delta ^{-1}\bigr)
=\psi (\beta _i)$ 
for $i=1,2$, 
and 
$$
\bar{\psi }\bigl(2\alpha _1\delta ^{-1}\bigr)=0,\ 
\bar{\psi }\bigl(2\alpha _2\delta ^{-1}\bigr)=1,\ 
\bar{\psi }\bigl((\alpha _1+2\beta _1)\delta ^{-1}\bigr)=-1,\ 
\bar{\psi }\bigl((\alpha _2+2\beta _2)\delta ^{-1}\bigr)=2. 
$$
Thus, 
$\bar{\psi }(R)$ is contained in $\Z $. 
Therefore, 
we get $\bar{\psi }(R)=\Z $.

We prove that $2$ is an irreducible element of $R$ 
by contradiction. 
Suppose that $2=pq$ for some $p,q\in R\sm \zs $ 
not belonging to $R^{\times }$. 
Since $p$ and $q$ belong to $R_0[\delta ^{-1}]$, 
we may write 
$p=p'\delta ^{-l}$ and $q=q'\delta ^{-m}$, 
where $p',q'\in R_0$ and $l,m\in \Zn $. 
Then, 
we have $2\delta ^{l+m}=p'q'$. 
Since 
$2$ and $\delta $ are irreducible elements of $R_0$, 
we may write 
$p'=2u\delta ^{l'}$ and $q'=u\delta ^{m'}$ 
by interchanging $p'$ and $q'$ if necessary, 
where $u\in \{ 1,-1\} $, and 
$l',m'\in \Zn $ are such that $l'+m'=l+m$. 
We show that $(l,m)=(l',m')$. 
Then, it follows that 
$q=(u\delta ^{m'})\delta ^{-m}=u$ 
belongs to $R^{\times }$, 
and we are led to a contradiction. 
Suppose to the contrary that $(l,m)\neq (l',m')$. 
Then, 
we have $l>l'$ or $m>m'$. 
If $l>l'$, 
then $u\delta ^{l-l'-1}p$ belongs to $R$, 
since so does $p$. 
Since $p=(2u\delta ^{l'})\delta ^{-l}$, 
we have $u\delta ^{l-l'-1}p=2\delta ^{-1}$. 
Hence, $2\delta ^{-1}$ belongs to $R$. 
Thus, 
$\bar{\psi }(2\delta ^{-1})=1/2$ belongs to 
$\bar{\psi }(R)=\Z $, a contradiction. 
Similarly, 
if $m>m'$, 
then $u\delta ^{m-m'-1}q$ belongs to $R$. 
Since $q=(u\delta ^{m'})\delta ^{-m}$, 
we have $u\delta ^{m-m'-1}q=\delta ^{-1}$. 
Hence, $\delta ^{-1}$ belongs to $R$. 
Thus, 
$\bar{\psi }(\delta ^{-1})=1/4$ belongs to 
$\bar{\psi }(R)=\Z $, a contradiction. 
Therefore, 
$2$ is an irreducible element of $R$.

We show that $I:=\alpha _1R+\alpha _2R$ 
is not a principal ideal of $R$. 
First, note that $I$ is not a unit ideal, 
since $\bar{\psi }(I)=2\Z $. 
Suppose that $I=sR$ for some $s\in R\sm R^{\times }$. 
Then, 
$s$ divides 
$$
\alpha _1\left((\alpha _2+2\beta _2)\delta ^{-1}\right)
-\alpha _2\left((\alpha _1+2\beta _1)\delta ^{-1}\right)
=2(\alpha _1\beta _2-\alpha _2\beta _1)\delta ^{-1}=2. 
$$
Since $s$ is not a unit of $R$, 
it follows that $sR=2R$ by the irreducibility of $2$. 
Hence, 
2 divides $\alpha _1$. 
Thus, 
$\alpha _1/2$ belongs to $R$, 
and so belongs to $R_0[\delta ^{-1}]$, 
a contradiction. 
Therefore, 
$I$ is not a principal ideal of $R$. 
Consequently, 
$\alpha _1/\alpha _2$ does not belong to $V(R)$ 
by Lemma~\ref{lem:V(R)} (ii). 
This proves that 
$\tau _5$ satisfies (b) of 
Definition~\ref{def:invariant coord} (5). 
Observe that 
$$
\gamma _{i,j}(-1)=-2\alpha _i\beta _j\delta ^{-1}
=-\beta _j(2\alpha _i\delta ^{-1})
\quad \text{and}\quad 
\gamma _{0,j}(-1)=-4\beta _j\delta ^{-1}
$$ 
belong to $R$ 
for each $i,j\in \{ 1,2\} $. 
Since $(-1)^{e'}=1$, 
we know that (c) holds for $\zeta =-1$. 
Therefore, 
$h_5$ is of type V\null.

Next, we show that $h_5$ 
is a coordinate of $\Rx $ over $R$. 
Consider the polynomial 
$$
h_5':=\delta ^{-1}\bigl( \tau _5(x_1)+2(h_5-17)\bigr) . 
$$
Then, 
we have 
$K[h_5,h_5']=K[h_5,\tau _5(x_1)]
=K[\tau _5(x_2),\tau _5(x_1)]=\Kx $, 
and 
\begin{equation}\label{eq:dh_5wedgedh_5'}
dh_5\wedge dh_5'
=\delta ^{-1}dh_5\wedge d\tau _5(x_1)
=\delta ^{-1}d\tau _5(x_2)\wedge d\tau _5(x_1)
=-dx_1\wedge dx_2. 
\end{equation}
We check that $h_5'$ belongs to $\Rx $. 
Note that $p:=\tau _5(x_1)^4-16$ 
is divisible by $\alpha _1x_1+\alpha _2x_2$. 
Since $2\delta ^{-1}(\alpha _1x_1+\alpha _2x_2)$ 
belongs to $\Rx $, 
we see that $2\delta ^{-1}p$ belongs to $\Rx $. 
Since $h_5-17=\tau (x_2)+p-1$, 
it follows that 
$$
h_5'=\delta ^{-1}\bigl( \tau _5(x_1)+2(\tau (x_2)+p-1)\bigr) 
=(\alpha _1+2\beta _1)\delta ^{-1}x_1+
(\alpha _2+2\beta _2)\delta ^{-1}x_2+2\delta ^{-1}p
$$
belongs to $\Rx $. 
Hence, 
we may define an endomorphism $\phi $ 
of the $R$-algebra $\Rx $ 
by $\phi (x_1)=h_5$ and $\phi (x_2)=h_5'$. 
Then, we have $\det J\phi =-1$ by (\ref{eq:dh_5wedgedh_5'}). 
Hence, $\det J\phi $ belongs to $R^{\times }$. 
Thus, 
$\phi $ belongs to $\Aut (\Rx /R)$ 
by Lemma~\ref{lem:Essen:Lem1.1.8}. 
Therefore, 
$h_5$ is a coordinate of $\Rx $ over $R$.

As remarked after Theorem~\ref{thm:Gamma (f)}, 
we can easily construct elements of $\Rx $ 
which are coordinates of $\Kx $ over $K$ 
satisfying (B) and (C) of Theorem~\ref{thm:Gamma (f)}, 
but are of none of the types I through V\null. 
For example, 
take any polynomial $f_1$ of type I, 
and $\Phi _i(z)\in R[z]$ 
with $e_i:=\deg _z\Phi _i(z)\geq 2$ 
for each $i\geq 1$. 
We define a sequence 
$(p_i)_{i=0}^{\infty }$ of elements of $\Rx $ by 
$$
p_0=x_1,\quad 
p_1=f_1\quad\text{and}\quad 
p_{i+1}=p_{i-1}+\Phi _i(p_i)
\quad \text{for}\quad i\geq 1. 
$$
Then, 
$p_i$ is a coordinate of $\Kx $ over $K$ for each $i\geq 0$, 
since $K[p_{0},p_{1}]=\Kx $ and 
$R[p_{l},p_{l+1}]=R[p_{l-1},p_{l}]$ for each $l\geq 1$. 
We show that $p_i$ satisfies (B) and (C), 
but is of none of the types I through V when $i\geq 3$. 
Let $c_i$ be the leading coefficient of $\Phi _i(z)$ 
for each $i\geq 1$. 
First, 
we show that 
\begin{equation}\label{eq:notI-V1}
p_{i}^{\w }=c_1\cdots c_{i-1}(f_1^{\w })^{e_1\cdots e_{i-1}}
\end{equation}
holds for each $\w \in \N ^2$ and $i\geq 1$ by induction on $i$. 
The assertion is clear if $i=1$. 
Assume that $i\geq 2$. 
Then, we have 
$p_{j}^{\w }=c_1\cdots c_{j-1}(f_1^{\w })^{e_1\cdots e_{j-1}}$ 
for $j=1,\ldots ,i-1$ by induction assumption. 
When $i\geq 3$, 
this implies that 
$\degw p_{i-1}=e_{i-2}\degw p_{i-2}$. 
Since $e_{i-2}\geq 2$, 
we get $\degw p_{i-1}>\degw p_{i-2}$. 
The same holds when $i=2$ 
since $\degw f_1>\degw x_1$. 
Hence, 
$\degw \Phi _{i-1}(p_{i-1})=e_{i-1}\degw p_{i-1}$ 
is greater than $\degw p_{i-2}$. 
Thus, 
it follows that 
\begin{align*}
&p_{i}^{\w }
=\bigl( p_{i-2}+\Phi _{i-1}(p_{i-1})\bigr) ^{\w }
=\Phi _{i-1}(p_{i-1})^{\w }
=c_{i-1}(p_{i-1}^{\w })^{e_{i-1}} \\
&\quad =c_{i-1}\bigl( 
c_1\cdots c_{i-2}(f_1^{\w })^{e_1\cdots e_{i-2}}
\bigr)^{e_{i-1}}
=c_1\cdots c_{i-1}(f_1^{\w })^{e_1\cdots e_{i-1}}. 
\end{align*}
Therefore, 
(\ref{eq:notI-V1}) holds for every $i\geq 1$. 
Since $p_i$ is a coordinate of $\Kx $ over $K$, 
we may write $p_i^{\w (p_i)}$ 
as in Proposition~\ref{prop:slope}. 
Thanks to (\ref{eq:notI-V1}), 
this implies that 
$f_1^{\w (p_i)}=ax_2+cx_1^{\lambda }$. 
Hence, 
$p_i^{\w (p_i)}$ 
is a power of $x_2+(c/a)x_1^{\lambda }$ 
multiplied by a nonzero constant. 
Since $f_1$ is of type I, 
we have $\lambda \geq 2$, 
and $c/a$ does not belong to $R$. 
Hence, 
$p_i$ is tamely reduced over $R$ 
by Proposition~\ref{prop:reduced coordinate} (ii). 
Thus, 
$p_i$ satisfies (B). 
By Proposition~\ref{prop:slope}, 
we have $\deg _{x_l}p_i=\deg _{x_l}p_i^{\w (p_i)}$ 
for $l=1,2$. 
Hence, 
we see from (\ref{eq:notI-V1}) 
that 
\begin{equation}\label{eq:notI-V}
\deg _{x_1}p_{i}=\lambda e_1\cdots e_{i-1}
\quad \text{and}\quad 
\deg _{x_2}p_{i}=e_1\cdots e_{i-1}. 
\end{equation}
Since $\lambda \geq 2$, 
it follows that $p_i$ satisfies (C). 
Moreover, 
we know that $p_i$ is not of type other than II 
owing to (\ref{eq:w(f_i)}). 
We show that $p_i$ is not of type II\null. 
Suppose to the contrary that $p_i$ is of type II\null. 
Then, 
we may find $q\in \Kx $ such that 
$\deg _{x_2}q=1$ and $K[p_i,q]=\Kx $. 
Since $K[p_i,p_{i-1}]=\Kx $, 
we may write $q=\alpha p_{i-1}+\Psi (p_i)$, 
where $\alpha \in K^{\times }$ and $\Psi (z)\in K[z]$. 
Since $i\geq 3$, we have 
$\deg _{x_2}p_i>\deg _{x_2}p_{i-1}\geq e_1\geq 2$ 
by (\ref{eq:notI-V}). 
Hence, 
$\deg _{x_2}q$ is equal to the maximum between 
$\deg _{x_2}p_{i-1}$ and $\deg _{x_2}\Psi (p_i)$, 
and is at least $\deg _{x_2}p_{i-1}\geq 2$. 
This is a contradiction. 
Thus, 
$p_i$ is not of type II\null. 
Therefore, 
$p_i$ is of none of the types I through V\null.

In closing, 
we construct counterexamples to 
statements similar to Lemma~\ref{lem:R/Sc} 
and the ``if" part of Theorem~\ref{thm:Gamma (f)} 
in the case where $R$ does not contain $\Z $. 
First, 
let $R=(\Z /2\Z )(z^2)$ and $S=(\Z /2\Z )(z)$, 
and consider the element 
$$
f:=x_2+x_1^2+z^2x_2^2 
$$
of $\Rx $. 
Since $f=x_2+(x_1+zx_2)^2$, 
we can define 
$\psi \in \Aut (\Sx /S)$ by 
$\psi (x_1)=f$ and $\psi (x_2)=x_1+zx_2$. 
Hence, 
$f$ is a coordinate of $\Sx $ over $S$. 
Observe that 
$f^{\w (f)}=(x_1+zx_2)^2$. 
Since $z$ does not belong to $R=K$, 
this implies that $f$ 
is not a coordinate of $\Kx $ over $K$ 
in view of Proposition~\ref{prop:slope}. 
Therefore, 
$f$ is a counterexample to 
a statement similar to Lemma~\ref{lem:R/Sc}.

Next, 
let $R=(\Z /2\Z )[z]$, 
and define $\sigma _i\in \Aut (\Kx /K)$ 
for $i=0,1,2$ by 
$(\sigma _0(x_1),\sigma _0(x_2))=(x_1,zx_2)$, 
$$
(\sigma _1(x_1),\sigma _1(x_2))=(x_1,x_2+x_1+x_1^2)
\ \text{ and }\ 
(\sigma _2(x_1),\sigma _2(x_2))=(x_1+x_2^2,x_2). 
$$
Then, 
consider the polynomial 
\begin{align*}
g&:=(\sigma _0\circ \sigma _1\circ \sigma _2\circ \sigma _1)(x_2) \\
&=(\sigma _0\circ \sigma _1)
\bigl( 
x_2+(x_1+x_2^2)+(x_1^2+x_2^4)
\bigr) \\
&=\sigma _0
\bigl( 
(x_2+x_1+x_1^2)+x_1+(x_2^2+x_1^2+x_1^4)+x_1^2
+(x_2^4+x_1^4+x_1^8)
\bigr)\\
&=x_1^2+x_1^8+zx_2+z^2x_2^2+z^4x_2^4. 
\end{align*}
Clearly, 
$g$ belongs to $\Rx $, 
and is a coordinate of $\Kx $ over $K$ by definition. 
We show that $g$ satisfies (A), (B) and (C) 
of Theorem~\ref{thm:Gamma (f)}, 
but is of none of the types I through V\null. 
Define 
$\phi \in \Aut (\Rx /R[x_2])$ by $\phi (x_1)=x_1+1$. 
Then, 
$\phi $ belongs to $\T (R,\x )$, 
and satisfies $\phi (g)=g$. 
Hence, 
$\phi $ belongs to $H(g)$. 
Thus, $g$ satisfies (A). 
Since $g^{\w (g)}=z^4(x_2+z^{-1}x_1^2)^4$, 
and $z^{-1}$ does not belong to $R$, 
we know that $g$ is tamely reduced over $R$ 
by Proposition~\ref{prop:reduced coordinate} (ii). 
Hence, $g$ satisfies (B). 
Since $\deg _{x_1}g=8$ and $\deg _{x_2}g=4$, 
we see that $g$ satisfies (C). 
Moreover, 
$g$ is not of type other than II 
because of (\ref{eq:w(f_i)}). 
Since $R$ 
does not contain a 
root of unity of positive order, 
we know that $g$ is not of type II\null. 
Therefore, 
$g$ is of none of the types I through V.

\section{Proof (I)}
\setcounter{equation}{0}

We begin with the proof of 
the ``only if" part of Theorem~\ref{thm:Gamma (f)}. 
Let $f\in \Rx $ be a coordinate of $\Kx $ over $K$ 
of one of the types I through V\null. 
Then, $f$ is tamely reduced over $R$ 
by the discussion after (\ref{eq:f_i^w}). 
Hence, $f$ satisfies (B). 
From (\ref{eq:w(f_i)}), 
we see that $f$ satisfies (C). 
By Proposition~\ref{prop:H_i infinite}, 
$H(f_i)$ is an infinite group 
for $i=1,3,4$. 
Hence, $f$ satisfies (A) 
if $f$ is of type I or III or IV\null. 
We show that 
$H_2$ and $H_5$ are not equal to $\{ \id _{\Rx }\} $. 
Then, 
when Theorem~\ref{thm:H(f)} is proved, 
it is also proved that 
$f$ satisfies (A) 
if $f$ is of type II or V.

For $f_2$, 
set $d=\lcm (e,\mu )$. 
Then, 
$h$ belongs to $R'[y_2^d]$, 
since $h$ belongs to $R'[y_2^e]$ and $R'[y_2^{\mu }]$ 
by assumption. 
By the maximality of $\mu $, 
it follows that $d=\mu $. 
Hence, $\mu $ is divisible by $e$. 
Since $\zeta ^e=1$ by assumption, 
we get 
$\zeta ^{\mu }=(\zeta ^e)^{\mu /e}=1$. 
Because $g_{\zeta }=(\zeta -1)(\zeta -1)^{-1}g=g$ 
belongs to $R[x_1]$, 
we know that $\zeta $ belongs to $Z$. 
Define $\phi \in \Aut (\Rx /R[x_1])$ 
by $\phi (x_2)=\zeta x_2+g$. 
Then, $\phi $ belongs to $H_2$, 
and is not equal to $\id _{\Rx }$. 
Therefore, 
$H_2$ is not equal to 
$\{ \id _{\Rx }\} $. 
Similarly, 
since $g'$ belongs to $K[x_1^{e'}]$ and $K[x_1^{\mu '}]$ 
by assumption, 
we know 
that $\mu '$ is divisible by $e'$ 
by the maximality of $\mu '$. 
Hence, 
we have $(\zeta ')^{\mu '}=((\zeta ')^{e'})^{\mu '/e'}=1$. 
Since $\gamma _{i,j}(\zeta ')$ 
belongs to $R$ for $i=0,1,2$ and $j=1,2$ by assumption, 
it follows that $\zeta '$ belongs to $Z'$. 
Thus, 
$H_5$ contains $\phi _{\zeta '}$. 
Since $\det J\phi _{\zeta '}=\zeta '\neq 1$, 
we have $\phi _{\zeta '}\neq \id _{\Rx }$. 
Therefore, 
$H_5$ is not equal to $\{ \id _{\Rx }\} $.

Next, we prove Theorem~\ref{thm:H(f)}. 
Assume that $i=1$. 
Since $f_1$ is tamely reduced over $R$, 
and $\deg _{x_1}f_1=\lambda $ 
is greater than $\deg _{x_2}f_1=1$, 
we know that 
$H(f_1)$ is contained in $J(R;x_1,x_2)$ 
thanks to Theorem~\ref{thm:reduced coordinate} (ii). 
By definition, 
$H_1$ is also contained in $J(R;x_1,x_2)$. 
Hence, it suffices to show that 
$\phi $ belongs to $H(f_1)$ 
if and only if $\phi $ belongs to $H_1$ 
for each $\phi \in J(R;x_1,x_2)$. 
Take any $\phi \in J(R;x_1,x_2)$. 
Then, $\phi $ belongs to $\T  (R,\x )$. 
Hence, 
$\phi $ belongs to $H(f_1)$ if and only if 
$\phi (f_1)=f_1$. 
Since $f_1=ax_2+g$ and $a\neq 0$, 
this condition is equivalent to 
$a\phi (x_2)+\phi (g)
=ax_2+g$, 
and to $\phi (x_2)=x_2+a^{-1}(g-\phi (g))$. 
This condition is equivalent to the condition that 
$\phi $ belongs to $H_1$. 
Thus, 
$\phi $ belongs to $H(f_1)$ if and only if $\phi $ 
belongs to $H_1$. 
Therefore, 
we get $H(f_1)=H_1$.

We treat the case $i=2$ later. 
For $i=3,4,5$, 
we have $\deg _{x_1}f_i=\deg _{x_2}f_i$ 
by (\ref{eq:w(f_i)}). 
Since $f_i$ is tamely reduced over $R$, 
we know that $H(f_i)$ is contained in $\Aff (R,\x )$ 
thanks to Theorem~\ref{thm:reduced coordinate} (i). 
By definition, 
$H_i$ is also contained in $\Aff (R,\x )$. 
So we show that $\phi $ belongs to $H(f_i)$ 
if and only if $\phi $ belongs to $H_i$ 
for each $\phi \in \Aff (R,\x )$. 
Since $\Aff (R,\x )$ is contained in $\T (R,\x )$, 
it suffices to check that $\phi (f_i)=f_i$ 
if and only if $\phi $ belongs to $H_i$.

Assume that $i=3$. 
Set $f_3'=f_3-b$. 
Then, 
we have $\phi (f_3)=f_3$ 
if and only if $\phi (f_3')=f_3'$. 
Write $\phi $ as in (\ref{eq:affine}). 
Then, 
since 
$f_3'=a_1x_1+a_2x_2=(x_1,x_2){\bf a}$, 
we have 
\begin{align*}
&\phi (f_3')
=\phi \bigl((x_1,x_2){\bf a}\bigr)
=(\phi (x_1),\phi (x_2)){\bf a} \\
&\quad =\bigl(
(x_1,x_2)A+(b_1,b_2)\bigr) {\bf a}
=(x_1,x_2)(A{\bf a})+a_1b_1+a_2b_2. 
\end{align*}
Hence, 
we know that $\phi (f_3')=f_3'$ if and only if 
$$
(x_1,x_2)(A{\bf a})+a_1b_1+a_2b_2=
(x_1,x_2){\bf a}, 
$$
and hence if and only if 
$A$, $b_1$ and $b_2$ satisfy (\ref{eq:H_3condition}). 
Thus, 
we have $\phi (f_3)=f_3$ 
if and only if $\phi $ 
belongs to $H_3$. 
This proves that $H(f_3)=H_3$.

We use the following lemma 
for the cases of $i=4,5$.

\begin{lem}\label{lem:inv coord H=H}
Let $\psi \in \Aff (K,\x )$ be such that 
$\psi (x_2+g)=x_2+g$ for some $g\in K[x_1]$ 
with $\deg _{x_1}g\geq 2$. 
Then, 
$\psi $ belongs to $J(K;x_1,x_2)$. 
\end{lem}
\begin{proof}
It suffices to show that 
$\deg _{x_2}\psi (x_1)\leq 0$. 
By assumption, 
we have 
$\psi (g)=\psi \bigl((x_2+g)-x_2\bigr)=x_2+g-\psi (x_2)$. 
Since $\deg _{x_2}g=0$ and $\psi $ is affine, 
this implies that 
$\deg _{x_2}\psi (g)\leq 1$. 
On the other hand, 
we have 
$$
\deg _{x_2}\psi (g)
=(\deg _{x_1}g)\deg _{x_2}\psi (x_1)\geq 2\deg _{x_2}\psi (x_1), 
$$
since $g$ is a polynomial 
in the single variable $x_1$ 
with $\deg _{x_1}g\geq 2$. 
Thus, 
we get $2\deg _{x_2}\psi (x_1)\leq 1$. 
This gives that $\deg _{x_2}\psi (x_1)\leq 0$. 
Therefore, 
$\psi $ belongs to $J(K;x_1,x_2)$. 
\end{proof}

Now, assume that $i=4$. 
Then, 
$\phi $ fixes $f_4=\tau _4(ax_1^2+x_2)$ if and only if 
$\phi ':=\tau _4^{-1}\circ \phi \circ \tau _4$ 
fixes $ax_1^2+x_2$. 
Since $\phi '$ belongs to $\Aff (K,\x )$, 
this implies that $\phi '$ belongs to $J(K;x_1,x_2)$ 
by Lemma~\ref{lem:inv coord H=H}, 
and so implies that 
$\phi '(x_1)=\ep x_1-\alpha $ 
for some $\ep \in K^{\times }$ and $\alpha \in K$. 
Hence, 
$\phi '$ fixes $ax_1^2+x_2$ 
if and only if 
$\phi '(x_1)=\ep x_1-\alpha $ 
and 
\begin{equation}\label{eq:inv coord 4a}
a(\ep x_1-\alpha )^2+\phi '(x_2)=\phi '(ax_1^2+x_2)=ax_1^2+x_2
\end{equation}
for some $\ep \in K^{\times }$ and $\alpha \in K$. 
Since $\phi '$ is affine, 
(\ref{eq:inv coord 4a}) implies that $\ep ^2=1$. 
Hence, 
for $\ep \in K^{\times }$ and $\alpha \in K$, 
we have (\ref{eq:inv coord 4a}) if and only if 
$\ep $ belongs to $\{ 1,-1\} $ and 
$$
\phi '(x_2)=2\ep \alpha ax_1+x_2-\alpha ^2a. 
$$
Since $\phi '(x_1)=\ep x_1-\alpha $, 
we may write $(\phi '(x_1),\phi '(x_2))$ 
as in (\ref{eq:H_4condition}). 
Thus, 
we know that $\phi (f_4)=f_4$ 
if and only if (\ref{eq:H_4condition}) holds 
for some $\ep \in \{ 1,-1\} $ and $\alpha \in K$. 
Therefore, 
we conclude that $H(f_4)=H_4$.

To prove Theorem~\ref{thm:H(f)} for $i=5$, 
we need the following lemma.

\begin{lem}\label{lem:H5}
For $\omega \in K^{\times }$, 
we have 
\begin{equation}\label{eq:lem:H5}
\phi _{\omega }(\tau _5(x_1))=\omega \tau _5(x_1),\quad 
\phi _{\omega }(\tau _5(x_2))=\tau _5(x_2). 
\end{equation}
\end{lem}
\begin{proof}
Put $\phi =\phi _{\omega }$, $\tau =\tau _5$ 
and $\gamma _{i,j}=\gamma _{i,j}(\omega )$ 
for each $i$ and $j$. 
Then, we have 
\begin{equation}\label{lem:H5pf}
\bigl( 
\phi (\tau (x_1)),\phi (\tau (x_2))
\bigr) 
=\bigl( 
\phi (x_1),\phi (x_2)
\bigr) 
\left(\begin{array}{@{\,}cc@{\,}}
\alpha _1& \beta _1\\
\alpha _2& \beta _2
\end{array}\right)
+(\alpha _0,\beta _0) 
\end{equation}
and 
$$
\bigl( \phi (x_1),\phi (x_2)\bigr) 
=(x_1,x_2)
\left(\begin{array}{@{\,}cc@{\,}}
1+\gamma _{1,2} & 
-\gamma _{1,1} \\ 
\gamma _{2,2} & 
1-\gamma _{2,1}
\end{array}\right)
+(\gamma _{0,2},-\gamma _{0,1}). 
$$
A direct computation shows that 
\begin{align*}
& \left(\begin{array}{@{\,}cc@{\,}}
1+\gamma _{1,2} & 
-\gamma _{1,1} \\ 
\gamma _{2,2} & 
1-\gamma _{2,1}
\end{array}\right)
\left(\begin{array}{@{\,}cc@{\,}}
\alpha _1& \beta _1\\
\alpha _2& \beta _2
\end{array}\right) \\
&\quad =
\left(\begin{array}{@{\,}cc@{\,}}
\alpha _1+\gamma _{1,2}\alpha _1-\gamma _{1,1}\alpha _2 & 
\beta _1+\gamma _{1,2}\beta _1-\gamma _{1,1}\beta _2 \\
\gamma _{2,2}\alpha _1+\alpha _2-\gamma _{2,1}\alpha _2 & 
\gamma _{2,2}\beta _1+\beta _2-\gamma _{2,1}\beta _2 
\end{array}\right) 
=\left(\begin{array}{@{\,}cc@{\,}}
\omega \alpha _1& \beta _1\\
\omega \alpha _2& \beta _2
\end{array}\right) 
\end{align*}
and 
$$
(\gamma _{0,2},-\gamma _{0,1})
\left(\begin{array}{@{\,}cc@{\,}}
\alpha _1& \beta _1\\
\alpha _2& \beta _2
\end{array}\right) 
=\frac{(\omega -1)\alpha _0}{\alpha _1\beta _2-\alpha _2\beta _1}
(\beta _2,-\beta _1)
\left(\begin{array}{@{\,}cc@{\,}}
\alpha _1& \beta _1\\
\alpha _2& \beta _2
\end{array}\right) 
=\bigl( (\omega -1)\alpha _0,0
\bigr) . 
$$
Hence, 
the right-hand side of (\ref{lem:H5pf}) is equal to 
\begin{align*}
(x_1,x_2)\left(\begin{array}{@{\,}cc@{\,}}
\omega \alpha _1& \beta _1\\
\omega \alpha _2& \beta _2
\end{array}\right) 
+\bigl( (\omega -1)\alpha _0,0
\bigr) +(\alpha _0,\beta _0)
=\bigl( \omega \tau (x_1),\tau (x_2)\bigr) . 
\end{align*}
Therefore, 
we get 
$\phi (\tau (x_1))=\omega \tau (x_1)$ 
and $\phi (\tau (x_2))=\tau (x_2)$. 
\end{proof}

We show that 
$\phi (f_5)=f_5$ if $\phi $ belongs to 
$H_5=\{ \phi _{\omega }\mid \omega \in Z'\} $. 
Take $\omega \in Z'$ such that $\phi =\phi _{\omega }$. 
Then, 
we have $\omega ^{\mu '}=1$. 
Hence, 
every element of $K[\tau _5(x_1)^{\mu '},\tau _5(x_2)]$ 
is fixed under $\phi _{\omega }$ by Lemma~\ref{lem:H5}. 
Since $g'$ belongs to $K[x_1^{\mu '}]$, 
we see that 
$f_5=\tau _5(x_2+g')$ belongs to 
$K[\tau _5(x_1)^{\mu '},\tau _5(x_2)]$. 
Hence, $f_5$ is fixed under $\phi _{\omega }$. 
Therefore, 
we get $\phi (f_5)=f_5$.

To prove the converse, 
we use the following lemma.

\begin{lem}\label{lem:inv coord H=H2}
Let $f$ be an element of $K[z]\sm K$, 
$m$ the maximal integer for which 
$f$ belongs to $K[z^m]$, 
and $\phi \in \Aut (K[z]/K)$ 
such that $\phi (z)=\alpha z$ 
for some $\alpha \in K^{\times }$. 
If $\phi (f)=f$, then we have $\alpha ^m=1$. 
\end{lem}
\begin{proof}
Let $S$ be the set of $i\in \Z $ 
such that the monomial $z^i$ appears in $f$ 
with nonzero coefficient. 
Then, 
we have $\alpha ^i=1$ for each $i\in S$ 
by the assumption that $\phi (f)=f$. 
Let $m'$ be the positive generator 
of the ideal $(S)$ of $\Z $ generated by $S$. 
Then, 
we may write 
$m'=\sum _{i\in S}in_i$, 
where $n_i\in \Z $ for each $i\in S$. 
Hence, 
we have 
$\alpha ^{m'}=\prod _{i\in S}(\alpha ^i)^{n_i}=1$. 
Since $(S)$ is generated by $m'$, 
we know that $S$ is contained in $m'\Z $. 
Hence, 
$f$ belongs to $K[z^{m'}]$. 
By assumption, 
$f$ also belongs to $K[z^m]$. 
Thus, 
$f$ belongs to $K[z^d]$, 
where $d:=\lcm (m,m')$. 
By the maximality of $m$, 
it follows that $d=m$. 
Hence, $m'$ divides $m$. 
Therefore, 
we conclude that 
$\alpha ^m=(\alpha ^{m'})^{m/m'}=1$. 
\end{proof}

Now, assume that $\phi \in \Aff (R,\x )$ 
fixes $f_5=\tau _5(x_2+g')$. 
Then, 
$\phi ':=\tau _5^{-1}\circ \phi \circ \tau _5$ 
fixes $x_2+g'$. 
Since $\phi '$ belongs to $\Aff (K,\x )$, 
and $\deg _{x_1}g'\geq 3$ by assumption, 
it follows that 
$\phi '$ belongs to $J(K;x_1,x_2)$ 
by Lemma~\ref{lem:inv coord H=H}. 
Write $\phi '(x_1)=\omega x_1+\beta $, 
where $\omega \in K^{\times }$ and $\beta \in K$. 
We show that $\beta =0$. 
Suppose to the contrary that $\beta \neq 0$. 
Then, 
the monomial $x_1^{\lambda _2-1}$ 
appears in 
$\phi '(c_2x_1^{\lambda _2})
=c_2(\omega x_1+\beta )^{\lambda _2}$. 
Set $g''=g'-c_2x_1^{\lambda _2}$. 
Then, 
we have 
$$
\phi '(c_2x_1^{\lambda _2})
=\phi '\bigl( 
(x_2+g')-(x_2+g'')
\bigr) 
=x_2+g'-\phi '(x_2+g''). 
$$
To obtain a contradiction, 
it suffices to check that $x_1^{\lambda _2-1}$ 
does not appear in $g'$ and $\phi '(x_2+g'')$. 
Since $g'$ belongs to $K[x_1^{\mu '}]$, 
we know that $\lambda _2=\deg _{x_1}g'$ 
is divisible by $\mu '$. 
Since $\mu '\geq 2$, 
this implies that 
$\lambda _2-1$ is not divisible by $\mu '$. 
Hence, 
$x_1^{\lambda _2-1}$ does not appear in $g'$. 
Note that $\lambda _2-2\geq 1$ 
by the assumption that $\lambda _2\geq 3$. 
Since $g''$ is an element of 
$K[x_1^{\mu '}]$ with $\deg _{x_1}g''<\lambda _2$, 
we have 
$\deg _{x_1}g''\leq \lambda _2-\mu '\leq \lambda _2-2$. 
Hence, 
$x_2+g''$ has total degree at most $\lambda _2-2$. 
Since an affine automorphism preserves the total degrees, 
it follows that 
$\deg \phi '(x_2+g'')\leq \lambda _2-2$. 
Thus, 
$x_1^{\lambda _2-1}$ 
does not appear in $\phi '(x_2+g'')$. 
Therefore, we are led to a contradiction. 
This proves that $\beta =0$. 
Hence, 
we have $\phi '(x_1)=\omega x_1$, 
and so $\phi '(x_1^i)=\omega ^ix_1^i$ 
for each $i\geq 0$. 
Since $g'$ belongs to $K[x_1^{\mu '}]$, 
it follows that $g'-\phi '(g')$ 
belongs to $x_1^{\mu '}K[x_1^{\mu '}]$. 
On the other hand, 
we have $g'-\phi '(g')=\phi '(x_2)-x_2$, 
since $\phi '(x_2+g')=x_2+g'$. 
Hence, 
$\phi '(x_2)-x_2$ belongs to $x_1^{\mu '}K[x_1^{\mu '}]$. 
Since 
$\phi '$ is affine and 
$\mu '\geq 2$, 
this implies that $\phi '(x_2)=x_2$. 
Consequently, 
we have $\phi '(g')=g'$. 
Since $\mu '$ is the maximal integer 
for which $g'$ belongs to $K[x_1^{\mu '}]$, 
we conclude that $\omega ^{\mu '}=1$ 
by means of Lemma~\ref{lem:inv coord H=H2}. 
Note that 
$\phi (\tau _5(x_1))=\omega \tau _5(x_1)$ 
and $\phi (\tau _5(x_2))=\tau _5(x_2)$, 
since 
$\phi '(x_1)=\omega x_1$, $\phi '(x_2)=x_2$ 
and $\phi '=\tau _5^{-1}\circ \phi \circ \tau _5$. 
By Lemma~\ref{lem:H5}, 
this implies that $\phi =\phi _{\omega }$. 
Because $\phi $ is an element of $\Aff (R,\x )$, 
it follows that 
$\omega =\det J\phi _{\omega }$ belongs to $R^{\times }$. 
Moreover, 
$\phi _{\omega }(x_i)=\phi (x_i)$ belongs to $\Rx $ 
for $i=1,2$. 
From this, we see that 
$\gamma _{i,j}(\omega )$ belongs to $R$ 
for $i=0,1,2$ and $j=1,2$. 
Thus, 
$\omega $ belongs to $Z'$. 
Therefore, 
$\phi =\phi _{\omega }$ belongs to $H_5$. 
This completes the proof of 
Theorem~\ref{thm:H(f)} for $i=5$.

Here is a useful identity to be 
used in the following discussions. 
Let $A$ be a domain, 
and $\gamma _i$ and $p_i$ 
elements of $A$ for $i=1,2$. 
Define an endomorphism $\phi $ of 
the $A$-algebra $A[z]$ by $\phi (z)=\gamma _1z+p_1$, 
and define 
$$
z'=\gamma _2z+p_2\quad\text{and}\quad 
p'=\gamma _2p_1-(\gamma _1-1)p_2. 
$$
Then, it holds that 
\begin{equation}\label{eq:ch1:useful id}
\phi (z')
=\gamma _2\phi (z)
+p_2
=\gamma _1(\gamma _2z+p_2)-(\gamma _1-1)p_2+\gamma _2p_1 
=\gamma _1z'+p'.
\end{equation}

Now, 
we show that $H_2$ is contained in $H(f_2)$. 
Take any $\phi \in H_2$. 
Then, 
we have $\phi (x_1)=x_1$ and 
$\phi (x_2)=\omega x_2+g_{\omega }$ 
for some $\omega \in Z$. 
In the notation above 
with $A=R[x_1]$, $z=x_2$ and 
$(\gamma _1,p_1,\gamma _2,p_2)
=(\omega ,g_{\omega },\zeta -1,g)$, 
we have 
$\phi (z)=\omega x_2+g_{\omega }=\phi (x_2)$, 
$z'=(\zeta -1)x_2+g=y_2$ 
and $p'=(\zeta -1)g_{\omega }-(\omega -1)g=0$. 
Hence, we get $\phi (y_2)=\gamma _1z'+p'=\omega y_2$ 
by (\ref{eq:ch1:useful id}). 
Since $\omega ^{\mu }=1$, 
it follows that 
$\phi $ fixes every element of $K[x_1,y_2^{\mu }]$. 
Hence, 
$\phi $ fixes $f_2=ax_1+h$, 
since $h$ belongs to $K[y_2^{\mu }]$ 
by the choice of $\mu $. 
Clearly, 
$\phi $ belongs to $\T (R,\x )$. 
Thus, $\phi $ belongs to $H(f_2)$. 
Therefore, 
$H_2$ is contained in $H(f_2)$.

The reverse inclusion is proved by using 
the following lemma and proposition, 
where we denote $H(x_1)=\Aut (\Rx /R[x_1])$ for simplicity.

\begin{lem}\label{lem:H_2=H(f_2)cap H(x_1)}
$H_2':=H(f_2)\cap H(x_1)$ 
is contained in $H_2$. 
\end{lem}
\begin{proof}
Take any $\phi \in H_2'$. 
Then, we have $\phi (x_1)=x_1$. 
Hence, 
we may write $\phi (x_2)=\omega x_2+p$, 
where $\omega \in R^{\times }$ and $p\in R[x_1]$.	 
We show that $\omega ^{\mu }=1$ and $p=g_{\omega }$. 
Then, 
it follows that $\phi $ belongs to $H_2$. 
By applying (\ref{eq:ch1:useful id}) with 
$A=R[x_1]$, $z=x_2$ and 
$(\gamma _1,p_1,\gamma _2,p_2)=(\omega ,p,\zeta -1,g)$, 
we get $\phi (y_2)=\omega y_2+q$, 
where 
$$
q:=(\zeta -1)p-(\omega -1)g. 
$$
Then, 
we have $p=g_{\omega }$ if and only if $q=0$. 
So we prove that $q=0$. 
Note that $h=f_2-ax_1$ is fixed under $\phi $, 
since $\phi $ belongs to $H(f_2)\cap H(x_1)$. 
Hence, 
we have 
$$
c_1(\omega y_2+q)^{\lambda _1}=\phi (c_1y_2^{\lambda _1})=
\phi (h-h')=h-\phi (h'), 
$$
where $h':=h-c_1y_2^{\lambda _1}$. 
Now, 
suppose to the contrary that $q\neq 0$. 
Then, 
the monomial $y_2^{\lambda _1-1}$ 
appears in $(\omega y_2+q)^{\lambda _1}$ 
as a polynomial in $y_2$ over $R'[x_1]$. 
Hence, 
$y_2^{\lambda _1-1}$ 
appears in $h$ or $\phi (h')$ 
by the preceding equality. 
Since $h$ belongs to $R'[y_2^{\mu }]$, 
we know that $\lambda _1=\deg _{y_2}h$ 
is divisible by $\mu $. 
Since $\mu \geq 2$, 
this implies that 
$\lambda _1-1$ is not divisible by $\mu $. 
Hence, 
$y_2^{\lambda _1-1}$ does not appear in $h$. 
Note that 
$\deg _{y_2}h'\leq \lambda _1-\mu \leq \lambda _1-2$, 
since $h'$ is an element of 
$R'[y_2^{\mu }]$ with $\deg _{y_2}h'<\lambda _1$. 
Because $\phi (y_2)=\omega y_2+q$, 
it follows that $\deg _{y_2}\phi (h')=\deg _{y_2}h'\leq \lambda _1-2$. 
Hence, 
$y_2^{\lambda _1-1}$ does not appear in $\phi (h')$. 
This is a contradiction, 
thus proving $q=0$. 
Therefore, we get $p=g_{\omega }$ 
and $\phi (y_2)=\omega y_2$. 
Since $\phi (h)=h$, 
and $\mu $ is the maximal number such that 
$h$ belongs to $R'[y_2^{\mu }]$, 
it follows that $\omega ^{\mu }=1$ 
by Lemma~\ref{lem:inv coord H=H2}. 
This proves that $\phi $ belongs to $H_2$. 
Therefore, 
$H_2'$ is contained in $H_2$. 
\end{proof}

The following proposition is a key 
to the proof of the ``if" part of 
Theorem~\ref{thm:Gamma (f)}, 
and Theorem~\ref{thm:H(f)} for $i=2$.

\begin{prop}\label{prop:inv}
Let $f\in \Rx $ be a coordinate of $\Kx $ over $K$ 
which is tamely reduced over $R$. 

\noindent{\rm (i)} 
Assume that $\deg _{x_1}f>\deg _{x_2}f\geq 2$. 
If $\phi (f)=f$ for some 
$\phi \in J(R;x_1,x_2)$ with $\phi \neq \id _{\Rx }$, 
then $\phi $ belongs to $H(x_1)$ and $f$ is of type II.

\noindent{\rm (ii)} 
Assume that $\deg _{x_1}f=\deg _{x_2}f\geq 2$. 
If $\phi (f)=f$ for some $\phi \in \Aff (R,\x)$ 
with $\phi \neq \id _{\Rx }$, 
then $f$ is of type IV, 
or $f$ satisfies the condition $(5)$ of 
Definition~$\ref{def:invariant coord}$ 
except for {\rm (c)}. 
\end{prop}

We prove Proposition~\ref{prop:inv} in the next section. 
In the rest of this section, 
we derive from Proposition~\ref{prop:inv} 
the ``if" part of Theorem~\ref{thm:Gamma (f)}, 
and Theorem~\ref{thm:H(f)} for $i=2$.

Let $f\in \Rx $ be a coordinate of $\Kx $ over $K$ 
which satisfies (A), (B) and (C) of Theorem~\ref{thm:Gamma (f)}. 
We show that $f$ is of one of the types I through V\null. 
Since $\deg _{x_1}f\geq \deg _{x_2}f\geq 1$ by (C), 
we have $|\w (f)|>1$. 
Hence, 
there exist $l,m\in \N $ 
and $\alpha ,\beta \in K^{\times }$ 
such that 
$\deg _{x_1}f=lm$, $\deg _{x_2}f=m$ 
and $f^{\w (f)}=\alpha (x_2+\beta x_1^l)^m$ 
by Proposition~\ref{prop:slope}. 
Because $f$ is tamely reduced over $R$ by (B), 
it follows from 
Proposition~\ref{prop:reduced coordinate} that 
$\beta $ does not belong to $V(R)$ if $l=1$, 
and to $R$ if $l\geq 2$.

Assume that $l=m=1$. 
Then, 
we have 
$\deg _{x_1}f=\deg _{x_2}f=1$ and 
$f^{\w (f)}=\alpha (x_2+\beta x_1)$. 
Hence, 
$f$ has the form 
$\alpha x_2+\alpha \beta x_1+b$ 
for some $b\in R$. 
Since $f$ belongs to $\Rx $, 
it follows that 
$a_2:=\alpha $ and $a_1:=\alpha \beta $ 
belong to $R\sm \zs $. 
Since $l=1$, 
we know that $a_1/a_2=\beta $ 
does not belong to $V(R)$. 
Therefore, 
$f$ is of type III.

Assume that $l\geq 2$ and $m=1$. 
Then, 
we have 
$\deg _{x_1}f=l$, $\deg _{x_2}f=1$ 
and $f^{\w (f)}=\alpha (x_2+\beta x_1^l)$. 
From this, we see that 
$g:=f-\alpha x_2$ is an element of $R[x_1]$ 
of degree $l\geq 2$ 
with leading coefficient $\alpha \beta $. 
Since $l\geq 2$, 
we know that 
$\beta $ does not belong to $R$. 
Hence, 
$\alpha \beta $ does not belong to $\alpha R$. 
Therefore, 
$f$ is of type I.

Assume that $l\geq 2$ and $m\geq 2$. 
Then, 
we have $\deg _{x_1}f>\deg _{x_2}f\geq 2$. 
Thanks to Theorem~\ref{thm:reduced coordinate} (ii), 
this implies that 
$H(f)$ is contained in $J(R;x_1,x_2)$ 
because of (B). 
By (A), 
there exists $\phi \in H(f)$ 
with $\phi \neq \id _{\Rx }$. 
Then, 
$\phi $ belongs to $J(R;x_1,x_2)$ 
and satisfies $\phi (f)=f$. 
Therefore, 
$f$ is of type II 
due to Proposition~\ref{prop:inv} (i).

Finally, 
assume that $l=1$ and $m\geq 2$. 
Then, 
we have 
$\deg _{x_1}f=\deg _{x_2}f\geq 2$. 
Thanks to Theorem~\ref{thm:reduced coordinate} (i), 
this implies that $H(f)$ is contained in $\Aff (R,\x )$ 
because of (B). 
By (A), 
there exists $\phi \in H(f)$ 
with $\phi \neq \id _{\Rx }$. 
Then, 
$\phi $ belongs to $\Aff (R,\x )$ 
and satisfies $\phi (f)=f$. 
Due to Proposition~\ref{prop:inv} (ii), 
this implies that $f$ is of type IV, 
or satisfies the condition of 
Definition~\ref{def:invariant coord} (5) 
except for (c).

We show that $f$ is of type V in the latter case. 
By assumption, 
we may write $f=\tau _5(x_2+g')$. 
Here, 
$\tau _5$ is an element of $\Aff (K,\x )$ 
satisfying (b) of 
Definition~\ref{def:invariant coord} (5), 
and $g'$ is an element of $x_1^{e'}K[x_1^{e'}]$ 
for some $e'\geq 2$ with $\deg _{x_1}g'\geq 3$. 
By Theorem~\ref{thm:H(f)} for $i=5$, 
we have $H(f)=H_5$. 
Hence, 
we get $Z'\sm  \{ 1\} \neq \emptyset $ by (A). 
Take $\omega \in Z'$ with $\omega \neq 1$. 
Then, we have $\omega ^{\mu '}=1$ and 
$\gamma _{i,j}(\omega )$ belongs to $R$ 
for $i=0,1,2$ and $j=1,2$ by definition. 
Since $g'$ belongs to $K[x_1^{\mu '}]$ and $x_1^{e'}K[x_1^{e'}]$, 
we know that $g'$ belongs to $x_1^{\mu '}K[x_1^{\mu '}]$. 
Thus, 
$f$ satisfies (a) and (c) of 
Definition~\ref{def:invariant coord} (5) 
with $e'$ replaced by $\mu '$. 
Therefore, $f$ is of type V\null.

This proves that 
$f$ is of one of the types I through V\null. 
Therefore, 
the``if" part of Theorem~\ref{thm:Gamma (f)} 
follows from Proposition~\ref{prop:inv}.

To complete the proof of 
Theorem~\ref{thm:H(f)} for $i=2$, 
it remains only to prove that 
$H(f_2)$ is contained in $H_2$. 
Thanks to Lemma~\ref{lem:H_2=H(f_2)cap H(x_1)}, 
it suffices to verify that 
$H(f_2)$ is contained in $H(x_1)$. 
By (\ref{eq:w(f_i)}), 
we have $\deg _{x_1}f_2>\deg _{x_2}f_2\geq 2$. 
In view of Theorem~\ref{thm:reduced coordinate} (ii), 
this implies that 
$H(f_2)$ is contained in $J(R;x_1,x_2)$, 
since $f_2$ is tamely reduced over $R$. 
By Proposition~\ref{prop:inv} (i), 
we know that 
$$
\Aut (\Rx /R[f_2])\cap J(R;x_1,x_2)
$$ 
is contained in $H(x_1)$. 
Thus, $H(f_2)$ is contained in $H(x_1)$. 
This proves that $H(f_2)$ is contained in $H_2$. 
Therefore, 
Proposition~\ref{prop:inv} (i) implies 
Theorem~\ref{thm:H(f)} for $i=2$.

\section{Proof (II)} 
\setcounter{equation}{0}
\label{sect:tame intersection pf 2}

The goal of this section is to prove 
Proposition~\ref{prop:inv}. 
Recall that elements of $\Aut (\Rx /R)$ 
are naturally regarded as 
elements of $\Aut (A[\x ]/A)$ 
for any $R$-domain $A$. 
For each $\phi \in \Aut (\Rx /R)$, 
we define an $A$-subalgebra of $A[\x ]$ by 
$$
A[\x ]^{\phi }=\{ f\in A[\x ]\mid \phi (f)=f\} . 
$$

\begin{lem}\label{lem:k[x^t,y]}
\noindent{\rm (i)} 
Let $t\geq 2$ be an integer, 
and $f\in R[x_1^t,x_2]$ 
a coordinate of $\Kx $ over $K$. 
Then, 
we have $f=\alpha x_2+p$ 
for some $\alpha \in R\sm \zs $ 
and $p\in R[x_1^t]$.

\noindent{\rm (ii)} 
Let $\phi \in \Aut (\Kx /K)$ 
be such that $\phi (x_1)=\alpha x_1$ 
and $\phi (x_2)=x_2+q$ for some $\alpha \in K^{\times }$ 
and $q\in K[x_1]\sm \zs $ with $\phi (q)=q$. 
Then, 
$\Kx ^{\phi }$ is contained in $K[x_1]$. 
\end{lem}
\begin{proof}
(i) 
Let $U$ be the set of $f\in K[x_1^t,x_2]$ 
which is a coordinate of $\Kx $ over $K$, 
but is not of the form 
$f=\alpha x_2+p$ 
for any $\alpha \in K^{\times }$ and $p\in K[x_1^t]$.  
For each $q\in K[x_1^t]$, 
we define $\psi _q\in \Aut (\Kx /K[x_1])$ 
by $\psi _q(x_2)=x_2-q$. 
Then, 
we remark that $\psi _q(U)$ 
is contained in $U$. 
Now, 
suppose to the contrary that 
there exists a coordinate 
$f'\in R[x_1^t,x_2]$ of $\Kx $ over $K$ 
which is not of the form 
$f'=\alpha x_2+p$ 
for any $\alpha \in R\sm \zs $ 
and $p\in R[x_1^t]$.  
Then, 
$f'$ belongs to $U$. 
Actually, 
if $f'=\alpha x_2+p$ 
for some $\alpha \in K^{\times }$ and $p\in K[x_1^t]$, 
then $\alpha $ and $p$ belong to 
$R\sm \zs $ and $R[x_1^t]$, respectively. 
Hence, 
$U$ is not empty. 
Choose $f\in U$ so that $|\w (f)|$ is minimal. 
First, 
we show that $|\w (f)|>1$. 
Suppose to the contrary that $|\w (f)|=1$. 
Then, 
$f$ is a linear polynomial in $x_i$ over $K$ 
for some $i\in \{ 1,2\} $. 
Since $f$ is an element of $K[x_1^t,x_2]$ with $t\geq 2$, 
we know that $i=2$. 
Hence, we have $f=\alpha x_2+\beta $ 
for some $\alpha \in K^{\times }$ and $\beta \in K$. 
Thus, 
$f$ does not belong to $U$, 
a contradiction. 
Therefore, 
we get $|\w (f)|>1$. 
By Proposition~\ref{prop:slope}, 
there exist 
$i,j\in \{ 1,2\} $ with $i\neq j$, 
$l,m\in \N $ and 
$\alpha ,\beta \in K^{\times }$ such that 
$\deg _{x_i}f=m$, $\deg _{x_j}f=lm$ and 
$f^{\w (f)}=\alpha (x_i+\beta x_j^l)^m$. 
Since $K$ is of characteristic zero, 
the monomials 
$x_ix_j^{l(m-1)}$ and $x_i^{m-1}x_j^{l}$ 
appear in $f^{\w (f)}$, 
and hence appear in $f$. 
Because $f$ is an element of $K[x_1^t,x_2]$ 
with $t\geq 2$, 
it follows that $(i,j)=(2,1)$, 
and $l$ is a multiple of $t$. 
Hence, 
$q:=\beta x_1^l$ belongs to $K[x_1^t]$. 
Put $\psi =\psi _q$. 
Then, 
$\psi (U)$ is contained in $U$ as remarked. 
Hence, 
$\psi (f)$ belongs to $U$. 
Thus, 
we get $|\w (\psi (f))|\geq |\w (f)|$ 
by the minimality of $|\w (f)|$. 
To obtain a contradiction, 
we show that $|\w (\psi (f))|<|\w (f)|$ 
using Lemma~\ref{lem:w-homoge autom}. 
Since $f$ is a coordinate of $\Kx $ over $K$ 
with $|\w (f)|>0$, 
we know that $f$ satsfies the equivalent conditions (a) and (b) 
before Proposition~\ref{prop:slopee}. 
By definition, 
we have $\psi (x_1)=x_1$ and 
$\psi (x_2)=x_2-\beta x_1^l$. 
Since $\w (f)=(m,lm)$, 
we see that $\psi $ is $\w (f)$-homogeneous. 
Since 
$$
\psi (f^{\w (f)})=\psi \bigl( 
\alpha (x_2+\beta x_1^l)^m
\bigr) 
=\alpha x_2^m, 
$$
we have 
$\deg _{x_1}\psi (f^{\w (f)})<\deg _{x_1}f^{\w (f)}$. 
Thus, it follows that 
$|\w (\psi (f))|<|\w (f)|$ 
by Lemma~\ref{lem:w-homoge autom}. 
Therefore, 
we are led to a contradiction, 
proving (i).

(ii) 
Define a $K$-linear endomorphism of 
$\Kx $ by 
$\delta :=\phi -\id _{\Kx }$. 
Then, we have $\ker \delta =\Kx ^{\phi }$, 
and 
$$
\delta (x_1^i)
=\phi (x_1)^i-x_1^i
=\alpha ^ix_1^i-x_1^i
=(\alpha ^i-1)x_1^i
$$ 
for each $i\in \Zn $. 
If $\delta ^2(cx_1^i)=c(\alpha ^i-1)^2x_1^i=0$ 
for $c\in K$ and $i\in \Zn $, 
then we have $c=0$ or $\alpha ^i=1$, 
and hence $\delta (cx_1^i)=c(\alpha ^i-1)x_1^i=0$. 
Thus, 
$\delta ^2(f)=0$ implies 
$\delta (f)=0$ for each $f\in K[x_1]$.

Now, 
take any $p\in \Kx ^{\phi }\sm \zs $, 
and write $p=\sum _{i=0}^lp_{l-i}x_2^i$, 
where $l\in \Zn $ and 
$p_0,\ldots ,p_l\in K[x_1]$ 
with $p_0\neq 0$. 
Then, 
it suffices to show that $l=0$. 
Since $\phi (K[x_1])$ is contained in $K[x_1]$, 
and $\phi (x_2)=x_2+q$, 
we have 
\begin{align*}
p&=\phi (p)=
\sum _{i=0}^l\phi (p_{l-i})(x_2+q)^i \\
&=\phi (p_0)x_2^l+(l\phi (p_0)q+\phi (p_1))x_2^{l-1}
+(\text{terms of lower degree in }x_2). 
\end{align*}
Hence, 
we get $\phi (p_0)=p_0$ 
and $l\phi (p_0)q+\phi (p_1)=p_1$. 
Thus, 
$p_0$ belongs to $\Kx ^{\phi }$ 
and 
$$
\delta (p_1)=\phi (p_1)-p_1=-l\phi (p_0)q=-lp_0q. 
$$ 
Since $q$ belongs to $\Kx ^{\phi }$ by assumption, 
$-lp_0q$ belongs to $\Kx ^{\phi }=\ker \delta $. 
Hence, it follows that $\delta ^2(p_1)=0$. 
This implies that 
$\delta (p_1)=0$ as mentioned. 
Thus, we get $lp_0q=0$. 
Since $p_0$ and $q$ are nonzero, 
we conclude that $l=0$ 
due to the assumption that $R$ contains $\Z $. 
Therefore, 
$p$ belongs to $K[x_1]$, 
proving (ii). 
\end{proof}

In the following discussion and two lemmas, 
$n\in \N $ may be arbitrary. 
For each endomorphism $\phi $ of the $R$-algebra $\Rx $ 
and each matrix 
$A=(a_{i,j})_{i,j}$ with entries in $\Rx $, 
we denote $\phi (A)=(\phi (a_{i,j}))_{i,j}$. 
Then, 
by chain rule, 
we have 
\begin{equation}\label{eq:Jacobian chain rule}
J(\phi \circ \psi )=\phi (J\psi )\cdot J\phi 
\end{equation}
for endomorphisms $\phi $ and $\psi $ 
of the $R$-algebra $\Rx $. 
If $\psi $ is an automorphism, 
then we get 
\begin{equation}\label{eq:chain}
J(\psi^{-1}\circ \phi \circ \psi )
=(\psi ^{-1}\circ \phi )(J\psi )\cdot 
\psi ^{-1}(J\phi )\cdot 
J(\psi ^{-1}). 
\end{equation}
Now, 
define an endomorphism $\ep _0$ 
of the $R$-algebra $\Rx $ 
by $\ep _0(x_i)=0$ for $i=1,\ldots,n$. 
Assume that 
$\ep _0\circ \phi =\ep _0\circ \psi =\ep _0$, 
i.e., 
$\phi (x_i)$ and $\psi (x_i)$ 
have no constant terms 
for each $i$. 
Then, 
the matrices 
$\ep _0\bigl(J(\psi^{-1}\circ \phi \circ \psi )\bigr)$ 
and $\ep _0(J\phi )$ are similar 
for the following reason. 
Since 
$\ep _0=(\ep _0\circ \psi )\circ \psi ^{-1}
=\ep  _0\circ \psi ^{-1}$, 
we have 
\begin{align*}
&\ep _0(J\psi )\cdot \ep _0\bigl(J(\psi ^{-1})\bigr)=
(\ep _0\circ \psi ^{-1})(J\psi )
\cdot \ep _0\bigl(J(\psi ^{-1})\bigr) \\
&\quad =\ep _0\bigl(\psi ^{-1}(J\psi )
\cdot J(\psi ^{-1})\bigr)
=\ep _0\bigl(J(\psi ^{-1}\circ \psi )\bigr)=E. 
\end{align*}
Hence, 
we get $\ep _0\bigl(J(\psi ^{-1})\bigr)
=\ep _0(J\psi )^{-1}$. 
Thus, 
it follows from (\ref{eq:chain}) that 
\begin{align*}
\ep _0\bigl(J(\psi^{-1}\circ \phi \circ \psi )\bigr)
&=\ep _0\bigl((\psi ^{-1}\circ \phi )(J\psi )\bigr) \cdot 
\ep _0\bigl(\psi ^{-1}(J\phi )\bigr) \cdot 
\ep _0\bigl(J(\psi ^{-1})\bigr)\\
&=(\ep _0\circ \psi ^{-1}\circ \phi )(J\psi ) \cdot 
(\ep _0\circ \psi ^{-1})(J\phi )\cdot 
\ep _0(J\psi )^{-1}\\
&=\ep _0(J\psi )\cdot \ep _0(J\phi )\cdot \ep _0(J\psi )^{-1}.
\end{align*}
Therefore, 
$\ep _0(J(\psi^{-1}\circ \phi \circ \psi ))$ 
and $\ep _0(J\phi )$ are similar.

\begin{lem}\label{lem:eigenvalue}
Let $\kappa $ be any field, 
and $\phi \in J(\kappa ;x_1,\ldots ,x_n)$ 
such that $\phi (f)=f$ 
for some coordinate $f$ of $\kapx $ over $\kappa $. 
Then, we have $\phi (x_i)=x_i+g$ 
for some $i\in \{ 1,\ldots ,n\} $ and 
$g\in \kappa [x_1,\ldots ,x_{i-1}]$. 
\end{lem}
\begin{proof}
Write $\phi (x_i)=\alpha _ix_i+g_i$ 
for $i=1,\ldots ,n$, 
where $\alpha _i\in \kappa ^{\times }$ 
and $g_i\in \kappa [x_1,\ldots ,x_{i-1}]$. 
We show that $\alpha _i=1$ for some $i$ 
by contradiction. 
Suppose that 
$\alpha _i\neq 1$ for all $i$. 
Then, 
we can define $c_1,\ldots ,c_n\in \kappa $ by 
$$
c_i:=-(\alpha _i-1)^{-1}
g_i(c_1,\ldots ,c_{i-1})
$$
by induction on $i$. 
Define $\tau \in \Aut (\kapx /\kappa )$ by 
$\tau (x_i)=x_i-c_i$ for $i=1,\ldots ,n$. 
Then, 
we have 
\begin{equation}\label{eq:eigen 1}
\begin{aligned}
&(\phi \circ \tau )(x_i)
=\phi (x_i-c_i)=\alpha _ix_i+g_i-c_i \\
&\quad =\alpha _i(x_i-c_i)+g_i+(\alpha _i-1)c_i
=\alpha _i(x_i-c_i)+g_i-g_i(c_1,\ldots ,c_{i-1})
\end{aligned}
\end{equation}
for $i=1,\ldots ,n$. 
Set $\phi _0=\tau ^{-1}\circ \phi \circ \tau $. 
Then, 
we see from (\ref{eq:eigen 1}) that 
$\phi _0(x_i)=\tau ^{-1}\bigl((\phi \circ \tau )(x_i)\bigr)$ 
has the form $\alpha _ix_i+g_i'$ for some 
$g_i'\in \kappa [x_1,\ldots ,x_{i-1}]$ for each $i$. 
Hence, 
$J\phi _0$ is a lower triangular matrix 
with diagonal entries $\alpha _1,\ldots ,\alpha _n$. 
Thus, the same holds for $\ep _0(J\phi _0)$. 
Therefore, 
$\alpha _1,\ldots ,\alpha _n$ 
are the eigenvalues of $\ep _0(J\phi _0)$.

Observe that (\ref{eq:eigen 1}) is sent to zero 
by the substitution 
$x_j\mapsto c_j$ for $j=1,\ldots ,i$. 
Since 
$(\ep _0\circ \tau ^{-1})(x_j)
=\ep _0(x_j+c_j)=c_j$ 
for $j=1,\ldots ,n$, 
it follows that 
$$
(\ep _0\circ \phi _0)(x_i)
=(\ep _0\circ \tau ^{-1})\bigl(
(\phi \circ \tau )(x_i)\bigr)=0. 
$$
Hence, 
we get $\ep _0\circ \phi _0=\ep _0$. 
Since $f$ is a coordinate of $\kapx $ 
over $\kappa $ by assumption, 
there exists $\sigma \in \Aut (\kapx /\kappa )$ 
such that $\sigma (x_1)=f$. 
Define $\sigma _0\in \Aut (\kapx /\kappa )$ by 
$$
\sigma _0(x_i)
=(\tau ^{-1}\circ \sigma )(x_i)
-\ep _0\bigl( (\tau ^{-1}\circ \sigma )(x_i)\bigr)
$$
for $i=1,\ldots ,n$. 
Then, 
we have $\ep _0(\sigma _0(x_i))=0$ for each $i$. 
Hence, we get 
$\ep _0\circ \sigma _0=\ep _0$. 
Set 
$\phi _1=\sigma _0^{-1}\circ \phi _0\circ \sigma _0$. 
Then, 
$\ep _0(J\phi _0)$ and $\ep _0(J\phi _1)$ are similar 
by the discussion above. 
Accordingly, 
$\alpha _1,\ldots ,\alpha _n$ 
are the eigenvalues of $\ep _0(J\phi _1)$. 
Since $\phi (\sigma (x_1))=\phi (f)=f=\sigma (x_1)$, 
and $\beta :=\ep _0\bigl( (\tau ^{-1}\circ \sigma )(x_1)\bigr)$ 
is a constant, 
we have 
\begin{align*}
&(\phi _0\circ \sigma _0)(x_1)
=(\tau ^{-1}\circ \phi \circ \tau )
\bigl( 
(\tau ^{-1}\circ \sigma )(x_1)
-\beta \bigr) \\
&\quad =\tau ^{-1}\bigl( 
\phi (\sigma (x_1))
\bigr) -
(\tau ^{-1}\circ \phi \circ \tau )(\beta ) 
=(\tau ^{-1}\circ \sigma )(x_1)-\beta 
=\sigma _0(x_1). 
\end{align*}
Hence, we get 
$$
\phi _1(x_1)
=(\sigma _0^{-1}\circ \phi _0\circ \sigma _0)(x_1)
=\sigma _0^{-1}\bigl( 
(\phi _0\circ \sigma _0)(x_1)
\bigr) 
=\sigma _0^{-1}\bigl( 
\sigma _0(x_1)
\bigr) 
=x_1. 
$$
This shows that the first row of $J\phi _1$ 
is $(1,0,\ldots ,0)$, 
and the same holds for $\ep _0(J\phi _1)$. 
Thus, 
$1$ is an eigenvalue of $\ep _0(J\phi _1)$. 
Since $\alpha _1,\ldots ,\alpha _n$ are the eigenvalues of 
$\ep _0(J\phi _1)$, 
this contradicts that $\alpha _i\neq 1$ for all $i$. 
Therefore, 
we have $\alpha _i=1$ for some $i$. 
\end{proof}

We also use locally nilpotent derivations 
to prove Proposition~\ref{prop:inv}. 
For each domain $A$, 
we denote by $Q(A)$ the field of fractions of $A$. 
Then, 
the following facts are well-known. 
Here, 
we recall that $k$ 
is an arbitrary field of characteristic zero. 
For $D\in \lnd _k\kx \sm \zs $, 
the transcendence degree of $Q(\ker D)$ over $k$ 
is equal to $n-1$ 
(cf.\ \cite[Proposition 1.3.32]{Essen}). 
If $D(f)\neq 0$ for $D\in \Der _k\kx $ and $f\in \kx $, 
then $f$ is transcendental over $Q(\ker D)$ 
(cf.~\cite[Proposition 5.2]{Lang}).

The following lemma is a consequence of these facts.

\begin{lem}\label{lem:e-n}
For $D\in \lnd _k\kx $, 
we put $\phi =\exp D$. Then, 
we have $\kx ^{\phi }=\ker D$. 
\end{lem}
\begin{proof}
Obviously, 
$\ker D$ is contained in $\kx ^{\phi }$. 
We prove the reverse inclusion by contradiction. 
Suppose that 
$D(f)\neq 0$ for some $f\in \kx ^{\phi }$. 
Then, 
$f$ is transcendental over $Q(\ker D)$ as mentioned.
Clearly, 
$D$ is nonzero. 
Hence, 
$Q(\ker D)$ has transcendence degree 
$n-1$ over $k$ as mentioned. 
Thus, 
the transcendence degree of $Q(\ker D)(f)$ 
over $k$ is equal to $n$. 
Since 
$L:=Q(\kx ^{\phi })$ contains $Q(\ker D)(f)$, 
we know that 
$\kxr :=Q(\kx )$ is a finite extension of $L$. 
Note that $\phi $ extends 
to an automorphism of $\kxr $ over $L$. 
Hence, 
$\phi $ must be of finite order. 
This contradicts that 
$\phi =\exp D$ with $D\neq 0$. 
Thus, 
$\kx ^{\phi }$ is contained in $\ker D$. 
Therefore, 
we have $\kx ^{\phi }=\ker D$. 
\end{proof}

For $\phi \in \Aut (\kx /k)$, 
consider the $k$-linear endomorphism 
$\delta =\phi -\id _{\kx }$ of $\kx $. 
For each $f,g\in \kx $, 
we have 
$$
\delta (fg)=\phi(fg)-fg
=\bigl(\phi (f)-f\bigr)g
+\phi (f)\bigl( \phi (g)-g\bigr) 
=\delta (f)g+\phi (f)\delta (g). 
$$
Since $\delta \circ \phi =\phi \circ \delta $, 
it follows that 
$$
\delta ^l(fg)=\sum _{i=0}^l\binom{l}{i}
\phi ^i(\delta ^{l-i}(f))\delta ^i(g)
$$
for each $l\geq 1$. 
By this formula, 
we know that $\delta ^{l+m}(fg)=0$ 
if $\delta ^l(f)=0$ and $\delta ^m(g)=0$ 
for $l,m\in \N $. 
Hence, 
we see that 
$\delta $ is locally nilpotent if 
$\delta ^{l_i}(x_i)=0$ for some $l_i\in \N $ 
for $i=1,\ldots ,n$.

\begin{lem}\label{lem:e-n2}
Let $\phi \in \Aut (\kx /k)$ be such that 
$\phi (x_i)=x_i+g_i$ for some 
$g_i\in k[x_1,\ldots ,x_{i-1}]$ 
for $i=1,\ldots ,n$. 
Then, 
we have $\phi =\exp D$ 
for some triangular derivation 
$D$ of $\kx $ over $k$. 
\end{lem}
\begin{proof}
By induction on $n$, 
we can check that 
$\delta =\phi -\id _{\kx }$ is locally nilpotent. 
In fact, 
assuming that the restriction of 
$\delta $ to $k[x_1,\ldots ,x_{n-1}]$ 
is locally nilpotent, 
we have $\delta ^l(x_n)=0$ for some $l\in \N $, 
since $\delta (x_n)=(x_n+g_n)-x_n=g_n$ 
belongs to $k[x_1,\ldots ,x_{n-1}]$. 
This implies that $\delta $ is locally nilpotent 
by the discussion above. 
Due to 
van den Essen~\cite[Proposition 2.1.3]{Essen} 
and Nowicki~\cite[Proposition 6.1.4 (6)]{Now}, 
it follows that 
$$
D:=\sum _{i\geq 1}(-1)^{i+1}
\frac{\delta ^i}{i}
$$ 
belongs to $\lnd _k\kx $ and satisfies 
$\phi =\exp D$. 
Since 
$\delta (x_i)=(x_i+g_i)-x_i=g_i$ belongs to 
$k[x_1,\ldots ,x_{i-1}]$, 
and $\delta (k[x_1,\ldots ,x_{i-1}])$ 
is contained in $k[x_1,\ldots ,x_{i-1}]$, 
we see that $D(x_i)$ 
belongs to 
$k[x_1,\ldots ,x_{i-1}]$ for $i=1,\ldots ,n$. 
Therefore, $D$ is triangular. 
\end{proof}

Now, 
we prove Proposition~\ref{prop:inv} (i). 
Let $f\in \Rx $ be a coordinate of $\Kx $ over $K$ 
which is tamely reduced over $R$ 
and satisfies $\deg _{x_1}f>\deg _{x_2}f\geq 2$. 
Take 
$\id _{\Rx }\neq \phi \in J(R;x_1,x_2)$ 
such that $\phi (f)=f$, 
and write 
\begin{equation}\label{eq:prop inv pf phi}
\phi (x_1)=u_1x_1+v \text{ \ and \ }\phi (x_2)=u_2x_2+g, 
\end{equation}
where $u_1,u_2\in R^{\times }$, 
$v\in R$ and $g\in R[x_1]$. 
Then, 
we have $u_1=1$ or $u_2=1$ 
by Lemma~\ref{lem:eigenvalue}. 
Since $f$ is a coordinate of $\Kx $ over $K$ 
with $\deg _{x_1}f>\deg _{x_2}f\geq 2$, 
we see from Proposition~\ref{prop:slope} that 
$f$ has the form 
$\alpha x_2^l+(\text{terms of lower degree in }x_2)$ 
for some $\alpha \in R\sm \zs $ and $l\geq 2$. 
Then, 
we have 
$$
\phi (f)=
\alpha u_2^lx_2^l+(\text{terms of lower degree in }x_2)
$$
because of (\ref{eq:prop inv pf phi}). 
Since $\phi (f)=f$, 
it follows that $u_2^l=1$. 
Hence, 
$u_2$ is a root of unity. 
We show that $u_1=1$ and $u_2\neq 1$ 
by contradiction.

Suppose that $u_1=u_2=1$. 
Then, 
there exists a triangular derivation $D$ of $\Kx $ 
over $K$ such that $\phi =\exp D$ 
by Lemma~\ref{lem:e-n2}. 
Here, we identify $\phi $ 
with the natural extension to $\Kx $. 
Since $\phi \neq \id _{\Rx }$, 
we have $D\neq 0$. 
Assume that $D(x_1)=0$. 
Then, 
we have $D(x_2)\neq 0$. 
Hence, we get $\ker D=K[x_1]$. 
Since $f$ belongs to $\Kx ^{\phi }$, 
and $\Kx ^{\phi }=\ker D$ 
by Lemma~\ref{lem:e-n}, 
it follows that $f$ belongs to $K[x_1]$. 
Hence, 
$f$ is a linear polynomial in $x_1$ over $K$. 
This contradicts that $\deg _{x_2}f\geq 2$. 
If $D(x_1)\neq 0$, 
then $D(x_1)$ belongs to $K^{\times }$ 
by the triangularity of $D$. 
By integrating $D(x_1)^{-1}D(x_2)$ in $x_1$, 
we can construct $q\in K[x_1]$ such that 
$dq/dx_1=D(x_1)^{-1}D(x_2)$. 
Then, 
we have 
$D(x_2-q)=D(x_2)-D(x_1)dq/dx_1=0$. 
Since $x_2-q$ is a coordinate of $\Kx $ over $K$, 
it follows that $\ker D=K[x_2-q]$ 
by Theorem~\ref{thm:Rentschler}. 
Hence, $f$ belongs to $K[x_2-q]$. 
Since $f$ is a coordinate, 
$f$ must be a linear polynomial in $x_2-q$ over $K$. 
Thus, we get $\deg _{x_2}f=1$, 
a contradiction. 
Therefore, 
we have $(u_1,u_2)\neq (1,1)$.

Next, 
suppose that $u_1\neq 1$ and $u_2=1$. 
Set 
$$
z_1=(u_1-1)x_1+v. 
$$
Then, 
we have $K[z_1]=K[x_1]$ since $u_1\neq 1$. 
Moreover, 
we get 
$$
\phi (z_1)=u_1z_1+(u_1-1)v-(u_1-1)v=u_1z_1
$$
by applying 
(\ref{eq:ch1:useful id}) with 
$(\gamma _1,p_1,\gamma _2,p_2)=(u_1,v,u_1-1,v)$. 
Since $g$ belongs to $K[x_1]=K[z_1]$, 
we may write $g=\sum _{i\geq 0}\beta _iz_1^i$, 
where $\beta _i\in K$ for each $i$. 
Let $I$ be the set of $i$ such that 
$\beta _i\neq 0$ and $u_1^i\neq 1$. 
Then, 
define 
$$
z_2=x_2+q,\quad \text{where}\quad 
q:=\sum _{i\in I}\beta _i(1-u_1^i)^{-1}z_1^i. 
$$
Since $u_2=1$, 
we have 
\begin{align*}
\phi (z_2)=\phi (x_2+q)
=(x_2+g)+\phi (q)=(z_2-q)+g+\phi (q)
=z_2+q', 
\end{align*}
where 
\begin{align*}
q'&:=g+\phi (q)-q=\sum _{i\geq 0}\beta _iz_1^i+
\sum _{i\in I}\beta _i(1-u_1^i)^{-1}(\phi (z_1^i)-z_1^i) \\
&\quad =\sum _{i\geq 0}\beta _iz_1^i+
\sum _{i\in I}\beta _i(1-u_1^i)^{-1}(u_1^i-1)z_1^i
=\sum _{i\not\in I}\beta _iz_1^i. 
\end{align*}
Note that 
$i$ does not belong to $I$ if and only if 
$\beta _i=0$ or $u_1^i=1$, 
and so if and only if 
$\phi (\beta _iz_1^i)=\beta _iu_1^iz_1^i$ 
is equal to $\beta _iz_1^i$. 
Hence, 
we get $\phi (q')=q'$. 
If $q'\neq 0$, 
then we know by Lemma~\ref{lem:k[x^t,y]} (ii) that 
$\Kx ^{\phi }$ is contained in $K[z_1]=K[x_1]$. 
Hence, 
$f$ is a linear polynomial in $x_1$ over $K$, 
a contradiction. 
If $q'=0$, 
then we have $\phi (z_2)=z_2$. 
Since $\phi (z_1)=u_1z_1$ with $u_1\neq 1$, 
it follows that 
$\Kx ^{\phi }=K[z_1,z_2]^{\phi }$ 
is equal to $K[z_1^t,z_2]$ for some $t\geq 2$ 
if $u_1$ is a root of unity, 
and to $K[z_2]$ otherwise. 
In the former case, 
$f$ is a linear polynomial in 
$z_2$ over $K[z_1^t]$ 
by Lemma~\ref{lem:k[x^t,y]} (i). 
In the latter case, 
$f$ is a linear polynomial in $z_2$ over $K$. 
In either case, 
we have $\deg _{x_2}f=1$, 
a contradiction. 
Therefore, 
we conclude that $u_1=1$ and $u_2\neq 1$.

Since $u_2$ is a root of unity as mentioned, 
we may find the maximal integer $e\geq 2$ 
such that $u_2^e=1$. 
Then, we have 
$$
\phi ^e(x_1)=x_1+ev\quad \text{and}\quad 
\phi ^e(x_2)=x_2+p
$$
for some $p\in R[x_1]$. 
Hence, $\phi ^e$ belongs to $J(R;x_1,x_2)$. 
Since $\phi ^e(f)=f$, 
we may conclude that 
$\phi ^e=\id _{\Rx }$ 
from the discussion for the case of $u_1=u_2=1$. 
Hence, 
we have $ev=0$, 
and so $v=0$. 
Thus, 
we get $\phi (x_1)=x_1$. 
Therefore, 
$\phi $ belongs to $H(x_1)$, 
proving the first part of 
Proposition~\ref{prop:inv} (i).

We check that $f$ is of type II\null. 
Set $u=u_2-1$, 
and define 
$$
y_2=ux_2+g 
\quad \text{and}\quad 
R'=R[u^{-1}]. 
$$ 
Then, 
we have $R'[x_1,y_2]=R'[\x ]$ since $u_2\neq 1$. 
Moreover, 
we get $\phi (y_2)=u_2y_2$ by applying 
(\ref{eq:ch1:useful id}) with 
$(\gamma _1,p_1,\gamma _2,p_2)=(u_2,g,u,g)$. 
Hence, 
it follows that 
$$
R'[\x ]^{\phi }=R'[x_1,y_2]^{\phi }=R'[x_1,y_2^e] 
$$
by the definition of $e$. 
Since $f$ belongs to this set, 
we know by Lemma~\ref{lem:k[x^t,y]} (i) that 
$f=a'x_1+h$ for some 
$a'\in R'\sm \zs $ and $h\in R'[y_2^e]$. 
Then, 
$h$ does not belong to $R'$ 
by the assumption that $\deg _{x_2}f\geq 2$. 
Furthermore, 
we have $\lambda :=\deg _{x_1}g\geq 2$ 
by the assumption that 
$\deg _{x_1}f>\deg _{x_2}f$. 
Let $c$ be the leading coefficient of $g$. 
Then, 
$f^{\w (f)}$ is equal to a power of 
$y_2^{\w (y_2)}=ux_2+cx_1^{\lambda }$ 
multiplied by an element of $R'\sm \zs $. 
Since $\lambda \geq 2$ 
and $f$ is tamely reduced over $R$ by assumption, 
it follows from Proposition~\ref{prop:reduced coordinate} (ii) 
that $c$ does not belong to $uR$. 
Thus, 
$f$ satisfies all the conditions of 
Definition~\ref{def:invariant coord} (2). 
Therefore, 
$f$ is of type II\null. 
This completes the proof of 
Proposition~\ref{prop:inv} (i).

We use the following lemma 
to prove Proposition~\ref{prop:inv} (ii).

\begin{lem}\label{lem:inv:affinecase}
Let $f\in \Rx $ be such that 
$f=\psi (\gamma x_2+q)$ for some 
$\psi \in \Aff (K,\x )$, 
$\gamma \in K^{\times }$ 
and $q\in K[x_1]$. 
If $f$ is tamely reduced over $R$, 
and $\deg _{x_1}f=\deg _{x_2}f$, 
then the following statements hold$:$

\noindent{\rm (i)} 
If $\deg _{x_1}q=2$, then $f$ is of type IV. 

\noindent{\rm (ii)} 
Assume that $q$ belongs to $K[(sx_1+t)^l]\sm K$ 
for some $s\in K^{\times }$, $t\in K$ 
and $l\geq 2$. 
If $\deg _{x_1}q\geq 3$, 
then $f$ satisfies the conditions of 
Definition~$\ref{def:invariant coord}$ $(5)$ 
except for {\rm (c)}. 
\end{lem}
\begin{proof}
Write $\psi (x_1)=\alpha _1'x_1+\alpha _2'x_2+\alpha _0'$, 
where $\alpha _1',\alpha _2',\alpha _0'\in K$. 
If $\deg _{x_1}q=2$, then 
we may write 
$$
q=\beta _2(x_1-\alpha _0')^2+\beta _1x_1+\beta _0, 
$$
where $\beta _0,\beta _1,\beta _2\in K$ with $\beta _2\neq 0$. 
Take $u\in K^{\times }$ such that $a:=u^2\beta _2$ 
belongs to $R$. 
Then, 
we can define $\tau \in \Aff (K,\x )$ by 
$$
\tau (x_1)=u^{-1}\psi (x_1-\alpha _0'),\quad 
\tau (x_2)=\psi (\beta _1x_1+\gamma x_2+\beta _0), 
$$
since $\gamma \neq 0$. 
Using this $\tau $, 
we may write 
\begin{equation*}
f=\psi (q+\gamma x_2)
=\psi \bigl(
\beta _2(x_1-\alpha _0')^2+\beta _1x_1+\beta _0+\gamma x_2\bigr)
=a\tau (x_1)^2+\tau (x_2). 
\end{equation*}
Put $\alpha _i=u^{-1}\alpha _i'$ for $i=1,2$. 
Then, 
we have $\tau (x_1)=\alpha _1x_1+\alpha _2x_2$. 
Since $\deg _{x_1}f=\deg _{x_2}f$ by assumption, 
we know that $\alpha _1$ and $\alpha _2$ are nonzero. 
Hence, 
we get 
$$
f^{\w (f)}=a\tau (x_1)^2
=a\alpha _2^2\bigl((\alpha _1/\alpha _2)x_1+x_2\bigr)^2. 
$$ 
Since $f$ is tamely reduced over $R$ by assumption, 
this implies that 
$\alpha _1/\alpha _2$ does not belong to $V(R)$ 
by Proposition~\ref{prop:reduced coordinate} (i). 
Therefore, 
$f$ is of type IV\null.

Next, 
assume that $q$ is as in (ii). 
Put $y_1=sx_1+t$, 
and write $q=q'+\delta $, 
where $q'\in y_1^lK[y_1^l]$ and $\delta \in K$. 
Let $g'$ be an element of $x_1^lK[x_1^l]$ obtained from 
$q'$ by replacing $y_1$ with $x_1$. 
Then, 
we have 
$\lambda :=\deg _{x_1}g'=\deg _{y_1}q'=\deg _{y_1}q\geq 3$. 
Define $\tau \in \Aff (K,\x )$ by 
$\tau (x_1)=\psi (y_1)$ and $\tau (x_2)=\psi (\gamma x_2+\delta )$. 
Then, we have 
\begin{equation}\label{eq:pf:lem:inv:affinecase2}
f=\psi (\gamma x_2+q'+\delta )
=\psi (\gamma x_2+\delta )+\tau (g')
=\tau (x_2+g'). 
\end{equation}
Hence, 
$f$ is written as in 
Definition~\ref{def:invariant coord} (5), 
and satisfies (a). 
Put $\alpha _i:=s\alpha _i'$ for $i=1,2$. 
Then, we have 
$$
\tau (x_1)
=\psi (sx_1+t)=\alpha _1x_1+\alpha _2x_2+s\alpha _0'+t. 
$$ 
Since $\deg _{x_1}f=\deg _{x_2}f$ by assumption, 
we see from (\ref{eq:pf:lem:inv:affinecase2}) that 
$\alpha _1$ and $\alpha _2$ are nonzero, 
and 
$$
f^{\w (f)}
=(g')^{\w (f)}
=c(\alpha _1x_1+\alpha _2x_2)^{\lambda }
=c\alpha _2^{\lambda }((\alpha _1/\alpha _2)x_1+x_2)^{\lambda }, 
$$ 
where $c\in K^{\times }$ is the leading coefficient of $g'$. 
Since $f$ is tamely reduced over $R$ by assumption, 
this implies that 
$\alpha _1/\alpha _2$ does not belong to $V(R)$ 
by Proposition~\ref{prop:reduced coordinate} (ii). 
Thus, 
(b) of Definition~\ref{def:invariant coord} (5) 
is satisfied. 
Therefore, 
$f$ satisfies the conditions of 
Definition~\ref{def:invariant coord} (5) 
except for (c). 
\end{proof}

Now, 
we prove Proposition~\ref{prop:inv} (ii). 
Let 
$f\in \Rx $ be a coordinate of $\Kx $ over $K$ 
which is tamely reduced over $R$ 
and satisfies $\deg _{x_1}f=\deg _{x_2}f\geq 2$. 
Assume that $\phi (f)=f$ for some 
$\id _{\Rx }\neq \phi \in \Aff (R,\x )$. 
Then, 
it suffices to show that 
$f$ is written as in Lemma~\ref{lem:inv:affinecase}. 
Write $(\phi (x_1),\phi (x_2))=(x_1,x_2)A+(b_1,b_2)$, 
where $A\in \GL (2,R)$ and $b_1,b_2\in R$. 
Let $K'$ be an extension field of $K$ 
to which the eigenvalues of $A$ belong. 
Then, 
we know from linear algebra 
that $P^{-1}AP$ is upper triangular 
for some $P\in \GL (2,K')$. 
We define $\psi \in \Aff (K',\x )$ 
by $(\psi (x_1),\psi (x_2))=(x_1,x_2)P$, 
and put $\phi '=\psi ^{-1}\circ \phi \circ \psi $. 
Then, we have 
\begin{align*}
&(\phi '(x_1),\phi '(x_2))
=\bigl(
(\psi ^{-1}\circ \phi )(x_1),(\psi ^{-1}\circ \phi )(x_2)\bigr)P \\
&\quad 
=\bigl( (\psi ^{-1}(x_1),\psi ^{-1}(x_2))A +(b_1,b_2)\bigr) P
=(x_1,x_2)P^{-1}AP+(b_1,b_2)P. 
\end{align*}
Since $P^{-1}AP$ is regular and upper triangular, 
we may write
$$
\phi '(x_1)=u_1x_1+\beta _1
\quad \text{and}\quad 
\phi '(x_2)=vx_1+u_2x_2+\beta _2,
$$
where 
$u_1,u_2\in (K')^{\times }$ 
are the eigenvalues of $A$, 
and $v,\beta _1,\beta _2\in K'$. 
Hence, 
$\phi '$ belongs to $J(K';x_1,x_2)$. 
Since $\phi (f)=f$ by assumption, 
$\phi '$ fixes the coordinate 
$\psi ^{-1}(f)$ of $K'[\x ]$ over $K'$. 
Thus, 
we have $u_1=1$ or $u_2=1$ 
by Lemma~\ref{lem:eigenvalue}. 
Since 
$u_1u_2=\det P^{-1}AP=\det A$ belongs to $R^{\times }$, 
it follows that $u_1$ and $u_2$ belong to $R^{\times }$. 
Hence, we may assume that $K'=K$. 
Since $u_1=1$ or $u_2=1$, 
we have $u_1\neq u_2$ if and only if $(u_1,u_2)\neq (1,1)$. 
If this is the case, 
then we may choose $P$ so that 
$P^{-1}AP$ is a diagonal matrix. 
Then, we have $v=0$. 
In this case, 
we may assume further that $u_1\neq 1$ and $u_2=1$ 
by replacing $P$ if necessary. 
Thus, 
we are reduced to the following two cases: 
$$
\left\{ 
\begin{array}{cl}
\phi '(x_1)&=x_1+\beta _1 \\ 
\phi '(x_2)&=x_2+\alpha x_1+\beta _2,\quad 
\end{array}
\right. 
\left\{ 
\begin{array}{cl}
\phi '(x_1)&=ux_1+\beta _1 \\ 
\phi '(x_2)&=x_2+\beta _2. 
\end{array}
\right. 
$$
Here, 
$\alpha $, $\beta _1$ and $\beta _2$ 
are elements of $K$, 
and $u\neq 1$ is an element of $R^{\times }$.

First, 
we consider the former case. 
Define $D\in \Der _K\Kx $ by 
$$
D(x_1)=\beta _1\quad \text{and}\quad 
D(x_2)=\alpha x_1+\beta _2-\frac{\alpha \beta _1}{2}. 
$$
Then, $D$ is triangular, 
and satisfies 
$\exp D=\phi '$, 
since 
\begin{align*}
(\exp D)(x_1)&=x_1+D(x_1)=x_1+\beta _1 \\
(\exp D)(x_2)&=x_2+D(x_2)+\frac{D^2(x_2)}{2}
=x_2+\left( 
\alpha x_1+\beta _2-\frac{\alpha \beta _1}{2}\right) 
+\frac{\alpha \beta _1}{2}. 
\end{align*}
Hence, 
we have $\Kx ^{\phi '}=\ker D$ by Lemma~\ref{lem:e-n}, 
and so $\Kx ^{\phi }=\psi (\ker D)$. 
Since $f$ is a coordinate of $\Kx $ over $K$ 
belonging to $\Kx ^{\phi }$, 
we know that, 
if $\ker D=K[h]$ for some $h\in \Kx $, 
then $f$ is a linear polynomial 
in $\psi (h)$ over $K$.

We prove that $\beta _1\neq 0$ by contradiction. 
If $\beta _1=0$, 
then we have $\ker D=K[x_1]$. 
Hence, 
$f$ is a linear polynomial 
in $\psi (x_1)$ over $K$. 
Since $\psi $ is affine, 
it follows that $\deg f=1$, 
a contradiction. 
Thus, 
we get $\beta _1\neq 0$. 
Define 
$$
q=\frac{\alpha }{2\beta _1}x_1^2+
\left(  \frac{\beta _2}{\beta _1}-\frac{\alpha }{2}
\right) x_1. 
$$
Then,  
we have $D(x_2-q)=0$, 
since 
$$
D(q)=D(x_1)\frac{\partial q}{\partial x_1}
=\beta _1\left( 
2\frac{\alpha }{2\beta _1}x_1
+\frac{\beta _2}{\beta _1}-\frac{\alpha }{2}
\right) 
=D(x_2). 
$$
Since $x_2-q$ is a coordinate of $\Kx $ over $K$, 
it follows that $\ker D=K[x_2-q]$ 
by Theorem~\ref{thm:Rentschler}. 
Hence, we have 
$$
f=\gamma \psi (x_2-q)+\gamma '
=\psi \bigl(\gamma x_2+(\gamma '-\gamma q)\bigr)
$$ 
for some $\gamma \in K^{\times }$ and $\gamma '\in K$ 
as remarked. 
Since $\deg f\geq 2$ by assumption, 
this implies that $\deg _{x_1}q\geq 2$. 
Thus, 
we get $\alpha \neq 0$, 
and so 
$\deg _{x_1}(\gamma '-\gamma q)
=\deg _{x_1}q=2$. 
Therefore, 
$f$ is expressed as in 
Lemma~\ref{lem:inv:affinecase} (i).

Next, 
we consider the latter case. 
Set $y_1=(u-1)x_1+\beta _1$. 
Then, 
we have $\Kx =K[y_1,x_2]$ since $u\neq 1$. 
Moreover, 
we get $\phi '(y_1)=uy_1$ 
by applying (\ref{eq:ch1:useful id}) with 
$(\gamma _1,p_1,\gamma _2,p_2)=(u,\beta _1,u-1,\beta _1)$. 
We prove that $\beta _2=0$ 
and $u$ is a root of unity. 
Suppose to the contrary that $\beta _2\neq 0$. 
Then, 
$\Kx ^{\phi '}=K[y_1,x_2]^{\phi '}$ 
is contained in $K[y_1]$ 
by Lemma~\ref{lem:k[x^t,y]} (ii). 
Hence, 
$f$ belongs to $\psi (K[y_1])$. 
This implies that $f$ 
is a linear polynomial in $\psi (y_1)$ over $K$. 
Since $\psi $ is affine, 
it follows that 
$\deg f=\deg \psi (y_1)=\deg y_1=1$, 
a contradiction. 
Thus, we get $\beta _2=0$, 
and so $\phi '(x_2)=x_2$. 
Suppose that $u$ is not a root of unity. 
Then, we have 
$\Kx ^{\phi '}=K[y_1,x_2]^{\phi '}=K[x_2]$, 
since $\phi '(y_1)=uy_1$. 
Hence, 
$f$ belongs to $\psi (K[x_2])$. 
This implies that $f$ 
is a linear polynomial in $\psi (x_2)$ over $K$. 
Thus, we get $\deg f=1$, 
a contradiction. 
Therefore, 
$u$ is a root of unity. 
Consequently, 
we have 
$\Kx ^{\phi '}=K[y_1^l,x_2]$ for some $l\geq 2$. 
Note that $\psi ^{-1}(f)$ belongs to $\Kx ^{\phi '}$, 
and is a coordinate of $\Kx $ over $K$. 
Hence, 
by virtue of Lemma~\ref{lem:k[x^t,y]} (i), 
we may write $\psi ^{-1}(f)=\gamma x_2+q$, 
where $\gamma \in K^{\times }$ and $q\in K[y_1^l]$. 
Then, 
we have $f=\psi (\gamma x_2+q)$. 
Since $\deg f\geq 2$ by assumption, 
we see that $q$ does not belong to $K$. 
Hence, 
we know that $\deg _{x_1}q=\deg _{y_1}q\geq l\geq 2$. 
If $\deg _{x_1}q=2$, 
then $f$ is expressed as in Lemma~\ref{lem:inv:affinecase} (i). 
If $\deg _{x_1}q\geq 3$, 
then $f$ is expressed as in Lemma~\ref{lem:inv:affinecase} (ii). 
This completes the proof of Proposition~\ref{prop:inv} (ii), 
and thereby completing the proof of 
Theorems~\ref{thm:Gamma (f)} and \ref{thm:H(f)}.

\section{Application}\label{sect:tame intersect appl}
\setcounter{equation}{0}

Assume that $R$ is a $\Q $-domain, 
and let $D$ be a triangular derivation of $\Rx $ over $R$ 
with $D(x_i)\neq 0$ for $i=1,2$. 
Take any $f\in \ker D\sm R$, 
and put $\phi =\exp fD$. 
In this section, 
we investigate 
when the coordinates $\phi (x_1)$ and $\phi (x_2)$ 
of $\Rx $ over $R$ 
are totally wild or quasi-totally wild. 
Note that 
$\phi (x_i)$ is quasi-totally wild if and only if 
$\phi (x_i)$ is exponentially wild for $i=1,2$ 
by Corollary~\ref{cor:expw=>qtw} (i). 
If $\phi (x_i)$ is exponentially wild, 
then $\phi (x_i)$ is wild 
as mentioned after Definition~\ref{defin:wild coordinates}. 
If $\phi (x_i)$ is wild for some $i\in \{ 1,2\} $, 
then $\phi $ does not belong to $\T (R,\x )$ by definition. 
Conversely, 
if $\phi $ does not belong to $\T (R,\x )$, 
then $\phi (x_i)$ is wild for $i=1,2$. 
In fact, 
if $\phi (x_i)$ is tame for some $i\in \{ 1,2\} $, 
and $\tau \in \T (R,\x )$ is such that $\tau (x_i)=\phi (x_i)$, 
then $\tau ^{-1}\circ \phi $ belongs to $\Aut (\Rx /R[x_i])$.

Write $D$ as in (\ref{eq:triangular/R}), 
and define $I$ as in 
(\ref{eq:II'}). 
Then, due to Theorem~\ref{thm:hD}, 
$\phi $ does not belong to $\T (R,\x )$ 
if and only if one of the following conditions holds: 

\noindent{\rm (W1) } 
$I\cap \{ 1,\ldots ,l\} \neq \emptyset $. 

\noindent{\rm (W2) } 
$I=\zs $, $b_0/a$ does not belong to $V(R)$ 
and $\deg _{x_2}f\neq 1$. 

\noindent 
Here, we note that $\deg _{x_2}f\neq 1$ 
if and only if $\deg _{x_2}f\geq 2$, 
since $f$ is not an element of $R$ by assumption.

Define $\tau \in \Aut (\Rx /R[x_1])$ 
as in (\ref{eq:tau:triangular/R}), 
and set $g_i=\tau (\phi (x_i))$ for $i=1,2$. 
Then, 
$\phi (x_i)$ is totally (resp.\ quasi-totally) wild 
if and only if 
$g_i$ is totally (resp.\ quasi-totally) wild for $i=1,2$.

With this notation, 
we have the following theorems.

\begin{thm}\label{prop:at2}
Assume that $\phi =\exp fD$ 
does not belong to $\T (R,\x )$. 
Then, the following assertions hold$:$ 

\noindent{\rm (i)} 
$\deg _{x_1}g_1\geq 2$, 
$\deg _{x_1}g_1\geq \deg _{x_2}g_2\geq 1$, 
and $\deg _{x_1}g_2\geq \deg _{x_2}g_2\geq 2$.

\noindent{\rm (ii)} 
$g_1$ and $g_2$ are tamely reduced over $R$.

\noindent{\rm (iii)} 
Assume that $H(g_i)$ is not equal to $\{ \id _{\Rx }\} $ 
for some $i\in \{ 1,2\} $. 
Then, 
$g_i$ is of type I or IV or V if $i=1$, 
and of type IV or V if $i=2$. 
\end{thm}

Next, consider the following condition: 
\begin{equation}\label{eq:inv:appl:cond}
I=\zs \text{ and }
b_0/a\text{ does not belong to }V(R),
\text{ and }l=0\text{ if }i=2. 
\end{equation}

\begin{thm}\label{prop:at2'}
The following assertions hold$:$ 

\noindent{\rm (i)} 
$g_1$ is of type I if and only if 
$\deg _{x_2}f=1$ and 
$I\cap \{ 1,\ldots ,l\} \neq \emptyset $.

\noindent{\rm (ii)} 
$g_i$ is of type IV for $i\in \{ 1,2\} $ 
if and only if $\deg _{x_2}f=2$ 
and $(\ref{eq:inv:appl:cond})$ holds.

\noindent{\rm (iii)} 
If $g_i$ is of type V for $i\in \{ 1,2\} $, 
then we have $\deg _{x_2}f\geq 3$ 
and $(\ref{eq:inv:appl:cond})$. 
\end{thm}

First, we describe $\phi (x_1)$ and $\phi (x_2)$ concretely. 
Since $D(x_1)=a$, 
we have 
$$
\phi (x_1)=(\exp fD)(x_1)=x_1+fD(x_1)=x_1+af. 
$$
Define $h$ as in (\ref{eq:triangular kernel h/R}). 
Then, 
we have $\phi (h)=h$, 
since $D(h)=0$. 
Hence, 
we get 
$$
ax_2-\sum _{i=0}^l\frac{b_i}{i+1}x_1^{i+1}
=h=\phi (h)
=a\phi (x_2)-\sum _{i=0}^l\frac{b_i}{i+1}(x_1+af)^{i+1}. 
$$
This gives that 
$$
\phi (x_2)=x_2+
\sum _{i=0}^l\frac{b_i}{(i+1)a}
\bigl((x_1+af)^{i+1}-x_1^{i+1}\bigr) . 
$$
Set $f_0=\tau (f)$. 
Then, we have 
\begin{equation}\label{eq:inv:appl:g1g2}
\begin{aligned}
g_1&=\tau (\phi (x_1))=x_1+af_0\\
g_2&=\tau (\phi (x_2))=x_2+
\sum _{i=0}^l\frac{b_i}{(i+1)a}(x_1+af_0)^{i+1}
-\sum _{i\in I}\frac{b_i}{(i+1)a}x_1^{i+1}, 
\end{aligned}
\end{equation}
since $I$ is the complement of $I'$ in $\{ 0,1,\ldots ,l\} $.

Now, we prove Theorems~\ref{prop:at2} and \ref{prop:at2'}. 
Since $\phi $ does not belong to $\T (R,\x )$ by assumption, 
we have $I\neq \emptyset $ by Theorem~\ref{thm:hD}. 
Set $t=\max I\geq 0$ and $h_0=\tau (h)$. 
Then, 
we see from (\ref{eq:triangular kernel h0/R}) that 
$h_0^{\w (h_0)}$ is written as in (\ref{eq:w(h_0)}). 
Since $\ker D$ is contained in $K[h]$ 
as mentioned before Theorem~\ref{thm:Rentschler}, 
$f_0$ belongs to $K[h_0]\sm R$. 
Hence, 
we have 
$m:=\deg _{x_2}f=\deg _hf=\deg _{h_0}f_0\geq 1$. 
We show that 
\begin{equation}\label{eq:pf:ta2:1}
g_1^{\w (g_1)}=c_1(x_2-bx_1^{t+1})^m,\quad 
g_2^{\w (g_2)}=c_2(x_2-bx_1^{t+1})^{(l+1)m}
\end{equation}
for some $c_1,c_2\in K^{\times }$, 
where $b=((t+1)a)^{-1}b_t$. 
Since 
$f_0^{\w (f_0)}$ is equal to 
$(h_0^{\w (h_0)})^m=a^m(x_2-bx_1^{t+1})^m$ 
up to a nonzero constant multiple, 
it suffices to verify that 
$\deg _{x_1}f_0\geq 2$, 
$\deg _{x_2}f_0\geq 1$ and 
$\deg _{x_2}f_0^{l+1}\geq 2$ 
in view of 
(\ref{eq:inv:appl:g1g2}). 
Note that $\deg _{x_1}h_0=t+1$, $\deg _{x_2}h_0=1$ 
and 
$$
\deg _{x_i}f_0=(\deg _{h_0}f_0)\deg _{x_i}h_0
=m\deg _{x_i}h_0
$$
for $i=1,2$. 
First, 
assume that $t\geq 1$. 
Then, it follows that 
$\deg _{x_1}f_0\geq 2m\geq 2$ 
and $\deg _{x_2}f_0=m\geq 1$. 
Since $I$ is a subset of $\{ 0,\ldots ,l\} $, 
we have $l\geq t$. 
Hence, 
we get $\deg _{x_2}f_0^{l+1}=(l+1)m\geq 2$. 
Thus, 
the assertion holds. 
Next, assume that $t=0$. 
Then, we have $I=\zs $. 
Since $\phi $ does not belong to $\T (R,\x )$ by assumption, 
this implies that $\phi $ satisfies (W2). 
Hence, we get $m=\deg _{x_2}f\geq 2$. 
Since $\deg _{x_i}h_0=1$ for $i=1,2$, 
it follows that 
$\deg _{x_i}f_0=m\geq 2$ for $i=1,2$, 
and $\deg _{x_2}f_0^{l+1}=(l+1)m\geq m\geq 2$. 
Thus, the assertion holds. 
Therefore, 
we obtain (\ref{eq:pf:ta2:1}).

Thanks to Proposition~\ref{prop:slope}, 
we see from (\ref{eq:pf:ta2:1}) that 
\begin{equation}\label{eq:pf:ta2:2}
\begin{aligned}
\deg _{x_1}g_1&=(t+1)m &
\quad \deg _{x_1}g_2&=(t+1)(l+1)m \\
\deg _{x_2}g_1&=m& 
\deg _{x_2}g_2&=(l+1)m. 
\end{aligned}
\end{equation}
By the preceding discussion, 
we have $l\geq t$, 
and $m\geq 2$ if $t=0$. 
Hence, (i) follows from (\ref{eq:pf:ta2:2}). 
We prove (ii) using Proposition~\ref{prop:reduced coordinate}. 
First, assume that $t=0$. 
Then, we have $b=b_0/a$. 
Since $I=\zs $, 
we know that $b$ does not belong to $V(R)$ by (W2). 
By (\ref{eq:pf:ta2:1}), 
this implies that 
$g_1$ and $g_2$ are tamely reduced over $R$ 
because of Proposition~\ref{prop:reduced coordinate} (i). 
Next, 
assume that $t\geq 1$. 
Since $t$ is an element of $I$, 
we see that $b_t$ does not belong to $aR$. 
Hence, 
$b=((t+1)a)^{-1}b_t$ does not belong to $R$. 
By (\ref{eq:pf:ta2:1}), 
this implies that 
$g_1$ and $g_2$ are tamely reduced over $R$ 
because of Proposition~\ref{prop:reduced coordinate} (ii). 
This proves (ii).

To prove (iii), take any $i\in \{ 1,2\} $, 
and assume that $H(g_i)$ 
is not equal to $\{ \id _{\Rx }\} $. 
Then, 
$g_i$ satisfies (A) of Theorem~\ref{thm:Gamma (f)}. 
By (ii) and (i) of Theorem~\ref{prop:at2}, 
$g_i$ also satisfies (B) and (C). 
Thus, 
$g_i$ must be of one of the types I through V 
by Theorem~\ref{thm:Gamma (f)}. 
Since $R$ is a $\Q $-domain, 
$g_i$ is not of type II as remarked after 
Definition~\ref{def:invariant coord}. 
Since $g_i$ is a coordinate of $\Rx $ over $R$, 
we know that 
$g_i$ is not of type III 
by Proposition~\ref{prop:f1f3}. 
Since $\deg _{x_2}g_2\geq 2$ by (i), 
we see that $g_2$ is not of type I\null. 
Therefore, 
$g_i$ is of type I or IV or V if $i=1$, 
and is of type IV or V if $i=2$. 
This proves (iii), 
and thus completing the proof of 
Theorem~\ref{prop:at2}.

Next, we prove Theorem~\ref{prop:at2'}. 
Note that $\deg _{x_2}f=1$ and 
$I\cap \{ 1,\ldots ,l\} \neq \emptyset $ 
if and only if $\deg _{h_0}f_0=1$ 
and $t\geq 1$. 
If these conditions are satisfied, 
then we see from 
(\ref{eq:triangular kernel h0/R}) and 
(\ref{eq:inv:appl:g1g2}) 
that $g_1$ has the form 
$a'(ax_2-g)+x_1$ for some $a'\in K^{\times }$ 
and $g\in K[x_1]$ with $\deg _{x_1}g=t+1\geq 2$ 
whose leading coefficient is equal to $b_t/(t+1)$. 
Since $b_t/(t+1)$ does not belong to $aR$, 
we know that $g_1$ is of type I\null. 
Conversely, 
if $g_1$ is of type I, 
then we see from 
(\ref{eq:triangular kernel h0/R}) and 
(\ref{eq:inv:appl:g1g2}) 
that $\deg _{h_0}f_0=1$ 
and $t\geq 1$. 
Hence, 
we have 
$\deg _{x_2}f=1$ and 
$I\cap \{ 1,\ldots ,l\} \neq \emptyset $. 
This proves (i).

We prove the ``only if" part of (ii), and (iii). 
Take any $i\in \{ 1,2\} $, 
and assume that 
$g_i$ is of type IV or V\null. 
Then, 
we have $\deg _{x_1}g_i=\deg _{x_2}g_i$ 
by (\ref{eq:w(f_i)}). 
This implies that $t=0$ by (\ref{eq:pf:ta2:2}). 
Hence, $\phi $ does not satisfy (W1). 
Since $g_i$ is a coordinate of $\Rx $ over $R$, 
and is of type IV or V, 
we know that $g_i$ is wild 
by the remark before Theorem~\ref{thm:Gamma (f)}. 
Hence, $\phi $ does not belong to $\T (R,\x )$. 
Thus, 
$\phi $ satisfies (W2). 
Consequently, 
$\phi $ satisfies the first two parts of (\ref{eq:inv:appl:cond}). 
Now, 
assume that $g_i$ is of type IV\null. 
Then, 
we have $\deg g_i=2$ by definition. 
Since $t=0$, 
we know by (\ref{eq:pf:ta2:2}) that $m=2$, 
and $l=0$ if $i=2$. 
Hence, we get $\deg _{x_2}f=2$ 
and the last part of (\ref{eq:inv:appl:cond}). 
This proves the ``only if" part of (ii). 
Next, 
assume that $g_i$ is of type V\null. 
Then, 
we have $\deg g_i\geq 3$ by definition. 
If $i=1$, 
or if $i=2$ and $l=0$, 
then it follows from (\ref{eq:pf:ta2:2}) 
that $\deg _{x_2}f=m\geq 3$, 
since $t=0$. 
To complete the proof of (iii), 
it suffices to verify that 
$g_2$ is not of type V when $l\geq 1$.

\begin{lem}\label{lem:notV}
Assume that $f_5$ is as in 
Definition~$\ref{def:invariant coord}$. 
Let $\bar{f}_5$ be a linear form 
in $x_1$ and $x_2$ over $K$ 
such that $f_5^{\w (f_5)}=\alpha \bar{f}_5^u$ 
for some $\alpha \in K^{\times }$ and $u\in \N $. 
Then, 
we have $d\bar{f}_5\wedge df_5=\beta dx_1\wedge dx_2$ 
for some $\beta \in K^{\times }$. 
\end{lem}
\begin{proof}
From (\ref{eq:f_i^w}), 
we see that 
$\bar{f}_5=c(\alpha _1x_1+\alpha _2x_2)$ 
for some $c\in K^{\times }$. 
Hence, 
$\tau _5(x_1)$ belongs to $K[\bar{f}_5]$. 
Since $f=\tau _5(x_2+g')$ for some $g'\in K[x_1]$, 
it follows that 
$$
d\bar{f}_5\wedge df_5=
d\bar{f}_5\wedge d\tau (x_2+g')
=d\bar{f}_5\wedge d\tau _5(x_2)
=c(\alpha _1\beta _2-\alpha _2\beta _1)dx_1\wedge dx_2. 
$$
Since $\tau _5$ is an element of $\Aff (K,\x )$, 
we know that 
$\det J\tau _5=\alpha _1\beta _2-\alpha _2\beta _1$ 
belongs to $K^{\times }$. 
Therefore, 
$c(\alpha _1\beta _2-\alpha _2\beta _1)$ 
belongs to $K^{\times }$. 
\end{proof}

Since $t=0$, 
we have $I=\zs $. 
Hence, 
we get $h_0=a(x_2-bx_1)$ 
by (\ref{eq:triangular kernel h0/R}). 
Thus, 
we see from (\ref{eq:pf:ta2:1}) 
that $g_2^{\w (g_2)}=ch_0^u$ 
for some $c\in K^{\times }$ and $u\in \N $. 
Now, 
suppose to the contrary that $l\geq 1$ and 
$g_2$ is of type V\null. 
Then, we have 
$dh_0\wedge dg_2=\beta dx_1\wedge dx_2$ 
for some $\beta \in K^{\times }$ by Lemma~\ref{lem:notV}. 
By (\ref{eq:inv:appl:g1g2}) with $I=\zs $, 
we get 
\begin{align*}
dh_0\wedge dg_2
&=dh_0\wedge dx_2
+\sum _{i=0}^l\frac{b_i}{a}(x_1+af_0)^idh_0\wedge dx_1
-\frac{b_0}{a}dh_0\wedge dx_1 \\
&=-\sum _{i=0}^lb_i(x_1+af_0)^idx_1\wedge dx_2, 
\end{align*}
since $dh_0\wedge df_0=0$, 
$dh_0\wedge dx_1=-adx_1\wedge dx_2$ 
and $dh_0\wedge dx_2=-b_0dx_1\wedge dx_2$. 
By the supposition that $l\geq 1$, 
this implies that 
$dh_0\wedge dg_2\neq \beta dx_1\wedge dx_2$ 
for any $\beta \in K^{\times }$, 
a contradiction. 
Therefore, 
$g_2$ is not of type V if $l\geq 1$. 
This completes the proof of 
Theorem~\ref{prop:at2'} (iii).

Finally, 
we prove the ``if" part of Theorem~\ref{prop:at2'} (ii). 
Assume that $\deg _{x_2}f=2$ 
and (\ref{eq:inv:appl:cond}) is satisfied. 
We show that $g_i$ is of type IV for $i=1,2$ 
by means of Lemma~\ref{lem:inv:affinecase} (i). 
Since (W2) is satisfied, 
$\phi $ does not belong to $\T (R,\x )$. 
By Theorems~\ref{prop:at2} (ii), 
it follows that 
$g_i$ is tamely reduced over $R$ for $i=1,2$. 
Since $I=\zs $ by (\ref{eq:inv:appl:cond}), 
we have $t=0$. 
Hence, 
we get $\deg _{x_1}g_i=\deg _{x_2}g_i$ for $i=1,2$ 
by (\ref{eq:pf:ta2:2}). 
Since $I=\zs $, 
we have $h_0=ax_2-b_0x_1$ 
by (\ref{eq:triangular kernel h0/R}). 
Define $\psi \in \Aff (K,\x )$ by 
$\psi (x_1)=h_0$ and $\psi (x_2)=x_1$, 
and take $q\in K[x_1]$ such that $\psi (q)=f_0$. 
Then, 
we have 
$g_1=x_1+af_0=\psi (x_2+aq)$ by (\ref{eq:inv:appl:g1g2}), 
and $\deg _{x_1}q=\deg _{h_0}f_0=\deg _{x_2}f=2$ by assumption. 
Therefore, 
we conclude that $g_1$ is of type IV thanks to 
Lemma~\ref{lem:inv:affinecase} (i). 
If $i=2$, 
then we have $l=0$ by the last part of (\ref{eq:inv:appl:cond}). 
Hence, 
we get $g_2=x_2+b_0f_0$ by (\ref{eq:inv:appl:g1g2}). 
Define $\psi '\in \Aff (K,\x )$ by 
$\psi '(x_1)=h_0$ and $\psi '(x_2)=x_2$. 
Then, 
we have $g_2=\psi '(x_2+b_0q)$. 
Therefore, 
we conclude that $g_2$ is of type IV similarly. 
This proves the ``if" part of (ii) Theorem~\ref{prop:at2'}, 
and thus completing the proof of Theorem~\ref{prop:at2'}.

\begin{thm}\label{thm:nagata qtw}
For $i\in \{ 1,2\} $, 
the following statements hold$:$

\noindent{\rm (i)} 
The coordinate $\phi (x_i)$ of $\Rx $ over $R$ 
is wild and is not quasi-totally wild 
if and only if one of the following conditions holds$:$

\noindent{\rm \ (1)} 
$i=1$, $\deg _{x_2}f=1$ and 
$I\cap \{ 1,\ldots ,l\} \neq \emptyset $.

\noindent{\rm \ (2)} 
$\deg _{x_2}f=2$ and $(\ref{eq:inv:appl:cond})$ holds.

\noindent{\rm (ii)} 
The coordinate $\phi (x_i)$ of $\Rx $ over $R$ 
is quasi-totally wild 
if and only if one of the following conditions holds$:$ 

\noindent{\rm \ (1)} 
$i=2$ or $\deg _{x_2}f\geq 2$, 
and $I\cap \{ 1,\ldots ,l\} \neq \emptyset $.

\noindent{\rm \ (2)} 
$i=2$, $l\geq 1$, $\deg _{x_2}f=2$, 
$I=\zs $ and $b_0/a$ does not belong to $V(R)$.

\noindent{\rm \ (3)} 
$\deg _{x_2}f\geq 3$, 
$I=\zs $ and 
$b_0/a$ does not belong to $V(R)$.

\noindent{\rm (iii)} 
If $I\cap \{ 1,\ldots ,l\} \neq \emptyset $ 
and $\deg _{x_2}f\geq 2$, 
then $\phi (x_1)$ is a totally wild coordinate of $\Rx $ over $R$.

\noindent{\rm (iv)} 
If one of the following conditions holds, 
then $\phi (x_2)$ is a totally wild coordinate of $\Rx $ over $R$$:$ 

\noindent{\rm \ (1)} 
$I\cap \{ 1,\ldots ,l\} \neq \emptyset $. 

\noindent{\rm \ (2)} 
$l\geq 1$, $\deg _{x_2}f\geq 2$, 
$I=\zs $ and $b_0/a$ does not belong to $V(R)$. 
\end{thm}
\begin{proof}
We may replace $\phi (x_i)$ with $g_i$ 
in the statements of the corollary if necessary.

(i) 
If $\phi (x_i)$ is wild, 
then $\phi $ does not belong to $\T (R,\x )$. 
When this is the case, 
we know by Theorem~\ref{prop:at2} (iii) 
and Proposition~\ref{prop:H_i infinite} that 
$g_i$ is not quasi-totally wild only if 
$i=1$ and $g_i$ is of type I, 
or $g_i$ is of type IV\null. 
Conversely, 
if $g_i$ is of type I or IV, 
then $g_i$ is wild and is not quasi-totally wild. 
Thus, 
$\phi (x_i)$ is wild and is not quasi-totally wild 
if and only if $i=1$ and $g_i$ is of type I, 
or $g_i$ is of type IV\null. 
Therefore, 
(i) follows from 
(i) and (ii) of Theorem~\ref{prop:at2'}.

(ii) 
Note that $\phi (x_i)$ is quasi-totally wild 
if and only if $\phi (x_i)$ 
is wild and (1) and (2) of (i) do not hold. 
We prove that this condition 
is equivalent to the condition 
that one of (1), (2) and (3) of (ii) holds. 
First, 
assume that 
$I\cap \{ 1,\ldots ,l\} \neq \emptyset $. 
Then, 
$\phi $ does not belong to $\T (R,\x )$. 
This implies that $\phi (x_i)$ is wild as mentioned. 
Since $I\neq \zs $, 
we see that (\ref{eq:inv:appl:cond}) does not hold. 
Hence, 
(2) of (i) is not satisfied. 
In this case, 
(1) of (i) does not hold 
if and only if $i=2$ or $\deg _{x_2}f\geq 2$. 
Hence, $g_i$ is quasi-totally wild 
if and only if (1) of (ii) holds. 
Next, assume that 
$I\cap \{ 1,\ldots ,l\} =\emptyset $. 
Then, (W1) and (1) of (i) do not hold. 
Hence, 
we know that 
$g_i$ is quasi-totally wild if and only if (W2) holds 
and (2) of (i) does not hold. 
Since the first two parts of (W2) and 
(\ref{eq:inv:appl:cond}) are the same, 
it follows that 
$g_i$ is quasi-totally wild 
if and only if (W2) holds, 
and $\deg _{x_2}f\neq 2$ or $i=2$ and $l\geq 1$. 
This condition is equivalent to 
the condition that (2) or (3) of (ii) is satisfied. 
Therefore, 
$g_i$ is quasi-totally wild 
if and only if one of (1), (2) and (3) 
of (ii) holds.

(iii) 
Since $I\cap \{ 1,\ldots ,l\} \neq \emptyset $ 
by assumption, 
$\phi $ does not belong to $\T (R,\x )$. 
Hence, 
it suffices to show that 
$g_1$ is not of type I or IV or V 
in view of  Theorem~\ref{prop:at2} (iii). 
Since $\deg _{x_2}f\geq 2$ by assumption, 
$g_1$ is not of type I by Theorem~\ref{prop:at2'} (i). 
Since $I\neq \zs $, 
we see that (\ref{eq:inv:appl:cond}) does not hold. 
Hence, $g_1$ is not of type IV or V by 
(ii) and (iii) of Theorem~\ref{prop:at2'}. 
Therefore, 
$g_1$ is totally wild.

(iv) 
Assume that (1) or (2) is satisfied. 
Then, 
$\phi $ does not belong to $\T (R,\x )$, 
since (1) is the same as (W1), 
and (2) implies (W2). 
Thus, 
it suffices to show that 
$g_2$ is not of type IV or V 
in view of Theorem~\ref{prop:at2} (iii). 
Note that (\ref{eq:inv:appl:cond}) does not hold, 
since $I\neq \zs $ in the case of (1), 
and $l\geq 1$ in the case of (2). 
Hence, 
we know by (ii) and (iii) of Theorem~\ref{prop:at2'} 
that $g_2$ is not of type IV or V\null. 
Therefore, 
$g_2$ is totally wild. 
\end{proof}

Since no elements of $\Rx $ is of type IV or V 
when $V(R)=K^{\times }$, 
Theorem~\ref{thm:nagata qtw} implies the following corollary.

\begin{cor}\label{cor:at:nagata}
Assume that $V(R)=K^{\times }$. 

\noindent{\rm (i)} The following conditions are equivalent$:$ 

\noindent{\rm \ (1)} 
$\phi $ does not belong to $\T (R,\x )$ 
and $\deg _{x_2}f\geq 2$. 

\noindent{\rm \ (2)}
$\phi (x_1)$ 
is a quasi-totally wild coordinate of $\Rx $ over $R$. 

\noindent{\rm \ (3)}
$\phi (x_1)$ 
is a totally wild coordinate of $\Rx $ over $R$.

\noindent{\rm (ii)} 
The following conditions are equivalent$:$ 

\noindent{\rm \ (1)} 
$\phi $ does not belong to $\T (R,\x )$. 

\noindent{\rm \ (2)}
$\phi (x_2)$ is a wild coordinate of $\Rx $ over $R$. 

\noindent{\rm \ (3)}
$\phi (x_2)$ 
is a quasi-totally wild coordinate of $\Rx $ over $R$. 

\noindent{\rm \ (4)}
$\phi (x_2)$ 
is a totally wild coordinate of $\Rx $ over $R$. 
\end{cor}
\begin{proof}
(i) 
Assume that (1) is satisfied. 
From the first part of (1), 
we have (W1) or (W2). 
Since $V(R)=K^{\times }$ by assumption, 
(W2) does not hold. 
Hence, 
we get (W1), 
and so $I\cap \{ 1,\ldots ,l\} \neq \emptyset $. 
Since $\deg _{x_2}f\geq 2$ by the last part of (1), 
we know that $\phi (x_1)$ is totally wild 
by Theorem~\ref{thm:nagata qtw} (iii). 
Therefore, 
(1) implies (3). 
Clearly, 
(3) implies (2). 
Assume that $\phi (x_1)$ is quasi-totally wild. 
Then, 
one of (1), (2) and (3) of 
Theorem~\ref{thm:nagata qtw} (ii) holds. 
Since $V(R)=K^{\times }$ by assumption, 
(2) and (3) do not hold. 
Hence, (1) of Theorem~\ref{thm:nagata qtw} (ii) holds. 
Since $i=1$, 
it follows that $I\cap \{ 1,\ldots ,l\} \neq \emptyset $ 
and $\deg _{x_2}f\geq 2$. 
This implies (1). 
Therefore, 
(1), (2) and (3) are equivalent.

(ii) 
Since $V(R)=K^{\times }$, 
we see that (1) implies (W1). 
By (1) of Theorem~\ref{thm:nagata qtw} (iv), 
(W1) implies (4). 
Hence, (1) implies (4). 
Clearly, (4) implies (3), 
and (2) implies (1). 
We have shown that (3) implies (2) 
after Definition~\ref{defin:wild coordinates}. 
Therefore, 
(1) through(4) are equivalent. 
\end{proof}

Finally, 
assume that $n=3$, 
and consider 
the element $\Phi _{g,h}^f=\exp fT_{g,h}$ of $\Aut (\kx /k)$ 
defined before 
Proposition~\ref{prop:triangular family} 
for $(g,h)\in \Lambda $ and 
$f\in k[x_1,gx_3+h]$ not belonging to $k[x_1]$. 
As mentioned, 
$\Phi _{g,h}^f$ does not belong to $\T (k,\x )$. 
We define 
$$
H_i=\Aut (\kx /k[x_1,\Phi _{g,h}^f(x_i)])\cap \T (k,\x )
$$
for $i=2,3$. 
Then, 
we have the following corollary.

\begin{cor}\label{cor:ta2Tghf}
For $i\in \{ 2,3\} $, 
we have $H_i\neq \{ \id _{\kx }\} $ 
if and only if $i=2$ and $\deg _{x_3}f=1$. 
\end{cor}
\begin{proof}
Let $R=k[x_1]$, $y_i=x_{i+1}$ for $i=1,2$ 
and $\y =\{ y_1,y_2\} $, 
and put $\phi =\Phi _{g,h}^f$. 
Then, 
we have 
$H_{i}=\Aut (R[\y ]/R[\phi (y_{i-1})])\cap \T (R,\y )$ 
for $i=2,3$ on account of Theorem~\ref{thm:tame23}. 
Hence, 
we have $H_{i}\neq \{ \id _{\kx }\} $ 
if and only if $\phi (y_{i-1})$ 
is not a totally wild element 
of $R[\y ]$ over $R$. 
Note that $T_{g,h}$ 
is a triangular derivation of $\Ry $ over $R$ 
with $T_{g,h}(y_j)=T_{g,h}(x_{j+1})\neq 0$ for $j=1,2$, 
and $\phi $ is an element of $\Aut (R[\y ]/R)$ 
not belonging to $\T (R,\y )$. 
Hence, 
$\phi (y_1)$ is totally wild 
if and only if $\deg _{y_2}f\geq 2$ 
by Corollary~\ref{cor:at:nagata} (i), 
and $\phi (y_2)$ is always totally wild 
by Corollary~\ref{cor:at:nagata} (ii). 
Since $f$ is not an element of $k[x_1]$, 
we have $\deg _{y_2}f=\deg _{x_3}f\geq 1$. 
Thus, 
$\phi (y_{i-1})$ is not totally wild 
if and only if $i=2$ and $\deg _{y_2}f=1$. 
Therefore, 
we have 
$H_{i}\neq \{ \id _{\kx }\} $ 
if and only if $i=2$ and 
$\deg _{x_3}f=1$. 
\end{proof}

In Chapter~\ref{chapter:atcoord}, 
we will give coordinates of $\kx $ over $k$ 
some of which are totally wild, 
and others are quasi-totally wild, 
but not totally wild.

\part{Applications of 
the generalized Shestakov-Umirbaev theory}
\label{chap:GSU}

\chapter{Generalized Shestakov-Umirbaev theory}
\label{sect:GSU}
\setcounter{equation}{0}

\section{Shestakov-Umirbaev reductions}
\setcounter{equation}{0}
\label{sect:criterion}

Part~\ref{chap:GSU} is devoted to applications of 
the generalized Shestakov-Umirbaev theory. 
Throughout, 
a {\it tame} (resp.\ {\it wild}) automorphism of $\kx $ 
will always mean an element of 
$\T (k,\x )$ (resp.\ $\Aut (\kx /k)\sm \T (k,\x )$). 
In this chapter, 
we briefly review the generalized Shestakov-Umirbaev theory, 
and derive some consequences needed later.

Let $\Gamma $ be a totally ordered additive group, 
and let $F=(f_1,f_2,f_3)$ and $G=(g_1,g_2,g_3)$ be 
triples of elements of $\kx $ 
such that $f_1$, $f_2$, $f_3$ and $g_1$, $g_2$, $g_3$ 
are algebraically independent over $k$, 
respectively. 
Here, $n\in \N $ may be arbitrary for the moment. 
We denote by $\Gammap $ the set of positive elements of $\Gamma $. 
For $\w \in (\Gammap )^n$, 
we say that the pair $(F,G)$ satisfies the 
{\it Shestakov-Umirbaev condition} for the weight $\w $ 
if the following conditions hold (cf.~\cite{SU2}): 

\medskip

\noindent
(SU1) $g_1=f_1+af_3^2+cf_3$ and $g_2=f_2+bf_3$ 
for some $a,b,c\in k$, 
and $g_3-f_3$ belongs to $k[g_1,g_2]$; 

\noindent
(SU2) $\deg _{\w }f_1\leq \deg _{\w }g_1$ 
and $\deg _{\w }f_2=\deg _{\w }g_2$; 

\noindent
(SU3) $(g_1^{\w })^2\approx (g_2^{\w })^s$ 
for some odd number $s\geq 3$; 

\noindent
(SU4) $\deg _{\w }f_3\leq \deg _{\w }g_1$, 
and $f_3^{\w }$ does not belong to 
$k[g_1^{\w }, g_2^{\w }]$; 

\noindent
(SU5) 
$\deg _{\w }g_3<\deg _{\w }f_3$; 

\noindent
(SU6) 
$\deg _{\w }g_3<\deg _{\w }g_1-\deg _{\w }g_2 
+\deg _{\w }dg_1\wedge dg_2$. 

\medskip

Here, 
$h_1\approx h_2$ 
(resp.\ $h_1\not\approx h_2$) denotes that 
$h_1$ and $h_2$ are linearly dependent 
(resp.\ linearly independent) over $k$ 
for each $h_1,h_2\in \kx \sm \zs $. 
We say that $(F,G)$ 
satisfies the {\it weak Shestakov-Umirbaev condition} 
for the weight $\w $ if (SU4), (SU5), (SU6) and 
the following conditions are satisfied (cf.~\cite{SU2}):

\medskip 

\noindent
(SU$1'$) $g_1-f_1$, 
$g_2-f_2$ and $g_3-f_3$ belong to $k[f_2,f_3]$, 
$k[f_3]$ and $k[g_1,g_2]$, respectively; 

\noindent
(SU$2'$) $\deg f_i\leq \deg g_i$ for $i=1,2$; 

\noindent
(SU$3'$) $\deg g_2<\deg g_1$, and 
$g_1^{\w }$ does not belong to $k[g_2^{\w }]$. 

\medskip 

\noindent
It is easy to check that 
(SU1), (SU2) and (SU3) imply 
(SU$1'$), (SU$2'$) and (SU$3'$), 
respectively. 
Hence, 
the Shestakov-Umirbaev condition 
implies the weak Shestakov-Umirbaev condition. 
As listed in \cite[Theorem 4.2]{SU2}, 
if $(F,G)$ satisfies the weak Shestakov-Umirbaev condition 
for the weight $\w $, 
then $(F,G)$ has certain special properties such as 

\medskip 

\noindent
{\rm (P1)} $(g_1^{\w })^2\approx (g_2^{\w })^s$ 
for some odd number $s\geq 3$, 
and so $\delta :=(1/2)\degw g_2$ belongs to $\Gamma $. 

\smallskip 

\noindent
{\rm (P2)} $\degw f_3\geq (s-2)\delta +\degw dg_1\wedge dg_2$. 

\smallskip 

\noindent
{\rm (P5)} If $\degw f_1<\degw g_1$, 
then $s=3$, $g_1^{\w }\approx (f_3^{\w })^2$, 
$\degw f_3=(3/2)\delta $ and 
$$
\degw f_1\geq \frac{5}{2}\delta +\degw dg_1\wedge dg_2. 
$$

\smallskip 

\noindent
{\rm (P7)} $\degw f_2<\degw f_1$, $\degw f_3\leq \degw f_1$, 
and $\delta <\degw f_i\leq s\delta $ for $i=1,2,3$. 

\medskip

In what follows, 
we simply say that elements of $\Gamma $ 
are linearly dependent (resp.\ linearly independent) 
if they are linearly dependent 
(resp.\ linearly independent) over $\Z $.

\begin{lem}\label{lem:SUdependent}
Assume that 
$(F,G)$ satisfies the Shestakov-Umirbaev condition 
for the weight $\w $. 
Then, $\degw f_i$ and $\degw f_2$ 
are linearly dependent for 
$i=1$ or $i=3$. 
\end{lem}
\begin{proof}
If $\degw f_1=\degw g_1$, 
then we have 
$$
2\degw f_1=2\degw g_1=s\degw g_2=s\degw f_2
$$ 
for some odd number $s\geq 3$ 
by (SU3) and (SU2). 
Hence, 
$\degw f_1$ and $\degw f_2$ are linearly dependent. 
If $\degw f_1\neq \degw g_1$, 
then we have $\degw f_1<\degw g_1$ by (SU2), 
and hence 
$\degw f_3=(3/2)\delta =(3/4)\degw g_2$ by (P5). 
Since $\degw g_2=\degw f_2$ by (SU2), 
it follows that $\degw f_3$ and $\degw f_2$ 
are linearly dependent. 
\end{proof}

We define the {\it rank} $\rank \w $ of $\w =(w_1,\ldots ,w_n)$ 
as the rank of the $\Z $-submodule of $\Gamma $ 
generated by $w_1,\ldots ,w_n$. 
If $\rank \w =n$, then 
$x_1^{a_1}\cdots x_n^{a_n}$'s 
have the different $\w $-degrees 
for different $(a_1,\ldots ,a_n)$'s. 
Hence, 
$f^{\w }$ and $g^{\w }$ 
are monomials for each $f,g\in \kx \sm \zs $. 
When this is the case, 
$f^{\w }$ and $g^{\w }$ 
are algebraically independent over $k$ 
if and only if $\degw f$ and $\degw g$ 
are linearly independent.

Now, assume that $n=3$. 
Then, we may identify 
$F\in \Aut (\kx /k)$ with the triple $(f_1,f_2,f_3)$, 
where $f_i:=F(x_i)$ for $i=1,2,3$. 
For a permutation $\sigma $ of $\{ 1,2,3\} $, 
we define 
$F_{\sigma }=
(f_{\sigma (1)},f_{\sigma (2)},f_{\sigma (3)})$. 
We say that $F$ admits a 
{\it Shestakov-Umirbaev reduction} 
for the weight $\w$ 
if there exist a permutation $\sigma $ of $\{ 1,2,3\} $ 
and $G\in \Aut (\kx /k)$ 
such that $(F_{\sigma },G_{\sigma })$ 
satisfies the Shestakov-Umirbaev condition 
for the weight $\w $. 
If this is the case, 
$\degw f_i$ and $\degw f_j$ must be linearly dependent 
for some $i\neq j$ by virtue of Lemma~\ref{lem:SUdependent}. 
If $\rank \w =3$, 
then this implies that 
$f_i^{\w }$ and $f_j^{\w }$ 
are algebraically dependent over $k$ for some $i\neq j$. 
Therefore, 
if $\rank \w =3$, 
and $f_1^{\w }$, $f_2^{\w }$ and $f_3^{\w }$  
are pairwise algebraically independent over $k$, 
then $F$ admits no Shestakov-Umirbaev reduction 
for the weight $\w $.

\begin{lem}\label{lem:sulem}
Assume that $\degw f_1>\degw f_2>\degw f_3$. 
If $F$ admits a Shestakov-Umirbaev reduction for the weight $\w $, 
then we have $3\degw f_2=4\degw f_3$, 
or $2\degw f_1=s\degw f_i$ for some odd number $s\geq 3$ 
and $i\in \{ 2,3\} $. 
\end{lem}
\begin{proof}
By definition, 
there exist $\sigma $ and $G$ such that 
$(F_{\sigma },G_{\sigma })$ 
satisfies the Shestakov-Umirbaev condition for the weight $\w $. 
Since $\degw f_1>\degw f_i$ for $i=2,3$ by assumption, 
we know that $\sigma (1)=1$ in view of (P7). 
If $\degw f_1=\degw g_1$, 
then we have $2\degw f_1=s\degw f_{\sigma (2)}$ 
for some odd number $s\geq 3$ by (SU3) and (SU2). 
Since $\sigma (2)$ must be 2 or 3, 
we get the last statement. 
If $\degw f_1\neq \degw g_1$, 
then we have 
$\degw f_{\sigma (3)}=(3/2)(1/2)\degw f_{\sigma (2)}$ 
by (P5) and (SU2). 
Hence, 
we get $3\degw f_{\sigma (2)}=4\degw f_{\sigma (3)}$. 
Since $\deg f_2>\degw f_3$ by assumption, 
it follows that $\sigma $ is the identity permutation. 
Therefore, 
we obtain $3\degw f_2=4\degw f_3$. 
\end{proof}

The following theorem is a generalization of 
the main result of Shestakov-Umirbaev~\cite{SU}.

\begin{thm}[{\cite[Theorem 2.1]{SU2}}]\label{thm:SUcriterion}
Assume that $n=3$. 
If $\degw \phi >|\w |$ for 
$\phi \in \T (k,\x )$ and $\w \in (\Gammap )^3$, 
then $\phi $ admits an elementary reduction 
for the weight $\w $, 
or a Shestakov-Umirbaev reduction for the weight $\w $. 
\end{thm}

We mention that $\phi \in \Aut (\kx /k)$ is tame 
if $\degw \phi =|\w |$ 
(cf.~\cite[Lemma~6.1]{SU2}). 
By Lemma~\ref{lem:minimal autom}, 
we have 
$\degw \phi =|\w |$ if and only if 
$\phi (x_1)^{\w }$, 
$\phi (x_2)^{\w }$ and 
$\phi (x_3)^{\w }$ are algebraically 
independent over $k$.

By Theorem~\ref{thm:SUcriterion}, 
it follows that $F=(f_1,f_2,f_3)\in \Aut (\kx /k)$ is wild 
if there exists $\w \in (\Gammap )^3$ with $\rank \w =3$ 
as follows: 

\smallskip 

\noindent
(1) $f_1^{\w }$, $f_2^{\w }$ and $f_3^{\w }$ 
are algebraically dependent over $k$, 
and are pairwise algebraically independent over $k$; 

\smallskip \noindent
(2) $f_i^{\w }$ does not belong to $k[\{ f_j^{\w }\mid j\neq i\} ]$ 
for $i=1,2,3$. 

\smallskip 

\noindent
In fact, 
the former part of (1) implies that $\degw F>|\w |$. 
Since $\rank \w =3$, 
the latter part of (1) implies that 
$F$ admits no Shestakov-Umirbaev reduction for the weight $\w $ 
as mentioned. 
The latter part of (1) also implies that 
$k[f_i,f_j]^{\w }=k[f_i^{\w },f_j^{\w }]$ for each $i\neq j$ 
by the discussion before Lemma~\ref{lem:minimal autom}. 
Hence, 
(2) implies that $F$ admits no elementary reduction 
for the weight $\w $.

\begin{definition}\label{defn:W-test}
We call $P\in \kx $ a {\it W-test polynomial} if 
there do not exist $\phi \in \T (k,\x )$, 
totally ordered additive group $\Gamma $ 
and $\w \in (\Gamma _{>0})^3$ with $\rank \w =3$ 
which satisfy the following conditions:

\smallskip 

\noindent{\rm (a)} 
$\degw \phi (P)<\degw \phi (x_{i_1})$ for some $i_1\in \{ 1,2,3\} $$;$ 

\smallskip 

\noindent{\rm (b)} 
$\degw \phi (x_{i_2})$ and $\degw \phi (x_{i_3})$ 
are linearly independent for some $i_2,i_3\in \{ 1,2,3 \} $. 
\end{definition}

Since $\rank \w =3$, 
the condition (b) is equivalent to the condition that 
$\phi (x_{i_2})^{\w }$ and $\phi (x_{i_3})^{\w }$ 
are algebraically independent over $k$ 
for some $i_2,i_3\in \{ 1,2,3 \} $. 
Note that $P$ is a W-test polynomial 
if and only if 
the following condition holds: 

\smallskip 

\noindent($\dag $) 
If $\phi \in \Aut (\kx /k)$ 
satisfies (a) and (b) for some 
totally ordered additive group $\Gamma $ 
and $\w \in (\Gamma _{>0})^3$ with $\rank \w =3$, 
then $\phi $ does not belong to $\T (k,\x )$.

\smallskip

For $P\in \kx $, 
we define 
$$
\mathcal{F}(P)=\{ 
P^{\vv }\mid \vv \in (\Lambda _{>0})^3,\ 
\Lambda \text{ is a totally ordered additive group}
\} . 
$$
The following result will be used in Chapter 7 
to prove the wildness of certain exponential automorphisms.

\begin{prop}\label{prop:criterion}
Assume that $P\in \kx $ does not belong to 
$k[\x \sm \{ x_i\} ]$ for $i=1,2,3$. 
If the following conditions hold for each $f\in \mathcal{F}(P)$, 
then $P$ is a W-test polynomial$:$ 

\noindent{\rm (i)} 
$f$ is not divisible by $x_i-g$ for any 
$i\in \{ 1,2,3\} $ and 
$g\in k[\x \sm \{ x_i\} ]\sm k$.

\noindent{\rm (ii)} 
$f$ is not divisible by $x_i^{s_i}-cx_j^{s_j}$ 
for any $i,j\in \{ 1,2,3\} $ with $i\neq j$, 
$s_i,s_j\in \N $ and 
$c\in k^{\times }$. 
\end{prop}
\begin{proof}
Let $P$ be as in the proposition. 
We verify that $P$ satisfies ($\dag $). 
Assume that $\phi \in \Aut (\kx /k)$ satisfies (a) and (b) 
for some totally ordered additive group $\Gamma $ 
and $\w \in (\Gamma _{>0})^3$ with $\rank \w =3$. 
Then, we show that $\phi $ is wild. 
Since $\rank \w =3$, 
it suffices to check the conditions (1) and (2) 
after Theorem~\ref{thm:SUcriterion}. 
Set $f_i=\phi (x_i)$ and $v_i=\degw f_i$ for $i=1,2,3$. 
Then, 
$\vv :=(v_1,v_2,v_3)$ belongs to $(\Gamma _{>0})^3$, 
since so does $\w $, 
and $f_1$, $f_2$ and $f_3$ are not constants. 
Because $P$ does not belong to $k[\x \sm \{ x_{i}\} ]$ 
for $i=1,2,3$ by assumption, 
there appears in $P$ a monomial involving $x_{i_1}$. 
Hence, 
we have $\degv P\geq v_{i_1}$. 
By (a), 
$v_{i_1}=\degw \phi (x_{i_1})$ 
is greater than $\degw \phi (P)$. 
Thus, 
we get $\degv P>\degw \phi (P)$. 
This implies that 
$\psi (P^{\vv })=0$, 
where $\psi $ is the endomorphism of the $k$-algebra $\kx $ 
defined by $\psi (x_i)=f_i^{\w }$ for $i=1,2,3$. 
Consequently, 
we know that 
$f_1^{\w }$, $f_2^{\w }$ and $f_3^{\w }$ 
are algebraically dependent over $k$.

We show that $f_1^{\w }$, $f_2^{\w }$ and $f_3^{\w }$ 
are pairwise algebraically independent over $k$ 
by contradiction. 
Suppose that $f_i^{\w }$ and $f_j^{\w }$ 
are algebraically dependent over $k$ for some $i\neq j$, 
say $i=2$ and $j=3$. 
Since $\rank \w =3$ by assumption, 
$f_2^{\w }$ and $f_3^{\w }$ are monomials. 
Hence, 
we have $(f_3^{\w })^s=c(f_2^{\w })^t$ 
for some $c\in k^{\times }$ and 
$s,t\in \N $ with $\gcd (s,t)=1$. 
Then, 
$x_3^s-cx_2^t$ is an irreducible element of $\kx $. 
Since $x_3^s-cx_2^t$ 
is a monic polynomial in $x_3$, 
it follows that $x_3^s-cx_2^t$ 
is an irreducible polynomial 
in $x_3$ over $k[x_1,x_2]$, 
and hence over $k(x_1,x_2)$. 
Note that $f_1^{\w }$ and $f_2^{\w }$ 
are algebraically independent over $k$ by (b), 
since $f_2^{\w }$ and $f_3^{\w }$ 
are algebraically dependent over $k$ by supposition. 
Let $z$ be an indeterminate over $\kx $. 
Then, $f_1^{\w }$, $f_2^{\w }$ and $z$ 
are algebraically independent over $k$. 
Hence, 
$p(z):=z^s-c(f_2^{\w })^t$ 
is an irreducible polynomial in $z$ 
over $k(f_1^{\w },f_2^{\w })$. 
Since $p(f_3^{\w })=(f_3^{\w })^s-c(f_2^{\w })^t=0$, 
it follows that 
$p(z)$ is the minimal polynomial 
of $f_3^{\w }$ over $k(f_1^{\w },f_2^{\w })$. 
Let $q(z)$ be the element of $k[f_1^{\w },f_2^{\w }][z]$ 
obtained from 
$P^{\vv }$ by the substitution 
$x_i\mapsto f_i^{\w }$ for $i=1,2$ 
and $x_3\mapsto z$. 
Then, we have $q(f_3^{\w })=\psi (P^{\vv })=0$. 
Hence, 
$q(z)$ is divisible by $p(z)$. 
Accordingly, 
$P^{\vv }$ is divisible by $x_3^s-cx_2^t$, 
a contradiction to (ii). 
Thus, 
$f_1^{\w }$, $f_2^{\w }$ and $f_3^{\w }$ 
are pairwise algebraically independent over $k$, 
proving (1). 
As a consequence, 
we know that 
$\ker \psi $ is a prime ideal of $\kx $ of height one, 
and hence a principal ideal of $\kx $.

Finally, 
we show (2) by contradiction. 
Suppose to the contrary that 
$f_i^{\w }$ belongs to 
$k[\{ f_j^{\w }\mid j\neq i\} ]$ 
for some $i$, say $i=1$. 
Then, there exists 
$g\in k[x_2,x_3]\sm k$ 
such that $f_1^{\w }=\psi (g)$. 
Note that 
$x_1-g$ is an irreducible element of $\kx $ 
such that $\psi (x_1-g)=f_1^{\w }-\psi (g)=0$. 
Since $\ker \psi $ is a principal prime ideal of $\kx $ 
as mentioned, 
this implies that $\ker \psi $ is generated by $x_1-g$. 
Since $P^{\vv }$ belongs to $\ker \psi $, 
it follows that $P^{\vv }$ is divisible by $x_1-g$. 
This contradicts (i). 
Therefore, 
$\phi $ satisfies (2). 
This proves that $\phi $ is wild, 
and thereby proving that 
$P$ is a W-test polynomial. 
\end{proof}

\section{Shestakov-Umirbaev inequality}
\label{sect:SUineq}
\setcounter{equation}{0}

In this section, 
we review the generalized 
Shestakov-Umirbaev inequality \cite{SU1}. 
Consider a nonzero polynomial 
$\Phi $ in one variable $z$ over $\kx $. 
Take $\w \in (\Gamma _{>0})^n$ 
and $g\in \kx \sm \zs $, 
and set $\w _g=(\w ,\degw g)$. 
Regard $\Phi $ as a polynomial in 
$n+1$ variables over $k$ with $x_{n+1}=z$. 
Then, we have 
$$
\degw ^g\Phi :=\deg _{\w _g}\Phi 
\geq \degw \Phi (g),\quad 
\degw ^g\Phi \geq (\deg _z\Phi )\degw g. 
$$
We denote by $\Phi ^{(i)}$ 
the $i$-th order derivative of $\Phi $ in $z$ 
for each $i\geq 0$. 
Then, we have 
$\degw ^g\Phi ^{(i)}=\degw \Phi ^{(i)}(g)$ 
for sufficiently large $i\geq 0$. 
We define $m_\w ^g(\Phi )$ 
to be the minimal $i\in \Zn$ 
such that 
$\degw ^g\Phi ^{(i)}=\degw \Phi ^{(i)}(g)$. 
We mention that $m_{\w }^g(\Phi )$ 
is equal to the minimal $i\in \Zn $ such that 
$(\Phi ^{\w _g})^{(i)}(g^{\w })\neq 0$ 
(see \cite[Lemma 3.1(ii)]{SU1}).

With this notation, 
we have the following theorem. 
This is a generalization of 
Shestakov-Umirbaev~\cite[Theorem 3]{SU'}.

\begin{thm}[{\cite[Theorem 2.1]{SU1}}]\label{thm:SUineq}
Let $f_1,\ldots ,f_r$ be elements of $\kx $ 
which are algebraically independent over $k$, 
and $\omega :=df_1\wedge \cdots \wedge df_r$, 
where $r\in \N $. 
If $\Phi $ belongs to $k[f_1,\ldots ,f_r][z]\sm \zs $, 
then we have 
$$
\degw \Phi (g)\geq \degw ^g\Phi 
+m_\w ^g(\Phi )(\degw \omega \wedge dg-\degw \omega -\degw g). 
$$
\end{thm}

In the situation of Theorem~\ref{thm:SUineq}, 
$\Phi ^{\w _g}$ belongs to 
$k[f_1,\ldots ,f_r]^{\w }[z]\sm \zs $. 
Let $K$ be the field 
of fractions of $k[f_1,\ldots ,f_r]^{\w }$, 
$\psi (z)$ the minimal polynomial 
of $g^{\w }$ over $K$, 
and $m:=m_{\w }^g(\Phi )$. 
Then, 
$\Phi ^{\w _g}$ is divisible by $\psi (z)^m$, 
since $m$ is equal to the minimal number such that 
$(\Phi ^{\w _g})^{(m)}(g^{\w })\neq 0$ as mentioned. 
Since $\deg _z\Phi ^{\w _g}\leq \deg _z\Phi $, 
it follows that 
$m$ is at most the quotient of $\deg _z\Phi $ 
divided by $[K(g^{\w }):K]$.

Now, 
let $S=\{ f,g\} \subset \kx $ 
be such that $f$ and $g$ are algebraically independent over $k$, 
and $\phi $ a nonzero element of $k[f,g]$. 
Then, 
we may uniquely write $\phi =\sum _{i,j}c_{i,j}f^ig^j$, 
where $c_{i,j}\in k$ for each $i,j\in \Zn $. 
We define $\degw ^S\phi $ to be the maximum among 
$\degw f^ig^j$ for $i,j\in \Zn $ with $c_{i,j}\neq 0$. 
We remark that, 
if $f^{\w }$ and $g^{\w }$ 
are algebraically independent over $k$, 
then $\degw ^S\phi $ is equal to $\degw \phi $. 
Take $\Phi \in k[f][y]$ such that $\Phi (g)=\phi $. 
Then, we have $\degw ^g\Phi =\degw ^S\phi $. 
Hence, 
it follows that $\degw \phi <\degw ^S\phi $ 
if and only if $m_{\w }^g(\Phi )\geq 1$.

\begin{lem}\label{lem:SUineq1}
If $\degw \phi <\degw ^S\phi $ for $\phi \in k[f,g]\sm \zs $, 
then the following assertions hold$:$

\noindent{\rm (i)} 
There exist $p,q\in \N $ with $\gcd (p,q)=1$ 
such that $(g^\w )^p\approx (f^\w )^q$. 

\noindent{\rm (ii)} 
$\degw \phi \geq q\degw f+\degw df\wedge dg-\degw f-\degw g$. 

\noindent{\rm (iii)} 
Assume that $\degw f<\degw g$, $\degw \phi \leq \degw g$ 
and $g^\w$ does not belong to $k[f^\w]$. 
Then, we have $p=2$, and $q\geq 3$ is an odd number. 
Moreover, 
$\delta :=(1/2)\degw f$ belongs to $\Gamma $, 
and 
\begin{equation*}
\degw \phi \geq 
(q-2)\delta +\degw df\wedge dg>\degw g-\degw f. 
\end{equation*}
If furthermore $\degw \phi \leq \degw f$, 
then we have $q=3$. 

\noindent{\rm (iv)} 
Let $\Phi \in \kx [z]$ be such that $\Phi (g)=\phi $. 
Then, 
$m_{\w }^g(\Phi )$ 
is at most the quotient of $\deg _z\Phi $ divided by $p$. 
\end{lem}
\begin{proof}
(i) and (ii), and (iii) follow 
from Lemmas 3.2 and 3.3 of \cite{SU2}, 
respectively. 
We prove (iv). 
Since $k[f]^{\w }=k[f^{\w }]$, 
the field of fractions of $k[f]^{\w }$ 
is equal to $k(f^{\w })$. 
Hence, 
$m_\w ^g(\Phi )$ is at most the quotient of 
$\deg _z\Phi $ divided by 
$[k(f^{\w })(g^{\w }):k(f^{\w })]$ as remarked above. 
Since $(g^\w )^p\approx (f^\w )^q$, 
there exists $c\in k^{\times }$ 
such that $(g^\w )^p=c(f^\w )^q$. 
Then, 
$z^p-c(f^{\w })^q$ is the minimal polynomial of $g^{\w }$ 
over $k(f^{\w })$, 
since $\gcd (p,q)=1$. 
Hence, 
we have $[k(f^{\w })(g^{\w }):k(f^{\w })]=p$. 
Therefore, 
$m_\w ^g(\Phi )$ is at most the quotient of 
$\deg _z\Phi $ divided by $p$. 
\end{proof}

Using the results above, 
we prove a technical lemma which will be used in 
Chapters 6 and 7. 
Assume that $f,g\in \kx $ 
satisfy the following conditions: 

\smallskip 

\noindent{\rm (1)} $\degw f<\degw g$; 

\noindent{\rm (2)} $g^\w$ does not belong to $k[f^\w ]$; 

\noindent{\rm (3)} 
$f$ and $g$ are algebraically independent over $k$. 

\smallskip 

\noindent
Then, we define 
$$
\eta _1=\degw f+\frac{3}{2}\degw g\quad \text{and}\quad 
\eta _2=2\degw f+\degw g. 
$$
Take any $\theta (z)\in k[z]$ with 
$d:=\deg _z\theta (z)\geq 1$. 
Then, there exists 
\begin{gather*}\eta (\theta ;f,g):=
\min \{ \degw (\theta (g)+fh)\mid h\in k[f,g]\} ,
\end{gather*}
since $\{ \degw h\mid h\in \kx \sm \zs \} $ 
is a well-ordered subset of $\Gamma $ 
by the assumption that 
$\w $ is an element of $(\Gammap )^3$ 
(cf.~\cite[Lemma 6.1]{SU2}).

With the notation and assumption above, 
we have the following lemma.

\begin{lem}\label{lem:91113}
If one of the following three conditions is satisfied, 
then we have 
$\eta (\theta ;f,g)>\eta _i$ for $i=1,2$$:$

\noindent{\rm (i)} $d=2$, 
$\degw df\wedge dg>\degw g$ 
and $(2l+1)\degw f=l\degw g$ 
for some integer $l\geq 3$. 

\noindent{\rm (ii)} 
$d\geq 3$ and $\degw df\wedge dg>(d-1)\degw f$.

\noindent{\rm (iii)} 
$d\geq 9$ and $d\neq 10,12$. 
\end{lem}
\begin{proof}
Take $h\in k[f,g]$ and $\Phi \in k[f][z]$ 
such that 
$\eta (\theta ;f,g)=\degw (\theta (g)+fh)$ 
and $h=\Phi (g)$, 
and put $\Psi =\theta +f\Phi $. 
Then, 
we have $\eta (\theta ;f,g)=\degw \Psi (g)$. 
Note that 
\begin{equation}\label{eq:pfstart:91113}
\deg _z\Psi 
=\max \{ \deg _z\theta ,\deg _zf\Phi \} 
\geq \deg _z\theta =d, 
\end{equation}
since the leading coefficient of $\theta $ 
is an element of $k^{\times }$, 
while that of $f\Phi $ is a multiple of $f$. 
Hence, 
we know that 
\begin{equation}\label{eq:SUineqpf0}
\degw ^g\Psi =\deg _{\w _g}\Psi 
\geq (\deg _z\Psi )\degw g\geq d\degw g. 
\end{equation}

First, 
assume that $m_{\w }^g(\Psi )=0$. 
Then, we have $\degw ^g\Psi =\degw \Psi (g)$. 
Since $\eta (\theta ;f,g)=\degw \Psi (g)$, 
we get 
$\eta (\theta ;f,g)\geq d\degw g$ by (\ref{eq:SUineqpf0}). 
Assume that (i) is satisfied. 
Then, 
it follows that $\eta (\theta ;f,g)\geq 2\degw g$, 
since $d=2$. 
Because 
$$
\degw f=\frac{l}{2l+1}\degw g
$$
for some $l\geq 3$, 
we know that $\degw f$ is less than $(1/2)\degw g$. 
Hence, 
we see that $\eta _i<2\degw g$ for $i=1,2$. 
Therefore, 
we get $\eta (\theta ;f,g)>\eta _i$ for $i=1,2$. 
If (ii) or (iii) is satisfied, 
then we have 
$\eta (\theta ;f,g)\geq 3\degw g$ for $i=1,2$, 
since $d\geq 3$. 
Because $\degw f<\degw g$ by (1), 
we see that $\eta _i<3\degw g$ for $i=1,2$. 
Therefore, 
we get $\eta (\theta ;f,g)>\eta _i$ for $i=1,2$.

Next, 
assume that $m_\w ^g(\Psi )\geq 1$. 
Then, 
$\degw \Psi (g)$ is less than 
$\degw ^g\Psi =\degw ^S\Psi (g)$, 
where $S:=\{ f,g\} $. 
By Lemma~\ref{lem:SUineq1} (i), 
there exist $p,q\in \N $ with $\gcd (p,q)=1$ 
such that $(g^\w )^p\approx (f^\w )^q$. 
Then, we have $2\leq p<q$ by (1) and (2). 
Let $a$ and $b$ 
be the quotient and remainder of $\deg _z\Psi $ 
divided by $p$. 
Then, 
we have 
$a\geq m_\w ^g(\Psi )$ by Lemma~\ref{lem:SUineq1} (iv). 
Since $m_\w ^g(\Psi )\geq 1$ by assumption, 
it follows that $a\geq 1$. 
Set $\delta =p^{-1}\degw f$. 
Then, 
we have $\degw f=p\delta $, $\degw g=q\delta $ 
and 
\begin{equation}\label{eq:eta pq pf}
\eta _1=\left(p+\frac{3}{2}q\right)\delta ,\quad 
\eta _2=(2p+q)\delta . 
\end{equation}
Since $f$ does not belong to $k$ by (3), 
and $\w $ is an element of $(\Gammap )^3$, 
we have $\degw df=\degw f$ 
by (\ref{eq:deg df = deg f}) 
and the note following it. 
Hence, we get 
\begin{align*}
\eta (\theta ;f,g)
=\degw \Psi (g)
\geq \degw ^g\Psi +m_{\w }^g(\Psi )
(\degw df\wedge dg-\degw f-\degw g)
\end{align*}
by Theorem~\ref{thm:SUineq}. 
Since $m_{\w }^g(\Psi )\leq a$, 
and $\degw df\wedge dg\leq \degw f+\degw g$ 
by (\ref{eq:inequomega}), 
we have 
\begin{align*}
m_{\w }^g(\Psi )
(\degw df\wedge dg-\degw f-\degw g)
&\geq a(\degw df\wedge dg-\degw f-\degw g) \\
&=a(\degw df\wedge dg-p\delta -q\delta ). 
\end{align*}
By (\ref{eq:SUineqpf0}), 
we know that $\degw ^g\Psi \geq 
(ap+b)q\delta $. 
Therefore, we get 
\begin{equation}\label{eq:SUineqpf1}
\begin{aligned}
\eta (\theta ;f,g)
&\geq 
\degw ^g\Psi +a(\degw df\wedge dg-p\delta -q\delta ) \\
&\geq (ap+b)q\delta +a(\degw df\wedge dg-p\delta -q\delta ). 
\end{aligned}
\end{equation}

First, 
assume that (i) is satisfied. 
Then, we have $(p,q)=(l,2l+1)$. 
Since 
$\degw df\wedge dg>\degw g=q\delta $ by assumption, 
and $a\geq 1$ and $b\geq 0$, it follows that 
\begin{align*}
\eta (\theta ;f,g)
>(ap)q\delta +a(q-p-q)\delta 
=ap(q-1)\delta\geq 2l^2\delta 
\end{align*}
by (\ref{eq:SUineqpf1}). 
By (\ref{eq:eta pq pf}), 
we have 
$\eta _1=(4l+3/2)\delta $ 
and $\eta _2=(4l+1)\delta $. 
Since $l\geq 3$ by assumption, 
we know that $\eta _i<2l^2\delta $ for $i=1,2$. 
Therefore, we conclude that 
$\eta (\theta ;f,g)>\eta _i$ for $i=1,2$.

Next, 
assume that (ii) is satisfied. 
Then, 
we have $\degw df\wedge dg>2\degw f=2p\delta $. 
Hence, we get 
$$
\eta (\theta ;f,g)
>(ap+b)q\delta +a\bigl(2p-p-q\bigr)\delta 
=\Bigl(a\bigl(p+(p-1)q\bigr)+bq\Bigr)\delta =:\alpha  
$$
by (\ref{eq:SUineqpf1}). 
We show that 
$\alpha \geq (p+2q)\delta $. 
If $p\geq 3$, then this is clear, 
since 
$a\geq 1$ and $b\geq 0$. 
Assume that $p=2$. 
Then, we have 
$2a+b=\deg _z\Psi \geq d\geq 3$ 
with $0\leq b\leq 1$. 
Hence, 
we get $a\geq 2$ or $(a,b)=(1,1)$. 
Thus, we know that 
$\alpha \geq (p+2q)\delta $. 
Since $p<q$, 
we see from (\ref{eq:eta pq pf}) that 
$\eta _i<(p+2q)\delta $ for $i=1,2$. 
Therefore, 
we conclude that 
$\eta (\theta ;f,g)>\eta _i$ for $i=1,2$.

Finally, 
assume that (iii) is satisfied. 
By (3), we have $df\wedge dg\neq 0$. 
Since $\w $ is an element of $(\Gammap )^3$, 
it follows that $\degw df\wedge dg>0$. 
Hence, (\ref{eq:SUineqpf1}) gives that 
\begin{align}\label{eq:SUineqpf11}
\eta (\theta ;f,g)&> 
\degw ^g\Psi -a(p+q)\delta \\
&\geq (ap+b)q\delta -a(p+q)\delta 
=\Bigl(a\bigl((p-1)q-p\bigr)+bq\Bigr)\delta =:\beta . \notag
\end{align}
Note that $\beta >a(p-2)q\delta $, 
since $a\geq 1$, $b\geq 0$ and $p<q$. 
First, assume that $p\geq 3$. 
We show that $a(p-2)\geq 3$. 
Since $a\geq 1$, 
this is clear if $p\geq 5$. 
If $p=4$, 
then we have $a\geq 2$, 
since $4a+b=\deg _z\Psi \geq d\geq 9$ 
with $0\leq b\leq 3$. 
Hence, 
we get $a(p-2)\geq 4$. 
If $p=3$, 
then we have $a\geq 3$, 
since $3a+b=\deg _z\Psi \geq d\geq 9$ 
with $0\leq b\leq 2$. 
Hence, we get $a(p-2)\geq 3$. 
Thus, 
we know that $\beta >3q\delta $. 
Because $\eta _i<3q\delta $ for $i=1,2$, 
we conclude that 
$\eta (\theta ;f,g)>\eta _i$ for $i=1,2$. 
Next, assume that $p=2$. 
Then, we have 
$\eta _1=((3/2)q+2)\delta $ and 
$\eta _2=(q+4)\delta $ by (\ref{eq:eta pq pf}). 
Since $q\geq p+1=3$, 
it follows that 
$\eta _i\leq (5q-8)\delta $ for $i=1,2$. 
First, 
consider the case where $\deg _z\Psi \neq 10,12$. 
Since $2a+b=\deg _z\Psi \geq d\geq 9$ 
with $0\leq b\leq 1$, 
we know that $a\geq 7$, 
or $4\leq a\leq 6$ and $b=1$. 
If $a\geq 7$, 
then we have 
$$
\beta \geq 7(q-2)\delta 
=\bigl( 5q+(2q-14)\bigr) \delta 
\geq (5q-8)\delta , 
$$
since $p=2$ and $q\geq 3$. 
If $4\leq a\leq 6$ and $b=1$, 
then we have 
$$
\beta \geq (4(q-2)+q)\delta =(5q-8)\delta . 
$$
Therefore, 
we conclude that 
$\eta (\theta ;f,g)>\eta _i$ for $i=1,2$. 
Next, 
consider the case where 
$\deg _z\Psi =10$ or $\deg _z\Psi =12$. 
Since $d\neq 10,12$ by assumption, 
$\deg _z\Psi$ is not equal to $d=\deg _z\theta $. 
Hence, we have 
$\deg _z\theta <\deg _zf\Phi $ 
by (\ref{eq:pfstart:91113}). 
Since $\theta $ is an element of $k[z]$, 
we get 
$$
\deg _{\w _g}\theta 
=(\deg _z\theta )\degw g
<(\deg _zf\Phi )\degw g
\leq \deg _{\w _g}f\Phi . 
$$
Thus, we obtain 
\begin{align*}
&\degw ^g\Psi 
=\deg _{\w _g}(\theta +f\Phi )=\deg _{\w _g}f\Phi  
=\degw f+\deg _{\w _g}\Phi \\
&\quad \geq \degw f+(\deg _z\Phi )\degw g 
=\bigl(2+(\deg _z\Phi )q\bigr)\delta . 
\end{align*}
Since $\deg _z\theta <\deg _zf\Phi $, 
we have 
$\deg _z\Phi =\deg _zf\Phi =\deg _z \Psi $ 
by (\ref{eq:pfstart:91113}). 
Because $\deg _z \Psi =2a+b\geq 2a$, 
it follows that 
$\degw ^g\Psi \geq (2+2aq)\delta $ 
by the preceding inequality. 
Hence, 
the first inequality of (\ref{eq:SUineqpf11}) 
gives that 
\begin{align*}
\eta (\theta ;f,g)>
\degw ^g\Psi -a(2+q)\delta 
\geq (2+2aq) \delta -a(2+q)\delta 
=\bigl(a(q-2)+2\bigr)\delta . 
\end{align*}
Since $2a+b=\deg _z\Psi \geq 10$ 
with $0\leq b\leq 1$, 
we have $a\geq 5$. 
Hence, we know that $a(q-2)+2\geq 5(q-2)+2=5q-8$, 
since $q\geq 3$. 
Therefore, we get 
$\eta (\theta ;f,g)>\eta _i$ 
for $i=1,2$. 
\end{proof}

The following 
criterion for wildness 
will be used in Section~\ref{sect:proof:tacoord}.

\begin{lem}\label{lem:SUcriterion}
$F=(f_1,f_2,f_3)\in \Aut (\kx /k)$ is wild 
if the following conditions hold$:$ 

\noindent{\rm (a)} 
$f_1^{\w }$, $f_2^{\w }$ and $f_3^{\w }$ 
are algebraically dependent over $k$. 

\noindent{\rm (b)} 
$f_1^{\w }$ and $f_2^{\w }$ do not belong to 
$k[f_2,f_3]^{\w }$ and $k[f_3^{\w }]$, respectively. 

\noindent{\rm (c)} 
$\degw f_1\geq \degw f_2+\degw f_3$. 

\noindent{\rm (d)} 
$\degw f_2>\degw f_3$. 

\noindent{\rm (e)} 
$2\degw f_1\neq 3\degw f_2$. 
\end{lem}
\begin{proof}
By Lemma~\ref{lem:minimal autom}, 
(a) implies that $\degw \phi >|\w |$. 
Hence, 
it suffices to check that $\phi $ admits 
no elementary reduction for the weight $\w $, 
and no Shestakov-Umirbaev reduction for the weight $\w $ 
by virtue of Theorem~\ref{thm:SUcriterion}. 
Suppose to the contrary that 
$\phi $ admits a Shestakov-Umirbaev reduction 
for the weight $\w $. 
By definition, 
there exist $\sigma $ and $G$ such that 
$(F_{\sigma },G_{\sigma })$ satisfies 
the Shestakov-Umirbaev condition for the weight $\w $. 
Then, $(F_{\sigma },G_{\sigma })$ has the properties 
listed before Lemma~\ref{lem:SUdependent}. 
By (P7), 
we know that  
$f_{\sigma (1)}>f_{\sigma (2)}$ and 
$f_{\sigma (1)}\geq f_{\sigma (3)}$. 
Since $\degw f_1>\degw f_2>\degw f_3$ 
by (c) and (d), 
it follows that $\sigma (1)=1$. 
Hence, we have 
$\degw g_1=s\delta $ and 
$\degw g_{\sigma (2)}=2\delta $ 
for some $s\geq 3$ by (P1), 
and $\degw g_1\geq \degw f_1$ 
and $\degw g_{\sigma (2)}=\degw f_{\sigma (2)}$ by (SU2). 
Thus, 
we get 
$$
\degw f_{\sigma (3)}>(s-2)\delta 
=\degw g_1-\degw g_{\sigma (2)}\geq 
\degw f_1-\degw f_{\sigma (2)} 
$$
by (P2). 
This contradicts (c). 
Therefore, 
$\phi $ admits no Shestakov-Umirbaev reduction 
for the weight $\w $.

Next, we show that $\phi $ admits 
no elementary reduction for the weight $\w $. 
Since $f_1^{\w }$ does not belong to 
$k[f_2,f_3]^{\w }$ by (b), 
we check that $f_2^{\w }$ and $f_3^{\w }$ 
do not belong to $k[f_1,f_3]^{\w }$ and 
$k[f_1,f_2]^{\w }$, 
respectively.

First, 
suppose to the contrary that 
$f_2^{\w }$ belongs to $k[f_1,f_3]^{\w }$. 
Then, 
there exists $h\in k[f_1,f_3]$ 
such that $h^{\w }=f_2^{\w }$. 
We show that $\degw f_2=\degw h$ 
is greater than $\degw f_1-\degw f_3$ 
by applying Lemma~\ref{lem:SUineq1} (iii) 
with $f=f_3$, $g=f_1$ and $\phi =h$. 
Then, we get a contradiction to (c). 
Since $f_2^{\w }$ 
does not belong to $k[f_3^{\w }]$ by (b), 
we know that $h$ belongs to 
$k[f_1,f_3]\sm k[f_3]$. 
Since 
$\degw h=\degw f_2$ is less than $\degw f_1$ by (c), 
this implies that $\degw h<\degw ^{S_2}h$, 
where $S_2:=\{ f_1,f_3\} $. 
By (c), 
we have $\degw f_1>\degw f_3$, 
and $\degw f_1>\degw f_2=\degw h$. 
By (b), 
$f_1^{\w }$ does not belong to $k[f_2^{\w },f_3^{\w }]$, 
and hence does not belong to $k[f_3^{\w }]$. 
Thus, 
we conclude from Lemma~\ref{lem:SUineq1} (iii) that 
$\degw h>\degw f_1-\degw f_3$. 
This proves that 
$f_2^{\w }$ does not belong to $k[f_1,f_3]^{\w }$.

Next, 
suppose to the contrary that 
$f_3^{\w }$ belongs to 
$k[f_1,f_2]^{\w }$. 
Then, 
there exists $h\in k[f_1,f_2]$ 
such that $h^{\w }=f_3^{\w }$. 
We show that 
$(f_1^{\w })^2\approx (f_2^{\w })^3$ 
by applying Lemma~\ref{lem:SUineq1} (i) and 
the last part of 
Lemma~\ref{lem:SUineq1} (iii) 
with $f=f_2$, $g=f_1$ and $\phi =h$. 
Then, we get a contradiction to (e). 
By (c) and (d), 
we have $\degw h=\degw f_3<\degw f_i$ for $i=1,2$. 
Hence, 
we get $\degw h<\degw ^{S_3}h$, 
where $S_3:=\{ f_1,f_2\} $. 
By (c) and (d), 
we have $\degw f_1>\degw f_2$ 
and $\degw f_2>\degw f_3=\degw h$. 
By (b), 
$f_1^{\w }$ does not belong to $k[f_2^{\w },f_3^{\w }]$, 
and hence does not belong to $k[f_2^{\w }]$. 
Thus, 
we conclude from Lemma~\ref{lem:SUineq1} (i) and the last part 
of Lemma~\ref{lem:SUineq1} (iii) 
that $(f_1^{\w })^2\approx (f_2^{\w })^3$. 
This proves that  
$f_3^{\w }$ does not belong to $k[f_1,f_2]^{\w }$. 
\end{proof}

\chapter{Totally wild coordinates} 
\label{chapter:atcoord}

\section{Main result}
\setcounter{equation}{0}
\label{sect:wild coordinate}

In what follows, 
we always assume that $n=3$ unless otherwise stated. 
The purpose of this chapter is to 
give coordinates of $\kx $ over $k$ 
some of which are totally wild, 
and others are quasi-totally wild, 
but not totally wild.

Let $\theta (z)$ be an element of $k[z]$ 
with $d:=\deg _z\theta (z)\geq 1$. 
Define $D_{\theta }\in \Der_k\kx $ by 
$$
D_{\theta }(x_1)=-\theta '(x_2),\quad 
D_{\theta }(x_2)=x_3,\quad D_{\theta }(x_3)=0. 
$$
Then, $D_{\theta }$ is locally nilpotent, 
since $D_{\theta }$ is triangular if $x_1$ and $x_3$ 
are interchanged. 
Since 
$f_{\theta }:=x_1x_3+\theta (x_2)$ 
belongs to $\ker D_{\theta }$, 
it follows that 
$f_{\theta }D_{\theta }$ is locally nilpotent. 
Set $\sigma _{\theta }=\exp f_{\theta }D_{\theta }$ 
and $y_i=\sigma _{\theta }(x_i)$ for $i=1,2,3$. 
Then, 
we consider the tame intersection 
$$
G_{\theta }:=\Aut (\kx /k[y_1])\cap \T (k,\x ). 
$$

Let us describe $y_1$ concretely. 
Since $D_{\theta }(f_{\theta })=D_{\theta }(x_3)=0$, 
we have 
$\sigma _{\theta }(f_{\theta })=f_{\theta }$ and $y_3=x_3$, 
and so 
$$
y_1x_3+\theta (y_2)
=\sigma _{\theta }(x_1x_3+\theta (x_2))
=\sigma _{\theta }(f_{\theta })=f_{\theta }=
x_1x_3+\theta (x_2). 
$$
Hence, we get 
\begin{equation}\label{eq:y_11}
y_1=x_1+\frac{\theta (x_2)-\theta (y_2)}{x_3}. 
\end{equation}
Note that 
$y_2=x_2+f_{\theta }x_3$, 
since 
$(f_{\theta }D_{\theta })(x_2)=f_{\theta }x_3$ 
and $(f_{\theta }D_{\theta })^2(x_2)=0$. 
Therefore, 
(\ref{eq:y_11}) gives that 
\begin{equation}\label{eq:y_1}
y_1=x_1-\frac{\theta (x_2+f_{\theta }x_3)-\theta (x_2)}{x_3}\\
=x_1-\sum _{i=1}^d
\frac{1}{i!}\theta ^{(i)}(x_2)
f_{\theta }^ix_3^{i-1}, 
\end{equation}
where $\theta ^{(i)}(z)$ denotes the 
$i$-th order derivative of $\theta (z)$.

Now, 
let $c$ and $c'$ be the coefficients of 
$z^d$ and $z^{d-1}$ in $\theta (z)$, 
respectively. 
Put $\kappa =-c'/(cd)$ 
and write 
$$
\theta (z)=\sum _{i=0}^du_i(z-\kappa )^i, 
$$ 
where $u_i\in k$ for each $i$. 
Then, 
we have 
$u_d=c$, $u_{d-1}=0$ and $u_0=\theta (\kappa )$. 
Let $e\in \N $ be the 
greatest common divisor of 
$U:=\{ 2i-1\mid i=1,\ldots ,d\text{ with }u_i\neq 0\} $, 
i.e., 
the positive generator of the ideal $(U)$ of $\Z $. 
Then, we define 
$$
T_{\theta }=\{ \zeta \in k^{\times }
\mid \zeta ^e=1\} . 
$$
For each $\zeta \in T_{\theta }$, 
we define an element $\phi _{\zeta }$ of 
$J(k;x_3,x_2,x_1)$ by 
\begin{align*}
\phi _{\zeta }(x_1)&=x_1+g_{\zeta } \\
\phi _{\zeta }(x_2)
&=\zeta ^2(x_2-\kappa )+\zeta (\zeta -1)\theta (\kappa )x_3
+\kappa \\
\phi _{\zeta }(x_3)&=\zeta x_3, 
\end{align*}
where 
$$
g_{\zeta }:=
\frac{\zeta \theta (x_2)-
\phi _{\zeta }(\theta (x_2))
+(1-\zeta )\theta (\kappa )}
{\zeta x_3}. 
$$
Here, 
we note that 
$g_{\zeta }$ belongs to $\kx $ 
if and only if $\zeta $ belongs to $T_{\theta }$ 
for $\zeta \in k^{\times }$. 
To see this, observe 
that the numerator of $g_{\zeta }$ 
is congruent to 
\begin{align*}
\zeta \sum _{i=0}^du_i(x_2-\kappa )^i
-\sum _{i=0}^du_i\bigl(
\zeta ^2(x_2-\kappa )\bigr)^i
+(1-\zeta )\theta (\kappa ) 
=\sum _{i=1}^du_i\zeta (1-\zeta ^{2i-1})
(x_2-\kappa )^i
\end{align*}
modulo $x_3\kx $, 
since $\theta (\kappa )=u_0$. 
Then, the right-hand side of this equality belongs to $x_3\kx $ 
if and only if it is equal to zero. 
This condition is satisfied if and only if 
$\zeta ^j=1$ for each $j\in U$, 
and hence if and only if 
$\zeta $ belongs to $T_{\theta }$.

The following is the main result.

\begin{thm}\label{thm:G_{theta}}
In the notation above, 
$\phi _{\zeta }$ belongs to $G_{\theta }$ 
for each $\zeta \in T_{\theta }$, 
and 
$$
\iota :T_{\theta }\ni \zeta \mapsto \phi _{\zeta }\in G_{\theta }
$$ 
is an injective homomorphism of groups. 
If $d\geq 9$ and $d\neq 10,12$, 
then $\iota $ is surjective. 
\end{thm}

Theorem~\ref{thm:G_{theta}} 
immediately implies the following corollary.

\begin{cor}\label{cor:aw}
Assume that $d\geq 9$ and $d\neq 10,12$. 
Then, the following assertions hold$:$ 

\noindent{\rm (i)} 
$y_1$ is a quasi-totally wild coordinate of $\kx $ over $k$. 

\noindent{\rm (ii)} 
$y_1$ is a totally wild coordinate of $\kx $ over $k$ 
if and only if $T_{\theta }=\{ 1\} $. 
\end{cor}

Recall that $k$ is said to be {\it real} 
if $-1$ is not a sum of squares in $k$. 
We remark that the roots of unity 
in a real field are only 1 and $-1$. 
Actually, 
real fields are ordered fields 
(cf.~\cite[Chapter XI, Section 2]{Lang}). 
Hence, if $\zeta ^2\neq 1$, 
then we have $\zeta ^2>1$ or $0\leq \zeta ^2<1$. 
If $\zeta ^e=1$ for some $e\geq 1$, 
then it follows that $1=(\zeta ^2)^e>1^e=1$ 
or $1=(\zeta ^2)^e<1^e=1$, a contradiction.

By the following proposition, 
we see that there exist a number of totally wild coordinates.

\begin{prop}
If one of the following conditions is satisfied, 
then we have $T_{\theta }=\{ 1\} $$:$ 

\noindent{\rm (1)} $k$ is a real field. 

\noindent{\rm (2)} $u_{d-i}\neq 0$ 
for some $i\geq 2$ such that $\gcd (i,2d-1)=1$. 
\end{prop}
\begin{proof}
Since $u_d$ is nonzero, 
$e$ must be a divisor of $2d-1$. 
Hence, $e$ is an odd number. 
This implies that $T_{\theta }=\{ 1\} $ 
if $k$ is a real field.

If $u_{d-i}\neq 0$, 
then $e$ divides $2(d-i)-1$. 
Hence, $e$ 
divides 
$$
\gcd (2(d-i)-1,2d-1)=
\gcd (2i,2d-1)=
\gcd (i,2d-1)=1. 
$$
Thus, we have $e=1$. 
Therefore, 
we get $T_{\theta }=\{ 1\} $. 
\end{proof}

We remark that $\sigma _{z^2}$ 
is equal to Nagata's automorphism 
defined in (\ref{eq:Nagataautom}). 
Clearly, 
$y_3=x_3$ is a tame coordinate, 
while $y_1$ and $y_2$ are wild coordinates 
due to Umirbaev-Yu~\cite{UY}. 
Observe that 
$y_2=x_1x_3^2+(x_2+x_2^2x_3)$ 
is killed by $D\in \Der _k\kx $ defined by 
\begin{equation}\label{eq:quasi-absolutely wild}
D(x_1)=1+2x_2x_3,\quad 
D(x_2)=-x_3^2,\quad 
D(x_3)=0. 
\end{equation}
Since $D$ is triangular if $x_1$ and $x_3$ 
are interchanged, 
we know that $\exp D$ is tame. 
Hence, 
$y_2$ is not exponentially wild. 
Thus, 
a wild coordinate of $\kx $ is not always exponentially wild. 
Since $\theta (z)=z^2$, we have $e=2d-1=3$. 
Hence, we get $T_{z^2}=\{ \zeta \in k\mid \zeta ^3=1\} $. 
Therefore, $y_1$ is not totally wild 
if $k$ contains a primitive third root of unity. 
The author believes that $y_1$ is quasi-totally wild, 
but it remains open.

In the rest of this section, 
we prove the first part of Theorem~\ref{thm:G_{theta}}.

\begin{lem}\label{lem:commute}
For each $\zeta \in T_{\theta}$, 
we have 
$$
\phi _{\zeta }(y_1)=y_1,\quad 
\phi _{\zeta }(y_2-\kappa )=\zeta ^2(y_2-\kappa ),\quad 
\phi _{\zeta }(y_3)=\zeta y_3. 
$$
If furthermore $\theta (\kappa )=0$, 
then we have 
$$
\phi _{\zeta }(x_1)=x_1,\quad 
\phi _{\zeta }(x_2-\kappa )=\zeta ^2(x_2-\kappa ),\quad 
\phi _{\zeta }(x_3)=\zeta x_3, 
$$
and so 
$\sigma _{\theta }\circ \phi _{\zeta }
=\phi _{\zeta }\circ \sigma _{\theta }$. 
\end{lem}
\begin{proof}
Write $\phi =\phi _{\zeta }$ for simplicity. 
Since $y_3=x_3$, 
we have $\phi (y_3)=\zeta y_3$ 
by the definition of $\phi _{\zeta }$. 
A direct computation shows that 
\begin{align*}
&\phi (f_{\theta })=\phi \bigl(x_1x_3+\theta (x_2)\bigr) 
=(x_1+g_{\zeta })(\zeta x_3)+\theta \bigl(\phi (x_2)\bigr)\\ 
&\quad =\zeta \bigl( x_1x_3+\theta (x_2)\bigr) 
-\zeta \theta (x_2)+g_{\zeta }(\zeta x_3)
+\theta \bigl(\phi (x_2)\bigr)
=\zeta f_{\theta }+(1-\zeta )\theta (\kappa ). 
\end{align*}
Since $y_2=x_2+f_{\theta }x_3$, 
it follows that 
\begin{align*}
&\phi (y_2-\kappa )
=\phi (x_2-\kappa )+\phi (f_{\theta })\phi (x_3)\\
&\quad =\zeta ^2(x_2-\kappa )
+\zeta (\zeta -1)\theta (\kappa )x_3+
\bigl(\zeta f_{\theta }+(1-\zeta )\theta (\kappa )\bigr)
(\zeta x_3) \\
&\quad =\zeta ^2(x_2-\kappa +f_{\theta }x_3)=\zeta ^2(y_2-\kappa ). 
\end{align*}
Hence, we have
\begin{equation}\label{eq:pf:commute1}
\phi \bigl(\theta (y_2)\bigr)=\sum _{i=0}^du_i
\phi (y_2-\kappa )^i
=\zeta \sum _{i=1}^du_i\zeta ^{2i-1}(y_2-\kappa )^i+u_0. 
\end{equation}
Since $\zeta $ is an element of $T_{\theta }$, 
we have $\zeta ^{2i-1}=1$ for each $i\geq 1$ with $u_i\neq 0$. 
Thus, 
the right-hand side of (\ref{eq:pf:commute1}) 
is equal to 
\begin{equation}\label{eq:pf:commute2}
\zeta \sum _{i=1}^du_i(y_2-\kappa )^i+u_0
=\zeta \theta (y_2)+(1-\zeta )u_0. 
\end{equation}
Therefore, 
it follows from (\ref{eq:y_11}) that 
\begin{align*}
\phi (y_1)
&=\phi \left( 
x_1+\frac{\theta (x_2)-\theta (y_2)}{x_3}
\right) \\
&=x_1+g_{\zeta }
+\frac{\phi \bigl(\theta (x_2)\bigr)-
\bigl(
\zeta \theta (y_2)+(1-\zeta )u_0\bigr)}{\zeta x_3}\\
&=x_1+\frac{\zeta \theta (x_2)
-\zeta \theta (y_2)}{\zeta x_3}=y_1. 
\end{align*}
Next, 
assume that $\theta (\kappa )=0$. 
Then, 
we have $\phi (x_2-\kappa )=\zeta ^2(x_2-\kappa )$. 
Hence, 
we obtain 
$\phi (\theta (x_2))=
\zeta \theta (x_2)+(1-\zeta )\theta (\kappa )$ 
from (\ref{eq:pf:commute1}) and (\ref{eq:pf:commute2}) 
with $y_2$ replaced by $x_2$. 
This implies that $g_{\zeta }=0$. 
Thus, 
we get $\phi (x_1)=x_1$. 
By definition, 
we have $\phi (x_3)=\zeta x_3$. 
Since $\sigma _{\theta }(x_i)=y_i$ for $i=1,2,3$, 
it is easy to see that 
$\sigma _{\theta }\circ \phi =\phi \circ \sigma _{\theta }$. 
\end{proof}

Since $\phi _{\zeta }$ is tame, 
and $\phi _{\zeta }(y_1)=y_1$ by 
Lemma~\ref{lem:commute}, 
we know that 
$\phi _{\zeta }$ belongs to $G_{\theta }$ 
for each $\zeta \in T_{\theta }$. 
By Lemma~\ref{lem:commute}, 
we have 
\begin{gather*}
(\phi _{\alpha }\circ \phi _{\beta })(y_1)
=y_1=\phi _{\alpha \beta }(y_1)\\
(\phi _{\alpha }\circ \phi _{\beta })(y_2-\kappa )
=\alpha ^2\beta ^2(y_2-\kappa )
=\phi _{\alpha \beta }(y_2-\kappa )\\
(\phi _{\alpha }\circ \phi _{\beta })(y_3)
=\alpha \beta y_3=\phi _{\alpha \beta }(y_3)
\end{gather*}
for each $\alpha ,\beta \in T_{\theta }$. 
Hence, 
$\iota $ is a homomorphism of groups. 
If $\phi _{\zeta }=\id _{\kx }$ for $\zeta \in T_{\theta }$, 
then we have $\zeta x_3=\phi _{\zeta }(x_3)=x_3$, 
and hence $\zeta =1$. 
Therefore, 
$\iota $ is injective.

Sections~\ref{sect:proof:tacoord}, 
\ref{sect:wild3pf2} and \ref{sect:tawildpf3} 
are devoted to proving the following theorem. 
Let $\Lambda $ be the totally ordered 
additive group $\Z ^2$ equipped with 
the lexicographic order such that $\e _1>\e _2$. 
We consider the element 
$\vv :=(\e _1,\e _2,\e _1)$ of $\Lambda ^3$.

\begin{thm}\label{thm:tawild}
{\rm (i)} 
Let $\Gamma $ be a 
totally ordered additive group, 
and $\w \in (\Gammap )^3$. 
If $d\geq 9$ and $d\neq 10,12$, 
then we have $\degw \phi (y_2)>\degw \phi (x_3)$ 
for each $\phi \in G_{\theta }$.

\noindent
{\rm (ii)} 
Let $\vv $ be as above, 
and assume that $d\geq 9$. 
If $\degv \phi (y_2)>\degv \phi (x_3)$ 
for $\phi \in G_{\theta }$, 
then there exists $\zeta \in T_{\theta }$ 
such that $\phi =\phi _{\zeta }$. 
\end{thm}

By this theorem, 
it follows that $\iota $ is surjective 
if $d\geq 9$ and $d\neq 10,12$. 
Thus, 
the proof of Theorem~\ref{thm:G_{theta}} is completed.

\section{Proof (I)}\label{sect:proof:tacoord}
\setcounter{equation}{0}

Let $\phi$ be an element of $\Aut (\kx /k[y_1])$. 
First, 
we investigate when $\phi $ 
is wild in general. 
Set $z_i=\phi (x_i)$ for $i=1,2,3$. 
Then, 
we have 
\begin{equation}\label{eq:f_1f_2}
\begin{aligned}
z_1&
=\bigl(\phi (f_{\theta })-\theta (z_2)\bigr)z_3^{-1}\\ 
z_2&
=\phi (y_2-f_{\theta }x_3)\\
z_3&
=\phi (y_3), 
\end{aligned}
\end{equation}
since 
$z_1z_3+\theta (z_2)
=\phi (f_{\theta })$, 
$x_2=y_2-f_{\theta }x_3$, 
and $x_3=y_3$. 
Let $\w =(w_1,w_2,w_3)$ be an element of 
$(\Gammap )^3$. 
Then, 
there exist
\begin{align*}
\gamma ^{\w }_0&
:=\min \{ \degw \bigl(\phi (f_{\theta })+h_0\bigr)\mid 
h_0\in k[z_3]\} \\
\gamma ^{\w }_1&:=\min \{ 
\degw \bigl(\theta (z_2)+h_1z_3\bigr)\mid 
h_1\in k[z_2,z_3]\} \\
\gamma ^{\w }_2&:=\min \{ \degw (z_2+h_2)\mid h_2\in k[z_3]\} 
\end{align*}
by the well-orderedness of $\{ \degw h\mid h\in \kx \sm \zs \} $.

Take $h_0,h_2\in k[z_3]$ and $h_1\in k[z_2,z_3]$ 
such that 
$$
\gamma ^{\w }_0=\degw \bigl(\phi (f_{\theta })+h_0\bigr),\quad 
\gamma ^{\w }_1=\degw \bigl(\theta (z_2)+h_1z_3\bigr),\quad 
\gamma ^{\w }_2=\degw (z_2+h_2). 
$$
Then, 
$\phi (f_{\theta })+h_0$ and $z_2+h_2$ 
do not belong to $k$, 
for otherwise 
$\phi (f_{\theta })$ or $z_2$ belongs to $k[z_3]$, 
and so $f_{\theta }$ or $x_2$ 
belongs to $k[x_3]$, 
a contradiction. 
Similarly, 
$\theta (z_2)+h_1z_3$ 
does not belong to $k$, 
since $z_2$ and $z_3$ 
are algebraically independent over $k$. 
Hence, 
we have $\gamma _i^{\w }>0$ for $i=0,1,2$ 
by the choice of $\w $. 
For this reason, 
the $\w $-degrees of 
$\phi (f_{\theta })+h_0$ and $z_2+h_2$ 
are not changed if we subtract 
the constant terms from $h_0$ and $h_2$. 
Therefore, 
we may assume that $h_0$ and $h_2$ are elements of $z_3k[z_3]$. 
We remark that 
$\bigl(\phi (f_{\theta })+h_0\bigr)^{\w }$ and $(z_2+h_2)^{\w }$ 
do not belong to $k[z_3]^{\w }=k[z_3^{\w }]$ 
by the minimality of $\gamma ^{\w }_0$ and $\gamma ^{\w }_2$.

Set $\gamma _3^{\w }=\degw z_3>0$. 
Then, we have the following lemma.

\begin{lem}\label{lem:taiguu}
If $\degw \phi (y_2)<\gamma ^{\w}_0+\gamma ^{\w }_3$, 
then we have $\gamma ^{\w }_i<\gamma ^{\w }_2$ for $i=0,3$. 
\end{lem}
\begin{proof}
Since $\gamma _i^{\w }>0$ for $i=0,3$, 
it suffices to show that 
$\gamma _0^{\w }+\gamma _3^{\w }\leq \gamma _2^{\w }$. 
Since $h_2z_3^{-1}$ belongs to $k[z_3]$ 
by the choice of $h_2$, 
we have 
$\degw (\phi (f_{\theta })-h_2z_3^{-1})\geq \gamma ^{\w }_0$ 
by the minimality of $\gamma ^{\w }_0$. 
Hence, 
we get 
\begin{equation}\label{eq:pf:taiguu}
\degw (\phi (f_{\theta })-h_2z_3^{-1})z_3\geq 
\gamma ^{\w }_0+\gamma ^{\w }_3>\degw \phi (y_2) 
\end{equation}
by assumption. 
Since $z_2=\phi (y_2)-\phi (f_{\theta })z_3$ 
by (\ref{eq:f_1f_2}), 
we have 
$$
\gamma _2^{\w }=\degw (z_2+h_2)=
\degw \bigl(\phi (y_2)-(\phi (f_{\theta })-h_2z_3^{-1})z_3\bigr). 
$$
Thus, 
we see from (\ref{eq:pf:taiguu}) that 
$\gamma _2^{\w }$ is equal to the left-hand side of 
(\ref{eq:pf:taiguu}), 
and is at least $\gamma ^{\w}_0+\gamma ^{\w }_3$. 
Therefore, 
we conclude that $\gamma ^{\w }_i<\gamma ^{\w }_2$ for $i=0,3$. 
\end{proof}

We define $\psi \in \Aut (\kx /k)$ by 
$$
\psi (x_1)=z_1+h_0z_3^{-1}-h_1,\quad 
\psi (x_2)=z_2+h_2,\quad 
\psi (x_3)=z_3. 
$$
Then, $\phi ^{-1}\circ \psi $ 
belongs to $J(k[x_3];x_2,x_1)$. 
Hence, 
$\phi $ is tame if and only if so is $\psi $. 
We note that 
$\degw \psi (x_i)=\gamma _i^{\w }$ for $i=2,3$.

\begin{lem}\label{lem:awc}
Assume that 
$\gamma ^{\w }_0<\gamma ^{\w }_1$, 
$\gamma ^{\w }_2>\gamma ^{\w }_3$ and 
\begin{equation}\label{eq:awc}
2\gamma ^{\w }_1>3\gamma ^{\w }_2+2\gamma ^{\w }_3,\quad 
\gamma ^{\w }_1\geq \gamma ^{\w }_2+2\gamma ^{\w }_3
\end{equation}
for some $\w \in (\Gammap )^3$. 
Then, 
$\phi $ is wild. 
\end{lem}
\begin{proof}
It suffices to prove that $\psi $ is wild. 
We check that $\psi $ satisfies the conditions 
(a) through (e) of Proposition~\ref{lem:SUcriterion}. 
Since 
$\degw \psi (x_i)=\gamma _i^{\w }$ for $i=2,3$, 
and $\gamma ^{\w }_2>\gamma ^{\w }_3$ by assumption, 
we see that $\psi $ satisfies (d). 
By (\ref{eq:f_1f_2}), we have 
$$
\psi (x_1)
=\bigl(\phi (f_{\theta })-\theta (z_2)\bigr)z_3^{-1}
+h_0z_3^{-1}-h_1
=\bigl(
(\phi (f_{\theta })+h_0)-(\theta (z_2)+h_1z_3)
\bigr)z_3^{-1}. 
$$
Since $\degw \bigl(\phi (f_{\theta })+h_0\bigr)
=\gamma ^{\w }_0
<\gamma ^{\w }_1
=\degw \bigl(\theta (z_2)+h_1z_3\bigr)$ 
by assumption, 
it follows that 
\begin{equation}\label{eq:awcpf1}
\psi (x_1)^{\w }
=-\bigl(\theta (z_2)+h_1z_3\bigr)^{\w }(z_3^{\w })^{-1}. 
\end{equation}
Hence, 
$\psi (x_1)^{\w }$ 
belongs to the field of fractions of $k[z_2,z_3]^{\w }$. 
Since $\psi (x_2)$ and $\psi (x_3)$ belong to $k[z_2,z_3]$, 
we know that $\psi (x_2)^{\w }$ 
and $\psi (x_3)^{\w }$ also belong to $k[z_2,z_3]^{\w }$. 
Thus, 
$\psi (x_1)^{\w }$, $\psi(x_2)^{\w }$ and $\psi (x_3)^{\w }$ 
are algebraically dependent over $k$, 
since the field of fractions of $k[z_2,z_3]^{\w }$ 
has transcendence degree at most two over $k$. 
Therefore, 
$\psi $ satisfies (a). 
Since 
$\degw \psi (x_1)=\gamma ^{\w }_1-\gamma ^{\w }_3$ 
by (\ref{eq:awcpf1}), 
we have 
$$
2\degw \psi (x_1)=2\gamma ^{\w }_1-2\gamma ^{\w }_3
>3\gamma ^{\w }_2=3\degw \psi (x_2)
$$
by the first part of (\ref{eq:awc}). 
This proves (e). 
Similarly, 
we have 
$$
\degw \psi (x_1)=\gamma ^{\w }_1
-\gamma ^{\w }_3\geq \gamma ^{\w }_2+\gamma ^{\w }_3
=\degw \psi (x_2)+\degw \psi (x_3)
$$
by the second part of (\ref{eq:awc}). 
This proves (c). 
Since 
$\psi (x_2)^{\w }=(z_2+h_2)^{\w }$ 
does not belong to 
$k[\psi (x_3)^{\w }]=k[z_3^{\w }]$ 
as mentioned, 
$\psi $ satisfies the second part of (b). 
We prove the first part of (b) by contradiction. 
Suppose that 
$\psi (x_1)^{\w }$ belongs to 
$k[\psi (x_2),\psi (x_3)]^{\w }$. 
Then, 
there exists $h\in k[\psi (x_2),\psi (x_3)]$ 
such that $h^{\w }=\psi (x_1)^{\w }$. 
By (\ref{eq:awcpf1}), 
it follows that 
$$
(hz_3)^{\w }+(\theta (z_2)+h_1z_3)^{\w }
=h^{\w }z_3^{\w }-\psi (x_1)^{\w }z_3^{\w }=0. 
$$ 
Hence, the $\w $-degree of 
$$
\theta (z_2)+(h_1+h)z_3
=hz_3+\bigl(\theta (z_2)+h_1z_3\bigr)
$$ 
is less than 
$\degw (\theta (z_2)+h_1z_3)=\gamma ^{\w }_1$. 
By the minimality of $\gamma _1^{\w }$, 
this implies that 
$h_1+h$ does not belong to $k[z_2,z_3]$. 
However, 
$h_1$ is an element of 
$k[z_2,z_3]$, 
and $h$ is an element of 
$k[\psi (x_2),\psi (x_3)]=
k[z_2+h_2,z_3]=k[z_2,z_3]$. 
Hence, 
$h_1+h$ belongs to $k[z_2,z_3]$, 
a contradiction. 
This proves the first part of (b). 
Thus, 
$\psi $ satisfies (a) through (e) 
of Proposition~\ref{lem:SUcriterion}. 
Therefore, 
we conclude that $\psi $ is wild. 
Consequently, 
$\phi $ is wild. 
\end{proof}

Since $h_2$ belongs to $z_3k[x_3]$, 
we see that 
$\theta \bigl(\psi (x_2)\bigr)=\theta (z_2+h_2)$ 
has the form 
$\theta (z_2)+hz_3$ for some $h\in k[z_2,z_3]$. 
Similarly, 
$\theta (z_2)=\theta \bigl(\psi (x_2)-h_2\bigr)$ 
has the form 
$\theta \bigl(\psi (x_2)\bigr)+h'z_3$ 
for some $h'\in k[\psi (x_2),z_3]$. 
Since $k[\psi (x_2),z_3]=k[z_2,z_3]$, 
it follows that 
$$
\{\theta \bigl(\psi (x_2)\bigr)+hz_3\mid 
h\in k[\psi (x_2),z_3]\} 
=\{\theta (z_2)+hz_3\mid 
h\in k[z_2,z_3]\} . 
$$
Hence, we have 
\begin{equation}\label{eq:gamma1}
\gamma _1^{\w }=
\min \{ \degw \bigl(\theta (\psi (x_2))+hz_3\bigr)\mid 
h\in k[\psi (x_2),z_3]\} . 
\end{equation}
In the notation of Lemma~\ref{lem:91113}, 
we may write 
$\gamma _1^{\w }=\eta (\theta ;z_3,\psi (x_2))$. 
We note that the conditions 
(2) and (3) before Lemma~\ref{lem:91113} 
are fulfilled for $f=z_3$ and $g=\psi (x_2)$, 
since $\psi (x_2)^{\w }=(z_2+h_2)^{\w }$ 
does not belong to $k[z_3^{\w }]$, 
and $\psi (x_2)$ and $z_3=\psi (x_3)$ 
are algebraically independent over $k$. 
Since $\degw \psi (x_2)=\gamma _2^{\w }$ 
and $\degw z_3=\gamma _3^{\w }$, 
the condition (1) is equivalent to 
$\gamma ^{\w }_2>\gamma ^{\w }_3$.

\begin{lem}\label{lem:awc1}
Assume that $\gamma ^{\w }_2>\gamma ^{\w }_3$. 
Then, 
$\phi $ satisfies $(\ref{eq:awc})$ 
if one of the following conditions holds$:$

\noindent{\rm (i)} $d\geq 3$ and 
$\psi (x_2)^{\w }$ 
and $z_3^{\w }$ are algebraically independent over $k$. 

\noindent{\rm (ii)}
$d\geq 3$ and 
$\degw dz_2\wedge dz_3>(d-1)\degw z_3$. 

\noindent{\rm (iii)}
$d\geq 9$ and $d\neq 10,12$. 
\end{lem}
\begin{proof}
First, assume that (i) is satisfied. 
Then, 
we have 
$k[\psi (x_2),z_3]^{\w }=k[\psi (x_2)^{\w },z_3^{\w }]$, 
since $\psi (x_2)^{\w }$ and $z_3^{\w }$ 
are algebraically independent over $k$. 
Take any $h\in k[\psi (x_2),z_3]$. 
Then, 
$h^{\w }$ belongs to $k[\psi (x_2)^{\w },z_3^{\w }]$. 
Hence, 
we know that 
$\theta (\psi (x_2))^{\w }\approx 
(\psi (x_2)^{\w })^d\not\approx h^{\w }z_3^{\w }=(hz_3)^{\w }$. 
This implies that 
$$
\degw \bigl(\theta (\psi (x_2))+hz_3\bigr) 
=\max \{ \degw \theta (\psi (x_2)),\degw hz_3\} 
\geq \degw \theta (\psi (x_2))
=d\gamma _2^{\w }. 
$$ 
Thus, we obtain
$\gamma _1^{\w }\geq d\gamma _2^{\w }$ 
in view of (\ref{eq:gamma1}). 
Since $d\geq 3$ and $\gamma ^{\w }_2>\gamma ^{\w }_3$ 
by assumption, 
it follows that 
$$
2\gamma _1^{\w }\geq 2d\gamma ^{\w }_2
>5\gamma ^{\w }_2>3\gamma ^{\w }_2+2\gamma ^{\w }_3,\quad 
\gamma _1^{\w }\geq d\gamma ^{\w }_2\geq 3\gamma ^{\w }_2
>\gamma ^{\w }_2+2\gamma ^{\w }_3. 
$$ 
Therefore, 
$\phi $ satisfies $(\ref{eq:awc})$.

Next, 
assume that (ii) or (iii) is satisfied. 
Since 
$d\psi (x_2)\wedge dz_3=dz_2\wedge dz_3$, 
we see that (ii) is equivalent to 
(ii) of Lemma~\ref{lem:91113}. 
Clearly, 
(iii) is the same as (iii) of Lemma~\ref{lem:91113}. 
Since $\gamma ^{\w }_2>\gamma ^{\w }_3$ by assumption, 
the conditions (1), (2) and (3) listed before Lemma~\ref{lem:91113} 
are fulfilled for $f=z_3$ and $g=\psi (x_2)$ as mentioned. 
Hence, we know by Lemma~\ref{lem:91113} that 
$\gamma _1^{\w }=\eta (\theta ;z_3,\psi (x_2))$ 
is greater than 
$\eta _1=\gamma ^{\w }_3+(3/2)\gamma ^{\w }_2$ 
and 
$\eta _2=2\gamma ^{\w }_3+\gamma ^{\w }_2$. 
Therefore, we get (\ref{eq:awc}). 
\end{proof}

Now, 
we prove Theorem~\ref{thm:tawild} (i). 
Assume that $d\geq 9$ and $d\neq 10,12$. 
Suppose to the contrary that 
$\degw \phi (y_2)\leq \degw \phi (x_3)$ 
for some $\phi \in G_{\theta }$. 
Since $\degw \phi (x_3)=\degw z_3=\gamma _3^{\w }$, 
and $\gamma _0^{\w }>0$, 
we have 
$\degw \phi (y_2)\leq \gamma _3^{\w }
<\gamma _0^{\w }+\gamma _3^{\w }$. 
By Lemma~\ref{lem:taiguu}, 
it follows that 
$\gamma _3^{\w }<\gamma _2^{\w }$ 
and $\gamma _0^{\w }<\gamma _2^{\w }$. 
Hence, 
$\phi $ satisfies the assumption of Lemma~\ref{lem:awc1}. 
Since $d\geq 9$ and $d\neq 10,12$, 
we know by (iii) 
that $\phi $ satisfies (\ref{eq:awc}). 
Since $\gamma ^{\w }_0<\gamma ^{\w }_2$, 
(\ref{eq:awc}) implies that $\gamma ^{\w }_0<\gamma ^{\w }_1$. 
Thus, 
we conclude from Lemma~\ref{lem:awc} that $\phi $ is wild, 
a contradiction. 
Therefore, 
we have $\degw \phi (y_2)>\degw \phi (x_3)$ for each $\phi \in G_{\theta }$. 
This completes the proof of Theorem~\ref{thm:tawild} (i).

\section{Proof (II)} 
\setcounter{equation}{0}
\label{sect:wild3pf2}

The goal of this section 
is to prove the following proposition, 
which is a key to the proof of 
Theorem~\ref{thm:tawild} (ii).

\begin{prop}\label{prop:tawild key}
Assume that $d\geq 9$. 
Let $\phi \in G_{\theta }$ 
be such that $\degv \phi (y_2)>\degv \phi (x_3)$. 
Then, the following assertions hold$:$ 	

\noindent{\rm (i)} $\degv \phi (y_2)=3\e _1$. 

\noindent{\rm (ii)} $\phi (x_3)=z_3=\alpha _3x_3+g_3$ for some 
$\alpha _3\in k^{\times }$ and $g_3\in k[x_2]$. 

\noindent{\rm (iii)} $\gamma _2^{\vv }<\e _1$. 
\end{prop}

We begin with the following lemma, 
which is proved by a technique similar to 
Hadas--Makar-Limanov~\cite[Corollary 3.3]{Hadas}.

\begin{lem}\label{lem:newton}
Let $\tau $ be an element of $\Aut (\kx /k)$, 
and $\w $ an element of $\Gamma^n$, 
where $n\geq 3$ may be arbitrary. 
Assume that there exist $i_1,i_2\in \{ 1,\ldots ,n\} $ 
such that $\tau (x_{i_1})^{\w }$ 
is divisible by $x_i$ for each $i\neq i_2$. 
Then, 
$\tau (x_j)^{\w }$ 
belongs to $k[\x \sm \{ x_{i_2}\} ]$ 
for $j=1,\ldots ,n$. 
\end{lem}
\begin{proof}
For each $0\neq D\in \Der _k\kx $, 
we define 
$$
\gamma _D=\max \{ \degw D(x_i)x_i^{-1}\mid 
i=1,\ldots ,n\} , 
$$
and $D^{\w }\in \Der _k\kx $ by 
\begin{equation}\label{eq:D^w}
D^{\w }(x_i)=\left\{ 
\begin{array}{ccl}
D(x_i)^{\w } &\text{ if }& \degw D(x_i)x_i^{-1}=\gamma _D \\
0 & \text{ if } & \degw D(x_i)x_i^{-1}<\gamma _D
\end{array}
\right. 
\end{equation}
for $i=1,\ldots ,n$. 
We show that $D^{\w }(f^{\w })\neq 0$ implies 
$D(f)^{\w }=D^{\w }(f^{\w })$ 
for each $f\in \kx \sm \zs $. 
For each $h\in \kx $ and $i=1,\ldots ,n$, 
we denote $h_{x_i}=\partial h/\partial x_i$ 
for simplicity. 
Then, 
we have 
$\degw f_{x_i}\leq \degw fx_i^{-1}$, 
and $\degw f_{x_i}=\degw fx_i^{-1}$ 
if and only if $f^{\w }$ 
does not belong to $k[\x \sm \{ x_i\} ]$. 
Hence, we know that 
\begin{equation}\label{eq:f^w_x}
(f^{\w })_{x_i}
=\left\{ 
\begin{array}{ccl}
(f_{x_i})^{\w } 
 &\text{ if }& \degw f_{x_i}=\degw fx_i^{-1} \\
0  &\text{ if }& \degw f_{x_i}<\degw fx_i^{-1}
\end{array}
\right. 
\end{equation}
for $i=1,\ldots ,n$. 
Let $I$ be the set of $i\in \{ 1,\ldots ,n\} $ such that 
$\degw D(x_i)x_i^{-1}=\gamma _D$ and 
$\degw f_{x_i}=\degw fx_i^{-1}$. 
Then, we have 
\begin{equation}\label{eq:HMLM}
D^{\w }(f^{\w })
=\sum _{i=1}^nD^{\w }(x_i)(f^{\w })_{x_i}
=\sum _{i\in I}D(x_i)^{\w }(f_{x_i})^{\w }
=\sum _{i\in I}\bigl(D(x_i)f_{x_i}\bigr)^{\w }. 
\end{equation}
Note that 
$$
\degw 
D(x_i)f_{x_i}
\leq (\gamma _D+\degw x_i)+(\degw f-\degw x_i)
=\gamma _D+\degw f 
$$
for $i=1,\ldots ,n$, 
in which the equality holds 
if and only if $i$ belongs to $I$. 
Since $D^{\w }(f^{\w })\neq 0$ by assumption, 
this implies that 
$$
\sum _{i\in I}\bigl(D(x_i)f_{x_i}\bigr)^{\w }
=\left( 
\sum _{i=1}^nD(x_i)f_{x_i}
\right)^{\w } =D(f)^{\w }. 
$$
Thus, 
we get 
$D(f)^{\w }=D^{\w }(f^{\w })$. 
Using this, 
we can prove by induction on $l$ 
that $(D^{\w })^l(f^{\w })\neq 0$ implies 
$D^l(f)^{\w }=(D^{\w })^l(f^{\w })$ 
for each $l\in \N $ and $f\in \kx \sm \zs $. 
Therefore, 
if $D$ is locally nilpotent, 
then $D^{\w }$ is also locally nilpotent.

Now, 
take any $i_0\in \{ 1,\ldots ,n\} \sm \{ i_1\} $, 
and put 
$D=\tau \circ (\partial /\partial x_{i_0})\circ \tau ^{-1}$. 
Then, $D$ is locally nilpotent. 
Hence, $D^{\w }$ is also locally nilpotent. 
Thus, 
$\ker D^{\w }$ is factorially closed in $\kx $, 
i.e., if $fg$ belongs to $\ker D^{\w }$ 
for $f,g\in \kx \sm \zs $, then 
$f$ and $g$ belong to $\ker D^{\w }$ 
(cf.~\cite[Lemma 1.3.1]{Miyanishi}). 
Since $i_1\neq i_0$, 
we have $D\bigl(\tau (x_{i_1})\bigr)=0$. 
This implies that $D^{\w }(\tau (x_{i_1})^{\w })=0$, 
for otherwise 
$D(\tau (x_{i_1}))^{\w }=D^{\w }(\tau (x_{i_1})^{\w })\neq 0$, 
a contradiction. 
Hence, 
$\tau (x_{i_1})^{\w }$ belongs to $\ker D^{\w }$. 
Since 
$x_i$ is a factor of $\tau (x_{i_1})^{\w }$ 
for each $i\neq i_2$ by assumption, 
it follows that 
$x_i$ belongs to $\ker D^{\w }$ for each $i\neq i_2$ 
by the factorially closedness of $\ker D^{\w }$. 
Hence, 
$k[\x \sm \{ x_{i_2}\} ]$ is contained in $\ker D^{\w }$. 
Since $D$ is nonzero, 
so is $D^{\w }$. 
Hence, 
the transcendence degree of $\ker D^{\w }$ over $k$ 
is less than $n$. 
Thus, 
we conclude that 
$\ker D^{\w }=k[\x \sm \{ x_{i_2}\} ]$. 
If $j\neq i_0$, 
then we have $D\bigl( \tau (x_j)\bigr) =0$. 
This implies that 
$D^{\w }\bigl( \tau (x_j)^{\w }\bigr) =0$ as mentioned. 
Hence, 
$\tau (x_j)^{\w }$ belongs to 
$\ker D^{\w }=k[\x \sm \{ x_{i_2}\} ]$. 
Since $n\geq 3$ by assumption, 
we may take 
$i_0'\in \{ 1,\ldots ,n\} \sm \{ i_0,i_1\} $. 
Then, 
by a similar argument with $i_0$ replaced by $i_0'$, 
we can verify that 
$\tau (x_{j})^{\w }$ belongs to $k[\x \sm \{ x_{i_2}\} ]$ 
for each $j\neq i_0'$. 
Thus, 
$\tau (x_{i_0})^{\w }$ also belongs to $k[\x \sm \{ x_{i_2}\} ]$. 
Therefore, 
$\tau (x_j)^{\w }$ belongs to $k[\x \sm \{ x_{i_2}\} ]$ 
for $j=1,\ldots ,n$. 
\end{proof}

Recall that $f_{\theta }=x_1x_3+\theta (x_2)$, 
$y_2=x_2+f_{\theta }x_3$ 
and $y_3=x_3$. 
Since $\vv =(\e _1,\e _2,\e _1)$, 
we have 
\begin{equation}\label{eq:y highest}
f_{\theta }^{\vv }=x_1x_3,\quad 
y_2^{\vv }
=x_1x_3^2,\quad 
y_3^{\vv }=x_3
\end{equation}
by the definition of $\Lambda $. 
Hence, 
we see from (\ref{eq:y_11}) that 
\begin{equation}\label{eq:y highesty_1}
y_1^{\vv }
=-\theta (y_2)^{\vv }x_3^{-1}
\approx x_1^dx_3^{2d-1}. 
\end{equation}
To prove Proposition~\ref{prop:tawild key}, 
assume that $d\geq 9$, 
and take $\phi \in G_{\theta }$ 
such that $\degv \phi (y_2)$ 
is greater than $\degv \phi (x_3)=\degv z_3$. 
Then, 
we have $\phi (y_1)=y_1$. 
Hence, 
$\phi (y_1)^{\vv }=y_1^{\vv }$ 
is divisible by $x_1$ and $x_3$ by (\ref{eq:y highesty_1}). 
Since $\phi (y_1)=(\phi \circ \sigma _{\theta })(x_1)$, 
it follows that 
$\phi (y_2)^{\vv }=(\phi \circ \sigma _{\theta })(x_2)^{\vv }$ 
and $z_3^{\vv }=(\phi \circ \sigma _{\theta })(x_3)^{\vv }$ 
belong to $k[x_1,x_3]$ by Lemma~\ref{lem:newton}. 
Hence, we know that 
$\degv \phi (y_2)=a\e _1$ and 
$\degv z_3=b\e _1$ for some $a,b\in \N $. 
Since $\degv \phi (y_2)>\degv z_3$ by assumption, 
we get $a\geq b+1$.

\begin{lem}\label{lem:(a,b)=(3,1)}
If $\degv y_1z_3=d\degv \phi (y_2)$, 
then we have 
$$
(a,b)=(3,1),\quad \degv \phi (y_2)=3\degv z_3
$$
and $z_3^{\vv}$ is a linear form 
in $x_1$ and $x_3$ over $k$. 
\end{lem}
\begin{proof}
By (\ref{eq:y highesty_1}), 
we see that $\degv y_1=(3d-1)\e _1$. 
Since $\degv \phi (y_2)=a\e _1$, $\degv z_3=b\e _1$, 
and $\degv y_1z_3=d\degv \phi (y_2)$ by assumption, 
we know that $3d-1+b=da$. 
Because $1\leq b\leq a-1$, 
it follows that 
$1\leq da-3d+1\leq a-1$. 
Hence, we have 
$$
3\leq a\leq \frac{3d-2}{d-1}=3+\frac{1}{d-1}<4, 
$$
since $d\geq 9$. 
Thus, we get $a=3$. 
Since $3d-1+b=da=3d$, 
it follows that $b=1$. 
Therefore, 
we have $\degv \phi (y_2)=3\e _1=3\degv z_3$. 
Since $\degv z_3=\e _1$ and $\vv =(\e _1,\e _2,\e _1)$, 
we know that $z_3^{\vv}$ is a linear form 
in $x_1$ and $x_3$ over $k$. 
\end{proof}

Using this lemma, 
we can prove the following proposition.

\begin{prop}\label{prop:awcpf1}
If $(y_1z_3)^{\vv }\approx (\phi (y_2)^{\vv })^d$, 
then we have $\phi (y_2)^{\vv }\approx x_1x_3^2$ 
and $z_3^{\vv }\approx x_3$. 
\end{prop}
\begin{proof}
Since $(y_1z_3)^{\vv }\approx (\phi (y_2)^{\vv })^d$, 
we have $\degv y_1z_3=d\degv \phi (y_2)$. 
Hence, 
we may write 
$z_3^{\vv }=\alpha x_1+\beta x_3$ 
by Lemma~\ref{lem:(a,b)=(3,1)}, 
where $\alpha ,\beta \in k$ 
are such that $\alpha \neq 0$ or $\beta \neq 0$. 
In view of (\ref{eq:y highesty_1}), 
we have 
$$
(\phi (y_2)^{\vv })^d\approx (y_1z_3)^{\vv }
\approx x_1^dx_3^{2d-1}(\alpha x_1+\beta x_3). 
$$
Since $d\geq 9$, 
this implies that $\alpha =0$, $\beta \neq 0$, 
$\phi (y_2)^{\vv }\approx x_1x_3^2$ 
and $z_3^{\vv }=\beta x_3\approx x_3$. 
\end{proof}

Since $\sigma _{\theta }(f_{\theta })=f_{\theta }$, 
$\sigma _{\theta }(x_3)=x_3$ and $\phi (y_1)=y_1$, 
we have 
\begin{equation}\label{eq:core:tawild}
\phi (f_{\theta })=\phi (\sigma _{\theta }(f_{\theta }))
=\phi (y_1x_3+\theta (y_2))
=y_1z_3+\theta (\phi (y_2)). 
\end{equation}
The following proposition forms the core of 
the proof of Proposition~\ref{prop:tawild key}.

\begin{prop}\label{prop:awcpf2}
If $(y_1z_3)^{\vv }\not\approx (\phi (y_2)^{\vv })^d$, 
then the following inequalities hold$:$

\noindent{\rm (i)} 
$\degv d\phi (f_{\theta })\wedge dz_3>(d-1)\degv \phi (y_2)$.

\noindent{\rm (ii)} 
$\degv dz_2\wedge dz_3>(d-1)\degv \phi (y_2)
+\degv z_3$.

\noindent{\rm (iii)} 
$\gamma _0^{\vv }+\gamma ^{\vv }_3>(d-1)\degv \phi (y_2)$. 
\end{prop}
\begin{proof}
(i) 
First, 
assume that $\degv y_1z_3=d\degv \phi (y_2)$. 
Then, 
we have 
\begin{equation}\label{eq:dFdz_31}
\phi (f_{\theta })^{\vv }
=\bigl(y_1z_3+\theta (\phi (y_2))\bigr)^{\vv }
=(y_1z_3)^{\vv }+c(\phi (y_2)^{\vv })^d 
\end{equation}
by (\ref{eq:core:tawild}), 
since 
$(y_1z_3)^{\vv }\not\approx (\phi (y_2)^{\vv })^d$ 
by assumption. 
Hence, we get 
\begin{equation}\label{eq:pf1:awcpf2}
\degv \phi (f_{\theta })=d\degv \phi (y_2). 
\end{equation}
Now, 
suppose that (i) is false. 
Then, we have 
\begin{align*}
\degv d\phi (f_{\theta })\wedge dz_3
\leq (d-1)\degv \phi (y_2)
<\degv \phi (f_{\theta })+\degv z_3 
\end{align*}
by (\ref{eq:pf1:awcpf2}). 
This implies that 
$\phi (f_{\theta })^{\vv }$ and $z_3^{\vv }$ 
are algebraically dependent over $k$. 
Since $\phi (f_{\theta })^{\vv }$ and $z_3^{\vv }$ 
are $\vv $-homogeneous polynomials of positive $\vv $-degrees, 
it follows that 
$(\phi (f_{\theta })^{\vv })^l=c_1(z_3^{\vv })^{m}$ 
for some $l,m\in \N $ with $\gcd (l,m)=1$ 
and $c_1\in k^{\times }$. 
By (\ref{eq:pf1:awcpf2}) and Lemma~\ref{lem:(a,b)=(3,1)}, 
we know that $\degv \phi (f_{\theta })=3d\degv z_3$. 
Hence, we get $(l,m)=(1,3d)$. 
By (\ref{eq:dFdz_31}), 
it follows that 
$c(\phi (y_2)^{\vv })^d=c_1(z_3^{\vv })^{3d}-y_1^{\vv }z_3^{\vv }$. 
From this, 
we see that 
$(\phi (y_2)^{\vv })^d$ is divisible by $z_3^{\vv }$. 
By Lemma~\ref{lem:(a,b)=(3,1)}, 
$z_3^{\vv }$ is a linear form, and so irreducible. 
Hence, $z_3^{\vv }$ divides $\phi (y_2)^{\vv }$. 
Since $d\geq 9$ by assumption, 
it follows that $z_3^{\vv }$ divides $y_1^{\vv }$. 
By (\ref{eq:y highesty_1}), 
we see that $z_3^{\vv }=c_2x_i$ 
for some $c_2\in k^{\times }$ and $i\in \{ 1,3\} $. 
Thus, 
$c(\phi (y_2)^{\vv })^d
=c_1(c_2x_i)^{3d}-c_2x_iy_1^{\vv }$ 
is a binomial. 
Since $k$ is of characteristic zero, 
no binomial is a proper power of a polynomial. 
Therefore, we get $d=1$, a contradiction.

Next, 
assume that $\degv y_1z_3<d\degv \phi (y_2)$. 
Then, we have $3d-1+b<da$. 
This yields that $a>3+(b-1)/d\geq 3$. 
Hence, we get $a\geq 4$. 
First, we prove that 
\begin{equation}\label{eq:pf:case2:awcpf2}
\degv y_1z_3^2\leq (d-1)\degv \phi (y_2)
\end{equation}
by contradiction. 
Suppose that (\ref{eq:pf:case2:awcpf2}) is false. 
Then, 
we have 
$$
(d-1)a\leq (3d-1+2b)-1. 
$$
Since $b\leq a-1$, 
it follows that $(d-1)a\leq 3d+2a-4$. 
Since $d\geq 9$ by assumption, 
this yields that 
$$
a\leq \frac{3d-4}{d-3}=3+\frac{5}{d-3}<4, 
$$
a contradiction. 
Therefore, 
(\ref{eq:pf:case2:awcpf2}) is true. 
By (\ref{eq:core:tawild}), 
we get 
\begin{equation}\label{eq:phi(F)}
d\phi (f_{\theta })\wedge dz_3
=d(y_1z_3)\wedge dz_3+
d\theta (\phi (y_2))\wedge dz_3. 
\end{equation}
Since $\phi (y_2)=\phi (\sigma _{\theta }(x_2))$ 
and $z_3=\phi (\sigma _{\theta }(x_3))$ 
are algebraically independent over $k$, 
we know that $d\phi (y_2)\wedge dz_3$ is nonzero. 
Hence, the $\vv $-degree of 
$d\theta (\phi (y_2))\wedge dz_3 =
\theta '(\phi (y_2))d\phi (y_2)\wedge dz_3$ 
is greater than 
$\degv \theta '(\phi (y_2))=(d-1)\degv \phi (y_2)$. 
Thus, 
it follows from (\ref{eq:pf:case2:awcpf2}) that
$$
\degv d\theta (\phi (y_2))\wedge dz_3 
>(d-1)\degv \phi (y_2)\geq \degv y_1z_3^2
\geq \degv d(y_1z_3)\wedge dz_3. 
$$
In view of (\ref{eq:phi(F)}), 
this implies that 
$\degv d\phi (f_{\theta })\wedge dz_3>(d-1)\degv \phi (y_2)$.

Finally, assume that 
$\degv y_1z_3>d\degv \phi (y_2)$. 
Then, we have 
\begin{align*}
&\degv d\theta (\phi (y_2))\wedge dz_3
\leq \degv \theta (\phi (y_2))+\degv z_3
=d\degv \phi (y_2)+\degv z_3 \\
&\quad <\degv y_1z_3+\degv z_3=\degv y_1z_3^2. 
\end{align*}
We show that 
\begin{equation}\label{eq:phi (F)!}
\degv y_1z_3^2=\degv d(y_1z_3)\wedge dz_3. 
\end{equation}
Then, 
it follows that 
$\degv d\theta (\phi (y_2))\wedge dz_3
<\degv d(y_1z_3)\wedge dz_3$. 
In view of (\ref{eq:phi(F)}), 
this implies that 
$\degv d\phi (f_{\theta })\wedge dz_3
=\degv d(y_1z_3)\wedge dz_3$. 
By (\ref{eq:phi (F)!}), 
this is equal to $\degv y_1z_3^2$, 
and is greater than $(d-1)\degv \phi (y_2)$ 
by the assumption that $\degv y_1z_3>d\degv \phi (y_2)$. 
Therefore, 
we get 
$\degv d\phi (f_{\theta })\wedge dz_3>(d-1)\degv \phi (y_2)$.

Since $d(y_1z_3)\wedge dz_3=z_3dy_1\wedge dz_3$, 
it suffices to verify that 
$y_1^{\vv }$ and $z_3^{\vv }$ 
are algebraically independent over $k$. 
Suppose to the contrary that $y_1^{\vv }$ and $z_3^{\vv }$ 
are algebraically dependent over $k$. 
Then, 
we have $(y_1^{\vv })^q\approx (z_3^{\vv })^r$ 
for some $q,r\in \N $ with $\gcd (q,r)=1$, 
since $y_1^{\vv }$ and $z_3^{\vv }$ 
are $\vv $-homogeneous polynomials of positive 
$\vv $-degrees. 
Because $\gcd (d,2d-1)=1$, 
we see that 
$y_1^{\vv }\approx x_1^dx_2^{2d-1}$ 
is not a proper power of a polynomial. 
Hence, 
we know that $r=1$. 
Thus, we get $q\degv y_1=\degv z_3$. 
Since $\degv y_1z_3>d\degv \phi (y_2)$ by assumption, 
and $\degv \phi (y_2)>\degv z_3$ by the choice of $\phi $, 
it follows that 
$$
2\degv z_3=q\degv y_1+\degv z_3
\geq \degv y_1z_3>d\degv \phi (y_2)>d\degv z_3. 
$$
This contradicts that $d\geq 9$. 
Therefore, 
$y_1^{\vv }$ and $z_3^{\vv }$ 
are algebraically independent over $k$, 
proving that 
$\degv d\phi (f_{\theta })\wedge dz_3>(d-1)\degv \phi (y_2)$.

(ii) 
By (i), 
we know that 
\begin{align*}
&\degv z_3d\phi (f_{\theta })\wedge dz_3
=\degv d\phi (f_{\theta })\wedge dz_3+\degv z_3 \\
&\quad >(d-1)\degv \phi (y_2)+\degv z_3 
>\degv d\phi (y_2)\wedge dz_3, 
\end{align*}
since $d\geq 9$. 
Because $z_2=\phi (y_2)-\phi (f_{\theta })z_3$ 
by (\ref{eq:f_1f_2}), 
we have 
$$
dz_2\wedge dz_3=d\phi(y_2)\wedge dz_3
-z_3d\phi (f_{\theta })\wedge dz_3. 
$$
Therefore, we get 
$$
\degv dz_2\wedge dz_3=
\degv z_3d\phi (f_{\theta })\wedge dz_3
>(d-1)\degv \phi (y_2)+\degv z_3. 
$$

(iii) 
Take $h_0\in k[z_3]$ such that 
$\gamma _0^{\vv }=\degv (\phi(f_{\theta })+h_0)$. 
Then, we have 
$$
\gamma _0^{\vv }+\gamma _3^{\vv }
=\degv (\phi(f_{\theta })+h_0)+\degv z_3
\geq \degv d(\phi(f_{\theta })+h_0)\wedge dz_3
=\degv d\phi(f_{\theta })\wedge dz_3. 
$$
Therefore, 
we get 
$\gamma _0^{\vv }+\gamma _3^{\vv }>(d-1)\degv \phi (y_2)$ 
by (i). 
\end{proof}

Now, 
let us complete the proof of 
Proposition~\ref{prop:tawild key}. 
First, 
we prove that 
$(y_1z_3)^{\vv }\approx (\phi (y_2)^{\vv })^d$. 
Suppose to the contrary that 
$(y_1z_3)^{\vv }\not\approx (\phi (y_2)^{\vv })^d$. 
Then, 
we have the three inequalities 
(i), (ii) and (iii) of Proposition~\ref{prop:awcpf2}. 
We deduce that $\phi $ is wild by means of Lemma~\ref{lem:awc}. 
By the inequality (iii), 
we have 
$$
\gamma _0^{\vv }+\gamma _3^{\vv }
>(d-1)\degv \phi (y_2)
>\degv \phi (y_2), 
$$
since $d\geq 9$. 
By Lemma~\ref{lem:taiguu}, 
it follows that 
$\gamma _3^{\w }<\gamma _2^{\w }$ 
and $\gamma _0^{\w }<\gamma _2^{\w }$. 
Hence, 
$\phi $ satisfies the assumption of Lemma~\ref{lem:awc1}. 
Since $\degv \phi (y_2)>\degv z_3$ 
by the choice of $\phi $, 
the inequality (ii) yields that 
$$
\degv dz_2\wedge dz_3
>(d-1)\degv \phi (y_2)+\degv z_3
>(d-1)\degv z_3. 
$$ 
Since $d\geq 9$, 
this implies that $\phi $ satisfies (\ref{eq:awc}) 
because of Lemma~\ref{lem:awc1} (ii). 
The second part of (\ref{eq:awc}) implies 
$\gamma _1^{\vv }>\gamma _2^{\vv }$. 
Since $\gamma _0^{\vv }<\gamma _2^{\vv }$, 
it follows that $\gamma _0^{\vv }<\gamma _1^{\vv }$. 
Thus, 
we conclude from 
Lemma~\ref{lem:awc} that $\phi $ is wild, 
a contradiction. 
This proves that 
$(y_1z_3)^{\vv }\approx (\phi (y_2)^{\vv })^d$. 
Therefore, 
we get $\phi (y_2)^{\vv }\approx x_1x_3^2$ 
and $z_3^{\vv }\approx x_3$ 
thanks to Proposition~\ref{prop:awcpf1}. 
Hence, 
we have 
$\degv \phi (y_2)=3\e _1$, 
proving Proposition~\ref{prop:tawild key} (i). 
Since $\vv =(\e _1,\e _2,\e _1)$, 
we know that 
$z_3=\alpha _3x_3+g_3$ 
for some $\alpha _3\in k^{\times }$ and $g_3\in k[x_2]$. 
This proves Proposition~\ref{prop:tawild key} (ii).

Next, 
we prove Proposition~\ref{prop:tawild key} (iii). 
First, 
we show that 
$\gamma _2^{\vv }\leq \e _1$. 
Suppose to the contrary that 
$\gamma _2^{\vv }>\e _1$. 
We deduce that $\phi $ is wild 
by means of Lemma~\ref{lem:awc}. 
Since $\gamma _3^{\vv }=\degv z_3=\e _1$, 
we have 
$\gamma _2^{\vv }>\gamma _3^{\vv }$. 
Hence, 
$\phi $ satisfies the assumption of Lemma~\ref{lem:awc1}. 
Define $\psi \in \Aut (\kx /k)$ as before Lemma~\ref{lem:awc}. 
Then, 
$\psi (x_2)^{\vv }=(z_2+h_2)^{\vv }$ 
does not belong to $k[z_3^{\vv }]=k[x_3]$ 
by the minimality of $\gamma _2^{\vv }$. 
Hence, 
$\psi (x_2)^{\vv }$ and $z_3^{\vv }\approx x_3$ 
are algebraically independent over $k$. 
Since $d\geq 9$, 
we know by Lemma~\ref{lem:awc1} (i) 
that $\phi $ satisfies (\ref{eq:awc}). 
Finally, 
we show that $\gamma _0 ^{\vv }<\gamma _1^{\vv }$. 
The second part of (\ref{eq:awc}) 
implies $\gamma _1^{\vv }>\gamma _2^{\vv }$. 
If $\gamma _0^{\vv }\leq \gamma _2^{\vv }$, 
then we get $\gamma _0^{\vv }<\gamma _1^{\vv }$. 
So assume that $\gamma _0^{\vv }>\gamma _2^{\vv }$. 
Then, we have 
$\degv \phi (y_2)\geq \gamma _0^{\vv }+\gamma _3^{\vv }$ 
by the contraposition of Lemma~\ref{lem:taiguu}. 
Since $\degv \phi (y_2)=3\e _1$ and $\gamma _3^{\vv }=\e _1$, 
it follows that $2\e _1\geq \gamma _0^{\vv }$. 
Hence, 
we get 
$$
\gamma _1^{\vv }\geq \gamma _2^{\vv }+2\gamma _3^{\vv }
>2\gamma _3^{\vv }=2\e _1\geq \gamma _0^{\vv }
$$ 
by the second part of (\ref{eq:awc}). 
Thus, 
we conclude from Lemma~\ref{lem:awc} that 
$\phi $ is wild, 
a contradiction. 
This proves that $\gamma _2^{\vv }\leq \e _1$. 
Therefore, 
we may write 
$$
z_2+h_2=\alpha x_1+\alpha 'x_3+g, 
$$
where $\alpha ,\alpha '\in k$ 
and $g\in k[x_2]$. 
In order to conclude that $\gamma _2^{\vv }<\e _1$, 
it suffices to show that 
$(\alpha ,\alpha ')=(0,0)$. 
Suppose to the contrary that 
$(\alpha ,\alpha ')\neq (0,0)$. 
Then, 
we have $\alpha \neq 0$, 
since $(z_2+h_2)^{\vv }=\alpha x_1+\alpha 'x_3$ 
does not belong to $k[z_3^{\vv }]=k[x_3]$. 
Since $h_2$ is an element of $k[z_3]$, 
and $z_3=\alpha _3x_3+g_3$ belongs to $k[x_2,x_3]$, 
we know that $h_2$ belongs to $k[x_2,x_3]$. 
Hence, $z_2=\alpha x_1+\alpha 'x_3+g-h_2$ 
is a linear polynomial in $x_1$ over $k[x_2,x_3]$ 
with leading coefficient $\alpha $. 
Regard 
$z_1z_3$ and $\theta (z_2)$ 
as polynomials in $x_1$ over $k[x_2,x_3]$. 
Then, 
the leading coefficient of $z_1z_3$ 
is a multiple of $z_3$, 
while that of $\theta (z_2)$ is an element of $k^{\times }$. 
Since $z_3$ does not belong to $k$, 
it follows that 
\begin{align*}
&\deg _{x_1}\phi (f_{\theta })
=\deg _{x_1}\bigl(z_1z_3+\theta (z_2) \bigr) 
=\max \{ \deg _{x_1}z_1z_3,\deg _{x_1}\theta (z_2)\}  . 
\end{align*}
Since $\deg _{x_1}z_3=0$ and $\deg _{x_1}z_2=1$, 
this gives that 
$\deg _{x_1}\phi (f_{\theta })=\max \{ \deg _{x_1}z_1,d\} $. 
Hence, we get 
$$
\deg _{x_1}\phi (y_2)=
\deg _{x_1}\bigr(z_2+\phi (f_{\theta })z_3\bigr) 
=\deg _{x_1}\phi (f_{\theta })
=\max \{ \deg _{x_1}z_1,d\} . 
$$ 
Consequently, 
we see from (\ref{eq:y_11}) that 
\begin{align*}
&\deg _{x_1}\phi (y_1)
=\deg _{x_1}\Bigl( 
z_1+\bigl( 
\theta (z_2)-\theta (\phi (y_2))
\bigr)z_3^{-1}
\Bigr) \\
&\quad =\deg _{x_1}\theta (\phi (y_2))
=d\deg _{x_1}\phi (y_2)
\geq d^2. 
\end{align*}
On the other hand, 
we have 
$\deg _{x_1}y_2=\deg _{x_1}(x_2+f_{\theta }x_3)=1$. 
Hence, we get 
$$
\deg _{x_1}y_1
=\deg _{x_1}\Bigl( 
x_1+\bigl( 
\theta (x_2)-\theta (y_2)
\bigr)x_3^{-1}
\Bigr) =\deg _{x_1}\theta (y_2)=d. 
$$
This contradicts that $\phi (y_1)=y_1$, 
thus proving that $(\alpha ,\alpha ')=(0,0)$. 
Therefore, 
we conclude that $\gamma _2^{\vv }<\e _1$. 
This completes the proof of 
Proposition~\ref{prop:tawild key} (iii).

\section{Proof (III)}
\setcounter{equation}{0}
\label{sect:tawildpf3}

In this section, 
we complete the proof of Theorem~\ref{thm:tawild} (ii). 
Assume that $d\geq 9$. 
Take $\phi \in G_{\theta }$ 
such that $\degv \phi (y_2)>\degv \phi (x_3)$. 
Then, 
we have 
\begin{equation}\label{eq:rabbit31}
\phi (x_3)=z_3=\alpha _3x_3+g_3
\end{equation}
for some $\alpha _3\in k^{\times }$ and $g_3\in k[x_2]$ 
by Proposition~\ref{prop:tawild key} (ii). 
We establish that $\phi =\phi _{\alpha _3}$ 
and $\alpha _3$ belongs to $T_{\theta }$. 
If $\phi =\id _{\kx }$, 
then we have $\alpha _3=1$. 
Since $\iota :T_{\theta }\to G_{\theta }$ 
is a homomorphism of groups, 
it follows that $\phi _{\alpha _3}=\rho (1)=\id _{\kx }$. 
Hence, 
the assertion is true. 
In what follows, 
we assume that $\phi \neq \id _{\kx }$.

By Proposition~\ref{prop:tawild key} (iii), 
$\degv (z_2+h_2)=\gamma _2^{\vv }$ 
is less than $\e _1$. 
Hence, $z_2+h_2$ belongs to $k[x_2]$. 
Since $z_2+h_2=\psi (x_2)$ 
is a coordinate of $\kx $ over $k$, 
it follows that $z_2+h_2$ 
is a linear polynomial in $x_2$ over $k$. 
Hence, we have 
\begin{equation}\label{eq:rabbit21}
\phi (x_2)=z_2=\alpha _2x_2+g_2 
\end{equation}
for some $\alpha _2\in k^{\times }$ 
and $g_2\in k[z_3]=k[\alpha _3x_3+g_3]$, 
since $h_2$ is an element of $k[z_3]$. 
From this and (\ref{eq:rabbit31}), 
we see that $\phi $ induces an automorphism of $k[x_2,x_3]$. 
Thus, 
we know that $k[x_1,\phi (x_2),\phi (x_3)]=\kx $. 
This implies that 
\begin{equation}\label{eq:rabbit11}
\phi (x_1)=\alpha _1x_1+g_1
\end{equation}
for some $\alpha _1\in k^{\times }$ and 
$g_1\in k[\phi (x_2),\phi (x_3)]=k[x_2,x_3]$ 
by the following lemma.

\begin{lem}\label{lem:useful}
Let $\tau \in \Aut (\kx /k)$ be such that 
$k[x_i,\tau (x_2),\tau (x_3)]=\kx $ 
for some $i\in \{ 1,2,3\} $. 
Then, 
we have $\tau (x_1)=\alpha x_i+g$ 
for some $\alpha \in k^{\times }$ and 
$g\in k[\tau (x_2),\tau (x_3)]$. 
\end{lem}
\begin{proof}
Since $k[x_i,\tau (x_2),\tau (x_3)]=\kx $, 
we can define $\rho \in \Aut (\kx /k)$ by 
$$
\rho (x_1)=x_i\quad \text{and}\quad 
\rho (x_j)=\tau (x_j)\quad \text{for}\quad j=2,3. 
$$
Then, 
$\rho ^{-1}\circ \tau $ belongs to $\Aut (\kx /k[x_2,x_3])$. 
Hence, 
we may write 
$(\rho ^{-1}\circ \tau )(x_1)=\alpha x_1+g'$, 
where $\alpha \in k^{\times }$ and $g'\in k[x_2,x_3]$. 
Then, 
we have 
$$
\tau (x_1)
=\rho ((\rho ^{-1}\circ \tau )(x_1))=\rho (\alpha x_1+g')
=\alpha x_i+\rho (g'), 
$$
in which $\rho (g')$ is an element of 
$\rho (k[x_2,x_3])=k[\tau (x_2),\tau (x_3)]$. 
\end{proof}

Since 
$y_2=x_2+f_{\theta }x_3=x_2+\bigl(x_1x_3+\theta (x_2)\bigr)x_3$ 
is a linear polynomial in 
$x_1$ with leading coefficient $x_3^2$, 
we see from (\ref{eq:y_11}) that 
$$
y_1=-cx_1^dx_3^{2d-1}+(\text{terms of lower degree in }x_1).
$$
In view of (\ref{eq:rabbit31}), 
(\ref{eq:rabbit21}) and (\ref{eq:rabbit11}), 
we have 
$$
\phi (y_1)=
-c(\alpha _1x_1)^d(\alpha _3x_3+g_3)^{2d-1}
+(\text{terms of lower degree in }x_1). 
$$ 
Since $\phi (y_1)=y_1$, 
we get 
$\alpha _1^d(\alpha _3x_3+g_3)^{2d-1}=x_3^{2d-1}$ 
by comparing the coefficients of $x_1^d$. 
This implies that $g_3=0$. 
Hence, 
we have 
\begin{equation}\label{eq:rabbit32}
\phi (x_3)=z_3=\alpha _3x_3=\phi _{\alpha _3}(x_3). 
\end{equation}
From this, we know that 
$\phi $ commutes with the substitution $x_3\mapsto 0$. 
By (\ref{eq:y_1}), 
we see that $y_1$ is sent to 
$x_1-\theta '(x_2)\theta (x_2)$ 
by the substitution $x_3\mapsto 0$. 
Since $y_1=\phi (y_1)$, 
it follows that 
$$
x_1-\theta '(x_2)\theta (x_2)
=\phi (x_1)-\phi (\theta '(x_2)\theta (x_2))
=\alpha _1x_1+(\text{an element of }k[x_2,x_3]) 
$$
by (\ref{eq:rabbit11}) and (\ref{eq:rabbit21}). 
This gives that $\alpha _1=1$. 
Therefore, 
we conclude that 
$$
\phi (x_1)=x_1+g_1. 
$$

Since $y_3=x_3$, 
we have $\phi (y_3)=\alpha _3y_3$ 
by (\ref{eq:rabbit32}). 
Hence, 
we get 
$$
k[y_1,y_2,y_3]=k[\phi (y_1),y_2,\phi (y_3)]. 
$$
By Lemma~\ref{lem:useful}, 
this implies that 
$$
\phi (y_2)=\beta _2y_2+h
$$
for some $\beta _2\in k^{\times }$ 
and $h\in k[y_1,y_3]=k[y_1,x_3]$. 
We note that $\degv h\leq 3\e _1$, 
since $\degv \phi (y_2)=3\e _1$ 
by Proposition~\ref{prop:tawild key} (i), 
and $\degv y_2=\degv x_1x_3^2=3\e _1$ by (\ref{eq:y highest}). 
We prove that $h$ belongs to $k[x_3]$. 
Suppose to the contrary that 
$h$ does not belong to $k[x_3]$. 
Then, 
we have $\degv ^{S}h\geq \degv y_1$, 
where $S:=\{ y_1,x_3\} $. 
Since $y_1^{\vv }\approx x_1^dx_3^{2d-1}$ and $x_3$ 
are algebraically independent over $k$, 
we have $\degv h=\degv ^Sh$. 
Hence, 
we get $\degv h\geq \degv y_1=(3d-1)\e _1>3\e _1$, 
a contradiction. 
This proves that $h$ belongs to $k[x_3]$. 
From (\ref{eq:y_11}), 
we see that 
\begin{align*}
0&=\alpha _3x_3\bigl(y_1-\phi (y_1)\bigr) \\
&=\alpha _3\bigl(x_1x_3+\theta (x_2)-\theta (y_2)\bigr) 
-\bigl(\alpha _3(x_1+g_1)x_3
+\theta (\alpha _2x_2+g_2)-\theta (\beta _2y_2+h)
\bigr) \\
&=\bigl( 
\theta (\beta _2y_2+h)-\alpha _3\theta (y_2)
\bigr) 
-\bigl( 
\alpha _3g_1x_3+\theta (\alpha _2x_2+g_2)
-\alpha _3\theta (x_2)
\bigr) . 
\end{align*}
Hence, we have 
\begin{equation}\label{eq:awcpf:p}
p:=
\theta (\beta _2y_2+h)-\alpha _3\theta (y_2)
=\alpha _3g_1x_3+\theta (\alpha _2x_2+g_2)
-\alpha _3\theta (x_2). 
\end{equation}
Since $h$ belongs to $k[x_3]=k[y_3]$, 
we see that $p$ belongs to $k[y_2,y_3]$. 
As an element of $k[y_2,y_3]$, 
we may consider the partial derivative 
$\partial p/\partial y_i$ for $i=2,3$. 
Then, 
we have 
$$
\frac{\partial p}{\partial x_1}
=\frac{\partial y_2}{\partial x_1}
\frac{\partial p}{\partial y_2}
+\frac{\partial y_3}{\partial x_1}
\frac{\partial p}{\partial y_3}
=x_3^2\frac{\partial p}{\partial y_2} 
$$
by chain rule. 
Since $g_1$ and $g_2$ 
are elements of $k[x_2,x_3]$, 
we know that $p$ belongs to $k[x_2,x_3]$ 
by (\ref{eq:awcpf:p}). 
Hence, 
we have $\partial p/\partial x_1=0$. 
Thus, 
we get 
$$
0=\frac{\partial p}{\partial y_2}
=\beta _2\theta '(\beta _2y_2+h)-
\alpha _3\theta '(y_2) 
$$
by the preceding equality. 
This implies that 
$h$ is algebraic over $k(y_2)$. 
Since $h$ is an element of $k[y_3]$, 
it follows that $h$ belongs to $k$. 
Consequently, 
$p$ belongs to $k[y_2]$. 
Since $\partial p/\partial y_2=0$, 
we conclude that $p$ belongs to $k$.

We prove that $\beta _2\neq 1$. 
Suppose to the contrary that $\beta _2=1$. 
Then, 
the coefficient of $y_2^d$ in $p$ 
is equal to $c(1-\alpha _3)$. 
Since $p$ belongs to $k$, 
it follows that $\alpha _3=1$. 
Hence, we get $\phi (y_3)=y_3$ by (\ref{eq:rabbit32}). 
Since $\beta _2=\alpha _3=1$, 
we have 
$$
p=\theta (y_2+h)-\theta (y_2)=\sum _{i=1}^d
\frac{\theta ^{(i)}(y_2)}{i!}h^i. 
$$ 
Since $p$ belongs to $k$, 
this implies that $h=0$. 
Thus, 
we get $\phi (y_2)=\beta _2y_2+h=y_2$. 
Since $\phi (y_1)=y_1$, 
we conclude that 
$\phi =\id _{\kx }$, 
a contradiction. 
This proves that $\beta _2\neq 1$. 
Set $\kappa '=h/(1-\beta _2)$, 
and write 
$$
\theta (z)=\sum _{i=0}^du_i'(z-\kappa ')^i,
$$
where $u_i'\in k$ for $i=0,\ldots ,d$. 
Then, we have 
$$
\phi (y_2-\kappa ')=
(\beta _2y_2+h)-\frac{h}{1-\beta _2}
=\beta _2\left(
y_2-\frac{h}{1-\beta _2}
\right) 
=\beta _2(y_2-\kappa '). 
$$
Hence, 
we know that
\begin{align*}
&p=\theta (\beta _2y_2+h)-\alpha _3\theta (y_2)
=\theta (\phi (y_2))-\alpha _3\theta (y_2) \\
&\quad 
=\sum _{i=0}^du_i'\phi (y_2-\kappa ')^i
-\alpha _3\sum _{i=0}^du_i'(y_2-\kappa ')^i
=\sum _{i=0}^du_i'(\beta _2^i-\alpha _3)(y_2-\kappa ')^i. 
\end{align*}
Since $p$ belongs to $k$, 
it follows that $p=u_0'(1-\alpha _3)$, 
and $u_i'=0$ or $\beta _2^i=\alpha _3$ 
for $i=1,\ldots ,d$.

We show that $\kappa '=\kappa $ by contradiction. 
Suppose that $\kappa '':=\kappa '-\kappa $ is nonzero. 
Since $u_{d-1}=0$ by the definition of $\kappa $, 
we may write 
$$
\theta (z)
=u_d(z-\kappa '+\kappa '')^d
+\sum _{i=0}^{d-2}u_i(z-\kappa )^i
=u_d(z-\kappa ')^d+d\kappa ''u_d(z-\kappa ')^{d-1}+\cdots . 
$$
Hence, we get 
$u_d'=u_d$ and $u_{d-1}'=d\kappa ''u_d$. 
Since $u_d\neq 0$, 
and $\kappa ''\neq 0$ by supposition, 
we know that 
$u_d'$ and $u_{d-1}'$ are nonzero. 
Hence, we have 
$\beta _2^i=\alpha _3$ for $i=d,d-1$. 
Thus, 
we get 
$\beta _2=1$, 
a contradiction. 
This proves that $\kappa '=\kappa $. 
Therefore, 
we have $u_i'=u_i$ for $i=0,\ldots ,d$. 
Consequently, we get 
\begin{equation}\label{eq:rabbit:p}
p=u_0'(1-\alpha _3)=u_0(1-\alpha _3)=(1-\alpha _3)\theta (\kappa ). 
\end{equation}

To complete the proof, 
we verify that 
$\phi (x_i)=\phi _{\alpha _3}(x_i)$ for $i=1,2$. 
Since $\sigma _{\theta }(f_{\theta })=f_{\theta }$, 
$\phi (y_1)=y_1$ and $\phi (y_3)=\alpha _3y_3$, 
we have 
\begin{align*}
&\phi (f_{\theta })=\phi (\sigma _{\theta }(f_{\theta }))=
\phi \bigl(y_1y_3+\theta (y_2)\bigr)
=y_1(\alpha _3y_3)+\theta (\phi (y_2)) \\
&\quad =\alpha _3f_{\theta }
-\alpha _3\theta (y_2)+\theta (\phi (y_2))
=\alpha _3f_{\theta }+p. 
\end{align*}
Hence, we get 
$$
\phi (y_2)
=\phi (x_2+f_{\theta }x_3)
=(\alpha _2x_2+g_2)+(\alpha _3f_{\theta }+p)(\alpha _3x_3). 
$$
On the other hand, 
we have 
$$
\phi (y_2)=\beta _2y_2+h
=\beta _2(x_2+f_{\theta }x_3)+h. 
$$
Since $g_2$ and $h$ are elements of $k[x_3]$ and $k$, 
respectively, 
the monomial $x_2$ does not appear in 
$g_2$ and $h$. 
Clearly, 
the monomial $x_2$ does not appear in 
a multiple of $x_3$. 
Thus, 
we get $\alpha _2=\beta _2$ 
by comparing the coefficients of $x_2$. 
Equating the two expressions of $\phi (y_2)$, 
we obtain 
$$
g_2-h
=\bigl(\beta _2f_{\theta }-
\alpha _3(\alpha _3f_{\theta }+p)\bigr)x_3
=\bigl((\alpha _2-\alpha _3^2)f_{\theta }
-\alpha _3p\bigr)x_3. 
$$
Note that $g_2-h$ belongs to $k[x_3]$, 
while $f_{\theta }$ does not belong to $k[x_3]$. 
Since $p$ is a constant, 
it follows that $\alpha _2=\alpha _3^2$ 
and $g_2=h-\alpha _3px_3$. 
Therefore, we have 
\begin{align*}
&\phi (x_2-\kappa )=\phi (x_2-\kappa ')
=\alpha _2x_2+g_2-\frac{h}{1-\beta _2}
=\alpha _2x_2+(h-\alpha _3px_3)-\frac{h}{1-\alpha _2} \\
&\quad 
=\alpha _2\left( x_2-\frac{h}{1-\alpha _2}\right) -\alpha _3px_3 
=\alpha _3^2(x_2-\kappa )+\alpha _3(\alpha _3-1)\theta (\kappa )x_3
\end{align*}
because of (\ref{eq:rabbit:p}). 
This proves that 
$\phi (x_2)=\phi _{\alpha _3}(x_2)$. 
By (\ref{eq:awcpf:p}) and (\ref{eq:rabbit:p}), 
we have 
$$
g_1=\frac{\alpha _3\theta (x_2)
-\theta (\alpha _2x_2+g_2)+p}{\alpha _3x_3}
=\frac{\alpha _3\theta (x_2)
-\phi (\theta (x_2))+
(1-\alpha _3)\theta (\kappa )}{\alpha _3x_3}
=g_{\alpha _3}. 
$$
Since $\phi (x_1)=x_1+g_1$, 
this proves that 
$\phi (x_1)=\phi _{\alpha _3}(x_1)$. 
Therefore, 
we conclude that $\phi =\phi _{\alpha _3}$. 
Since $g_{\alpha _3}=g_1$ belongs to $\kx $, 
we know that $\alpha _3$ belongs to $T_{\theta }$ 
as noted before Theorem~\ref{thm:G_{theta}}. 
This completes the proof of 
Theorem~\ref{thm:tawild} (ii), 
and thereby completing the proof of Theorem~\ref{thm:G_{theta}}.

\chapter{Locally nilpotent derivations of rank three}
\label{chapter:rank3}

\section{Main result}
\setcounter{equation}{0}
\label{sect:maxrank}

Recall that the {\it rank} of $D\in \Der _k\kx $ 
is defined to be the minimal number $r\geq 0$ 
for which there exists $\sigma \in \Aut (\kx /k)$ 
such that $D(\sigma (x_i))\neq 0$ for $i=1,\ldots ,r$. 
It is difficult to handle $D\in \lnd _k\kx $ 
with $\rank D=3$. 

In this chapter, 
we construct 
families of elements of $\lnd _k\kx $ 
using 
``local slice construction" due to Freudenburg~\cite{Flsc}. 
Then, we prove that most of the members 
of the families are of rank three, 
for which the exponential automorphisms are wild.

For $i=0,1$, 
take $t_i\in \N $ 
and $\alpha _j^i\in k$ for $j=1,\ldots ,t_i-1$. 
First, 
we define a sequence $(b_i)_{i=0}^{\infty }$ 
of integers by 
$$
b_0=b_1=0
\text{ \ and \ }
b_{i+1}=t_ib_i-b_{i-1}+\xi _i
\text{ \ for \ }
i\geq 1, 
$$
where $t_i:=t_0$ if $i$ is an even number, 
and $t_i:=t_1$ otherwise, 
and where $\xi _i:=1$ if $i\equiv 0,1\pmod{4}$, 
and $\xi _i:=-1$ otherwise. 
Next, 
for each $i\geq 0$, 
we define a polynomial $\eta _i(y,z)\in k[y,z]$ by 
\begin{align*}
\eta _i(y,z)&=
z^{t_ib_i+1}+\sum _{j=1}^{t_i-1}
\alpha _j^iy^jz^{(t_i-j)b_i}+y^{t_i}
&
\text{ if }
i\equiv 0,1\pmod{4}\\
\eta_i(y,z)&=
y^{t_i}+\sum _{j=1}^{t_i-1}\alpha _j^iz^{jb_i-1}y^{t_i-j}
+z^{t_ib_i-1}
&
\text{otherwise,\qquad \ }
\end{align*}
where 
$\alpha _j^i:=\alpha _j^0$ if $i$ is an even number, 
and $\alpha _j^i:=\alpha _j^1$ otherwise. 
Set 
$$
r=x_1x_2x_3-\sum _{i=1}^{t_0}\alpha _i^0x_2^i
-\sum _{j=1}^{t_1}\alpha _j^1x_1^j, 
$$
where $\alpha _{t_0}^0:=\alpha _{t_1}^1:=1$. 
Then, we define a sequence $(f_i)_{i=0}^{\infty }$ 
of rational functions by 
$$
f_0=x_2,\quad 
f_1=x_1\quad 
\text{and}\quad 
f_{i+1}=\eta _i(f_i,r)f_{i-1}^{-1}\quad
\text{for}\quad i\geq 1. 
$$
Set $q_i=\eta _i(f_i,r)$ for each $i\geq 1$. 
Then, we have 
\begin{equation}\label{eq:q_1}
q_1=\eta _1(x_1,r)
=r+\sum _{j=1}^{t_1}\alpha _j^1x_1^j
=x_1x_2x_3-\sum _{i=1}^{t_0}\alpha _i^0x_2^{i}. 
\end{equation}
Hence, we get 
\begin{equation}\label{eq:f_2}
f_2=q_1x_2^{-1}=x_1x_3-\theta (x_2),\quad 
r=x_2f_2-\sum _{j=1}^{t_1}\alpha _j^1x_1^j, 
\end{equation}
where $\theta (z):=\sum _{i=1}^{t_0}\alpha _i^0z^{i-1}$. 
When $t_0=2$, 
we have $f_2=x_1x_3-x_2-\alpha _1^0$. 
In this case, 
we can define $\tau _2\in \T (k,\x )$ 
by 
\begin{equation}\label{eq:tau_2}
\tau _2(x_1)=f_2,\quad 
\tau _2(x_2)=x_1\quad\text{and}\quad 
\tau _2(x_3)=x_3. 
\end{equation}

For each $g_1,g_2\in \kx $, 
we define 
$\Delta _{(g_1,g_2)}\in \Der _k\kx $ 
by 
$$
dg_1\wedge dg_2\wedge dg=\Delta _{(g_1,g_2)}(g)dx_1\wedge dx_2\wedge dx_3
$$ 
for $g\in \kx $. 
Then, $k[g_1,g_2]$ 
is always contained in $\ker \Delta _{(g_1,g_2)}$. 
When $f_i$ and $f_{i+1}$ belong to $\kx $, 
we consider 
the derivation 
$$
D_i:=\Delta _{(f_{i+1},f_i)}
$$ 
of $\kx $. 
For example, 
since $f_0=x_2$ and $f_1=x_1$ by definition, 
$D_0=\Delta _{(x_1,x_2)}$ 
is the partial derivation of $\kx $ in $x_3$. 
By (\ref{eq:f_2}), 
we see that $D_1=\Delta _{(f_2,x_1)}$ 
is the triangular derivation of $\kx $ defined by 
\begin{equation}\label{eq:D_1}
D_1(x_1)=0,\ \ D_1(x_2)=x_1\ \text{ and }\ 
D_1(x_3)=\theta '(x_2)
=\sum _{i=2}^{t_0}(i-1)\alpha _i^0x_2^{i-2}, 
\end{equation}
since 
$$
df_2\wedge dx_1\wedge dx_2=
\frac{\partial f_2}{\partial x_3}
dx_1\wedge dx_2\wedge dx_3,\ \ 
df_2\wedge dx_1\wedge dx_3=
-\frac{\partial f_2}{\partial x_2}
dx_1\wedge dx_2\wedge dx_3. 
$$

We say that $D\in \Der _k\kx $ 
is {\it irreducible} 
if $D(\kx )$ is contained in no proper principal ideal of $\kx $, 
or equivalently $D(x_1)$, $D(x_2)$ and $D(x_3)$ 
have no common factor. 
Here, 
``no common factor" means 
``no non-constant common factor". 
Then, $D_0$ is clearly irreducible, 
and $D_1$ is irreducible if and only if $t_0\geq 2$. 
We mention that, 
if $t_0=1$, 
then $\ker D_1=k[x_1,x_3]$ 
is not equal to $k[f_1,f_2]=k[x_1,x_1x_3-1]$.

We can similarly construct the sequence 
of rational functions from the same data 
\begin{equation}\label{eq:data}
t_0,\quad t_1,\quad (\alpha _j^0)_{j=1}^{t_0-1}\quad 
\text{and}\quad (\alpha _j^1)_{j=1}^{t_1-1}
\end{equation}
by interchanging the roles of $t_0$ and $t_1$, 
and $(\alpha _j^0)_{j=1}^{t_0-1}$ and $(\alpha _j^1)_{j=1}^{t_1-1}$, 
respectively. 
To distinguish it from the original one, 
we denote it by $(f_i')_{i=0}^{\infty }$. 
If $t_1=2$, 
then we can define $\tau _2'\in \T (k,\x )$ 
by $\tau _2'(x_1)=f_2'$, 
$\tau _2'(x_2)=x_1$ and $\tau _2'(x_3)=x_3$. 
When $f_i'$ and $f_{i+1}'$ belong to $\kx $, 
we define $D_i'=\Delta _{(f_{i+1}',f_i')}$.

Let $I$ be the set of $i\in \N $ such that 
$$
a_j:=t_jb_j+\xi _j>0
$$ 
for $j=1,\ldots ,i$. 
Then, 
we have 
\begin{align}\begin{split}\label{eq:I}
I=\left\{ 
\begin{array}{cl}
\{ 1\} & \text{if } t_0=1\\
\{ 1,2\} & \text{if }(t_0,t_1)=(2,1)\\ 
\{ 1,2,3,4 \} & \text{if }(t_0,t_1)=(3,1)\\
\N & \text{otherwise}
\end{array}
\right. 
\end{split}\end{align}
as will be shown in Section~\ref{sect:sequence}. 
We note that $a_i=\xi _i=1$ for $i=0,1$, 
since $b_0=b_1=0$. 
For each $i\geq 1$, 
we have 
\begin{equation}\label{eq:a_i zenkasiki}
a_{i+1}=t_{i+1}a_i-a_{i-1}, 
\end{equation}
since 
\begin{align*}
a_{i+1}-t_{i+1}a_i+a_{i-1} 
&=(t_{i+1}b_{i+1}+\xi _{i+1})
-t_{i+1}(t_ib_i+\xi _i)+(t_{i-1}b_{i-1}+\xi _{i-1}) \\ 
&=t_{i+1}(b_{i+1}-t_ib_i+b_{i-1}-\xi _i)
+\xi _{i+1}+\xi _{i-1}=0. 
\end{align*}

Now, 
we are ready to state the main results of this chapter.

\begin{thm}\label{thm:lsc1}
In the notation above, 
the following assertions hold for each $i\in I$$:$

\noindent{\rm (i)} 
$f_i$ and $f_{i+1}$ belong to $\kx $, 
$D_i$ belongs to $\lnd _k\kx $ 
and $D_i(r)=f_if_{i+1}$. 
Moreover, 
we have the following$:$

\noindent{\rm \ (a)} 
If $i$ is the maximum of $I$, 
then $D_i$ is not irreducible and 
$\ker D_i\neq k[f_i,f_{i+1}]$. 

\noindent
{\rm \ (b)} 
If $i$ is not the maximum of $I$, then 
$D_i$ is irreducible and $\ker D_i=k[f_i,f_{i+1}]$. 

\noindent{\rm (ii)} 
If $t_0=2$, 
then we have 
$\tau _2^{-1}\circ D_i\circ \tau _2 =D_{i-1}'$. 
Hence, 
$D_2$ is tamely triangularizable. 
If $t_0=t_1=2$, 
then we have 
$\tau ^{-1}\circ D_i\circ \tau =D_0$, 
where 
$$
\tau :=\left\{ 
\begin{array}{cl}
(\tau _2\circ \tau _2')^{i/2} & \text{ if $i$ is an even number} \\
(\tau _2\circ \tau _2')^{(i-1)/2}\circ \tau _2 & 
\text{ otherwise. }
\end{array}
\right. 
$$

\noindent{\rm (iii)} 
If $t_0=2$, $t_1\geq 3$ and $i\geq 3$, 
or if $t_0\geq 3$ and $i\geq 2$, 
then $\exp hD_i$ is wild for each $h\in \ker D_i\sm \zs $. 
\end{thm}

By means of this theorem, 
we can completely determine when 
$\exp hD_i$ is tame or wild 
for $h\in \ker D_i\sm \zs $ and $i\in I$ as follows. 
First, consider the case of $i=1$. 
Since $D_1$ is triangular and $D_1(x_1)=0$ 
by (\ref{eq:D_1}), 
we see that $hD_1$ is triangular if $h$ belongs to $k[x_1]$. 
In this case, 
$\exp hD_1$ is tame. 
Assume that $h$ does not belong to $k[x_1]$. 
If $t_0=1$, 
then we have $D_1(x_3)=0$. 
Hence, 
$\exp hD_1$ is elementary, 
and so tame. 
When $t_0\geq 2$, 
we have $D_1(x_j)\neq 0$ for $j=2,3$. 
By Theorem~\ref{thm:triangular3}, 
we know 
that $\exp hD_1$ is tame if and only if 
$$
\frac{\partial D_1(x_3)}{\partial x_2}
=\sum _{i=3}^{t_0}(i-1)(i-2)\alpha _i^0x_2^{i-3}
\quad \text{belongs to}\quad D(x_2)k[x_1,x_2]=x_1k[x_1,x_2], 
$$ 
and hence if and only if $t_0=2$. 
Therefore, 
$\exp hD_1$ is tame 
if and only if $h$ belongs to $k[x_1]$ or $t_0\leq 2$. 
When $t_0=1$, 
we have $I=\{ 1\} $. 
Hence, 
this case is completed. 
Assume that $t_0=2$. 
Then, we have 
$$
hD_2=\tau _2\circ (\tau _2^{-1}(h)D_1')\circ \tau _2^{-1}
$$ 
by (ii). 
Since $\tau _2$ is tame, 
$\exp hD_2$ is tame if and only if so is 
$\exp \tau _2^{-1}(h)D_1'$. 
From the preceding discussion, 
it follows that $\exp \tau _2^{-1}(h)D_1'$ 
is tame if and only if 
$\tau _2^{-1}(h)$ belongs to $k[x_1]$ or $t_1\leq 2$. 
Since $\tau _2(x_1)=f_2$ by definition, 
$\tau _2^{-1}(h)$ belongs to $k[x_1]$ 
if and only if $h$ belongs to $k[f_2]$. 
Therefore, 
we conclude that $\exp hD_2$ 
is tame if and only if 
$h$ belongs to $k[f_2]$ or $t_1\leq 2$ when $t_0=2$. 
If $(t_0,t_1)=(2,1)$, 
then we have $I=\{ 1,2\} $. 
Hence, 
this case is completed. 
If $t_0=t_1=2$, 
then we have 
$hD_i=\tau \circ (\tau ^{-1}(h)D_0)\circ \tau ^{-1}$ 
for each $i\in I$ by (ii). 
Since $D_0(x_j)=0$ for $j=1,2$, 
we see that 
$\exp \tau ^{-1}(h)D_0$ is elementary. 
Hence, 
it follows that $\exp hD_i$ is tame 
by the tameness of $\tau $. 
If $t_0=2$, $t_1\geq 3$ and $i\geq 3$, 
or if $t_0\geq 3$ and $i\geq 2$, 
then $\exp hD_i$ is wild due to (iii). 
Therefore, 
we have completely determined when $\exp hD_i$ 
is tame or wild for $h\in \ker D_i\sm \zs $ and $i\in I$.

By the discussion above, 
we see that $\exp hD_i$ is tame only if 
$i=1$ or $(i,t_0)=(2,2)$ or $(t_0,t_1)=(2,2)$. 
In this case, 
$D_i$ is tamely triangularizable, 
and hence kills a tame coordinate of $\kx $ over $k$. 
Thanks to Theorem~\ref{thm:anstoF}, 
this implies that 
$hD_i$ is tamely triangularizable if $\exp hD_i$ is tame. 
Clearly, 
$\exp hD_i$ is tame if $hD_i$ 
is tamely triangularizable. 
Therefore, 
we get the following corollary to Theorem~\ref{thm:lsc1}.

\begin{cor}\label{cor:tametriangular}
For $h\in \ker D_i\sm \zs $ and $i\in I$, 
it holds that $\exp hD_i$ is tame 
if and only if $hD_i$ is tamely triangularizable. 
\end{cor}

If $D$ is a derivation of $\kx $, 
then 
$$
\pl D:=D(\kx )\cap \ker D
$$ 
forms an ideal of $\ker D$, 
and is called the {\it plinth ideal} of $D$. 
Assume that $D$ is locally nilpotent. 
Then, 
$\pl D$ is not a zero ideal unless $D=0$. 
Moreover, 
$\pl D$ is always a principal ideal of $\ker D$ 
by Daigle-Kaliman~\cite[Theorem 1]{Kaliman}. 
Hence, 
the notation $\pl D=(p)$ 
will mean that $\pl D=p\ker D$. 
The plinth ideals are important 
in the study of the ranks of locally nilpotent derivations.

In the following theorem, 
we completely determine 
$\pl D_i$ and $\rank D_i$ for $i\in I$.

\begin{thm}\label{prop:rank}
The following assertions hold for each $i\in I $$:$

\noindent{\rm (i)} 
If $t_0=1$ or $(t_0,t_1,i)=(2,1,2)$, 
then we have $\pl D_i=(f_i)$ and $\rank D_i=1$. 

\noindent{\rm (ii)} 
If $(t_0,i)=(2,1)$ 
or $t_0=t_1=2$, 
then we have 
$\pl D_i=\ker D_i$ and $\rank D_i=1$.

\noindent{\rm (iii)} If 
$t_0=2$, $t_1\geq 3$ and $i=2$, 
or if $t_0\geq 3$ and $i=1$, 
then we have $\pl D_i=(f_i)$ and $\rank D_i=2$. 

\noindent{\rm (iv)} 
If $t_0=2$, $t_1\geq 3$ and $i\geq 3$, 
or if $t_0\geq 3$, $(t_0,t_1)\neq (3,1)$ and $i\geq 2$, 
then we have $\pl D_i=(f_if_{i+1})$ 
and $\rank D_i=3$. 

\noindent{\rm (v)} 
If $(t_0,t_1)=(3,1)$, 
then we have $\pl D_2=(f_3)$ and $\rank D_2=2$, 
$\pl D_3=\ker D_3$ and $\rank D_3=1$, 
and $\pl D_4=(f_4)$ and $\rank D_4=1$. 
\end{thm}

On the homogeneity of $f_i$'s, 
we have the following result. 
Assume that $\alpha _j^i=0$ 
for $i=0,1$ and $j=1,\ldots ,t_i-1$. 
Then, 
we have 
$r=x_1x_2x_3-x_2^{t_0}-x_1^{t_1}$ 
and $f_2=x_1x_3-x_2^{t_0-1}$. 
Put 
$$
\bt :=(t_0,t_1,t_0t_1-t_0-t_1). 
$$
Then, it is easy to check that 
$r$ and $f_2$ 
are $\bt $-homogeneous. 
Moreover, 
we have the following proposition.

\begin{prop}\label{prop:homogeneous}
If $\alpha _j^i=0$ for $i=0,1$ and $j=1,\ldots ,t_i-1$, then 
$f_i$ and $f_{i+1}$ are $\bt $-homogeneous 
for each $i\in I$. 
\end{prop}

Next, 
we construct another family of elements of $\lnd _k\kx $ 
by making use of $(f_i)_{i=0}^{\infty }$. 
Take any 
$$
\lambda (y)\in k[y]\sm \zs \text{ \ and \ }
\mu (y,z)=\sum _{j\geq 1}\mu _j(y)z^j\in zk[y,z],
$$
where $\mu _j(y)\in k[y]$ for each $j$. 
For $i\geq 2$, 
we set 
$$
r_i=\lambda (f_i)\tilde{r}-\mu (f_i,f_{i-1}),
\text{ \ where \ } 
\tilde{r}:=\left\{\begin{array}{cll}
x_2 & \text{if}& i=2\\
r & \text{if}& i\geq 3. 
\end{array}\right.
$$
Then, we define 
$$
\tilde{f}_{i+1}=
\tilde{\eta }_i\left(
f_i,r_i\lambda (f_i)^{-1}\right) 
f_{i-1}^{-1}\lambda (f_i)^{a_i}, 
\text{ \ where \ } 
\tilde{\eta }_i(y,z):=
\left\{\begin{array}{cll}
y+\theta (z) & \text{if}& i=2\\
\eta _i(y,z) & \text{if}& i\geq 3. 
\end{array}\right.
$$ 
When $\tilde{f}_{i+1}$ belongs to $\kx $, 
we consider the derivation 
$\tilde{D}_i:=\Delta _{(\tilde{f}_{i+1},f_i)}$ 
of $\kx $.

To study $\tilde{D}_i$, 
we may assume that $\lambda (y)$ and $\mu (y,z)$ 
have no common factor for the following reason. 
Let $\nu (y)$ be a common factor of 
$\lambda (y)$ and $\mu (y,z)$, 
and let 
$$
\lambda _0(y):=\lambda (y)\nu (y)^{-1}\text{ \ and \ } 
\mu _0(y,z):=\mu (y,z)\nu (y)^{-1}. 
$$
Then, 
we have 
\begin{align*}
g:&=\tilde{\eta }_i
\Bigl(
f_i,
\bigl(\lambda _0(f_i)\tilde{r}-\mu _0(f_i,f_{i-1})\bigr) 
\lambda _0(f_i)^{-1}\Bigr) 
\lambda _0(f_i)^{a_i}f_{i-1}^{-1} \\
&=\tilde{\eta }_i\left(
f_i,r_i\lambda (f_i)^{-1}\right)
\lambda (f_i)^{a_i}
f_{i-1}^{-1}\nu (f_i)^{-a_i}
=\nu (f_i)^{-a_i}\tilde{f}_{i+1}. 
\end{align*}
Hence, we get 
$\tilde{D}_i=\Delta _{(\tilde{f}_{i+1},f_i)}
=\nu (f_i)^{a_i}\Delta _{(g,f_i)}$. 
Therefore, 
the study of $\tilde{D}_i$ 
is reduced to the study of $\Delta _{(g,f_i)}$.

If $\mu (y,z)=0$, 
then we have $r_i\lambda (f_i)^{-1}=\tilde{r}$. 
In this case, 
we cannot obtain a new derivation as $\tilde{D}_i$ 
for the following reason. 
Since $\mu (y,z)=0$, 
we may assume that 
$\lambda (y)=c$ for some $c\in k^{\times }$ 
by the preceding discussion. 
Then, 
we have 
\begin{equation}\label{eq:mu=0}
\begin{aligned}
&\tilde{f}_3
=\tilde{\eta }_2(f_2,c^{-1}r_2)f_1^{-1}c^{a_2}
=\tilde{\eta }_2(f_2,x_2)f_1^{-1}c^{a_2} \\
&\quad =c^{a_2}\bigl(f_2+\theta (x_2)\bigr) f_1^{-1}
=c^{a_2}x_1x_3x_1^{-1}=c^{a_2}x_3 
\end{aligned}
\end{equation}
when $i=2$. 
This gives that $\tilde{D}_2=c^{a_2}\Delta _{(x_3,f_2)}$. 
Since $f_2$ is a symmetric polynomial 
in $x_1$ and $x_3$ over $k[x_2]$, 
we may regard $\tilde{D}_2$ as 
$c^{a_2}\Delta _{(x_1,f_2)}=-c^{a_2}D_1$ 
by interchanging $x_1$ and $x_3$. 
When $i\geq 3$, 
we have $\tilde{r}=r$ 
and $\tilde{\eta }_i(y,z)=\eta _i(y,z)$ 
by definition. 
Since $r_i\lambda (f_i)^{-1}=\tilde{r}$ 
and $\lambda (y)=c$, 
it follows that $\tilde{f}_{i+1}=c^{a_i}f_{i+1}$. 
Thus, we get $\tilde{D}_i=c^{a_i}D_i$. 
Therefore, 
we may assume that $\mu (y,z)\neq 0$.

In the notation above, 
we have the following theorem.

\begin{thm}\label{thm:lsc2}
Assume that $t_0\geq 3$ and $i=2$, 
or $t_0\geq 3$, $(t_0,t_1)\neq (3,1)$ and $i\geq 3$. 
If $\lambda (y)\in k[y]\sm \zs $ 
and $\mu (y,z)\in zk[y,z]\sm \zs $ have no common factor, 
then the following assertions hold$:$

\noindent{\rm (i)} 
$\tilde{f}_{i+1}$ belongs to $\kx $, 
and $\tilde{D}_i$ is irreducible and locally nilpotent. 
Moreover, we have $\ker \tilde{D}_i=k[f_i,\tilde{f}_{i+1}]$, 
and 
\begin{equation}\label{eq:tilde{D}_i(r_i)}
\tilde{D}_i(r_i)=
\left\{ 
\begin{array}{cl}
\lambda (f_2)\tilde{f}_{3} & \text{ if }i=2 \\
\lambda (f_i)f_i\tilde{f}_{i+1}& \text{ if }i\geq 3. 
\end{array}
\right. 
\end{equation}

\noindent{\rm (ii)} 
For $h\in \ker \tilde{D}_i\sm \zs $, 
it holds that $\exp h\tilde{D}_i$ is tame 
if and only if 
$i=2$, 
$\lambda (y)$ belongs to $k^{\times }$, 
$\mu (y,z)$ belongs to $zk[z]\sm \zs $, 
and $h$ belongs to $k[\tilde{f}_3]\sm \zs $. 
If this is the case, 
then $h\tilde{D}_i$ is tamely triangularizable. 
\end{thm}

In the statement (iii) of the following theorem, 
$\sqrt{(\lambda (y))}$ 
denotes the radical of the ideal 
$(\lambda (y))$ of $k[y]$.

\begin{thm}\label{prop:rank2}
Under the assumption of 
Theorem~{\rm \ref{thm:lsc2}}, 
the following assertions hold$:$

\noindent{\rm (i)} 
If $i\geq 3$, 
then we have $\rank \tilde{D}_i=3$.

\noindent{\rm (ii)} 
If $\lambda (y)$ belongs to $k^{\times }$, 
then we have 
$\rank \tilde{D}_2=2$ and 
$\pl \tilde{D}_2=(\tilde{f}_3)$. 
Moreover, 
$\tilde{f}_3$ is a coordinate of $\kx $ over $k$.

\noindent{\rm (iii)} 
Assume that $\lambda (y)$ does not belong to $k^{\times }$. 
If $t_0\geq 4$, 
or $t_0=3$ and 
$\mu _j(y)$ does not belong to $\sqrt{(\lambda (y))}$ 
for some $j\geq 2$, 
then we have $\rank \tilde{D}_2=3$. 
If $t_0=3$ 
and $\mu _j(y)$ belongs to $\sqrt{(\lambda (y))}$ 
for every $j\geq 2$, 
then we have $\pl \tilde{D}_2=(\tilde{f}_3)$. 
\end{thm}

In the last case of (iii), 
we do not know the rank of $\tilde{D}_2$ in general. 
However, 
if $\tilde{f}_3$ is 
a coordinate of $\kx $ over $k$, 
then we may conclude that 
$\rank \tilde{D}_2=2$ as follows. 
Since $\tilde{D}_2$ kills $\tilde{f}_3$, 
we have $1\leq \rank \tilde{D}_2\leq 2$ 
by the assumption that $\tilde{f}_3$ is 
a coordinate of $\kx $ over $k$. 
Suppose that $\rank \tilde{D}_2=1$. 
Then, 
we have $\ker \tilde{D}_2=\sigma (k[x_2,x_3])$ 
for some $\sigma \in \Aut (\kx /k)$. 
Since 
$\ker \tilde{D}_2=k[f_2,\tilde{f}_3]$ 
by Theorem~\ref{thm:lsc2} (i), 
it follows that 
$$
k[f_2,\tilde{f}_3][\sigma (x_1)]
=k[\sigma (x_2),\sigma (x_3)][\sigma (x_1)]
=\sigma (\kx )=\kx . 
$$
This implies that $f_2$ is a coordinate of $\kx $ over $k$. 
Since $t_0=3$, 
we have 
$f_2=x_1x_3-(\alpha _1^0+\alpha _2^0x_2+x_2^2)$. 
Hence, 
$f_2$ is changed to a polynomial with no linear monomial 
by the substitution $x_2\mapsto x_2-\alpha _2^0/2$. 
This implies that $f_2$ 
is not a coordinate of $\kx $ over $k$, 
a contradiction. 
Therefore, 
if $\tilde{f}_3$ is a coordinate of $\kx $ over $k$, 
then we have $\rank \tilde{D}_2=2$. 
Possibly, 
$\tilde{f}_3$ is always a coordinate of $\kx $ over $k$.

Finally, 
we discuss the locally nilpotent derivations 
of rank three given by Freudenburg~\cite[Section 4]{Flsc} 
(see also~\cite[Sections 5.5.2 and 5.5.3]{Fbook}). 
Assume that 
$t_0=t_1=3$ and $\alpha _j^i=0$ 
for $i=0,1$ and $j=1,2$. 
Then, 
we have $r=x_2f_2-x_1^3$ and 
$f_2=x_1x_3-x_2^2$ by (\ref{eq:f_2}). 
Moreover, 
we have $f_0=x_2$ and $f_1=x_1$, 
and 
$$
f_{i-1}f_{i+1}=f_i^3+r^{a_i}\text{ \ for each \ }i\geq 1 
$$ 
by the definition of $\eta _i(y,z)$'s. 
As mentioned, 
we have $a_0=a_1=1$, 
and $a_{i+1}=3a_i-a_{i-1}$ for $i\geq 1$ 
by (\ref{eq:a_i zenkasiki}). 
From these conditions, 
we know that $-D_i=\Delta _{(f_i,f_{i+1})}$ 
is the same as the locally nilpotent derivation of 
``Fibonacci type" by Freudenburg for each $i\geq 1$, 
where we regard $x_1$, $x_2$ and $x_3$ as 
$x$, $-y$ and $z$, respectively. 
It is previously known that 
$\Delta _{(f_i,f_{i+1})}$ has rank three 
if $i\geq 2$. 
In this case, 
$\exp h\Delta _{(f_i,f_{i+1})}$ 
is wild for each $h\in \ker \Delta _{(f_i,f_{i+1})}\sm \zs $ 
by Theorem~\ref{thm:lsc1} (iii).

Next, 
let $\lambda (y)=y^l$ and $\mu (y,z)=-z^{m}$ 
for $l\geq 1$ and $m\geq 1$. 
Then, 
we have $r_2=f_2^lx_2+x_1^{m}$. 
Since $\tilde{\eta }_2(y,z)=y+z^2$ 
and $a_2=t_2a_1-a_0=2$ by (\ref{eq:a_i zenkasiki}), 
it follows that 
\begin{align*}
&f_2\tilde{f}_3=\tilde{\eta }_2\left(
f_2,r_2\lambda (f_2)^{-1}
\right) \lambda (f_2)^{a_2}
=\left(f_2+
(r_2f_2^{-l})^2
\right)
f_2^{2l}=f_2^{2l+1}+r_2^2 \\
&\quad =f_2^{2l}(x_1x_3-x_2^2)+(f_2^lx_2+x_1^{m})^2 
=x_1(f_2^{2l}x_3+2f_2^lx_1^{m-1}x_2+x_1^{2m-1}). 
\end{align*}
Hence, we get 
$$
\tilde{f}_3
=f_2^{2l}x_3+2f_2^lx_1^{m-1}x_2+x_1^{2m-1}. 
$$
This shows that, 
if $m=2l+1$, then 
$-\tilde{D}_2=\Delta _{(f_2,\tilde{f}_3)}$ 
is the same as the 
rank three homogeneous locally nilpotent derivation 
of ``type $(2,4l+1)$" due to Freudenburg, 
where we regard $x_1$, $x_2$ and $x_3$ as 
$x$, $y$ and $z$, respectively. 
If $l\geq 1$ and $m\geq 1$, 
then $\exp h\Delta _{(f_2,\tilde{f}_3)}$ 
is wild for each $h\in \ker \Delta _{(f_2,\tilde{f}_3)}\sm \zs $ 
by Theorem~\ref{thm:lsc2} (iii), 
since $\lambda (y)=y^l$ does not belong to $k^{\times }$. 
If $m\geq 2$, 
then we have $\rank \Delta _{(f_2,\tilde{f}_3)}=3$ 
by Theorem~\ref{prop:rank2} (iii), 
since 
$t_0=3$ and $\mu _m(y)=-1$ does not belong to 
$\sqrt{(\lambda (y))}=(y)$. 

\bigskip 

\noindent
Note: Recently, 
Prof.\ Freudenburg kindly informed the author 
that he independently realized that 
Kara\'s-Zygad\l o~\cite[Theorem 2.1]{KZ} 
implies the wildness of $\exp D$ 
for his homogeneous locally nilpotent derivation of type 
$(2,5)$ as follows: 
In this case, 
it holds that 
$$
\deg {(\exp D)(x_1)}=9,\quad 
\deg {(\exp D)(x_2)}=25,\quad 
\deg {(\exp D)(x_3)}=41. 
$$
On the other hand, 
the above-mentioned result of Kara\'s-Zygad\l o 
implies that $\phi \in \Aut _{\C }\C [x_1,x_2,x_3]$ 
is wild if the following conditions hold: 

\noindent{\rm (i)} 
$\deg \phi (x_3)\geq \deg \phi (x_2)>\deg \phi (x_1)\geq 3$. 

\noindent{\rm (ii)} $\deg \phi (x_1)$ and $\deg \phi (x_2)$ 
are mutually prime odd numbers. 

\noindent{\rm (iii)} $\deg \phi (x_3)$ does not belong to 
$\Zn \deg \phi (x_1)+\Zn \deg \phi (x_2)$. 

It is easy to check that $\exp D$ satisfies (i), (ii) and (iii). 
The author would like to thank Prof.\ Freudenburg 
for informing him the remark.

\section{A reduction lemma}
\setcounter{equation}{0}
\label{sect:lscpf1}

Recall that we defined elements 
$\tau _2$ and $\tau _2'$ of $\T (k[x_3],\{ x_1,x_2\} )$ 
when $t_0=2$ and $t_1=2$, respectively. 
Let $\kxr $ be the field 
of fractions of $\kx $. 
Then, 
$\tau _2$ and $\tau _2'$ uniquely extend to 
automorphisms of $\kxr $. 
By abuse of notation, 
we denote them 
by the same symbols $\tau _2$ and $\tau _2'$. 
The purpose of this section 
is to prove the following lemma, 
and Theorem~\ref{thm:lsc1} (ii) 
on the assumption that $f_i$ and $f_{i+1}$ 
belong to $\kx $ for each $i\in I$.

\begin{lem}\label{lem:t_0=2}
\noindent{\rm (i)} 
If $t_0=2$, 
then we have $\tau _2(f_i')=f_{i+1}$ 
for each $i\geq 0$. 

\noindent{\rm (ii)} 
If $t_1=2$, 
then we have $\tau _2'(f_i)=f_{i+1}'$ for each $i\geq 0$.

\noindent{\rm (iii)} 
If $t_0=t_1=2$, then we have 
$(\tau _2\circ \tau _2')^{j}(f_i)=f_{i+2j}$ 
for each $i,j\geq 0$. 
\end{lem}

From the data (\ref{eq:data}), 
we can similarly define 
$(b_i)_{i=0}^{\infty }$, 
$r$, $(\eta _i(y,z))_{i=0}^{\infty }$, 
and $(a_i)_{i=0}^{\infty }$ 
by interchanging the roles of 
$t_0$ and $t_1$, 
and $(\alpha _j^0)_{j=1}^{t_0-1}$ 
and $(\alpha _j^1)_{j=1}^{t_1-1}$, 
respectively. 
To distinguish them from the original ones, 
we denote them by 
$$
(b_i')_{i=0}^{\infty },\quad 
r',\quad 
(\eta _i'(y,z))_{i=0}^{\infty } 
\quad \text{ and }\quad (a_i')_{i=0}^{\infty }, 
$$ 
respectively. 
We note that $b_0'=b_1'=0$ 
and $b_{i+1}'=t_{i+1}b_i'-b_{i-1}'+\xi _i$ 
for each $i\geq 1$.

In the rest of this section, 
we always assume that $t_0=2$.

\begin{lem}\label{lem:reduction1}
For each $i\geq 0$, 
we have 
$$
b_i'=b_{i+1}+\frac{\xi _{i+1}-\xi _i}{2}=
\left\{ 
\begin{array}{cl}
b_{i+1}-1 & \text{ if }i\equiv 1\pmod{4} \\
b_{i+1}+1 & \text{ if }i\equiv 3\pmod{4} \\
b_{i+1} & \text{ otherwise. }
\end{array}
\right. 
$$
\end{lem}
\begin{proof}
We prove the first equality by induction on $i$. 
Since $b_0'=b_1'=0$, and 
$$
b_1+\frac{\xi _1-\xi _0}{2}=
\frac{1-1}{2}=0\quad \text{and}\quad 
b_2+\frac{\xi _2-\xi _1}{2}=
1+\frac{(-1)-1}{2}=0, 
$$
the equality holds for $i=0,1$. 
Assume that $i\geq 2$. 
Then, 
we have 
$b_j'=b_{j+1}+(\xi _{j+1}-\xi _j)/2$ 
for $j=i-2,i-1$ by induction assumption. 
Hence, we get 
\begin{align*}
b_{i}'=t_{i}b_{i-1}'-b_{i-2}'+\xi _{i-1} 
=t_{i}\left( 
b_{i}+\frac{\xi _{i}-\xi _{i-1}}{2}
\right)
-\left( 
b_{i-1}+\frac{\xi _{i-1}-\xi _{i-2}}{2}
\right) +\xi _{i-1}. 
\end{align*}
Since 
$t_{i}b_{i}-b_{i-1}=b_{i+1}-\xi _{i}$ 
and $-(\xi _{i-1}-\xi _{i-2})=\xi _{i+1}-\xi _i$, 
the right-hand side is equal to 
\begin{align*}
b_{i+1}-\xi _i+t_{i}\frac{\xi _{i}-\xi _{i-1}}{2}
+\frac{\xi _{i+1}-\xi _{i}}{2}
+\xi _{i-1}
=b_{i+1}+(t_{i}-2)\frac{\xi _{i}-\xi _{i-1}}{2}
+\frac{\xi _{i+1}-\xi _{i}}{2}. 
\end{align*}
Since $t_{i}=2$ if $i$ is an even number, 
and $\xi _{i}-\xi _{i-1}=0$ otherwise, 
the right-hand side is equal to 
$b_{i+1}+(\xi _{i+1}-\xi _i)/2$. 
This proves that the first equality holds for every $i\geq 0$. 
The second equality is readily verified. 
\end{proof}

Since $t_0=2$ by assumption, 
we know by Lemma~\ref{lem:reduction1} 
that $\eta _i'(y,z)$ is equal to 
\begin{align*}
&z^{t_{1}b_{i}'+1}+\sum _{j=1}^{t_{1}}
\alpha _j^{1}y^jz^{(t_{1}-j)b_{i}'}
=z^{t_1b_{i+1}+1}+\sum _{j=1}^{t_1}
\alpha _j^1y^jz^{(t_1-j)b_{i+1}}
&\text{if}\quad & 
i\equiv 0\pmod{4}\\
&z^{2b_i'+1}+\alpha _1^0yz^{b_i'}+y^2
=y^{2}+\alpha _1^0yz^{b_{i+1}-1}+z^{2b_{i+1}-1}
&\text{if}\quad & 
i\equiv 1\pmod{4}\\
&y^{t_{1}}+\sum _{j=1}^{t_{1}}
\alpha _j^{1}y^{t_{1}-j}z^{jb_{i}'-1}
=y^{t_1}+\sum _{j=1}^{t_1}
\alpha _j^1y^{t_1-j}z^{jb_{i+1}-1}
&\text{if}\quad & 
i\equiv 2\pmod{4}\\
&y^2+\alpha _1^0yz^{b_i'-1}+z^{2b_i'-1}
=z^{2b_{i+1}+1}+\alpha _1^0yz^{b_{i+1}}+y^{2}
&\text{if}\quad & 
i\equiv 3\pmod{4} 
\end{align*}
for $i\geq 0$. 
In each case, 
the right-hand side is equal to 
$\eta _{i+1}(y,z)$. 
Hence, we get 
$$
\eta _i'(y,z)=\eta _{i+1}(y,z)
$$ 
for $i\geq 0$. 
Since $a_i'=\deg _z\eta '_i(y,z)$ 
and $a_{i+1}=\deg _z\eta _{i+1}(y,z)$, 
this implies that 
$a_i'=a_{i+1}$ for each $i\geq 0$.

Now, 
let us prove Lemma~\ref{lem:t_0=2}. 
First, 
we prove (i) by induction on $i$. 
Since 
$f_0'=x_2$ and $f_1'=x_1$, 
we get 
$$
\tau _2(f_0')=\tau _2(x_2)=x_1=f_1
\quad \text{and}\quad 
\tau _2(f_1')=\tau _2(x_1)=f_2 
$$ 
by (\ref{eq:tau_2}). 
Hence, 
the statement holds for $i=0,1$. 
Assume that $i\geq 2$. 
Since $t_0=2$, 
we see that 
$$
r'=x_1x_2x_3-\sum _{i=1}^{t_1}
\alpha _i^1x_2^i-
\sum _{j=1}^{2}\alpha _j^0x_1^j
=x_1(x_2x_3-\alpha _1^0-x_1)
-\sum _{i=1}^{t_1}
\alpha _i^1x_2^i. 
$$
Hence, we get 
$$
\tau _2(r')
=f_2(x_1x_3-\alpha _1^0-f_2)-\sum _{i=1}^{t_1}
\alpha _i^1x_1^i 
=f_2x_2-\sum _{i=1}^{t_1}\alpha _i^1x_1^i=r
$$
in view of (\ref{eq:f_2}). 
Since $\tau _2(f_{j}')=f_{j+1}$  
for $j=i-2,i-1$ 
by induction assumption, 
and $\eta _{i-1}'(y,z)=\eta _i(y,z)$ as shown above, 
it follows that 
$$
\tau _2(f_{i}')
=\tau _2\left( 
\eta _{i-1}'(f_{i-1}',r')(f_{i-2}')^{-1}
\right) 
=\eta _{i}(f_{i},r)f_{i-1}^{-1}
=f_{i+1}. 
$$
This proves (i). 
We can prove (ii) similarly. 
To prove (iii), 
assume that $t_0=t_1=2$. 
Then, 
we have $\tau _2(f_i')=f_{i+1}$ and 
$\tau _2'(f_i)=f_{i+1}'$ for each $i\geq 0$ 
by (i) and (ii). 
Hence, 
know that 
$$
(\tau _2\circ \tau _2')(f_i)=
\tau _2(f_{i+1}')=f_{i+2}. 
$$ 
Therefore, 
we get $(\tau _2\circ \tau _2')^{j}(f_i)=f_{i+2j}$ 
for each $i,j\geq 0$. 
This proves (iii), 
and thus completes the proof of Lemma~\ref{lem:t_0=2}.

For $g_1,g_2,g_3\in \kx $, 
define an endomorphism $\psi $ of the $k$-algebra $\kx $ 
by $\psi (x_i)=g_i$ for $i=1,2,3$. 
Then, 
we have $\Delta _{(g_1,g_2)}(g_3)=\det J\psi $. 
From this, 
it follows that 
\begin{equation}\label{eq:Jacobian}
\phi ^{-1}\circ \Delta _{(\phi (g_1),\phi (g_2))}\circ \phi 
=(\det J\phi )\Delta _{(g_1,g_2)}. 
\end{equation}
Actually, 
since 
$J(\phi \circ \psi )=\phi (J\psi )\cdot J\phi $, 
and $\det J\phi $ is a constant, 
we have 
\begin{align*}
&(\phi ^{-1}\circ \Delta _{(\phi (g_1),\phi (g_2))}\circ \phi )(g_3)
=\phi ^{-1}(\det J(\phi \circ \psi ))
=\phi ^{-1}\bigl(
\det \bigl(\phi (J\psi )\cdot J\phi \bigr)\bigr) 
\\ &\quad 
=\phi ^{-1}(\phi (\det J\psi )\det J\phi ) 
=(\det J\psi )\det J\phi 
=(\det J\phi )\Delta _{(g_1,g_2)}(g_3). 
\end{align*}

Now, 
we prove Theorem~\ref{thm:lsc1} (ii) 
on the assumption that $f_i$ and $f_{i+1}$ 
belong to $\kx $ for each $i\in I$. 
Assume that $t_0=2$, 
and take any $i\in I$. 
Then, 
it follows that $f_{j-1}'=\tau _2^{-1}(f_j)$ 
belongs to $\kx $ for $j=i,i+1$ 
by Lemma~\ref{lem:t_0=2} (i), 
since $i\geq 1$. 
Hence, we may consider $D_{i-1}'$. 
By definition, 
we see that $\det J\tau _2=1$. 
Thus, we know by (\ref{eq:Jacobian}) that 
\begin{equation}\label{eq:jacob}
\tau _2^{-1}\circ D_i\circ \tau _2
=\tau _2^{-1}\circ 
\Delta _{(\tau _2(f_{i}'),\tau _2(f_{i-1}'))}
\circ \tau _2=(\det J\tau _2)\Delta _{(f_{i}',f_{i-1}')}
=D_{i-1}'. 
\end{equation}
This proves the first part of Theorem~\ref{thm:lsc1} (ii). 
Since $D_1'$ is triangular, 
and $\tau _2$ is tame, 
the second part follows immediately.

To prove the last part, 
assume that $t_0=t_1=2$, 
and define $\tau $ as in the theorem. 
Then, 
we know by Lemma~\ref{lem:t_0=2} (iii) that 
$$
\tau (x_2)=(\tau _2\circ \tau _2')^{i/2}(f_0)=f_i,\quad 
\tau (x_1)=(\tau _2\circ \tau _2')^{i/2}(f_1)=f_{i+1}
$$
if $i$ is an even number. 
Since $\tau _2(f_0')=f_1$ and $\tau _2(f_1')=f_2$ 
by Lemma~\ref{lem:t_0=2} (i), 
we have 
$$
\tau (x_2)=
((\tau _2\circ \tau _2')^{(i-1)/2}\circ \tau _2)(f_0')=f_i,\quad 
\tau (x_1)=
((\tau _2\circ \tau _2')^{(i-1)/2}\circ \tau _2)(f_1')=f_{i+1}
$$
if $i$ is an odd number. 
Hence, 
$\tau (x_2)=f_i$ and $\tau (x_1)=f_{i+1}$ 
hold in either case. 
Because $\det J\tau _2=\det J\tau _2'=1$, 
we get $\det J\tau =1$. 
Thus, we conclude that 
$$
\tau ^{-1}\circ D_i\circ \tau 
=\tau ^{-1}\circ 
\Delta _{(\tau (x_1),\tau (x_2))}
\circ \tau =(\det J\tau )\Delta _{(x_1,x_2)}=D_{0}
$$
by the formula (\ref{eq:Jacobian}). 
This proves the last part of Theorem~\ref{thm:lsc1} (ii), 
and thereby completing the proof of Theorem~\ref{thm:lsc1} (ii).

\section{Properties of $(a_i)_{i=0}^{\infty }$}
\setcounter{equation}{0}
\label{sect:sequence}

Let us prove (\ref{eq:I}). 
By (\ref{eq:a_i zenkasiki}), 
we can check the first three cases immediately. 
Actually, we have 
$a_1=1$ and $a_2=0$ if $t_0=1$, 
$a_1=a_2=1$ and $a_3=0$ if $(t_0,t_1)=(2,1)$, 
and $a_1=a_3=a_4=1$, $a_2=2$ and $a_5=0$ if $(t_0,t_1)=(3,1)$. 
If $t_0=t_1=2$, 
then we have $a_{i+1}-a_i=a_i-a_{i-1}$ for each $i\geq 1$ 
by (\ref{eq:a_i zenkasiki}). 
Since $a_0=a_1=1$, 
it follows that $a_i=1$ for every $i\in \N $. 
Hence, we get $I=\N $.

By (ii) of the following lemma, 
we have $I=\N $ 
if $t_0\geq 3$ and $(t_0,t_1)\neq (3,1)$. 
Then, 
it follows that $I=\N $ 
if $t_0=2$ and $t_1\geq 3$, 
since $a_{i+1}=a_i'$ for each $i\geq 0$ 
as mentioned after Lemma~\ref{lem:reduction1}. 
Thus, 
(\ref{eq:I}) is proved in all cases.

\begin{lem}\label{lem:a_i}
If $t_0\geq 3$ and $(t_0,t_1)\neq (3,1)$, 
then the following assertions hold$:$ 

\noindent{\rm (i)} 
For each $i\geq 1$, 
we have $a_{i+2}-a_i\geq 2$ 
if $i$ is an even number, 
and $a_{i+2}-a_i\geq 1$ otherwise.

\noindent{\rm (ii)} 
For each $i\geq 2$, 
we have $a_i\geq 2$ and $b_i\geq 1$.

\noindent{\rm (iii)} 
For each $i\geq 3$, we have $a_i>t_i$.

\noindent{\rm (iv)} 
For each $i\geq 3$, 
we have $b_i\geq 1$ if $i\equiv 0,1\pmod{4}$, 
and $b_i\geq 2$ otherwise.

\noindent{\rm (v)} 
For each $i\geq 2$ and $l\in \N $, 
we have $la_i\neq t_i$. 
\end{lem}

To prove Lemma~\ref{lem:a_i}, 
we use the following lemma.

\begin{lem}\label{lem:skip}
Let $\gamma _{i-2},\ldots ,\gamma _{i+2}$ 
be five elements of 
an abelian group for some $i\in \Z $. 
Assume that there exists $l\in \Z $ such that 
$\gamma _{j+1}=t_{j+l}\gamma _j-\gamma _{j-1}$ 
for $j=i-1,i,i+1$. 
Then, 
we have 
$\gamma _{i+2}-\gamma _i
=(t_0t_1-4)\gamma _i+(\gamma _i-\gamma _{i-2})$. 
\end{lem}
\begin{proof}
Since $t_{j+l}\gamma _j=\gamma _{j+1}+\gamma _{j-1}$ 
for $j=i-1,i,i+1$, 
we get 
\begin{align*}
&t_{i+l-1}t_{i+l}\gamma _i
=t_{i+l-1}(\gamma _{i+1}+\gamma _{i-1}) \\
&\quad =t_{(i+1)+l}\gamma _{i+1}+t_{(i-1)+l}\gamma _{i-1}
=(\gamma _{i+2}+\gamma _{i})+(\gamma _i+\gamma _{i-2}). 
\end{align*}
This gives that 
$\gamma _{i+2}-\gamma _i
=(t_0t_1-4)\gamma _i+(\gamma _i-\gamma _{i-2})$. 
\end{proof}

Now, let us prove Lemma~\ref{lem:a_i}. 
To prove (i), 
we demonstrate that 
$$
a_{2i-1}\geq 0,\quad 
a_{2i}\geq 0,\quad 
a_{2i+1}-a_{2i-1}\geq 1,\quad 
a_{2i+2}-a_{2i}\geq 2
$$
for each $i\geq 1$ 
by induction on $i$. 
Since $t_0\geq 3$ and $(t_0,t_1)\neq (3,1)$ by assumption, 
we have $t_0\geq 4$ and $t_1=1$, 
or $t_0\geq 3$ and $t_1\geq 2$. 
Assume that $i=1$. 
Then, 
we have $a_1=1$ and $a_2=t_0-1\geq 2$ 
by the definition of $a_i$. 
From (\ref{eq:a_i zenkasiki}), 
it follows that 
\begin{align*}
a_3-a_1=t_3a_2-2a_1
&=t_1(t_0-1)-2\geq 1,\\
a_4-a_2=t_4a_3-2a_2
&=t_0(t_1(t_0-1)-1)-2(t_0-1) \\
&=t_0(t_1(t_0-1)-3)+2\geq 2. 
\end{align*}
Thus, 
the statement holds when $i=1$. 
Assume that $i\geq 2$. 
Then, 
we have 
$$
a_{2i-3}\geq 0,\quad 
a_{2i-2}\geq 0,\quad 
a_{2i-1}-a_{2i-3}\geq 1,\quad 
a_{2i}-a_{2i-2}\geq 2 
$$ 
by induction assumption. 
This implies that 
$a_{2i-1}\geq 0$ and $a_{2i}\geq 0$. 
By (\ref{eq:a_i zenkasiki}), 
we know that 
$a_{j+1}=t_{j+1}a_j-a_{j-1}$ 
for $j=2i-2,2i-1,2i$, 
since $j\geq 2i-2\geq 2$. 
Hence, we get 
$$
a_{2i+1}-a_{2i-1}
=(t_0t_1-4)a_{2i-1}+(a_{2i-1}-a_{2i-3})
$$
by Lemma~\ref{lem:skip}. 
Since 
$t_0\geq 4$ and $t_1=1$, 
or $t_0\geq 3$ and $t_1\geq 2$, 
we have $t_0t_1\geq 4$. 
Because $a_{2i-1}\geq 0$, 
we conclude that 
$a_{2i+1}-a_{2i-1}\geq a_{2i-1}-a_{2i-3}\geq 1$. 
Similarly, 
since $a_{j+1}=t_{i+1}a_j-a_{j-1}$ 
for $j=2i-1,2i,2i+1$ by (\ref{eq:a_i zenkasiki}), 
and since $a_{2i}\geq 0$ and $a_{2i}-a_{2i-2}\geq 2$,
we get 
$$
a_{2i+2}-a_{2i}
=(t_0t_1-4)a_{2i}+(a_{2i}-a_{2i-2})
\geq a_{2i}-a_{2i-2}\geq 2 
$$ 
by Lemma~\ref{lem:skip}. 
Therefore, 
the statement holds for every $i\geq 1$. 
This proves (i).

By (i), 
we know that $a_i$ is at least $a_2$ or $a_3$ 
for each $i\geq 2$. 
Since $a_2=t_0-1\geq 2$ 
and $a_3=t_1(t_0-1)-1\geq 2$, 
it follows that $a_i\geq 2$ for each $i\geq 2$. 
This implies that $b_i\geq 1$ for each $i\geq 2$, 
for otherwise $a_i=t_{i+1}b_i+\xi _i\leq 1$. 
Therefore, we get (ii).

To prove (iii), 
it suffices to check that $a_i>t_i$ for $i=3,4$, 
since $a_{i+2j}\geq a_i$ by (i) 
and $t_i=t_{i+2j}$ for each $j\geq 0$. 
Since $t_0\geq 4$ and $t_1=1$, 
or $t_0\geq 3$ and $t_1\geq 2$, 
we see that 
$a_3=t_1(t_0-1)-1$ is greater than $t_1=t_3$. 
By (i), 
we have $a_4\geq a_2+2=t_0+1>t_4$. 
Therefore, 
we get (iii).

By (iii), 
it follows that 
$t_ib_i+\xi _i=a_i\geq t_i+1$ for each $i\geq 3$. 
Hence, 
we have $b_i\geq 1+(1-\xi _i)t_i^{-1}$. 
Since $\xi _i=1$ if $i\equiv 0,1\pmod{4}$, 
and $\xi _i=-1$ otherwise, 
we get $b_i\geq 1$ if $i\equiv 0,1\pmod{4}$, 
and $b_i>1$ otherwise. 
This proves (iv).

By (iii), 
we have $la_i>t_i$ for any $l\in \N $ if $i\geq 3$. 
Hence, 
(v) holds when $i\geq 3$. 
Suppose that 
$la_2=t_2$ for some $l\in \N $. 
Then, we have $l(t_0-1)=t_0$. 
Hence, 
$t_0/(t_0-1)=l$ must be an integer. 
This contradicts that $t_0\geq 3$. 
Therefore, we get (v). 
This completes the proof of Lemma~\ref{lem:a_i}.

\section{Theory of local slice construction}
\setcounter{equation}{0}
\label{sect:lsc}

The main part of Theorem~\ref{thm:lsc1} (i), 
and Theorem~\ref{thm:lsc2} (i) 
are proved by means of {\it local slice construction} 
due to Freudenburg~\cite{Flsc}. 
In this section, 
we briefly review 
basic facts about locally nilpotent derivations, 
and summarize the theory of local slice construction.

Let $f,g\in \kx $ be such that 
$\ker D=k[f,g]$ for some $D\in \lnd _k\kx $. 
Then, 
$\ker D$ has transcendence degree two over $k$ 
(cf.~\cite[1.4]{Miyanishi}). 
Hence, 
$f$ and $g$ are algebraically independent over $k$. 
Since $k[f,g]$ is the polynomial ring in $f$ and $g$ over $k$, 
it is clear that $f$ is an irreducible element of $k[f,g]$, 
and $k[f]$ is factorially closed in $k[f,g]$, 
i.e., 
if $pq$ belongs to $k[f]$ for $p,q\in k[f,g]\sm \zs $, 
then $p$ and $q$ belong to $k[f]$. 
Recall that 
$\ker D$ is {\it factorially closed} in $\kx $ 
(cf.~\cite[1.3.1]{Miyanishi}). 
Thus, 
it follows that 
$f$ is an irreducible element of $\kx $, 
and $k[f]$ is factorially closed in $\kx $. 
In particular, 
$f\kx $ is a prime ideal of $\kx $. 
Since $g\not\approx f$, 
and $g$ is also an irreducible element of $\kx $, 
we know that $g$ does not belong to $f\kx $. 
Here, 
$p\approx q$ 
(resp.\ $p\not\approx q$) denotes that 
$p$ and $q$ are linearly dependent 
(resp.\ linearly independent) over $k$ 
for $p,q\in \kx $.

We summarize these facts 
in the following lemma.

\begin{lem}\label{lem:facts}
Let $f,g\in \kx $ be such that 
$\ker D=k[f,g]$ for some $D\in \lnd _k\kx $. 
Then, $f$ and $g$ are algebraically independent over $k$, 
$f$ is an irreducible element of $\kx $, 
$k[f]$ and $k[f,g]$ are factorially closed in $\kx $, 
and $f\kx $ is a prime ideal of $\kx $ 
to which $g$ does not belong. 
\end{lem}

In the situation of the lemma above, 
assume that $K$ is an extension field of $k$. 
Then, 
$D$ naturally extends to 
a locally nilpotent derivation 
$\bar{D}:=\id _K\otimes D$ of $\Kx :=K\otimes _k\kx $. 
Since $K$ is flat over $k$, 
we have $\ker \bar{D}=K\otimes _k\ker D=K[f,g]$. 
Hence, 
$f=1\otimes f$ is an irreducible element of $\Kx $ 
(see also~\cite[Corollary 1.7]{Daigleh} 
for a stronger statement). 
Similarly, 
$f+\alpha $ is an irreducible element of $\Kx $ 
for each $\alpha \in K$, 
since $K[f+\alpha ,g]=K[f,g]$.

\begin{lem}\label{lem:minimal}
\noindent{\rm (i)} 
Assume that $f,g\in \kx $ are algebraically 
independent over $k$. 
If $\ker D=k[f,g]$ holds for some $D\in \Der _k\kx $, 
then we have $k[f]\cap g\kx =\zs $. 

\noindent{\rm (ii)} 
If $f,s\in \kx \sm k$ are such that 
$D(f)=0$ and $D(s)\neq 0$ 
for some $D\in \Der _k\kx $, 
then $f$ and $s$ are algebraically independent over $k$. 

\noindent{\rm (iii)} 
Let $f,g,s\in \kx $ be such that 
$\ker D=k[f,g]$ and $D(s)\neq 0$ 
for some $D\in \lnd _k\kx $. 
Then, 
$P:=k[f,s]\cap g\kx $ is a principal prime ideal of $k[f,s]$. 
If $\eta (y,z)$ is an irreducible element of $k[y,z]$ 
such that $\eta (f,s)$ belongs to $g\kx $, 
then $P$ is generated by $\eta (f,s)$. 
\end{lem}
\begin{proof}
(i) Suppose to the contrary that $k[f]\cap g\kx \neq \zs $. 
Then, 
we may find 
$\lambda (y)\in k[y]\sm \zs $ and $g'\in \kx \sm \zs $ 
such that $\lambda (f)=gg'$. 
Since $D(f)=D(g)=0$, 
it follows that $gD(g')=D(\lambda (f))=0$. 
Hence, we get $D(g')=0$. 
Thus, 
$g'$ belongs to $\ker D$. 
Since $\ker D=k[f,g]$, 
we may write $g'=\mu (f,g)$, 
where $\mu (y,z)\in k[y,z]\sm \zs $. 
Then, 
we have $\lambda (f)=gg'=g\mu (f,g)$. 
Since $\lambda (y)$ and $\mu (y,z)$ are nonzero, 
this contradicts that $f$ and $g$ 
are algebraically independent over $k$. 
Therefore, 
we get $k[f]\cap g\kx =\zs $.

(ii)
Since $k$ is of characteristic zero, 
$k(f,s)$ is a separable extension of $k(f)$. 
Since $D(f)=0$ and $D(s)\neq 0$, 
it follows that $k(f,s)$ 
is not a finite extension of $k(f)$ 
(cf.~\cite[Proposition 5.2]{Lang}). 
Hence, 
$s$ is transcendental over $k(f)$. 
Since $f$ does not belong to $k$ by assumption, 
$k(f)$ is a transcendence extension of $k$. 
Therefore, 
$f$ and $s$ are algebraically independent over $k$.

(iii) 
Since $D$ is locally nilpotent 
and $\ker D=k[f,g]$ by assumption, 
$f$ and $g$ are algebraically independent over $k$ 
by Lemma~\ref{lem:facts}. 
Hence, we get $k[f]\cap g\kx =\zs $ by (i). 
By Lemma~\ref{lem:facts}, 
$g\kx $ is a prime ideal of $\kx $ 
to which $f$ does not belong. 
Hence, $P=k[f,s]\cap g\kx $ 
is a prime ideal of $k[f,s]$. 
Let $\bar{f}$ be the image of $f$ 
in the $k$-domain $\kx /g\kx $. 
Then, 
$\bar{f}$ is transcendental over $k$, 
since $k[f]\cap g\kx =\zs $. 
Hence, 
$k[f,s]/P$ has transcendence degree 
at least one over $k$. 
Accordingly, 
$P$ is of height at most one. 
Since $k[f,s]$ is the polynomial ring 
in $f$ and $s$ over $k$ by (ii), 
we see that $k[f,s]$ is a noetherian UFD. 
Therefore, 
$P$ is a principal ideal of $k[f,s]$ 
(cf.~\cite[Theorem 20.1]{Matsumura}).

Assume that $\eta (y,z)$ is 
an irreducible element of $k[y,z]$ 
for which $q:=\eta (f,s)$ belongs to $g\kx $. 
Then, 
$q$ belongs to $P$, 
and is an irreducible element of $k[f,s]$. 
Since $P$ is a principal prime ideal of $k[f,s]$, 
it follows that $P$ is generated by $q$. 
\end{proof}

Now, 
we briefly 
summarize the theory of local slice construction. 
Assume that $D\in \lnd _k\kx $ is irreducible 
and satisfies the following conditions:  

\medskip 

\noindent(LSC1) There exist $f,g\in \kx $ such that 
$D=\Delta _{(f,g)}$ and $\ker D=k[f,g]$; 

\noindent(LSC2) There exist $s\in \kx \sm g\kx $ 
and $F\in k[f]\sm \zs $ such that $D(s)=gF$.

\medskip

Since $\ker D=k[f,g]$ by (LSC1), 
we know that $f$ and $g$ are algebraically independent over $k$ 
by Lemma~\ref{lem:facts}. 
In particular, 
we have $g\neq 0$. 
Hence, we get $D(s)=gF\neq 0$ due to (LSC2). 
Thus, 
$f$ and $s$ are algebraically independent over $k$ 
by Lemma~\ref{lem:minimal} (ii), 
and $P:=k[f,s]\cap g\kx $ 
is a principal prime ideal of $k[f,s]$ 
by Lemma~\ref{lem:minimal} (iii).

We show that $P$ is not the zero ideal. 
Since $D(g)=0$, 
we see that $D$ induces a derivation $\bar{D}$ of 
$R:=\kx /g\kx $ over $k$. 
Then, 
$\bar{D}$ is nonzero, 
since $D(\kx )$ is not contained in $g\kx $ 
by the irreducibility of $D$. 
Because $g\kx $ is a prime ideal of $\kx $ 
by Lemma~\ref{lem:facts}, 
and is of height one, 
we know that $R$ is a domain 
having transcendence degree two over $k$. 
Accordingly, 
$\ker \bar{D}$ has transcendence degree 
at most one over $k$. 
On the other hand, 
the images 
$\bar{f}$ and $\bar{s}$ of $f$ and $s$ in $R$ 
belong to $\ker \bar{D}$, 
since $D(f)=0$ and $D(s)=gF$. 
Thus, 
$\bar{f}$ and $\bar{s}$ are algebraically dependent over $k$. 
Consequently, 
some element of $k[f,s]\sm \zs $ belongs to $g\kx $. 
Therefore, $P$ is not the zero ideal.

Since $P$ is a principal prime ideal of $k[f,s]$, 
we may find an irreducible element $q$ of $k[f,s]$ 
such that $P=qk[f,s]$. 
Then, $h:=g^{-1}q$ belongs to $\kx $. 
We note that $q$ does not belong to $k[f]$. 
In fact, 
$q=gh$ belongs to $g\kx \sm \zs $, 
and $k[f]\cap g\kx =\zs $ by Lemma~\ref{lem:minimal} (i).

The following theorem is due to 
Freudenburg~\cite{Flsc} 
(see also~\cite[Theorem 5.24]{Fbook}).

\begin{thm}[Freudenburg]\label{thm:fbook}
In the notation above, 
we have the following assertions$:$

\noindent{\rm (a)} 
$\Delta _{(f,h)}$ belongs to $\lnd _k\kx $.

\noindent{\rm (b)} 
$\Delta _{(f,h)}(s)=-hF$.

\noindent{\rm (c)} 
If $\Delta _{(f,h)}$ is irreducible, 
then $\ker \Delta _{(f,h)}=k[f,h]$. 
\end{thm}

The element 
$\Delta _{(f,h)}$ of $\lnd _k\kx $ is called 
a locally nilpotent derivation obtained by 
{\it local slice construction} from the data $(f,g,s)$.

\section{Irreducibility criteria}
\setcounter{equation}{0}

In the situation of Theorem~\ref{thm:fbook}, 
the following proposition is useful to check 
the irreducibility of $\Delta _{(f,h)}$.

\begin{prop}\label{prop:lsc:F:rem}
If $F$ and $\Delta _{(f,h)}(g_0)$ have no common factor 
for some $g_0\in \kx $, 
then $\Delta _{(f,h)}$ is irreducible. 
\end{prop}
\begin{proof}
Suppose to the contrary that 
$E:=\Delta _{(f,h)}$ is not irreducible. 
Then, 
there exists $p\in \kx \sm k$ such that 
$E(\kx )$ is contained in $p\kx $. 
Without loss of generality, 
we may assume that $p$ 
is an irreducible element of $\kx $. 
Then, $F$ does not belong to $p\kx $, 
since $F$ and $E(g_0)$ have no common factor by assumption, 
and $E(g_0)$ belongs to $p\kx $ by supposition. 
By Theorem~\ref{thm:fbook} (b), 
we have $-hF=E(s)$. 
Since $E(s)$ belongs to $p\kx $, 
it follows that $h$ belongs to $p\kx $. 
Hence, $q=gh$ belongs to $p\kx $. 
Recall that $q$ is an element of $k[f,s]$, 
and $f$ and $s$ are algebraically independent over $k$. 
Hence, 
we may consider the partial derivatives 
$\partial q/\partial f$ and $\partial q/\partial s$. 
Since $E(h)=E(f)=0$ and $E(s)=-hF$, 
we get 
$$
E(g)=
E(qh^{-1})
=E(q)h^{-1}
=\left(\frac{\partial q}{\partial f}E(f)
+\frac{\partial q}{\partial s}E(s)
\right)
h^{-1}
=-\frac{\partial q}{\partial s}F 
$$
by chain rule. 
Because $E(g)$ belongs to $p\kx $, 
and $F$ does not belong to $p\kx $ as mentioned, 
we conclude that 
$\partial q/\partial s$ belongs to $p\kx $.

Now, take $\phi (y,z)\in k[y,z]$ such that $\phi (f,s)=q$. 
Then, 
$\phi (y,z)$ is an irreducible element of $k[y,z]$ 
by the irreducibility of $q$ in $k[f,s]$. 
Since $q$ does not belong to $k[f]$, 
it follows that 
$\phi (y,z)$ does not belong to $k[y]$. 
Let $\bar{f}$ and $\bar{s}$ 
denote the images of $f$ and $s$ in $\kx /p\kx $, 
respectively. 
We show that $\bar{f}$ is algebraic over $k$. 
Define elements of the polynomial ring $(\kx /p\kx )[z]$ by 
$\psi (z):=\phi (\bar{f},z)$ and $\psi '(z):=d\psi (z)/dz$. 
Then, 
we have $\psi (\bar{s})=\psi '(\bar{s})=0$, 
since $\phi (f,s)=q$ and 
$(\partial \phi /\partial z)(f,s)=\partial q/\partial s$ 
belong to $p\kx$ as mentioned. 
Now, 
suppose to the contrary that $\bar{f}$ is transcendental over $k$. 
Then, 
$\bar{f}$ and $z$ are algebraically independent over $k$, 
since $z$ is an indeterminate over $\kx /p\kx $. 
Because $\phi (y,z)$ is an irreducible element of $k[y,z]$ 
not belonging to $k[y]$, 
it follows that 
$\psi (z)$ is an irreducible element of $k[\bar{f},z]$ 
not belonging to $k[\bar{f}]$. 
Consequently, 
$\psi (z)$ is an irreducible polynomial in 
$z$ over $k[\bar{f}]$, 
and hence over $k(\bar{f})$. 
Since $k$ is of characteristic zero, 
this contradicts that 
$\psi (\bar{s})=\psi '(\bar{s})=0$. 
Therefore, 
$\bar{f}$ is algebraic over $k$.

Let $\mu _1(y)$ be the minimal polynomial 
of $\bar{f}$ over $k$. 
Then, 
we have $\mu _1(\bar{f})=0$. 
Hence, 
$\mu _1(f)$ belongs to $p\kx $. 
On the other hand, 
$\mu _1(f)$ is an irreducible element of $k[f]$, 
since $f$ is not a constant. 
By Lemma~\ref{lem:facts}, 
$k[f]$ is factorially closed in $\kx $. 
Hence, it follows that 
$\mu _1(f)$ is an irreducible element of $\kx $. 
Therefore, 
we conclude that $\mu _1(f)\approx p$.

By definition, 
$\psi (z)$ belongs to $k[\bar{f}][z]$. 
We claim that $\psi (z)$ is nonzero. 
In fact, if $\psi (z)=0$, 
then $\phi (y,z)$ is divisible by $\mu _1(y)$. 
By the irreducibility of $\phi (y,z)$, 
it follows that 
$\phi (y,z)\approx \mu _1(y)$. 
Hence, $\phi (y,z)$ belongs to $k[y]$, 
a contradiction. 
Thus, 
$\psi (z)$ belongs to $k[\bar{f}][z]\sm \zs $. 
Since $\psi (\bar{s})=0$ 
and $\bar{f}$ is algebraic over $k$, 
this implies that $\bar{s}$ is algebraic over $k$. 
Let $\mu _2(z)$ be the minimal polynomial of $\bar{s}$ over $k$. 
Then, 
we have $\mu _2(s)=ph_0$ 
for some $h_0\in \kx $, 
while $\mu _2'(s)$ does not belong to $p\kx $, 
where $\mu _2'(z):=d\mu _2(z)/dz$. 
Since $p\approx \mu _1(f)$ is killed by $D$, 
and $D(s)=gF$ by (LSC2), 
we get 
$$
pD(h_0)=D(ph_0)=D(\mu _2(s))=\mu _2'(s)D(s)=\mu _2'(s)gF. 
$$
Because 
$\mu _2'(s)$ and $F$ do not belong to $p\kx $, 
this implies that $g$ belongs to $p\kx $. 
By the irreducibility of $g$, 
it follows that $g\approx p$. 
Since $p\approx \mu _1(f)$, 
we conclude that 
$f$ and $g$ 
are algebraically independent over $k$, 
a contradiction. 
Therefore, 
$E$ must be irreducible. 
\end{proof}

To prove the irreducibility of polynomials 
in two variables, 
we use the following lemma.

\begin{lem}\label{lem:x^a+y^b}
Let $q\in k[x_1,x_2]$ be such that 
$q^{\vv }=x_1^a+\alpha x_2^b$ 
for some $\vv \in \N ^2$, 
$a,b\in \N $ with $\gcd (a,b)=1$ 
and $\alpha \in k^{\times }$. 
Then, 
$q$ is an irreducible element of $k[x_1,x_2]$. 
\end{lem}
\begin{proof}
Since $\gcd (a,b)=1$ and $\alpha \neq 0$ by assumption, 
we see that $x_1^a+\alpha x_2^b$ 
is an irreducible element of $k[x_1,x_2]$. 
Suppose to the contrary that $q=q_1q_2$ 
for some $q_1,q_2\in k[x_1,x_2]\sm k$. 
Then, 
$q_1^{\vv }$ and $q_2^{\vv }$ 
belong to $k[x_1,x_2]\sm k$ 
by the choice of $\vv $. 
Since $x_1^a+\alpha x_2^b=q^{\vv }=q_1^{\vv }q_2^{\vv }$, 
this contradicts the irreducibility of 
$x_1^a+\alpha x_2^b$. 
Therefore, 
$q$ is an irreducible element of $k[x_1,x_2]$. 
\end{proof}

Let $\Gamma $ be a totally ordered additive group. 
Then, 
we have the following lemma.

\begin{lem}\label{lem:initial}
Let $\vv =(a,t)\in \Gamma ^2$ be such that 
$a>0$ or $t>0$. 
Then, 
for each $i\geq 0$, 
we have 
$$
\eta _i(x_1,x_2)^{\vv }=\left\{ 
\begin{array}{ccl}
x_1^{t_i} & \text{ if }&t_ia>a_it\\
x_1^{t_i}+x_2^{a_i} & \text{ if }&t_ia=a_it\\
x_2^{a_i} & \text{ if }&t_ia<a_it. 
\end{array}
\right. 
$$
\end{lem}
\begin{proof}
Assume that $t\leq 0$. 
Then, 
we have $a>0$ by assumption. 
Hence, we get $t_ia>0\geq a_it$. 
By definition, 
$\eta _i(x_1,x_2)$ is a monic polynomial 
in $x_1$ over $k[x_2]$ of degree $t_i\geq 1$. 
Since $a>0$ and $t\leq 0$, 
it follows that 
$\eta _i(x_1,x_2)^{\vv }=x_1^{t_i}$. 
Therefore, 
the assertion is true.

Assume that $t>0$. 
If $i\equiv 0,1\pmod{4}$, 
then we have $a_i=t_ib_i+1$. 
Hence, 
we know that 
\begin{align*}
&\deg _{\vv }x_1^jx_2^{(t_i-j)b_i}
=ja+(t_i-j)b_it
=ja+(t_i-j)\left(\frac{a_i-1}{t_i}\right)t \\
&\quad =a_it+\frac{t}{t_i}(j-t_i)
+\frac{j}{t_i}(t_ia-a_it)
\end{align*}
for $j=1,\ldots ,t_i$. 
If $t_ia\geq a_it$, 
then $\deg _{\vv }x_1^jx_2^{(t_i-j)b_i}$ 
has the maximum value 
$\deg _{\vv }x_1^{t_i}=t_ia$ when $j=t_i$. 
Hence, 
we have 
$$
\bigl( 
\eta _i(x_1,x_2)-x_2^{t_ib_i+1}
\bigr) ^{\vv }
=\left( 
\sum _{j=1}^{t_i}
\alpha _j^ix_1^jx_2^{(t_i-j)b_i}
\right) ^{\vv }
=x_1^{t_i}. 
$$
Since $x_2^{t_ib_i+1}=x_2^{a_i}$, 
it follows that 
$\eta _i(y,z)^{\vv }=x_1^{t_i}$ if $t_ia>a_it$, 
and $\eta _i(y,z)^{\vv }=x_1^{t_i}+x_2^{a_i}$ 
if $t_ia=a_it$. 
If $t_ia<a_it$, 
then 
$\deg _{\vv }x_1^jx_2^{(t_i-j)b_i}$ 
is less than 	$a_it=\deg _{\vv }x_2^{a_i}$ 
for $j=1,\ldots ,t_i$. 
Hence, 
we get $\eta _i(y,z)^{\vv }=x_2^{a_i}$. 
Therefore, 
the lemma is true when $i\equiv 0,1\pmod{4}$.

If $i\equiv 2,3\pmod{4}$, 
then we have $a_i=t_ib_i-1$. 
Hence, we know that 
\begin{align*}
&\deg _{\vv }x_1^{t_i-j}x_2^{jb_i-1}
=(t_i-j)a+(jb_i-1)t
=(t_i-j)a+j\frac{a_i+1}{t_i}t-t\\
&\quad =t_ia+\frac{t}{t_i}(j-t_i)
+\frac{j}{t_i}(a_it-t_ia) 
\end{align*}
for $j=1,\ldots ,t_i$. 
If $t_ia\leq a_it$, 
then $\deg _{\vv }x_1^{t_i-j}x_2^{jb_i-1}$ 
has the maximum value 
$\degv x_2^{t_ib_i-1}=a_it$
when $j=t_i$. 
This implies that 
$$
\bigl( 
\eta _i(x_1,x_2)-x_1^{t_i}
\bigr) ^{\vv }
=\left( 
\sum _{j=1}^{t_i}\alpha _j^ix_1^{t_i-j}x_2^{jb_i-1}
\right) ^{\vv }
=x_2^{t_ib_i-1}=x_2^{a_i}. 
$$
Hence, 
we have 
$\eta _i(y,z)^{\vv }=x_2^{a_i}$ if $t_ia<a_it$, 
and $\eta _i(y,z)^{\vv }=x_1^{t_i}+x_2^{a_i}$ 
if $t_ia=a_it$. 
If $t_ia>a_it$, 
then $\deg _{\vv }x_1^{t_i-j}x_2^{jb_i-1}$ 
is less than 
$t_ia=\degv x_1^{t_i}$ for $j=1,\ldots ,t_i$. 
Hence, we get 
$\eta _i(y,z)^{\vv }=x_1^{t_i}$. 
Therefore, 
the lemma is true when $i\equiv 2,3\pmod{4}$. 
\end{proof}

By Lemma~\ref{lem:initial}, 
we have $\eta _i(x_1,x_2)^{\vv _i}=x_1^{t_i}+x_2^{a_i}$ 
for $\vv _i:=(a_i,t_i)$ for each $i\geq 0$. 
Moreover, 
$t_i$ and $a_i=t_ib_i+\xi _i$ 
are mutually prime, 
since $\xi _i=1,-1$. 
Therefore, 
we conclude 
that $\eta _i(x_1,x_2)$ 
is an irreducible element of $k[x_1,x_2]$ 
by Lemma~\ref{lem:x^a+y^b}.

For each $i\geq 3$, 
we define 
$$
\tilde{h}_i=
\eta _i(x_1,x_2\lambda (x_1)^{-1})\lambda (x_1)^{a_i}. 
$$ 
Then, 
$\tilde{h}_i$ belongs to $k[x_1,x_2]$, 
since $\eta _i(x_1,x_2)$ 
is a monic polynomial in $x_2$ over $k[x_1]$ 
of degree $a_i$. 
We show that $\tilde{h}_i$ 
is an irreducible element of $k[x_1,x_2]$. 
Let $\tilde{\vv }_i=(a_{i},ua_i+t_{i})$, 
where $u:=\deg _y\lambda (y)$. 
Then, 
we have $\deg _{\tilde{\vv }_i}x_1=a_i=\deg _{\vv _i}x_1$ 
and 
$\deg _{\tilde{\vv }_i}x_2\lambda (x_1)^{-1}
=t_i=\deg _{\vv _i}x_2$. 
Since 
$\eta _i(x_1,x_2)^{\vv _i}=x_1^{t_i}+x_2^{a_i}$, 
we see that 
$$
\eta _i\left(
x_1,x_2\lambda (x_1)^{-1}
\right)^{\tilde{\vv }_i}
=x_1^{t_i}+\left( \left( x_2\lambda (x_1)^{-1}
\right) ^{\tilde{\vv }_i}\right)^{a_i} 
=x_1^{t_i}+\left(x_2(cx_1^u)^{-1}\right)^{a_i}, 
$$
where $c$ is the leading coefficient of $\lambda (y)$. 
Hence, we get 
$\tilde{h}_i^{\tilde{\vv }_i}=c^{a_i}x_1^{ua_i+t_i}+x_2^{a_{i}}$. 
Since $t_i$ and $a_i$ are mutually prime, 
it follows that $ua_i+t_i$ and $a_i$ are mutually prime. 
Therefore, 
$\tilde{h}_i$ is an irreducible element of $k[x_1,x_2]$ 
thanks to Lemma~\ref{lem:x^a+y^b}.

Since $\deg _z\theta (z)=t_0-1=a_2$, 
we see that 
$$
\tilde{h}_2:=
\tilde{\eta }_2\left(
x_1,x_2\lambda (x_1)^{-1}\right) 
\lambda (x_1)^{a_2}
=\Bigl( x_1+
\theta \bigl( x_2\lambda (x_1)^{-1}\bigr) 
\Bigr) 
\lambda (x_1)^{a_2}
$$
belongs to $k[x_1,x_2]$. 
We show that $\tilde{h}_2$ 
is an irreducible element of $k[x_1,x_2]$. 
Let $\tilde{\vv }_2=(a_2,ua_2+1)$. 
Then, 
we have $\deg _{\tilde{\vv }_2}x_1=a_2$ and 
$\deg _{\tilde{\vv }_2}x_2\lambda (x_1)^{-1}=1$. 
Hence, 
we get 
$$
\tilde{h}_2^{\tilde{\vv }_2}
=\left( x_1+\left(x_2(cx_1^u)^{-1}\right)^{a_2} 
\right) (cx_1^u)^{a_2}=c^{a_2}x_1^{ua_2+1}+x_2^{a_2}. 
$$
Since $ua_2+1$ and $a_2$ are mutually prime, 
we conclude that $\tilde{h}_2$ 
is an irreducible element of $k[x_1,x_2]$ 
thanks to Lemma~\ref{lem:x^a+y^b}.

The following lemma is also a consequence of 
Lemma~\ref{lem:initial}.

\begin{lem}\label{lem:rhq}
Let $f,p\in \kxr \sm \zs $ and $\w \in \Gamma ^3$ 
be such that $\degw f>0$ or $\degw p>0$. 
If $t_i\degw f\neq a_i\degw p$ for $i\in \Zn $, 
then we have 
$$
\degw \eta _i(f,p)=\max \{ t_i\degw f,a_i\degw p\} . 
$$ 
\end{lem}
\begin{proof}
Set $\vv =(\degw f,\degw p)$. 
First, 
assume that $t_i\degw f>a_i\degw p$. 
Then, 
we have 
$\eta _i(x_1,x_2)^{\vv }=x_1^{t_i}$ 
by Lemma~\ref{lem:initial}, 
since $\degw f>0$ or $\degw p>0$ by assumption. 
This implies that 
$\eta _i(f,p)^{\w }=(f^{\w })^{t_i}$. 
Hence, 
we get $\degw \eta _i(f,p)=t_i\degw f$. 
Thus, 
the lemma is true in the case 
where $t_i\degw f>a_i\degw p$. 
Next, assume that 
$t_i\degw f\leq a_i\degw p$. 
Then, 
we have $t_i\degw f<a_i\degw p$, 
since $t_i\degw f\neq a_i\degw p$ by assumption. 
By Lemma~\ref{lem:initial}, 
it follows that 
$\eta _i(x_1,x_2)^{\vv }=x_2^{a_i}$. 
This implies that 
$\eta _i(f,p)^{\w }=(p^{\w })^{a_i}$. 
Hence, 
we get $\degw \eta _i(f,p)=a_i\degw p$. 
Thus, 
the lemma is true in the case where $t_i\degw f\leq a_i\degw p$, 
and therefore in all cases. 
\end{proof}

\section{Local slice constructions (I)}
\setcounter{equation}{0}
\label{sect:lsc1}

The goal of this section is to prove 
Theorem~\ref{thm:lsc1} (i), 
except for (a) when $(t_0,t_1,i)=(3,1,4)$. 
The exceptional case 
is postponed to Section~\ref{sect:exceptional}. 
At the end of this section, 
we also prove Proposition~\ref{prop:homogeneous}.

Consider the following statements for $i\in \zs \cup I$: 

\smallskip

\noindent{\rm (1)} 
$f_i$ and $f_{i+1}$ belong to $\kx \sm \zs $, 
$D_i$ belongs to $\lnd _k\kx $, 
$D_i(r)=f_if_{i+1}$, 
and $r$ does not belong to $f_i\kx $. 
If $i\geq 1$, 
then $-D_i$ is obtained by a local slice construction 
from the data $(f_i,f_{i-1},r)$.

\noindent{\rm (2)} 
If $i\neq \max I$, 
then $D_i$ is irreducible 
and $\ker D_i=k[f_i,f_{i+1}]$.

\noindent{\rm (3)} 
If $i\neq \max I$, 
then $q_{i+1}$ is an irreducible element of $k[f_{i+1},r]$ 
belonging to $f_i\kx $. 

\noindent{\rm (4)} 
If $i\geq 2$ and $i\neq \max I$, 
then $f_i$ and $D_i(f_{i-2})$ have no common factor.

\smallskip

We note that, 
if $i\neq \max I$, 
then (1), (2) and (3) imply that 
\begin{equation}\label{eq:pre:lsc1a}
k[f_{i+1},r]\cap f_i\kx =q_{i+1}k[f_{i+1},r]. 
\end{equation}
To see this, 
it suffices to check that 
the assumptions of Lemma~\ref{lem:minimal} (iii) 
are fulfilled for $D=D_i$, $f=f_{i+1}$, $g=f_i$, 
$s=r$ and $\eta (y,z)=\eta _{i+1}(y,z)$. 
By (2), 
we have $\ker D_i=k[f_i,f_{i+1}]$. 
By (1), 
we have $D_i(r)=f_if_{i+1}\neq 0$, 
and $D_i$ belongs to $\lnd _k\kx $. 
Moreover, 
$\eta _{i+1}(y,z)$ 
is an irreducible element of $k[y,z]$ 
as mentioned after Lemma~\ref{lem:initial}, 
and $\eta _{i+1}(f_{i+1},r)=q_{i+1}$ 
belongs to $f_i\kx $ by (3). 
Thus, 
the assumptions of Lemma~\ref{lem:minimal} (iii) 
are fulfilled. 
Therefore, 
we get (\ref{eq:pre:lsc1a}). 
Similarly, 
if $i\neq \max I$ and $f_{i-1}$ belongs to $\kx $, 
then (1) and (2) imply that 
\begin{equation}\label{eq:pre:lsc1b}
k[f_{i},r]\cap f_{i+1}\kx =q_{i}k[f_{i},r]. 
\end{equation}
Actually, 
since $q_{i}=\eta _i(f_i,r)=f_{i-1}f_{i+1}$ 
belongs to $f_{i+1}\kx $, 
we obtain (\ref{eq:pre:lsc1b}) 
by applying Lemma~\ref{lem:minimal} (iii) 
with $D=D_i$, $f=f_i$, $g=f_{i+1}$, 
$s=r$ and $\eta (y,z)=\eta _i(y,z)$.

We prove the following proposition 
using the theory of local slice construction.

\begin{prop}\label{prop:lsc1}
The statements {\rm (1)} through {\rm (4)} 
hold for each $i\in \zs \cup I$. 
\end{prop}
\begin{proof}
We proceed by induction on $i$. 
First, 
assume that $i=0$. 
Recall that $f_0=x_2$, $f_1=x_1$ and $D_0=\partial /\partial x_3$. 
Since $r$ has the form $x_1x_2x_3+h$ for some 
$h\in k[x_1,x_2]\sm x_2\kx $, 
we have $D_0(r)=x_1x_2$, 
and $r$ does not belong to $x_2\kx $. 
From these conditions, 
we know that (1) and (2) are true. 
As for (3), 
we see from (\ref{eq:q_1}) 
that $q_1$ is an irreducible element of $k[x_1,r]$ 
belonging to $x_2\kx $. 
Since $i<2$, 
(4) is obvious. 
Therefore, 
(1) through (4) hold for $i=0$.

Next, 
take any $j\in I$, 
and assume that (1) through (4) hold for $i<j$. 
First, 
we prove that (1) holds for $i=j$. 
Put $l=j-1$ and $l'=j+1$. 
Then, 
(1) through (4) hold for $i=l$ by induction assumption. 
By (1), it follows that 
$f_{l}$ and $f_j$ belong to $\kx \sm \zs $, 
$D_{l}$ belongs to $\lnd _k\kx $, 
$D_{l}(r)=f_{l}f_j$, 
and $r$ does not belong to $f_{l}\kx $. 
Since $j$ is an element of $I$, 
we have $l\neq \max I$. 
Hence, 
$D_{l}$ is irreducible 
and $\ker D_{l}=k[f_{l},f_j]$ 
by (2). 
Since $D_{l}=\Delta _{(f_j,f_{l})}$ by definition, 
we know that $D_{l}$ satisfies 
(LSC1) for $f=f_j$ and $g=f_{l}$, 
and (LSC2) for $s=r$ and $F=f_j$. 
By (3), 
$q_j$ is an irreducible element of $k[f_j,r]$ 
belonging to $f_{l}\kx $. 
Hence, 
$f_{l'}=q_jf_{l}^{-1}$ belongs to $\kx \sm \zs $. 
This proves the first part of (1). 
By (a) and (b) of Theorem~\ref{thm:fbook}, 
we know that 
$D_j=\Delta _{(f_{l'},f_j)}=-\Delta _{(f_j,f_{l'})}$ 
belongs to $\lnd _k\kx $ 
and satisfies 
$$
D_j(r)=(-\Delta _{(f_j,f_{l'})})(r)
=-(-f_{l'}f_j)=f_jf_{l'}, 
$$ 
and $-D_j=\Delta _{(f_j,f_{l'})}$ 
is obtained by a local slice construction 
from the data $(f_j,f_{l},r)$. 
Finally, 
we prove that $r$ does not belong to $f_j\kx $. 
Suppose to the contrary that 
$r=f_jg'$ for some $g'\in \kx $. 
Then, 
$g'$ does not belong to $k$, 
for otherwise $f_jf_{l'}=D_j(r)=D_j(f_j)g'=0$. 
Since $\ker D_l=k[f_l,f_j]$, 
we see that $f_j$ also does not belong to $k$ 
by Lemma~\ref{lem:facts}. 
Hence, 
$r$ is not an irreducible element of $\kx $. 
On the other hand, 
$r$ is a linear and primitive polynomial 
in $x_3$ over $k[x_1,x_2]$. 
Hence, $r$ is an irreducible element of $\kx $, 
a contradiction. 
Thus, 
$r$ does not belong to $f_j\kx $. 
Therefore, (1) holds for $i=j$.

Next, 
we prove that (2) holds for $i=j$. 
So assume that $j\neq \max I$. 
Then, we have $t_0\geq 2$, 
since $I=\{ 1\} $ if $t_0=1$. 
In view of Theorem~\ref{thm:fbook} (c), 
it suffices to check that $D_j$ is irreducible. 
Since $t_0\geq 2$, 
we know that $D_1$ is irreducible 
as mentioned after (\ref{eq:D_1}). 
Hence, the assertion is true if $j=1$. 
So assume that $j\geq 2$. 
We prove that (4) holds for $i=j$. 
Then, 
it follows that $D_j$ is irreducible 
thanks to Proposition~\ref{prop:lsc:F:rem}, 
since $F=f_j$. 
Because $D_{l}$ is an element of $\lnd _k\kx $ 
with $\ker D_{l}=k[f_{l},f_j]$, 
we know that 
$\mathfrak{p}_j:=f_j\kx $ is a prime ideal of $\kx $ 
by Lemma~\ref{lem:facts}. 
Therefore, 
it suffices to prove that $D_j(f_{j-2})$ 
does not belong to $\mathfrak{p}_j$.

Since (1) holds for $i\leq j$, 
we have $D_i(r)=f_if_{i+1}\neq 0$ 
for each $i\leq j$. 
On the other hand, 
$D_i=\Delta _{(f_{i+1},f_i)}$ 
kills $f_i$ and $f_{i+1}$ for any $i\geq 0$. 
Hence, 
we know that $r$ and $f_i$ are algebraically independent over $k$ 
for $i\leq l'$ by Lemma~\ref{lem:minimal} (ii). 
Thus, 
we may regard $q_i$ as a polynomial in $r$ and $f_i$ 
over $k$ for each $i\leq l'$. 
First, 
we show that 
\begin{equation}\label{eq:q g_1}
q:=a_j\frac{\partial q_{l}}{\partial f_{l}}f_{l}
+\frac{\partial q_{l}}{\partial r}r
\end{equation}
does not belong to $\mathfrak{p}_j$. 
Suppose to the contrary that 
$q$ belongs to $\mathfrak{p}_j$. 
Then, 
$q$ belongs to $\mathfrak{p}':=\mathfrak{p}_j\cap k[f_{l},r]$. 
Since $j\geq 2$, 
we have $l-1=j-2\geq 0$. 
Hence, 
(1) holds for $i=l-1$, 
and so $f_{l-1}$ belongs to $\kx $. 
Since (1) and (2) hold for $i=l$, 
it follows that $\mathfrak{p}'=q_lk[f_{l},r]$ 
by (\ref{eq:pre:lsc1b}) with $i=l$. 
Thus, 
$q':=q-a_lq_l$ belongs to $q_lk[f_{l},r]$. 
Write $q_l=\eta _l(f_l,r)=f_l^{t_l}+r^{a_l}+h$, 
where $h\in k[f_l,r]$. 
Then, 
we can easily check that $\deg _{f_l}h<t_l$. 
When $l\equiv 0,1\pmod{4}$, 
we have $\deg _rh\leq (t_l-1)b_l<b_lt_l+1=a_l$. 
When $l\equiv 2,3\pmod{4}$, 
we have $l\geq 2$, 
and so $b_l\geq 1$ by Lemma~\ref{lem:a_i} (ii). 
Hence, we get 
$$
\deg _rh\leq (t_l-1)b_l-1<t_lb_l-1=a_l. 
$$
Thus, 
we may write 
\begin{align*}
a_j\frac{\partial q_{l}}{\partial f_{l}}f_l
&=a_j\left( 
t_lf_l^{t_l-1}+\frac{\partial h}{\partial f_l}
\right) f_l
=a_jt_lf_l^{t_l}+h_1 \\
\frac{\partial q_{l}}{\partial r}r
&=\left( 
a_lr^{a_l-1}+\frac{\partial h}{\partial r}
\right) r
=a_lr^{a_l}+h_2, 
\end{align*}
where $h_1,h_2\in k[f_{l},r]$ are such that 
$\deg _{f_l}h_i<t_l$ and $\deg _rh_i<a_l$ 
for $i=1,2$. 
Since $t_{l}a_j-a_{l}=a_{l'}$ 
by (\ref{eq:a_i zenkasiki}), 
it follows that  
\begin{equation}\label{eq:q'lsc1}
\begin{aligned}
&q'=q-a_lq_l
=(a_jt_lf_l^{t_l}+h_1)
+(a_lr^{a_l}+h_2)
-a_l(f_l^{t_l}+r^{a_l}+h) \\
&\quad =(t_{l}a_j-a_{l})f_l^{t_l}+h_1+h_2-a_lh
=a_{l'}f_l^{t_l}+h_1+h_2-a_lh. 
\end{aligned}
\end{equation}
From this, 
we see that 
$\deg _{r}q'<a_l=\deg _rq_l$. 
Since 
$q'$ belongs to $q_lk[f_{l},r]$ as mentioned, 
it follows that $q'=0$. 
On the other hand, 
we have $a_{l'}=a_{j+1}\neq 0$ 
by the assumption that $j\neq \max I$. 
Hence, 
we see from (\ref{eq:q'lsc1}) that 
$\deg _{f_l}q'=t_l$, a contradiction. 
Therefore, 
we conclude that 
$q$ does not belong to $\mathfrak{p}_j$.

By chain rule, 
we have 
$$
D_{j}(q_i)=\frac{\partial q_{i}}{\partial f_i}D_{j}(f_i)
+\frac{\partial q_{i}}{\partial r}D_{j}(r)
$$
for each $i$. 
Since 
$D_j=\Delta _{(f_{l'},f_j)}$ kills $f_{l'}$ and $f_j$, 
and since $D_j(r)=f_jf_{l'}$ by (1) with $i=j$, 
it follows that  
\begin{equation}\label{eq:D_i(f_{i-1})}
D_j(f_{l})=D_j(q_jf_{l'}^{-1})=
D_j(q_j)f_{l'}^{-1}=
\frac{\partial q_{j}}{\partial r}D_j(r)f_{l'}^{-1}
=\frac{\partial q_{j}}{\partial r}f_j. 
\end{equation}
Similarly, 
we have 
\begin{align}\begin{split}\label{eq:D_i(f_{i-2})}
&D_j(f_{j-2})=D_j(q_{l}f_j^{-1})
=D_j(q_{l})f_j^{-1} \\
&\quad =\left( 
\frac{\partial q_{l}}{\partial f_{l}}D_j(f_{l})+
\frac{\partial q_{l}}{\partial r}D_j(r)
\right) f_j^{-1}
=\frac{\partial q_{l}}{\partial f_{l}}
\frac{\partial q_{j}}{\partial r}+
\frac{\partial q_{l}}{\partial r}f_{l'}, 
\end{split}
\end{align}
where we use (\ref{eq:D_i(f_{i-1})}) 
for the last equality. 
Put $(y)=yk[y,z]$. 
Then, 
we have 
$\eta _j(y,z)\equiv z^{a_j}\pmod{(y)}$, 
and so 
$\partial \eta _j(y,z)/\partial z
\equiv a_jz^{a_j-1}\pmod{(y)}$. 
Since $f_jk[f_j,r]$ is contained in $\mathfrak{p}_j$, 
it follows that 
\begin{equation}\label{eq:lsc1cong}
f_{l}f_{l'}=q_j\equiv r^{a_j}
\pmod{\mathfrak{p}_j}\quad \text{and}\quad 
\frac{\partial q_{j}}{\partial r}
\equiv a_jr^{a_j-1}\pmod{\mathfrak{p}_j}. 
\end{equation}
Thus, 
we know by (\ref{eq:D_i(f_{i-2})}) that 
$f_{l}D_j(f_{j-2})$ is congruent to 
$$
\frac{\partial q_{l}}{\partial f_{l}}f_l
(a_jr^{a_j-1})+
\frac{\partial q_{l}}{\partial r}
r^{a_j}=qr^{a_j-1} 
$$
modulo $\mathfrak{p}_j$. 
Since $r$ does not belong to $\mathfrak{p}_j$ by (1) with $i=j$, 
and $q$ does not belong to $\mathfrak{p}_j$ as shown above, 
we conclude that $D_j(f_{j-2})$ 
does not belong to $\mathfrak{p}_j$. 
This proves that (4) holds for $i=j$, 
and thereby proving that (2) holds for $i=j$.

Finally, 
we prove that (3) holds for $i=j$. 
Since $f_{l'}$ and $r$ are algebraically independent over $k$ 
as mentioned, 
we know that $q_{l'}=\eta _{l'}(f_{l'},r)$ 
is an irreducible element of $k[f_{l'},r]$ 
by the irreducibility of $\eta _{l'}(y,z)$ in $k[y,z]$. 
We show that 
$q_{l'}$ belongs to $\mathfrak{p}_j$. 
By (1) with $i=j$, 
we see that $r$ does not belong to $\mathfrak{p}_j$. 
Since $\ker D_l=k[f_l,f_j]$ by (2) with $i=l$, 
we know that 
$f_l$ does not belong to $\mathfrak{p}_j$ 
in view of Lemma~\ref{lem:facts}. 
Since $j\neq \max I$ by the assumption of (3), 
we have $\ker D_j=k[f_j,f_{l'}]$ by (2) with $i=j$. 
Hence, 
$f_{l'}$ also does not belong to $\mathfrak{p}_j$ 
by Lemma~\ref{lem:facts}. 
Since $b_{l'}=t_jb_j-b_{l}+\xi _j$ by definition, 
we have 
$$
a_j=t_jb_j+\xi _j=b_l+b_{l'}. 
$$
Hence, 
we get 
$f_{l}f_{l'}\equiv r^{b_{l}+b_{l'}}\pmod{\mathfrak{p}_j}$ 
by the first part of (\ref{eq:lsc1cong}). 
This gives that 
\begin{equation}\label{eq:fraction equiv}
f_{l}r^{-b_{l}}\equiv 
f_{l'}^{-1}r^{b_{l'}},\quad
f_{l}^{-1}r^{b_{l}}\equiv 
f_{l'}r^{-b_{l'}}
\pmod{\mathfrak{p}_j\kx _{\mathfrak{p}_j}}. 
\end{equation}
Put $\theta _i(z)=\sum _{m=1}^{t_i}\alpha _m^{i}z^m$ 
for each $i\geq 0$. 
Then, 
we have 
\begin{equation}\label{eq:irekae}
\begin{aligned}
\eta _{i}(f_i,r)&=
r^{t_{i}b_{i}}(r+\theta _i(f_ir^{-b_{i}}))
& \text{ if }\ i\equiv 0,1\pmod{4}& \\
\eta _{i}(f_i,r)&=
f_i^{t_{i}}r^{-1}(r+\theta _i(f_i^{-1}r^{b_{i}}))
& \text{ if }\ i\equiv 2,3\pmod{4}&, 
\end{aligned}
\end{equation}
and $r=x_1x_2x_3-\theta _0(x_2)-\theta _1(x_1)$. 
We show that $\eta _{l}(f_l,r)$ 
belongs to $\mathfrak{p}_j$. 
First, 
assume that $j=1$. 
Then, we have $l=0$. 
Since $b_0=0$, 
we see from (\ref{eq:irekae}) that 
$$
\eta _0(f_0,r)=r+\theta _0(x_2)
=x_1x_2x_3-\theta _1(x_1). 
$$
Hence, $\eta _0(f_0,r)$ 
belongs to $\mathfrak{p}_1=x_1\kx $. 
Next, 
assume that $j\geq 2$. 
Then, $f_{j-2}$ belongs to $\kx $ by (1) with $i=j-2$. 
Hence, 
$\eta _l(f_l,r)=f_{j-2}f_j$ belongs to $\mathfrak{p}_j$. 
Thus, 
$\eta _l(f_l,r)$ belongs to $\mathfrak{p}_j$ 
in all cases. 
Thanks to (\ref{eq:irekae}), 
it follows that 
$r+\theta _l(f_lr^{-b_{l}})$ and 
$r+\theta _l(f_l^{-1}r^{b_{l}})$ 
belong to $\mathfrak{p}_j\kx _{\mathfrak{p}_j}$ 
if $l\equiv 0,1\pmod{4}$ and if 
$l\equiv 2,3\pmod{4}$, 
respectively. 
Assume that $l\equiv 0,1\pmod{4}$. 
Then, 
this implies that 
$r+\theta _l(f_{l'}^{-1}r^{b_{l'}})$ 
belongs to $\mathfrak{p}_j\kx _{\mathfrak{p}_j}$ 
by (\ref{eq:fraction equiv}). 
Since $l'=l+2$, 
we have $\theta _{l'}(z)=\theta _l(z)$ 
and $l'\equiv 2,3\pmod{4}$. 
Hence, 
we see from the second equality of 
(\ref{eq:irekae}) that 
$q_{l'}=\eta _{l'}(f_{l'},r)$ belongs to 
$\mathfrak{p}_j\kx _{\mathfrak{p}_j}$. 
Since 
$\mathfrak{p}_j\kx _{\mathfrak{p}_j}\cap \kx =\mathfrak{p}_j$, 
it follows that $q_{l'}$ belongs to $\mathfrak{p}_j$. 
Similarly, 
we can check that $q_{l'}$ 
belongs to $\mathfrak{p}_j$ when $l\equiv 2,3\pmod{4}$. 
Therefore, 
(3) holds for $i=j$. 
This proves that 
(1) through (4) hold for every $i\in \zs \cup I$. 
\end{proof}

As a consequence of Proposition~\ref{prop:lsc1}, 
we know that (\ref{eq:pre:lsc1a}) holds for each 
$i\in \zs \cup I$ with $i\neq \max I$, 
and (\ref{eq:pre:lsc1b}) holds for each 
$i\in I$ with $i\neq \max I$. 
This proves the first part of the following lemma.

\begin{lem}\label{lem:local slice}
Assume that $(j,l)=(i,i-1)$ for some $i\in I$, 
or $(j,l)=(i-1,i)$ for some $i\in I\sm \{ 1\} $. 
Then, we have 
$$
k[f_j,r]\cap f_{l}\kx =q_jk[f_j,r]. 
$$ 
If $a_j\geq 2$, 
then we have 
$(r+k[f_j])\cap f_l\kx =\emptyset $. 
\end{lem}
Here, 
for $s\in \kx $ 
and a $k$-vector subspace $A$ of $\kx $, 
we define 
$$
s+A=\{ s+f\mid f\in A\} . 
$$
The last part of the lemma is proved as follows. 
Suppose to the contrary that there exists 
$h\in (r+k[f_j])\cap f_l\kx $. 
Then, 
we have $\deg _rh=1$, 
and so $h\neq 0$. 
Since $h$ belongs to 
$k[f_j,r]\cap f_{l}\kx =q_jk[f_j,r]$, 
it follows that 
$\deg _rh\geq \deg _rq_j=a_j\geq 2$, 
a contradiction. 
Therefore, 
we get $(r+k[f_j])\cap f_l\kx =\emptyset $.

Now, 
let us complete the proof of Theorem~\ref{thm:lsc1} (i), 
except for (a) when $(t_0,t_1,i)\neq (3,1,4)$. 
Proposition~\ref{prop:lsc1}, 
we know that (1) through (4) hold for each $i\in \zs \cup I$. 
By (1), 
we get the first part of Theorem~\ref{thm:lsc1} (i). 
By (2), 
we get (b) of Theorem~\ref{thm:lsc1} (i). 
Hence, it remains only to check 
(a) of Theorem~\ref{thm:lsc1} (i) 
in the cases where 
$t_0=i=1$ and $(t_0,t_1,i)=(2,1,2)$. 
If $t_0=1$, 
then $D_1$ is not irreducible and $\ker D_1\neq k[f_1,f_2]$ 
as mentioned after (\ref{eq:D_1}). 
Hence, 
(a) holds when $t_0=i=1$. 
Assume that $(t_0,t_1)=(2,1)$. 
Then, 
we have $D_2=\tau _2\circ D_1'\circ \tau _2^{-1}$ 
by Theorem~\ref{thm:lsc1} (ii). 
Since $t_1=1$, 
we know that 
$D_1'$ is not irreducible 
and $\ker D_1'\neq k[f_1',f_2']$. 
Hence, 
it follows that 
$D_2$ is not irreducible 
and $\ker D_2\neq k[f_2,f_3]$. 
Thus, 
(a) holds when $(t_0,t_1,i)=(2,1,2)$. 
This proves Theorem~\ref{thm:lsc1} (i), 
except for (a) when $(t_0,t_1,i)\neq (3,1,4)$.

Finally, 
we prove Proposition~\ref{prop:homogeneous}. 
So assume that 
$\alpha _j^i=0$ for $i=0,1$ and $j=1,\ldots ,t_i-1$. 
Then, 
we have 
$q_i=\eta _i(f_i,r)=f_i^{t_i}+r^{a_i}$ 
for each $i\geq 1$. 
Set $d_i=\deg _{\bt }f_i$ for each $i$. 
We prove that 
$f_{i+1}$ is $\bt $-homogeneous 
and $t_{i+1}d_{i+1}=t_0t_1a_{i+1}$ 
for each $i\in \zs \cup I$ by induction on $i$. 
Since $f_1=x_1$, $d_1=t_0$ and $a_1=1$, 
we see that the statement holds for $i=0$. 
Since $f_2$ is $\bt $-homogeneous as mentioned, 
and $d_2=t_1(t_0-1)$ and $a_2=t_0-1$, 
we see that the statement also holds for $i=1$. 
So assume that $i\geq 2$. 
Then, 
$f_l$ is $\bt $-homogeneous 
and $t_{l}d_{l}=t_0t_1a_{l}$ for $l=i-1,i$ 
by induction assumption. 
Hence, we have 
$\deg _{\bt }f_i^{t_i}=t_id_i=t_0t_1a_i=\deg _{\bt }r^{a_i}$. 
Since $f_i$ and $r$ are $\bt $-homogeneous, 
this implies that 
$q_i=f_i^{t_i}+r^{a_i}$ is $\bt $-homogeneous. 
Because $f_{i-1}$ is $\bt $-homogeneous, 
it follows that 
$f_{i+1}=q_if_{i-1}^{-1}$ is $\bt $-homogeneous. 
Note that $\deg _{\bt }q_i=\deg _{\bt }f_i^{t_i}=t_id_i$. 
Since $q_i=f_{i-1}f_{i+1}$, 
we get $d_{i-1}+d_{i+1}=t_id_i$. 
Hence, we know that 
\begin{align*}
&t_{i+1}d_{i+1}-t_0t_1a_{i+1}
=t_{i+1}(t_id_i-d_{i-1})-t_0t_1(t_{i+1}a_i-a_{i-1}) \\
&\quad =t_{i+1}(t_{i}d_{i}-t_0t_1a_{i})
-(t_{i-1}d_{i-1}-t_0t_1a_{i-1})
\end{align*}
in view of (\ref{eq:a_i zenkasiki}). 
Since $t_{l}d_{l}=t_0t_1a_{l}$ for $l=i-1,i$, 
the right-hand side of the preceding equality is zero. 
Thus, 
we get $t_{i+1}d_{i+1}=t_0t_1a_{i+1}$. 
This proves that 
the statement holds for every $i\in \zs \cup I$. 
Therefore, 
$f_i$ and $f_{i+1}$ are $\bt $-homogeneous 
for each $i\in \zs \cup I$. 
This completes the proof of Proposition~\ref{prop:homogeneous}.

\section{Local slice constructions (II)}
\label{sect:lsc2}
\setcounter{equation}{0}

In this section, we prove Theorem~\ref{thm:lsc2} (i). 
So assume that $i=2$ and $t_0\geq 3$, 
or $i\geq 3$, $t_0\geq 3$ and $(t_0,t_1)\neq (3,1)$. 
Then, 
$l:=i-1$ belongs to $I$, 
and is not the maximum of $I$. 
By Theorem~\ref{thm:lsc1} (i), 
it follows that $D_l=\Delta _{(f_i,f_l)}$ 
is irreducible and locally nilpotent, 
and satisfies $\ker D_l=k[f_l,f_i]$. 
Hence, 
$D_l$ satisfies (LSC1) for $f=f_i$ and $g=f_l$.

We show that 
$D_l$ satisfies (LSC2) for $s=r_i$, 
and for $F=\lambda (f_2)$ if $i=2$, 
and $F=\lambda (f_i)f_i$ if $i\geq 3$. 
First, 
we check that $r_i$ does not belong to 
$\mathfrak{p}_l:=f_l\kx $. 
Suppose to the contrary that 
$r_i=\lambda (f_i)\tilde{r}-\mu (f_i,f_l)$ 
belongs to $\mathfrak{p}_l$. 
Then, 
$\lambda (f_i)\tilde{r}$ belongs to $\mathfrak{p}_l$, 
since $\mu (f_i,f_l)$ belongs to $f_lk[f_i,f_l]$ 
by the choice of $\mu (y,z)$. 
Since $\ker D_l=k[f_l,f_i]$, 
we know by Lemmas~\ref{lem:facts} and \ref{lem:minimal} (i) 
that $\mathfrak{p}_l$ 
is a prime ideal of $\kx $ with 
$k[f_i]\cap \mathfrak{p}_l=\zs $. 
Because $f_i$ is not a constant, 
we have $\lambda (f_i)\neq 0$ 
by the assumption that $\lambda (y)\neq 0$. 
Hence, 
$\lambda (f_i)$ does not belong to $\mathfrak{p}_l$. 
Thus, 
$\tilde{r}$ belongs to $\mathfrak{p}_l$. 
However, 
$\tilde{r}=x_2$ 
does not belong to $\mathfrak{p}_1=x_1\kx $ 
if $i=2$, 
and $\tilde{r}=r$ does not belong to 
$\mathfrak{p}_l$ if $i\geq 3$ 
by (1) of Proposition~\ref{prop:lsc1}. 
This is a contradiction. 
Therefore, 
$r_i$ does not belong to $\mathfrak{p}_l$. 
If $i=2$, 
then we have 
$D_l(\tilde{r})=D_1(x_2)=x_1=f_1$ by (\ref{eq:D_1}). 
If $i\geq 3$, 
then we have $D_l(\tilde{r})=D_l(r)=f_lf_i$ 
by Theorem~\ref{thm:lsc1} (i). 
Since 
$$
D_l(r_i)=D_l
\bigl(\lambda (f_i)\tilde{r}-\mu (f_i,f_l)\bigr)
=\lambda (f_i)D_l(\tilde{r}), 
$$
it follows that 
$D_1(r_2)=f_1\lambda (f_2)$ if $i=2$, 
and 
$D_l(r_i)=f_l\lambda (f_i)f_i$ 
if $i\geq 3$. 
Thus, 
we get $D_l(r_i)=f_lF$ for each $i\geq 2$ 
for the $F$ mentioned above. 
Therefore, 
$D_l$ satisfies (LSC2) for $s=r_i$ and this $F$.

Recall that $\tilde{h}_i$ is an irreducible element of $k[x_1,x_2]$ 
as shown after Lemma~\ref{lem:x^a+y^b}. 
Since $D_l(f_i)=0$ and $D_l(r_i)=f_lF\neq 0$, 
we know that 
$f_{i}$ and $r_i$ are algebraically independent over $k$ 
by Lemma~\ref{lem:minimal} (ii). 
Hence, 
it follows that 
\begin{equation}\label{eq:tilde{q}_i}
\tilde{q}_i:=\tilde{h}_i(f_i,r_i)=
\tilde{\eta }_i\left(f_i,
r_i\lambda (f_i)^{-1}
\right)
\lambda (f_i)^{a_i}
\end{equation}
is an irreducible element of $k[f_i,r_i]$. 
We show that $\tilde{q}_i$ belongs to $\mathfrak{p}_l$. 
Since 
$\mu (f_i,f_l)$ belongs to $\mathfrak{p}_l$, 
we have 
$r_i\equiv \lambda (f_i)\tilde{r}\pmod{\mathfrak{p}_l}$. 
Hence, we get 
$$
\tilde{q}_i
\equiv \tilde{\eta }_i(f_i,\tilde{r})\lambda (f_i)^{a_i}
\pmod{\mathfrak{p}_l}. 
$$ 
If $i=2$, 
then $\tilde{\eta }_2(f_2,x_2)=x_1x_3$ 
belongs to $\mathfrak{p}_{1}=x_1\kx $. 
If $i\geq 3$, 
then $\tilde{\eta }_i(f_i,\tilde{r})=\eta _i(f_i,r)=f_lf_{i+1}$ 
belongs to $\mathfrak{p}_l$. 
Therefore, 
it follows that 
$\tilde{q}_i$ belongs to $\mathfrak{p}_l$. 
Hence, 
$\tilde{f}_{i+1}=\tilde{q}_if_l^{-1}$ belongs to $\kx $. 
From (a) and (b) of Theorem~\ref{thm:fbook}, 
we conclude that 
$\tilde{D}_i=\Delta _{(\tilde{f}_{i+1},f_i)}$ 
belongs to $\lnd _k\kx $, 
and $\tilde{D}_i(r_i)=F\tilde{f}_{i+1}$ 
is as in Theorem~\ref{thm:lsc2}~(i).

In view of Theorem~\ref{thm:fbook} (c), 
it remains only to prove that $\tilde{D}_i$ is irreducible 
to complete the proof of Theorem~\ref{thm:lsc2} (i).

Under the assumption of Theorem~{\rm \ref{thm:lsc2}}, 
the following lemma holds.

\begin{lem}\label{lem:D(f_{i-2})}
$\lambda (f_i)$ and $\tilde{D}_i(f_{i-2})$ 
have no common factor. 
\end{lem}
\begin{proof}
Suppose to the contrary that 
$\lambda (f_i)$ and $\tilde{D}_i(f_{i-2})$ 
have a common factor $p\in \kx \sm k$. 
By Lemma~\ref{lem:facts}, 
$k[f_i]$ is factorially closed in $\kx $, 
since $\ker D_l=k[f_l,f_i]$ 
by Theorem~\ref{thm:lsc1} (i). 
Hence, 
$p$ belongs to $k[f_i]$. 
Thus, 
$p$ is divisible by $f_i-\alpha $ 
for some $\alpha \in \bar{k}$, 
where $\bar{k}$ is an algebraic closure of $k$. 
Since $p$ is a factor of $\lambda (f_i)$, 
we have $\lambda (\alpha )=0$. 
By the assumption that 
$\lambda (y)$ and $\mu (y,z)$ have no common factor, 
it follows that $\mu (\alpha ,z)\neq 0$. 
Since $\mu (y,z)$ is an element of $zk[y,z]$, 
we may write $\mu (\alpha ,z)=-z\nu (z)$, 
where $\nu (z)\in k[z]\sm \zs $. 
Then, 
$r_i$ is congruent to 
$\lambda (\alpha )\tilde{r}-\mu (\alpha ,f_l)
=f_l\nu (f_l)$ 
modulo $\mathfrak{p}:=(f_i-\alpha )\bar{k}[\x ]$. 
Note that $\tilde{\eta }_i(y,z)$ 
is a monic polynomial in $z$ over $k[y]$ of degree $a_i$ 
for any $i\geq 2$. 
Hence, 
we see from (\ref{eq:tilde{q}_i}) that 
$\tilde{q}_i$ is congruent to $r_i^{a_i}$, 
and hence to $f_l^{a_i}\nu (f_l)^{a_i}$ 
modulo $\mathfrak{p}$. 
Put $\psi (z)=z^{a_i-1}\nu (z)^{a_i}$. 
Then, it follows that 
$$
\tilde{f}_{i+1}=\tilde{q}_if_l^{-1}\equiv 
\psi (f_l)
\pmod{\mathfrak{p}}. 
$$
If $i=2$, then we have $a_2=t_0-1\geq 2$, 
since $t_0\geq 3$. 
If $i\geq 3$, 
then we have $a_i\geq 2$ by Lemma~\ref{lem:a_i} (ii), 
since $t_0\geq 3$ and $(t_0,t_1)\neq (3,1)$. 
Hence, 
$\psi (z)$ is not a constant, 
and so 
the derivative $\psi'(z)$ is nonzero.

Now, 
regard $\Delta :=\Delta _{(f_i,f_{i-2})}$ 
as a derivation of $\bar{k}[\x ]$. 
Then, 
we have $\Delta (f_{i}-\alpha )=0$. 
Hence, 
$\Delta (\mathfrak{p})$ is contained in $\mathfrak{p}$. 
Since 
$\tilde{f}_{i+1}\equiv \psi (f_l)\pmod{\mathfrak{p}}$, 
it follows that 
$$
\Delta (\tilde{f}_{i+1})\equiv \Delta (\psi (f_l))
\pmod{\mathfrak{p}}. 
$$
Since 
$\Delta (\tilde{f}_{i+1})
=\Delta _{(f_i,f_{i-2})}(\tilde{f}_{i+1})
=\tilde{D}_i(f_{i-2})$, 
and $\tilde{D}_i(f_{i-2})$ is divisible by $p$ 
by supposition, 
this implies that 
$\Delta (\psi (f_l))$ belongs to $\mathfrak{p}$. 
By chain rule, 
we have 
$\Delta (\psi (f_l))=\psi '(f_l)\Delta (f_l)$, 
in which 
$$
\Delta (f_l)=
\Delta _{(f_i,f_{i-2})}(f_l)=-D_l(f_{i-2})
=-\frac{\partial q_l}{\partial r}f_l
$$ 
by (\ref{eq:D_i(f_{i-1})}). 
Hence, 
$\psi '(f_l)f_l(\partial q_l/\partial r)$ 
belongs to $\mathfrak{p}$. 
We show that 
$\partial q_l/\partial r$ 
belongs to $\mathfrak{p}$. 
As discussed after Lemma~\ref{lem:facts}, 
we may extend $D_l$ to a locally nilpotent derivation 
$\bar{D}_l$ of $\bar{k}[\x ]$ 
such that 
\begin{equation}\label{eq:ker f_l,f_i-alpha}
\ker \bar{D}_l
=\bar{k}[f_l,f_i]
=\bar{k}[f_l,f_i-\alpha ]. 
\end{equation}
Hence, 
we know by Lemmas~\ref{lem:facts} and \ref{lem:minimal} (i) 
that $\mathfrak{p}$ is a prime ideal of $\bar{k}[\x ]$ 
such that $\bar{k}[f_l]\cap \mathfrak{p}=\zs $. 
Since $\psi '(f_l)\neq 0$ as mentioned, 
it follows that $\psi '(f_l)f_l$ 
does not belong to $\mathfrak{p}$. 
Thus, 
we conclude that 
$\partial q_l/\partial r$ belongs to $\mathfrak{p}$. 
Note that 
$\bar{D}_l(\mathfrak{p})$ is contained in $\mathfrak{p}$, 
since $\bar{D}_l(f_i-\alpha )=0$. 
Hence, 
$D_l^j(\partial q_l/\partial r)$ belongs to $\mathfrak{p}$ 
for any $j\geq 0$. 
Since $q_l=\eta _l(f_l,r)$ 
is a monic polynomial in $r$ over $k[f_l]$ 
of degree $a_l$, 
we have $\partial ^{a_l}q_l/\partial r^{a_l}=a_l!$. 
Because $D_l(f_l)=D_l(f_i)=0$ and $D_l(r)=f_lf_i$, 
it follows that 
\begin{align*}
&D_l^{a_l-1}\left(
\frac{\partial q_l}{\partial r}\right)
=D_l^{a_l-2}\left(
\frac{\partial ^2q_l}{\partial r^2}
D_l(r)
\right)
=D_l^{a_l-2}\left(
\frac{\partial ^2q_l}{\partial r^2}f_lf_i
\right) 
\\
&\quad 
=D_l^{a_l-2}\left(
\frac{\partial ^2q_l}{\partial r^2}\right) f_lf_i
=\cdots 
=\frac{\partial ^{a_l}q_l}{\partial r^{a_l}}
(f_lf_i)^{a_l-1}
=a_l!(f_lf_i)^{a_l-1}
\end{align*}
by chain rule. 
Therefore, 
$a_l!(f_lf_i)^{a_l-1}$ 
belongs to $\mathfrak{p}$.

When $i=2$, 
we have $a_l!(f_lf_i)^{a_l-1}=1$, 
since $a_l=a_1=1$. 
This implies that $\mathfrak{p}=\bar{k}[\x ]$, 
a contradiction. 
Assume that $i\geq 3$. 
Then, 
we have $a_l=a_{i-1}\geq 2$ by Lemma~\ref{lem:a_i} (ii). 
Hence, 
$f_l$ or $f_i$ belongs to 
$\mathfrak{p}$. 
In view of (\ref{eq:ker f_l,f_i-alpha}), 
we know by Lemma~\ref{lem:facts} 
that $f_l$ does not belong to $\mathfrak{p}$. 
Hence, 
$f_i$ belongs to $\mathfrak{p}$. 
Thus, we get $\alpha =0$. 
Since $f_i$ is an element of $\kx $, 
it follows that 
$\mathfrak{p}\cap \kx =f_i\bar{k}[\x ]\cap \kx =f_i\kx $, 
to which 
$\partial q_l/\partial r$ belongs. 
Hence, 
$\partial q_l/\partial r$ 
belongs to $f_i\kx \cap k[f_l,r]$. 
By Lemma~\ref{lem:local slice}, 
we have $f_i\kx \cap k[f_l,r]=q_lk[f_l,r]$. 
Thus, 
$\partial q_l/\partial r$ is divisible by $q_l$. 
This implies that $\partial q_l/\partial r=0$. 
Hence, we get $a_l=\deg _rq_l=0$, 
a contradiction. 
Therefore, 
$\lambda (f_i)$ and $\tilde{D}_i(f_{i-2})$ 
have no common factor. 
\end{proof}

If $i=2$, 
then $\lambda (f_2)$ and $\tilde{D}_2(f_0)$ 
have no common factor by Lemma~\ref{lem:D(f_{i-2})}. 
Since $F=\lambda (f_2)$, 
it follows that $\tilde{D}_2$ is irreducible 
by Proposition~\ref{prop:lsc:F:rem}. 
Assume that $i\geq 3$. 
Then, 
we have $F=\lambda (f_i)f_i$. 
By Lemma~\ref{lem:D(f_{i-2})}, 
we know that 
$\lambda (f_i)$ and $\tilde{D}_i(f_{i-2})$ 
have no common factor. 
Hence, 
it suffices to prove that 
$f_i$ and $\tilde{D}_i(f_{i-2})$ 
have no common factor 
by virtue of Proposition~\ref{prop:lsc:F:rem}. 
Since $f_i$ is an irreducible element of $\kx $, 
we verify that $\tilde{D}_i(f_{i-2})$ 
does not belong to $\mathfrak{p}_i=f_i\kx $.

If $\lambda (0)=0$, 
then $f_i$ is a factor of $\lambda (f_i)$. 
Since $\lambda (f_i)$ and $\tilde{D}_i(f_{i-2})$ 
have no common factor, 
it follows that 
$\tilde{D}_i(f_{i-2})$ does not belong to $\mathfrak{p}_i$. 
So assume that $\lambda (0)\neq 0$. 
If $\mu (0,z)=0$, 
then $\mu (f_i,f_l)$ 
belongs to $\mathfrak{p}_i$. 
Since $\tilde{r}=r$, 
it follows that 
$r_i\equiv \lambda (f_i)r\pmod{\mathfrak{p}_i}$. 
Hence, 
we see from (\ref{eq:tilde{q}_i}) that 
$$
\tilde{q}_i
\equiv 
\eta _i(f_i,r)\lambda (f_i)^{a_i}
\equiv \lambda (0)^{a_i}q_i\pmod{\mathfrak{p}_i}. 
$$ 
Since $\tilde{q}_i=f_{l}\tilde{f}_{i+1}$ 
and $q_i=f_{l}f_{i+1}$, 
it follows that 
$f_{l}(\tilde{f}_{i+1}-\lambda (0)^{a_i}f_{i+1})$ 
belongs to $\mathfrak{p}_i$. 
Because $f_{l}$ does not belong to $\mathfrak{p}_i$ 
by Lemma~\ref{lem:facts}, 
we conclude that  
$\tilde{f}_{i+1}-\lambda (0)^{a_i}f_{i+1}$ 
belongs to $\mathfrak{p}_i$. 
Write 
$\tilde{f}_{i+1}-\lambda (0)^{a_i}f_{i+1}=f_ig$, 
where $g\in \kx $. 
Then, we have 
$$
\tilde{D}_i
=\Delta _{(\tilde{f}_{i+1},f_i)}
=\lambda (0)^{a_i}\Delta _{(f_{i+1},f_i)}+f_i\Delta _{(g,f_i)}
=\lambda (0)^{a_i}D_i+f_i\Delta _{(g,f_i)}. 
$$ 
Since $f_i\Delta _{(g,f_i)}(f_{i-2})$ 
belongs to $\mathfrak{p}_i$, 
we get 
$\tilde{D}_i(f_{i-2})\equiv \lambda (0)^{a_i}D_i(f_{i-2})
\pmod{\mathfrak{p}_i}$. 
By (4) of Proposition~\ref{prop:lsc1}, 
$f_i$ and $D_i(f_{i-2})$ have no common factor. 
Since $\lambda (0)\neq 0$ by assumption, 
it follows that 
$\lambda (0)^{a_i}D_i(f_{i-2})$ 
does not belong to $\mathfrak{p}_i$. 
Therefore, 
$\tilde{D}_i(f_{i-2})$ 
does not belong to $\mathfrak{p}_i$.

Finally, 
assume that 
$\lambda (0)\neq 0$ and $\mu (0,z)\neq 0$. 
Then, 
we have $\mu _j(0)\neq 0$ for some $j\geq 1$. 
Put 
$$
g'=a_i\mu _z(0,f_l)
f_l-\mu (0,f_l)
=\sum _{j\geq 1}\mu _j(0)(ja_i-1)f_l^j, 
$$
where 
$$
\mu _z(y,z):=\frac{\partial \mu (y,z)}{\partial z}
=\sum _{j\geq 1}j\mu _j(y)z^{j-1}. 
$$
Then, we have $g'\neq 0$, 
since $a_i\geq 2$ by Lemma~\ref{lem:a_i} (ii).

Now, we prove that 
$\tilde{D}_i(f_{i-2})$ 
does not belong to $\mathfrak{p}_i$ 
by contradiction. 
Suppose to the contrary that 
$\tilde{D}_i(f_{i-2})$ belongs to $\mathfrak{p}_{i}$. 
Then, 
we have the following claim. 
Here, 
we define 
$q\in k[f_l,r]$ as in (\ref{eq:q g_1}) with $j=i$, 
and $g=g'\partial q_{l}/\partial r$.

\begin{claim}\label{lem:D(f_{i-2})'}
$\lambda (0)q+g$ belongs to $\mathfrak{p}_{i}$. 
\end{claim}

First, 
we assume this claim, 
and derive a contradiction. 
Since $\lambda (0)q+g$ 
is an element of $k[f_l,r]$, 
it follows from Claim~\ref{lem:D(f_{i-2})'} 
that $\lambda (0)q+g$ 
belongs to $\mathfrak{p}_i\cap k[f_l,r]$. 
By Lemma~\ref{lem:local slice}, 
we have $\mathfrak{p}_i\cap k[f_l,r]=q_lk[f_l,r]$. 
Hence, 
$$
g_1:=\lambda (0)(q-a_lq_l)+g
=(\lambda (0)q+g)-\lambda (0)a_lq_l
$$ 
belongs to $q_lk[f_l,r]$. 
This implies that $\deg _rg_1\geq \deg _rq_l$ or $g_1=0$. 
From (\ref{eq:q'lsc1}), 
we see that $\deg _r(q-a_lq_l)<a_l=\deg _rq_l$. 
Since $g'$ is an element of $k[f_l]\sm \zs $, 
we have 
$$
\deg _rg
=\deg _rg'\frac{\partial q_l}{\partial r}
=\deg _r\frac{\partial q_l}{\partial r}<\deg _rq_l. 
$$
Hence, we get $\deg _rg_1<\deg _rq_l$. 
Thus, we conclude that $g_1=0$. 
Therefore, 
we obtain $q-a_lq_l\approx g$.

First, assume that $i=3$. 
Then, 
we have $b_l=b_2=1$, 
and so 
$$
q_l=q_2=\eta _2(f_2,r)
=f_2^{t_0}+\sum _{j=1}^{t_0}\alpha _j^0f_2^{t_0-j}r^{j-1}. 
$$
Regard $q_l$ 
as a polynomial in $r$ over $k[f_2]$. 
Then, 
the coefficient of $r^{t_0-2}$ in $q_l$ 
is equal to $f_2$ multiplied by a constant. 
From (\ref{eq:q g_1}), 
we see that the same holds for 
the coefficient of $r^{t_0-2}$ in $q$, 
and hence for 
the coefficient of $r^{t_0-2}$ in 
$$
q-a_lq_l\approx 
g=g'\frac{\partial q_{l}}{\partial r}
=g'\sum _{j=2}^{t_0}\alpha _j^0(j-1)f_2^{t_0-j}r^{j-2}. 
$$
Because $\alpha _{t_0}^0(t_0-1)g'\neq 0$, 
it follows that $g'\approx f_2$. 
Hence, we get 
$\deg _{f_2}(q-a_lq_l)=t_0-1$. 
On the other hand, 
we have 
$\deg _{f_2}(q-a_lq_l)=\deg _{f_l}(q-a_lq_l)=t_l=t_0$ 
by (\ref{eq:q'lsc1}). 
This is a contradiction.

Next, assume that $i\geq 4$. 
Regard $q_l=\eta _l(f_l,r)$ 
as a polynomial in $r$ over $k[f_l]$. 
We show that $r^{a_l-1}$ does not appear in $q_l$. 
When $l\equiv 0,1\pmod{4}$, 
it is easy to see that 
$r^{a_l-1}=r^{t_lb_l}$ does not appear in $q_l$. 
When $l\equiv 2,3\pmod{4}$, 
we have $b_l\geq 2$ by Lemma~\ref{lem:a_i} (iv), 
since $l=i-1\geq 3$. 
From this, 
we know that $r^{a_l-1}=r^{t_lb_l-2}$ 
does not appear in $q_l$. 
By (\ref{eq:q g_1}), 
it follows that 
$r^{a_l-1}$ does not appear in $q$, 
and hence in $q-a_lq_l\approx g$. 
Since $\deg _rq_l=a_l$, 
however, 
we see that 
$r^{a_l-1}$ appears in $\partial q_l/\partial r$ 
with coefficient $a_l\neq 0$, 
and hence in $g$ with coefficient $a_lg'\neq 0$. 
This is a contradiction. 
Therefore, 
we conclude that $\tilde{D}_i(f_{i-2})$ 
does not belong to $\mathfrak{p}_{i}$.

Finally, 
we prove Claim~\ref{lem:D(f_{i-2})'}. 
Recall that $f_i$ and $r_i$ 
are algebraically independent over $k$. 
Hence, we may consider the partial derivatives of 
$\tilde{q}_i=\tilde{h}_i(f_i,r_i)$ in $f_i$ and $r_i$. 
Since $f_l=\tilde{q}_i\tilde{f}_{i+1}^{-1}$, 
$\tilde{D}_i(f_i)=\tilde{D}_i(\tilde{f}_{i+1})=0$ 
and $\tilde{D}_i(r_i)=\lambda (f_i)f_i\tilde{f}_{i+1}$, 
we have 
\begin{equation}\label{eq:claimpf:rank3-2}
\tilde{D}_i(f_l)
=\tilde{D}_i(\tilde{q}_i\tilde{f}_{i+1}^{-1})
=\tilde{D}_i(\tilde{q}_i)\tilde{f}_{i+1}^{-1}
=\frac{\partial \tilde{q}_i}{\partial r_i}
\tilde{D}_i(r_i)\tilde{f}_{i+1}^{-1}
=\frac{\partial \tilde{q}_i}{\partial r_i}
\lambda (f_i)f_i
\end{equation}
by chain rule. 
Since $r_i=\lambda (f_i)r-\mu (f_i,f_{l})$, 
it follows that 
\begin{align*}
&\lambda (f_i)f_i\tilde{f}_{i+1}=\tilde{D}_i(r_i)
=\tilde{D}_i\bigl(\lambda (f_i)r-\mu (f_i,f_{l})\bigr) \\
&\quad =\lambda (f_i)\tilde{D}_i(r)
-\mu _z(f_i,f_l)\tilde{D}_i(f_l)
=\lambda (f_i)\tilde{D}_i(r)
-\mu _z(f_i,f_l)\frac{\partial \tilde{q}_i}{\partial r_i}
\lambda (f_i)f_i. 
\end{align*}
Since $\lambda (f_i)\neq 0$, 
this gives that 
\begin{equation}\label{eq:claimpf:rank3-3}
\tilde{D}_i(r)
=f_i\left(
\mu _z(f_i,f_l)\frac{\partial \tilde{q}_i}{\partial r_i}
+\tilde{f}_{i+1}\right) . 
\end{equation}
From (\ref{eq:claimpf:rank3-2}) 
and (\ref{eq:claimpf:rank3-3}), 
we obtain 
\begin{align*}
\tilde{D}_i(q_l)
&=\frac{\partial q_l}{\partial f_l}\tilde{D}_i(f_l)
+\frac{\partial q_l}{\partial r}\tilde{D}_i(r) \\
&=f_i\left( 
\frac{\partial q_l}{\partial f_l}
\frac{\partial \tilde{q}_i}{\partial r_i}\lambda (f_i)
+\frac{\partial q_l}{\partial r}
\left(
\mu _z(f_i,f_l)\frac{\partial \tilde{q}_i}{\partial r_i}+
\tilde{f}_{i+1}\right)
\right) . 
\end{align*}
Since $\tilde{f}_{i+1}=\tilde{q}_if_l^{-1}$, 
the right-hand side 
of this equality is written as $f_if_l^{-1}h_1$, 
where 
$$
h_1:=f_l\frac{\partial \tilde{q}_i}{\partial r_i}
\left(
\lambda (f_i)
\frac{\partial q_l}{\partial f_l}
+\mu _z(f_i,f_l)\frac{\partial q_l}{\partial r}
\right)
+\tilde{q}_i\frac{\partial q_l}{\partial r}. 
$$
Because $\tilde{D}_i(f_{i-2})$ 
belongs to $\mathfrak{p}_i$ by assumption, 
$$
h_1=f_i^{-1}f_l\tilde{D}_i(q_l)
=f_i^{-1}f_l\tilde{D}_i(f_{i-2}f_i)
=f_l\tilde{D}_i(f_{i-2})
$$
belongs to $\mathfrak{p}_i$. 
Since $\eta _i(0,z)=z^{a_i}$, 
we see from (\ref{eq:tilde{q}_i}) that 
\begin{equation}\label{eq:(y)}
\tilde{q}_i\equiv 
\eta _i(0,\lambda (f_i)^{-1}r_i)\lambda (f_i)^{a_i}
=r_i^{a_i}
\pmod{f_ik[f_i,r_i]}. 
\end{equation}
This gives that 
$$
\frac{\partial \tilde{q}_i}{\partial r_i}
\equiv a_ir_i^{a_i-1}
\pmod{f_ik[f_i,r_i]}.
$$
Hence, we have 
$h_1\equiv r_i^{a_i-1}h_2\pmod{\mathfrak{p}_i}$, 
where 
$$
h_2:=a_if_l\left(
\lambda (0)\frac{\partial q_l}{\partial f_l}
+\mu _z(0,f_l)\frac{\partial q_l}{\partial r}\right)
+r_i\frac{\partial q_l}{\partial r}. 
$$
Since 
$a_l=a_{i-1}\geq 2$ by Lemma~\ref{lem:a_i} (ii), 
we have $(r+k[f_l])\cap \mathfrak{p}_i=\emptyset $ 
by Lemma~\ref{lem:local slice}. 
Since $\lambda (0)\neq 0$ by assumption, 
and $r_i\equiv \lambda (0)r-\mu (0,f_l)\pmod{\mathfrak{p}_i}$, 
it follows that $r_i$ does not belong to $\mathfrak{p}_i$. 
Thus, we know that $h_2$ belongs to $\mathfrak{p}_i$. 
Now, 
observe that 
\begin{align*}
\lambda (0)q+g&=\lambda (0)\left( 
a_if_l\frac{\partial q_l}{\partial f_l}
+r\frac{\partial q_l}{\partial r}\right) 
+\bigl(a_i\mu _z(0,f_l)f_l-\mu (0,f_l)\bigr)
\frac{\partial q_l}{\partial r}\\
&=a_if_l\left(
\lambda (0)\frac{\partial q_l}{\partial f_l}
+\mu _z(0,f_l)\frac{\partial q_l}{\partial r}\right)
+(\lambda (0)r-\mu (0,f_l))
\frac{\partial q_l}{\partial r}. 
\end{align*}
Since the right-hand side of this equality 
is congruent to $h_2$ modulo $\mathfrak{p}_i$, 
we conclude that 
$\lambda (0)q+g$ belongs to $\mathfrak{p}_i$. 
This proves Claim~\ref{lem:D(f_{i-2})'}, 
and thereby 
completing the proof of Theorem~\ref{thm:lsc2} (i).

\section{Recurrence equations}
\setcounter{equation}{0}

In what follows, 
let $\Gamma $ be the totally ordered additive group $\Z ^3$ 
equipped with the lexicographic order such that 
$\e _1<\e _2<\e _3$. 
From $\Gamma ^3$, 
we take the weight $\w :=(\e _1,\e _2,\e_ 3)$, 
and define 
$$
\delta _i=\degw f_i
$$ 
for each $i\geq 0$. 
Then, 
we have 
\begin{equation}\label{eq:delta vector}
\delta _0=(0,1,0),\quad 
\delta _1=(1,0,0),\quad 
\delta _2=(1,0,1),\quad 
\degw r=(1,1,1). 
\end{equation}

We remark that, 
if 
$$
\gamma _i:=t_i\delta _i-a_i\degw r>0
$$ 
for $i\in \N $, 
then the $\w $-degree of 
$q_i=\eta _i(f_i,r)$ is equal to 
$\max \{ t_i\delta _i,a_i\degw r\} =t_i\delta _i$ 
by Lemma~\ref{lem:rhq}, 
since $\degw r>0$. 
When this is the case, 
we have 
$$
\delta _{i+1}=\degw f_{i+1}=\degw q_if_{i-1}^{-1}
=t_i\delta _i-\delta _{i-1}. 
$$

\begin{prop}\label{prop:delta}
If $t_0\geq 3$, 
then $\delta _{i+1}=t_i\delta _i-\delta _{i-1}$ holds 
for each $i\in I\sm \{ 1\} $. 
\end{prop}
\begin{proof}
We prove that the following statements hold for 
each $i\in I$ by induction on $i$: 

\smallskip 

\noindent
(1) If $(t_0,t_1)=(3,1)$ and $i\neq \max I$, 
then we have $\gamma _{i+1}>0$.

\noindent
(2) 
If $(t_0,t_1)\neq (3,1)$, 
then we have $\gamma _{i+1}>0$. 
If furthermore $i\geq 2$, 
then we have $\gamma _{i+1}-\gamma _{i-1}>0$.

\noindent
(3) If $i\geq 2$, then we have 
$\delta _{i+1}=t_{i}\delta _i-\delta _{i-1}$ 
and $\gamma _{i+1}=t_{i+1}\gamma _i-\gamma _{i-1}$. 

\smallskip

Since 
$$
\gamma _2=
t_2\delta _2-a_2\degw r
=t_0(1,0,1)-(t_0-1)(1,1,1)=(1,1-t_0,1)>0, 
$$
we see that the statements hold when $i=1$. 
Take $i\in I$ with $i\geq 2$. 
Then, 
$i-1$ belongs to $I$, 
and is not the maximum of $I$. 
Hence, 
we have $\gamma _i>0$ 
by the induction assumption of (1) and (2). 
As remarked above, 
this implies that 
$\delta _{i+1}=t_i\delta _i-\delta _{i-1}$, 
proving the first part of (3). 
Since $a_{i+1}=t_{i+1}a_i-a_{i-1}$ 
by (\ref{eq:a_i zenkasiki}), 
and $t_{i+1}=t_{i-1}$, 
it follows that 
\begin{align*}
&\gamma _{i+1}=t_{i+1}\delta _{i+1}-a_{i+1}\degw r
=t_{i+1}(t_i\delta _i-\delta _{i-1})
-(t_{i+1}a_i-a_{i-1})\degw r \\
&\quad =t_{i+1}(t_i\delta _i-a_i\degw r)
-(t_{i-1}\delta _{i-1}-a_{i-1}\degw r)
=t_{i+1}\gamma _i-\gamma _{i-1}. 
\end{align*}
This proves the second part of (3). 
Hence, we have 
$\gamma _{j+1}=t_{j+1}\gamma _j-\gamma _{j-1}$ 
for $j=2,\ldots,i$, 
where the case $j<i$ is due to 
the induction assumption of (3). 
When $i=2$, 
we have 
$\gamma _3=t_1\gamma _2-\gamma _1$, 
and so 
$\gamma _3-\gamma _1=t_1\gamma _2-2\gamma _1$. 
Since $\gamma _2>0$ and 
$$
\gamma _1
=t_1\delta _1-a_1\degw r
=t_1(1,0,0)-(1,1,1)=(t_1-1,-1,-1)<0, 
$$
it follows that 
$\gamma _3$ and $\gamma _3-\gamma _1$ 
are positive. 
Hence, (1) and (2) are true if $i=2$. 
When $i=3$, 
we have 
$\gamma _{j+1}=t_{j+1}\gamma _j-\gamma _{j-1}$ 
for $j=2,3$. 
This gives that 
$$
\gamma _4-\gamma _2
=t_0\gamma _3-2\gamma _2
=t_0(t_1\gamma _2-\gamma _1)-2\gamma _2
=(t_0t_1-2)\gamma _2-t_0\gamma _1. 
$$
Since $t_0\geq 3$ by assumption, 
and $\gamma _2>0$ and $\gamma _1<0$, 
it follows that $\gamma _4-\gamma _2>0$, 
and so $\gamma _4>0$. 
Thus, 
(1) and (2) are true if $i=3$. 
When $(t_0,t_1)=(3,1)$, 
we have $I=\{ 1,\ldots ,4\} $. 
Hence, the proof is completed in this case. 
So 
assume that $(t_0,t_1)\neq (3,1)$ and $i\geq 4$. 
Then, 
we have 
$\gamma _{j+1}=t_{j+1}\gamma _j-\gamma _{j-1}$ 
for $j=i-2,i-1,i$, since $i-2\geq 2$. 
Hence, 
we get 
$$
\gamma _{i+1}-\gamma _{i-1}=
(t_0t_1-4)\gamma _{i-1}+(\gamma _{i-1}-\gamma _{i-3}) 
$$
by Lemma~\ref{lem:skip}. 
Since $(t_0,t_1)\neq (3,1)$, 
and $t_0\geq 3$ by assumption, 
we have $t_0t_1\geq 4$. 
Since $i-2\geq 2$, 
we know that 
$\gamma _{i-1}$ and $\gamma _{i-1}-\gamma _{i-3}$ 
are positive 
by the induction assumption of (2). 
Thus, 
we conclude that 
$\gamma _{i+1}-\gamma _{i-1}>0$, 
and so $\gamma _{i+1}>0$. 
Therefore, 
(2) is true if $i\geq 4$. 
This proves that 
(1), (2) and (3) hold for every $i\in I$. 
Consequently, 
we know by (3) that 
$\delta _{i+1}=t_{i}\delta _i-\delta _{i-1}$ 
holds for each $i\in I\sm \{ 1\} $. 
\end{proof}

We derive some consequences 
of Proposition~\ref{prop:delta}.  
When $(t_0,t_1)=(3,1)$, 
we have 
\begin{equation*}
\delta _3=3\delta _2-\delta _1=(2,0,3),\quad 
\delta _4=\delta _3-\delta _2=(1,0,2),\quad 
\delta _5=3\delta _4-\delta _3=(1,0,3) 
\end{equation*}
by Proposition~\ref{prop:delta} 
and (\ref{eq:delta vector}). 
From this and (\ref{eq:delta vector}), 
we can easily check that the following proposition holds 
when $(t_0,t_1)=(3,1)$, 
since $I=\{ 1,\ldots ,4\} $.

\begin{prop}\label{prop:deltap}
If $t_0\geq 3$, 
then the following assertions hold$:$

\noindent{\rm (i)} 
$\delta _{i+1}>0$ for each $i\in \zs \cup I$. 

\noindent{\rm (ii)} 
$\delta _i$ and $\delta _{i+1}$ are linearly independent 
for each $i\in I$. 

\noindent{\rm (iii)} 
If $t_1\geq 2$, 
then $\delta _{i+1}-\delta _i>0$ for each $i\geq 1$. 

\noindent{\rm (iv)} 
If $(t_0,t_1)\neq (3,1)$, 
then $\delta _{i+2}-\delta _i>0$ for each $i\geq 1$. 

\noindent{\rm (v)} 
The second component of $\delta _{i+1}$ is zero 
for each $i\in \zs \cup I$. 
\end{prop}
\begin{proof}
By the discussion above, 
we may assume that $(t_0,t_1)\neq (3,1)$. 
Then, we have $I=\N $, 
since $t_0\geq 3$ by assumption.

First, 
we prove (v) by induction on $i$. 
By (\ref{eq:delta vector}), 
the second components of 
$\delta _0$ and $\delta _1$ are zero. 
Assume that $i\geq 2$. 
Then, 
we have 
$\delta _{i+1}=t_{i}\delta _i-\delta _{i-1}$ 
by Proposition~\ref{prop:delta}. 
Hence, the second component of $\delta _{i+1}$ is zero, 
proving (v).

We prove (i) through (iv) 
simultaneously by induction on $i$. 
When $i=0$, 
(ii), (iii) and (iv) are obvious. 
Since $\delta _1>0$ by (\ref{eq:delta vector}), 
we get (i). 
When $i=1$, 
(i), (ii) and (iii) follow from 
(\ref{eq:delta vector}). 
By Proposition~\ref{prop:delta}, 
we have $\delta _3-\delta _1=t_0\delta _2-2\delta _1$. 
Since the third component of 
$t_0\delta _2-2\delta _1$ is $t_0>0$, 
we know that $\delta _3-\delta _1>0$. 
This proves (iv). 
Assume that $i\geq 2$. 
Then, 
we have $\delta _{i-1}>0$ 
and $\delta _{i+1}-\delta _{i-1}>0$ 
by the induction assumption of (i) and (iv), 
since $i-1\geq 1$. 
Hence, 
we get $\delta _{i+1}>0$, 
proving (i). 
By the induction assumption of (ii), 
$\delta _{i-1}$ and $\delta _i$ are linearly independent. 
Hence, 
$\delta _i$ and $\delta _{i+1}=t_i\delta _i-\delta _{i-1}$  
are linearly independent, 
proving (ii). 
To show (iii), 
assume that $t_1\geq 2$. 
Then, 
we have $t_i-2\geq 0$ 
independently of the parity of $i$. 
By the induction assumption of (i) and (iii), 
we have $\delta _i>0$ and $\delta _i-\delta _{i-1}>0$. 
Since $\delta _{i+1}=t_i\delta _i-\delta _{i-1}$, 
it follows that 
$$
\delta _{i+1}-\delta _i
=(t_i-2)\delta _i+(\delta _i-\delta _{i-1})
\geq \delta _i-\delta _{i-1}>0,  
$$
proving (iii). 
Finally, we show (iv). 
When $i=2$, 
we have 
\begin{equation}\label{eq:delta 3,4}
\begin{gathered}
\delta _4-\delta _2=t_1\delta _3-2\delta _2
=t_1(t_0\delta _2-\delta _1)-2\delta _2
=(t_0t_1-2)\delta _2-t_1\delta _1. 
\end{gathered}
\end{equation}
Since the third component of 
$(t_0t_1-2)\delta _2-t_1\delta _1$ 
is $t_0t_1-2>0$, 
we know that $\delta _4-\delta _2>0$. 
Thus, 
(iv) is true if $i=2$. 
Assume that $i\geq 3$. 
Then, 
we have $\delta _{j+1}=t_j\delta _j-\delta _{j-1}$ 
for $j=i-1,i,i+1$, 
since $i-1\geq 2$. 
Hence, we get 
$$
\delta _{i+2}-\delta _i
=(t_0t_1-4)\delta _i+(\delta _i-\delta _{i-2})
$$ 
by Lemma~\ref{lem:skip}. 
Since $t_0\geq 3$ and $(t_0,t_1)\neq (3,1)$ by assumption, 
we have $t_0t_1\geq 4$. 
By the induction assumption of (i) and (iv), 
we have $\delta _i>0$ and $\delta _i-\delta _{i-2}>0$, 
since $i\geq 3$. 
Thus, 
it follows that $\delta _{i+2}-\delta _i>0$, 
proving (iv). 
Therefore, 
(i) through (iv) hold for every $i$. 
\end{proof}

\section{Wildness (I)}
\label{sect:lscwildness1}
\setcounter{equation}{0}

In this section, 
we give a sufficient condition for wildness 
of certain exponential automorphisms. 
Assume that $i\geq 2$, 
and let $g,s\in \kx $ and $\nu (y)\in k[y]\sm \zs $ 
be such that $E:=\Delta _{(g,f_i)}$ is locally nilpotent, 
and 
$$
f_{i-1}g=q:=\tilde{\eta }_i(f_i,s')\nu (f_i)^{a_i},
\quad\text{where}\quad 
s':=s\nu (f_i)^{-1}. 
$$ 
Take any $h\in k[f_i,g]\sm \zs $, 
and set 
$$
\phi =\exp hE,\quad 
\alpha =\degw \phi (s)\quad\text{and}\quad 
\beta =\degw \phi (\tilde{r}). 
$$ 
Then, 
we have $\phi (f_i)=f_i$, 
since $E(f_i)=0$. 
For each $j\geq 0$, 
we define 
$$
\ep _j=\degw \phi (f_j)\quad
\text{and}\quad 
d(j)=t_j\ep _j-a_j\beta . 
$$
Then, 
we have $\ep _i=\delta _i$, 
since $\phi (f_i)=f_i$. 
Because $\phi (\tilde{r})$ and $\phi (f_j)$ 
are not constants, 
and all the components of $\w $ are positive, 
we have $\beta >0$, and $\ep _j>0$ for each $j\geq 0$.

In this situation, 
consider the following conditions: 

\smallskip 

\noindent(a) 
$\delta _i$ and $\delta :=\degw g$ 
are linearly independent. 

\noindent(b) 
Set $v=\deg _y\nu (y)$. 
Then, 
we have 
$\alpha \geq v\delta _2+\delta $ 
and $\delta >\delta _2$ if $i=2$, 
and $\alpha =(v+v_1)\delta _i+v_2\delta $ 
for some $v_1,v_2\in \N $ if $i\geq 3$.

\noindent(c) 
If $i=2$, then $\beta $ and $\ep _1$ are linearly independent. 
If $i\geq 3$, then we have $\beta \geq \alpha $. 

\noindent(d) 
If $i\geq 3$, then we have $d(i-1)\neq 0$. 

\noindent(e) 
If $i\geq 3$ and $d(i-1)<0$, 
then we have $\beta \geq \ep _{i-1}-v\delta _i$ 
and $\delta >a_iv\delta _i$, 
and $a_{i-1}\beta -\delta _i$ and $\ep _{i-1}$ 
are linearly independent. 

\smallskip 

We mention that (a) implies that $f_i$ and $g$ 
are algebraically independent over $k$, 
and hence implies that $E$ is nonzero.

In the notation above, 
we have the following theorem.

\begin{thm}\label{thm:wild general}
Assume that $i=2$ and $t_0\geq 3$, 
or $i\geq 3$, $t_0\geq 3$ and $(t_0,t_1)\neq (3,1)$. 
Let $g,s\in \kx $, $\nu (y)\in k[y]\sm \zs $ 
and $h\in k[f_i,g]\sm \zs $ be such that 
{\rm (a)} through {\rm (e)} are fulfilled. 
Then, $\phi =\exp hE$ is wild. 
\end{thm}

First, 
we show that (b) implies 
\begin{equation}\label{eq:b conseq}
\ep _{i-1}=a_i\alpha -\delta 
\end{equation}
when $i=2$ and $t_0\geq 3$, 
or $i\geq 3$, $t_0\geq 3$ and $(t_0,t_1)\neq (3,1)$. 
Since $E(g)=0$, we have 
$\phi (q)=\phi (f_{i-1}g)=\phi (f_{i-1})g$. 
Hence, 
we get $\degw \phi (q)=\ep _{i-1}+\delta $. 
Thus, 
it suffices to prove that 
$\degw \phi (q)=a_i\alpha $.

Since $\phi (f_i)=f_i$ and 
$\deg _y\nu (y)=v$, 
we have 
$$
\alpha '
:=\degw \phi (s')
=\degw \phi (s\nu (f_i)^{-1})
=\degw \phi (s)\nu (f_i)^{-1}
=\alpha -v\delta _i. 
$$
First, 
assume that $i=2$ and $t_0\geq 3$. 
Then, 
(b) implies that 
\begin{equation}\label{eq:i=2 alpha'}
\alpha '\geq (v\delta _2+\delta )-v\delta _2
=\delta >\delta _2. 
\end{equation}
Hence, 
$\degw \theta (\phi (s'))=(t_0-1)\alpha '$ 
is greater than $\degw f_2=\delta _2$. 
Since $\phi (f_2)=f_2$ and $t_0-1=a_2$, 
it follows that 
\begin{align*}
&\degw \phi (q)
=\degw \tilde{\eta }_2(f_2,\phi (s'))\nu (f_2)^{a_2}
=\degw \bigl( 
f_2+\theta (\phi (s'))
\bigr) \nu (f_2)^{a_2} \\
&\quad =(t_0-1)\alpha '+a_2v\delta _2
=a_2(\alpha '+v\delta _2)
=a_2\alpha . 
\end{align*}
This proves (\ref{eq:b conseq}).

Next, 
assume that 
$i\geq 3$, $t_0\geq 3$ and $(t_0,t_1)\neq (3,1)$. 
Then, 
we have $a_i>t_i$ by Lemma~\ref{lem:a_i} (iii), 
and (b) implies that 
$$
\alpha '=\bigl(
(v+v_1)\delta _i+v_2\delta \bigr) -v\delta _i
=v_1\delta _i+v_2\delta \geq \delta _i+\delta . 
$$
Hence, 
it follows that  
\begin{equation}\label{eq:alpha _1}
a_i\alpha '-t_i\delta _i
\geq a_i(\delta _i+\delta )-t_i\delta _i
=(a_i-t_i)\delta _i+a_i\delta >0. 
\end{equation}
Thus, 
we get $\degw \eta _i(f_i,\phi (s'))=a_i\alpha '$ 
by applying Lemma~\ref{lem:rhq} with 
$f=f_i$ and $p=\phi (s')$. 
Since $\tilde{\eta }_i(y,z)=\eta _i(y,z)$, 
we know that 
$$
\degw \phi (q)
=\degw \eta _i(f_i,\phi (s'))\nu (f_i)^{a_i}
=a_i\alpha '+a_iv\delta _i=a_i\alpha, 
$$
proving (\ref{eq:b conseq}). 
As a consequence, it follows that 
\begin{equation}\label{eq:ep _{i-1}}
\ep _{i-1}
=a_i\alpha -\delta 
=a_i(v+v_1)\delta _i+(a_iv_2-1)\delta , 
\end{equation}
since $\alpha =(v+v_1)\delta _i+v_2\delta $ by (b).

Now, 
let us prove Theorem~\ref{thm:wild general}. 
Recall the notion of W-test polynomial 
introduced in Section~\ref{sect:criterion} 
(see Definition~\ref{defn:W-test}). 
We show that $f_2$ is a W-test polynomial if $t_0\geq 3$ 
by means of Proposition~\ref{prop:criterion}. 
Take any totally ordered additive group $\Lambda $, 
and $\uu \in (\Lambda _{>0})^3$. 
Then, 
we have $\deg _{\uu }x_2^i<\deg _{\uu }x_2^{t_0-1}$ 
for $i=0,\ldots ,t_0-2$. 
Hence, 
$f_2^{\uu }$ must be 
$x_1x_3$ or $-x_2^{t_0-1}$ or $x_1x_3-x_2^{t_0-1}$. 
Since $t_0\geq 3$, 
we see that 
these three polynomials are not divisible by 
$x_i-g$ for any $i\in \{ 1,2,3\} $ 
and $g\in k[\x \sm \{ x_i\} ]\sm k$, 
and by $x_i^{s_i}-cx_j^{s_j}$ 
for any $i,j\in \{ 1,2,3\} $ with $i\neq j$, 
$s_i,s_j\in \N $ and $c\in k^{\times }$. 
Therefore, 
$f_2$ is a W-test polynomial 
by Proposition~\ref{prop:criterion}.

First, 
assume that  $i=2$ and $t_0\geq 3$. 
Then, 
we have 
$$
\degw \phi (x_1)=
\ep _1=a_2\alpha -\delta  
=(t_0-1)\alpha -\delta 
\geq 2\alpha '-\delta >\delta _2=\ep _2=\degw \phi (f_2) 
$$
by (\ref{eq:b conseq}) and (\ref{eq:i=2 alpha'}). 
Since 
$\degw \phi (x_2)=\degw \phi (\tilde{r})=\beta $, 
we know by (c) that 
$\degw \phi (x_1)$ and $\degw \phi (x_2)$ 
are linearly independent. 
Thus, 
we conclude that $\phi $ is wild 
because $f_2$ is a W-test polynomial. 
This proves 
Theorem~\ref{thm:wild general} 
when $i=2$ and $t_0\geq 3$.

Next, assume that 
$i\geq 3$, $t_0\geq 3$ and $(t_0,t_1)\neq (3,1)$. 
Then, 
by Proposition~\ref{prop:wild general} to follow, 
we know that 
$\degw \phi (f_2)=\ep _2$ 
is less than $\degw \phi (x_2)=\ep _0$, 
and $\degw \phi (x_2)=\ep _0$ and 
$\degw \phi (x_1)=\ep _1$ are linearly independent. 
Hence, we conclude that $\phi $ is wild similarly. 
Thus, 
the proof of 
Theorem~\ref{thm:wild general} is completed.

\begin{prop}\label{prop:wild general}
In the situation of Theorem~$\ref{thm:wild general}$, 
assume that 
$t_0\geq 3$, $(t_0,t_1)\neq (3,1)$ and $i\geq 3$. 
Set $i'=i$ if $d(i-1)>0$, 
and $i'=i-1$ if $d(i-1)<0$. 
Then, the following statements hold 
for each $l\in \{ 1,\ldots ,i'\} $$:$

\noindent{\rm (i)} 
We have $d(l-1)>0$, and 
$\ep _{l-1}$ and $\ep _l$ are linearly independent.

\noindent{\rm (ii)} 
If $l\neq i'$, 
then we have $\ep _{l-1}=t_l\ep _l-\ep _{l+1}$ 
and $\ep _{l-1}>\ep _{l+1}$. 
\end{prop}
\begin{proof}
First, 
we show that $d(l)>0$ implies 
$\ep _{l-1}=t_l\ep _l-\ep _{l+1}$ for $l\geq 1$. 
Since $i\geq 3$, 
we have $\tilde{r}=r$. 
Hence, we know that 
$t_l\degw \phi (f_l)>a_l\degw \phi (r)$ 
by the assumption that 
$d(l)=t_l\degw \phi (f_l)-a_l\degw \phi (\tilde{r})$ 
is positive. 
By applying Lemma~\ref{lem:rhq} 
with $f=\phi (f_l)$ and $p=\phi (r)$, 
we obtain 
$$
\degw \phi \bigl(\eta _l(f_l,r)\bigr)
=\degw \eta _l\bigl(\phi (f_l),\phi (r)\bigr)
=t_l\degw \phi (f_l)
=t_l\ep _l. 
$$ 
Since $\eta _l(f_l,r)=f_{l-1}f_{l+1}$, 
it follows that 
$$
\ep _{l-1}+\ep _{l+1}=\degw \phi (f_{l-1}f_{l+1})
=\degw \phi \bigl(\eta _l(f_l,r)\bigr)=t_l\ep _l. 
$$
Therefore, 
we get $\ep _{l-1}=t_l\ep _l-\ep _{l+1}$.

Similarly, 
if $d(i-1)<0$, then we have 
$$
\degw 
\phi \bigl(\eta _{i-1}(f_{i-1},r)\bigr)
=a_{i-1}\degw \phi (r)
=a_{i-1}\degw \phi (\tilde{r})
=a_{i-1}\beta
$$ 
by Lemma~\ref{lem:rhq}. 
Since $\degw 
\phi (\eta _{i-1}(f_{i-1},r))
=\phi (f_{i-2}f_i)=\ep _{i-2}+\ep _i$, 
this gives that 
$\ep _{i-2}
=a_{i-1}\beta -\ep _i
=a_{i-1}\beta -\delta _i$.

Now, we prove (i) and (ii) simultaneously 
by descending induction on $l$. 
When $l=i'$, 
we have only to check (i). 
Assume that $d(i-1)>0$ and $l=i'=i$. 
Then, the first part of (i) is obvious. 
Since 
$a_i\geq 2$ by Lemma~\ref{lem:a_i} (ii), 
and $v_2\geq 1$, 
we have $a_iv_2-1>0$. 
Hence, 
we see from (\ref{eq:ep _{i-1}}) that 
$\ep _{i-1}$ and $\ep _i=\delta _i$ 
are linearly independent by virtue of (a), 
proving the second part of (i). 
Thus, 
the statements hold. 
Assume that $d(i-1)<0$ and $l=i'=i-1$. 
Then, we have 
$\ep _{i-2}=a_{i-1}\beta -\delta _i$ 
as mentioned. 
Since $t_{i-2}=t_i$ and $t_ia_{i-1}-a_{i-2}=a_i$ 
by (\ref{eq:a_i zenkasiki}), 
it follows that 
\begin{align*}
&d(l-1)=d(i-2)=t_{i-2}\ep _{i-2}-a_{i-2}\beta 
=t_{i}(a_{i-1}\beta -\delta _i)-a_{i-2}\beta \\
&\quad =(t_{i}a_{i-1}-a_{i-2})\beta -t_{i}\delta _i
=a_i\beta -t_i\delta _i. 
\end{align*}
Since 
$\beta \geq \alpha $ by (c) 
and $\alpha \geq \alpha '$, 
we know that 
\begin{equation}\label{eq:inu}
a_i\beta -t_i\delta _i\geq a_{i}\alpha '-t_{i}\delta _{i}>0
\end{equation} 
by (\ref{eq:alpha _1}). 
Hence, we get $d(l-1)>0$. 
Since 
$\ep _{i-2}=a_{i-1}\beta -\delta _i$, 
it follows from (e) that 
$\ep _{l-1}=\ep _{i-2}$ and $\ep _l=\ep _{i-1}$ 
are linearly independent. 
Therefore, 
the statements are true.

Next, 
assume that $1\leq l\leq i'-1$. 
Then, 
we have $d(l)>0$ by induction assumption. 
This implies $\ep _{l-1}=t_l\ep _l-\ep _{l+1}$ as remarked. 
Hence, we get the first part of (ii). 
Since $\ep _{l}$ and $\ep _{l+1}$ 
are linearly independent by induction assumption, 
it follows that $\ep _{l-1}$ and $\ep _{l}$ 
are linearly independent, 
proving the latter part of (i).

We show that $\ep _{l-1}>\ep _{l+1}$. 
Note that 
\begin{equation}\label{eq:pf:wildgeneral:ep}
\ep _{j-1}=t_j\ep _j-\ep _{j+1}
\quad \text{for}\quad 
l\leq j\leq i'-1, 
\end{equation}
where the case $j=l$ is just mentioned above, 
and the case $l<j\leq i'-1$ 
is due to the induction assumption. 
Using this equality for $j=l$, we get 
\begin{equation}\label{eq:wild general ep 1}
\ep _{l-1}-\ep _{l+1}
=(t_l\ep _l-\ep _{l+1})-\ep _{l+1}
=t_{l}\ep _{l}-2\ep _{l+1}. 
\end{equation}

First, 
consider the case where $l=i'-1$. 
When $d(i-1)>0$, 
we have $l=i-1$. 
Since $a_i\geq 2$, $v_1\geq 1$ and $v_2\geq 1$, 
we know by (\ref{eq:ep _{i-1}}) that 
$\ep _{i-1}>2\delta _i=2\ep _i$. 
Hence, we have $\ep _l>2\ep _{l+1}$. 
Thus, 
we obtain $\ep _{l-1}>\ep _{l+1}$ 
from (\ref{eq:wild general ep 1}). 
When $d(i-1)<0$, 
we have $l=i-2$. 
Since 
$\ep _{i-2}=a_{i-1}\beta -\delta _i$ 
as remarked, 
it follows from (\ref{eq:wild general ep 1}) that 
\begin{equation}\label{eq:wild general ep 2}
\ep _{l-1}-\ep _{l+1}
=t_{i-2}\ep _{i-2}-2\ep _{i-1}
=t_{i}(a_{i-1}\beta -\delta _i)-2\ep _{i-1}. 
\end{equation}
Since $\beta \geq \ep _{i-1}-v\delta _i$ by (e), 
the right-hand side of this equality is at least 
$$
t_{i}\bigl(a_{i-1}(\ep _{i-1}-v\delta _i)-\delta _i\bigr)
-2\ep _{i-1}
=(t_{i}a_{i-1}-2)\ep _{i-1}-t_{i}(a_{i-1}v+1)\delta _i. 
$$
Note that (\ref{eq:ep _{i-1}}) 
implies 
$$
\ep _{i-1}
\geq a_i(v+1)\delta _i+\delta  
>a_i(v+1)\delta _i+a_iv\delta _i
=a_i(2v+1)\delta _i, 
$$
since $v_1\geq 1$, $a_iv_2-1\geq 1$, 
and $\delta >a_iv\delta _i$ by (e). 
Hence, 
the right-hand side of the preceding equality 
is greater than 
\begin{align}\begin{split}\label{eq:ep _{i-2}-ep _{i-1}}
&\bigl((t_{i}a_{i-1}-2)a_i(2v+1)
-t_{i}(a_{i-1}v+1)\bigr)\delta _i \\
&\quad =\bigl(
t_{i}a_{i-1}\bigl(a_i(2v+1)-v\bigr)
-2a_i(2v+1)-t_i
\bigr)\delta _i\\
&\quad =\bigl(
a_{i-2}\bigl(a_i(2v+1)-v\bigr)
+a_i\bigl((a_i-2)(2v+1)-v\bigr)
-t_i
\bigr)\delta _i, 
\end{split}\end{align}
where we use $t_ia_{i-1}=a_{i-2}+a_{i}$ 
for the last equality. 
When $a_i\geq 3$, 
we have $(a_i-2)(2v+1)-v\geq 1$. 
Since 
$a_i>t_i$ by Lemma~\ref{lem:a_i} (iii), 
and $a_i(2v+1)-v>0$, 
we see that the right-hand side of 
(\ref{eq:ep _{i-2}-ep _{i-1}}) is positive. 
When $a_i\leq 2$, 
we have $a_i=2$ and $i=3$ 
in view of (i) and (ii) of Lemma~\ref{lem:a_i}. 
Since 
$a_3=t_1(t_0-1)-1$, and 
$t_0\geq 3$ and $(t_0,t_1)\neq (3,1)$ by assumption, 
it follows that $(t_0,t_1)=(4,1)$. 
Hence, 
the right-hand side of 
(\ref{eq:ep _{i-2}-ep _{i-1}}) 
is equal to 
$$
\bigl(
\bigl(2(2v+1)-v\bigr)+2(-v)-t_3\bigr)
\delta _3=(v+1)\delta _3>0. 
$$
Thus, 
we know that (\ref{eq:wild general ep 2}) 
is positive. 
Therefore, 
we get $\ep _{l-1}>\ep _{l+1}$.

Next, 
consider the case where $l=i'-2$. 
In this case, 
we have 
$\ep _l>\ep _{l+2}$ by the induction assumption of (ii). 
By (\ref{eq:pf:wildgeneral:ep}), 
we have 
$\ep _{j-1}=t_{j}\ep _{j}-\ep _{j+1}$ 
for $j=l,l+1$. 
Hence, 
it follows from 
(\ref{eq:wild general ep 1}) that 
\begin{align*}
&\ep _{l-1}-\ep _{l+1}
=\frac{t_l}{2}\ep _{l}+\frac{t_l}{2}\ep _{l}-2\ep _{l+1}
=\frac{t_l}{2}\ep _l
+\frac{t_l}{2}
(t_{l+1}\ep _{l+1}-\ep _{l+2})-2\ep _{l+1} \\
&\quad =\frac{t_l}{2}(\ep _l-\ep _{l+2})
+\frac{1}{2}(t_lt_{l+1}-4)\ep _{l+1}
\geq \frac{t_l}{2}(\ep _l-\ep _{l+2})>0, 
\end{align*}
since $t_lt_{l+1}=t_0t_1\geq 4$. 
Therefore, 
we get $\ep _{l-1}>\ep _{l+1}$.

Finally, 
consider the case where $1\leq l\leq i'-3$. 
Since $l+2=i'-1$, 
we have $\ep _{j-1}=t_j\ep _j-\ep _{j+1}$ for $j=l,l+1,l+2$ 
by (\ref{eq:pf:wildgeneral:ep}), 
and $\ep _{l+1}>\ep _{l+3}$ by the induction assumption of (ii). 
By Lemma~\ref{lem:skip}, it follows that 
$$
\ep _{l-1}-\ep _{l+1}=(t_0t_1-4)\ep _{l+1}+(\ep _{l+1}-\ep _{l+3}) 
\geq \ep _{l+1}-\ep _{l+3}>0. 
$$
Therefore, 
we get $\ep _{l-1}>\ep _{l+1}$. 
This proves the second part of (ii).

It remains only to show that $d(l-1)>0$. 
Since $\ep _{l-1}=t_l\ep _l-\ep _{l+1}$ 
and $a_{l+1}=t_{l+1}a_l-a_{l-1}$, 
we have 
\begin{align*}
&d(l+1)+d(l-1)=
(t_{l+1}\ep _{l+1}-a_{l+1}\beta )
+(t_{l-1}\ep _{l-1}-a_{l-1}\beta )\\
&\quad =t_{l+1}(\ep _{l+1}+\ep _{l-1})-(a_{l+1}+a_{l-1})\beta  
=t_{l+1}t_{l}\ep _{l}-t_{l+1}a_{l}\beta
=t_{l+1}d(l). 
\end{align*}
Since $d(l)>0$ by induction assumption, 
it follows that $d(l-1)>-d(l+1)$. 
When $l=i'-1$ and $d(i-1)>0$, 
we have 
$$
d(l+1)=d(i')=d(i)=t_i\ep _i-a_i\beta =t_i\delta _i-a_i\beta <0
$$ 
by (\ref{eq:inu}). 
Hence, we get $d(l-1)>-d(l+1)>0$. 
When $l=i'-1$ and $d(i-1)<0$, 
we have $d(l+1)=d(i')=d(i-1)<0$. 
Hence, we get $d(l-1)>0$ similarly.

Assume that $1\leq l\leq i'-2$. 
Then, 
we have $d(l'-1)>0$ for $l'=i'-1,i'$ 
by induction assumption, 
and $l-1\leq i'-3<i'-1$ if $l\equiv i'\pmod{2}$, 
and $l-1\leq i'-4<i'-2$ otherwise. 
Note that $\ep _{j-1}>\ep _{j+1}$ for $l\leq j\leq i'-1$, 
where the case $j=l$ is just verified above, 
and the case $l<j\leq i'-1$ is due to the induction assumption. 
Hence, we know that 
$\ep _{l-1}>\ep _{i'-1}$ 
if $l\equiv i'\pmod{2}$, 
and $\ep _{l-1}>\ep _{i'-2}$ 
otherwise. 
On the other hand, 
we have $a_0=1<t_0-1=a_2$, 
and $a_{j-1}<a_{j+1}$ for $j\geq 2$ 
by Lemma~\ref{lem:a_i} (i). 
Hence, we similarly obtain that 
$a_{l-1}<a_{i'-1}$ if $l\equiv i'\pmod{2}$, 
and $a_{l-1}<a_{i'-2}$ otherwise. 
Since $d(l'-1)>0$ for $l'=i'-1,i'$, 
we have 
$a_{l'-1}\beta <t_{l'-1}\ep _{l'-1}$ 
for $l'=i'-1,i'$. 
Thus, 
it follows that 
$$
a_{l-1}\beta <a_{l'-1}\beta <t_{l'-1}\ep _{l'-1}
=t_{l-1}\ep _{l'-1}
<t_{l-1}\ep _{l-1}, 
$$
where $l':=i'$ if $l\equiv i'\pmod{2}$, 
and $l':=i'-1$ otherwise. 
Therefore, 
we get 
$d(l-1)=t_{l-1}\ep _{l-1}-a_{l-1}\beta >0$. 
This proves that (i) and (ii) 
hold for every $l$. 
\end{proof}

\section{Wildness (II)}
\label{sect:lscwildness2}
\setcounter{equation}{0}

Thanks to Theorem~\ref{thm:lsc1} (ii), 
the former case of 
Theorem~\ref{thm:lsc1} (iii) 
is reduced to the latter case. 
The latter case of Theorem~\ref{thm:lsc1} (iii) 
is divided into the following three cases: 

\smallskip 

\noindent
(w1) $t_0\geq 3$, $(t_0,t_1)\neq (3,1)$ and $i\geq 3$. 

\noindent
(w2) $(t_0,t_1,i)=(3,1,3), (3,1,4)$. 

\noindent
(w3) $t_0\geq 3$ and $i=2$. 

\smallskip 

We show that the case (w3) 
is contained in Theorem~\ref{thm:lsc2} (ii). 
Let $\lambda (y)=y$ 
and $\mu (y,z)=\sum _{j=1}^{t_1}\alpha _j^1z^j$. 
Then, we have 
$$
r_2=
\lambda (f_2)x_2-\mu (f_2,f_1)
=f_2x_2-\sum _{j=1}^{t_1}\alpha _j^1x_1^j=r
$$
by (\ref{eq:f_2}). 
Hence, we  get 
\begin{align*}
&f_1\tilde{f}_3
=\tilde{\eta }_2\left(
f_2,r_2f_2^{-1}
\right)f_2^{a_2}
=\tilde{\eta }_2\left(
f_2,rf_2^{-1}
\right)f_2^{a_2}
=\left(
f_2+\sum _{j=1}^{t_0}\alpha _j^0
\left(rf_2^{-1}
\right)^{j-1}\right) f_2^{t_0-1} \\
&\quad =f_2^{t_0}+\sum _{j=1}^{t_0}\alpha _j^0r^{j-1}f_2^{t_0-j}
=\eta _2(f_2,r)=f_1f_3. 
\end{align*}
Thus, 
it follows that $f_3=\tilde{f}_3$. 
Therefore, we conclude that $D_2=\tilde{D}_2$. 
Since $\lambda (y)=y$ does not belong to $k$, 
the wildness of $\exp hD_2$ follows from 
Theorem~\ref{thm:lsc2} (ii).

This section is devoted to proving 
Theorem~\ref{thm:lsc1} (iii) 
in the case of (w1), 
and Theorem~\ref{thm:lsc2} (ii). 
The case (w2) of Theorem~\ref{thm:lsc1} (iii) will be treated 
in Section~\ref{sect:exceptional}.

First, 
we prove Theorem~\ref{thm:lsc1} (iii) 
in the case of (w1). 
Take any $h\in \ker D_i\sm \zs $ 
and put $\phi =\exp hD_i$. 
By definition, 
we have $D_i=\Delta _{(f_{i+1},f_i)}$. 
Since $i\geq 3$, 
we have 
$$
f_{i-1}f_{i+1}=q_i=\eta _i(f_i,r)
=\tilde{\eta }_i(f_i,r). 
$$ 
Hence, 
$D_i$ is obtained by the construction 
stated before Theorem~\ref{thm:wild general} 
from the data $(g,s,\nu (y))=(f_{i+1},r,1)$. 
Since 
$\ker D_i=k[f_i,f_{i+1}]$ by Theorem~\ref{thm:lsc1} (i), 
$h$ belongs to $k[f_i,f_{i+1}]\sm \zs $. 
Therefore, 
it suffices to check (a) through (e) 
by virtue of Theorem~\ref{thm:wild general}.

Since $i\geq 3$, 
we have $s=r=\tilde{r}$. 
Hence, we get $\alpha =\beta $, 
proving (c). 
Since $\delta _i$ and $\delta _{i+1}$ 
are linearly independent by Proposition~\ref{prop:deltap} (ii), 
we get (a). 
Since $D_i(r)=f_if_{i+1}$ by Theorem~\ref{thm:lsc1} (i), 
we have 
$$
\phi (r)=r+hf_if_{i+1}. 
$$
We show that $\degw r<\degw hf_if_{i+1}$. 
Since $t_0\geq 3$ by assumption, 
$\degw r=(1,1,1)$ 
is less than $(t_0,0,t_0)=t_0\delta _2$. 
By Proposition~\ref{prop:delta}, 
we have 
$t_0\delta _2=\delta _1+\delta _3$. 
Since $\delta _1<\delta _2$ in view of (\ref{eq:delta vector}), 
we get $\delta _1+\delta _3<\delta _2+\delta _3$. 
By Proposition~\ref{prop:deltap} (iv), 
we know that 
$\delta _{j}+\delta _{j+1}<\delta _{j+1}+\delta _{j+2}$ 
for each $j\geq 1$. 
Hence, 
it follows that $\delta _2+\delta _3<
\delta _i+\delta _{i+1}$, 
since $i\geq 3$. 
Since $h\neq 0$, 
we have $\delta _i+\delta _{i+1}\leq \degw hf_if_{i+1}$. 
Thus, we conclude that 
$\degw r<\degw hf_if_{i+1}$. 
Therefore, we get 
$\alpha =\degw \phi (r)=\degw hf_if_{i+1}$.

Since $\degw f_i=\delta _i$ 
and $\degw f_{i+1}=\delta _{i+1}$ 
are linearly independent, 
we know that $f_i^{\w }$ and $f_{i+1}^{\w }$ 
are algebraically independent over $k$. 
Hence, 
we have 
$k[f_i,f_{i+1}]^{\w }=k[f_i^{\w },f_{i+1}^{\w }]$ 
by the discussion before Lemma~\ref{lem:minimal autom}. 
Since $h$ belongs to $k[f_i,f_{i+1}]\sm \zs $, 
it follows that 
$\degw h=v_1'\delta _i+v_2'\delta _{i+1}$ 
for some $v_1',v_2'\in \Zn $. 
Thus, we get 
$$
\alpha =\degw hf_if_{i+1}
=(v_1'+1)\delta _i+(v_2'+1)\delta _{i+1}. 
$$
Because $v=\deg _y\nu (y)=0$, 
we see that 
(b) holds for 
$v_j=v_j'+1$ for $j=1,2$. 
Consequently, 
we have 
$\ep _{i-1}=a_i\alpha -\delta _{i+1}$ 
by (\ref{eq:b conseq}). 
Since $t_{i+1}a_i=a_{i-1}+a_{i+1}$ 
by (\ref{eq:a_i zenkasiki}) 
and $\alpha >\delta _{i+1}$, 
it follows that 
$$
t_{i-1}\ep _{i-1}=t_{i+1}(a_i\alpha -\delta _{i+1})
=(a_{i-1}+a_{i+1})\alpha -t_{i+1}\delta _{i+1} 
>a_{i-1}\alpha +(a_{i+1}-t_{i+1})\delta _{i+1}. 
$$
Since 
$a_{i+1}>t_{i+1}$ by Lemma~\ref{lem:a_i} (iii), 
the right-hand side of this inequality is greater than 
$a_{i-1}\alpha=a_{i-1}\beta $. 
Thus, 
we get $d(i-1)=t_{i-1}\ep _{i-1}-a_{i-1}\beta >0$, 
proving (d) and (e). 
Therefore, (a) through (e) are fulfilled. 
This completes the proof of 
Theorem~\ref{thm:lsc1} (iii) 
in the case of (w1).

Next, 
we prove Theorem~\ref{thm:lsc2} (ii). 
First, 
we consider the case where $i=2$ and $t_0\geq 3$. 
Let $J$ be the set of $j\geq 1$ such that 
$u_j:=\deg _y\mu _j(y)$ 
is equal to $v:=\deg _y\lambda (y)$, 
and let $c$ and $c_j$ be 
the leading coefficients of $\lambda (y)$ 
and $\mu _j(y)$ for each $j\in J$, 
respectively. 
Since $\mu (y,z)$ is an element of $zk[y,z]$, 
and $\lambda (y)$ and $\mu (y,z)$ 
have no common factor by assumption, 
we see that 
$$
\bar{\mu }(y,z)
:=\mu (y,z)-c^{-1}\lambda (y)\sum _{j\in J}c_jz^j
$$ 
belongs to $zk[y,z]$, 
and $\lambda (y)$ and $\bar{\mu }(y,z)$ 
have no common factor. 
We show that $\bar{\mu }(y,z)=0$ if and only if
$\lambda (y)$ belongs to $k^{\times }$ 
and $\mu (y,z)$ belongs to $zk[z]$. 
If $\bar{\mu }(y,z)=0$, 
then we have 
$\mu (y,z)=
c^{-1}\lambda (y)\sum _{j\in J}c_jz^j$. 
Since $\lambda (y)$ and $\mu (y,z)$ 
have no common factor, 
it follows that $\lambda (y)$ belongs to $k^{\times }$, 
and so $\mu (y,z)$ belongs to $zk[z]$. 
Conversely, 
if $\lambda (y)$ belongs to $k^{\times }$ 
and $\mu (y,z)$ belongs to $zk[z]$, 
then we have $\lambda (y)=c$ 
and $\mu (y,z)=\sum _{j\in J}c_jz^j$. 
Hence, we get $\bar{\mu }(y,z)=0$.

Define $\tau \in J(k[x_1];x_2,x_3)$ by 
$$
\tau (x_2)=x_2+c^{-1}\sum _{j\in J}c_jx_1^{j}
\quad \text{and}\quad 
\tau (x_3)=x_3+
x_1^{-1}\bigl(\theta (\tau (x_2))-\theta (x_2)\bigr). 
$$ 
Then, 
we have 
$\tau (f_1)=\tau (x_1)=x_1=f_1$, 
$$
\tau (f_2)
=x_1\tau (x_3)-\theta (\tau (x_2))
=x_1x_3-\theta (x_2)=f_2
$$
and 
$$
\tau (r_2)=
\lambda (f_2)\tau (x_2)-\mu (f_2,x_1)
=\lambda (f_2)x_2-\bar{\mu }(f_2,x_1). 
$$
Hence, 
$\tau (\tilde{f}_3)$ is equal to the polynomial 
obtained similarly to $\tilde{f}_3$ 
from $\lambda (y)$ and $\bar{\mu }(y,z)$ 
instead of $\lambda (y)$ and $\mu (y,z)$. 
By the formula (\ref{eq:Jacobian}), 
we get 
\begin{equation}\label{eq:lsc2:wild:pf:jacobi}
D':=\Delta _{(\tau (\tilde{f}_3),f_2)}
=\Delta _{(\tau (\tilde{f}_3),\tau (f_2))}
=\tau \circ \tilde{D}_2\circ \tau ^{-1}, 
\end{equation}
since $\det J\tau =1$.

Now, 
take any $h\in \ker \tilde{D}_2\sm \zs $. 
Assume that $\lambda (y)$ belongs to $k^{\times }$ 
and $\mu (y,z)$ belongs to $zk[z]\sm \zs $. 
Then, 
we show that $\exp h\tilde{D}_2$ is tame if and only if 
$h$ belongs to $k[\tilde{f}_3]$. 
By assumption, 
it follows that $\lambda (y)=c$ and $\bar{\mu }(y,z)=0$. 
Hence, 
we have $\tau (\tilde{f}_3)=c^{a_2}x_3$ 
by (\ref{eq:mu=0}). 
Since $f_2$ is a symmetric polynomial 
in $x_1$ and $x_3$ over $k[x_2]$, 
we may regard $D'=c^{a_2}\Delta _{(x_3,f_2)}$ as 
$c^{a_2}\Delta _{(x_1,f_2)}=-c^{a_2}D_1$ 
by interchanging $x_1$ and $x_3$. 
Put $h'=\tau (h)$. 
Then, $h'$ belongs to $\ker D'\sm \zs $ 
by (\ref{eq:lsc2:wild:pf:jacobi}). 
Note that 
$\exp h\tilde{D}_2$ is tame if and only if 
$\exp h'D'$ is tame. 
By the discussion after Theorem~\ref{thm:lsc1}, 
it follows that 
$\exp h'D'$ is tame if and only if 
$h'$ belongs to $k[x_3]$ 
or $t_0\leq 2$, 
and hence if and only if 
$h'$ belongs to $k[x_3]$ 
by the assumption that $t_0\geq 3$. 
Since $k[x_3]=k[\tau (\tilde{f}_3)]$, 
we conclude that 
$\exp h\tilde{D}_2$ is tame if and only if 
$h$ belongs to $k[\tilde{f}_3]$. 
When this is the case, 
$h'D'$ is triangular if $x_1$ and $x_3$ 
are interchanged. 
Hence, 
$h\tilde{D}_2$ is tamely triangularizable, 
proving the last part of Theorem~\ref{thm:lsc2} (ii).

Take any $h\in \ker \tilde{D}_i\sm \zs $. 
To complete the proof of Theorem~\ref{thm:lsc2} (ii), 
it suffices to prove that 
$\exp h\tilde{D}_i$ is wild 
in the following cases: 

\noindent{\rm (1)} 
$i=2$, $t_0\geq 3$ and $\bar{\mu }(y,z)\neq 0$. 

\noindent{\rm (2)} 
$i\geq 3$, $t_0\geq 3$ and $(t_0,t_1)\neq (3,1)$. 

In the case of (1), 
we may replace $\mu (y,z)$ with $\bar{\mu }(y,z)$ 
because of (\ref{eq:lsc2:wild:pf:jacobi}). 
Therefore, 
we may assume that $u_j\neq v$ for each $j\geq 1$, 
since $u_j\neq v$ if $j$ does not belong to $J$, 
and $\deg _y(\mu _j(y)-c^{-1}c_j\lambda (y))<v$ 
otherwise.

By definition, 
we have 
$\tilde{D}_i=\Delta _{(\tilde{f}_{i+1},f_i)}$, 
and 
$$
f_{i-1}\tilde{f}_{i+1}=\tilde{q}_i
=\tilde{\eta _i}(f_i,r_i')\lambda (f_i)^{a_i}, 
$$
where $r_i':=r_i\lambda (f_i)^{-1}$. 
Hence, 
$\tilde{D}_i$ is obtained by the construction 
stated before Theorem~\ref{thm:wild general} from the data 
$(g,s,\nu (y))=(\tilde{f}_{i+1},r_i,\lambda (y))$. 
Since $\ker \tilde{D}_i=k[f_i,\tilde{f}_{i+1}]$ 
by Theorem~\ref{thm:lsc2} (i), 
$h$ belongs to $k[f_i,\tilde{f}_{i+1}]\sm \zs $. 
Therefore, 
it suffices to check (a) through (e) 
by virtue of Theorem~\ref{thm:wild general}.

Since $\mu (y,z)\neq 0$ by assumption, 
$J':=\{ j\geq 1\mid \mu _j(y)\neq 0\} $ 
is not empty. 
Since $\delta _{i-1}$ and $\delta _i$ 
are linearly independent by Proposition~\ref{prop:deltap} (ii), 
we see that 
$\degw \mu _j(f_i)f_{i-1}^j=u_j\delta _i+j\delta _{i-1}$'s 
are different for different elements $j$'s of $J'$. 
Hence, 
we may find $l_1\in J'$ such that 
$$
\degw \mu (f_i,f_{i-1})
=\degw \mu _{l_1}(f_i)f_{i-1}^{l_1}
=u_{l_1}\delta _i+l_1\delta _{i-1}. 
$$ 
By Proposition~\ref{prop:deltap} (v), 
the second components of 
$\delta _{i-1}$ and $\delta _i$ are zero, 
while $\degw \tilde{r}$ equals 
$(0,1,0)$ if $i=2$, 
and $(1,1,1)$ if $i\geq 3$. 
Because $\delta _{i-1}$ and $\delta _i$ 
are linearly independent, 
it follows that 
$\delta _{i-1}$, $\delta _i$ and $\degw \tilde{r}$ 
are linearly independent. 
Hence, 
$\degw \mu (f_i,f_{i-1})$ is not equal to 
$\degw \lambda (f_i)\tilde{r}=v\delta _i+\degw \tilde{r}$. 
Thus, we know that 
\begin{equation}\label{eq:r_i}
\degw r_i=
\degw \bigl( 
\lambda (f_i)\tilde{r}+\mu (f_i,f_{i-1})\bigr) 
=\max \{ v\delta _i+\degw \tilde{r},
u_{l_1}\delta _i+l_1\delta _{i-1}\} , 
\end{equation}
and so 
$$
\degw r_i'=\degw r_i-v\delta _i\geq \degw \tilde{r}>0. 
$$ 
Since 
$\delta _{i-1}$, $\delta _i$ and $\degw \tilde{r}$ 
are linearly independent, 
and $l_1\geq 1$, 
we see from (\ref{eq:r_i}) that 
$\delta _i$ and $\degw r_i$ are linearly independent. 
Therefore, 
$\delta _i$ and $\degw r_i'$ are linearly independent. 
In the case of (1), 
it follows that 
$$
\degw \tilde{\eta }_2(f_2,r_2')
=\degw \left( f_2+\theta (r_2')\right) 
=\max \{ \delta _2,a_2\degw r_2'\} , 
$$ 
since $\deg _z\theta (z)=t_0-1=a_2$. 
In the case of (2), 
we have $t_i\delta _i\neq a_i\degw r_i'$. 
Hence, 
we get 
$$
\degw \tilde{\eta }_i(f_i,r_i')
=\degw \eta _i(f_i,r_i')
=\max \{ t_i\delta _i,a_i\degw r_i'\} 
$$ 
by applying Lemma~\ref{lem:rhq} 
with $f=f_i$ and $p=r_i'$. 
Set $\tilde{t}_2=1$ in the case of (1), 
and $\tilde{t}_i=t_i$ in the case of (2). 
Then, 
we have 
\begin{align}\begin{split}\label{eq:tilde{delta}}
\delta :=\degw \tilde{f}_{i+1}
&=\degw \tilde{\eta }_i(f_i,r_i')
\lambda (f_i)^{a_i}f_{i-1}^{-1}\\
&=\max \{ \tilde{t}_i\delta _i,a_i\degw r_i'\} 
+a_iv\delta _i-\delta _{i-1}\\
&=\max \{ (\tilde{t}_i+a_iv)\delta _i,a_i\degw r_i\} 
-\delta _{i-1}. 
\end{split}\end{align}
Now, let us prove (a). 
From (\ref{eq:tilde{delta}}) and (\ref{eq:r_i}), 
we see that $\delta $ and $r_i$ 
have two possibilities. 
In the case where $\delta =a_i\degw r_i-\delta _{i-1}$ 
and $\degw r_i=u_{l_1}\delta _i+l_1\delta _{i-1}$, 
we have 
$$
\delta =a_iu_{l_1}\delta _i+(a_il_1-1)\delta _{i-1}. 
$$ 
Note that $a_2=t_0-1\geq 2$ in the case of (1), 
and $a_i\geq 2$ in the case of (2) by Lemma~\ref{lem:a_i} (ii). 
Hence, 
we have $a_il_1-1\geq a_i-1\geq 1$. 
Thus, we know that 
$\delta _i$ and $\delta $ are linearly independent, 
since so are $\delta _i$ and $\delta _{i-1}$. 
In the other cases, 
we can easily check that 
$\delta _i$ and $\delta $ are linearly independent 
because 
$\delta _{i-1}$, $\delta _i$ and $\degw \tilde{r}$ 
are linearly independent. 
Therefore, we get (a).

Recall that $u_j\neq v$ for each $j$ 
in the case of (1). 
Hence, 
we have $(u_{l_1},v)\neq (0,0)$. 
From (\ref{eq:r_i}), 
we see that 
$\degw r_2>\delta _2$. 
Since $a_2\geq 2$, 
and $\delta _1=(1,0,0)$ is less than $\delta _2=(1,0,1)$, 
it follows from (\ref{eq:tilde{delta}}) that 
$$
\delta \geq a_2\degw r_2-\delta _1
\geq \degw r_2+(\degw r_2-\delta _2)+(\delta _2-\delta _1)
>\degw r_2>\delta _2. 
$$
Therefore, 
we get the second part of (b) for $i=2$. 
Since $\tilde{D}_2(r_2)=\lambda (f_2)\tilde{f}_3$ 
by Theorem~\ref{thm:lsc2} (i), 
we have $\phi (r_2)=r_2+h\lambda (f_2)\tilde{f}_3$. 
By the preceding inequality, 
we know that 
$\degw h\lambda (f_2)\tilde{f}_3\geq \delta >\degw r_2$. 
Hence, we get 
$$
\alpha =\degw \phi (r_2)=\degw h\lambda (f_2)\tilde{f}_3
\geq \degw \lambda (f_2)\tilde{f}_3=v\delta _2+\delta , 
$$
proving the first part of (b) for $i=2$.

In the case of (2), 
it follows from (\ref{eq:tilde{delta}}) that 
$$
\delta \geq (t_i+a_iv)\delta _i-\delta _{i-1}
=(t_i\delta _i-\delta _{i-1})+a_iv\delta _i
=\delta _{i+1}+a_iv\delta _i>a_iv\delta _i, 
$$
since $t_i\delta _i-\delta _{i-1}=\delta _{i+1}$ 
by Proposition~\ref{prop:delta}. 
This proves the second part of (e).

Since $\degw f_i=\delta _i$ and $\degw \tilde{f}_{i+1}=\delta $ 
are linearly independent by (a), 
we have 
$k[f_i,\tilde{f}_{i+1}]^{\w }
=k[f_i^{\w },\tilde{f}_{i+1}^{\w }]$. 
Hence, 
we may write 
$$
\degw h=v_1'\delta _i+v_2'\delta , 
$$
where $v_1',v_2'\in \Zn $. 
In the case of (1), 
we have 
$\alpha =\degw h\lambda (f_2)\tilde{f}_3$ as mentioned. 
Hence, 
we get 
\begin{equation}\label{eq:alpha 2}
\alpha 
=\degw h\lambda (f_2)\tilde{f}_3
=(v_1'+v)\delta _2+(v_2'+1)\delta . 
\end{equation}

Consider the case of (2). 
Since 
$\tilde{D}_i(r_i)=\lambda (f_i)f_i\tilde{f}_{i+1}$ 
by Theorem~\ref{thm:lsc2} (i), 
we have 
$$
\phi (r_i)=r_i+h\lambda (f_i)f_i\tilde{f}_{i+1}. 
$$ 
We show that $(v+1)\delta _i+\delta $ 
is greater than $\degw r_i$. 
Then, 
it follows that 
$$
\degw h\lambda (f_i)f_i\tilde{f}_{i+1}\geq 
(v+1)\delta _i+\delta >\degw r_i. 
$$ 
Consequently, 
we obtain 
\begin{equation}\label{eq:alpha 3}
\alpha =\degw \phi (r_i)=\degw h\lambda (f_i)f_i\tilde{f}_{i+1}
=(v_1'+v+1)\delta _i+(v_2'+1)\delta , 
\end{equation}
and thereby proving that (b) for $i\geq 3$ 
holds with $v_1=v_1'+1$ and $v_2=v_2'+1$. 
First, 
assume that $a_i\degw r_i'<t_i\delta _i$. 
Then, 
we have 
$$
\degw r_i=\degw r_i'+v\delta _i
<\left(\frac{t_i}{a_i}+v\right) 
\delta _i<(1+v)\delta _i<(v+1)\delta _i+\delta , 
$$ 
since $t_i<a_i$ by Lemma~\ref{lem:a_i} (iii). 
Hence, the assertion holds. 
Next, 
assume that $a_i\degw r_i'\geq t_i\delta _i$. 
Since $\degw r_i'$ and $\delta _i$ are linearly independent, 
it follows that $a_i\degw r_i'>t_i\delta _i$, 
and so $a_i\degw r_i>(t_i+a_iv)\delta _i$. 
Hence, 
we know that 
$\delta =a_i\degw r_i-\delta _{i-1}$ 
by (\ref{eq:tilde{delta}}). 
Thus, 
we have 
\begin{align*}
&(v+1)\delta _i+\delta -\degw r_i 
=(v+1)\delta _i+(a_i-1)\degw r_i-\delta _{i-1}\\
&\quad >\left(
v+1+(a_i-1)\frac{t_i+a_iv}{a_i}\right)
\delta _i-\delta _{i-1}\\
&\quad =\left(1+a_iv-\frac{t_i}{a_i}\right)
\delta _i+t_i\delta _i-\delta _{i-1}
>t_i\delta _i-\delta _{i-1}=\delta _{i+1}>0. 
\end{align*}
Therefore, 
we get $(v+1)\delta _i+\delta >\degw r_i$. 
This proves (\ref{eq:alpha 3}), 
and (b) for $i\geq 3$.

Recall that (b) implies (\ref{eq:b conseq}). 
By (\ref{eq:alpha 2}) and (\ref{eq:alpha 3}), 
it follows that 
$\ep _{i-1}=a_i\alpha -\delta $ and $\delta _i$ are linearly independent, 
since $\delta $ and $\delta _i$ are linearly independent, 
and $a_i(v_2'+1)-1\geq 1$. 
Hence, 
$\degw \mu _j(f_i)\phi (f_{i-1})^j=u_j\delta _i+j\ep _{i-1}$'s 
are different for different elements $j$'s of $J'$. 
Thus, 
we may find $l_2\in J'$ such that 
$$
\gamma :=\degw \mu \bigl(f_i,\phi (f_{i-1})\bigr)
=\degw \mu _{l_2}(f_i)\phi (f_{i-1})^{l_2}
=u_{l_2}\delta _i+l_2\ep _{i-1}. 
$$ 
We show that $\gamma >\alpha $. 
If $\alpha >\delta $, 
then we have 
$$
\gamma \geq 
l_2\ep _{i-1}\geq \ep _{i-1}=a_i\alpha -\delta 
>(a_i-1)\alpha \geq \alpha . 
$$
In particular, 
we have $\gamma >\alpha $ in the case of (2), 
since $\alpha >\delta $ by (\ref{eq:alpha 3}). 
Assume that $\alpha \leq \delta $ in the case of (1). 
Then, we see from 
(\ref{eq:alpha 2}) that 
$v_1'=v_2'=v=0$ and $\alpha =\delta $. 
Since $u_{l_2}\neq v$ by assumption, 
it follows that $u_{l_2}>0$. 
Hence, we get $\gamma \geq \delta _2+\ep _{1}$. 
Since $\alpha =\delta $ and $a_2\geq 2$, 
we have $\ep _{1}=a_2\alpha -\delta \geq \alpha $. 
Thus, 
we conclude that $\gamma >\alpha $. 
Therefore, 
$\degw \mu \bigl(f_i,\phi (f_{i-1})\bigr)=\gamma $ 
is greater than $\degw \phi (r_i)=\alpha $. 
On the other hand, 
we have 
\begin{equation}\label{eq:wildII:pf:phi(r_i)}
\phi (r_i)=\phi \bigl(
\lambda (f_i)\tilde{r}-\mu (f_i,f_{i-1})\bigr)
=\lambda (f_i)\phi (\tilde{r})
-\mu \bigl(f_i,\phi (f_{i-1})\bigr), 
\end{equation}
since $\phi (f_i)=f_i$. 
This implies that 
$\degw \lambda (f_i)\phi (\tilde{r})=
\degw \mu \bigl(f_i,\phi (f_{i-1})\bigr)=\gamma $. 
Hence, we get 
\begin{equation}\label{eq:beta}
\beta =\degw \phi (\tilde{r})
=\gamma -\degw \lambda (f_i)
=(u_{l_2}-v)\delta _i+l_2\ep _{i-1}. 
\end{equation}
From this, 
we know that $\beta \geq \ep _{i-1}-v\delta _i$. 
This proves the first part of (e). 
In the case of (2), 
we have $a_{i-1}\geq 2$ by Lemma~\ref{lem:a_i} (ii), 
since $i-1\geq 2$. 
Hence, we get 
$a_{i-1}(u_{l_2}-v)-1\neq 0$. 
Because 
$$
a_{i-1}\beta -\delta _i
=\bigl(a_{i-1}(u_{l_2}-v)-1\bigr)\delta _i
+a_{i-1}l_2\ep _{i-1}
$$
by (\ref{eq:beta}), 
and $\delta _i$ and $\ep _{i-1}$ 
are linearly independent as mentioned, 
it follows that $a_{i-1}\beta -\delta _i$ 
and $\ep _{i-1}$ are linearly independent, 
proving 
the last part of (e).

In the case of (1), 
we have $u_{l_2}\neq v$. 
Since $\delta _2$ and $\ep _1$ are linearly independent, 
we know by (\ref{eq:beta}) 
that $\beta $ and $\ep _1$ are linearly independent, 
proving (c) for $i=2$. 
In the case of (2), 
it follows from 
(\ref{eq:beta}), (\ref{eq:b conseq}) and (\ref{eq:alpha 3}) 
that 
\begin{align*}
&\beta -\alpha 
\geq (\ep _{i-1}-v\delta _i)-\alpha 
=(a_i-1)\alpha -\delta -v\delta _i\geq 
\alpha -\delta -v\delta _i \\
&\quad 
\geq \bigl((v+1)\delta _i+\delta \bigr)
-\delta -v\delta _i=\delta _i>0. 
\end{align*}
This proves (c) for $i\geq 3$.

Finally, 
we prove (d) by contradiction. 
Suppose that $d(i-1)=0$. 
Then, 
we have 
$$
t_{i-1}\ep _{i-1}=a_{i-1}\beta 
=a_{i-1}\bigl((u_{l_2}-v)\delta _i+l_2\ep _{i-1}\bigr)
$$ 
by (\ref{eq:beta}). 
Since $\ep _{i-1}$ and $\delta _i$ are linearly 
independent, 
it follows that $t_{i-1}=a_{i-1}l_2$. 
Since $i-1\geq 2$, 
this contradicts Lemma~\ref{lem:a_i} (v), 
proving (d). 
Thus, 
we have verified (a) through (e). 
Therefore, 
$\phi $ is wild in the case of (1) and (2) 
by Theorem~\ref{thm:wild general}. 
This completes the proof of 
Theorem~\ref{thm:lsc2} (ii).

\section{Exceptional case}
\label{sect:exceptional}
\setcounter{equation}{0}

The goal of this section is to 
complete the proof of Theorem~\ref{thm:lsc1} 
by proving (a) of (i) in the case of $(t_0,t_1,i)=(3,1,4)$, 
and (ii) in the case of (w2). 
Assume that $(t_0,t_1)=(3,1)$. 
Then, 
we have 
$b_2=t_1b_1-b_0+\xi _1=1$, 
$b_3=t_2b_2-b_1+\xi _2=2$ 
and $b_4=t_3b_3-b_2+\xi _3=0$, 
since $b_0=b_1=0$. 
Hence, 
we see from (\ref{eq:irekae}) that 
\begin{equation}\label{eq:(t_0,t_1)=(3,1)}
\begin{aligned}
f_0f_2&=q_1=r+f_1\\
f_1f_3&=q_2=f_2^3r^{-1}\bigl( 
r+\theta _0(f_2^{-1}r)
\bigr) 
=f_2^3+r^2+\alpha _2^0f_2r+\alpha _1^0f_2^2\\
f_2f_4&=q_3=f_3r^{-1}\bigl( 
r+f_3^{-1}r^{2}
\bigr) 
=f_3+r\\
f_3f_5&=q_4=r+\theta (f_4)f_4, 
\end{aligned}
\end{equation}
since 
$\theta (z)=z^2+\alpha _2^0z+\alpha _1^0$, 
$\theta _0(z)=\theta (z)z$ 
and $\theta _1(z)=z$. 
By (\ref{eq:f_2}), 
we have $r=x_2f_2-x_1$. 
Hence, 
the second equality of 
(\ref{eq:(t_0,t_1)=(3,1)}) gives that 
\begin{align*}
f_1f_3&=f_2^3+(x_2f_2-x_1)^2+\alpha _2^0(x_2f_2-x_1)f_2
+\alpha _1^0f_2^2\\
&=x_1^2-x_1(2x_2+\alpha _2^0)f_2
+(f_2+x_2^2+\alpha _2^0x_2+\alpha _1^0)f_2^2. 
\end{align*}
Since $f_1=x_1$ and 
$f_2+x_2^2+\alpha _2^0x_2+\alpha _1^0
=f_2+\theta (x_2)
=x_1x_3$ 
by (\ref{eq:f_2}), 
it follows that 
\begin{equation}\label{eq:f_3}
f_3=x_1-(2x_2+\alpha _2^0)f_2+x_3f_2^2. 
\end{equation}
Thus, 
we know by the third equality of 
(\ref{eq:(t_0,t_1)=(3,1)}) that 
$$
f_2f_4=r+f_3
=(x_2f_2-x_1)+(x_1-(2x_2+\alpha _2^0)f_2+x_3f_2^2)
=f_2(x_3f_2-x_2-\alpha _2^0). 
$$
Therefore, 
we get 
$$
f_4=x_3f_2-x_2-\alpha _2^0. 
$$

With the notation of Chapter~\ref{chapter:atcoord}, 
define $\tilde{D}=D_{-\theta }$. 
Then, 
we have $\tilde{D}(f_2)=0$, 
since $f_2=x_1x_3-\theta (x_2)=f_{-\theta }$. 
Hence, 
$-f_2\tilde{D}$ belongs to $\lnd _k\kx $. 
Set $\tilde{\sigma }=\exp (-f_2\tilde{D})$ 
and $y_i=\tilde{\sigma }(x_i)$ for $i=1,2,3$. 
Then, 
we have 
\begin{equation}\label{eq:y_3y_2}
y_3=x_3\quad \text{and}\quad 
y_2=x_2-f_2x_3=-f_4-\alpha _2^0, 
\end{equation}
since $\tilde{D}(x_3)=0$, 
and $\tilde{D}(x_2)=x_3$ 
and $\tilde{D}(f_2)=0$. 
Since $f_2$ is fixed under $\tilde{\sigma }$, 
we see that 
$f_2=x_1x_3-\theta (x_2)$ 
is equal to 
$\tilde{\sigma }(f_2)=y_1x_3-\theta (x_2-f_2x_3)$. 
From this, 
we obtain 
$$
y_1=x_1+
\bigl( \theta (x_2-f_2x_3)-\theta (x_2)\bigr)x_3^{-1}. 
$$
Hence, we get 
\begin{equation}\label{eq:excep:y_1}
y_1=x_1-\theta '(x_2)f_2+\dfrac{1}{2}\theta ''(x_2)f_2^2x_3=f_3, 
\end{equation}
since 
$\theta '(x_2)=2x_2+\alpha _2^0$ and $\theta ''(x_2)=2$. 
Thus, 
we can define $\sigma _3\in \Aut (\kx /k)$ by 
\begin{equation}\label{eq:sigma _3}
\sigma _3(x_1)=y_1=f_3,\quad 
\sigma _3(x_2)=-y_2-\alpha _2^0=f_4,\quad 
\sigma _3(x_3)=y_3=x_3.
\end{equation}
Note that the sum of the two roots of 
$\theta (z)$ is equal to $-\alpha _2^0$. 
Hence, $\theta (z)$ and 
$\theta (-z-\alpha _2^0)$ have exactly the same roots. 
Since $\theta (z)$ and 
$\theta (-z-\alpha _2^0)$ are monic polynomials, 
we conclude that $\theta (-z-\alpha _2^0)=\theta (z)$. 
Thus, we have 
$$
\sigma _3(f_2)
=y_1y_3-\theta (-y_2-\alpha _2^0)
=y_1y_3-\theta (y_2)
=\tilde{\sigma }(f_2)=f_2. 
$$
Therefore, 
$\sigma _3$ belongs to $\Aut (\kx /k[f_2,x_3])$. 
Since $x_1x_3=f_2+\theta (x_2)$ by (\ref{eq:f_2}), 
we know that 
$$
f_3x_3=\sigma _3(x_1x_3)
=\sigma _3(f_2+\theta (x_2))
=f_2+\theta (f_4). 
$$
Hence, 
it follows that 
\begin{align*}
f_3(x_3f_4-1-f_5)
&=(f_3x_3)f_4-(f_3+r)-(f_3f_5-r)\\
&=(f_2+\theta (f_4))f_4
-f_2f_4-\theta (f_4)f_4 
=0 
\end{align*}
by the last two equalities of 
(\ref{eq:(t_0,t_1)=(3,1)}). 
Therefore, 
we get $f_5=x_3f_4-1$.

Now, 
let us prove (a) of Theorem~\ref{thm:lsc1} (i) 
in the case of $(t_0,t_1,i)=(3,1,4)$. 
Since 
\begin{equation}\label{eq:D_4}
D_4=\Delta _{(f_5,f_4)}=\Delta _{(x_3f_4-1,f_4)}
=f_4\Delta _{(x_3,f_4)}, 
\end{equation}
we see that 
$D_4$ is not irreducible, 
and $x_3$ belongs to $\ker D_4$. 
By (\ref{eq:sigma _3}), 
we have $\sigma _3^{-1}(x_3)=x_3$, 
$\sigma _3^{-1}(f_4)=x_2$ 
and $\sigma _3^{-1}(f_5)=\sigma _3^{-1}(x_3f_4-1)=x_2x_3-1$. 
Hence, we know that 
$\sigma _3^{-1}(x_3)=x_3$ does not belong to 
$\sigma _3^{-1}(k[f_4,f_5])=k[x_2,x_2x_3-1]$. 
Thus, 
$x_3$ does not belong to $k[f_4,f_5]$. 
Therefore, 
we conclude that $\ker D_4\neq k[f_4,f_5]$. 
This proves (a) of Theorem~\ref{thm:lsc1} (i) 
when $(t_0,t_1,i)=(3,1,4)$.

Next, 
we prove Theorem~\ref{thm:lsc1} (iii) 
in the case of $(t_0,t_1,i)=(3,1,4)$. 
Take any $h\in \ker D_4\sm \zs $. 
We show that 
$\exp hD_4=\exp hf_4\Delta _{(x_3,f_4)}$ is wild. 
Since 
\begin{gather*}
\Delta _{(x_3,f_4)}(x_1)
=-\frac{\partial f_4}{\partial x_2}
=1-x_3\frac{\partial f_2}{\partial x_2}
=1+\theta '(x_2)x_3 \\
\Delta _{(x_3,f_4)}(x_2)
=\frac{\partial f_4}{\partial x_1}
=x_3\frac{\partial f_2}{\partial x_1}
=x_3^2 \quad\text{and}\quad 
\Delta _{(x_3,f_4)}(x_3)=0, 
\end{gather*}
we see that 
$\Delta _{(x_3,f_4)}$ is triangular 
if $x_1$ and $x_3$ are interchanged. 
Since $hf_4$ belongs to $\ker \Delta _{(x_3,f_4)}\sm k[x_3]$, 
and 
$\partial (\Delta _{(x_3,f_4)}(x_1))/\partial x_2
=\theta ''(x_2)x_3=2x_3$ 
is not divisible by $\Delta _{(x_3,f_4)}(x_2)=x_3^2$, 
we conclude from Theorem~\ref{thm:triangular3} 
that $\exp hf_4\Delta _{(x_3,f_4)}$ is wild. 
Therefore, 
Theorem~\ref{thm:lsc1} (iii) 
is true when $(t_0,t_1,i)=(3,1,4)$.

The rest of this section is devoted to 
proving Theorem~\ref{thm:lsc1} (iii) 
in the case of $(t_0,t_1,i)=(3,1,3)$. 
Take any $h\in \ker D_3\sm \zs $ 
and put $\phi =\exp hD_3$. 
Since $q_3=r+f_3$ by (\ref{eq:(t_0,t_1)=(3,1)}), 
we have 
$D_3(f_2)=(\partial q_3/\partial r)f_3=f_3$ 
by (\ref{eq:D_i(f_{i-1})}) with $j=3$. 
Hence, 
we get 
$$
\phi (f_2)=f_2+hf_3. 
$$
First, 
we describe $z_i:=\phi (x_i)$ for $i=1,2,3$. 
By (\ref{eq:y_3y_2}) and (\ref{eq:excep:y_1}), 
we have 
$$
D_3=\Delta _{(f_4,f_3)}
=\Delta _{(-y_2-\alpha _2^0,y_1)}
=\Delta _{(y_1,y_2)}. 
$$
Since $x_3=y_3$, 
it follows that 
$D_3(x_3)
=\Delta _{(y_1,y_2)}(y_3) 
=\det J\tilde{\sigma }$.

Here, 
we remark that $\det J(\exp D)=1$ 
for any $D\in \lnd _k\kx $. 
This is verified as follows. 
Let $R=k[t]$ 
be the polynomial ring in one variable over $k$. 
Then, 
$D$ naturally extends to an element $\bar{D}$ of $\lnd _R\Rx $. 
Since 
$\Psi :=\exp t\bar{D}$ belongs to $\Aut (\Rx /R)$, 
we know that 
$\det J\Psi $ belongs to $R^{\times }=k^{\times }$. 
Put $R_0=R/(t)$ and $R_1=R/(t-1)$, 
and define an automorphism of 
$R_i\otimes _R\Rx =R_i[\x ]=\kx $ over $R_i=k$ 
by $\psi _i=\id _{R_i}\otimes \Psi $ for $i=0,1$. 
Then, 
we have $\psi _0=\id _{\kx }$ and $\psi _1=\exp D$. 
Hence, 
the images of $\det J\Psi $ in $R_0[\x ]$ 
and $R_1[\x ]$ are $\det J(\id _{\kx })=1$ 
and $\det J(\exp D)$, respectively. 
Since $\det J\Psi $ belongs to $k^{\times }$, 
it follows that $\det J(\exp D)=\det J\Psi =1$.

By the remark, 
we know that 
$\det J\tilde{\sigma }=\det J(\exp (-f_2\tilde{D}))=1$. 
Thus, we conclude that $D_3(x_3)=1$. 
Therefore, we get 
$$
z_3=x_3+h. 
$$
Since 
$D_3=\Delta _{(y_1,y_2)}$ kills $y_1$ and $y_2$, 
we have $\phi (y_i)=y_i$ for $i=1,2$. 
Hence, it follows from 
(\ref{eq:y_3y_2}) that 
$$
y_2=\phi (y_2)=\phi (x_2-f_2x_3)=z_2-\phi (f_2)z_3. 
$$
Thus, we get 
\begin{equation}\label{eq:excep:z_2}
z_2=y_2+\phi (f_2)z_3=y_2+(f_2+hy_1)(x_3+h).  
\end{equation}
Since $f_2=x_1x_3-\theta (x_2)$ is equal to 
$\tilde{\sigma }(f_2)=y_1x_3-\theta (y_2)$, 
we know that 
$\phi (f_2)=z_1z_3-\theta (z_2)$ is equal to 
$\phi (\tilde{\sigma }(f_2))=y_1z_3-\theta (y_2)$. 
Therefore, we have 
\begin{equation}\label{eq:excep:z_1}
\begin{aligned}
z_1
&=y_1+\bigl( \theta (z_2)-\theta (y_2)\bigr) z_3^{-1} \\
&=y_1+\Bigl( \theta \bigl(y_2+\phi (f_2)z_3\bigr)
-\theta (y_2)\Bigr) z_3^{-1}\\
&=y_1+(2y_2+\alpha _2^0)\phi (f_2)
+\phi (f_2)^2z_3 \\
&=y_1+(2y_2+\alpha _2^0)(f_2+hy_1)
+(f_2+hy_1)^2(x_3+h). 
\end{aligned}
\end{equation}

We show that $y_2$ is a W-test polynomial. 
Take any totally ordered additive group $\Lambda $ 
and $\vv \in (\Lambda _{>0})^3$. 
Then, 
$f_2^{\vv }$ is equal to 
$x_1x_3$ or $-x_2^{t_0-1}$ or $x_1x_3-x_2^{t_0-1}$ 
as mentioned before Proposition~\ref{prop:wild general}, 
since $t_0=3$. 
Hence, 
we have 
$$
y_2^{\vv }=(x_2-f_2x_3)^{\vv }=-f_2^{\vv }x_3, 
$$
and is equal to 
$-x_1x_3^2$ or $x_2^{t_0-1}x_3$ or $-(x_1x_3-x_2^{t_0-1})x_3$. 
Thus, we know that $y_2^{\vv }$ 
is not divisible by 
$x_i-g$ for any $i\in \{ 1,2,3\} $ and 
$g\in k[\x \sm \{ x_i\} ]\sm k$, 
and by $x_i^{s_i}-cx_j^{s_j}$ 
for any $i,j\in \{ 1,2,3\} $ with $i\neq j$, 
$s_i,s_j\in \N $ and $c\in k^{\times }$. 
Therefore, 
$y_2$ is a W-test polynomial 
due to Proposition~\ref{prop:criterion}.

Recall that $\Gamma =\Z ^3$ 
has the lexicographic order with 
$\e _1<\e _2<\e _3$, 
and $\w =(\e _1,\e _2,\e_ 3)$. 
Hence, 
we have $f_2^{\w }=x_1x_3$. 
By (\ref{eq:excep:y_1}), 
(\ref{eq:f_3}) and (\ref{eq:y_3y_2}), 
it follows that 
\begin{equation}\label{eq:y head}
y_1^{\w }
=f_3^{\w }
=(x_3f_2^2)^{\w }=x_1^2x_3^3
\quad \text{and}\quad 
y_2^{\w }=
(-f_2x_3)^{\w }=-x_1x_3^2. 
\end{equation}
Hence, 
we get $\degw y_1>\degw y_2>\degw f_2$. 
In view of (\ref{eq:excep:z_2}) and (\ref{eq:excep:z_1}), 
we see that  
\begin{equation}\label{eq:z head}
z_1^{\w}=(h^2y_1^2)^{\w }(x_3+h)^{\w }
\quad \text{and}\quad 
z_2^{\w }=(hy_1)^{\w }(x_3+h)^{\w }. 
\end{equation}
Thus, 
we have $\degw z_1>\degw y_1$. 
Since $\phi (y_2)=y_2$, 
it follows that 
$$
\degw \phi (x_1)=\degw z_1>\degw y_1>\degw y_2=\degw \phi (y_2). 
$$
Therefore, 
if $z_i^{\w }$ and $z_j^{\w }$ 
are algebraically independent over $k$ 
for some $i,j\in \{ 1,2,3\} $, 
then we may conclude that $\phi $ is wild, 
because $y_2$ is a W-test polynomial.

Assume that $h$ belongs to $k^{\times }$. 
Then, 
we have $(x_3+h)^{\w }=x_3$. 
Hence, 
we get $z_1^{\w }\approx (y_1^{\w })^2x_3$ 
and $z_2^{\w }\approx y_1^{\w }x_3$ 
by (\ref{eq:z head}). 
Since $y_1^{\w }$ and $x_3$ 
are algebraically independent over $k$ 
by (\ref{eq:y head}), 
it follows that $z_1^{\w }$ and $z_2^{\w }$ 
are algebraically independent over $k$. 
Therefore, 
we conclude that $\phi $ is wild.

Assume that $h$ does not belong to $k^{\times }$. 
Then, 
$h^{\w }$ does not belong to $k^{\times }$. 
Since $\max I=4>3$, 
we have $\ker D_3=k[f_3,f_4]$ 
by Theorem~\ref{thm:lsc1} (i). 
By (\ref{eq:y_3y_2}) and (\ref{eq:excep:y_1}), 
we see that $k[f_3,f_4]=k[y_1,y_2]$. 
Because $y_1^{\w }$ and $y_2^{\w }$ 
are algebraically independent over $k$ 
by (\ref{eq:y head}), 
we have $k[y_1,y_2]^{\w }=k[y_1^{\w },y_2^{\w }]$. 
Thus, 
$h^{\w }$ belongs to $k[y_1^{\w },y_2^{\w }]\sm k$. 
This implies that 
$\degw h\geq \degw y_i$ for some $i\in \{ 1,2\} $. 
Since $\degw y_1>\degw y_2>\degw x_3$, 
it follows that $(x_3+h)^{\w }=h^{\w }$. 
Therefore, we get 
$z_1^{\w }=(h^3y_1^2)^{\w }$ and 
$z_2^{\w }=(h^2y_1)^{\w }$ from (\ref{eq:z head}).

Assume that $h^{\w }$ belongs to 
$k[y_1^{\w },y_2^{\w }]\sm k[y_1^{\w }]$. 
Then, 
$h^{\w }$ and $y_1^{\w }$ 
are algebraically independent over $k$, 
since so are $y_1^{\w }$ and $y_2^{\w }$. 
This implies that 
$z_2^{\w }$ and $z_3^{\w }$ 
are algebraically independent over $k$. 
Therefore, we conclude that $\phi $ is wild.

In the rest of this section, 
we consider the case where 
$h^{\w }$ belongs to $k[y_1^{\w }]\sm k$. 
In this case, 
there exists $l\in \N $ such that $h^{\w }\approx (y_1^{\w })^l$. 
Then, 
we have 
\begin{equation}\label{eq:z head3}
z_1^{\w }\approx (y_1^{\w })^{3l+2},\ \ 
z_2^{\w }\approx (y_1^{\w })^{2l+1},\ \ 
z_3^{\w }\approx (y_1^{\w })^l,\ \ 
\phi (f_2)^{\w }\approx (y_1^{\w })^{l+1}. 
\end{equation}
Put $S=\{ y_1,y_2\} $. 
Then, 
we have $\degw ^Sh=\degw h=l\degw y_1$, 
since $y_1^{\w }$ and $y_2^{\w }$ 
are algebraically independent over $k$.

First, assume that $l=1$. 
Then, 
we have $\degw ^Sh=\degw y_1<\degw y_2^2$. 
Hence, 
we may write 
$$
h=a_{1,0}y_1+a_{0,1}y_2+a_{0,0}, 
$$ 
where $a_{1,0},a_{0,1},a_{0,0}\in k$ 
with $a_{1,0}\neq 0$. 
We define $\phi '\in \Aut (\kx /k)$ 
by 
$$
\phi '(x_1)=a_{1,0}^2z_1+(z_3^3-2a_{1,0}z_2)z_3^2,\quad 
\phi '(x_2)=a_{1,0}z_2-z_3^3,\quad 
\phi '(x_3)=z_3. 
$$
Then, 
$\phi ^{-1}\circ \phi '$ belongs to $J(k[x_3];x_2,x_1)$. 
Hence, 
$\phi $ is wild if and only if $\phi '$ is wild. 
Consider the polynomial 
\begin{align*}
g:=z_3^2-a_{1,0}\phi (f_2)
&=(x_3+h)^2-a_{1,0}(f_2+hy_1)\\
&=h(2x_3+h-a_{1,0}y_1)-a_{1,0}f_2+x_3^2 \\
&=h(2x_3+a_{0,1}y_2+a_{0,0})-a_{1,0}f_2+x_3^2. 
\end{align*}
Since $\degw h=\degw y_1$ is greater than 
$\degw f_2$ and $\degw x_3^2$, 
we see that 
$g^{\w }=(a_{0,1}hy_2)^{\w }\approx (y_1y_2)^{\w }$ 
if $a_{0,1}\neq 0$, 
and $g^{\w }=(2hx_3)^{\w }\approx (y_1x_3)^{\w }$ otherwise. 
In either case, 
we have $\degw g>\degw y_1$, 
and $g^{\w }$ and $z_3^{\w }\approx y_1^{\w }$ 
are algebraically independent over $k$. 
A direct computation shows that 
\begin{align*}
\phi '(x_1)
&=a_{1,0}^2z_1+
\bigl(z_3^3-2a_{1,0}\bigl(y_2+\phi (f_2)z_3\bigr)\bigr)z_3^2\\
&=a_{1,0}^2z_1+
\bigl(z_3^4-2a_{1,0}\phi (f_2)z_3^2-2a_{1,0}y_2z_3\bigr)z_3\\
&=a_{1,0}^2z_1+\bigl(
g^2-a_{1,0}^2\phi (f_2)^2-2a_{1,0}y_2z_3\bigr)z_3\\
&=(g^2-2a_{1,0}y_2z_3)z_3+
a_{1,0}^2\bigl(z_1-\phi (f_2)^2z_3
\bigr)\\
&=(g^2-2a_{1,0}y_2z_3)z_3+
a_{1,0}^2\bigl(
y_1+(2y_2+\alpha _2^0)(f_2+hy_1)
\bigr), 
\end{align*}
where we use 
(\ref{eq:excep:z_2}) and (\ref{eq:excep:z_1}) 
for the first and last equalities. 
Since $\degw g^2z_3>3\degw y_1>\degw hy_1y_2$, 
this gives that 
$\phi '(x_1)^{\w }=(g^2z_3)^{\w }$. 
By (\ref{eq:excep:z_2}), 
we have 
$$
\phi '(x_2)
=a_{1,0}(y_2+\phi (f_2)z_3)-z_3^3
=a_{1,0}y_2-(z_3^2-a_{1,0}\phi (f_2))z_3
=a_{1,0}y_2-gz_3. 
$$
Hence, 
we get $\phi '(x_2)^{\w }=(-gz_3)^{\w }$. 
Thus, 
$\phi '(x_1)^{\w }$, 
$\phi '(x_2)^{\w }$ and $\phi '(x_3)^{\w }=z_3^{\w }$ 
are algebraically dependent over $k$, 
and are pairwise algebraically independent over $k$, 
since $g^{\w }$ and $z_3^{\w }$ 
are algebraically independent over $k$. 
It is easy to check that 
$(g^2z_3)^{\w }$, $(gz_3)^{\w }$ and $z_3^{\w }$ 
do not belong to 
$k[(gz_3)^{\w },z_3^{\w }]$, 
$k[(g^2z_3)^{\w },z_3^{\w }]$ 
and $k[(g^2z_3)^{\w },(gz_3)^{\w }]$, 
respectively. 
Therefore, 
$\phi '$ satisfies the conditions (1) and (2) 
after Theorem~\ref{thm:SUcriterion}. 
This proves that $\phi '$ is wild, 
thereby proving that $\phi $ is wild.

For the case of $l\geq 2$, 
we need the following lemma.

\begin{lem}\label{lem:f_2wedge f_3}
We have 
$\degw dz_2\wedge dz_3\geq \degw z_2x_3$. 
\end{lem}
\begin{proof}
Since 
$z_2=y_2+\phi (f_2)z_3$, 
we may write 
$dz_2\wedge dz_3=dy_2\wedge dz_3+z_3\eta $, 
where $\eta :=d\phi (f_2)\wedge dz_3$. 
We show that $\degw \eta \geq \degw hy_1x_3$ below. 
Then, 
it follows that 
$$
\degw z_3\eta \geq \degw \degw hy_1x_3z_3
>\degw y_2z_3\geq \degw dy_2\wedge dz_3. 
$$
This implies that 
$\degw dz_2\wedge dz_3=\degw  z_3\eta$, 
and is at least 
$$
\degw hy_1x_3z_3=(2l+1)\degw y_1+\degw x_3=\degw z_2x_3 
$$
by (\ref{eq:z head3}). 
Thus, the lemma is proved.

Now, we show that $\degw \eta \geq \degw hy_1x_3$. 
Since $\phi (f_2)=f_2+hy_1$ and $z_3=x_3+h$, 
we have 
$$
\eta =d(f_2+hy_1)\wedge dz_3
=df_2\wedge dz_3 +d(hy_1)\wedge d(x_3+h)
=\eta '
+hdy_1\wedge dh, 
$$
where $\eta ':=df_2\wedge dz_3 +d(hy_1)\wedge dx_3$. 
Since 
$(hy_1)^{\w }\approx (y_1^{\w })^{l+1}$ and $x_3$ 
are algebraically independent over $k$, 
we have 
$$
\degw d(hy_1)\wedge dx_3
=\degw hy_1x_3>\degw f_2z_3\geq \degw df_2\wedge dz_3. 
$$
Hence, 
we get $\degw \eta '=\degw hy_1x_3$. 
If $h$ belongs to $k[y_1]$, 
then we have $dy_1\wedge dh=0$, 
and so $\degw \eta =\degw \eta '=\degw hy_1x_3$. 
Thus, the assertion is true. 
Assume that $h$ does not belong to $k[y_1]$. 
Since $h$ is an element of $\ker D_3=k[y_1,y_2]$, 
we get 
$$
dy_1\wedge dh=\frac{\partial h}{\partial y_2}dy_1\wedge dy_2
\text{ \ and \ }\frac{\partial h}{\partial y_2}\neq 0. 
$$ 
Since $y_1^{\w }$ and $y_2^{\w }$ 
are algebraically independent over $k$, 
we have $\degw dy_1\wedge dy_2=\degw y_1y_2$. 
Hence, 
it follows that
$$
\degw hdy_1\wedge dh
=\degw h\frac{\partial h}{\partial y_2}dy_1\wedge dy_2
\geq \degw hy_1y_2>\degw hy_1x_3=\degw \eta '. 
$$
Since $\eta =\eta '+hdy_1\wedge dh$, 
this implies that $\degw \eta =\degw hdy_1\wedge dh$, 
and is greater than $\degw hy_1x_3$. 
Therefore, 
the assertion is true. 
\end{proof}

Now, 
assume that $l=2$. 
Then, we have $\degw ^Sh=\degw h=\degw y_1^2$. 
Since $\degw y_2^3<\degw y_1^2<\degw y_2^4$ 
by (\ref{eq:y head}), 
we may write 
$$
h=a_{2,0}y_1^2+a_{0,3}y_2^3+a_{1,1}y_1y_2
+a_{0,2}y_2^2+a_{1,0}y_1+a_{0,1}y_2+a_{0,0}, 
$$
where $a_{i,j}$'s are elements of $k$ 
such that $a_{2,0}\neq 0$. 
Define $\psi \in \Aut (\kx /k)$ by 
$$
\psi (x_1)=f:=a_{2,0}z_1+(a_{1,0}z_2-z_3^3)z_3
$$ 
and $\psi (x_i)=z_i$ for $i=2,3$. 
Then, $\phi ^{-1}\circ \psi $ is elementary. 
Hence, 
$\phi $ is wild if and only if $\psi $ is wild. 
So we prove that $\psi $ is wild.

In the notation above, 
we have the following lemma.

\begin{lem}\label{lem:l=2}
\noindent{\rm (i)} 
$\degw h^3x_3\leq \degw f<\degw h^4$. 

\noindent{\rm (ii)} 
$f^{\w }$ and $y_1^{\w }$ 
are algebraically independent over $k$. 
\end{lem}
\begin{proof}
Set $\bar{h}=h-(a_{2,0}y_1^2+a_{1,0}y_1)$. 
Then, 
by (\ref{eq:excep:z_2}), 
we have 
\begin{align*}
g_1&:=a_{1,0}z_2-h^2(h-a_{2,0}y_1^2)\\
&=a_{1,0}\bigl(y_2+(f_2+hy_1)(x_3+h)\bigr)
-h^2(\bar{h}+a_{1,0}y_1)\\
&=a_{1,0}\bigl(y_2+f_2(x_3+h)+hy_1x_3\bigr)-h^2\bar{h}. 
\end{align*}
If $\bar{h}\neq a_{0,0}$, 
then $\bar{h}^{\w }$ is equal to one of 
$(y_2^3)^{\w }$, $(y_1y_2)^{\w }$, 
$(y_2^2)^{\w }$ and $y_2^{\w }$ 
up to nonzero constant multiples, 
since $y_1^{\w }$ and $y_2^{\w }$ 
are algebraically independent over $k$. 
If this is the case, then 
we have $\degw \bar{h}>\degw x_3$. 
Hence, we get 
$\degw h^2\bar{h}>\degw hy_1x_3$. 
Thus, 
we know that 
$g_1^{\w }\approx (h^2\bar{h})^{\w }$ 
and $\degw g_1>\degw h^2x_3$. 
If $\bar{h}=a_{0,0}$, 
then we have $\degw g_1\leq \degw h^2$. 
Since $z_3=x_3+h$, we have 
\begin{equation}\label{eq:excep:g_2}
\begin{aligned}
g_2&:=a_{2,0}(f_2+hy_1)^2+a_{1,0}z_2-z_3^3\\
&=a_{2,0}(f_2^2+2f_2hy_1)-(x_3^3+3hx_3^2+3h^2x_3)+g_1. 
\end{aligned}
\end{equation}
From this, 
we see that $(g_2-g_1)^{\w }\approx (h^2x_3)^{\w }$. 
Therefore, we get 
$$
g_2^{\w }\approx g_1^{\w }
\approx (h^2\bar{h})^{\w } \ 
\text{ if } \ \bar{h}\neq a_{0,0},\ 
\text{ and } \ 
g_2^{\w }=(g_2-g_1)^{\w }\approx (h^2x_3)^{\w }\ \text{ otherwise}.
$$
In either case, we have 
$\degw h^2x_3\leq \degw g_2<\degw h^3$. 
If $\bar{h}\neq a_{0,0}$, 
then $g_2^{\w }\approx (h^2\bar{h})^{\w }$ 
is equal to one of $(y_1^4y_2^3)^{\w }$, $(y_1^5y_2)^{\w }$, 
$(y_1^4y_2^2)^{\w }$ and $(y_1^4y_2)^{\w }$ 
up to nonzero constant multiples. 
Since $y_1^{\w }$ and $y_2^{\w }$ are 
algebraically independent over $k$, 
we know that $g_2^{\w }$ and $y_1^{\w }$ 
algebraically independent over $k$. 
The same holds when $\bar{h}=a_{0,0}$ 
since $y_1^{\w }$ and $x_3$ are 
algebraically independent over $k$. 
Thus, 
it suffices to show that 
$f^{\w }=(g_2h)^{\w }$. 
By (\ref{eq:excep:z_1}), 
we have 
$$
f=a_{2,0}z_1+g_2z_3-a_{2,0}(f_2+hy_1)^2z_3
=a_{2,0}\bigl( y_1+(2y_2+\alpha _2^0)(f_2+hy_1)
\bigr)+g_2z_3. 
$$
Since $\degw h^2x_3\leq \degw g_2$, 
it follows that 
$$
\degw (f-g_2z_3)=\degw hy_1y_2 <\degw h^3x_3
\leq \degw g_2h=\degw g_2z_3. 
$$
Therefore, 
we get $f^{\w }=(g_2z_3)^{\w }=(g_2h)^{\w }$. 
\end{proof}

We prove that 
$\psi $ admits no elementary reduction 
and no Shestakov-Umirbaev reduction 
for the weight $\w $. 
Then, 
it follows that $\psi $ 
is wild due to Theorem~\ref{thm:SUcriterion}. 
By (\ref{eq:z head3}) with $l=2$, 
we have $z_2^{\w }\approx (y_1^{\w })^5$ 
and $z_3^{\w }\approx (y_1^{\w })^2$. 
Hence, 
$f^{\w }$ and $z_i^{\w }$ 
are algebraically independent over $k$ 
for $i=2,3$ in view of Lemma~\ref{lem:l=2} (ii). 
Thus, 
we know that  
$k[f,z_i]^{\w }=k[f^{\w },z_i^{\w }]$ for $i=2,3$. 
Since $z_i^{\w }$ does not belong to $k[z_j^{\w }]$ 
for $(i,j)=(2,3),(3,2)$, 
we see that 
$z_i^{\w }$ does not belong to $k[f^{\w },z_j^{\w }]$ 
for $(i,j)=(2,3),(3,2)$. 
Therefore, 
$\psi (x_i)^{\w }$ does not belong to 
$k[\psi (x_1),\psi (x_j)]^{\w }$ for $(i,j)=(2,3),(3,2)$. 
To conclude that $\psi $ admits no 
elementary reduction for the weight $\w $, 
it suffices to check that 
$f^{\w }$ does not belong to $k[z_2,z_3]^{\w }$.

Suppose to the contrary that 
$f^{\w }$ belongs to $k[z_2,z_3]^{\w }$. 
Then, 
there exists $g\in k[z_2,z_3]$ such that 
$g^{\w }=f^{\w }$. 
By Lemma~\ref{lem:l=2} (ii), 
$f^{\w }$ cannot be an element of $k[y_1^{\w }]$. 
Hence, 
$g^{\w }$ does not belong to $k[z_2^{\w },z_3^{\w }]$. 
This implies that $\degw g<\degw ^Tg$, 
where $T:=\{ z_2,z_3\} $. 
By (i) and (ii) of Lemma~\ref{lem:SUineq1}, 
there exist $p,q\in \N $ with $\gcd (p,q)=1$ 
such that $(z_2^{\w })^p\approx (z_3^{\w })^q$ 
and 
$$
\degw f=\degw g\geq 
q\degw z_3+\degw dz_2\wedge dz_3-\degw z_2-\degw z_3. 
$$
Then, we have $(p,q)=(2,5)$. 
Since $\degw dz_2\wedge dz_3\geq \degw z_2x_3$ 
by Lemma~\ref{lem:f_2wedge f_3}, 
the right-hand side of the preceding inequality 
is at least 
\begin{align*}
5\degw z_3+\degw z_2x_3-\degw z_2-\degw z_3
>4\degw z_3=\degw h^4. 
\end{align*}
Hence, we get $\degw f>\degw h^4$. 
This contradicts Lemma~\ref{lem:l=2} (i). 
Thus, 
$f^{\w }$ does not belong to $k[z_2,z_3]^{\w }$. 
Therefore, 
$\psi $ admits no elementary reduction 
for the weight $\w $.

Next, 
we check that $\psi $ admits no Shestakov-Umirbaev 
reduction for the weight $\w $ 
by means of Lemma~\ref{lem:sulem}. 
By Lemma~\ref{lem:l=2} (i), 
we have $\degw f\geq \degw h^3x_3>6\degw y_1$. 
Since $\degw z_2=5\degw y_1$ 
and $\degw z_3=2\degw y_1$, 
it follows that 
$\degw \psi (x_1)>\degw \psi (x_2)>\degw \psi (x_3)$. 
Clearly, 
$3\degw \psi (x_2)=15\degw y_1$
is not equal to $4\degw \psi (x_3)=8\degw y_1$. 
Since $\rank \w =3$, 
we know by Lemma~\ref{lem:l=2} (ii) that 
$\degw f$ and $\degw y_1$ are linearly independent. 
Hence, 
$2\degw \psi (x_1)$ is not equal to 
$m\degw \psi (x_i)$ for any $m\in \N $ and $i=2,3$. 
Thus, 
we conclude from Lemma~\ref{lem:sulem} that 
$\psi $ admits no Shestakov-Umirbaev reduction 
for the weight $\w $. 
This proves that $\psi $ is wild, 
and thereby proving that $\phi $ is wild.

Finally, 
we prove that $\phi $ is wild when $l\geq 3$ 
using Lemma~\ref{lem:awc}. 
Define $\iota \in \Aut (\kx /k[x_2,x_3])$ 
by $\iota (x_1)=-x_1$. 
First, 
we show that 
$D':=
\iota ^{-1}\circ (f_{\theta }D_{\theta })\circ \iota $ 
is equal to $-f_2\tilde{D}$. 
Since 
$\iota ^{-1}(f_{\theta })=-x_1x_3+\theta (x_2)=-f_2$ 
and $D_{\theta }(-x_1)=\theta '(x_2)=\tilde{D}(x_1)$, 
we have 
$$
D'(x_1)
=\iota ^{-1}(f_{\theta })
\iota ^{-1}\bigl(D_{\theta }(-x_1)\bigr)
=-f_2\theta '(x_2)
=-f_2\tilde{D}(x_1). 
$$
Since $D_{\theta }(x_2)=\tilde{D}(x_2)=x_3$, 
$D_{\theta }(x_3)=\tilde{D}(x_3)=0$ 
and $\iota (x_i)=x_i$ for $i=2,3$, 
we have $D'(x_i)
=\iota ^{-1}(f_{\theta })
\iota ^{-1}\bigl(D_{\theta }(x_i)\bigr)
=-f_2\tilde{D}(x_i)$ 
for $i=2,3$. 
Hence, 
we get $D'=-f_2\tilde{D}$. 
Thus, 
it follows that 
$$
\iota ^{-1}\circ \sigma _{\theta }\circ \iota 
=\iota ^{-1}\circ (\exp f_{\theta }D_{\theta })\circ \iota 
=\exp D'
=\exp (-f_2\tilde{D})
=\tilde{\sigma }. 
$$
Since $\phi =\exp hD_3$ fixes $f_3=\tilde{\sigma }(x_1)$, 
we know that 
$\phi ':=\iota \circ \phi \circ \iota ^{-1}$ 
fixes 
$\iota (\tilde{\sigma }(x_1))
=\sigma _{\theta }(\iota (x_1))=-\sigma _{\theta }(x_1)$. 
Therefore, 
$\phi '$ belongs to $\Aut (\kx /k[\sigma _{\theta }(x_1)])$. 
So we prove that $\phi '$ is wild 
by means of Lemma~\ref{lem:awc}. 
Then, 
it follows that $\phi $ is wild.

Define $\gamma _0^{\w },\ldots ,\gamma _3^{\w }$ 
as in Section~\ref{sect:proof:tacoord} 
for $\phi '$. 
Thanks to Lemma~\ref{lem:awc}, 
it suffices to verify that 
\begin{equation}\label{eq:excep:awc}
\gamma ^{\w }_0<\gamma ^{\w }_1,\quad 
\gamma ^{\w }_2>\gamma ^{\w }_3,\quad 
2\gamma ^{\w }_1>3\gamma ^{\w }_2+2\gamma ^{\w }_3\quad \text{and}\quad 
\gamma ^{\w }_1\geq \gamma ^{\w }_2+2\gamma ^{\w }_3. 
\end{equation}
For $i=2,3$, 
we have $\phi '(x_i)=\iota (\phi (x_i))=\iota (z_i)$. 
Since $\iota $ preserves the $\w $-degree of 
each element of $\kx $, 
it follows that $\degw \phi '(x_i)=\degw z_i$ 
for $i=2,3$. 
Since 
$\phi '(f_{\theta })=\iota (\phi (-f_2))=-\iota (\phi (f_2))$, 
we have $\degw \phi '(f_{\theta })=\degw \phi (f_2)$ 
similarly. 
Put $\gamma =\degw y_1$. 
Then, 
it follows from (\ref{eq:z head3}) that 
\begin{equation}\label{eq:excep:ta:z'}
\degw \phi '(x_2)=(2l+1)\gamma ,\ \ 
\degw \phi '(x_3)=l\gamma ,\ \ 
\degw \phi '(f_{\theta })=(l+1)\gamma . 
\end{equation}
Since $l\geq 3$ by assumption, 
$2l+1$ is not a multiple of $l$. 
Hence, we know that 
$\phi '(x_2)^{\w }$ does not belong to $k[\phi '(x_3)]^{\w }$. 
By the definition of $\gamma _2^{\w }$, 
this implies that 
$\gamma _2^{\w }=\degw \phi '(x_2)$. 
Hence, 
$\gamma _2^{\w }$ 
is greater than $\gamma _3^{\w }=\degw \phi '(x_3)$ 
by (\ref{eq:excep:ta:z'}). 
This proves the second part of (\ref{eq:excep:awc}). 
Since $l+1$ is not a multiple of $l$, 
we know that 
$\phi '(f_{\theta })^{\w }$ 
does not belong to $k[\phi '(x_3)]^{\w }$. 
By the definition of $\gamma _0^{\w }$, 
this implies that 
$\gamma _0^{\w }=\degw \phi '(f_{\theta })$. 
Hence, $\gamma _0^{\w }$ 
is less than $\gamma _2^{\w }=\degw \phi '(x_2)$ 
by (\ref{eq:excep:ta:z'}). 
By the definition of $\gamma _1^{\w }$, 
we may write 
$\gamma _1^{\w }=\eta (\theta ;\phi '(x_3),\phi '(x_2))$ 
in the notation of Lemma~\ref{lem:91113}. 
Note that 
$$
\theta (\phi '(x_2))+\rho \bigl(\phi '(x_2),\phi '(x_3)\bigr)\phi '(x_3)
=\iota 
\bigl( 
\theta (z_2)+\rho (z_2,z_3)z_3
\bigr)
$$
for any $\rho (y,z)\in k[y,z]$. 
Since $\iota $ preserves the $\w $-degree, 
it follows that 
$$
\gamma _1^{\w }=
\eta (\theta ;\phi '(x_3),\phi '(x_2))=
\eta (\theta ;z_3,z_2). 
$$
By (\ref{eq:z head3}), 
we see that (1) and (2) 
before Lemma~\ref{lem:91113} are fulfilled 
for $f=z_3$ and $g=z_2$. 
Clearly, 
(3) is satisfied. 
Moreover, 
the condition (i) of Lemma~\ref{lem:91113} is satisfied, 
because $\deg _z\theta (z)=2$, 
$\degw dz_3\wedge dz_2>\degw z_2$ 
by Lemma~\ref{lem:f_2wedge f_3}, 
$(2l+1)\degw z_3=l\degw z_2$ 
by (\ref{eq:z head3}), 
and $l\geq 3$ by assumption. 
Thus, we conclude 
that $\gamma _1^{\w }=\eta (\theta ;z_3,z_2)$ 
is greater than 
$$
\degw z_3+\frac{3}{2}\degw z_2\quad\text{and}\quad 
2\degw z_3+\degw z_2. 
$$
Since $\degw z_i=\degw \phi '(x_i)=\gamma _i^{\w }$ 
for $i=2,3$, 
this implies the last two parts of (\ref{eq:excep:awc}). 
Since $\gamma _0^{\w }<\gamma _2^{\w }$ as mentioned, 
it follows that $\gamma _0^{\w }<\gamma _1^{\w }$. 
Therefore, 
we get the four inequalities of (\ref{eq:excep:awc}). 
This proves that $\phi '$ is wild, 
and thereby proving that $\phi $ is wild.

This completes the proof of Theorem~\ref{thm:lsc1} (iii) 
in the case of (w2), 
and thus completing the proof of Theorem~\ref{thm:lsc1}.

\section{Plinth ideals (I)} 
\setcounter{equation}{0}
\label{sect:plinth1}

The goal of this section 
is to prove Theorem~\ref{prop:rank}. 
First, we remark that 
$$
\pl \sigma \circ D\circ \sigma ^{-1}=\sigma (\pl D) 
$$
for each $\sigma \in \Aut (\kx /k)$ 
and $D\in \Der _k\kx $, 
since 
$$
\ker \sigma \circ D\circ \sigma ^{-1}=\sigma (\ker D)
\quad \text{ and }\quad 
(\sigma \circ D\circ \sigma ^{-1})(\kx )=\sigma (D(\kx )). 
$$
By definition, we have 
$$
\rank \sigma \circ D\circ \sigma ^{-1}=\rank D. 
$$

\begin{lem}\label{lem:slice irred}
For $D\in \lnd _k\kx $, 
the following conditions are equivalent$:$ 

\noindent{\rm (1)} $\pl D=\ker D$. 

\noindent{\rm (2)} $D(s)=1$ for some $s\in \kx $. 

\noindent{\rm (3)} 
$D$ is irreducible and $\rank D=1$. 
\end{lem}
\begin{proof}
By the definition of $\pl D$, 
we see that (1) and (2) are equivalent. 
Since $D(s)=1$ belongs to no proper ideal of $\kx $, 
(2) implies that $D$ is irreducible. 
Recall that, 
if $D(s)=1$ for $D\in \lnd _k\kx $ and $s\in \kx $, 
then we have 
$\kx =(\ker D)[s]$ 
(cf.~\cite[Proposition 1.3.21]{Essen}). 
By Miyanishi~\cite{MiyanishiD}, 
we have $\ker D=k[f,g]$ for some $f,g\in \kx $ 
for any $D\in \lnd _k\kx \sm \zs $. 
Hence, 
(2) implies that $k[f,g,s]=\kx $ for some $f,g\in \kx $, 
and so implies that $\rank D=1$. 
Thus, 
(2) implies (3). 
If $\rank D=1$ for $D\in \lnd _k\kx $, 
then we have 
$\sigma \circ D\circ \sigma ^{-1}=f\partial /\partial x_1$ 
for some $\sigma \in \Aut (\kx /k)$ and $f\in k[x_2,x_3]\sm \zs $. 
If furthermore $D$ is irreducible, 
then $f$ belongs to $k^{\times }$. 
When this is the case, $s:=f^{-1}\sigma (x_1)$ belongs to $\kx $, 
and satisfies $D(s)=1$. 
Thus, 
(3) implies (2). 
Therefore, 
(1), (2) and (3) are equivalent. 
\end{proof}

We begin with the proof of (i) and (ii) of Theorem~\ref{prop:rank}. 
If $t_0=1$, 
then we have $I=\{ 1\} $. 
From (\ref{eq:D_1}), 
we see that 
$\rank D_1=1$ and 
$$
\pl D_1=x_1\kx \cap k[x_1,x_3]=
x_1k[x_1,x_3]=x_1\ker D_1, 
$$ 
proving the first part of (i). 
Hence, we get 
$\rank D_1'=1$ and $\pl D_1'=x_1\ker D_1'$ if $t_1=1$. 
If furthermore $t_0=2$, 
then we have 
$D_2=\tau _2\circ D_1'\circ \tau _2^{-1}$ 
by Theorem~\ref{thm:lsc1} (ii). 
Thus, 
it follows that $\rank D_2=1$ and 
$$
\pl D_2=\tau _2(\pl D_1')=
\tau _2(x_1\ker D_1')=f_2\ker D_2, 
$$ 
proving the second part of (i). 
If $(t_0,i)=(2,1)$ or $i=0$, 
then we have $D_i(x_3)=1$. 
Hence, we get $\pl D_i=\ker D_i$ 
and $\rank D_i=1$ by Lemma~\ref{lem:slice irred}. 
By the last part of Theorem~\ref{thm:lsc1} (ii), 
the same holds for every $i\geq 1$ if $t_0=t_1=2$, 
proving (ii).

To analyze the detailed structure of plinth ideals, 
the following lemma is useful. 
For an irreducible element $p$ of $\kx $, 
we consider the $p$-{\it adic valuation} 
$v_p$ of $\kxr $, 
i.e., 
the map $v_p:\kxr \to \Z \cup \{ \infty \} $ 
defined by $v_p(f)=m$ for each $f\in \kxr \sm \zs $ 
and $v_p(0)=\infty $, 
where $m\in \Z $ is such that 
$f=p^mg/h$ for some $g,h\in \kx \sm p\kx $. 
Now, 
take $D\in \lnd _k\kx $ and $s\in \kx $ such that 
$D(s)\neq 0$ and $D^2(s)=0$. 
Then, 
$D(s)$ belongs to $\pl D$. 
Since $\pl D$ is a nonzero principal ideal of $\ker D$, 
and $D(s)\neq 0$ by assumption, 
we may find a factor $q$ of $D(s)$ 
(may be an element of $k^{\times }$) 
such that $\pl D=q\ker D$. 
Let $p\in \kx $ be any factor of $D(s)$ 
which is an irreducible element of $\kx $. 
Then, 
$p$ belongs to $\ker D$ 
by the factorially closedness of $\ker D$ in $\kx $. 
Note that $i:=v_p(q)$ is the maximal number 
such that $p^{-i}q$ belongs to $\kx $. 
Since $p^{-i}q$ belongs to $\ker D$, 
we see that $i$ is equal to 
the maximal number such that 
$q$ belongs to $p^i\ker D$. 
Thus, 
we know that 
$$
i=v_p(q)=\max \{ i\geq 0\mid \pl D\text{ is contained in }p^{i}\ker D\} . 
$$
We define 
$$
j=\max \{ j\geq 0\mid (s+\ker D)\cap p^{j}\kx \neq \emptyset \} . 
$$

In the notation above, 
we have the following lemma.

\begin{lem}\label{lem:rank criterion2}
It holds that $i+j=v_p(D(s))$. 
Hence, 
$\pl D$ is contained in $p\ker D$ 
if $(s+\ker D)\cap p\kx =\emptyset $. 
\end{lem}
\begin{proof}
Put $l=v_p(D(s))-v_p(q)$. 
We show that 
$$
S_1:=(s+\ker D)\cap p^{l}\kx \neq \emptyset
\quad\text{and}\quad  
S_2:=(s+\ker D)\cap p^{l+1}\kx =\emptyset . 
$$
Then, 
it follows that $l=j$. 
Hence, we get $i+j=v_p(q)+l=v_p(D(s))$. 
Since $q\ker D=\pl D$, 
we have $q=D(t)$ for some $t\in \kx $. 
Since $D(s)$ belongs to $\pl D$, 
there exists $h_1\in \ker D$ such that 
$D(s)=qh_1=D(t)h_1$. 
Then, it follows that $D(s-th_1)=0$. 
Hence, 
$h_2:=s-th_1$ belongs to $\ker D$. 
Thus, 
$th_1=s-h_2$ belongs to $s+\ker D$. 
Since $h_1=D(s)q^{-1}$, 
we have 
$$
v_p(th_1)=v_p\bigl( tD(s)q^{-1}\bigr) 
\geq v_p(D(s))-v_p(q)=l. 
$$
Hence, $th_1$ belongs to $p^{l}\kx $. 
Thus, 
$th_1$ belongs to $S_1$. 
Therefore, $S_1$ is not empty. 
Next, 
suppose to the contrary that 
$S_2$ is not empty. 
Then, 
we have $s+h=p^{l+1}u$ 
for some $h\in \ker D$ and $u\in \kx $. 
Since $D(h)=D(p)=0$, 
it follows that $D(u)=D(s)p^{-l-1}$. 
Hence, we get 
$$
v_p(D(u))=v_p(D(s)p^{-l-1})=v_p(D(s))-l-1=v_p(q)-1. 
$$
On the other hand, 
$D(u)$ belongs to $\pl D$, 
since $D^2(u)=D^2(s)p^{-l-1}=0$. 
Hence, 
we have $D(u)=qh'$ for some $h'\in \ker D$. 
Thus, 
we get $v_p(D(u))\geq v_p(q)$, 
a contradiction. 
Therefore, $S_2$ is empty.

To prove the last part, 
assume that 
$(s+\ker D)\cap p\kx =\emptyset $. 
Then, we have $j=0$. 
Since $v_p(D(s))=i+j$, 
it follows that $i=v_p(D(s))$. 
Because $p$ is a factor of $D(s)$, 
we have $v_p(D(s))\geq 1$. 
Thus, we get $i\geq 1$. 
Therefore, 
$\pl D$ is contained in $p\ker D$. 
\end{proof}

We remark that the former case of 
Theorem~\ref{prop:rank} (iii) is reduced to the latter case 
thanks to 
Theorem~\ref{thm:lsc1} (ii). 
In fact, 
since $t_0=2$, 
we have 
$D_2=\tau _2\circ D_1'\circ \tau _2^{-1}$ 
by Theorem~\ref{thm:lsc1} (ii). 
Since $t_1\geq 3$, 
it follows that $\rank D_2=\rank D_1'=2$ 
by the latter case of Theorem~\ref{prop:rank} (iii). 
Since $\tau _2(f_1')=f_2$ by Lemma~\ref{lem:t_0=2} (i), 
the latter case of Theorem~\ref{prop:rank} (iii) 
implies that 
$$
\pl D_2
=\tau _2(\pl D_1')
=\tau _2(f_1'\ker D_1')
=\tau _2(f_1')\tau _2(\ker D_1')
=f_2\ker D_{2}. 
$$
Similarly, 
the former case of 
Theorem~\ref{prop:rank} (iv) 
is reduced to the latter case.

Now, 
we prove Theorem~\ref{prop:rank} (iii). 
By the remark, 
we may assume that $t_0\geq 3$ and $i=1$. 
Note that $D_1(x_2)=x_1$ and $D_1^2(x_2)=0$. 
Hence, we have $\pl D_1=(q)$ for some factor $q$ of $x_1$ 
by the discussion before Lemma~\ref{lem:rank criterion2}. 
We prove that $q\approx x_1$. 
Then, 
it follows that $\pl D_1=(x_1)=(f_1)$. 
Since $D_1$ is irreducible when $t_0\geq 3$ 
by (b) of Theorem~\ref{thm:lsc1} (i), 
this implies that $\rank D_1\geq 2$ 
because of Lemma~\ref{lem:slice irred}. 
Since $D_1(x_1)=0$, 
we have $\rank D_1\leq 2$. 
Hence, we get $\rank D_1=2$. 
Thus, 
the proof is completed. 
Now, 
suppose to the contrary that $q\not\approx x_1$. 
Then, 
$q$ must be an element of 
$k^{\times }$. 
Hence, 
we have $\pl D_1=\ker D_1$, 
and is not contained in $x_1\ker D_1$. 
Since $x_1$ is an irreducible factor of $D_1(x_2)$, 
it follows that $(x_2+\ker D_1)\cap x_1\kx \neq \emptyset $ 
by the last part of Lemma~\ref{lem:rank criterion2}. 
Take $g_1\in \ker D_1$ and $g_2\in \kx $ 
such that $x_2+g_1=x_1g_2$. 
Since $\ker D_1=k[x_1,f_2]$ by Theorem~\ref{thm:lsc1} (i), 
we may write $g_1=\nu (x_1,f_2)$, 
where $\nu (y,z)\in k[y,z]$. 
Then, 
we have $x_2=-\nu (x_1,f_2)+x_1g_2$. 
Substituting zero for $x_1$, 
we get 
$x_2=-\nu \bigl(0,-\theta (x_2)\bigr) $. 
Hence, 
$\nu (0,z)$ does not belong to $k$. 
Thus, 
we know that 
$$
1=\deg _{x_2}x_2=
(\deg _z\nu (0,z))\deg _{x_2}\theta (x_2)
\geq \deg _{x_2}\theta (x_2)=t_0-1\geq 2, 
$$ 
a contradiction. 
This proves that $q\approx x_1$, 
and thereby completing the proof of (iii).

To prove (iv), 
we need the following result.

\begin{lem}\label{lem:rank criterion}
Assume that 
$D\in \lnd _k\kx \sm \zs $ is irreducible. 
If $D(p)=0$ for 
a coordinate $p$ 
of $\kx $ over $k$, 
then we have $D(\kx )\cap k[p]\neq \zs $. 
Consequently, 
we have $\pl D=(q)$ for some $q\in k[p]\sm \zs $. 
\end{lem}
\begin{proof}
Without loss of generality, 
we may assume that $p=x_1$. 
Then, 
$D$ extends to a locally nilpotent derivation 
$\tilde{D}$ of $R:=k(x_1)[x_2,x_3]$ over $k(x_1)$. 
We show that $\tilde{D}$ is irreducible by contradiction. 
Suppose to the contrary that 
$\tilde{D}$ is not irreducible. 
Then, 
we may find $q\in R
\sm k(x_1)$ 
such that $\tilde{D}(R)$ 
is contained in $qR$. 
Multiplying by an element of $k(x_1)^{\times }$, 
we may assume that $q$ belongs to $\kx \sm k[x_1]$, 
and is not divisible by any element of $k[x_1]\sm k$. 
Then, 
$qh$ does not belong to $\kx $ for any $h\in R\sm \kx $. 
Hence, 
$qR\cap \kx $ is contained in $q\kx $. 
Since $D(\kx )=\tilde{D}(\kx )$ 
is contained in $qR\cap \kx $, 
it follows that $D(\kx )$ is contained in $q\kx $. 
This contradicts that $D$ is irreducible. 
Therefore, 
$\tilde{D}$ is irreducible.

Now, 
by Theorem~\ref{thm:Rentschler original}, 
there exist $\tau \in \Aut (R/k(x_1))$ 
and $f\in k(x_1)[x_2]\sm \zs $ such that 
$D':=\tau ^{-1}\circ \tilde{D}\circ \tau 
=f\partial /\partial x_3$. 
Since $\tilde{D}$ is irreducible, 
so is $D'$. 
Hence, 
$f$ belongs to $k(x_1)^{\times }$. 
Take $g\in k(x_1)^{\times }$ 
such that $fg$ and $s:=\tau (x_3)g$ 
belong to $\kx $. 
Then, 
$fg$ belongs to $k[x_1]\sm \zs $. 
Since $\tau ^{-1}(s)=x_3g$ and $D'(x_3g)=fg$, 
it follows that 
$$
D(s)=\tau \bigl(D'(\tau ^{-1}(s))\bigr)
=\tau \bigl(D'(x_3g)\bigr)=\tau (fg)=fg. 
$$ 
Thus, 
$D(s)$ belongs to $k[x_1]\sm \zs $. 
Therefore, 
we get $D(\kx )\cap k[x_1]\neq \zs $.

Since $D\neq 0$ by assumption, 
we have $\pl D=(q)$ for some $q\in \kx \sm \zs $. 
Take any 
$h\in D(\kx )\cap k[x_1]\sm \zs $. 
Then, 
$h$ belongs to $\pl D$. 
Hence, $q$ is a factor of $h$. 
Since $k[x_1]$ is factorially closed in $\kx $, 
it follows that $q$ belongs to $k[x_1]\sm \zs $. 
\end{proof}

Using Lemma~\ref{lem:rank criterion}, 
we get the following proposition.

\begin{prop}\label{prop:rank3criterion}
Let $D\in \lnd _k\kx \sm \zs $ and $p_1,p_2,q\in \kx $ 
be such that 
$\pl D=(q)$, and $p_1$ and $p_2$ are factors of $q$. 
If $D$ is irreducible, 
and 
$p_1$ and $p_2$ are algebraically independent over $k$, 
then we have $\rank D=3$. 
\end{prop}
\begin{proof}
Suppose to the contrary that $\rank D\leq 2$. 
Then, 
we have $D(p)=0$ for some coordinate $p$ of $\kx $ over $k$. 
Since $D$ is irreducible by assumption, 
it follows that $\pl D=(q')$ for some $q'\in k[p]\sm \zs $ 
by Lemma~\ref{lem:rank criterion}. 
Then, we have $q\approx q'$. 
Hence, 
$q$ belongs to $k[p]\sm \zs $. 
Since $k[p]$ is factorially closed in $\kx $, 
we know that $p_1$ and $p_2$ belong to $k[p]$. 
Thus, 
$p_1$ and $p_2$ are algebraically dependent over $k$, 
a contradiction. 
Therefore, 
we conclude that $\rank D=3$. 
\end{proof}

Let $D\in \lnd _k\kx \sm \zs $ and $q\in \kx $ 
be such that $D$ is irreducible and $\pl D=(q)$, 
and let $K$ be the subfield of $\kxr $ 
generated over $k$ by all the factors of $q$. 
Then, 
we have $\trd _kK\leq 2$, 
since every factor of $q$ belongs to $\ker D$, 
and $\trd _k\ker D=2$. 
By Lemma~\ref{lem:slice irred}, 
we have $\trd _kK=0$ if and only if $\rank D=1$. 
Hence, 
we know by Proposition~\ref{prop:rank3criterion} 
that $\trd _kK=1$ if $\rank D=2$. 
We do not know whether $\rank D=2$ whenever $\trd _kK=1$.

Now, 
let us prove Theorem~\ref{prop:rank} (iv). 
Without loss of generality, 
we may assume that 
$t_0\geq 3$, $(t_0,t_1)\neq (3,1)$ and $i\geq 2$ 
as remarked. 
Since $I=\N $, 
we see that $i$ belongs to $I$, 
and is not the maximum of $I$. 
Hence, 
$D_i$ is irreducible, $\ker D_i=k[f_i,f_{i+1}]$ 
and $D_i(r)=f_if_{i+1}$ by Theorem~\ref{thm:lsc1} (i). 
Since $D_i^2(r)=0$, 
we have $\pl D_i=(q)$ for some factor 
$q$ of $D_i(r)=f_if_{i+1}$ 
by the discussion before Lemma~\ref{lem:rank criterion2}. 
We show that $q\approx f_if_{i+1}$. 
Since 
$f_i$ and $f_{i+1}$ are irreducible elements of $\kx $ 
with $f_i\not\approx f_{i+1}$ by Lemma~\ref{lem:facts}, 
it suffices to verify that $v_{f_l}(q)=1$ for $l=i,i+1$. 
Suppose to the contrary that 
$v_{f_l}(q)=0$ for some $l\in \{ i,i+1\} $. 
Then, 
$\pl D_i=(q)$ is not contained in $f_l\ker D_i$. 
Hence, 
$(r+\ker D_i)\cap f_l\kx $ is not empty 
due to the last part of Lemma~\ref{lem:rank criterion2}. 
Since $\ker D_i=k[f_i,f_{i+1}]$, 
it follows that 
$(r+k[f_j])\cap f_l\kx \neq \emptyset $, 
where $j\in \{ i,i+1\} $ with $j\neq l$. 
By the last part of Lemma~\ref{lem:local slice}, 
this implies that $a_j<2$. 
However, 
since $t_0\geq 3$, $(t_0,t_1)\neq (3,1)$ and 
$j\geq i\geq 2$, 
we have $a_j\geq 2$ by Lemma~\ref{lem:a_i} (ii), 
a contradiction. 
This proves that $q\approx f_if_{i+1}$. 
Therefore, we get $\pl D_i=(f_if_{i+1})$. 
Since $f_i$ and $f_{i+1}$ 
are algebraically independent over $k$, 
it follows that $\rank D_i=3$ 
by Proposition~\ref{prop:rank3criterion}. 
This completes the proof of (iv).

Finally, we prove Theorem~\ref{prop:rank} (v). 
Assume that $(t_0,t_1)=(3,1)$. 
Then, 
we can define 
$\sigma _3\in \Aut (\kx /k)$ as in (\ref{eq:sigma _3}). 
Since 
$f_2=\sigma _3(f_2)$, 
$f_3=\sigma _3(x_1)=\sigma _3(f_1)$ 
and $f_4=\sigma _3(x_2)=\sigma _3(f_0)$, 
we have 
\begin{gather*}
D_2=-\Delta _{(f_2,f_3)}
=-\Delta _{(\sigma _3(f_2),\sigma _3(f_1))}
=-(\det J\sigma _3)
(\sigma _3\circ D_1\circ \sigma _3^{-1})\\
D_3=-\Delta _{(f_3,f_4)}
=-\Delta _{(\sigma _3(f_1),\sigma _3(f_0))} 
=-(\det J\sigma _3)
(\sigma _3\circ D_0\circ \sigma _3^{-1})
\end{gather*}
by the formula (\ref{eq:Jacobian}). 
Since $t_0=3$, 
we have 
$\pl D_1=f_1\ker D_1$ and $\rank D_1=2$ 
by Theorem~\ref{prop:rank} (iii). 
Hence, 
it follows that 
$$
\pl D_2
=\sigma _3(f_1\ker D_1)
=\sigma _3(f_1)\sigma _3(\ker D_1)
=f_3\ker D_2
$$ 
and $\rank D_2=2$. 
Since $\pl D_0=\ker D_0$ and $\rank D_0=1$, 
we get $\pl D_3=\ker D_3$ and $\rank D_3=1$ 
similarly. 
Recall that 
$D_4=f_4\Delta _{(x_3,f_4)}$ by (\ref{eq:D_4}). 
Since $x_3=\sigma _3(x_3)$ 
and $f_4=\sigma _3(x_2)$, 
we have 
$$
\Delta :=\Delta _{(x_3,f_4)}
=\Delta _{(\sigma _3(x_3),\sigma _3(x_2))}
=(\det J\sigma _3)\bigl( \sigma _3\circ 
\Delta _{(x_3,x_2)}\circ \sigma _3^{-1}\bigr). 
$$
Since 
$\Delta _{(x_3,x_2)}=-\partial /\partial x_1$, 
it follows that 
$\pl \Delta =\ker \Delta $ and $\rank \Delta =1$. 
Therefore, 
we conclude that 
$\pl D_4=f_4\ker D_4$ and $\rank D_4=1$. 
This completes the proof of (v), 
and thus completing the proof of 
Theorem~\ref{prop:rank}.

\section{Plinth ideals (II)}
\label{sect:plinth2}
\setcounter{equation}{0}

In this section, 
we prove Theorem~\ref{prop:rank2}. 
Assume that $t_0\geq 3$ and $i=2$, 
or $t_0\geq 3$, $(t_0,t_1)\neq (3,1)$ 
and $i\geq 3$. 
By Theorem~\ref{thm:lsc2} (i), 
$\tilde{D}_i$ is irreducible 
and locally nilpotent, 
and satisfies $\ker \tilde{D}_i=k[f_i,\tilde{f}_{i+1}]$. 
Hence, 
$f_i$ and $\tilde{f}_{i+1}$ 
are irreducible elements of $\kx $ 
by Lemma~\ref{lem:facts}. 
From (\ref{eq:tilde{D}_i(r_i)}), 
we see that $\tilde{f}_{i+1}$ 
is a factor of $\tilde{D}_i(r_i)$. 
If $i\geq 3$, 
then $f_i$ is also a factor of 
$\tilde{D}_i(r_i)$. 
Since $\tilde{D}_i^2(r_i)=0$, 
we have $\pl \tilde{D}_i=(\tilde{g}_i)$ 
for some factor $\tilde{g}_i$ of $\tilde{D}_i(r_i)$ 
by the discussion before Lemma~\ref{lem:rank criterion2}.

We note that $a_i\geq 2$ under the assumption above. 
Actually, 
we have $a_2=t_0-1\geq 2$ if $i=2$, 
and $a_i\geq 2$ if $i\geq 3$ 
by Lemma~\ref{lem:a_i} (ii).

\begin{lem}\label{lem:rank 1}
$\tilde{f}_{i+1}$ is a factor of $\tilde{g}_i$. 
Hence, 
we have $\rank \tilde{D}_i\geq 2$. 
If $\tilde{g}_i$ and 
$\tilde{D}_i(r_i)\tilde{f}_{i+1}^{-1}$ 
have a non-constant common factor, 
then we have $\rank \tilde{D}_i=3$. 
\end{lem}
\begin{proof}
Suppose to the contrary that 
$\tilde{f}_{i+1}$ is not a factor of $\tilde{g}_i$. 
Then, 
$\pl \tilde{D}_i$ is not contained in 
$\tilde{f}_{i+1}\ker \tilde{D}_i$. 
By the last part of Lemma~\ref{lem:rank criterion2}, 
it follows that 
$(r_i+\ker \tilde{D}_i)\cap \tilde{f}_{i+1}\kx \neq \emptyset $. 
Since $\ker \tilde{D}_i=k[f_i,\tilde{f}_{i+1}]$, 
this implies that 
$$
S:=(r_i+k[f_i])\cap \tilde{f}_{i+1}\kx \neq \emptyset . 
$$ 
Note that $S$ is contained in 
$\tilde{P}:=k[f_i,r_i]\cap \tilde{f}_{i+1}\kx $. 
Since $\ker \tilde{D}_i=k[f_i,\tilde{f}_{i+1}]$ 
and $\tilde{D}_i(r_i)\neq 0$, 
we know by Lemma~\ref{lem:minimal} (iii) that 
$\tilde{P}$ 
is a principal prime ideal of $k[f_i,r_i]$. 
As mentioned, 
$\tilde{q}_i$ is an irreducible element of $k[f_i,r_i]$ 
by (\ref{eq:tilde{q}_i}). 
Since $\tilde{q}_i=f_{i-1}\tilde{f}_{i+1}$ 
belongs to $\tilde{f}_{i+1}\kx $, 
we conclude that 
$\tilde{P}$ is generated by $\tilde{q}_i$. 
Because $S$ is contained in $\tilde{P}$, 
this implies that 
$\deg _{r_i}\tilde{q}_i\leq 1$. 
On the other hand, 
we see from (\ref{eq:tilde{q}_i}) that 
$$
\deg _{r_i}\tilde{q}_i=\deg _z\tilde{h}_i(y,z)
=\deg _z\tilde{\eta }_i(y,z)=a_i\geq 2, 
$$
a contradiction. 
Therefore, 
$\tilde{f}_{i+1}$ is a factor of $\tilde{g}_i$. 
In particular, 
we get $\pl \tilde{D}_i\neq \ker \tilde{D}_i$. 
Since $\tilde{D}_i$ is irreducible, 
this implies that $\rank \tilde{D}_i\geq 2$ 
by Lemma~\ref{lem:slice irred}.

Assume that 
$\tilde{g}_i$ and 
$\tilde{D}_i(r_i)\tilde{f}_{i+1}^{-1}$ 
have a common factor $p\in \kx \sm k$. 
Since $\tilde{D}_i(r_i)\tilde{f}_{i+1}^{-1}$ 
belongs to $k[f_i]$, 
and $k[f_i]$ is factorially closed in $\kx $ 
by Lemma~\ref{lem:facts}, 
it follows that $p$ belongs to $k[f_i]\sm k$. 
Because $f_i$ and $\tilde{f}_{i+1}$ 
are algebraically independent over $k$, 
we know that $p$ and $\tilde{f}_{i+1}$ are 
algebraically independent over $k$. 
Therefore, 
we get $\rank \tilde{D}_i=3$ by 
Proposition~\ref{prop:rank3criterion}. 
\end{proof}

Now, 
we prove Theorem~\ref{prop:rank2} (ii). 
Assume that $i=2$ and 
$c:=\lambda (y)$ belongs to $k^{\times }$. 
Then, 
we have $\tilde{D}_2(r_i)=c\tilde{f}_3$. 
Since $\tilde{g}_2$ is a factor of $\tilde{D}_2(r_i)$, 
and is divisible by 
$\tilde{f}_3$ by Lemma~\ref{lem:rank 1}, 
it follows that $\tilde{g}_2\approx \tilde{f}_3$. 
Hence, 
we get $\pl \tilde{D}_2=(\tilde{f}_3)$. 
This implies that $\pl \tilde{D}_2\neq \ker \tilde{D}_2$. 
Hence, 
we have $\rank \tilde{D}_2\geq 2$ 
by Lemma~\ref{lem:slice irred}. 
We show that $\tilde{f}_3$ 
is a coordinate of $\kx $ over $k$. 
Then, 
it follows that $\rank \tilde{D}_2=2$, 
and the proof is completed. 
Put $g=c^{-1}\mu (f_2,x_1)x_1^{-1}$. 
Then, 
$g$ belongs to $k[x_1,f_2]$, 
since $\mu (y,z)$ is an element of $zk[y,z]$. 
Let $\bar{D}$ be the triangular derivation of $\kx $ 
defined by $\bar{D}(x_1)=0$, 
$\bar{D}(x_2)=x_1$ and $\bar{D}(x_3)=\theta '(x_2)$. 
Then, 
we have $\bar{D}(f_2)=0$, 
and so $\bar{D}(g)=0$. 
Hence, 
$-g\bar{D}$ belongs to $\lnd _k\kx $. 
Define $\bar{\sigma }=\exp (-g\bar{D})$. 
Then, 
we have $\bar{\sigma }(x_1)=x_1$ and 
$$
\bar{\sigma }(x_2)
=x_2-gx_1
=x_2-c^{-1}\mu (f_2,x_1)
=c^{-1}r_2. 
$$ 
Since $f_2=x_1x_3-\theta (x_2)$ is equal to 
$\bar{\sigma }(f_2)=x_1\bar{\sigma }(x_3)
-\theta (c^{-1}r_2)$, 
we know that  
$$
\bar{\sigma }(x_3)=x_3+\bigl(
\theta (c^{-1}r_2)-\theta (x_2)\bigr)x_1^{-1}
=\bigl( 
x_1x_3-\theta (x_2)-\theta (c^{-1}r_2)
\bigr) x_1^{-1}
=c^{-a_2}\tilde{f}_3. 
$$
Therefore, 
$\tilde{f}_3$ is a coordinate of $\kx $ over $k$. 
This proves 
Theorem~\ref{prop:rank2} (ii).

\begin{lem}\label{lem:rank 2}
Assume that $i\geq 3$. 
If $\lambda (0)\neq 0$, 
then $f_i$ is a factor of $\tilde{g}_i$. 
\end{lem}
\begin{proof}
Suppose to the contrary that 
$f_i$ is not a factor of $\tilde{g}_i$. 
Then, 
$\pl \tilde{D}_i$ is not contained in $f_i\ker \tilde{D}_i$. 
Set $\mathfrak{p}_i=f_i\kx $. 
Then, 
we know by the last part of Lemma~\ref{lem:rank criterion2} 
that $(r_i+\ker \tilde{D}_{i})\cap \mathfrak{p}_i$ is not empty. 
Since $\ker \tilde{D}_i=k[f_i,\tilde{f}_{i+1}]$, 
it follows that 
$(r_i+k[\tilde{f}_{i+1}])\cap \mathfrak{p}_i$ is not empty. 
Take $\psi (z)\in k[z]$ such that 
$r_i-\psi (\tilde{f}_{i+1})$ belongs to $\mathfrak{p}_i$. 
Then, we have 
\begin{equation}\label{eq:pf:lem:rank 2}
\lambda (0)r-\mu (0,f_{i-1})-\psi (\tilde{f}_{i+1})
\equiv r_i-\psi (\tilde{f}_{i+1})
\equiv 0\pmod{\mathfrak{p}_i}. 
\end{equation}
Note that 
$(\lambda (0)r+k[f_{i-1}])\cap \mathfrak{p}_i=\emptyset $ 
by the last part of Lemma~\ref{lem:local slice}, 
since $\lambda (0)\neq 0$ by assumption, and $a_i\geq 2$. 
Hence, 
we see from (\ref{eq:pf:lem:rank 2}) 
that $\psi (z)$ does not belong to $k$. 
Define 
$\psi _1(z),\psi _2(z)\in k[z,z^{-1}]$ by 
$$
\psi _1(z)=\psi (z)^{a_i}z^{-1}
\quad\text{and}\quad 
\psi _2(z)=\lambda (0)^{-1}\bigl(
\mu (0,\psi _1(z))+\psi (z)\bigr). 
$$
We show that $\eta _{i-1}(\psi _1(z),\psi _2(z))=0$. 
Since $k[\tilde{f}_{i+1}]\cap \mathfrak{p}_i=\zs $ 
by Lemma~\ref{lem:minimal} (i), 
the image of $\tilde{f}_{i+1}$ in $\kx /\mathfrak{p}_i$ 
is transcendental over $k$. 
Hence, it suffices to verify that 
$$
h:=\eta _{i-1}(\psi _1(\tilde{f}_{i+1}),\psi _2(\tilde{f}_{i+1}))
$$ 
belongs to 
$\mathfrak{p}_i\kx _{\mathfrak{p}_i}$. 
By (\ref{eq:(y)}), 
we know that $f_{i-1}\tilde{f}_{i+1}=\tilde{q}_i$ 
is congruent to $r_i^{a_i}$ modulo $\mathfrak{p}_i$. 
Since 
$r_i\equiv \psi (\tilde{f}_{i+1})\pmod{\mathfrak{p}_i}$ 
by the choice of $\psi (z)$, 
it follows that 
$f_{i-1}\tilde{f}_{i+1}\equiv 
\psi (\tilde{f}_{i+1})^{a_i}\pmod{\mathfrak{p}_i}$. 
Hence, 
$\psi _1(\tilde{f}_{i+1})
=\psi (\tilde{f}_{i+1})^{a_i}\tilde{f}_{i+1}^{-1}$ 
is congruent to $f_{i-1}$ 
modulo $\mathfrak{p}_i\kx _{\mathfrak{p}_i}$. 
This implies that 
$$
\psi _2(\tilde{f}_{i+1})
\equiv \lambda (0)^{-1}
\bigl(\mu (0,f_{i-1})+\psi (\tilde{f}_{i+1})\bigr)
\equiv r
\pmod{\mathfrak{p}_i\kx _{\mathfrak{p}_i}} 
$$
by (\ref{eq:pf:lem:rank 2}). 
Thus, 
$h$ is congruent to $\eta _{i-1}(f_{i-1},r)=f_{i-2}f_i$ 
modulo $\mathfrak{p}_i\kx _{\mathfrak{p}_i}$, 
and therefore belongs to $\mathfrak{p}_i\kx _{\mathfrak{p}_i}$. 
This proves that 
$\eta _{i-1}(\psi _1(z),\psi _2(z))=0$.

Put $l_j=\deg _z \psi _j(z)$ for $j=1,2$. 
Then, 
we claim that 
\begin{equation}\label{eq:t_{i-1}l_1=a_{i-1}l_2}
t_{i-1}l_1=a_{i-1}l_2. 
\end{equation}
In fact, 
if $t_{i-1}l_1\neq a_{i-1}l_2$, 
then it follows from Lemma~\ref{lem:rhq} that 
$$
\deg _z\eta _{i-1}(\psi _1(z),\psi _2(z))
=\max \{ t_{i-1}l_1,a_{i-1}l_2\} >-\infty , 
$$ 
a contradiction. 
Since $\psi (z)$ does not belong to $k$ 
as mentioned, 
we have $l:=\deg _z\psi (z)\geq 1$. 
Since $a_i\geq 2$, 
it follows that 
\begin{equation}\label{eq:plinth2:pf:l_2}
l_1=\deg _z\psi _1(z)=\psi (z)^{a_i}z^{-1}=a_il-1\geq 1. 
\end{equation}
Put $m=\deg _z\mu (0,z)$. 
Then, we have 
\begin{equation}\label{eq:plinth2:pf:l_22}
l_2=\deg _z\psi _2(z)=\deg _z\bigl( 
\mu (0,\psi _1(z))+\psi (z)
\bigr) \leq \max \{ l_1m,l\} , 
\end{equation}
in which the equality holds when 
$l_1m\neq l$.

First, 
assume that $\mu (0,z)=0$. 
Then, 
we have $\psi _2(z)\approx \psi (z)$, 
and so $l_2=l$. 
Hence, 
it follows from 
(\ref{eq:plinth2:pf:l_2}) and (\ref{eq:a_i zenkasiki}) 
that 
\begin{align*}
t_{i-1}l_1-a_{i-1}l_2
=t_{i-1}(a_il-1)-a_{i-1}l
=(t_{i-1}a_i-a_{i-1})l-t_{i-1}
=a_{i+1}l-t_{i-1}. 
\end{align*}
Since $l\geq 1$, 
and $a_{i+1}>t_{i+1}$ by Lemma~\ref{lem:a_i} (iii), 
we know that $a_{i+1}l-t_{i-1}>0$. 
Thus, 
we get a contradiction to 
(\ref{eq:t_{i-1}l_1=a_{i-1}l_2}).

Next, 
assume that $\mu (0,z)\neq 0$. 
Then, 
$\mu (0,z)$ belongs to $zk[z]\sm \zs $, 
since $\mu (y,z)$ is an element of $zk[y,z]$. 
Hence, 
we have $m=\deg _z\mu (0,z)\geq 1$.

Assume that $i=3$. 
If $l_2=l_1m$, 
then we have $t_2l_1=a_2l_2=a_2l_1m$ by 
(\ref{eq:t_{i-1}l_1=a_{i-1}l_2}). 
Hence, 
we get $m=t_2/a_2=t_0/(t_0-1)$. 
Since $t_0\geq 3$, 
it follows that $m$ is not an integer, 
a contradiction.

Assume that $l_2\neq l_1m$. 
Then, 
we have $l_1m\leq l$ 
in view of (\ref{eq:plinth2:pf:l_22}) 
and the note following it. 
By (\ref{eq:plinth2:pf:l_2}), 
it follows that 
$$
l\geq l_1m=(a_3l-1)m=\bigl( (a_3-1)l+(l-1)\bigr) m. 
$$
Since $a_3\geq 2$, $l\geq 1$ and $m\geq 1$, 
this implies that 
$l=m=l_1=1$ and $a_3=2$. 
Since $l_1m\leq l$, 
we have $l_2\leq l$ 
by (\ref{eq:plinth2:pf:l_22}), 
and so $l_2\leq 1$. 
By the assumption that $l_2\neq l_1m$, 
it follows that $l_2\leq 0$. 
Hence, 
we get $3\leq t_2=t_2l_1=a_2l_2\leq 0$ 
by (\ref{eq:t_{i-1}l_1=a_{i-1}l_2}), 
a contradiction.

Finally, 
assume that $i\geq 4$. 
Then, 
we have $a_i\geq 3$ by (i) and  (ii) of Lemma~\ref{lem:a_i}. 
Since $l\geq 1$, 
it follows that $l_1=a_il-1>l$ 
by (\ref{eq:plinth2:pf:l_2}). 
Hence, we get $l_1m>l$. 
This implies that $l_2=l_1m$ 
by (\ref{eq:plinth2:pf:l_22}) 
and the note following it. 
Since $i-1\geq 3$, 
we have $t_{i-1}<a_{i-1}$ by Lemma~\ref{lem:a_i} (iii). 
Thus, we know that  
$t_{i-1}l_1<a_{i-1}l_1\leq a_{i-1}l_1m=a_{i-1}l_2$, 
a contradiction to (\ref{eq:t_{i-1}l_1=a_{i-1}l_2}). 
Therefore, 
$f_i$ is a factor of $\tilde{g}_i$. 
\end{proof}

If $i\geq 3$, 
then $f_i$ is a factor of 
$\tilde{D}_i(r_i)$, 
and hence a factor of 
$\tilde{D}_i(r_i)\tilde{f}_{i+1}^{-1}$. 
By Lemma~\ref{lem:rank 2}, 
it follows that 
$f_i$ is a common factor of 
$\tilde{g}_i$ and  
$\tilde{D}_i(r_i)\tilde{f}_{i+1}^{-1}$ 
if 
$i\geq 3$ and $\lambda (0)\neq 0$. 
If this is the case, 
then we have $\rank \tilde{D}_i=3$ 
due to the last part of Lemma~\ref{lem:rank 1}. 
This proves Theorem~\ref{prop:rank2} (i) 
when $\lambda (0)\neq 0$. 
To complete the proof of Theorem~\ref{prop:rank2}, 
we have only to prove (i) 
when $\lambda (0)=0$, and (iii). 
Therefore, 
we assume that $\lambda (y)$ is not an element of 
$k$ in what follows.

Let $\alpha \in \bar{k}$ be a root of $\lambda (y)$, 
and $\lambda _{\alpha }(y)$ 
the minimal polynomial of $\alpha $ over $k$. 
We define $m\in \N $ to be the maximal number such that 
$\lambda _{\alpha }(y)^m$ is a factor of 
$\lambda (y)$ if $i=2$, 
and of $y\lambda (y)$ if $i\geq 3$. 
Then, we have the following lemma.

\begin{lem}\label{lem:rank alpha}
If $a_i\geq 3$ or $\deg _z\mu (\alpha ,z)\geq 2$, 
then $\lambda _{\alpha }(f_i)^m$ 
is a factor of $\tilde{g}_i$. 
\end{lem}
\begin{proof}
Since $p:=\lambda _{\alpha }(f_i)$  
is an irreducible element of $k[f_i]$, 
and $k[f_i]$ is factorially closed in $\kx $ 
by Lemma~\ref{lem:facts}, 
it follows that $p$ is an irreducible element of $\kx $. 
Since $\tilde{D}_i(r_i)$ is divisible by 
$\lambda (f_i)$ if $i=2$, 
and by $\lambda (f_i)f_i$ if $i\geq 3$, 
we have $v_p(\tilde{D}_i(r_i))\geq m$ 
by the definition of $m$. 
We show that 
$(r_i+\ker \tilde{D}_i)\cap p\kx =\emptyset $. 
Then, 
it follows that 
$v_p(\tilde{g}_i)=v_p(\tilde{D}_i(r_i))\geq m$ 
by Lemma~\ref{lem:rank criterion2}, 
and the lemma is proved.

Suppose to the contrary that 
$(r_i+\ker \tilde{D}_i)\cap p\kx $ 
is not empty. 
Since $\ker \tilde{D}_i=k[f_i,\tilde{f}_{i+1}]$, 
we may find $\nu (y,z)\in k[y,z]$ such that 
$r_i+\nu (f_i,\tilde{f}_{i+1})$ belongs to $p\kx $. 
Then, 
$r_i+\nu (f_i,\tilde{f}_{i+1})$ 
belongs to 
$\mathfrak{p}:=(f_i-\alpha )\bar{k}[\x ]$. 
Put $\bar{\mu }(z)=-\mu (\alpha ,z)$ 
and $\bar{\nu }(z)=\nu (\alpha ,z)$. 
Then, 
$\bar{\mu}(z)$ belongs to $z\bar{k}[z]$, 
since $\mu (y,z)$ is an element of $zk[y,z]$. 
We show that 
$\psi (z):=\bar{\mu }(z)
+\bar{\nu }\bigl(\bar{\mu}(z)^{a_i}z^{-1}\bigr)=0$. 
As remarked after Lemma~\ref{lem:facts}, 
$D_i$ extends 
to an element of 
$\lnd _{\bar{k}}\bar{k}[\x ]$ with kernel 
$\bar{k}[f_{i-1},f_i-\alpha ]$. 
Hence, 
$\mathfrak{p}$ is a prime ideal of $\bar{k}[\x ]$ 
by Lemma~\ref{lem:facts}. 
Since $\bar{k}[f_{i-1}]\cap \mathfrak{p}=\zs $ 
by Lemma~\ref{lem:minimal} (i), 
the image of $f_{i-1}$ in $\bar{k}[\x ]/\mathfrak{p}$ 
is transcendental over $\bar{k}$. 
Hence, 
it suffices to verify that 
$\psi (f_{i-1})\equiv 0\pmod{\mathfrak{p}}$. 
Note that 
$$
f_{i-1}\tilde{f}_{i+1}=
\tilde{\eta }_i(f_i,r_i\lambda (f_i)^{-1})\lambda (f_i)^{a_i}
$$
and $\tilde{\eta }_i(y,z)$ is a monic polynomial 
in $z$ of degree $a_i$ over $k[y]$. 
Since $\lambda (\alpha )=0$, 
we see that 
$\tilde{\eta }_i(f_i,r_i\lambda (f_i)^{-1})\lambda (f_i)^{a_i}
\equiv r_i^{a_i}\pmod{\mathfrak{p}}$, 
while 
$r_i=\lambda (f_i)\tilde{r}-\mu (f_i,f_{i-1})$ is congruent to 
$\lambda (\alpha )\tilde{r}-\mu (\alpha ,f_{i-1})
=\bar{\mu }(f_{i-1})$ 
modulo $\mathfrak{p}$. 
Hence, 
it follows that 
$f_{i-1}\tilde{f}_{i+1}\equiv \bar{\mu }(f_{i-1})^{a_i}
\pmod{\mathfrak{p}}$, 
and so 
$\tilde{f}_{i+1}\equiv \bar{\mu }(f_{i-1})^{a_i}f_{i-1}^{-1}
\pmod{\mathfrak{p}}$. 
Thus, we get 
\begin{align*}
&\psi (f_{i-1})
=\bar{\mu }(f_{i-1})
+\bar{\nu }(\bar{\mu}(f_{i-1})^{a_i}f_{i-1}^{-1}) \\ 
&\quad 
\equiv r_i+\bar{\nu }(\tilde{f}_{i+1})
\equiv r_i+\nu (f_i,\tilde{f}_{i+1})
\equiv 0\pmod{\mathfrak{p}}
\end{align*}
by the choice of $\nu (y,z)$. 
This proves that $\psi (z)=0$. 
Hence, we know that 
$\bar{\mu }(z)=-\bar{\nu }(\bar{\mu}(z)^{a_i}z^{-1})$. 
From this, 
we obtain 
\begin{equation}\label{eq:plinth2:alpha:pf}
l:=\deg _z\bar{\mu }(z)
=\deg _z\bar{\nu }(\bar{\mu}(z)^{a_i}z^{-1})
=(a_il-1)\deg _z\bar{\nu }(z). 
\end{equation}
Note that 
$a_i\geq 3$ or $l\geq 2$ by assumption. 
We claim that $l\geq 1$. 
In fact, 
$\bar{\mu }(z)=-\mu (\alpha ,z)$ belongs to $z\bar{k}[z]$, 
and is nonzero 
by the assumption that $\lambda (y)$ and $\mu (y,z)$ 
have no common factor. 
Since $a_i\geq 2$, 
it follows from (\ref{eq:plinth2:alpha:pf}) 
that $\deg _z\bar{\nu }(z)=1$, 
$a_i=2$ and $l=1$, a contradiction. 
Therefore, 
$(r_i+\ker \tilde{D}_i)\cap p\kx $ is empty. 
\end{proof}

In the situation of Lemma~\ref{lem:rank alpha}, 
we have $\rank \tilde{D}_i=3$ 
due to the last part of Lemma~\ref{lem:rank 1}. 
In particular, 
we have $\rank \tilde{D}_i=3$ if $a_i\geq 3$. 
We note that $a_i\geq 3$ in the following cases: 

\smallskip 

\noindent{\rm (I)} 
$i=2$ and $t_0\geq 4$. 

\smallskip 

\noindent{\rm (II)} 
$i\geq 3$ and $(t_0,t_1,i)\neq (4,1,3)$. 

\smallskip 

Actually, 
we have $a_2=t_0-1\geq 3$ in the case of (I). 
If $i\geq 3$, 
then we have $t_0\geq 3$ and $(t_0,t_1)\neq (3,1)$ by assumption. 
Hence, 
we know that 
$a_3=t_1(t_0-1)-1\geq 3$ 
if $(t_0,t_1)\neq (4,1)$. 
If $i\geq 4$, 
then it follows that $a_i\geq 3$ 
by (i) and (ii) of Lemma~\ref{lem:a_i}.

Similarly, 
we have $\rank \tilde{D}_i=3$ 
if $\deg _z\mu (\alpha ,z)\geq 2$ 
for some $\alpha \in \bar{k}$ with $\lambda (\alpha )=0$. 
This condition is equivalent to the following condition: 

\smallskip 

\noindent{\rm (III)} $\mu _j(y)$ does not belong to 
$\sqrt{(\lambda (y))}$ for some $j\geq 2$. 

\smallskip 

From (I) and (III), 
we get the first part of Theorem~\ref{prop:rank2} (iii). 
From (II), (III) and the discussion after Lemma~\ref{lem:rank 2}, 
we see that Theorem~\ref{prop:rank2} (i) is proved 
except for the case where 
$\lambda (0)=0$, $(t_0,t_1,i)=(4,1,3)$ 
and $\mu _j(y)$ belongs to $\sqrt{(\lambda (y))}$ 
for every $j\geq 2$. 
To complete the proof of Theorem~\ref{prop:rank2}, 
it suffices to prove 
(i) in this exceptional case, 
and the last part of (iii). 
Therefore, 
we are reduced to proving the following statements 
under the assumption that 
$\mu _j(y)$ belongs to $\sqrt{(\lambda (y))}$ 
for every $j\geq 2$: 

\smallskip 

\noindent(1) 
If $\lambda (0)=0$ and $(t_0,t_1)=(4,1)$, 
then we have $\rank \tilde{D}_3=3$. 

\smallskip 

\noindent(2) 
If 
$t_0=3$, 
then we have $\tilde{g}_2\approx \tilde{f}_3$. 

\smallskip

We remark that 
$\lambda (y)$ and $\mu (y,z)-\mu _1(y)z$ 
have a non-constant common factor 
under the assumption above, 
since we are assuming that 
$\lambda (y)$ is not an element of $k$. 
Hence, 
$\lambda (y)$ and $\mu _1(y)$ 
have no common factor, 
since so do $\lambda (y)$ and $\mu (y,z)$.

First, we prove (1). 
By Lemma~\ref{lem:facts}, 
$f_3$ and $\tilde{f}_4$ are irreducible elements of $\kx $ 
with $f_3\not\approx \tilde{f}_4$. 
Since $\tilde{D}_3(r_3)=\lambda (f_3)f_3\tilde{f}_4$, 
and $\lambda (0)=0$ by assumption, 
we have 
$$
a:=v_{f_3}(\tilde{D}_3(r_3))\geq 2
\quad\text{and}\quad 
v_{f_3}(\lambda (f_3))=a-1, 
$$
and so $v_{f_3}(\lambda (f_3)^2)=2(a-1)\geq a$. 
Hence, 
we get 
$\lambda (f_3)^2\equiv 0\pmod{f_3^a\kx }$. 
Set 
$$
f=\bigl( 
\mu (f_3,f_2)^2-2\lambda (f_3)\mu (f_3,f_2)r
\bigr)f_2^{-1}. 
$$
Then, 
$f$ belongs to $\kx $, 
since $\eta (y,z)$ is an element of  $zk[y,z]$. 
We show that 
$\tilde{f}_4\equiv f\pmod{f_3^a\kx }$. 
Since $(t_0,t_1)=(4,1)$, 
we have $a_3=t_1(t_0-1)-1=2$ 
and $\tilde{\eta }_3(y,z)=\eta _3(y,z)=y+z^2$. 
Hence, we get 
$$
f_2\tilde{f}_4
=\tilde{\eta }_3(f_3,r_3\lambda (f_3)^{-1})\lambda (f_3)^2
=f_3\lambda (f_3)^2+r_3^2. 
$$
From this, 
we see that $f_2\tilde{f}_4$ is congruent to 
$$
r_3^2=(\lambda (f_3)\tilde{r}-\mu (f_3,f_2))^2
=\lambda (f_3)^2r^2-2\lambda (f_3)r\mu (f_3,f_2)
+\mu (f_3,f_2)^2, 
$$
and hence to $f_2f$ modulo $f_3^a\kx $. 
Since $f_2$ and $f_3$ have no common factor, 
it follows that 
$\tilde{f}_4\equiv f\pmod{f_3^a\kx }$.

To conclude that $\rank D=3$, 
it suffices to show that 
$f_3$ is a factor of $\tilde{g}_3$ 
thanks to the last part of Lemma~\ref{lem:rank 1}, 
since $f_3$ is a factor of 
$\tilde{D}_3(r_3)\tilde{f}_4^{-1}=\lambda (f_3)f_3$. 
Suppose to the contrary that $v_{f_3}(\tilde{g}_3)=0$. 
Then, 
$(r_3+\ker \tilde{D}_3)\cap f_3^a\kx $ 
is not empty 
by Lemma~\ref{lem:rank criterion2}. 
Since $\ker \tilde{D}_3=k[f_3,\tilde{f}_4]$, 
we may find $\nu (y,z)\in k[y,z]$ such that 
$r_3+\nu (f_3,\tilde{f}_4)$ belongs to $f_3^a\kx $. 
Since $\tilde{f}_4\equiv f\pmod{f_3^a\kx }$, 
it follows that 
$r_3+\nu (f_3,f)$ belongs to $f_3^a\kx $. 
We note that 
$r_3$ and $f$ are linear polynomials in $r$ 
over $k[f_2,f_3]$ 
whose leading coefficients are multiples of $\lambda (f_3)$. 
Here, $r$, $f_2$ and $f_3$ 
are algebraically independent over $k$ 
because  
$$df_3\wedge df_2\wedge dr
=D_3(r)dx_1\wedge dx_2\wedge dx_3\neq 0. 
$$
Hence, 
we may uniquely write 
$$
r_3+\nu (f_3,f)
=\sum _{i,j}\psi _{i,j}(f_3)r^if_2^j, 
$$
where $\psi _{i,j}(y)$ 
is an element of $\lambda (y)^ik[y]$ 
for each $i$ and $j$. 
We show that 
$v_y(\psi _{1,0}(y))$ or $v_y(\psi _{0,1}(y))$ 
is less than $a$. 
Let $(r,f_2)$ be the ideal of $k[r,f_2,f_3]$ 
generated by $r$ and $f_2$. 
Then, 
we have $\mu (f_3,f_2)\equiv \mu _1(f_3)f_2\pmod{(r,f_2)^2}$. 
Hence, we get 
\begin{gather*}
r_3\equiv 
\lambda (f_3)r-\mu _1(f_3)f_2
\pmod{(r,f_2)^2} \\
f\equiv \mu _1(f_3)^2f_2-2\lambda (f_3)\mu _1(f_3)r
\pmod{(r,f_2)^2}. 
\end{gather*}
Write $\nu (y,z)=\sum _{j\geq 0}\nu _j(y)z^j$, 
where $\nu _j(y)\in k[y]$ for each $j$. 
Then, 
we have 
$$
\nu (f_3,f)\equiv \nu _1(f_3)f+\nu _0(f_3)
\pmod{(r,f_2)^2}, 
$$ 
since $f$ belongs to $(r,f_2)$. 
Hence, 
$r_3+\nu (f_3,f)-\nu _0(f_3)$ is congruent to 
\begin{align*}
&\bigl(\lambda (f_3)r-\mu _1(f_3)f_2\bigr)
+\nu _1(f_3)
\bigl(
\mu _1(f_3)^2f_2-2\lambda (f_3)\mu _1(f_3)r\bigr)\\
&\quad =
\bigl(1-2\nu _1(f_3)\mu _1(f_3)\bigr)\lambda (f_3)r
+\mu _1(f_3)\bigl(\nu _1(f_3)\mu _1(f_3)-1\bigr)f_2
\end{align*}
modulo $(r,f_2)^2$. 
From this, 
we see that  
$$
\psi _{1,0}(y)=\bigl(1-2\nu _1(y)\mu _1(y)\bigr)\lambda (y)
\quad \text{and}\quad 
\psi _{0,1}(y)=\mu _1(y)\bigl(\nu _1(y)\mu _1(y)-1\bigr). 
$$
Now, 
assume that $v(\psi _{1,0}(y))\geq a$. 
Then, 
$1-2\nu _1(y)\mu _1(y)$ belongs to $yk[y]$, 
since $v_y(\lambda (y))=v_{f_3}(\lambda (f_3))=a-1$. 
This implies that 
$\nu _1(y)\mu _1(y)-1$ does not belong to $yk[y]$. 
Since 
$\lambda (0)=0$ by assumption, 
and $\lambda (y)$ and $\mu _1(y)$ 
have no common factor as mentioned, 
we know that 
$\mu _1(y)$ does not belong to $yk[y]$. 
Hence, 
$\psi _{0,1}(y)$ does not belong to $yk[y]$. 
Thus, 
we get $v_y(\psi _{0,1}(y))=0<a$. 
Therefore, 
$v_y(\psi _{1,0}(y))$ or $v_y(\psi _{0,1}(y))$ 
is less than $a$. 
Consequently, 
we have 
$$
b:=\min \{ v_{y}(\psi _{i,j}(y))\mid i,j\} <a. 
$$ 
Since 
$r_3+\nu (f_3,f)$ belongs to $f_3^a\kx $, 
it follows that 
$g:=(r_3+\nu (f_3,f))f_3^{-b}$ belongs to $f_3\kx $. 
Set $\bar{\psi }_{i,j}(y)=\psi _{i,j}(y)y^{-b}$ 
for each $i$ and $j$. 
Then, $\bar{\psi }_{i,j}(y)$ belongs to $k[y]$ 
for every $i$ and $j$, 
and $\bar{\psi }_{i,j}(0)\neq 0$ 
for some $i$ and $j$. 
Hence, 
$h:=\sum _{i,j}\bar{\psi }_{i,j}(0)r^if_2^j$ 
is a nonzero element of $k[r,f_2]$. 
We note that $\deg _rh\leq 1$. 
In fact, 
we have 
$$
v_y(\psi _{i,j}(y))\geq v_{y}(\lambda (y)^2)
=v_{f_3}(\lambda (f_3)^2)\geq a 
$$
if $i\geq 2$, 
since $\psi _{i,j}(y)$ 
is an element of $\lambda (y)^ik[y]$. 
Since 
$g=\sum _{i,j}\bar{\psi }_{i,j}(f_3)r^if_2^j$, 
we have $h\equiv g\pmod{f_3\kx }$. 
Because $g$ belongs to $f_3\kx $, 
it follows that $h$ belongs to $f_3\kx $. 
Thus, 
$h$ belongs to $k[f_2,r]\cap f_3\kx \sm \zs $. 
By Lemma~\ref{lem:local slice}, 
we have $k[f_2,r]\cap f_3\kx =q_2k[f_2,r]$. 
Therefore, 
we get $\deg _rh\geq \deg _rq_2=a_2=t_0-1=3$, 
a contradiction. 
This proves that $\rank \tilde{D}_3=3$, 
and thereby completing the proof of (1).

Next, 
we prove (2). 
Suppose to the contrary that 
$\tilde{g}_2\not\approx \tilde{f}_3$. 
Then, 
$\tilde{g}_2\tilde{f}_3^{-1}$ 
belongs to $\kx \sm k$, 
since 
$\tilde{f}_3$ is a factor of $\tilde{g}_2$ 
by Lemma~\ref{lem:rank 1}. 
Let $p$ be a factor of $\tilde{g}_2\tilde{f}_3^{-1}$ 
which is an irreducible element of $\kx $. 
Since $\tilde{g}_2$ 
is a factor of $\lambda (f_2)\tilde{f}_3$, 
we know that 
$\tilde{g}_2\tilde{f}_3^{-1}$ 
is a factor of $\lambda (f_2)$. 
Hence, 
$p$ is a factor of $\lambda (f_2)$. 
Since $k[f_2]$ is factorially closed in $\kx $ 
by Lemma~\ref{lem:facts}, 
it follows that $p$ belongs to $k[f_2]$. 
By Lemma~\ref{lem:facts}, 
$\tilde{f}_3$ is an irreducible element of $\kx $, 
and does not belong to $k[f_2]$. 
Hence, 
$p$ does not divide $\tilde{f}_3$. 
Thus, we get 
\begin{equation}\label{eq:rank a}
a:=v_p\bigl(\tilde{D}_2(r_2)\bigr)
=v_p\bigl(\lambda (f_2)\tilde{f}_3\bigr)
=v_p\bigl(\lambda (f_2)\bigr)
\geq 1. 
\end{equation}
Set $\mathfrak{q}=p^a\kx $. 
Then, 
we have $(r_2+\ker \tilde{D}_2)\cap \mathfrak{q}=\emptyset $ 
by Lemma~\ref{lem:rank criterion2}, 
since $v_p(\tilde{g}_2)
=v_p(\tilde{g}_2\tilde{f}_3^{-1})\geq 1$ 
by the choice of $p$.

We define 
$\tilde{f}=\mu (f_2,x_1)^2x_1^{-1}$. 
Then, $\tilde{f}$ belongs to $x_1k[x_1,f_2]$. 
We show that 
$\tilde{f}_3\equiv \tilde{f}\pmod{\mathfrak{q}}$. 
By (\ref{eq:rank a}), 
we have 
$\lambda (f_2)\equiv 0\pmod{\mathfrak{q}}$. 
Since $\tilde{\eta }_2(y,z)=y+\alpha _1^0+\alpha _2^0z+z^2$, 
we see that 
$x_1\tilde{f}_3=
\tilde{\eta }_2(f_2,r_2\lambda (f_2)^{-1})\lambda (f_2)^{2}$ 
is congruent to $r_2^{2}$ modulo $\mathfrak{q}$. 
On the other hand, 
$r_2=\lambda (f_2)x_2-\mu (f_2,x_1)$ 
is congruent to $-\mu (f_2,x_1)$ 
modulo $\mathfrak{q}$. 
Hence, 
we get 
$x_1\tilde{f}_3\equiv \mu (f_2,x_1)^2\pmod{\mathfrak{q}}$. 
Since $p$ is an element of $k[f_2]$, 
we have $x_1\not\approx p$. 
Therefore, 
$\tilde{f}_3$ 
is congruent to 
$\mu (f_2,x_1)^2x_1^{-1}=\tilde{f}$ 
modulo $\mathfrak{q}$.

The following is a key claim.

\begin{claim}\label{claim:plinth2calS}
There exists 
$h\in \mu (f_2,x_1)+k[f_2,\tilde{f}]$ 
such that $v_p(h)\geq a$. 
\end{claim}

Assuming this claim, 
we can complete the proof of (2) as follows. 
By the claim, 
there exists $\nu (y,z)\in k[y,z]$ 
such that $\mu (f_2,x_1)+\nu (f_2,\tilde{f})$ 
belongs to $\mathfrak{q}$. 
Then, 
we have 
$$
r_2-\nu (f_2,\tilde{f}_3)
\equiv -\mu (f_2,x_1)-\nu (f_2,\tilde{f})\equiv 0
\pmod{\mathfrak{q}}, 
$$ 
since $r_2\equiv -\mu (f_2,x_1)\pmod{\mathfrak{q}}$ 
and $\tilde{f}_3\equiv \tilde{f}\pmod{\mathfrak{q}}$.  
Hence, 
$r_2-\nu (f_2,\tilde{f}_3)$ belongs to $\mathfrak{q}$. 
This contradicts that 
$(r_2+\ker \tilde{D}_2)\cap \mathfrak{q}=\emptyset $, 
and thereby proving 
$\tilde{g}_2\approx \tilde{f}_3$.

Let us prove Claim~\ref{claim:plinth2calS}. 
Note that 
$\mu (f_2,x_1)+\tilde{f}k[f_2,\tilde{f}]$ 
is contained in $x_1k[x_1,f_2]$, 
since $\mu (f_2,x_1)$ and $\tilde{f}$ belong to 
$x_1k[x_1,f_2]$. 
Hence, 
each element of $\mu (f_2,x_1)+\tilde{f}k[f_2,\tilde{f}]$ 
is written as 
$$
h=\sum _{j\geq 1}h_jx_1^j,
\text{ \ where \ }h_j\in k[f_2]\text{ \ for each \ 	}j\geq 1. 
$$
By assumption, 
$\mu_j(y)$ belongs to $\sqrt{(\lambda (y))}$ 
for each $j\geq 2$. 
Since $p$ is a factor of $\lambda (f_2)$, 
and is an irreducible element of $\kx $, 
this implies that 
$p$ divides $\mu _j(f_2)$ for each $j\geq 2$. 
Hence, 
we have 
$$
b:=\min \left\{ 
\frac{v_p(\mu _j(f_2))}{j-1}\Bigm| j\geq 2
\right\} >0. 
$$
Let ${\mathcal S}$ be the set of 
$h\in \mu (f_2,x_1)+\tilde{f}k[f_2,\tilde{f}]$ 
such that 
\begin{equation}\label{eq:plinth2:cal S}
v_p(h_j)\geq (j-1)b
\text{ \ for each \ }j\geq 2. 
\end{equation}
Then, 
$\mu (f_2,x_1)$ belongs to ${\mathcal S}$, 
since 
\begin{equation}\label{eq:by the definition of b}
v_p(\mu _j(f_2))\geq (j-1)b\text{ \ for each \ }j\geq 2
\end{equation}
by the definition of $b$. 
Hence, 
${\mathcal S}$ is not empty. 
To prove Claim~\ref{claim:plinth2calS}, 
it suffices to verify that 
$v_p(h)\geq a$ for some $h\in {\mathcal S}$. 
Suppose to the contrary that 
$v_p(h)<a$ for all $h\in {\mathcal S}$. 
Then, 
we can find 
$$
l(h):=\min \{ l\in \N \mid v_p(h_l)<a\} 
$$
for each $h\in {\mathcal S}$. 
Actually, 
if $v_p(h_l)\geq a$ for every $l\in \N $ 
for $f\in {\mathcal S}$, 
then we have $v_p(h)\geq a$. 
By (\ref{eq:plinth2:cal S}), 
we have 
\begin{equation}\label{eq:plinth2:cal S1}
\bigl( l(h)-1\bigr)b\leq v_p(h_{l(h)})<a
\end{equation}
for each $h\in {\mathcal S}$. 
Hence, 
we can find 
$l:=\max \{ l(h)\mid h\in {\mathcal S}\} $. 
Take $h\in {\mathcal S}$ and $u(y)\in k[y]$ 
such that $l=l(h)$ and $h_l=u(f_2)$. 
Since $\lambda (y)$ and $\mu _1(y)$ 
have no common factor, 
the same holds for 
$\lambda (y)$ and $\mu _1(y)^{2l}$. 
Hence, 
we may find $u_1(y),u_2(y)\in k[y]$ such that 
\begin{equation}\label{eq:Euclid}
u_1(y)\lambda (y)+u_2(y)\mu _1(y)^{2l}=u(y). 
\end{equation}
Since $k[f_2]$ is factorially closed in $\kx $, 
we see that 
$\lambda (f_2)$ and $\mu _1(f_2)$ have no common factor. 
Since $p$ is a factor of $\lambda (f_2)$, 
it follows that $v_p(\mu _1(f_2))=0$. 
Hence, we get 
\begin{equation}\label{eq:plinth2:cal S2}
\begin{aligned}
&v_p\bigl(u_2(f_2)\bigr)
=v_p\bigl(u_2(f_2)\mu _1(f_2)^{2l}\bigr)
=v_p\bigl(
u(f_2)-u_1(f_2)\lambda (f_2)
\bigr) \\
&\quad 
=v_p\bigl(
h_l-u_1(f_2)\lambda (f_2)
\bigr) 
\geq 
\min \{ v_p\bigl( h_l\bigr),
v_p\bigl(u_1(f_2)\lambda (f_2)\bigr)\} \\
&\quad \geq 
\min \{ (l-1)b,a\}=
(l-1)b 
\end{aligned}
\end{equation}
in view of (\ref{eq:rank a}) and 
(\ref{eq:plinth2:cal S1}).

Write 
$$
\mu (f_2,x_1)^{2l}=
\left( \sum _{j\geq 1}\mu _j(f_2)x_1^j
\right) ^{2l}=
\sum _{j\geq l}g_{j}x_1^{j+l}, 
$$ 
where 
$g_{j}$ is the sum of 
$\prod _{t=1}^{2l}\mu _{j_t}(f_2)$ 
for $j_1,\ldots ,j_{2l}\in \N $ 
such that $\sum _{t=1}^{2l}j_t=j+l$ 
for each $j\geq l$. 
Then, 
we have $g_l=\mu _1(f_2)^{2l}$, 
and $v_p(g_j)\geq (j-l)b$ for each $j\geq l$, 
since 
$$
v_p\left(\prod _{t=1}^{2l}\mu _{j_t}(f_2)\right)
=\sum _{t=1}^{2l}v_p\bigl(\mu _{j_t}(f_2)\bigr)
\geq \sum _{t=1}^{2l}(j_t-1)b=(j-l)b
$$
by (\ref{eq:by the definition of b}). 
Hence, 
it follows from (\ref{eq:plinth2:cal S2}) that 
\begin{equation}\label{eq:cal S}
v_p\bigl(u_2(f_2)g_{j}\bigr)=v\bigl(u_2(f_2)\bigr)+v_p(g_{j})
\geq (l-1)b+(j-l)b=(j-1)b 
\end{equation}
for each $j\geq l$.

Now, consider the polynomial 
$$
h':=h-u_2(f_2)\tilde{f}^l=h-u_2(f_2)\mu (f_2,x_1)^{2l}x_1^{-l}
=\sum _{j\geq 1}\bigl(h_j-u_2(f_2)g_{j}\bigr)x_1^{j}, 
$$
where $g_j:=0$ for $1\leq j<l$. 
Since $h$ is an element of 
$\mu (f_2,x_1)+\tilde{f}k[f_2,\tilde{f}]$, 
and $u_2(f_2)\tilde{f}^l$ belongs to $\tilde{f}k[f_2,\tilde{f}]$, 
it follows that $h'$ belongs to 
$\mu (f_2,x_1)+\tilde{f}k[f_2,\tilde{f}]$. 
By (\ref{eq:plinth2:cal S}) and (\ref{eq:cal S}), 
we have 
$$
v_p\bigl( 
h_j-u_2(f_2)g_{j}
\bigr) \geq \min \{ v_p(h_j),v_p\bigl( 
u_2(f_2)g_{j}
\bigr) \} \geq (j-1)b 
$$
for each $j\geq 1$. 
Hence, 
$h'$ belongs to ${\mathcal S}$. 
For each $1\leq j<l$, 
we have 
$$
v_p\bigl(h_j-u_2(f_2)g_{j}\bigr)=v_p(h_j)\geq a
$$ 
by the definition of $l=l(h)$. 
From (\ref{eq:rank a}) and (\ref{eq:Euclid}), 
we see that 
$$
v_p\bigl(h_l-u_2(f_2)g_l\bigr)
=v_p\bigl(u(f_2)-u_2(f_2)\mu _1(f_2)^{2l}\bigr)
=v_p\bigl(u_1(f_2)\lambda (f_2)\bigr)
\geq v_p(\lambda (f_2))=a. 
$$
Hence, 
we get $l(h')>l$. 
This contradicts the maximality of $l$. 
Thus, 
we have proved Claim~\ref{claim:plinth2calS}, 
and thereby proving (2). 
This completes the proof of Theorem~\ref{prop:rank2}.

\section{Further local slice constructions}
\setcounter{equation}{0}

In this section, 
we discuss how to get more examples of 
elements of $\lnd _k\kx $. 
First, we note that 
$r$ may be replaced with a polynomial 
of the more general form $x_1x_2x_3-\psi (x_1,x_2)$ 
in the construction of $(f_i)_{i=0}^{\infty }$. 
Here, 
$\psi (x_1,x_2)\in k[x_1,x_2]$ 
is such that 
$x_1x_2x_3-\psi (0,x_2)-\psi (x_1,0)=r$. 
In fact, 
since $\psi (0,0)=0$, 
we can define $\tau \in \Aut (\kx /k[x_1,x_2])$ by 
$$
\tau (x_3)=x_3+
(\psi (x_1,x_2)-\psi (0,x_2)-\psi (x_1,0))x_1^{-1}x_2^{-1}. 
$$ 
Then, we have $\tau (f_0)=f_0$, $\tau (f_1)=f_1$ and 
$\tau \bigl(x_1x_2x_3-\psi (x_1,x_2)\bigr)=r$. 
Therefore, 
we get 
$\bigl(\tau ^{-1}(f_i)\bigr)_{i=0}^{\infty }$ 
instead of $(f_i)_{i=0}^{\infty }$ 
by this construction.

In the construction of $\tilde{f}_{i+1}$ for $i\geq 3$, 
we may interchange the role of $f_{i-1}$ and $f_{i+1}$. 
Namely, 
when $i\neq \max I$, 
we consider 
$$
s_{i}:=\lambda (f_i)r-\mu (f_i,f_{i+1})
\quad\text{and}\quad 
\tilde{g}_{i-1}:=
\tilde{h}_i(f_i,s_i)\lambda (f_i)^{a_i}f_{i+1}^{-1}
$$
instead of $r_i$ and $\tilde{f}_{i+1}$. 
Since $i\neq \max I$, 
we have $\ker D_i=k[f_i,f_{i+1}]$, 
and $D_i$ is irreducible due to Theorem~\ref{thm:lsc1} (i). 
Hence, 
$-D_i=\Delta _{(f_i,f_{i+1})}$ satisfies 
(LSC1) for $f=f_{i}$ and $g=f_{i+1}$. 
By Lemma~\ref{lem:facts}, 
$\mathfrak{p}_{i+1}:=f_{i+1}\kx $ 
is a prime ideal of $\kx $. 
We claim that 
$s_{i}$ does not belong to $\mathfrak{p}_{i+1}$. 
Actually, 
$\lambda (f_i)$ and $r$ do not belong 
to $\mathfrak{p}_{i+1}$ by Lemma~\ref{lem:minimal} (i) 
and by (1) of Proposition~\ref{prop:lsc1}, 
while 
$\mu (f_i,f_{i+1})$ 
belongs to $\mathfrak{p}_{i+1}$ 
since $\mu (y,z)$ belongs to $zk[y,z]$. 
Since $D_i(f_i)=D_i(f_{i+1})=0$, 
we get 
$-D_i(s_i)=-\lambda (f_i)D_i(r)=-\lambda (f_i)f_if_{i+1}$. 
Thus, 
$-D_i$ satisfies (LSC2) for $s=s_i$ and $F=-\lambda (f_i)f_i$. 
By the irreducibility of $\tilde{h}_i(x_1,x_2)$, 
it follows that 
$f_{i+1}\tilde{g}_{i-1}=\tilde{h}_i(f_i,s_i)$ 
is an irreducible element of $k[f_i,s_i]$. 
Since $s_i\equiv \lambda (f_i)r\pmod{\mathfrak{p}_{i+1}}$, 
we see that $\tilde{h}_i(f_i,s_i)$ 
is congruent to 
$\eta _i(f_i,r)\lambda (f_i)^{a_i}=
f_{i-1}f_{i+1}\lambda (f_i)^{a_i}$ 
modulo $\mathfrak{p}_{i+1}$. 
Hence, 
$\tilde{g}_{i-1}$ belongs to $\kx $. 
Therefore, 
we know by (a) and (b) of Theorem~\ref{thm:fbook} that
$\Delta _{(f_i,\tilde{g}_{i-1})}$ is locally nilpotent 
and $\Delta _{(f_i,\tilde{g}_{i-1})}(s_i)
=\tilde{g}_{i-1}\lambda (f_i)f_i$.

Next, 
let $\Delta _{(f,h)}$ be an element of $\lnd _k\kx $ 
obtained from a data $(f,g,s)$ 
by a local slice construction. 
We consider when $\Delta _{(f,h)}$ has rank three. 
In view of Theorems~\ref{prop:rank} and \ref{prop:rank2}, 
we see that there exist many examples 
in which $\Delta _{(f,g)}$ and $\Delta _{(f,h)}$ 
both have rank three. 
We are interested in the case where 
$\rank \Delta _{(f,g)}\leq 2$ 
and $\rank \Delta _{(f,h)}=3$.

\begin{prop}\label{prop:further}
Assume that 
$\rank \Delta _{(f,g)}\leq 2$ 
and $\rank \Delta _{(f,h)}=3$. 
Then, 
we have $\rank \Delta _{(f,g)}=2$, 
$\pl \Delta _{(f,g)}=(g)$, 
and $g$ is a coordinate of $\kx $ over $k$. 
\end{prop}
\begin{proof}
By the assumption of local slice construction, 
$\Delta _{(f,g)}$ is irreducible. 
Since $\rank \Delta _{(f,g)}\leq 2$ by assumption, 
there exists a coordinate $p$ of $\kx $ over $k$ 
such that $\Delta _{(f,g)}(p)=0$. 
Hence, 
we know by Lemma~\ref{lem:rank criterion} that 
$\pl \Delta _{(f,g)}=(q)$ 
for some $q\in k[p]\sm \zs $.

Suppose to the contrary that $\rank \Delta _{(f,g)}=1$. 
Then, 
we have $\ker \Delta _{(f,g)}=\sigma (k[x_2,x_3])$ 
for some $\sigma \in \Aut (\kx /k)$. 
Since $\ker \Delta _{(f,g)}=k[f,g]$ by (LSC1), 
it follows that $f$ and $g$ 
are coordinates of $\kx $ over $k$. 
Since $\Delta _{(f,h)}(f)=0$, 
we have $\rank \Delta _{(f,h)}\leq 2$, 
a contradiction. 
Thus, 
we get $\rank \Delta _{(f,g)}=2$. 
By Lemma~\ref{lem:slice irred}, 
this implies that 
$q$ does not belong to $k^{\times }$. 
By (LSC2), 
there exists $F\in k[f]\sm \zs $ such that 
$\Delta _{(f,g)}(s)=gF$. 
Since $D^2(s)=0$, 
we see that $D(s)$ belongs to $\pl \Delta _{(f,g)}$. 
Therefore, 
$q$ is a factor of $gF$.

We show that $q\approx g$. 
First, 
suppose that $q$ is not divisible by $g$. 
Then, 
$q$ is a factor of $F$ by the irreducibility of $g$. 
Since $q$ is an element of $k[p]\sm k$, 
we may find $\alpha \in \bar{k}$ 
such that $p-\alpha $ divides $q$. 
Then, $p-\alpha $ is a factor of $F$. 
By Lemma~\ref{lem:facts} and the remark following it, 
$\bar{k}[f]$ is factorially closed in $\bar{k}[\x ]$. 
Hence, $p-\alpha $ belongs to $\bar{k}[f]$. 
On the other hand, 
$p$ is a coordinate of 
$\bar{k}[\x ]$ over $\bar{k}$, 
since a coordinate of $\kx $ over $k$ 
is necessarily a coordinate of 
$\bar{k}[\x ]$ over $\bar{k}$. 
Hence, $p-\alpha $ is a coordinate of 
$\bar{k}[\x ]$ over $\bar{k}$. 
In particular, 
$p-\alpha $ is an irreducible element of $\bar{k}[\x ]$. 
Thus, 
$p-\alpha $ must be a linear polynomial in $f$ 
over $\bar{k}$. 
Since $p$ and $f$ are elements of $\kx $, 
this implies that $p$ 
is a linear polynomial in $f$ over $k$. 
Hence, 
$f$ is a coordinate of $\kx $ over $k$. 
Thus, 
we get $\rank \Delta _{(f,h)}\leq 2$, 
a contradiction. 
Therefore, 
$q$ is divisible by $g$. 
Next, 
suppose that $q\not\approx g$. 
Then, 
$q':=g^{-1}q$ is a factor of $F$. 
Hence, $q'$ belongs to $k[f]\sm k$ 
by the factorially closedness of $k[f]$ in $\kx $. 
Since $f$ and $g$ are algebraically independent over $k$, 
it follows that $q'$ and $g$ 
are algebraically independent over $k$. 
This contradicts that $q'g=q$ belongs to $k[p]$. 
Thus, we get $q\approx g$, 
proving that $\pl \Delta _{(f,g)}=(g)$. 
Consequently, 
$g$ belongs to $k[p]$. 
Since $g$ is irreducible in $\bar{k}[\x ]$, 
this implies that 
$q$ is a linear polynomial in $p$ over $k$. 
Therefore, 
$g$ is a coordinate of $\kx $ over $k$. 
\end{proof}

For example, 
$D_2$ is of rank three 
if $t_0\geq 3$ and $(t_0,t_1)\neq (3,1)$ 
by Theorem~\ref{prop:rank} (iv), 
and is obtained from the data $(f_2,x_1,r)$ 
by a local slice construction. 
In this case, 
$D_1=\Delta _{(f_2,x_1)}$ is triangular. 
It is previously not known whether 
there exists an example in which 
$\rank \Delta _{(f,h)}=3$, 
$\rank \Delta _{(f,g)}=2$ 
and $\Delta _{(f,g)}$ is not triangularizable. 
In closing this section, 
we construct $f,s\in \kx $ such that 
the data $(f,x_1,s)$ yields 
a rank three locally nilpotent derivation, 
and $\Delta _{(f,x_1)}$ is not triangularizable.

Define $p_1,p_2,f\in \kx $ by 
\begin{gather*}
p_1=(x_1+1)x_2-x_1^2x_3^2,\quad 
p_2=x_1x_3+(x_1+1)p_1^2\\
f=(x_1+1)^{-1}(p_1+p_2^2)=x_2+2x_1x_3p_1^2+(x_1+1)p_1^4. 
\end{gather*}
Then, 
we have 
$k(x_1)[f,p_2]=k(x_1)[p_1,p_2]=k(x_1)[p_1,x_3]=k(x_1)[x_2,x_3]$. 
We show that 
$D:=\Delta _{(f,x_1)}$ is locally nilpotent. 
Since 
$$
df\wedge dx_1\wedge dp_2=(x_1+1)^{-1}dp_1\wedge dx_1\wedge dp_2
=-x_1dx_1\wedge dx_2\wedge dx_3, 
$$
we have $D(p_2)=-x_1$. 
Since $D(f)=D(x_1)=0$, 
it follows that 
$D$ extends to a 
locally nilpotent 
derivation 
of $k(x_1)[x_2,x_3]=k(x_1)[p_2,f]$ over $R:=k(x_1)[f]$. 
Therefore, 
$D$ is locally nilpotent. 
We mention that $\ker D$ 
is contained in $R$, 
and hence in $k(x_1,f)$.

Now, following Daigle~\cite[Example 3.5]{Daigle}, 
we show that $D$ is not triangularizable 
by contradiction. 
Suppose to the contrary that $D$ is triangularizable. 
Then, 
there exists a coordinate $p$ of $\kx $ over $k[x_1]$ 
such that $k(x_1)[f,p]=k(x_1)[x_2,x_3]$ 
(see~\cite[Corollary 3.4]{Daigle}). 
Since 
$k(x_1)[x_2,x_3]=k(x_1)[f,p_2]$, 
it follows that 
$$
R[p]
=k(x_1)[f,p]
=k(x_1)[x_2,x_3]
=k(x_1)[f,p_2]
=R[p_2]. 
$$
This implies that 
$p=ap_2+b$ 
for some $a\in R^{\times }=k(x_1)^{\times }$ 
and $b\in R$. 
Hence, we may write 
$$
a_0(x_1)p=a_1(x_1)p_2+\sum _{i\geq 0}b_i(x_1)f^i, 
$$
where $a_i(x_1)\neq 0$ for $i=0,1$ 
and $b_i(x_1)$ for $i\geq 0$ 
are elements of $k[x_1]$ with no common factor. 
Since 
\begin{equation}\label{eq:further mod x_1}
p_1\equiv x_2,\quad 
p_2\equiv x_2^2,\quad 
f\equiv x_2+x_2^4\pmod{x_1\kx }, 
\end{equation}
it follows that 
$$
a_0(0)p=a_1(0)x_2^2+\sum _{i\geq 0}b_i(0)(x_2+x_2^4)^i. 
$$
When $a_0(0)=a_1(0)=0$, 
the preceding equality implies that 
$b_i(0)=0$ for every $i\geq 0$. 
Hence, 
$x_1$ is a common factor of 
$a_0(x_1)$, $a_1(x_1)$ and $b_i(x_1)$ for $i\geq 0$, 
a contradiction. 
If $a_0(0)\neq 0$ and $a_1(0)=0$, 
then we have $p\approx \sum _{i\geq 0}b_i(0)(x_2+x_2^4)^i$. 
Hence, 
$p$ is an element of $k[x_2]$, 
and is not a linear polynomial. 
This contradicts that $p$ 
is a coordinate of $\kx $ over $k[x_1]$. 
If $a_0(0)=0$ and $a_1(0)\neq 0$, 
then we have 
$$
b_0(0)+b_1(0)(x_2+x_2^4)+\left( 
a_1(0)+\sum _{i\geq 2}b_i(0)(1+x_2^3)^ix_2^{i-2}
\right) x_2^2=0. 
$$
This gives that 
$b_0(0)\equiv 0\pmod{x_2k[x_2]}$, 
and so $b_0(0)=0$. 
Hence, 
we have $b_1(0)x_2\equiv 0\pmod{x_2^2k[x_2]}$, 
and so $b_1(0)=0$. 
Then, 
it follows that $b_i(0)=0$ for every $i\geq 2$, 
and consequently $a_1(0)=0$. 
Thus, 
$x_1$ is a common factor of 
$a_0(x_1)$, $a_1(x_1)$ and $b_i(x_1)$ for $i\geq 0$, 
a contradiction. 
Therefore, 
$D$ is not triangularizable. 
In particular, 
we have $\rank D\neq 1$.

We show that $D$ is irreducible. 
Suppose that $D$ is not irreducible. 
Then, 
$D(\kx )$ is contained in $x_1\kx $, 
since $D(p_2)=-x_1$ as mentioned. 
Hence, 
$x_1^{-1}D$ belongs to $\lnd _k\kx $. 
Since $(x_1^{-1}D)(-p_2)=1$, 
we get $\rank D=\rank x_1^{-1}D=1$ 
by Lemma~\ref{lem:slice irred}, 
a contradiction. 
Therefore, $D$ is irreducible. 
Since $\ker D$ is contained in $k(f,x_1)$ 
as mentioned, 
we conclude that $\ker D=k[x_1,f]$ 
by the ``kernel criterion" 
(cf.~\cite[Proposition~5.12]{Fbook}). 
Hence, 
$D=\Delta _{(f,x_1)}$ satisfies (LSC1) 
for $(f,g)=(f,x_1)$. 
Note that $s:=p_2f+x_1^2$ does not belong to $x_1\kx $, 
since 
$$
s\equiv p_2f\equiv x_2^2(x_2+x_2^4)
\not\equiv 0\pmod{x_1\kx }. 
$$ 
Moreover, 
we have 
$$
D(s)=D(p_2f+x_1^2)=D(p_2)f=-x_1f. 
$$
Hence, 
$D$ satisfies (LSC2) for $s=p_2f+x_1^2$ and $F=-f$.

Now, define 
$q=(s^2-f^3)^2-f^3s$. 
Since $s\equiv p_2f\pmod{x_1\kx }$, 
we see that 
$q$ is congruent to 
$$
(p_2^2f^2-f^3)^2-f^4p_2=
f^4\bigl((p_2^2-f)^2-p_2\bigr)
$$
modulo $x_1\kx $. 
By (\ref{eq:further mod x_1}), 
this polynomial is congruent to 
$$
f^4\Bigl(\bigl((x_2^2)^2-(x_2+x_2^4)\bigr)^2-x_2^2\Bigr)=0
$$ 
modulo $x_1\kx $. 
Hence, 
$q$ belongs to $x_1\kx $. 
We show that $q$ is an irreducible element of $k[f,s]$. 
Write 
$q=f^4((t^2-f)^2-t)$, 
where $t:=s/f$. 
Then, 
$q':=(t^2-f)^2-t$ is a coordinate of $k[f,t]$ over $k$, 
since $k[q',t^2-f]=k[f,t]$. 
Hence, 
$q'$ is an irreducible element of $k[f,t]$. 
Since $q'$ does not belong to $k[f]$, 
it follows that $q'$ is an irreducible polynomial 
in $t$ over $k[f]$, 
and hence over $k(f)$. 
Thus, 
$q$ is an irreducible polynomial in $t$ over $k(f)$. 
Since $t=s/f$, 
this implies that $q$ 
is an irreducible polynomial in $s$ over $k(f)$. 
Because $q$ is a primitive polynomial in $s$ over $k[f]$, 
we conclude that $q$ is an irreducible element of $k[f,s]$. 
Put $h=qx_1^{-1}$. 
Then, 
it follows from (a) and (b) of Theorem~\ref{thm:fbook} 
that $D':=\Delta _{(f,h)}$ is locally nilpotent and $D'(s)=hf$.

We show that $D'$ is irreducible. 
Thanks to Proposition~\ref{prop:lsc:F:rem}, 
it suffices to check that $f$ and $D'(p_2)$ 
have no common factor. 
Since $\ker D=k[f,x_1]$, 
we know by Lemma~\ref{lem:facts} that 
$f\kx $ is a prime ideal of $\kx $. 
Hence, 
it suffices to verify that $D'(p_2)\not\equiv 0\pmod{f\kx }$. 
Since $x_1=qh^{-1}$, 
$D'(f)=D'(h)=0$ 
and $D'(s)=hf$, 
we have 
$$
D'(x_1)=D'(qh^{-1})=D'(q)h^{-1}=
\frac{\partial q}{\partial s}D'(s)h^{-1}
=(4s(s^2-f^3)-f^3)f
$$
by chain rule. 
On the other hand, 
we have 
$$
hf=D'(s)=D'(p_2f+x_1^2)=D'(p_2)f+2x_1D'(x_1). 
$$
From the two equalities above, 
it follows that 
$$
D'(p_2)=h-2x_1\bigl(4s(s^2-f^3)-f^3\bigr). 
$$
Note that $s\equiv x_1^2$, 
$q\equiv s^4\equiv x_1^8$ and 
$h=qx_1^{-1}\equiv x_1^7\pmod{f\kx }$. 
By Lemma~\ref{lem:facts}, 
$x_1$ does not belong to $f\kx $. 
Thus, 
we know that 
$$
D'(p_2)\equiv h-8x_1s^3\equiv -7x_1^7\not\equiv 0
\pmod{f\kx }. 
$$ 
Therefore, 
$D'$ is irreducible. 
Consequently, 
we get $\ker D'=k[f,h]$ by Theorem~\ref{thm:fbook} (c).

Since $D'(s)=hf$ belongs to $\pl D'$, 
there exists a factor $g'$ of $hf$ 
such that $\pl D'=(g')$. 
We show that $g'\approx hf$. 
Since $\ker D'=k[f,h]$, 
we know by Lemma~\ref{lem:facts} 
that $f$ and $h$ are irreducible 
elements of $\kx $ with $f\not\approx h$. 
Hence, 
it suffices to check that $g'$ 
is divisible by $h$ and $f$.

Suppose to the contrary that 
$g'$ is not divisible by $h$. 
Then, 
$\pl D'$ is not contained in $h\kx $. 
Hence, 
we have $(s+\ker D')\cap h\kx \neq \emptyset $ 
by the last part of Lemma~\ref{lem:rank criterion}. 
Since $\ker D'=k[f,h]$, 
it follows that 
$(s+k[f])\cap h\kx \neq \emptyset $. 
Hence, 
$k[f,s]\cap h\kx $ 
contains a linear polynomial in $s$ over $k[f]$. 
On the other hand, 
$k[f,s]\cap h\kx $ 
is a principal prime ideal of $k[f,s]$ 
by Lemma~\ref{lem:minimal} (iii). 
Since $q=x_1h$ belongs to $k[f,s]\cap h\kx $, 
and is an irreducible element of $k[f,s]$, 
we know that $k[f,s]\cap h\kx $ 
is generated by $q$. 
Because $\deg _sq=4>1$, 
it follows that 
$k[f,s]\cap h\kx $ 
contains no linear polynomial in $s$ over $k[f]$, 
a contradiction. 
Therefore, 
$g'$ is divisible by $h$.

Suppose to the contrary that 
$g'$ is not divisible by $f$. 
Then, 
$\pl D'$ is not contained in $f\kx $. 
Hence, 
we have $(s+\ker D')\cap f\kx \neq \emptyset $ by 
Lemma~\ref{lem:rank criterion}. 
Since $\ker D'=k[f,h]$, 
we get 
$(s+k[h])\cap f\kx \neq \emptyset $. 
Hence, 
$k[h,s]\cap f\kx $ 
contains a linear polynomial in $s$ over $k[h]$. 
On the other hand, 
$k[h,s]\cap f\kx $ 
is a principal prime ideal of $k[h,s]$ 
by Lemma~\ref{lem:minimal} (iii). 
Since 
$s^7-h^2$ is an irreducible element of $k[h,s]$ 
with $s^7-h^2\equiv 0\pmod{f\kx }$, 
it follows that 
$k[h,s]\cap f\kx $ is generated by $s^7-h^2$. 
This contradicts that 
$k[h,s]\cap f\kx $ 
contains a linear polynomial in $s$ over $k[f]$. 
Therefore, 
$g'$ is divisible by $f$. 
This proves that 
$g'\approx hf$. 
Because $f$ and $h$ are algebraically independent over $k$, 
we conclude that $\rank D'=3$ 
by Proposition~\ref{prop:rank3criterion}.

\chapter*{Conclusion}
\setcounter{equation}{0}
\label{chap:quest}

In closing this monograph, 
we list some problems, questions and conjectures. 
We assume that $n=3$, 
and $D$ is a nonzero element of $\lnd _k\kx $ 
unless otherwise stated.

\begin{conj}\label{conj:tt}\rm 
If $\exp D$ is tame, 
then $D$ is tamely triangularizable. 
\end{conj}

We claim that Conjecture~\ref{conj:tt} 
is equivalent to the following conjecture.

\begin{conj}\label{conj:no tame coord}\rm 
If $\exp D$ is tame, 
then $D$ kills a tame coordinate of $\kx $ over $k$. 
\end{conj}

In fact, 
if $D$ is tamely triangularizable, 
then $D$ kills a tame coordinate of $\kx $ over $k$, 
since a triangular derivation always kills a tame coordinate. 
Conversely, 
Conjecture~\ref{conj:no tame coord} 
implies Conjecture~\ref{conj:tt} 
by virtue of Theorem~\ref{thm:anstoF} (i).

By the remark after Problem~\ref{q:strong}, 
Conjecture~\ref{conj:tt} implies the following conjecture. 
This conjecture is true 
if $D$ kills a tame coordinate of $\kx $ over $k$.

\begin{conj}\label{conj:fixed points}\rm 
If $\exp fD$ is tame for some $f\in \ker D\sm \zs $, 
then $\exp D$ is tame. 
\end{conj}

Conjecture~\ref{conj:no tame coord} immediately 
implies the following conjecture.

\begin{conj}\label{conj:rank3}\rm 
If $\rank D=3$, then $\exp D$ is wild. 
\end{conj}

Next, we discuss totally 
(quasi-totally, exponentially) wildness of 
coordinates of $\kx $ over $k$.

\begin{pp}
For $D\in \lnd _k\kx $ with $\rank D=3$, 
study totally 
(quasi-totally, exponentially) wildness of 
$(\exp D)(x_i)$ for $i=1,2,3$. 
\end{pp}

As we have seen in Chapter~\ref{chapter:atcoord}, 
it is very hard to prove that 
a coordinate is totally wild 
or quasi-totally wild.

\begin{pp}
Find simple criteria for totally 
(quasi-totally, exponentially) wildness 
of coordinates of $\kx $ over $k$. 
\end{pp}

The following question is not answered 
even for Nagata's automorphism.

\begin{qqq}\label{q:ac}\rm 
Assume that $\sigma \in \Aut (\kx /k)$ is wild. 
Is one of $\sigma (x_1)$, $\sigma (x_2)$ and $\sigma (x_3)$ 
always quasi-totally wild? 
\end{qqq}

As remarked after Definition~\ref{defin:wild coordinates}, 
``quasi-totally wild" implies ``exponentially wild". 
In view of Corollary~\ref{cor:expw=>qtw} (i), 
we may ask the following question.

\begin{qqq}\label{qqq:awcexp}\rm 
Is every exponentially wild coordinate of $\kx $ over $k$ 
quasi-totally wild? 
\end{qqq}

Recall that there exists a wild coordinate 
of $\kx $ over $k$ 
which is not exponentially wild.

\begin{conj}\label{conj:qawc}\rm 
Let $f$ be a wild coordinate of $\kx $ over $k$ 
which is not exponentially wild. 
Then, there always 
exists $D\in \lnd _k\kx $ 
with $\rank D=1$ such that $D(f)=0$ and 
$D$ is tamely triangularizable. 
\end{conj}

For example, 
let $\psi $ be Nagata's automorphism 
defined in (\ref{eq:Nagataautom}). 
Then, 
$\psi (x_2)$ is wild due to Umirbaev-Yu~\cite{UY}, 
but is not exponentially wild as mentioned. 
In fact, 
$\psi (x_2)$ is killed by $D\in \lnd _k\kx $ 
defined in (\ref{eq:quasi-absolutely wild}), 
which is triangular if $x_1$ and $x_3$ are interchanged. 
Since $D$ kills $x_3=\psi (x_3)$ as well as $\psi (x_2)$, 
we know that $\rank D=1$. 
Hence, the conjecture is true in this case.

We show that Conjectures~\ref{conj:tt} 
implies Conjecture~\ref{conj:qawc}\null. 
Since $f$ is not exponentially wild, 
there exists 
$D\in \lnd _k\kx \sm \zs $ 
such that $D(f)=0$ and $\exp D$ is tame. 
Then, 
$D$ is tamely triangularizable by Conjecture~\ref{conj:tt}, 
and hence 
kills a tame coordinate $g$ of $\kx $ over $k$. 
Since $f$ is wild and $g$ is tame, 
we know that $f$ is not a linear polynomial in $g$ over $k$. 
Hence, we have $k[f]\neq k[g]$. 
Therefore, 
we conclude that $\rank D=1$ 
from the following lemma.

\begin{lem}\label{lem:2coord}
Let $f$ and $g$ be coordinates of $\kx $ over $k$ 
such that $k[f]\neq k[g]$. 
If $D(f)=D(g)=0$ for $D\in \lnd _k\kx \sm \zs $, 
then we have $\rank D=1$. 
\end{lem}
\begin{proof}
First, we show that $f$ and $g$ 
are algebraically independent over $k$. 
Suppose to the contrary that 
$f$ and $g$ are algebraically dependent over $k$. 
Then, 
$f$ is algebraic over $k(g)$. 
Since $g$ is a coordinate of $\kx $ over $k$, 
we see that $k(g)$ is algebraically closed in $\kxr $, 
and $k(g)\cap \kx =k[g]$. 
Hence, $f$ belongs to $k[g]$. 
Similarly, 
$g$ belongs to $k[f]$. 
Thus, we get $k[f]=k[g]$, 
a contradiction. 
Therefore, $f$ and $g$ 
are algebraically independent over $k$.

Now, we prove that $\rank D=1$. 
Without loss of generality, 
we may assume that $D$ is irreducible. 
Then, 
we have $\pl D=(p)=(q)$ for some 
$p\in k[f]\sm \zs $ and $q\in k[g]\sm \zs $ 
by Proposition~\ref{lem:rank criterion}. 
Since $(\ker D)^{\times }=k^{\times }$, 
it follows that $p=cq$ for some $c\in k^{\times }$. 
Because $f$ and $g$ 
are algebraically independent over $k$, 
this implies that 
$p$ and $q$ belong to $k^{\times }$. 
Thus, we get $\pl D=\ker D$. 
Therefore, 
we know by Lemma~\ref{lem:slice irred} that $\rank D=1$. 
\end{proof}

For any $n\geq 3$, 
it is widely believed that every $\sigma \in \Aut (\kx /k)$ 
is {\it stably tame}, 
i.e., the natural extension of $\sigma $ to an element of 
$$
\Aut (k[x_1,\ldots ,x_r]/k[x_{n+1},\ldots ,x_r])
$$ 
belongs to $\T (k,\{ x_1,\ldots ,x_r\} )$ 
for some $r\geq n$ 
(Stable Tameness Conjecture). 
Assume that $n=3$. 
Then, 
every element of $\Aut (\kx /k[x_3])$ is stably tame 
due to Berson-van den Essen-Wright~\cite{BEW} 
(see also \cite{Smith}). 
Hence, 
$\exp D$ is stably tame for each 
$D\in \lnd _k\kx $ with $D(x_3)=0$.

The following question is open.

\begin{qqq}
Is $\exp D$ stably tame 
for $D\in \lnd _k\kx $ with $\rank D=3$? 
\end{qqq}

We mention that some $D\in \lnd _k\kx $ 
with $\rank D=3$ can be extended to 
elements of $\lnd _kk[x_1,\ldots ,r]$ 
with $\rank D=r$ for each $r\geq 4$ 
in a very simple manner 
(cf.~\cite[Section 3]{Frank}). 
In view of this fact, 
we may ask the following question, 
in contrast with Conjecture~\ref{conj:rank3}.

\begin{qqq}
Assume that $n\geq 4$. 
Does there exist $D\in \lnd _k\kx $ such that 
$\rank D=n$ and $\exp D$ is tame? 
\end{qqq}

For each subgroup $G$ of $\Aut (\kx /k)$, 
we define $G^*$ 
to be the normal subgroup of $\Aut (\kx /k)$ 
generated by 
$$
\bigcup _{\sigma \in \Aut (\kx /k)}
\sigma ^{-1}\circ G\circ \sigma . 
$$
We say that $\phi \in \Aut (\kx /k)$ is 
{\it absolutely wild} if 
$\phi $ does not belong to $\T (k,\x )^*$. 
We note that the wild automorphism 
defined as in (\ref{eq:scalar auto}) 
is not absolutely wild. 
We do not know whether 
Nagata's automorphism is absolutely wild.

The following is an 
``absolute" version of the Tame Generators Problem.

\begin{pp}\rm 
Decide whether $\T (k,\x )^*$ equal to $\Aut (\kx /k)$. 
\end{pp}

For each $\alpha \in k^{\times }$, 
define $\psi _{\alpha }\in \Aut (\kx /A_1)$ 
by $\psi _{\alpha }(x_1)=\alpha x_1$, 
where $A_i$ is as in (\ref{eq:A_i}) 
for each $i\in \{ 1,2,3\} $. 
Then, 
$$
\iota :k^{\times }\ni \alpha 
\mapsto \psi _{\alpha }\in \Aut (\kx /k)
$$ 
is an injective homomorphism of groups. 
We set $G_m(k)=\iota (k^{\times })$.

\begin{prop}\label{prop:G_m}
We have 
$G_m(k)^*=\T (k,\x )^*$. 
\end{prop}
\begin{proof}
Since $\T(k,\x )=\E (k,\x )$, 
it suffices to show that $G_m(k)^*$ 
contains $\Aut (\kx /A_i)$ for $i=1,2,3$. 
For $i=2,3$, 
define $\tau _i\in \Aut (\kx /k)$ by 
$\tau _i(x_1)=x_i$, $\tau _i(x_i)=x_1$ 
and $\tau _i(x_j)=x_j$, 
where $j\in \{ 2,3\} $ with $j\neq i$. 
Then, 
we have 
$\tau _i^{-1}\circ \Aut (\kx /A_1)\circ \tau _i=\Aut (\kx /A_i)$. 
So we show that 
$\Aut (\kx /A_1)$ is contained in $G_m(k)^*$. 
Take any $\phi \in \Aut (\kx /A_1)$. 
Then, we have $\phi (x_1)=\alpha x_1+h$ 
for some $\alpha \in k^{\times }$ and $h\in A_1$. 
Define $\sigma \in G_m(k)^*$ by 
$\sigma (x_1+h)=2(x_1+h)$ 
and $\sigma (x_i)=x_i$ for $i=2,3$. 
Then, 
we have 
$\sigma (x_1)=2x_1+h$, 
and so 
$(\psi _{\alpha /2}\circ \sigma )(x_1)=\alpha x_1+h$. 
Since $(\psi _{\alpha /2}\circ \sigma )(x_i)=x_i$ 
for $i=2,3$, 
we conclude that $\psi _{\alpha /2}\circ \sigma =\phi $. 
Hence, $\phi $ belongs to $G_m(k)^*$. 
Thus, 
$\Aut (\kx /A_1)$ is contained in $G_m(k)^*$. 
Therefore, 
we get $G_m(k)^*=\T (k,\x )^*$. 
\end{proof}

Note that 
$$
\pi :\Aut (\kx /k)\ni \phi \mapsto \det J\phi \in k^{\times }
$$ 
is a homomorphism of groups such that 
$\pi \circ \iota =\id _{k^{\times }}$. 
Hence, 
$\Aut (\kx /k)$ is a semidirect product of 
$\ker \pi $ and $G_m(k)$. 
The {\it Exponential Generators Conjecture} says that 
$\ker \pi $ is generated by $\exp D$ for $D\in \lnd _k\kx $. 
For $\emptyset \neq I\subset \{ 1,2,3\} $, 
we define $\Exp _k^I\kx $ to be the subgroup of $\Aut (\kx /k)$ 
generated by $\exp D$ for $D\in \lnd _k\kx $ with $\rank D=r$ 
for $r\in I$. 
Then, 
$\Exp _k^I\kx $ is a normal subgroup of $\Aut (\kx /k)$ 
for each $\emptyset \neq I\subset \{ 1,2,3\} $. 
However, 
we know nothing about the structure of $\Exp _k^I\kx $.

We would like to conclude this monograph 
with the following remark. 
When $n\geq 4$, 
there exist complicated tame automorphisms. 
Indeed, 
some wild automorphisms in three variables 
extend to tame automorphisms in four or more variables 
by the stable tameness. 
When $n=3$, in contrast, 
the Shestakov-Umirbaev theory suggests that 
the tame automorphisms are simple. 
Indeed, 
the tame automorphisms in three variables 
can be completely controlled 
by elementary reductions and 
some types of reductions. 
We believe that 
no tame automorphism in three variables 
is beyond imagination. 
In other words, 
automorphisms in three variables are 
{\it usually} wild unless 
they are {\it obviously} tame.

\bibliographystyle{amsalpha}

\end{document}